\documentclass[11pt]{preprint}
\usepackage[utf8]{inputenc}
\usepackage[english]{babel}
\usepackage{amssymb}
\usepackage{amsmath}
\usepackage{hyperref}
\usepackage{breakurl}
\usepackage{mhenvs}
\usepackage{mhequ}
\usepackage{mhsymb}
\usepackage{booktabs}
\usepackage{tikz}
\usepackage{mathrsfs}
\usepackage{times}
\usepackage{microtype}
\usepackage{comment}
\usepackage{wasysym}
\usepackage{centernot}
\usepackage{leftidx}
\usepackage{arydshln}
\usepackage{footnote}
\makesavenoteenv{tabular}
\usepackage{enumitem}
\usepackage{stackrel}
\usepackage{halloweenmath}
\usepackage{longtable}
\usepackage{scalerel}
\usepackage{array}
\usepackage{cprotect}
\usepackage{ulem}

\DeclareSymbolFont{timesoperators}{T1}{ptm}{m}{n}
\SetSymbolFont{timesoperators}{bold}{T1}{ptm}{b}{n}

\makeatletter
\renewcommand{\operator@font}{\mathgroup\symtimesoperators}
\makeatother

\colorlet{symbols}{blue!30!black!50}
\colorlet{testcolor}{green!60!black}
\definecolor{purple}{rgb}{0.55,0.05,0.8}

\usetikzlibrary{shapes.misc}
\usetikzlibrary{shapes.symbols}
\usetikzlibrary{shapes.geometric}
\usetikzlibrary{snakes}
\usetikzlibrary{decorations}
\usetikzlibrary{decorations.markings}

\usetikzlibrary{calc}
\usetikzlibrary{external}

\newtheorem{example}[lemma]{Example}

\let\oldskull\skull
\def\skull{\mathord{\oldskull}}

\def\sol{{\mathop{\mathrm{sol}}}}
\def\lsol{{\mathop{\mathrm{lsol}}}}
\def\pr{\textcolor{purple}}
\def\expan{\mcb{H}} 
\def\initial{{\mathop{\mathrm{init}}}}

\def\restr{\mathbin{\upharpoonright}}

\def\Tran{\mathrm{Tran}}

\DeclareMathAlphabet{\mathbbm}{U}{bbm}{m}{n}

\DeclareFontFamily{U}{BOONDOX-calo}{\skewchar\font=45 }
\DeclareFontShape{U}{BOONDOX-calo}{m}{n}{
  <-> s*[1.05] BOONDOX-r-calo}{}
\DeclareFontShape{U}{BOONDOX-calo}{b}{n}{
  <-> s*[1.05] BOONDOX-b-calo}{}
\DeclareMathAlphabet{\mcb}{U}{BOONDOX-calo}{m}{n}
\SetMathAlphabet{\mcb}{bold}{U}{BOONDOX-calo}{b}{n}

\setlist{noitemsep,topsep=4pt}

\makeatletter 
\newcommand*{\bigcdot}{}
\DeclareRobustCommand*{\bigcdot}{%
  \mathbin{\mathpalette\bigcdot@{}}%
}
\newcommand*{\bigcdot@scalefactor}{.5}
\newcommand*{\bigcdot@widthfactor}{1.15}
\newcommand*{\bigcdot@}[2]{%
  \sbox0{$#1\vcenter{}$}
  \sbox2{$#1\cdot\m@th$}%
  \hbox to \bigcdot@widthfactor\wd2{%
    \hfil
    \raise\ht0\hbox{%
      \scalebox{\bigcdot@scalefactor}{%
        \lower\ht0\hbox{$#1\bullet\m@th$}%
      }%
    }%
    \hfil
  }%
}
\makeatother

\def\symbol#1{\textcolor{symbols}{#1}}
\def\1{\mathbf{\symbol{1}}}

\def\bone{\mathbf{1}}

\def\higgsvec{\mathbf{V}}

\let\p\partial

\def\bPsi{\boldsymbol{\Psi}}
\def\tbPsi{\tilde{\boldsymbol{\Psi}}}
\def\id{\mathrm{id}}
\def\PPi{\boldsymbol{\Pi}}

\def\bbrho{\bar{\boldsymbol{\rho}}}
\def\brho{\boldsymbol{\rho}}

\def\Lab{\mathfrak{L}}

\def\can{\textnormal{\scriptsize can}}

\def\BPHZ{\textnormal{\tiny \textsc{bphz}}}
\def\YMH{\textnormal{\tiny \textsc{ymh}}}
\def\YM{\textnormal{\tiny \textsc{ym}}}

\def\Higgs{\textnormal{\tiny {Higgs}}}
\def\Gauge{\textnormal{\tiny {Gauge}}}
\def\GaugeH{\textnormal{\tiny {GaugeH}}}
\def\even{\textnormal{\tiny even}}

\def\A{\textnormal{A}}
\def\h{\textnormal{h}}
\def\lead{\textnormal{\tiny lead}}

\def\moll{\chi}
\def\Hol#1{#1\textnormal{-H{\"o}l}}

\def\hol{\textnormal{hol}}
\def\gr#1{#1\textnormal{-gr}}

\def\state{\mathcal{S}}
\def\init{\mathcal{I}}
\let\ymhflow\CF
\def\flow{\CE}
\def\Area{\textnormal{Area}}

\usepackage{stmaryrd}
\def\fancynorm#1{{\talloblong #1 \talloblong}}
\def\heatgr#1{{|\!|\!| #1 |\!|\!|}}

\newcommand{\mrd}{\mathrm{d}}

\newcommand{\roof}[1]{\lceil #1 \rceil}

\colorlet{darkblue}{blue!90!black}
\colorlet{darkgreen}{green!50!black}

\def\s{\mathfrak{s}}
\def\KK{\mathfrak{K}}

\def\reg{\mathrm{reg}}

\def\Func{\mathbf{F}}
\let\eref\eqref

\newcommand{\e}{\varepsilon}

\def\MM{\mathscr{M}}

\def\${|\!|\!|}

\def\E{\mathbf{E}}
\def\T{\mathbf{T}}

\def\P{\mathbf{P}}

\def\-{\mbox{-}}

\def\f{\mathfrak{f}}

\def\bXi{\boldsymbol{\Xi}}

\def\Poly{\mathop{\mathrm{Poly}}}

\newcommand{\mfu}{\mathfrak{u}}
\newcommand{\mfU}{\mathfrak{U}}

\newcommand{\mfF}{\mathfrak{F}}

\newcommand{\mfT}{\mathfrak{T}}
\newcommand{\mfn}{\mathfrak{n}}
\newcommand{\mfo}{\mathfrak{o}}

\newcommand{\mfL}{\mathfrak{L}}

\newcommand{\mfa}{\mathfrak{a}}

\newcommand{\mfr}{\mathfrak{r}}
\newcommand{\mft}{\mathfrak{t}}
\newcommand{\mfm}{\mathfrak{m}}

\newcommand{\mfl}{\mathfrak{l}}
\newcommand{\mfz}{\mathfrak{z}}
\newcommand{\mfh}{\mathfrak{h}}
\newcommand{\mfq}{\mathfrak{q}}
\newcommand{\mfb}{\mathfrak{b}}

\newcommand{\mfG}{\mathfrak{G}}
\newcommand{\mfO}{\mathfrak{O}}

\newcommand{\mfg}{\mathfrak{g}}
\newcommand{\mfk}{\mathfrak{k}}

\newcommand{\mcH}{\mathcal{H}}

\newcommand{\mcA}{\mathcal{A}}

\newcommand{\mcR}{\mathcal{R}}
\newcommand{\mcC}{\mathcal{C}}
\newcommand{\mcB}{\mathcal{B}}
\newcommand{\mcJ}{\mathcal{J}}

\newcommand{\mcU}{\mathcal{U}}
\newcommand{\mcL}{\mathcal{L}}

\newcommand{\mcF}{\mathcal{F}}
\newcommand{\mcY}{\mathcal{Y}}

\newcommand{\mcD}{\mathcal{D}}

\newcommand{\mcX}{\mathcal{X}}

\def\symset{\mcb{s}}

\newcommand{\mbF}{\mathbf{F}}

\newcommand{\mbW}{\mathbf{W}}

\newcommand{\mbX}{\mathbf{X}}

\newcommand{\ad}{\mathrm{ad}}
\newcommand{\Ad}{\mathrm{Ad}}
\newcommand{\End}{\mathrm{End}}

\newcommand{\noise}{\mathrm{noise}}

\def\cS{\mathscr{S}}
\def\cD{\mathscr{D}}

\def\sig{\boldsymbol{\sigma}}

\let\tree\sigma

\def\emptyset{{\centernot\Circle}}

\def\act{\bigcdot}

\def\bUpsilon{\boldsymbol{\Upsilon}}
\def\bbUpsilon{\boldsymbol{\bar{\Upsilon}}}

\def\Cas{\mathord{\mathrm{Cas}}}
\def\Cov{\mathrm{Cov}}
\def\bCov{\mathbf{Cov}}

\numberwithin{equation}{section}

\def\dash{\leavevmode\unskip\kern0.18em--\penalty\exhyphenpenalty\kern0.18em}
\def\slash{\leavevmode\unskip\kern0.15em/\penalty\exhyphenpenalty\kern0.15em}

\def\hstar{\mathbin{\hat{*}}}
\let\emph\textit

\definecolor{pagebackground}{rgb}{1,1,1}

\makeatletter
\pgfdeclareshape{crosscircle}
{
  \inheritsavedanchors[from=circle] 
  \inheritanchorborder[from=circle]
  \inheritanchor[from=circle]{north}
  \inheritanchor[from=circle]{north west}
  \inheritanchor[from=circle]{north east}
  \inheritanchor[from=circle]{center}
  \inheritanchor[from=circle]{west}
  \inheritanchor[from=circle]{east}
  \inheritanchor[from=circle]{mid}
  \inheritanchor[from=circle]{mid west}
  \inheritanchor[from=circle]{mid east}
  \inheritanchor[from=circle]{base}
  \inheritanchor[from=circle]{base west}
  \inheritanchor[from=circle]{base east}
  \inheritanchor[from=circle]{south}
  \inheritanchor[from=circle]{south west}
  \inheritanchor[from=circle]{south east}
  \inheritbackgroundpath[from=circle]
  \foregroundpath{
    \centerpoint%
    \pgf@xc=\pgf@x%
    \pgf@yc=\pgf@y%
    \pgfutil@tempdima=\radius%
    \pgfmathsetlength{\pgf@xb}{\pgfkeysvalueof{/pgf/outer xsep}}%
    \pgfmathsetlength{\pgf@yb}{\pgfkeysvalueof{/pgf/outer ysep}}%
    \ifdim\pgf@xb<\pgf@yb%
      \advance\pgfutil@tempdima by-\pgf@yb%
    \else%
      \advance\pgfutil@tempdima by-\pgf@xb%
    \fi%
    \pgfpathmoveto{\pgfpointadd{\pgfqpoint{\pgf@xc}{\pgf@yc}}{\pgfqpoint{-0.707107\pgfutil@tempdima}{-0.707107\pgfutil@tempdima}}}
    \pgfpathlineto{\pgfpointadd{\pgfqpoint{\pgf@xc}{\pgf@yc}}{\pgfqpoint{0.707107\pgfutil@tempdima}{0.707107\pgfutil@tempdima}}}
    \pgfpathmoveto{\pgfpointadd{\pgfqpoint{\pgf@xc}{\pgf@yc}}{\pgfqpoint{-0.707107\pgfutil@tempdima}{0.707107\pgfutil@tempdima}}}
    \pgfpathlineto{\pgfpointadd{\pgfqpoint{\pgf@xc}{\pgf@yc}}{\pgfqpoint{0.707107\pgfutil@tempdima}{-0.707107\pgfutil@tempdima}}}
  }
}
\makeatother

\colorlet{greennode}{green!50!black}
\colorlet{rednode}{red!50!black}
\colorlet{lbluenode}{blue!25}
\colorlet{dbluenode}{blue}
\colorlet{orangenode}{orange}

\definecolor{connection}{rgb}{0.7,0.1,0.1}

\tikzset{
root/.style={circle,fill=black!50,inner sep=0pt, minimum size=3mm},
        var/.style={circle,fill=black!10,draw=black,inner sep=0pt, minimum size=3mm},
        kernel/.style={semithick,shorten >=2pt,shorten <=2pt},
        kernel1/.style={thick},
        kernels/.style={snake=zigzag,shorten >=2pt,shorten <=2pt,segment amplitude=1pt,segment length=4pt,line before snake=2pt,line after snake=5pt,},
        rho/.style={densely dashed,semithick,shorten >=2pt,shorten <=2pt},
           testfcn/.style={dotted,semithick,shorten >=2pt,shorten <=2pt},
           tau/.style={circle,inner sep=1pt,draw=black,fill=white,text=black,thin},
        renorm/.style={shape=circle,fill=white,inner sep=1pt},
        labl/.style={shape=rectangle,fill=white,inner sep=1pt},
        xi/.style={very thin,circle,fill=lbluenode,draw=symbols,inner sep=0pt,minimum size=1.2mm},
	xis/.style={very thin,diamond,fill=lbluenode,draw=symbols,inner sep=0pt,minimum size=1.5mm},
        xigreen/.style={very thin,circle,fill=greennode,draw=symbols,inner sep=0pt,minimum size=1.2mm},
        xigreen1/.style={very thin,rectangle,fill=greennode,draw=symbols,inner sep=0pt,minimum size=1.2mm},
        xired/.style={very thin,circle,fill=rednode,draw=symbols,inner sep=0pt,minimum size=1.2mm},
        xilblue/.style={very thin,circle,fill=lbluenode,draw=symbols,inner sep=0pt,minimum size=1.2mm},
        xiorange/.style={very thin,circle,fill=orangenode,draw=symbols,inner sep=0pt,minimum size=1.2mm},
        xix/.style={crosscircle,fill=lbluenode,draw=symbols,inner sep=0pt,minimum size=1.2mm},
xix-green-red/.style={circle, fill=greennode!70!white,draw=rednode,inner sep=0pt,minimum size=1.6mm,append after command={node [inner sep=0pt,minimum size=0.8mm,thick, draw = rednode, cross out]{}}},
xix-green-red1/.style={rectangle, fill=greennode!70!white,draw=rednode,inner sep=0pt,minimum size=1.5mm,append after command={node [inner sep=0pt,minimum size=1mm,thick, draw = rednode, cross out]{}}},
	xib/.style={very thin,circle,fill=lbluenode,draw=symbols,inner sep=0pt,minimum size=1.6mm},
	xisb/.style={very thin,diamond,fill=lbluenode,draw=symbols,inner sep=0pt,minimum size=1.9mm},
	xie/.style={very thin,circle,fill=greennode,draw=symbols,inner sep=0pt,minimum size=1.6mm},
	xid/.style={very thin,circle,fill=lbluenode,draw=symbols,inner sep=0pt,minimum size=1.6mm},
	xibx/.style={crosscircle,fill=lbluenode,draw=symbols,inner sep=0pt,minimum size=1.6mm},
	kernels2/.style={ultra thick,draw=symbols,segment length=12pt},
	not/.style={thin,regular polygon, regular polygon sides=3,draw=connection,fill=connection,inner sep=0pt,minimum size=1.2mm},
	notlblue/.style={thin,regular polygon, regular polygon sides=3,draw=lbluenode,fill=lbluenode,inner sep=0pt,minimum size=1.2mm},
	notorange/.style={thin,regular polygon, regular polygon sides=3,draw=orangenode,fill=orangenode,inner sep=0pt,minimum size=1.2mm},
	notgreen/.style={thin,regular polygon, regular polygon sides=3,draw=greennode,fill=greennode,inner sep=0pt,minimum size=1.2mm},
	>=stealth,
  }

\makeatletter
\def\DeclareSymbol#1#2#3{%
	\expandafter\gdef\csname MH@symb@#1\endcsname{\tikzsetnextfilename{symbol#1}%
	\tikz[baseline=#2,scale=0.15,draw=symbols,line join=round]{#3}}%
	\expandafter\gdef\csname MH@symb@#1s\endcsname{\scalebox{0.75}{\tikzsetnextfilename{symbol#1}%
	\tikz[baseline=#2,scale=0.15,draw=symbols,line join=round]{#3}}}%
	\expandafter\gdef\csname MH@symb@#1ss\endcsname{\scalebox{0.65}{\tikzsetnextfilename{symbol#1}%
	\tikz[baseline=#2,scale=0.15,draw=symbols,line join=round]{#3}}}%
	}
\def\<#1>{\ifmmode\mathchoice{\csname MH@symb@#1\endcsname}{\csname MH@symb@#1\endcsname}{\csname MH@symb@#1s\endcsname}{\csname MH@symb@#1ss\endcsname}\else\csname MH@symb@#1\endcsname\fi}
\makeatother

\DeclareSymbol{generic}{0}{
\draw (0,0.6) node[xi] {};
}

\DeclareSymbol{thin}{1.4}{\draw[pagebackground] (-0.3,0) -- (0.3,0); \draw  (0,0) -- (0,2);}
\DeclareSymbol{thick}{1.4}{\draw[pagebackground] (-0.3,0) -- (0.3,0); \draw[kernels2]  (0,0) -- (0,2);}

\DeclareSymbol{KPsi*Psi d1}{0}{
\draw[kernels2] (0,0) node[not] {} -- (1,1) node[xi] {} ; 
\draw (0,0) node {}  -- (-1,1) node {};
\draw  (-1,1) node {} -- (0,2) node[xi] {};
}
\DeclareSymbol{KPsi*Psi d2}{0}{
\draw (0,0) -- (1,1) node[xi] {} ; 
\draw[kernels2] (0,0) node[not] {}  -- (-1,1) node[not] {};
\draw  (-1,1) node {} -- (0,2) node[xi] {};
}
\DeclareSymbol{KPsi*Psi d3}{0}{
\draw (0,0) -- (1,1) node[xi] {} ; 
\draw (0,0) node {}  -- (-1,1) node {};
\draw[kernels2]  (-1,1) node[not] {} -- (0,2) node[xi] {};
}

\DeclareSymbol{IXi^2asym}{-1}{\draw (-1,1) node[xired] {} -- (0,-.4);
\draw (1,1) node[xigreen] {} -- (0,-.4);}

\DeclareSymbol{IXi^2asym2}{-1}{\draw (-1,1) node[xigreen] {} -- (0,-.4);
\draw (1,1) node[xired] {} -- (0,-.4);}

\DeclareSymbol{IXi^2asym3}{-1}{\draw (-1,1) node[xired] {} -- (0,-.4);
\draw (1,1) node[xigreen] {} -- (0,-.4);}

\DeclareSymbol{IXi^2}{-1}{\draw (-1,1) node[xi] {} -- (0,-.4);
\draw (1,1) node[xi] {} -- (0,-.4);}

\DeclareSymbol{IXiS^2}{-1}{\draw (-1,1) node[xis] {} -- (0,-.4);
\draw (1,1) node[xis] {} -- (0,-.4);}

\DeclareSymbol{rIXi^2}{-1}{\draw (-1,1) node[xired] {} -- (0,-.4);
\draw (1,1) node[xired] {} -- (0,-.4);}

\DeclareSymbol{IXi^2green}{-1}{\draw (-1,1) node[xigreen] {} -- (0,-.4);
\draw (1,1) node[xigreen] {} -- (0,-.4);}

\DeclareSymbol{IXi^2_typed}{-1}{\draw (-1,1) node[xiorange] {} -- (0,-.4);
\draw (1,1) node[xigreen] {} -- (0,-.4);}

\DeclareSymbol{IXi^2green*}{-1}{\draw (-1,1) node[xigreen1] {} -- (0,-.4);
\draw (1,1) node[xigreen1] {} -- (0,-.4);}

\DeclareSymbol{IXiI'Xi_notriangle}{-1}{\draw (-1,1) node[xi] {} -- (0,0);
\draw (1,1)[kernels2] node[xi] {} -- (0,0) {};}

\DeclareSymbol{IXiSI'XiS}{-1}{\draw (-1,1) node[xis] {} -- (0,0);
\draw (1,1)[kernels2] node[xis] {} -- (0,0) {};}

\DeclareSymbol{IXiI'Xi}{-1}{\draw (-1,1) node[xi] {} -- (0,0);
\draw (1,1)[kernels2] node[xi] {} -- (0,0) node[not] {};}

\DeclareSymbol{IXiI'Xi_typed1}{-1}{\draw (-1,1) node[xigreen] {} -- (0,0);
\draw (1,1)[kernels2,greennode] node[xired] {} -- (0,0) node[not] {};}

\DeclareSymbol{IXiI'Xi_typed2}{-1}{\draw (-1,1) node[xigreen] {} -- (0,0);
\draw (1,1)[kernels2,rednode] node[xigreen] {} -- (0,0) node[not] {};}

\DeclareSymbol{IXiI'Xi_typed2_notriangle}{-1}{\draw (-1,1) node[xigreen] {} -- (0,0);
\draw (1,1)[kernels2,rednode] node[xigreen] {} -- (0,0) {};}

\DeclareSymbol{IXi^3}{-1}{\draw (-1.3,1) node[xi] {} -- (0,0);
\draw (1.3,1) node[xi] {} -- (0,0);
\draw (0,1.5) node[xi] {} -- (0,0) node[not] {};}

\DeclareSymbol{IXiS^3}{-1}{\draw (-1.3,1) node[xis] {} -- (0,0);
\draw (1.3,1) node[xis] {} -- (0,0);
\draw (0,1.5) node[xis] {} -- (0,0) {};}

\DeclareSymbol{IXi^3_notriangle}{0}{\draw (-1.3,1) node[xi] {} -- (0,0);
\draw (1.3,1) node[xi] {} -- (0,0);
\draw (0,1.5) node[xi] {} -- (0,0) {};}

\DeclareSymbol{I[IXi^3]}{-1}{\draw (-1.3,1) node[xi] {} -- (0,0);
\draw (1.3,1) node[xi] {} -- (0,0);
\draw (0,1.5) node[xi] {} -- (0,0);
\draw (0,-1.3) -- (0,0);}

\DeclareSymbol{IXi^3_typed}{-1}{\draw (-1.3,1) node[xigreen] {} -- (0,0);
\draw (1.3,1) node[xired] {} -- (0,0);
\draw (0,1) node[xigreen] {} -- (0,0) node[not] {};}

\DeclareSymbol{IXi^3-typed112}{-1}{\draw (-1.3,1) node[xi] {} -- (0,0);
\draw (1.3,1) node[xigreen] {} -- (0,0);
\draw (0,1) node[xi] {} -- (0,0) node[not] {};}

\DeclareSymbol{I'Xi}{-1}{
\draw[kernels2] (0,1) node[xi] {} -- (0,-0.5) {};}

\DeclareSymbol{I'Xi_notriangle}{-1}{
\draw[kernels2] (1,1) node[xi] {} -- (0,0) {};}

\DeclareSymbol{I'Xigreen-red}{-1}{
\draw[kernels2,rednode] (0,1) node[xi,fill=greennode] {} -- (0,-.4);}

\DeclareSymbol{IXi}{-1}{
\draw (0,1) node[xi] {} -- (0,-.5);}

\DeclareSymbol{IXiS}{-1}{
\draw (0,1) node[xis] {} -- (0,-.5);}

\DeclareSymbol{I'XiS}{-1}{
\draw[kernels2] (0,1) node[xis] {} -- (0,-0.5);}

\DeclareSymbol{IXigreen}{-1}{
\draw (0,1) node[xigreen] {} -- (0,-.5);}

\DeclareSymbol{green}{-1}{
\draw[greennode,very thick,line cap=round,scale=1.3] (0.05,0.4) -- (0.2,0.55) -- (0,0) -- (0.6,0.6) -- (0.4,0.05) -- (0.55,0.2);}
\DeclareSymbol{red}{-1}{
\draw[rednode,very thick,line cap=round,scale=1.3] (0.05,0.4) -- (0.2,0.55) -- (0,0) -- (0.6,0.6) -- (0.4,0.05) -- (0.55,0.2);}
\DeclareSymbol{lblue}{-1}{
\draw[lbluenode,very thick,line cap=round,scale=1.3] (0.05,0.4) -- (0.2,0.55) -- (0,0) -- (0.6,0.6) -- (0.4,0.05) -- (0.55,0.2);}
\DeclareSymbol{dblue}{-1}{
\draw[dbluenode,very thick,line cap=round,scale=1.3] (0.05,0.4) -- (0.2,0.55) -- (0,0) -- (0.6,0.6) -- (0.4,0.05) -- (0.55,0.2);}
\DeclareSymbol{orange}{-1}{
\draw[orangenode,very thick,line cap=round,scale=1.3] (0.05,0.4) -- (0.2,0.55) -- (0,0) -- (0.6,0.6) -- (0.4,0.05) -- (0.55,0.2);}

\DeclareSymbol{I'XXi}{-1}{
\draw[kernels2] (1,1) node[xi] {} -- (0,0) node[not] {};
\draw (1,1) node[xix] {};}

\DeclareSymbol{I'XXi_notriangle}{-1}{
\draw[kernels2] (1,1) node[xi] {} -- (0,0) {};
\draw (1,1) node[xix] {};}

\DeclareSymbol{IXiI[IXi]}{-1}{
\draw (0,2) node[xi] {} -- (1,1) node {};
\draw (1,1) -- (0,0) node {};
\draw (-1,1) node[xi] {} -- (0,0);}

\DeclareSymbol{IXiI'[IXiI'Xi]}{-1}{\draw[kernels2] (2,2) node[xi] {} -- (1,1);
\draw (0,2) node[xi] {} -- (1,1) node[not] {};
\draw[kernels2] (1,1) -- (0,0) node[not] {};
\draw (-1,1) node[xi] {} -- (0,0);}

\DeclareSymbol{IXiI'[IXiI'Xi]_notriangle}{1}{\draw[kernels2] (2,2) node[xi] {} -- (1,1);
\draw (0,2) node[xi] {} -- (1,1) {};
\draw[kernels2] (1,1) -- (0,0) {};
\draw (-1,1) node[xi] {} -- (0,0);}

\DeclareSymbol{I[IXiI'[IXiI'Xi]]}{-1}{\draw[kernels2] (2,2) node[xi] {} -- (1,1);
\draw (0,2) node[xi] {} -- (1,1) {};
\draw[kernels2] (1,1) -- (0,0) {};
\draw (-1,1) node[xi] {} -- (0,0);
\draw (1,-1) -- (0,0);}

\DeclareSymbol{I[IXi^2]IXi-typed112}{-1}{\draw (-2,2) node[xi] {} -- (-1,1);
\draw (0,2) node[xi] {} -- (-1,1) node[not] {};
\draw (-1,1) -- (0,0) node[not] {};
\draw (1,1) node[xigreen] {} -- (0,0);}

\DeclareSymbol{I[IXi^2]IXi-typed211}{-1}{\draw (-2,2) node[xigreen] {} -- (-1,1);
\draw (0,2) node[xi] {} -- (-1,1) node[not] {};
\draw (-1,1) -- (0,0) node[not] {};
\draw (1,1) node[xi] {} -- (0,0);}

\DeclareSymbol{I[IXiI'Xi]I'Xi}{-1}{\draw (-2,2) node[xi] {} -- (-1,1);
\draw[kernels2] (0,2) node[xi] {} -- (-1,1) node[not] {};
\draw (-1,1) -- (0,0) node[not] {};
\draw[kernels2] (1,1) node[xi] {} -- (0,0);}

\DeclareSymbol{I[IXiI'Xi]I'Xi_notriangle}{1}{\draw (-2,2) node[xi] {} -- (-1,1);
\draw[kernels2] (0,2) node[xi] {} -- (-1,1) {};
\draw (-1,1) -- (0,0) {};
\draw[kernels2] (1,1) node[xi] {} -- (0,0);}

\DeclareSymbol{I[I[IXiI'Xi]I'Xi]}{-1}{\draw (-2,2) node[xi] {} -- (-1,1);
\draw[kernels2] (0,2) node[xi] {} -- (-1,1) {};
\draw (-1,1) -- (0,0) {};
\draw[kernels2] (1,1) node[xi] {} -- (0,0);
\draw (1,-1) -- (0,0);}

\DeclareSymbol{Psi2I[Psi3]_notriangle}{-1}{
\draw (2,2) node[xi] {} -- (1,1) {};
\draw (0,2) node[xi] {} -- (1,1) {};
\draw (1,2.5) node[xi] {} -- (1,1) {};
\draw (1,1) -- (1,0) {};
\draw (0,0.8) node[xi] {} -- (1,0);
\draw (2,0.8) node[xi] {} -- (1,0);
}

\DeclareSymbol{Psi2I[Psi2]_notriangle}{1}{
\draw (2,2) node[xi] {} -- (1,1) {};
\draw (0,2) node[xi] {} -- (1,1) {};
\draw (1,1) -- (1,0) {};
\draw (0,0.8) node[xi] {} -- (1,0);
\draw (2,0.8) node[xi] {} -- (1,0);
}

\DeclareSymbol{PsiYY_notriangle}{-1}{
\draw[kernels2] (1.8,2) node[xi] {} -- (1,1);
\draw (0.7, 2.3) node[xi] {} -- (1,1) {};
\draw (1,1) -- (0,-.2) {};
\draw[kernels2] (-1.8,2) node[xi] {} -- (-1,1);
\draw (-0.7,2.3) node[xi] {} -- (-1,1) {};
\draw (-1,1) -- (0,-.2) {};
\draw (0,1) node[xi] {} -- (0,-.2);}

\DeclareSymbol{YY_notriangle}{1}{
\draw[kernels2] (1.8,2) node[xi] {} -- (1,1);
\draw (0.7, 2.3) node[xi] {} -- (1,1) {};
\draw (1,1) -- (0,-.2) {};
\draw[kernels2] (-1.8,2) node[xi] {} -- (-1,1);
\draw (-0.7,2.3) node[xi] {} -- (-1,1) {};
\draw (-1,1) -- (0,-.2) {};}

\DeclareSymbol{PsiYI[Psi']_notriangle}{1}{
\draw[kernels2] (1,2.2) node[xi] {} -- (1,1);
\draw (1,1) -- (0,-.2) {};
\draw[kernels2] (-1.8,2) node[xi] {} -- (-1,1);
\draw (-0.7,2.3) node[xi] {} -- (-1,1) {};
\draw (-1,1) -- (0,-.2) {};
\draw (0,1) node[xi] {} -- (0,-.2);}

\DeclareSymbol{PsiPsiY_notriangle}{-1}{
\draw[kernels2] (1.8,2.3) node[xi] {} -- (1,1);
\draw (0.7, 2.3) node[xi] {} -- (1,1) {};
\draw (1,1) -- (0,-.2) {};
\draw (-1.2,1) node[xi] {} -- (0,-.2);
\draw (0,1) node[xi] {} -- (0,-.2);}

\DeclareSymbol{PsiI'[Psi3]_notriangle}{-1}{
\draw (2,2.3) node[xi] {} -- (1,1);
\draw (1, 2.3) node[xi] {} -- (1,1) {};
\draw (0, 2.3) node[xi] {} -- (1,1) {};
\draw (1,1) -- (1,-.2) {};
\draw[kernels2] (2,1) node[xi] {} -- (1,-.2);}

\DeclareSymbol{Psi'I[Psi3]_notriangle}{-1}{
\draw (2,2.3) node[xi] {} -- (1,1);
\draw (1, 2.3) node[xi] {} -- (1,1) {};
\draw (0, 2.3) node[xi] {} -- (1,1) {};
\draw[kernels2] (1,1) -- (1,-.2) {};
\draw (2,1) node[xi] {} -- (1,-.2);}

\DeclareSymbol{Y'Y_notriangle}{-1}{
\draw[kernels2] (1.8,2) node[xi] {} -- (1,1);
\draw (0.7, 2.3) node[xi] {} -- (1,1) {};
\draw[kernels2] (1,1) -- (0,-.2) {};
\draw[kernels2] (-1.8,2) node[xi] {} -- (-1,1);
\draw (-0.7,2.3) node[xi] {} -- (-1,1) {};
\draw (-1,1) -- (0,-.2) {};}

\DeclareSymbol{PsiXPsiY_notriangle}{-1}{
\draw[kernels2] (1.8,2.3) node[xi] {} -- (1,1);
\draw (0.7, 2.3) node[xi] {} -- (1,1) {};
\draw (1,1) -- (0,-.2) {};
\draw (-1.2,1) node[xix] {} -- (0,-.2);
\draw (0,1) node[xi] {} -- (0,-.2);}

\DeclareSymbol{Psi'I[YPsi']_notriangle}{-5}{
\draw (-2,2) node[xi] {} -- (-1,1);
\draw[kernels2] (0,2) node[xi] {} -- (-1,1) {};
\draw (-1,1) -- (0,0) {};
\draw[kernels2] (1,1) node[xi] {} -- (0,0);
\draw (0,0)  -- (1,-1);
\draw[kernels2] (2,0) node[xi] {} -- (1,-1) {};
}

\DeclareSymbol{Psi'I[Psi'I[YPsi']]_notriangle}{-5}{
\draw (-2,2) node[xi] {} -- (-1,1);
\draw[kernels2] (0,2) node[xi] {} -- (-1,1) {};
\draw (-1,1) -- (0,0) {};
\draw[kernels2] (1,1) node[xi] {} -- (0,0);
\draw (0,0)  -- (1,-1);
\draw[kernels2] (2,0) node[xi] {} -- (1,-1) {};
\draw (1,-1) -- (2,-2) {};
\draw[kernels2] (3,-1) node[xi] {} -- (2,-2);
}
	
	\DeclareSymbol{Psi'I[Psi'I[LPsi']]_notriangle}{-5}{
	\draw[kernels2] (0,2) node[xi] {} -- (-1,1) {};
	\draw (-1,1) -- (0,0) {};
	\draw[kernels2] (1,1) node[xi] {} -- (0,0);
	\draw (0,0)  -- (1,-1);
	\draw[kernels2] (2,0) node[xi] {} -- (1,-1) {};
	\draw (1,-1) -- (2,-2) {};
	\draw[kernels2] (3,-1) node[xi] {} -- (2,-2);
	}

\DeclareSymbol{Psi'I[Psi'I[Y'Psi]]_notriangle}{-5}{
\draw (-2,2) node[xi] {} -- (-1,1);
\draw[kernels2] (0,2) node[xi] {} -- (-1,1) {};
\draw[kernels2] (-1,1) -- (0,0) {};
\draw (1,1) node[xi] {} -- (0,0);
\draw (0,0)  -- (1,-1);
\draw[kernels2] (2,0) node[xi] {} -- (1,-1) {};
\draw (1,-1) -- (2,-2) {};
\draw[kernels2] (3,-1) node[xi] {} -- (2,-2);
}

	\DeclareSymbol{Psi'I[Psi'I[L'Psi]]_notriangle}{-5}{
	\draw[kernels2] (0,2) node[xi] {} -- (-1,1) {};
	\draw[kernels2] (-1,1) -- (0,0) {};
	\draw (1,1) node[xi] {} -- (0,0);
	\draw (0,0)  -- (1,-1);
	\draw[kernels2] (2,0) node[xi] {} -- (1,-1) {};
	\draw (1,-1) -- (2,-2) {};
	\draw[kernels2] (3,-1) node[xi] {} -- (2,-2);
	}
	\DeclareSymbol{Psi'I[Psi'I[Y']]_notriangle}{-5}{
	\draw (-2,2) node[xi] {} -- (-1,1);
	\draw[kernels2] (0,2) node[xi] {} -- (-1,1) {};
	\draw[kernels2] (-1,1) -- (0,0) {};
	\draw (0,0)  -- (1,-1);
	\draw[kernels2] (2,0) node[xi] {} -- (1,-1) {};
	\draw (1,-1) -- (2,-2) {};
	\draw[kernels2] (3,-1) node[xi] {} -- (2,-2);
	}

\DeclareSymbol{Psi'I[PsiI'[YPsi']]_notriangle}{-5}{
\draw (-2,2) node[xi] {} -- (-1,1);
\draw[kernels2] (0,2) node[xi] {} -- (-1,1) {};
\draw (-1,1) -- (0,0) {};
\draw[kernels2] (1,1) node[xi] {} -- (0,0);
\draw[kernels2] (0,0)  -- (1,-1);
\draw (2,0) node[xi] {} -- (1,-1) {};
\draw (1,-1) -- (2,-2) {};
\draw[kernels2] (3,-1) node[xi] {} -- (2,-2);
}

	\DeclareSymbol{Psi'I[PsiI'[LPsi']]_notriangle}{-5}{
	\draw[kernels2] (0,2) node[xi] {} -- (-1,1) {};
	\draw (-1,1) -- (0,0) {};
	\draw[kernels2] (1,1) node[xi] {} -- (0,0);
	\draw[kernels2] (0,0)  -- (1,-1);
	\draw (2,0) node[xi] {} -- (1,-1) {};
	\draw (1,-1) -- (2,-2) {};
	\draw[kernels2] (3,-1) node[xi] {} -- (2,-2);
	}
	\DeclareSymbol{Psi'I[I'[YPsi']]_notriangle}{-5}{
	\draw (-2,2) node[xi] {} -- (-1,1);
	\draw[kernels2] (0,2) node[xi] {} -- (-1,1) {};
	\draw (-1,1) -- (0,0) {};
	\draw[kernels2] (1,1) node[xi] {} -- (0,0);
	\draw[kernels2] (0,0)  -- (1,-1);
	\draw (1,-1) -- (2,-2) {};
	\draw[kernels2] (3,-1) node[xi] {} -- (2,-2);
	}

\DeclareSymbol{PsiI'[Psi'I[YPsi']]_notriangle}{-5}{
\draw (-2,2) node[xi] {} -- (-1,1);
\draw[kernels2] (0,2) node[xi] {} -- (-1,1) {};
\draw (-1,1) -- (0,0) {};
\draw[kernels2] (1,1) node[xi] {} -- (0,0);
\draw (0,0)  -- (1,-1);
\draw[kernels2] (2,0) node[xi] {} -- (1,-1) {};
\draw[kernels2] (1,-1) -- (2,-2) {};
\draw (3,-1) node[xi] {} -- (2,-2);
}

	\DeclareSymbol{PsiI'[Psi'I[LPsi']]_notriangle}{-5}{
	\draw[kernels2] (0,2) node[xi] {} -- (-1,1) {};
	\draw (-1,1) -- (0,0) {};
	\draw[kernels2] (1,1) node[xi] {} -- (0,0);
	\draw (0,0)  -- (1,-1);
	\draw[kernels2] (2,0) node[xi] {} -- (1,-1) {};
	\draw[kernels2] (1,-1) -- (2,-2) {};
	\draw (3,-1) node[xi] {} -- (2,-2);
	}

\DeclareSymbol{Psi'I[PsiI'[Y'Psi]]_notriangle}{-5}{
\draw (-2,2) node[xi] {} -- (-1,1);
\draw[kernels2] (0,2) node[xi] {} -- (-1,1) {};
\draw[kernels2] (-1,1) -- (0,0) {};
\draw (1,1) node[xi] {} -- (0,0);
\draw[kernels2] (0,0)  -- (1,-1);
\draw (2,0) node[xi] {} -- (1,-1) {};
\draw (1,-1) -- (2,-2) {};
\draw[kernels2] (3,-1) node[xi] {} -- (2,-2);
}
	
	\DeclareSymbol{Psi'I[PsiI'[L'Psi]]_notriangle}{-5}{
	\draw[kernels2] (0,2) node[xi] {} -- (-1,1) {};
	\draw[kernels2] (-1,1) -- (0,0) {};
	\draw (1,1) node[xi] {} -- (0,0);
	\draw[kernels2] (0,0)  -- (1,-1);
	\draw (2,0) node[xi] {} -- (1,-1) {};
	\draw (1,-1) -- (2,-2) {};
	\draw[kernels2] (3,-1) node[xi] {} -- (2,-2);
	}
	\DeclareSymbol{Psi'I[PsiI'[Y']]_notriangle}{-5}{
	\draw (-2,2) node[xi] {} -- (-1,1);
	\draw[kernels2] (0,2) node[xi] {} -- (-1,1) {};
	\draw[kernels2] (-1,1) -- (0,0) {};
	\draw[kernels2] (0,0)  -- (1,-1);
	\draw (2,0) node[xi] {} -- (1,-1) {};
	\draw (1,-1) -- (2,-2) {};
	\draw[kernels2] (3,-1) node[xi] {} -- (2,-2);
	}
	\DeclareSymbol{Psi'I[I'[Y'Psi]]_notriangle}{-5}{
	\draw (-2,2) node[xi] {} -- (-1,1);
	\draw[kernels2] (0,2) node[xi] {} -- (-1,1) {};
	\draw[kernels2] (-1,1) -- (0,0) {};
	\draw (1,1) node[xi] {} -- (0,0);
	\draw[kernels2] (0,0)  -- (1,-1);
	\draw (1,-1) -- (2,-2) {};
	\draw[kernels2] (3,-1) node[xi] {} -- (2,-2);
	}

\DeclareSymbol{PsiI'[Psi'I[Y'Psi]]_notriangle}{-5}{
\draw (-2,2) node[xi] {} -- (-1,1);
\draw[kernels2] (0,2) node[xi] {} -- (-1,1) {};
\draw[kernels2] (-1,1) -- (0,0) {};
\draw (1,1) node[xi] {} -- (0,0);
\draw (0,0)  -- (1,-1);
\draw[kernels2] (2,0) node[xi] {} -- (1,-1) {};
\draw[kernels2] (1,-1) -- (2,-2) {};
\draw (3,-1) node[xi] {} -- (2,-2);
}

	\DeclareSymbol{PsiI'[Psi'I[L'Psi]]_notriangle}{-5}{
	\draw[kernels2] (0,2) node[xi] {} -- (-1,1) {};
	\draw[kernels2] (-1,1) -- (0,0) {};
	\draw (1,1) node[xi] {} -- (0,0);
	\draw (0,0)  -- (1,-1);
	\draw[kernels2] (2,0) node[xi] {} -- (1,-1) {};
	\draw[kernels2] (1,-1) -- (2,-2) {};
	\draw (3,-1) node[xi] {} -- (2,-2);
	}
	\DeclareSymbol{PsiI'[Psi'I[Y']]_notriangle}{-5}{
	\draw (-2,2) node[xi] {} -- (-1,1);
	\draw[kernels2] (0,2) node[xi] {} -- (-1,1) {};
	\draw[kernels2] (-1,1) -- (0,0) {};
	\draw (0,0)  -- (1,-1);
	\draw[kernels2] (2,0) node[xi] {} -- (1,-1) {};
	\draw[kernels2] (1,-1) -- (2,-2) {};
	\draw (3,-1) node[xi] {} -- (2,-2);
	}

\DeclareSymbol{PsiI'[PsiI'[YPsi']]_notriangle}{-5}{
\draw (-2,2) node[xi] {} -- (-1,1);
\draw[kernels2] (0,2) node[xi] {} -- (-1,1) {};
\draw (-1,1) -- (0,0) {};
\draw[kernels2] (1,1) node[xi] {} -- (0,0);
\draw[kernels2] (0,0)  -- (1,-1);
\draw (2,0) node[xi] {} -- (1,-1) {};
\draw[kernels2] (1,-1) -- (2,-2) {};
\draw (3,-1) node[xi] {} -- (2,-2);
}
	
	\DeclareSymbol{PsiI'[PsiI'[LPsi']]_notriangle}{-5}{
	\draw[kernels2] (0,2) node[xi] {} -- (-1,1) {};
	\draw (-1,1) -- (0,0) {};
	\draw[kernels2] (1,1) node[xi] {} -- (0,0);
	\draw[kernels2] (0,0)  -- (1,-1);
	\draw (2,0) node[xi] {} -- (1,-1) {};
	\draw[kernels2] (1,-1) -- (2,-2) {};
	\draw (3,-1) node[xi] {} -- (2,-2);
	}
	\DeclareSymbol{PsiI'[I'[YPsi']]_notriangle}{-5}{
	\draw (-2,2) node[xi] {} -- (-1,1);
	\draw[kernels2] (0,2) node[xi] {} -- (-1,1) {};
	\draw (-1,1) -- (0,0) {};
	\draw[kernels2] (1,1) node[xi] {} -- (0,0);
	\draw[kernels2] (0,0)  -- (1,-1);
	\draw[kernels2] (1,-1) -- (2,-2) {};
	\draw (3,-1) node[xi] {} -- (2,-2);
	}

\DeclareSymbol{PsiI'[PsiI'[Y'Psi]]_notriangle}{-5}{
\draw (-2,2) node[xi] {} -- (-1,1);
\draw[kernels2] (0,2) node[xi] {} -- (-1,1) {};
\draw[kernels2] (-1,1) -- (0,0) {};
\draw (1,1) node[xi] {} -- (0,0);
\draw[kernels2] (0,0)  -- (1,-1);
\draw (2,0) node[xi] {} -- (1,-1) {};
\draw[kernels2] (1,-1) -- (2,-2) {};
\draw (3,-1) node[xi] {} -- (2,-2);
}

	\DeclareSymbol{PsiI'[PsiI'[L'Psi]]_notriangle}{-5}{
	\draw[kernels2] (0,2) node[xi] {} -- (-1,1) {};
	\draw[kernels2] (-1,1) -- (0,0) {};
	\draw (1,1) node[xi] {} -- (0,0);
	\draw[kernels2] (0,0)  -- (1,-1);
	\draw (2,0) node[xi] {} -- (1,-1) {};
	\draw[kernels2] (1,-1) -- (2,-2) {};
	\draw (3,-1) node[xi] {} -- (2,-2);
	}
	\DeclareSymbol{PsiI'[PsiI'[Y']]_notriangle}{-5}{
	\draw (-2,2) node[xi] {} -- (-1,1);
	\draw[kernels2] (0,2) node[xi] {} -- (-1,1) {};
	\draw[kernels2] (-1,1) -- (0,0) {};
	\draw[kernels2] (0,0)  -- (1,-1);
	\draw (2,0) node[xi] {} -- (1,-1) {};
	\draw[kernels2] (1,-1) -- (2,-2) {};
	\draw (3,-1) node[xi] {} -- (2,-2);
	}
	\DeclareSymbol{PsiI'[I'[Y'Psi]]_notriangle}{-5}{
	\draw (-2,2) node[xi] {} -- (-1,1);
	\draw[kernels2] (0,2) node[xi] {} -- (-1,1) {};
	\draw[kernels2] (-1,1) -- (0,0) {};
	\draw (1,1) node[xi] {} -- (0,0);
	\draw[kernels2] (0,0)  -- (1,-1);
	\draw[kernels2] (1,-1) -- (2,-2) {};
	\draw (3,-1) node[xi] {} -- (2,-2);
	}

\DeclareSymbol{I'[IXiI'Xi]}{-1}{\draw[kernels2] (2,2) node[xi] {} -- (1,1);
\draw (0,2) node[xi] {} -- (1,1) node[not] {};
\draw[kernels2] (1,1) -- (0,0) node[not] {};}

\DeclareSymbol{I'[IXiI'Xi]_notriangle}{1}{\draw[kernels2] (2,2) node[xi] {} -- (1,1);
\draw (0,2) node[xi] {} -- (1,1) {};
\draw[kernels2] (1,1) -- (0,0) {};}

\DeclareSymbol{I[IXiI'Xi]_notriangle}{1}{\draw[kernels2] (2,2) node[xi] {} -- (1,1);
\draw (0,2) node[xi] {} -- (1,1) {};
\draw (1,1) -- (1,-0.3) {};}

\DeclareSymbol{I[IXiSI'XiS]_notriangle}{1}{\draw[kernels2] (2,2) node[xis] {} -- (1,1);
\draw (0,2) node[xis] {} -- (1,1) {};
\draw (1,1) -- (1,-0.3) {};}

\DeclareSymbol{I[I[Xi]]Xi_notriangle}{1}{
\draw (0,2) node[xi] {} -- (-1,1) {};
\draw[snake=snake , segment length=1pt, segment amplitude=1pt] (-1,1) -- (0,0) node[xi] {};}

\DeclareSymbol{I[I[Xi]]Xi_typed}{1}{
\draw (0,2) node[xigreen] {} -- (-1,1) {};
\draw[snake=snake , segment length=1pt, segment amplitude=1pt] (-1,1) -- (0,0) node[xigreen] {};}

\DeclareSymbol{I[I[Xi]]Xi_typed*}{1}{
\draw (0,2) node[xigreen1] {} -- (-1,1) {};
\draw[snake=snake , segment length=1pt, segment amplitude=1pt] (-1,1) -- (0,0) node[xigreen1] {};}

\DeclareSymbol{I[I'Xi]I'Xi}{1}{\draw[kernels2] (0,2) node[xi] {} -- (-1,1);
\draw (-1,1) node[fill=rednode,not] {} -- (0,0);
\draw[kernels2] (1,1) node[xi] {} -- (0,0);}

\DeclareSymbol{I[I'Xi]I'Xi_notriangle}{1}{\draw[kernels2] (0,2) node[xi] {} -- (-1,1);
\draw (-1,1) {} -- (0,0);
\draw[kernels2] (1,1) node[xi] {} -- (0,0);}

\DeclareSymbol{I[I'Xi]I'Xi_typed}{1}{\draw[kernels2,rednode] (0,2) node[xigreen] {} -- (-1,1);
\draw (-1,1) node[notorange] {} -- (0,0);
\draw[kernels2,rednode] (1,1) node[xigreen] {} -- (0,0);}

\DeclareSymbol{I[I'Xi]I'Xi_typed*}{1}{\draw[kernels2,rednode] (0,2) node[xigreen1] {} -- (-1,1);
\draw (-1,1) node[notorange] {} -- (0,0);
\draw[kernels2,rednode] (1,1) node[xigreen1] {} -- (0,0);}

\DeclareSymbol{I[I'Xi]I'Xi_typed-h}{1}{\draw[kernels2,greennode] (0,2) node[xigreen] {} -- (-1,1);
\draw[snake=snake , segment length=3pt, segment amplitude=1pt] (-1,1) node[notgreen] {} -- (0,0);
\draw[kernels2,greennode] (1,1) node[xigreen] {} -- (0,0);}

\DeclareSymbol{I[I'Xi]I'Xi_typed-h*}{1}{\draw[kernels2,greennode] (0,2) node[xigreen1] {} -- (-1,1);
\draw[snake=snake , segment length=3pt, segment amplitude=1pt] (-1,1) node[notgreen] {} -- (0,0);
\draw[kernels2,greennode] (1,1) node[xigreen1] {} -- (0,0);}

\DeclareSymbol{I[I'Xi]I'Xib}{1}{\draw[kernels2] (0,2) node[xi] {} -- (1,1);
\draw (1,1) node[not] {} -- (0,0);
\draw[kernels2] (-1,1) node[xi] {} -- (0,0);}

\DeclareSymbol{I[I'Xi]I'Xib_notriangle}{1}{\draw[kernels2] (0,2) node[xi] {} -- (1,1);
\draw (1,1) {} -- (0,0);
\draw[kernels2] (-1,1) node[xi] {} -- (0,0);}

\DeclareSymbol{IXiI'[I'Xi]}{1}{\draw[kernels2] (0,2) node[xi] {} -- (1,1);
\draw[kernels2] (1,1) node[not] {} -- (0,0);
\draw (-1,1) node[xi] {} -- (0,0);}

\DeclareSymbol{IXiI'[I'Xi]_notriangle}{1}{\draw[kernels2] (0,2) node[xi] {} -- (1,1);
\draw[kernels2] (1,1)  {} -- (0,0);
\draw (-1,1) node[xi] {} -- (0,0);}

\DeclareSymbol{IXiI'[I'Xi]_typed}{1}{\draw[kernels2,rednode] (0,2) node[xigreen] {} -- (1,1);
\draw[kernels2,rednode] (1,1) node[notorange] {} -- (0,0);
\draw (-1,1) node[xigreen] {} -- (0,0);}

\DeclareSymbol{IXiI'[I'Xi]_typed-h}{1}{\draw[kernels2,greennode] (0,2) node[xigreen] {} -- (1,1);
\draw[kernels2,greennode,snake=snake , segment length=3pt, segment amplitude=1pt] (1,1) node[notgreen] {} -- (0,0);
\draw (-1,1) node[xigreen] {} -- (0,0);}

\DeclareSymbol{IXiI'[I'Xi]_typed*}{1}{\draw[kernels2,rednode] (0,2) node[xigreen1] {} -- (1,1);
\draw[kernels2,rednode] (1,1) node[notorange] {} -- (0,0);
\draw (-1,1) node[xigreen1] {} -- (0,0);}

\DeclareSymbol{IXiI'[I'Xi]_typed-h*}{1}{\draw[kernels2,greennode] (0,2) node[xigreen1] {} -- (1,1);
\draw[kernels2,greennode,snake=snake , segment length=3pt, segment amplitude=1pt] (1,1) node[notgreen] {} -- (0,0);
\draw (-1,1) node[xigreen1] {} -- (0,0);}

\DeclareSymbol{IXiI'[IXi]_notriangle}{1}{\draw (0,2) node[xi] {} -- (1,1);
\draw[kernels2] (1,1)  {} -- (0,0);
\draw (-1,1) node[xi] {} -- (0,0);}

\DeclareSymbol{I[I'Xi]}{-1}{\draw[kernels2] (0,2) node[xi] {} -- (1,1);
\draw (1,1) node[not] {} -- (0,0) node[not] {};}

\DeclareSymbol{I[I'Xi]_typed}{1}{\draw[kernels2,rednode] (0,2) node[xigreen] {} -- (-1,1);
\draw (-1,1) node[notorange] {} -- (0,0);}

\DeclareSymbol{I'[I'Xi]}{-1}{\draw[kernels2] (0,2) node[xi] {} -- (1,1);
\draw[kernels2] (1,1) node[not] {} -- (0,0) node[not] {};}

\DeclareSymbol{I'[I'Xi]_notriangle}{-1}{\draw[kernels2] (0,2) node[xi] {} -- (1,1);
\draw[kernels2] (1,1) {} -- (0,0) {};}

\DeclareSymbol{I'[IXi]_notriangle}{1}{\draw (0,2) node[xi] {} -- (1,1);
\draw[kernels2] (1,1) {} -- (0,0) {};}

\DeclareSymbol{I'[I'Xi]_typed}{1}{\draw[kernels2,rednode] (0,2) node[xigreen] {} -- (1,1);
\draw[kernels2,rednode] (1,1) node[notorange] {} -- (0,0)  {};}

\DeclareSymbol{IXiI'XXi}{-1}{\draw (-1,1) node[xi] {} -- (0,0);
\draw[kernels2] (1,1) node[xi] {} -- (0,0) node[not] {};
\draw (1,1) node[xix] {};}

\DeclareSymbol{IXiI'XXi_notriangle}{-1}{\draw (-1,1) node[xi] {} -- (0,0);
\draw[kernels2] (1,1) node[xi] {} -- (0,0) {};
\draw (1,1) node[xix] {};}

\DeclareSymbol{IXiI'XXi_typed}{-1}{\draw (-1,1) node[xigreen] {} -- (0,-0.3);
\draw[kernels2,rednode] (1,1)  -- (0,-0.3) {};
\draw (1,1) node[xix-green-red] {};
}

\DeclareSymbol{IXiI'XXi_typed*}{-1}{\draw (-1,1) node[xigreen1] {} -- (0,-0.3);
\draw[kernels2,rednode] (1,1)  -- (0,-0.3) {};
\draw (1,1) node[xix-green-red1] {};
}

\DeclareSymbol{IXXiI'Xi}{-1}{\draw (-1,1) node[xix] {} -- (0,0);
\draw[kernels2] (1,1) node[xi] {} -- (0,0) node[not] {};}

\DeclareSymbol{IXXiI'Xi_notriangle}{-1}{\draw (-1,1) node[xix] {} -- (0,0);
\draw[kernels2] (1,1) node[xi] {} -- (0,0) {};}

\DeclareSymbol{IXXiI'Xi_typed}{-1}{\draw (-1,1) node[xix-green-red] {} -- (0,-0.3);
\draw[kernels2,rednode] (1,1) node[xigreen] {} -- (0,-0.3) {};}

\DeclareSymbol{IXXiI'Xi_typed*}{-1}{\draw (-1,1) node[xix-green-red1] {} -- (0,-0.3);
\draw[kernels2,rednode] (1,1) node[xigreen1] {} -- (0,-0.3) {};}

\DeclareSymbol{XiX}{-2.8}{\node[xibx] {};}
\DeclareSymbol{XXi}{-2.8}{\node[xibx] {};}
\DeclareSymbol{tauX}{-2.8}{ \draw[kernels2] (0,0) node[xibx] {};}
\DeclareSymbol{Xi}{-2.8}{\node[xib] {};}
\DeclareSymbol{XiS}{-2.8}{\node[xisb] {};}
\DeclareSymbol{Xigreen}{-2.8}{\node[xigreen] {};}
\DeclareSymbol{Xired}{-2.8}{\node[xired] {};}
\DeclareSymbol{not}{-2.8}{\node[not,minimum size=1.6mm] {};}
\DeclareSymbol{XiY}{-2.8}{\node[xie] {};}
\DeclareSymbol{XiZ}{-2.8}{\node[xid] {};}
\DeclareSymbol{IXiX}{0}{\draw (0,-0.25) node[not] {} -- (0,1.5) node[xix] {};}

\let\f\frac

\begin{document}

\title{Stochastic quantisation of Yang--Mills--Higgs in 3D}

\author{Ajay Chandra$^1$, Ilya Chevyrev$^2$, Martin Hairer$^{1,3}$, and Hao Shen$^4$}

\institute{Imperial College London, \email{\{achandra,m.hairer\}@imperial.ac.uk} \and University of Edinburgh, \email{ichevyrev@gmail.com} 
\and EPFL, \email{martin.hairer@epfl.ch}
\and  University of 
Wisconsin--Madison, \email{pkushenhao@gmail.com}}

\maketitle

\begin{abstract}
We define a state space and a Markov process
associated to the stochastic quantisation equation of Yang--Mills--Higgs (YMH) theories.
The state space $\state$ is a nonlinear metric space of distributions, elements of which can be used as initial conditions for the
(deterministic and stochastic) YMH flow
with good continuity properties.
Using gauge covariance of the deterministic YMH flow, we extend gauge equivalence $\sim$ to $\state$ and thus define a quotient space of ``gauge orbits'' $\mfO$.
We use the theory of regularity structures to prove local in time solutions to the renormalised stochastic YMH flow.
Moreover, by leveraging symmetry arguments in the small noise limit, we show that there is a 
unique choice of
renormalisation counterterms such that these solutions are gauge covariant in law.
This allows us to define a canonical Markov process on $\mfO$ (up to a potential finite time blow-up) associated to the stochastic YMH flow.  \\[.4em]
\end{abstract}


\setcounter{tocdepth}{2}
\tableofcontents

\section{Introduction}\label{sec:intro}
The purpose of this paper is to construct and study the Langevin dynamic associated to the Euclidean Yang--Mills--Higgs (YMH) measure on the torus $\T^{d}$, for $d=3$.
Postponing precise definitions to Section~\ref{sec:notation},
the YMH measure should be the probability measure formally given by 
\begin{equ}\label{eq:YMH_measure}
	\mrd\mu_\YMH(A,\Phi)  \propto \exp\big[-
	S_{\YMH}(A,\Phi)\big]\, \mrd A\, \mrd \Phi\;,
\end{equ}
where $\mrd A$ (resp.\ $\mrd \Phi$) is a formal Lebesgue measure on the space of principal $G$-connections on a principal bundle $\CP\to\T^d$ (resp.\ associated vector bundle $\CV$), for $G$ a compact Lie group with Lie algebra $\mfg$.
The YMH action $S_{\YMH}(A,\Phi)$ is given by 
\begin{equ}\label{eq:YM_energy}
	S_{\YMH}(A,\Phi) \eqdef \int_{\T^{d}}
	\Big(|F_A(x)|^2 +  |\mrd_A \Phi (x)|^2 + m^2|\Phi (x)|^2 + \frac12 |\Phi(x)|^4\Big)\, \mrd x\;,
\end{equ}
where $F_A$ is the curvature $2$-form of $A$ and $\mrd_{A}$ denotes the covariant 
derivative defined by the connection $A$.
In particular, this includes the pure Yang--Mills (YM) action,
in which the Higgs field $\Phi$ is absent (i.e., the associated vector bundle has dimension $0$).

The YMH action is an important component of the Standard Model Lagrangian, and the mathematical study of the YMH measure on various manifolds has a long history.
Much work has been done to give a meaning to~\eqref{eq:YMH_measure} for the \textit{pure} YM theory in $d=2$, including $\R^2$ and compact surfaces.
These works are primarily based on an elegant integrability property of 2D YM; a sample list of contributions is~\cite{Migdal75,GKS89, Driver89, Fine91, Sengupta97, Levy03}.
In the Abelian case $G=U(1)$ on $\R^d$ the measure \eqref{eq:YMH_measure} is essentially a Gaussian measure (e.g. \cite{MR0838562}), and \cite{Gross83} and \cite{Driver87} proved  convergence of the $U(1)$ lattice Yang--Mills theory to the continuum theory in $d=3$ and $d=4$ respectively.
Still in the Abelian case, one can also make sense of the full YMH theory in $d=2$~\cite{BFS79,BFS80,BFS81}
and $d=3$~\cite{King86I,King86II}.
A form of ultraviolet stability in finite volume (i.e.\ on $\T^d$) was also shown for the pure YM theory in $d=4$ using a continuum regularisation in~\cite{MRS93}  and in $d=3,4$ using renormalisation group methods on the lattice in a series of works by Balaban culminating in~\cite{Balaban85IV,Balaban87,Balaban89}.
The above list of works is far from exhaustive
and we refer to the surveys~\cite{JW06, LevySengupta17, Chatterjee18}
for further references and history.
We mention that the construction of the pure YM measure in $d=3$ and of the non-Abelian YMH measure in $d=2$ (and $d=3$), even in finite volume, is open.

In this paper we will always assume that the bundles $\CP$ and $\CV$ are trivial.
Upon fixing a global section of $\CP$, we can therefore identify
a connection $A$ with a $\mfg$-valued $1$-form and a section of $\CV$ with a function $\Phi\colon\T^d\to\higgsvec$, where $\higgsvec$ is a real finite-dimensional inner product space carrying an orthogonal representation $\brho$ of $G$.
We will often drop $\brho$ from our notation and simply write $g\phi$ for $\brho(g)\phi$, and similarly $h\phi$ for $h\in\mfg$ will denote the derivative representation of $\mfg$ on $\higgsvec$.

Besides the non-existence of the Lebesgue measure on infinite-dimensional spaces
and the usual divergencies which arise from the non-quadratic terms in an action, a major problem when 
trying to give a meaning to \eqref{eq:YMH_measure} is the fact that it is (formally) invariant under the action 
of an infinite-dimensional group of transformation.
Indeed, 
assuming for now that all the objects are smooth,
a \textit{gauge transformation} $g $ is an element 
$g \in \mfG^\infty \eqdef \mcC^\infty(\T^d,G)$ which acts on 
$(A,\Phi)$ by
\begin{equ}[e:gauge-transformation]
	g\act A =   A^g \eqdef \Ad_g(A)  - (\mrd g)g^{-1} \;,
	\qquad
	\mbox{and}
	\qquad
	g\act \Phi =  \Phi^g \eqdef  g\Phi\;.
\end{equ}
Geometrically, gauge transformations are automorphisms of the principal bundle $\CP$, so
$(A^g,\Phi^g)$ can be interpreted as representing $(A,\Phi)$ in a new coordinate system i.e., using a different global section of $\CP$.
An important feature of gauge theory is that all coordinate systems should give rise to the same physical quantities.
In particular, one can verify that 
\begin{equ}
	S_{\YMH}(A^g, \Phi^g) = S_{\YMH}(A,\Phi)\;,
\end{equ}
which suggests that $\mu_\YMH$ should be invariant under the action of 
any gauge transformations $g$. Since the group of all gauge transformations is infinite-dimensional,
such a measure cannot exist as a \textit{bona fide} $\sigma$-additive probability measure.

A potential way to make rigorous sense of~\eqref{eq:YMH_measure} is through stochastic quantisation, which was introduced in the physics literature for gauge theories by Parisi--Wu~\cite{ParisiWu}.
In this approach, one aims to construct $\mu_{\YMH}$ as the invariant measure of the Langevin dynamic associated to the action $S_{\YMH}$.
Formally, this is given by 
\begin{equs}[e:SYMH_no_deturck]
	\partial_t A =& -\mrd_A^* F_A - \mathbf{B}(\mrd_{A}\Phi\otimes\Phi) + \xi\;,\\
	\partial_t\Phi  =& - \mrd_{ A}^* \mrd_A \Phi - m^{2} \Phi - |\Phi|^2 \Phi + \zeta\;,
\end{equs}
where $\mrd_{A}^{\ast}$ denotes the adjoint of $\mrd_{A}$, and   $\mathbf{B}\colon \higgsvec \otimes \higgsvec \rightarrow \mathfrak{g}$  is the unique  $\R$-linear form such that
\begin{equ}[e:def-B]
	\scal{\mathbf{B}(u\otimes v),h}_{\mfg} =  
	\scal{u,hv}_\higgsvec
\end{equ}
for all $u,v \in \higgsvec$ and $h \in \mfg$.
The noises $\xi$ and $\zeta$ are independent space-time white noises
with respect to our metrics, i.e.\ for $\xi = \xi_i \mrd x_i$,   
\begin{equs}[e:explicit_covar_noise]
	\E[\xi_{i}(t,x) \otimes \xi_{j}(s,y)] &= \delta_{ij} \delta(t-s)\delta(x-y)\, \Cas\;,
	\\
	\E[\zeta(t,x) \otimes \zeta(s,y)] &= \delta(t-s)\delta(x-y)\,  \Cov \;.
\end{equs}
Here $\Cas\in \mfg\otimes_s \mfg$ is the quadratic Casimir element 
and $ \Cov \in \higgsvec\otimes_s \higgsvec$ is identity map if we identify 
$\higgsvec \simeq \higgsvec^*$ using the metric on $\higgsvec$.
We normalise the Casimir similarly and note that these covariances satisfy the invariant properties
$(\Ad_g \otimes \Ad_g) \Cas =\Cas$  and 
$(g \otimes g) \Cov =\Cov$ for every $g\in G$.
The dynamic \eqref{e:SYMH_no_deturck} is  formally
invariant {\it in law} under any {\it time-independent} gauge transformation $g$ by
\eqref{e:gauge-transformation}, namely, if $(A,\Phi)$ a solution to \eqref{e:SYMH_no_deturck} then, 
for any fixed time-independent gauge transformation $g$ we have that $(A^{g},\Phi^{g})$ solves \eqref{e:SYMH_no_deturck} with $(\xi,\zeta)$ replaced by the rotated noise $(\Ad_{g} \xi,g \zeta)$ which is equal in law. 
In particular, the bilinear form 
$\mathbf{B}$ satisfies the covariance property $\Ad_g \mathbf{B}(u\otimes v) = \mathbf{B}(gu\otimes gv)$.

In spatial coordinates,
\eqref{e:SYMH_no_deturck} reads
\begin{equs}[eq:langevin_coordinate]
	\partial_t A_i
	&= \Delta A_i  - \partial_{ji}^2 A_j + [A_j,2\partial_j A_i - \partial_i A_j + [A_j,A_i]]\\
	{}& \qquad 
	+ [\partial_j A_j, A_i] -\mathbf{B} ( (\partial_{i} \Phi + A_{i}\Phi)\otimes \Phi) + \xi_i \;,\\
	\partial_t\Phi  &= 
	\Delta \Phi + (\partial_{j}A_{j})\Phi + 2 A_{j} \partial_{j}\Phi + A_{j}^{2}\Phi -  m^2 \Phi - |\Phi|^2 \Phi + \zeta\;,
\end{equs}
for $i\in [d] = \{1,\ldots, d\}$ with the summation over $j$  implicit.
A major problem with this equation is the lack of parabolicity in the equations for $A$, which is a reflection of the invariance of the action $S_\YMH$ under the gauge group.
As discussed in \cite{CCHS2d},
this problem can be circumvented by taking any sufficiently regular time-dependent $0$-form $\omega\colon[0,T]\to \CC^\infty(\T^d,\mfg)$
and consider instead of \eqref{e:SYMH_no_deturck}
the equation
\begin{equs}[eq:SYMH_deturck]
	\partial_t A =& -\mrd_A^* F_A 
	+\mrd_A\omega
	-\mathbf{B}(\mrd_{A}\Phi\otimes\Phi)+ \xi\;,\\
	\partial_t\Phi  =& - \mrd_{ A}^* \mrd_A \Phi - \omega\Phi - m^{2} \Phi - |\Phi|^2 \Phi + \zeta\;.
\end{equs}
Then, at least formally, solutions to \eqref{eq:SYMH_deturck} are gauge equivalent in law to those of \eqref{e:SYMH_no_deturck}
under a time-dependent gauge transformation. 
To get a parabolic flow for $A$ in \eqref{eq:SYMH_deturck} a convenient choice of $\omega$ is $\omega = -\mrd^* A=\partial_jA_j$   which yields the 
so-called DeTurck--Zwanziger term~\cite{zwanziger81,deturck83,Donaldson} in the first equation of \eqref{eq:SYMH_deturck}. 
After making this choice our main focus will be the stochastic Yang--Mills--Higgs (SYMH) 
 flow which in coordinates reads
\begin{equs}
	\partial_t A_i &= \Delta A_i + \big[A_j,2\partial_j A_i - \partial_i A_j + [A_j,A_i]\big] 
	- \mathbf{B}((\partial_{i} \Phi + A_{i}\Phi) \otimes \Phi)+ \xi_i \;,\\
	\partial_t\Phi  &= 
	\Delta \Phi + 2 A_{j} \partial_{j}\Phi + A_{j}^{2}\Phi -  m^2 \Phi - |\Phi|^2 \Phi + \zeta\;. \label{eq:SYMH}
\end{equs}

%
The system \eqref{eq:SYMH} is formally gauge covariant in the following sense: if $(A,\Phi)$ is a solution to \eqref{eq:SYMH} then, given a time evolving gauge transformation $g$ solving $g^{-1} \partial_{t}g = -\mrd_A^* (g^{-1}  \mrd g)$,
$(A^{g},\Phi^{g})$ solves the same equation with $(\xi,\zeta)$ replaced by $(\Ad_{g}\xi,g\zeta)$ which, 
since $g$ is adapted, is equal in law to $(\xi,\zeta)$.

\subsection{Outline of the paper}

In Sections~\ref{sec:notation} and~\ref{subsec:main_theorems} we introduce important notation and summarise the main theorems.

In Section~\ref{sec:state_space} we construct a nonlinear metric space $\state$ of distributions which serves as the state space for SYMH on $\T^3$.
By using the regularising and gauge covariant properties of the deterministic YMH flow (with DeTurck--Zwanziger term), we define regularised gauge-invariant observables on $\state$ which have good continuity properties and show that gauge equivalence $\sim$ extends to $\state$ in a canonical way.
Furthermore sufficiently regular gauge transformations $g\in\mcC^\rho(\T^3,G)$ act on $\state$ and preserve $\sim$.
The main idea in the definition of $\state$ is to specify how the heat semigroup $t\mapsto \CP_t X$ behaves 
on elements $X\in\state$;
the space $\state$ is nonlinear because we force control on quadratic terms of the form $\CP_t X \cdot \nabla \CP_t X$ arising from the most singular terms in~\eqref{eq:SYMH}.
In Section~\ref{sec:SHE} we show that the stochastic heat equation defines a continuous stochastic process with values in $\state$.

In Section~\ref{sec:sym_renorm_reg_struct} we build on the ``basis free'' regularity structures framework in~\cite{CCHS2d} to show that certain symmetries are preserved by BPHZ renormalisation.
This is then used in Section~\ref{sec:renorm-A} to show that SYMH can be renormalised with mass counterterms to admit local solutions through mollifier approximations, which defines a continuous stochastic process with values in $\state$ (possibly with blow-up).

In Section~\ref{sec:gauge} we show that there exists a choice for the mass renormalisation so that SYMH is genuinely (not only formally) gauge covariant in law.
That is, the pushforwards to $\mfO = \state/{\sim}$ of the laws of two solutions $(A(t),\Phi(t))$ and $(\bar A(t),\bar \Phi(t))$ to SYMH for $t>0$ with initial conditions $(A(0),\Phi(0))=(a,\phi)$ and $(\bar A(0),\bar \Phi(0))=(a^g,\phi^g)$ respectively are equal (modulo subtleties involving restarting the equation).
This is done through an argument based on preservation of gauge symmetry in the small noise limit,
which is inspired by~\cite{StringManifold}.
In Section~\ref{sec:Markov} we show that there exists a Markov process on the space of ``gauge orbits'' $\mfO$, which is unique in a suitable sense and onto which the solution to SYMH from Section~\ref{sec:gauge} projects.

Finally in Appendix~\ref{app:Singular modelled distributions}
we collect some results concerning modelled distributions
with singular behaviour at $t=0$, which allows us 
to construct solutions to our SPDEs starting from suitable singular initial conditions. In Appendix~\ref{app:YMH flow without DeTurck term}
we collect some results on the deterministic YMH flow
which are useful for defining gauge equivalence $\sim$ on $\state$ regularised observables.
In Appendix~\ref{app:evolving_rough_g},
we extend the well-posedness result for $(A,\Phi)$ in
Section~\ref{sec:renorm-A} to a coupled
$(A,\Phi,g)$ system
which is used in Sections~\ref{sec:gauge} and~\ref{sec:Markov}.
In Appendix~\ref{app:injectivity}
we prove that our solution maps are injective in law as functions
of renormalisation constants, which is useful in showing 
gauge covariance in Section~\ref{sec:gauge}.

\subsection{Related works and open problems}

Let us mention several other works related to this one.
The idea to use the regularising properties of the deterministic YMH flow to define regularised gauge-invariant observables for singular gauge fields
was advocated in~\cite{CG13,GrossSmoothen} and more recently in~\cite{Sourav_flow,Sourav_state};
see also~\cite{BHST87_regularization,NN06,Luscher10,Fodor12, MR3133916} for related ideas in the physics literature.
The results in~\cite{Sourav_flow,Sourav_state} are closely related to those obtained in 
Sections~\ref{sec:state_space} and~\ref{sec:SHE}; see Remarks~\ref{rem:Sourav_exponents},~\ref{rem:Sourav_state},
and~\ref{rem:Sourav_flow} for a brief comparison.

In further work, Gross established a solution theory for the YM flow with initial conditions in $H^{1/2}$ in~\cite{Gross2016FiniteAction,GrossSingular}. This space is natural since it is scaling critical for the 
YM flow in three dimensions. (This is in the sense that it has small-scales scaling exponent 
$\alpha = -1$, so that $\Delta u$,
$u\cdot Du$ and $u^3$ all have the same scaling exponent $\alpha -2 = 2\alpha-1 = 3\alpha = -3$.)
Note however that the solution to the \textit{stochastic} YM flow does not belong to that 
space. Instead, it belongs 
to $\CC^{-\f12-\kappa}$ for any $\kappa > 0$ (but not to $\CC^{-\f12}$) which, although it is a 
space of rather badly behaved distributions, has scaling
exponent $-\f12-\kappa$ which is subcritical for the YM flow. Unfortunately, these spaces are 
sufficiently badly behaved so that there do exist vector-valued $X \in \CC^{-\f12}$ for which the map 
$t \mapsto \CP_t X \otimes \nabla \CP_t X$ fails to be locally integrable 
at $t=0$ (this does not happen when $u \in H^{\f12}$).
In fact, there exists no Banach space of distributions supporting the 3D Gaussian free field to which the YM flow extends continuously~\cite{C23}.

There are also number of recent results, from a probability theory perspective, on lattice gauge theories in $d=3$ and $d=4$ and their scaling limits.
Some works in this direction include~\cite{Chatterjee16} on the YM free energy,~\cite{Chatterjee20Wilson,Cao20,FLV20,FLV21} on the analysis of Wilson loops with discrete gauge group and~\cite{CC22, GS21} with $G=U(1)$,
and~\cite{Chatterjee21Confine} on the confinement problem.
See also~\cite{SSZZ22,SZZ23,SZZ24} for work on lattice gauge theory using Langevin dynamics.

The idea of stochastic quantisation was introduced in the physics literature by Nelson~\cite{Nelson66} and Parisi--Wu~\cite{ParisiWu}.
With the development of regularity structures~\cite{Hairer14} and paracontrolled distributions~\cite{GIP15} to solve singular SPDEs (see also~\cite{Kupiainen2016,Duch21,CF24}),
this idea has been  applied to the rigorous construction and study of scalar quantum fields, especially the $\Phi^4_3$ theory~\cite{MW17Phi43,MoinatWeber20,HS21,GH21,AK20}. See also~\cite{BG20,BG21,GM24} for another stochastic analytic approach.

We finally mention some earlier works of the authors.
In~\cite{CCHS2d} we studied the Langevin dynamic on for the pure YM measure on $\T^2$ (though the results therein carry over without fundamental problems to the YMH theory).
Because the equations in $d=2$ are much less singular than in $d=3$, we were able to obtain stronger results with different methods.
Namely,
\begin{itemize}
\item The state space constructed in~\cite{CCHS2d} was a linear Banach space and came with an action of a gauge group which determined completely the ``gauge equivalence'' relation $\sim$; 
we do not know here if there exists a gauge group acting on $\state$ which determines $\sim$ (we suspect it does not with the current definitions).

\item Gauge covariance of the SYM process in~\cite{CCHS2d} was shown through a direct computation of the renormalisation constants coming from just three stochastic objects.
Here such a computation is infeasible by hand due to the presence of a large number of logarithmic 
divergences, which is why we rely on more subtle symmetry arguments.

\item The SYMH~\eqref{eq:SYMH} for $d=2$ (and its deterministic version for any $d$)
admits arbitrary initial conditions in the H{\"o}lder--Besov space $\CC^\eta$ only for $\eta>-\frac12$, and the state space considered in~\cite{CCHS2d} embeds into such a H{\"o}lder--Besov space.
Here $\state$ embeds at best into $\CC^{\eta}$ for $\eta<-\frac12$, which causes significant complications in the 
short-time analysis.
This problem becomes even worse for the multiplicative noise versions of~\eqref{eq:SYMH}
considered in Section~\ref{sec:gauge} to the extent that we require a substantial change to the
fixed point problem when restarting the equation, in particular solving for a suitable `remainder',
in order to obtain maximal solutions.
\end{itemize}

We also mention that the first work to study the stochastic quantisation of a gauge theory using regularity structures is~\cite{Shen21} (using a lattice regularisation of $\T^2$ with $G=U(1)$ and a Higgs field),
and the first work to give a representation of the YM measure on $\T^2$ as a random variable taking values in a (linear) state space of distributional connections for which certain Wilson loops 
are defined pathwise is~\cite{Chevyrev19YM}.
The state space in~\cite{Chevyrev19YM} served as the basis for that in~\cite{CCHS2d},
and part of the definition of $\state$ in the present work is inspired by these works (see Sections~\ref{subsec:final_state_space} and~\ref{subsec:est_gauge_trans}).

We close with some open problems.
One of the main questions is whether the Markov process on $\mfO$ constructed in this paper admits an invariant measure, which should then be unique due to the strong 
Feller~\cite{HM16} and full support~\cite{HS19} properties of the solution.
We do conjecture that such an invariant measure exists, which then yields a reasonable 
candidate for the YMH measure on $\T^3$.
Unfortunately, we do not even know how to show the weaker statement that the 
Markov process survives for all times.
For the pure YM measure and Langevin dynamic on $\T^2$, this question was recently answered in~\cite{CS23}.

Another question is whether one can strengthen or change the construction of $(\state,\sim)$ 
in such a way that there exists a topological group $\bar \mfG$ containing 
$\mfG^\infty = \CC^\infty(\T^3,G)$ as a dense subgroup, acting on $\state$, having
closed orbits, and such that $\sim$ is given by its orbits.
For $d=2$, it was shown in~\cite{CCHS2d} that (the closure of smooth functions in) $\CC^{\alpha}(\T^2,G)$, for some $\alpha\in(\frac12,1)$ is such a group;
this is both aesthetically pleasing and a tool to simplify a number of arguments.
The lack of a nice gauge group in $d=3$, for example, leads to difficulties in studying 
the topology of $\mfO$
(we only know it is separable and completely Hausdorff here as opposed to Polish in~\cite{CCHS2d}),
and complicates the construction of the Markov process on $\mfO$.

We finally mention that it would be of interest to extend the results of this paper to $\R^3$, but this problem is entirely open.
In fact, this problem is open even in the 2D setting of~\cite{CCHS2d,CS23}. Regarding global-in-time solutions, see the very recent progress \cite{BringmannCaoHiggs} for Abelian-Higgs on $\T^2$.

\subsection{Notation and conventions}
\label{sec:notation}

We collect some conventions and notations used throughout the article.
Part of this notation follows~\cite{CCHS2d}.
We equip the torus $\T^3=\R^3/\Z^3$ with the geodesic distance denoted by $|x-y|$, and $\R\times \T^3 $ the parabolic distance $|(t,x)-(s,y)|=\sqrt{|t-s|}+|x-y|$.
We will tacitly identify $\T^3$ with the set $[-\frac12,\frac12)^3$.

Throughout the article we fix a compact Lie group $G$, with associated Lie algebra $\mfg$,
endowed with an adjoint invariant metric denoted by  $\langle \;,\;\rangle_\mfg$.
We will often assume $G$ to be embedded into a space of matrices.
Let $\higgsvec$
be a real vector space  of finite dimension $\dim \higgsvec \ge 0$
endowed with 
a scalar product $\langle \;,\;\rangle_\higgsvec$ and an orthogonal representation $\brho$ of $G$.
As mentioned before, we often drop $\brho$ from our notation and simply write $g\phi\equiv \brho(g)\phi$ and $h\phi\equiv\brho(h)\phi$ for $g\in G$, $h\in\mfg$, and $\phi\in\higgsvec$, where $\brho(h)\phi$ is understood as the derivative representation.

\begin{remark}
Even if $\higgsvec$ is a complex vector space endowed with a Hermitian inner product 
$\langle \;,\;\rangle_\higgsvec$, and the representation of $G$
on $\higgsvec$ is unitary, we instead view $\higgsvec$ as a real vector space
with $\dim_\R (\higgsvec) = 2\dim_\C (\higgsvec)$ endowed with 
the Euclidean inner product given by $\Re \langle \;,\;\rangle_\higgsvec$ and view the representation as an orthogonal 
representation on  $\higgsvec$. In fact, the definition of the YMH action \eqref{eq:YM_energy},
the deterministic part of the SPDE \eqref{e:SYMH_no_deturck} which is its gradient flow, 
and the definition of white noise in \eqref{e:explicit_covar_noise} all only depend on $\Re \langle \;,\;\rangle_\higgsvec$ and not on $\Im \langle \;,\;\rangle_\higgsvec$.
\end{remark}

A ``gauge field'' $A$ is an equivariant connection on the trivial\footnote{We emphasise again that 
the case of non-trivial bundles is of course 
very interesting but will not be treated here.} principal bundle
$\CP\simeq\T^3\times G$ viewed as a $1$-form $A=(A_1,A_2,A_3)\colon\T^3\to \mfg^3$,
which determines a covariant derivative $\mrd_{A}$ on the associated bundle
$\CV=\CP\otimes_{\brho}\higgsvec \simeq \T^3\times \higgsvec$ by
$\mrd_A\Phi=\mrd\Phi+A\Phi= (\partial_i\Phi+A_i\Phi)\,\mrd x_i$
(see e.g.\ \cite[Sec.~2.1.1]{DK90}).
A ``Higgs field'' is a section of $\CV$, viewed simply as a function $\Phi\colon\T^3\to\higgsvec$.
The curvature of a gauge field $A$ is the $\mfg$-valued $2$-form $(F_A)_{ij}=\partial_iA_j -\partial_jA_i+[A_i,A_j]$.

A space-time {\it mollifier} $\moll$ is a compactly supported smooth function on  $\R\times\R^3$  such that $\int\moll = 1$ and 
for every $i\in \{1,2,3\}$, $\moll$ is invariant under $x_i \mapsto -x_i$ along with $(x_1,x_2,x_3) \mapsto (x_{\sigma(1)},x_{\sigma(2)},x_{\sigma(3)})$ for any permutation $\sigma$ on $\{1,2,3\}$.
$\moll$ is called {\it non-anticipative}
if it is supported in the set $\{(t,x): t \geq 0\}$.

Assume that we are given a finite-dimensional normed space $(E, |\cdot|)$ 
and a metric space $(F,d)$. 
For $\alpha\in(0,1]$,  we define as usual
\begin{equ}
\mcC^\alpha(F,E)
=\{ f\colon F\to E \;:\;
|f|_{\mcC^{\alpha}} 
\eqdef |f|_{\infty} + |f|_{\Hol\alpha} < \infty \}\;,
\end{equ}
where $|f|_{\Hol\alpha} \eqdef \sup_{x\neq y\in F} \frac{|f(x)-f(y)|}{d(x,y)^\alpha} < \infty
$ denotes the H{\"o}lder seminorm and $|f|_\infty = \sup_{x\in F} |f(x)|$ denotes the sup norm.

For $\alpha>1$, we define 
$\mcC^\alpha(\T^3,E)$
(resp.\ $\mcC^\alpha(\R\times\T^3,E)$)  to be the space of $k\eqdef\roof\alpha-1$ times differentiable functions (resp.\ $k_0$-times differentiable in $t$ and $k_1$-times differentiable in $x$ for all $2k_0+k_1\le k$), with $(\alpha-k)$-H{\"o}lder continuous $k$-th derivatives.

For $\alpha<0$, let $r \eqdef -\roof{\alpha-1}$ and we define
\begin{equ}
\mcC^\alpha(\T^3,E)
= \{ \xi \in \mcD'(\T^3,E)\;:\;
|\xi|_{\mcC^\alpha} \eqdef \sup_{\lambda\in(0,1]}\sup_{\psi \in \mcB^r} \sup_{x \in \T^3} \lambda^{-\alpha} |\scal{\xi,\psi^\lambda_x}| < \infty\}\;,
\end{equ}
where $\mcB^r$ denotes the set of all smooth functions $\psi \in \mcC^\infty(\T^3)$ with $|\psi|_{\mcC^r} \leq 1$ and support in $\{z\in \T^3\,:\,|z|\leq \frac14\}$
and where $\psi^\lambda_x(y) \eqdef \lambda^{-3}\psi(\lambda^{-1}(y-x))$.

For $\alpha=0$, we define $\mcC^0$ to simply be $L^\infty(\T^3,E)$,
and use $\mcC(\T^3,E)$ to denote the space of continuous functions,
both spaces being equipped with the sup norm.
For any $\alpha\in\R$, we denote by $\mcC^{0,\alpha}$ the closure  of smooth functions in $\mcC^\alpha$.
We drop $E$ from the notation
whenever it is clear from the context. 

If $\CB$ is a space of $\mfg$-valued distributions equipped with a (semi-)norm $|\cdot|$,
then $\Omega\CB$ denotes the space of $\mfg$-valued distributional $1$-forms $A= A_i\, \mrd x_i$ where $A_i\in\CB$, equipped with the corresponding (semi-)norm $|A|\eqdef \sum_{i=1}^3 |A_i|$.
When $\CB$ is of the form $\mcC(\T^3,\mfg)$, $\mcC^{\alpha}(\T^3,\mfg)$, etc., we write simply $\Omega\mcC$, $\Omega\mcC^{\alpha}$, etc. for $\Omega\CB$.
For $\rho\in[0,\infty]$,
we write $\mfG^\rho \eqdef \CC^\rho(\T^3,G)$\label{mfG page ref}
and let $\mfG^{0,\rho}$\label{mfG^0 page ref}
denote the closure of smooth functions in $\mfG^\rho$,
where we understand $G$ as embedded into a space of matrices.
We often call $\mfG^\rho$ a \textit{gauge group} and its elements \textit{gauge transformations}.

In the remainder of the article, unless otherwise stated, we denote $E=\mfg^3\oplus \higgsvec$. \label{def of E}
This implies $\CD'(\T^3,E)\simeq \Omega\CD'\oplus \CD'(\T^3,\higgsvec)$ and, for $X\in\CD'(\T^3,E)$, we write $X=(A,\Phi)$ for the corresponding decomposition.
In particular, the configuration space of smooth (connection-Higgs) pairs $(A,\Phi)$ is $\CC^\infty(\T^3,E)$.

We similarly combine noises into a single variable $\xi = ((\xi_i)_{i=1}^{3},\zeta)$,
which in view of \eqref{e:explicit_covar_noise} has covariance 
$\E[\xi(t,x) \otimes \xi(s,y)] = \delta(t-s)\delta(x-y)  \bCov$, with 
\begin{equ}[e:def-bCov]
\bCov = \Cas^{\oplus 3} \oplus \Cov\;.
\end{equ}
Note that $E$ carries a representation of $G$ given for $g\in G$ by
\begin{equ}\label{eq:group_action_E}
E\ni x = ((x_i)_{i=1}^3,y)\mapsto g x \eqdef ((\Ad_g x_i)_{i=1}^3,gy) \in E\;.
\end{equ}
If $g$ and $x$ are functions (or distributions) with values in $G$ and $E$ respectively, we let $g\xi$ denote the above operation pointwise whenever it makes sense.

For $X,Y\in\CC^\infty(\T^3,E)$ we write $X\sim Y$ if there exists $g\in\mfG^{\infty}$
such that $X^g=Y$, where $X^g=(A^g,\Phi^g)$ is defined in~\eqref{e:gauge-transformation}.
Recall that $X\mapsto X^g$ defines a left group action of $\mfG^\infty$ on $\CC^\infty(\T^3,E)$.
We denote by $\mfO^\infty \eqdef \CC^\infty(\T^3,E)/{\sim}$ the corresponding quotient space.  \label{mfO infinity}
(The action $X\mapsto X^g = g\act X$ of $\mfG^\infty$ on $\CC^\infty(\T^3,E)$ should not be confused with the action of $\mfG^\infty$ on $\CC^\infty(\T^3,E)$ given by~\eqref{eq:group_action_E}.)


We will often use the following streamlined notation for writing the nonlinear terms of our equations. 
For any $X\in E$ and $i \in \{1,2,3\} $, we write $X |_{\mfg_i} \in \mfg$
and $X |_{\higgsvec} \in \higgsvec$
to be the projections of $X$ onto the $i$-th component of $X$ which is a copy of $\mfg$
and the last component of $X$
which is a copy of $\higgsvec$ respectively.
Given any $X\in E$ and $\p X = (\p_1 X,\p_2 X,\p_3 X)\in E^3$,
where $\p_i X$ is just a generic element in $E$ (which does not necessarily
mean a derivative in general), and similarly $\bar X\in E$ and $\p \bar X$,
 we introduce
the shorthand notation
$X\p \bar X, X^3 \in E$ defined as 
  \begin{equs}[2]
X\p \bar X |_{\mfg_i}
&=
\big[A_j,2\partial_j \bar A_i - \partial_i \bar A_j \big] 
- \mathbf{B}(\partial_{i} \bar \Phi \otimes \Phi) \;, 
&\,
X\p \bar X |_{\higgsvec}
&=2 A_{j} \partial_{j}\bar\Phi \;,
\\ \notag
X^3 |_{\mfg_i}
&= 
\big[A_j,   [A_j,A_i]\big] - \mathbf{B}(( A_{i}\Phi) \otimes \Phi)  \;,
\qquad&\,
X^3 |_{\higgsvec}
&=
- |\Phi|^2 \Phi	\label{e:XdX-X3}
 \end{equs}
 where, as above, the summation over $j$ is implicit and we have written
\begin{equ}[e:def-A_i and Phi]
 A_i \eqdef X|_{\mfg_i} \;,\;
 \p_j A_i \eqdef \p_j X|_{\mfg_i} \;,\;
 \Phi \eqdef X|_{\higgsvec} \;,\;
 \p_j \Phi \eqdef \p_j X|_{\higgsvec}\;,
 \end{equ}
and $\bar A_i, \partial_j \bar A_i, \bar\Phi, \partial_j \bar \Phi$ are understood in the analogous way.

Recall the following notation from \cite[Sec.~1.5.1]{CCHS2d}.
Given a metric
space $F$, we extend it with a cemetery state $\skull$ by setting
$\hat F = F\sqcup\{\skull\}$ and postulating that the complement of every closed ball in $F$ is a neighbourhood of $\skull$ in $\hat F$.
We then recall the definition of the metric space $F^\sol$\label{ref:Fsol} from \cite[Sec.~1.5.1]{CCHS2d}, which should be thought of as the space of continuous trajectories with values in $\hat F$ which can 
blow up in finite time but cannot be ``reborn''. The purpose of the rather convoluted definition of
the metric of  $F^\sol$ is to guarantee that it is separable and complete (provided that $F$ is).

\subsection{Main theorems}
\label{subsec:main_theorems}

We first collect our results on the state space of our 3D stochastic YMH process,
with references to more precise statements in later sections. Here, we will use the notation 
$\CF_t$ for the DeTurck--YMH flow (the solution to \eqref{eq:SYMH} with $\xi = \zeta = 0$ and the $\Phi |\Phi|^2$ term dropped\footnote{Dropping the $\Phi |\Phi|^2$ term is done for purely technical convenience since this allows us 
to reuse results from \cite{HongTian2004} and has no effect on our statements \dash see Remark~\ref{rem:ymh_flow_deturck}.}) at time $t$. 
Standard parabolic PDE theory shows that this is well-posed for short times for all initial conditions
in $\CC^\nu(\T^3,E)$ as soon as $\nu > -\f12$. In other words, writing $\CO_t \subset \CC^\nu$
for the set of initial conditions admitting a solution up to time $t$, these are a decreasing family of open sets
with $\bigcup_{t \in (0,\eps]}\CO_t = \CC^\nu$. On the other hand, the projection $\tilde \CF_t$ of 
the DeTurck--YMH flow onto $\mfO^\infty$ is globally well-posed for all initial conditions in $\CC^\nu$. (This allows for
a time-dependent gauge transformation which doesn't change the projection of the flow to $\mfO^\infty$ but
can prevent it from blowing up, see Appendix~\ref{app:YMH flow without DeTurck term}, in particular
Corollary~\ref{cor:flow_continuous}.)
Recalling that $E=\mfg^3\oplus \higgsvec$, we can state our results regarding the state space 
$\CS$ (see Definition~\ref{def:state} below) as follows.

\begin{theorem}[State space]\label{thm:main_state_space}
For every $\eta\in (-\frac12-\kappa,-\frac12)$, where $\kappa>0$ is sufficiently small,
there exists a complete (nonlinear) metric space $(\state,\Sigma)$ of $E$-valued distributions 
satisfying the following properties.
\begin{enumerate}[label=(\roman*)]
\item There is a canonical embedding $\CS \hookrightarrow \CC^\eta(\T^3,E)$ and there exists $\bar\nu < 0$ such that 
$\CC^{\bar\nu}(\T^3,E) \hookrightarrow \state$ densely. Furthermore, $\CS$ is closed under scalar multiplication 
when viewed as a space of distributions.
See Lemmas~\ref{lem:heatgr_Besov_embed} and~\ref{lem:perturbation}\ref{pt:fancynorm_0}.

\item The deterministic DeTurck--YMH flow $X \mapsto \ymhflow_t(X)\in\CC^\infty(\T^3,E)$
extends continuously to the closure of $\CO_{t+s}$ in $\CS$ for every $s,t>0$. It follows that 
$t \mapsto \ymhflow_t(X)$ is well-posed for every $X \in \CS$ and every 
sufficiently small $t$ (depending on $X$) and 
the flow on gauge orbits $X\mapsto \tilde\ymhflow_t\in\mfO^\infty(X)$ extends continuously to
all of $\CS$ for every $t>0$. One furthermore has
\begin{equ}
\lim_{t\to0} \Sigma(\mcF_t(X),X)=0 \qquad \forall X \in\CS\;.
\end{equ}
See Proposition~\ref{prop:YM_flow_minus_heat} and Lemma~\ref{lem:Sigmas_converg}.

\item\label{pt:gauge_equiv_extend} Define the equivalence relation on $\state$ by $X\sim Y \Leftrightarrow\tilde\ymhflow(X) = \tilde\ymhflow(Y)$.
Then $\sim$ extends
the notion of gauge-equivalence defined for smooth functions.
Moreover, the quotient space $\mfO\eqdef \state/{\sim}$ 
is a separable completely Hausdorff space.
See Proposition~\ref{prop:Hausdorff}.   

\item\label{pt:rho_gauge} There exists $\rho \in (\frac12,1)$ and a continuous left group action $\mfG^{\rho}\times\state \ni (g,X)\mapsto X^g \in\state$
for which $X\sim X^g$
and which agrees with the action~\eqref{e:gauge-transformation} for smooth $g$ 
and $X$. See Proposition~\ref{prop:group_action}.

\item\label{pt:bound_g} There exist $\nu\in(0,\frac12)$ and $C,q>0$ such that, for all $g \in \mfG^ \rho$ and $X \in \CS$, one has
$|g|_{\CC^\nu} \leq C(1+\Sigma(X,0)+\Sigma(X^g,0))^q$.  See Theorem~\ref{thm:g_bound}.
\end{enumerate}
\end{theorem}

\begin{remark}\label{rem:observables}
A consequence of Theorem~\ref{thm:main_state_space}\ref{pt:gauge_equiv_extend} is that classical gauge-invariant observables
 (Wilson loops, string observables, etc) have ``smoothened'' analogues defined on
 $\mfO$ obtained by precomposing the classical observable with $\tilde\ymhflow_t$ for some (typically small) $t>0$. 
These smoothened observables are sufficient to separate points in $\mfO$, see Section~\ref{subsec:regularised_Wilson}.
\end{remark}

\begin{remark}
The significance of point~\ref{pt:bound_g}
may seem unclear at this stage.
However, this estimate is crucial in the construction of the Markov process on $\mfO$ in Section~\ref{sec:Markov}.
\end{remark}

\begin{remark}\label{rem:two_parts}
The definition of the metric $\Sigma$ is given in two parts: $\Sigma(X,Y)=\Theta(X,Y)+\heatgr{X-Y}$.
The metric $\Theta$ is defined in Section~\ref{sec:Space of initial conditions}
and guarantees continuity of $\tilde\ymhflow_t$ with respect to the initial condition.
The norm $\heatgr{\cdot}$
is defined in Section~\ref{subsec:est_gauge_trans} and implies point~\ref{pt:bound_g}.
Both metrics come with several parameters, the final values of which are given at the beginning of Section~\ref{sec:renorm-A}.
\end{remark}

Fix i.i.d.\ $\mfg$-valued white noises $(\xi_i)_{i=1}^3$ on $\R\times\T^3$ and an independent $\higgsvec$-valued space-time white noise $\zeta$ on $\R \times \T^{3}$ and write
$\xi^\eps_i \eqdef \xi_i * \moll^\eps$ along with $\zeta^\eps \eqdef \zeta * \moll^{\eps}$.
Here $\moll$ is a mollifier as in Sec.~\ref{sec:notation}
and $\moll^{\eps}(t,x) =\eps^{-5} \moll(\eps^{-2}t,\eps^{-1}x)$.
For each $\eps\in(0,1]$ consider the system of SPDEs on $\R_+ \times \T^3$
with $i \in \{1,2,3\}$\footnote{One would have a term $- m^2 \Phi $ here in view of \eqref{eq:YM_energy}, but we absorb it into the term $C^\eps_{\Phi} \Phi$.} 
\begin{equs}
\partial_t A_i &= 
\Delta A_i  + [A_j,2\partial_j A_i - \partial_i A_j + [A_j,A_i]] 
\\
&\qquad\qquad -\mathbf{B}((\partial_{i} \Phi + A_{i}\Phi) \otimes \Phi)+ C^{\eps}_{\A} A_i +  \xi^\eps_i \;,
\\
\partial_t\Phi  &= 
\Delta \Phi  
+ 2 A_{j} \partial_{j}\Phi + A_{j}^{2}\Phi  - |\Phi|^2 \Phi + C^{\eps}_{\Phi} \Phi +  \zeta^{\eps} \;,\\
(A(0),& \Phi(0)) = (a,\phi) \in \state\;,
		\label{eq:SPDE_for_A}
\end{equs}
where the summation over $j$ is again implicit, and we fix some choice of $(C_{\A}^{\eps}, C_{\Phi}^{\eps}: \eps \in (0,1])$ with 
\[
C^{\eps}_{\A} \in L_{G}(\mfg,\mfg)\;,
\qquad
\mbox{and}
\qquad 
C_{\Phi}^{\eps}\in L_{G}(\higgsvec,\higgsvec)\;. 
\]   
Here $L_G(\higgsvec,\higgsvec)$ (resp.\   $L_G(\mfg,\mfg)$) is the space of 
all the linear operators from $\higgsvec$  (resp.\  $\mfg$) to itself which commute with the  action of $G$ (resp.\ adjoint action of $G$).
Recall also that $\mathbf{B}\colon \higgsvec \otimes \higgsvec \rightarrow \mathfrak{g}$  is the bilinear form 
determined by~\eqref{e:def-B}.  In view of the notation introduced in Section~\ref{sec:notation}, 
we  may also write \eqref{eq:SPDE_for_A} as 
\begin{equ}[e:simpleNotation]
\partial_t X = \Delta X + X \partial X + X^3 
+( (C_{\A}^{\eps})^{\oplus 3} \oplus C_{\Phi}^{\eps} ) X + \xi^\eps\;,
\end{equ}
where for any $ C \in L(\mfg,\mfg) $ we write $C^{\oplus 3}$ for $C\oplus C\oplus C\in L(\mfg^3,\mfg^3)$.

\begin{remark}\label{rem:V=g}
One particular example in our setting is that $\CV$ is the adjoint bundle, i.e.\ $\higgsvec=\mfg$, and $G$ acts on $\higgsvec$ by adjoint action. In this case, $\Phi$ is also $\mfg$-valued and
the bilinear form is simply given by the Lie bracket $ \mathbf{B}(\mrd_{A}\Phi\otimes\Phi)  = - [\mrd_{A}\Phi, \Phi]$.
\end{remark}
 
Recall the definition of $\state^\sol$
from Section~\ref{sec:notation} (see $F^\sol$ below \eqref{e:def-A_i and Phi} for any metric space $F$).

\begin{theorem}[Local existence]\label{thm:local_exist}
Consider any $\mathring{C}_{\A} \in L_{G}(\mfg,\mfg) $
and a space-time mollifier $\moll$.   
Then there exist $(C_{\YM}^{\eps}, C_{\Higgs}^{\eps})_{\eps \in (0,1]}$ 
with $C_{\YM}^{\eps} \in L_{G}(\mfg,\mfg)$ and  $C_{\Higgs}^{\eps} \in L_{G}(\higgsvec,\higgsvec)$ and which depend only on $\moll$, such that the following statements hold.

\begin{enumerate}[label=(\roman*)]
\item\label{pt:loc_sols}
The solution $(A,\Phi)$ to the system \eqref{eq:SPDE_for_A}
where
\begin{equ}
C_{\A}^{\eps} 
= 
	C_{\YM}^{\eps}
	+\mathring{C}_{\A} \;,
\qquad
C_{\Phi}^{\eps} 
=  C_{\Higgs}^{\eps}
 -m^2 
\end{equ}
converges in $\state^\sol$ in probability as $\eps \to 0$. 

\item
The limit in item~\ref{pt:loc_sols} depends only on $\mathring C_\A$ and not on $\moll$. 
\end{enumerate}
\end{theorem} 

\begin{remark}\label{rem:BPHZ_consts}
See Theorem~\ref{thm:local-exist-sigma} for a slightly more general version of Theorem~\ref{thm:local_exist}.
The operators $C^\eps_{\YM}$ and $C^\eps_{\Higgs}$ will be determined by the BPHZ character that deforms the canonical lift of $\xi^{\eps}$ into the BPHZ lift of $\xi^{\eps}$.

While the equation \eqref{eq:SPDE_for_A} does fall under the ``blackbox'' local existence theory of 
\cite{Hairer14,CH16,BHZ19,BCCH21} and the vectorial regularity structures of \cite{CCHS2d}, this theory does not directly give us Theorem~\ref{thm:local_exist}. 
There are three issues we must overcome:  (i) we want our solution to take values in the non-standard state space $\CS$, (ii) we want to start the dynamic from arbitrary, rough (not necessary ``modelled'') initial data in $\CS$, and (iii) we must verify that the renormalisation counter-term is given by $((C_{\YM}^{\eps})^{\oplus 3} \oplus C_{\Higgs}^{\eps})X$ for some $C_{\YM}^{\eps} \in L_{G}(\mfg,\mfg)$ and $C_{\Higgs}^{\eps} \in L_G(\higgsvec,\higgsvec)$ which is not obvious from the formulae for counterterms provided in \cite{BCCH21, CCHS2d}.  
\end{remark}

Our next result is about gauge covariance of the limiting solution,
provided that the operator $\mathring{C}_{\A}$ is suitably chosen. 
See \cite[Sec.~2.2]{CCHS2d} for a discussion in the 2D case (which extends \textit{mutatis mutandis} to the 3D case) on gauge covariance, and lack thereof before the limit, from a geometric perspective.
To formulate this result, we consider a gauge transformation $g(0)$ acting on the initial condition $(a,\phi)$ of $(A,\Phi)$ as in~\eqref{eq:SPDE_for_A},
and define a new dynamic $(B,\Psi)$ with initial condition $g(0)\act (a,\phi)$
such that $(B,\Psi)=g\act (A,\Phi)$ for some suitable time-dependent gauge-transformation $g$.
The transformation $g$ is chosen in such a way as to ensure that $(B,\Psi)$ converges in law to 
the solution to SYMH with initial condition $g(0)\act (a,\phi)$, provided that $\mathring{C}_{\A}$ is suitably chosen.
The resulting dynamics $(B,\Psi)$ and $g$ satisfy the equations
\begin{equs}[2]
\partial_t B_i &= 
\Delta B_i 
+[B_j,2\partial_j B_i - \partial_i B_j + [B_j,B_i]]
 -\mathbf{B}((\partial_{i} \Psi + B_{i}\Psi) \otimes \Psi)
 \\
&\qquad\qquad\qquad\qquad
+ C^{\eps}_{\A} B_i + C^{\eps}_{\A} (\partial_i g) g^{-1} +g\xi^\eps_i g^{-1} \;,
\\
\partial_t\Psi  &= 
\Delta \Psi 
+ 2 B_{j} \partial_{j}\Psi + B_{j}^{2}\Psi  - |\Psi|^2 \Psi + C_{\Phi}^{\eps} \Psi +  g\zeta^{\eps}\;,
\\
(\partial_t g) g^{-1} &= \partial_j((\partial_j g)g^{-1})+ [B_j, (\partial_j g)g^{-1}]\;,\label{eq:SPDE_for_B}
\\
& \qquad (B(0),\Psi(0))  = g(0)\act (a, \phi)\;,
\quad
g(0) \in \mfG^{0,\rho}\;.
\end{equs}
Above and for the rest of this section, we let $\rho \in (\frac12,1)$ be as in Theorem~\ref{thm:main_state_space}\ref{pt:rho_gauge}.

As mentioned above, for this choice of $(B,\Psi)$ and $g$, one has $(B,\Psi)=g\act (A,\Phi)$
for any fixed $\eps > 0$. 
Furthermore, $g$ as given by the solutions to \eqref{eq:SPDE_for_B} also solves
\begin{equ}\label{eq:SPDE_for_g_wrt_A}
g^{-1}(\partial_t g)
= \partial_j(g^{-1}\partial_j g)+ [A_j,g^{-1}\partial_j g]\;.
\end{equ}
Note that  \eqref{eq:SPDE_for_g_wrt_A} with $A$ given as in Theorem~\ref{thm:local_exist} is also classically ill-posed as $\eps \downarrow 0$ but can be shown to converge using regularity structures (however, it might blow up before $A$ does).  
Since the products in $g \act (A, \Phi)$ are well defined in the spaces where convergence takes place, this gives one way of seeing that the solutions to \eqref{eq:SPDE_for_B} also converge as $\eps \downarrow 0$. 

For any $\mathring{C}_{\A} \in L(\mfg,\mfg)$, let $\mathcal{A}_{\mathring{C}_{\A}}\colon \state \rightarrow \state^{\sol}$ be the solution map taking initial data $(a,\phi) \in \state$ to the limiting maximal solution of \eqref{eq:SPDE_for_A} promised by Theorem~\ref{thm:local_exist}. \label{SYM_solmap}
Our main result on the construction of a gauge covariant process can be stated as follows. 
\begin{theorem}[Gauge covariance]\label{theo:meta}
There exists a unique $\mathring{C}_{\A} \in L_{G}(\mfg,\mfg)$, independent of our choice of mollifier $\moll$, with the following properties.
\begin{enumerate}[label=(\roman*)]
 \item\label{pt:gauge_covar_informal} For all $g(0) \in \mathfrak{G}^{\rho}$ and $(a,\phi) \in \state$, one has, modulo finite time blow-up 
 \[
 g\act \mathcal{A}_{\mathring{C}_{\A}}(a,\phi)
\eqlaw \mathcal{A}_{\mathring{C}_{\A}} \big(g(0)\act (a,\phi) \big)
\quad \quad \quad
\text{}
\]
 where  $g$ is given by \eqref{eq:SPDE_for_g_wrt_A}
with $A$ therein given by the corresponding component of $\mathcal{A}_{\mathring{C}_{\A}}(a,\phi)$
and initial condition $g(0)$. 
See Theorem~\ref{thm:gauge_covar} for a precise statement.

\item \label{pt:Markov_informal}
There exists a unique
Markov process $\mathscr{X}$ 
on $\mfO$ such that, for every $(a,\phi) \in \CS$, if $\mathscr{X}$ is started from $[(a,\phi)]$ then
there exists a random time $t > 0$ such that, for all $s \in [0,t]$, $\mathscr{X}_s = [\mathcal{A}_{\mathring{C}_{\A}}(a,\phi)_s]$.  
See Theorem~\ref{thm:Markov_process}.
\end{enumerate}
\end{theorem}

\begin{remark}\label{rem:not_precise}
One reason why statement~\ref{pt:gauge_covar_informal} above is not precise is that
it is not clear that $g\act \mathcal{A}_{\mathring{C}_{\A}}(a,\phi)$ as given above
belongs to $\state^{\sol}$ \dash we cannot exclude that both $g$ and $\mathcal{A}_{\mathring{C}_{\A}}(a,\phi)$ blow up at some finite time $T$
but in such a way that $g\act \mathcal{A}_{\mathring{C}_{\A}}(a,\phi)$ converges to a finite limit at $T$.
\end{remark}

\subsection*{Acknowledgements}

{\small
We would like to thank Leonard Gross for helpful discussions. We are also very grateful to the
referees whose many comments allowed to substantially improve the article.
IC acknowledges support from the EPSRC via the New Investigator Award EP/X015688/1.
MH gratefully acknowledges support from the Royal Society through a 
research professorship, grant RP\textbackslash R1\textbackslash 191065. HS gratefully acknowledges support by NSF grants DMS-1954091 and CAREER DMS-2044415.}

\section{State space and gauge-invariant observables}
\label{sec:state_space}

We construct in this section the metric space $(\state,\Sigma)$ described in Theorem~\ref{thm:main_state_space}.
As mentioned in Remark~\ref{rem:two_parts},
we first define a metric space $(\init,\Theta)$ in Section~\ref{sec:Space of initial conditions} which will serve as the space of initial conditions for the (stochastic and deterministic) YMH flow.
The final state space $(\state, \Sigma)$ is defined in Section~\ref{subsec:final_state_space} using an additional norm on $\init$.
The definitions of the spaces $\init$ and $\state$ will depend on several parameters,
and we will formulate a condition \eqref{eq:CI} (which stands for {\it initial}) for the parameters
 under which we can solve the deterministic initial value problem in the space $\CI$,
and a condition 
 \eqref{eq:CGI} (resp. \eqref{eq:CGS})
 under which we have a continuous left group action of $\mathfrak G^{0,\rho}$ on $\CI$ (resp. on $\CS$).

\begin{remark}
All the results of this section hold with $\T^3$ replaced by $\T^d$ for $d=2,3$.
Furthermore, the only result which requires $d\leq 3$ is
the global existence of the deterministic YMH flow without DeTurck term $\flow$ from Appendix~\ref{app:YMH flow without DeTurck term}, which is used in Definition~\ref{def:ymhflow}
(and even this could be disposed of by redefining $\sim$ in Definition~\ref{def:ymhflow} to use only local in time solutions).
Note that long-time existence in $d=4$ of the YM flow (i.e.\ no Higgs component) was recently shown in~\cite{Waldron19}.
However, the state space $\init$ in Definition~\ref{def:init} below would not support the Gaussian free field in dimension $d \geq 4$ (or distributions of similar regularity).
Furthermore, the results of all subsequent sections break down badly in $d\ge 4$.
\end{remark}

We write ``Let $\eta=\alpha-$, $\beta=\nu+$, etc. Then \ldots''
to indicate that there exists $\eps>0$ such that for all $\eta\in(\alpha-\eps,\alpha)$, $\beta\in(\nu,\nu+\eps)$, etc. ``\ldots''
holds.
If ``\ldots'' involves a statement of the form ``there exists $\rho=\gamma-$'', this means that there exists $\rho < \gamma$ and $\rho\to \gamma$ as $\eta\to\alpha$, $\beta\to\nu$, etc.,
and similarly for ``$\rho=\gamma+$''.

We will also use the shorthands $x\leq \Poly(K)y$ and $x\lesssim y$
to denote that
$x\leq C (K+1)^q y$ and $x\leq C y$ respectively
for some $C,q>0$ which, unless otherwise stated, depend only the Greek letters $\alpha$, $\beta$, etc.

\subsection{Space of initial conditions}
\label{sec:Space of initial conditions}

\begin{definition}\label{def:N-B}
Recall the notation $E=\mfg^3\oplus \higgsvec$.
For $X\in\CD'(\T^3,E)$,
define
\begin{equ}
\CP X \colon(0,\infty) \to \CC^\infty(\T^3,E)\;,
\qquad \CP_t X \eqdef e^{t\Delta}X\;,
\end{equ}
the solution to the heat equation with initial condition $X$,
and
\begin{equ}\label{eq:CN_def}
\CN(X) \colon (0,\infty) \to \CC^\infty(\T^3, E\otimes E^3)\;,\qquad \CN_t(X) \eqdef \CP_t X\otimes \nabla \CP_t X\;.
\end{equ}
For $\delta$, $\beta\in\R$,
let $\CB^{\beta,\delta}$ denote the Banach space of continuous functions $Y\colon(0,1]\to\CC^\beta(\T^3,E\otimes E^3)$ for which \label{def:B-beta-delta}
\begin{equ}
|Y|_{\CB^{\beta,\delta}} \eqdef \sup_{t\in(0,1)} t^\delta |Y_t|_{\CC^{\beta}} < \infty\;.
\end{equ}
\end{definition}

For $\eta\in \R$,
we define the space $D(\CN) \eqdef \{X\in \CC^\eta(\T^3,E) \,:\, \CN(X)\in\CB^{\beta,\delta}\}$
endowed with the topology induced from $\CC^\eta$.
Although, for the parameters $\eta$, $\beta$, $\delta$ in the regime we care about, the function 
$\CN\colon D(\CN) \to \CB^{\beta,\delta}$ is not continuous,
 the continuity of $\CC^\eta(\T^3,E)\ni X\mapsto \CN_t(X)\in\CC^\beta(\T^3,E\otimes E^3)$ for each $t>0$
 easily yields the following:
 
\begin{lemma}\label{lem:closed_graph}
For every $\eta$, $\beta$, $\delta\in\R$, the graph of $\CN\colon D(\CN)\to \CB^{\beta,\delta}$ is closed.\qed
\end{lemma}

\begin{definition}\label{def:init}
For $X,Y\in\CD'(\T^3,E)$ and $\delta,\beta,\eta\in\R$, define the (extended) pseudometric and (extended) metric
\begin{equ}
\fancynorm{X;Y}_{\beta,\delta} \eqdef |\CN(X)-\CN(Y)|_{\CB^{\beta,\delta}}\;,
\quad
\Theta_{\eta,\beta,\delta}(X,Y) \eqdef |X-Y|_{\CC^\eta} + \fancynorm{X;Y}_{\beta,\delta}\;.
\end{equ}
Let $\init=\init_{\eta,\beta,\delta}$
denote the closure of smooth functions under the metric $\Theta\equiv \Theta_{\eta,\beta,\delta}$.
We define the shorthands $\fancynorm{X}_{\beta,\delta}\eqdef \fancynorm{X;0}_{\beta,\delta}$ and $\Theta(X)\eqdef \Theta(X,0)$.
We will often drop the reference to $\eta,\beta,\delta$ in the notation $(\init,\Theta)$.
Unless otherwise stated, we equip $\init$ with the metric $\Theta$.
\end{definition}

\begin{remark}
By Lemma~\ref{lem:closed_graph},
$\init$ can be identified with a subset of $\CC^{0,\eta}$.
\end{remark}

\begin{remark}
We will later choose $\eta$, $\beta$, $\delta$ such that the additive stochastic heat equation 
defines a continuous $\init$-valued process.
An (essentially optimal) example is $\eta=-\frac12-$, $\delta\in (\frac34,1)$, 
and $\beta = -2(1-\delta)-$.
\end{remark}


For a (possibly time-dependent) distribution $X$ taking values in $E$,
we will often write $X=(A,\Phi)$ to denote the two components of $X$ in the decomposition $\CD'(\T^3,E)\simeq \Omega\CD'\oplus \CD'(\T^3,\higgsvec)$. 

\begin{definition}\label{def:space-CI}
We say that $(\eta,\beta,\delta)\in \R^3$ satisfies condition~\eqref{eq:CI} if
\begin{equation}\label{eq:CI}
\begin{split}
&\eta\in(-2/3,0]\;,\quad
\beta \in (-1,0]\;,\quad
\delta\in(0,1)\;,\\
&\hat\beta\eqdef \beta+2(1-\delta) \in (-1/2,0]\;,\quad
\text{and}
\quad
\eta+\hat\beta>-1\;.
\end{split}\tag{$\CI$}
\end{equation} 
\end{definition}

\begin{remark}\label{rem:Sourav_exponents}
The space $\init$ is essentially the same as the space of possible initial conditions
appearing in \cite[Thm.~2.9]{Sourav_flow}, provided that the exponents $\gamma_i$ appearing there
are identified with $\gamma_1 = -\eta/2$, $\gamma_2 = \delta-1-\beta/2$. In particular,
the condition $\eta+\hat\beta>-1$ in \eqref{eq:CI} 
(guaranteeing that the initial data is scaling subcritical) 
corresponds to their condition $\gamma_1 + \gamma_2 < 1/2$.
\end{remark}

\begin{proposition}[Local well-posedness of YMH flow]\label{prop:YM_flow_minus_heat}
Suppose that $(\eta,\beta,\delta)$ satisfy~\eqref{eq:CI}.
For $T>0$, let $B$ denote the Banach space of functions
\begin{equ}
R\in\CC([0,T],\CC^{\hat\beta}(\T^3,E))
\end{equ}
for which
\begin{equ}
|R|_B \eqdef \sup_{t\in(0,T)} |R_t|_{\CC^{\hat\beta}} + t^{-\frac{\hat\beta}{2}}|R_t|_\infty + t^{\frac12-\frac{\hat\beta}{2}}|R_t|_{\CC^1} < \infty\;.
\end{equ}
Then for all $X=(A,\Phi)\in\init$
and $T\leq \Poly (\Theta(X)^{-1})$, there exists a unique function $\CR(X)\in B$ such that
\begin{equ}
\ymhflow(X) \eqdef (a,\phi)\eqdef \CR(X)+\CP X \colon (0,T]\to\CC^\infty(\T^3,E)
\end{equ}
solves the YMH  flow (with DeTurck term)
\begin{equs}\label{eq:ymh_flow_deturck}
\partial_t a =& -\mrd_a^* F_a - \mrd_a \mrd^* a -\mathbf{B}(\mrd_{a}\phi\otimes\phi)\;,\\
\partial_t\phi  =& - \mrd_{a}^* \mrd_a \phi + (\mrd^*a)\phi\;,
\end{equs}
with initial condition $X$ in the sense that
$\lim_{t\to 0} |\CR_t(X)|_{\CC^{\hat\beta}} = 0$.
Furthermore,
\begin{equ}
|\CR(X)|_{B} \lesssim \fancynorm{X}_{\beta,\delta}+1\;,
\end{equ}
where the proportionality constant depends only on $\eta,\beta,\delta$,
and the map $\init\ni X\mapsto \CR(X)\in B$ is locally Lipschitz continuous.
\end{proposition}

\begin{remark}\label{rem:ymh_flow_deturck}
In coordinates,~\eqref{eq:ymh_flow_deturck}  reads
\begin{equs}
\partial_t a_i
&= \Delta a_i  + [a_j,2\partial_j a_i - \partial_i a_j + [a_j,a_i]]
-\mathbf{B} ( (\partial_{i} \phi + a_{i}\phi)\otimes \phi)\;,\\
\partial_t\phi  &= 
\Delta \phi + 2 a_{j} \partial_{j}\phi + a_{j}^{2}\phi \;,
\end{equs}
with implicit summation over $j$. We will use the flow induced by this equation to define our space of ``gauge orbits'' in Section~\ref{subsec:gauge_orbits}.
For this purpose, we could equally have used
\textit{any} regularising and gauge covariant flow with nonlinearities of the same order,
such as the deterministic analogue of~\eqref{eq:SYMH}.
\end{remark}

\begin{proof}
For $X\in\init$, 
consider 
any map $\CM^X\colon B \to B$ of the form
\begin{equ}
\CM^X_t(R)
\eqdef \int_0^t \CP_{t-s}\Big((\CP_s X +R_s)\partial (\CP_s X + R_s) + P(R_s,\CP_s X)\Big) \mrd s\;,
\end{equ}
where $P\colon E\times E\to  E$ is a polynomial of degree at most $3$ such that $P(0)=0$.
Observe that
$|\CP_s X|_{\infty} \lesssim s^{\frac\eta2}|X|_{\CC^\eta}$ and
$| \partial \CP_sX|_{\infty}\lesssim s^{\frac\eta2-\frac12}|X|_{\CC^\eta}$.
Furthermore, since $\frac{\beta-\hat\beta}2=-(1-\delta)$,
\begin{equ}[e:PN-bound]
\int_0^t|\CP_{t-s}(\CN_s(X))|_{\CC^{\hat\beta}}\mrd s\lesssim \int_0^t(t-s)^{-(1-\delta)} s^{-\delta}\fancynorm{X}_{\beta,\delta} \mrd s \asymp \fancynorm{X}_{\beta,\delta}\;,
\end{equ}
and the same bound holds with $|\cdot|_{\CC^{\hat\beta}}$ replaced by $|\cdot|_\infty$ and $|\cdot|_{\CC^1}$, $(t-s)^{-(1-\delta)}$ replaced by $(t-s)^{\frac\beta2}$ and $(t-s)^{-\frac12 +\frac\beta2}$,
and the final $\fancynorm{X}_{\beta,\delta}$ replaced by $t^{\frac{\hat\beta}2}\fancynorm{X}_{\beta,\delta}$ and $t^{-\frac12+\frac{\hat\beta}2}\fancynorm{X}_{\beta,\delta}$ respectively.

It readily follows that for $\kappa=\frac12\min\{\eta+\hat\beta+1,2\hat\beta+1,3\eta+2\}$ and all $R\in B$
\begin{equ}
|\CM^X(R)|_B \lesssim \fancynorm{X}_{\beta,\delta}+T^{\kappa}(|X|^3_{\CC^\eta}+|X|_{\CC^\eta}+|R|^3_B+|R|_B)
\end{equ}
and
for $\bar X\in\init$, $\bar R \in B$, and denoting $Q(x)\eqdef x^2+x$,
\begin{equs}
|\CM^X(R)-\CM^{\bar X}(\bar R)|_B
&\lesssim \fancynorm{X;\bar X}_{\beta,\delta} + T^\kappa|X-\bar X|_{\CC^\eta} Q(|X|_{\CC^\eta}+|\bar X|_{\CC^\eta})
\\
&\qquad+ T^{\kappa}|R-\bar R|_B Q(|X|_{\CC^\eta} + |\bar X|_{\CC^\eta} + |R|_B + |\bar R|_B)\;.
\end{equs}
(The terms $\CP_s X \partial R_s$, $R_s \partial\CP_sX$ account for the condition 
$\eta+\hat\beta>-1$,
$R_s \partial R_s$ accounts for $\hat\beta> - \frac12$,
and $(\CP_sX)^3$ accounts for $\eta>-\frac23$.)
It follows that for $T\leq \Poly (\Theta(X)^{-1})$ sufficiently small,
$\CM^X$ stabilises the ball in $B$ of radius $\asymp\fancynorm{X}_{\beta,\delta}+1$, is a contraction on this ball,
and the unique fixed point of $\CM^X$ is a locally Lipschitz function of $X\in\init$.
Moreover, since $\init$ is the closure of smooth functions,
\begin{equ}
\lim_{t\to 0} t^{\delta}|\CN_t(X)|_{\CC^\beta} = 0\;,
\end{equ}
and thus $\lim_{t\to 0} |\CM^X_t(R)|_{\CC^{\hat\beta}} = 0$ for all $R\in B$.
It remains to observe that the desired $\CR(X)$ is the fixed point of a map $\CM^X$ of the above form.
\end{proof}

\begin{definition}\label{def:ymhflow}
Let $(\eta,\beta,\delta)$ satisfy~\eqref{eq:CI}.
For $X\in\init$,
let $\ymhflow(X)\in\CC^\infty((0,T_X)\times\T^3,E)$ denote the solution to~\eqref{eq:ymh_flow_deturck}
with initial condition $X$  where $T_X$ denotes the maximal existence time of the solution in $\CC^\infty(\T^3,E)$.
We write $\ymhflow_t(X)=(\ymhflow_{A,t}(X),\ymhflow_{\Phi,t}(X))$ for its decomposition in
$\CC^\infty(\T^3,E)\simeq \Omega\CC^\infty\oplus \CC^\infty(\T^3,\higgsvec)$.

Define $\tilde\ymhflow(X)\colon(0,\infty)\to\mfO^\infty$ by
\begin{equ}
\tilde\ymhflow_t(X)
\eqdef
\begin{cases}
[\ymhflow_t(X)] \quad&\textnormal{ if } t\in(0,T_X/2]\;,
\\
[\flow_{t-T_X/2}(\ymhflow_{T_X/2})] \quad&\textnormal{ if } t>T_X/2\;,
\end{cases}
\end{equ}
where $\flow$
is the flow of the YMH equation without the DeTurck and $\Phi^3$ terms, see Definition~\ref{def:flow_no_deturck}.

We define an equivalent relation $\sim$ on $\init$
by $X\sim Y \Leftrightarrow \tilde\ymhflow(X)=\tilde\ymhflow(Y)$.
Let $[X]$ denote the equivalence class of $X\in\init$.
\end{definition}

\begin{remark}
We will see below in Proposition~\ref{prop:back_unique}
that $\sim$ extends the notion of gauge equivalence for smooth functions.
\end{remark}

\begin{remark}
By Lemma~\ref{lem:deturck_to_no_deturck},
$\ymhflow_t(X)\in\tilde\ymhflow_t(X)$
for all $t\in (0,T_X)$.
This allows us interpret $\tilde\ymhflow$ as an extension of $\ymhflow$ to all times modulo gauge equivalence.
\end{remark}

\begin{remark}\label{rem:Sourav_state}
In~\cite{Sourav_state}, the authors introduce a topological space $\mcX$ as a candidate state space of the 3D Yang--Mills measure.
Roughly speaking, $\mcX$ contains all functions $X\colon (0,\infty)\to \mfO^\infty$ which solve the YM flow $\ymhflow$ 
with the DeTurck--Zwanziger term modulo gauge equivalence (with no restrictions at $t=0$ although some 
restrictions could easily be added).
They furthermore show in~\cite{Sourav_flow} that the Gaussian free field has a natural representative 
as a probability measure on $\mcX$. Note that we will use a space we call $\mcX$ in Section~\ref{subsec:final_state_space} which
is unrelated to the space $\mcX$ in~\cite{Sourav_state}.

The space $\init/{\sim}$ (and thus also the smaller space $\mfO = \state/{\sim}$ that we ultimately 
work with) embeds canonically into $\mcX$.
The main difference is that $\init$ and $\state$ are actual spaces of distributional connections 
that are not themselves defined in terms of the flow $\ymhflow$, although the analytic conditions we 
impose control how fast $\ymhflow_t$ can blow up at time $t=0$. One advantage is that this yields a large 
number of continuous operations on the state space $\mfO$.
Furthermore, we later show that the SHE takes values in $\CC(\R_+,\state)$, thus effectively recovering the main result of~\cite{Sourav_flow}; 
see Remark~\ref{rem:Sourav_flow}.

A result in~\cite{Sourav_state} that we do not consider here is a tightness criterion for $\mcX$
although it is possible in principle to formulate such criteria for $\init$ and $\state$ using 
the compact embedding results in Section~\ref{subsec:embeddings}.
\end{remark}

\subsection{Backwards uniqueness on gauge orbits for the YMH flow}
\label{subsec:backwards_unique}

In this subsection we show a backwards uniqueness property for the YMH flow which will be a key step in defining our space of ``gauge orbits''.
For $\rho\in(\frac12,1]$, recall from~\cite[Sec.~3]{CCHS2d} the spaces $\Omega_{\gr\rho}$ and $\Omega^1_{\gr\rho}$, the closure of $\mfg$-valued $1$-forms  in $\Omega_{\gr\rho}$.
More precisely, $\Omega_{\gr\rho}$ is equipped with a norm $|\cdot|_{\gr\rho}$, defined in the same way as in \cite[Def.~3.7]{CCHS2d} except that $\T^2$ is replaced by $\T^3$. In Section~\ref{def:heatgr} we will introduce a generalisation of it,
but for this subsection we only need $\Omega^1_{\gr\rho}$.
Here we only recall that 
$\Omega^1_{\gr\rho}$ is embedded into $\CC^{\rho-1}$, see e.g.
\eqref{eq:CC_gr_bound} below.
For\label{omega rho page ref}
\begin{equ}
X=(A,\Phi),\; \bar X=(\bar A,\bar \Phi)\in \Omega^\rho\eqdef \Omega^1_{\gr\rho}\times \CC^{\rho-1}(\T^3,\higgsvec)\;,
\end{equ}
we say that $X\sim\bar X$ in $\Omega^\rho$ if there exists $g\in\mfG^\rho$
such that $\bar X=X^g \eqdef (\Ad_g A - (\mrd g)g^{-1},g\Phi)$.
(Here, we recall the action of $\mfG^\rho$ on $\Omega^1_{\gr\rho}$ in \cite[Section 3.4, Definition 3.26]{CCHS2d};
in particular the formal expression $A^g=\Ad_g A - (\mrd g)g^{-1}$ is defined there as an element of 
$\Omega^1_{\gr\rho}$.)
Remark that $\sim$ in $\Omega^\rho$ extends the usual notion of gauge equivalence of smooth functions.
Remark also that $\Omega^\rho$ embeds into $\init$ for some parameters satisfying~\eqref{eq:CI},
so that the YMH flow $\ymhflow$ is well-defined on $\Omega^\rho$.
The following is the main result of this subsection.

\begin{proposition}\label{prop:back_unique}
Let $\rho\in(\frac12,1]$ and $X, Y \in \Omega^\rho$.
The following statements are equivalent.
\begin{enumerate}[label=(\roman*)]
\item\label{pt:X_Y_sim} $X\sim Y$ in $\Omega^\rho$.
\item\label{pt:all_t} $\tilde\ymhflow(X)=\tilde\ymhflow(Y)$, i.e.\ $X\sim Y$ in $\init$.
\item\label{pt:one_t} $\tilde\ymhflow_t(X) = \tilde\ymhflow_t(Y)$ for some $t >0$.
\end{enumerate}
\end{proposition}

\begin{remark}
In the setting of Proposition~\ref{prop:back_unique},
it also holds that if $\ymhflow_t(X)=\ymhflow_t(Y)$ for some $t \in (0,T_X\wedge T_Y)$, then $X=Y$.
This follows from Lemma~\ref{lem:analyticBasic} below or from classical backwards uniqueness statements for parabolic
PDEs, e.g.\ \cite[Thm.~2.2]{Chen98}.
\end{remark}

The proof of Proposition~\ref{prop:back_unique},
which is based on analytic continuation and the YMH flow without DeTurck term,
is given at the end of this subsection.
Denote by $\CK_0$ the real Banach space of pairs $(A,\Phi)\in\init$ which are
continuously differentiable and
write $\CK$ for the complexification of $\CK_0$. Note that even if $\higgsvec$ happens to have a complex structure already, we view it as real in this construction, so that the $\Phi$-component
of an element of $\CK$ takes values in $\higgsvec \oplus \higgsvec$ endowed with its canonical complex structure. A similar remark applies to $\mfg$.
We also write $\CL_0 \subset \CK_0$ and $\CL \subset \CK$ for the 
subspaces consisting of \textit{twice} continuously differentiable functions,
endowed with the corresponding norms.

For the remainder of the subsection, fix $\alpha<\pi/2$ and $T>0$ and define the sector $S_\alpha = \{t \in \C \,:\, \Re(t) \in [0,T] \,\&\, |{\arg t}| \le \alpha\}$.
Write $\CB \CK$ for the set of holomorphic functions
$X \colon S_\alpha \to \CK$ such that 
\begin{equ}
\|X\|_{\CB \CK} \eqdef \sup_{t \in S_\alpha} \|X(t)\|_{\CK} < \infty\;,
\end{equ}
and similarly for $\CB \CL$.

\begin{lemma}\label{lem:analyticBasic}
Let $X_0\in \CL_0$.
Then, for $T$ sufficiently small,
$\ymhflow(X_0)\colon[0,T]\to \CL_0$ admits a (necessarily unique) analytic continuation to an element of $\CB \CL$.
\end{lemma}

\begin{proof}
Note that $X\eqdef\ymhflow(X_0)$ solves
\begin{equ}
\partial_t X = \Delta X + F(X)\;,
\end{equ}
for a holomorphic function $F \colon \CL \to \CK$ (see~\cite{Mujica} for a definition) 
that is bounded
on bounded sets. This is equivalent to the integral equation
\begin{equ}[e:FPhol]
X_t = e^{t\Delta} X_0 + t \int_0^1 e^{tu\Delta}F(X_{t(1-u)})\mrd u\;,
\end{equ}
which we now interpret as a fixed point problem on a space of holomorphic functions 
on $S_\alpha$.
Recall that $\|e^{t\Delta} u\|_{\CL} \lesssim |t|^{-1/2}\|u\|_{\CK}$ for all $t \in S_\alpha$ as can easily be seen from the explicit expression of the
heat kernel. It immediately follows that the operator 
$\CM \colon Y \mapsto (t \mapsto t \int_0^1 e^{tu\Delta}Y_{t(1-u)}\mrd u)$ is bounded
from $\CB \CK$ into $\CB \CL$ with norm of order $T^{1/2}$, and therefore that 
the fixed point problem \eqref{e:FPhol} admits a unique solution in 
$\CB \CL$.
\end{proof}

\begin{lemma}\label{lem:analyticg}
Let $H \in \CB \CK$ and let $g$ denote the solution to
\begin{equ}[e:evolgGeneric]
g^{-1} \partial_t g = H\;,
\end{equ}
with some initial condition $g_0\in \CC^1$.
Set $h_i = g^{-1}\partial_i g$ and $U = \brho(g)$ for some 
finite-dimensional representation $\brho$ of $G$.
Then, $h_i$ and $U$ can
be extended uniquely to holomorphic functions $S_\alpha \to \CC$ and
$S_\alpha \to \CC^1$ respectively. 
\end{lemma}

\begin{proof}
It suffices to note that $h_i$ and $U$ solve pointwise the ODEs
\begin{equ}
\partial_t h_i = \partial_i H + [h_i, H]\;,\qquad
\partial_t U = U \brho(H)\;.
\end{equ}
We can interpret these as ODEs in $\CC$ and $\CC^1$ respectively, which admit solutions 
since the right-hand sides are holomorphic functions of $H$, $h$ and $U$. 
These solutions are global since the equations are linear.
\end{proof}

Recall
that $\flow(X)$ denotes the solution of the YMH flow without DeTurck term and with initial condition $X$.

\begin{lemma}\label{lem:YMH}
For $X \in \CC^\infty(\T^3,E)$, $\flow(X)$ is real 
analytic on $\R_+$ with values in $\CC$.
\end{lemma}

\begin{proof}
Denoting $\ymhflow(X)= (a,\phi)$,
by Lemma~\ref{lem:deturck_to_no_deturck},
if $g$ solves the ODE
\begin{equ}
g^{-1} \partial_t g = -\mrd^* a\;,\quad g_0 = \id\;,
\end{equ}
then $\flow(X) = \ymhflow(X)^{g}= U(\ymhflow(X) - h)$ where $h$ (understood as $(h_1,h_2,h_3,0)$) and $U$ are defined from $g$ as in Lemma~\ref{lem:analyticg} but with with $\brho$ replaced by $\Ad^3\oplus\brho$ on $\mfg^3 \oplus \higgsvec$.
Lemma~\ref{lem:analyticBasic} implies that $\ymhflow(X)$ is real analytic with values
in $\CC^2$ for short times,
and Lemma~\ref{lem:analyticg} implies that $h$ and $U$ are real analytic
with values in $\CC$ and $\CC^1$ respectively.
It follows that $\flow(X)$ is real analytic for short times (with a lower bound on these times depending
only on its $\CC^2$ norm)
and thus on $\R_+$ since
$\flow(X)$ exists globally by Lemma~\ref{lem:global_sol_no_deturck}.
\end{proof}

\begin{lemma}\label{lem:gauge_orbits_closed_rho}
Let $\rho\in(\frac12,1]$.
The set $\{(X,Y)\in\Omega^\rho\times \Omega^\rho\,:\,X\sim Y \textnormal{ in } \Omega^\rho\}$ is closed in $\Omega^\rho\times \Omega^\rho$.
\end{lemma}

\begin{proof}
Suppose $(X_n,Y_n)\to (X,Y)$ in $\Omega^\rho\times \Omega^\rho$
with $X_n^{g_n} = Y_n$.
Consider $\bar\rho\in(\frac12,\rho)$.
By the estimate $|g|_{\Hol\rho}\lesssim |A|_{\gr\rho}+|A^g|_{\gr\rho}$ (see~\cite[Prop.~3.35]{CCHS2d})
and
the compact embedding $\mfG^\rho\hookrightarrow \mfG^{\bar\rho}$,
it follows that $g_n \to g$ in $\mfG^{\bar\rho}$ along a subsequence with $g\in\mfG^\rho$.
Furthermore, by continuity of the maps
\begin{equs}
\mfG^{\bar\rho}\times \Omega_{\gr\rho}\ni (g,A)
&\mapsto
A^g \in \Omega_{\gr{\bar\rho}}\;,
\\
\mfG^{\bar\rho} \times \CC^{\rho-1}(\T^3,\higgsvec) \ni (g,\Phi)
&\mapsto
g\Phi \in\CC^{\rho-1}(\T^3,\higgsvec)\;,
\end{equs}
(see~\cite[Lem.~3.30,~3.32]{CCHS2d} for the first; the second is obvious),
we have $X^g - Y_n \to 0$ in $\Omega^{\bar\rho}$, which implies $X^g = Y$.
\end{proof}

\begin{proof}[of Proposition~\ref{prop:back_unique}]
The implication~\ref{pt:all_t}$\Rightarrow$\ref{pt:one_t} is trivial.
To show~\ref{pt:X_Y_sim}$\Rightarrow$\ref{pt:all_t},
suppose $X^g=Y$ for some $g\in\mfG^\rho$.
Then $\ymhflow_t(X)^{h_t}=\ymhflow_t(Y)$ for all sufficiently small $t>0$, where $h$ solves a parabolic PDE with initial condition $g$
(see~\cite[Sec.~2.2]{CCHS2d} or Remark~\ref{rem:H_meaning} below).
Hence \ref{pt:X_Y_sim}$\Rightarrow$\ref{pt:all_t}.
For the final implication \ref{pt:one_t}$\Rightarrow$\ref{pt:X_Y_sim}, suppose first that $X, Y\in\CC^\infty(\T^3,E)$.
By Lemma~\ref{lem:deturck_to_no_deturck}, $\flow_t(X)\sim \ymhflow_t(X)$, and likewise for $Y$, hence $\flow_t(X)\sim\flow_t(Y)$.
By Lemma~\ref{lem:gauge_for_flow}, there exists $g\in\mfG^\infty$
such that $\flow_s(X)^g = \flow_s(Y)$ for all $s \geq t$.
By Lemma~\ref{lem:YMH} and uniqueness of analytic continuations, it follows that 
$\flow_s(X)^g = \flow_s(Y)$ for all $s \geq 0$,
in particular $X\sim Y$.
The case of general $(X,Y)\in\Omega^\rho$ follows by Lemma~\ref{lem:gauge_orbits_closed_rho}
and the fact that $\lim_{s\to0}\ymhflow_s(X) = X$
in $\Omega^{\bar\rho}$
for any $\bar\rho\in(\frac12,\rho)$.
Indeed, if 
$\tilde\ymhflow_t(X) = \tilde\ymhflow_t(Y)$ for some $t >0$,
since $\ymhflow_s(X)$ and $\ymhflow_s(Y)$ are $\CC^\infty$
for all $s\in (0,T_X \wedge T_Y \wedge t)$,
one has $\ymhflow_s(X)\sim \ymhflow_s(Y)$ by the  $\CC^\infty$ case proved above,
and in particular $(\ymhflow_s(X), \ymhflow_s(Y))$
is in the set defined in Lemma~\ref{lem:gauge_orbits_closed_rho}.
One then has  $X\sim Y$ in $\Omega^\rho$ by sending $s\to 0$.
Hence~\ref{pt:one_t}$\Rightarrow$\ref{pt:X_Y_sim}.
\end{proof}

\subsection{Final state space}
\label{subsec:final_state_space}

In this subsection, we introduce the ``second half'' of our state space $(\state,\Sigma)$.
We do this by introducing an additional norm $\heatgr{\cdot}_{\alpha,\theta}$.
The role of this norm will become clear in Section~\ref{subsec:est_gauge_trans} where it will be used to control the H{\"o}lder norm of any gauge transformation $g$ in terms of $X$ and $X^g$ (Theorem~\ref{thm:g_bound} below).

Let $\mcX$\label{mcX page ref} denote the set of oriented line segments in $\T^3$ of length at most $\frac14$,
i.e.
\begin{equ}\label{eq:def_mcX}
\mcX \eqdef \T^3 \times B_{1/4}
\end{equ}
where $B_r \eqdef \{v \in \R^3 \,:\, |v| \leq r\}$, (the first coordinate of $\mcX$ is the initial point, the second coordinate is the direction).
We say that $\ell = (x,v)$ and $\bar \ell = (\bar x, \bar v)$ are joinable if $\bar x = x+v$,
there exists $c \in \R$ such that $\bar v = c v$, and $|v+\bar v| \le 1/4$. In that case, we define
$\ell \sqcup \bar \ell = (x, v+\bar v)$.

Let $\Omega$ denote the space of additive $\mfg$-valued functions on $\mcX$ (see~\cite[Def.~3.1]{CCHS2d} for the definition in 2D \dash precisely the same definition applies here with $\T^2$ and $\R^2$ replaced by $\T^3$ and $\R^3$).
Observe that every $A\in\Omega\CB$,
where $\CB$ denotes the space of bounded measurable $\mfg$-valued functions,
canonically defines an element of $\Omega$ by
\begin{equ}[eq:A_ell_def]
A(x,v) \eqdef \int_0^1 \sum_{i=1}^3 v_i A_i(x+tv)\,\mrd t\;.
\end{equ}

\begin{definition}\label{def:heatgr}
For $A\in\Omega$, $\alpha\in(0,1]$, and $t>0$, define
\begin{equ}
|A|_{\gr\alpha;<t} \eqdef \sup_{|\ell|<t} \frac{|A(\ell)|}{|\ell|^\alpha}\;,
\end{equ}
where the supremum is taken over all lines $\ell\in \mcX$ of length less than $t$.
For $A\in\Omega\CD'$ and
$\theta\in \R$, define the (extended) norm
\begin{equ}
\heatgr{A}_{\alpha,\theta} \eqdef \sup_{t\in(0,1)} |\CP_tA|_{\gr\alpha;<t^\theta}\;.
\end{equ}
For $\eta,\beta,\delta\in\R,\alpha\in(0,1],\theta\in\R$,
recalling Definition~\ref{def:init} and the notational convention $X=(A,\Phi)$,
define on $\CD'(\T^3,E)$ the (extended) metric
$\Sigma \equiv \Sigma_{\eta,\beta,\delta,\alpha,\theta}$ by
\begin{equ}[e:def-Sigma]
\Sigma(X,\bar X) \eqdef \Theta(X,\bar X) + \heatgr{A-\bar A}_{\alpha,\theta}\;.
\end{equ}
Similar to before, we use the shorthand $\Sigma(X) \eqdef \Sigma(X,0)$
and will often drop the reference to $\eta,\beta,\delta,\alpha,\theta$ in the notation $\Sigma$.
\end{definition}

\begin{definition}\label{def:state}
For $\alpha\in(0,1]$ and $\theta\in\R$, let $\state\equiv\state_{\eta,\beta,\delta,\alpha,\theta}\subset \init_{\eta,\beta,\delta}$  denote the closure of smooth functions 
under $\Sigma$.
Unless otherwise stated, we equip $\state$ with the metric $\Sigma$.
\end{definition}

Lemma~\ref{lem:heatgr_Besov_embed} below provides a basic relation between $\heatgr{\cdot}_{\alpha,\theta}$ and the H{\"o}lder--Besov spaces $|\cdot|_{\CC^\eta}$
that generalises the estimates
\begin{equ}\label{eq:CC_gr_bound}
|A|_{\CC^{\alpha-1}}\lesssim |A|_{\gr\alpha} \leq |A|_{\gr1} \asymp |A|_{L^\infty}
\end{equ}
(which hold for all $A\in\Omega\CC$ and $\alpha\in(0,1]$, see~\cite[Prop.~3.21]{Chevyrev19YM}).
We first state the following lemma.

\begin{lemma}\label{lem:local_to_global_gr}
For all $\alpha\in(0,1]$, $t\in(0,1)$, and $A\in\Omega$,
\begin{equ}
|A|_{\gr\alpha} \leq t^{-(1-\alpha)} |A|_{\gr\alpha;<t}\;.
\end{equ}
\end{lemma}
\begin{proof}
Identical to~\cite[Ex.~4.24]{FrizHairer}.
\end{proof}

The proof of the following lemma is obvious.

\begin{lemma}\label{lem:local_lower_exponent}
For all $0<\alpha\leq \beta\leq 1$, $t\in(0,1)$, and $A\in\Omega$,
\begin{equ}
|A|_{\gr\alpha;<t} \leq t^{\beta-\alpha}|A|_{\gr\beta;<t}\;.
\end{equ}
\end{lemma}

We now have the ingredients in place for the generalisation of \eqref{eq:CC_gr_bound} we just announced.

\begin{lemma}\label{lem:heatgr_Besov_embed}
For $A\in\Omega\CD'$, $\alpha\in(0,1]$, and $\theta\geq0$,
\begin{equ}
|A|_{\CC^{(1+2\theta)(\alpha-1)}} \lesssim \heatgr{A}_{\alpha,\theta} \lesssim |A|_{\CC^{2\theta(\alpha-1)}}\;,
\end{equ}
where the proportionality constants depend only on $\alpha,\theta$.
\end{lemma}

\begin{proof}
Denoting $\eta\eqdef(1+2\theta)(\alpha-1)$,
\begin{equs}
|A|_{\mcC^{\eta}}
&\asymp \sup_{t\in(0,1)}t^{\frac{\alpha-1-\eta}{2}}|\CP_t A|_{\mcC^{\alpha-1}}
\lesssim
\sup_{t\in(0,1)}t^{\frac{\alpha-1-\eta}{2}}|\CP_t A|_{\gr{\alpha}}
\\
&\leq \sup_{t\in(0,1)}t^{\frac{\alpha-1-\eta}{2}}t^{-\theta(1-\alpha)}|\CP_t A|_{\gr{\alpha};< t^\theta}
= \sup_{t\in(0,1)}|\CP_t A|_{\gr{\alpha};< t^\theta}
\\
&\leq \sup_{t\in(0,1)} t^{\theta(1-\alpha)}|\CP_t A|_{\gr1} \asymp |A|_{\CC^{-2\theta(1-\alpha)}}\;,
\end{equs}
where we used~\eqref{eq:CC_gr_bound} in the first line, Lemma~\ref{lem:local_to_global_gr} in the second line,
and Lemma~\ref{lem:local_lower_exponent}
in the third line.
\end{proof}

\subsection{Gauge transformations}
\label{sec:Gauge transformations}

Throughout this subsection, let us fix 
\begin{equ}[e:paramG]
\rho\in(1/2,1],
\quad \eta\in (-\rho,\rho-1],
\quad \beta\in(-\rho,0],
\quad\mbox{and}\quad \delta\in\R \;.
\end{equ}
Recall that $\CC^{0,\eta}$ denotes the closure of smooth functions in $\CC^\eta$.
Since $\rho+\eta>0$, $2\rho-1>0$, and $\eta\leq \rho-1$,
the group $\mfG^{0,\rho}$ (resp.\ $\mfG^{\rho}$)
acts continuously on $\CC^{0,\eta}(\T^3,E)$ (resp.\ $\CC^{\eta}(\T^3,E)$)
via $(A,\Phi) \mapsto (A,\Phi)^g \eqdef (\Ad_g A - (\mrd g) g^{-1},g\Phi)$.
The following result shows that, under further conditions, the action of $\mfG^{0,\rho}$ extends to $\init$.

\begin{definition}\label{def:admissible}
We say that $(\rho,\eta,\beta,\delta)\in\R^4$ satisfies condition~\eqref{eq:CGI} if
\begin{equation}\label{eq:CGI}
\begin{split}
&\rho\in [3/2-\delta, 1] \;,\quad
\eta\in [2-2\delta-\rho,\rho-1]\;,
\\
&\beta\in(-\rho,0]\;,\quad
\text{and}
\quad
\delta \in [1/2,1) \;.
\end{split}\tag{$\CG\CI$}
\end{equation}
We say that $(\rho,\eta,\beta,\delta,\alpha,\theta) \in\R^6$ satisfies condition~\eqref{eq:CGS} if
\begin{equation}\label{eq:CGS}
\begin{split}
&(\rho,\eta,\beta,\delta) \text{ satisfies condition~\eqref{eq:CGI}, and}
\\
&\alpha\in(0,1]\;,\quad
\theta\geq0\;,
\quad
\rho \geq 1+2\theta(\alpha-1)\;,
\quad \rho+(1+2\theta)(\alpha-1)>0\;.
\end{split}\tag{$\CG\CS$}
\end{equation}
\end{definition}

\begin{remark}\label{rem:GI-equiv}
The conditions \eqref{e:paramG} stated at the start of this subsection together with
$\eta \geq 2-2\delta-\rho$ and $\delta<1$
are equivalent to~\eqref{eq:CGI}. 
\end{remark}

\begin{proposition}\label{prop:group_action}
Suppose $(\rho,\eta,\beta,\delta)$ satisfies~\eqref{eq:CGI}.
Then $(g,X) \mapsto X^g$ defines a continuous left group action $\mfG^{0,\rho}\times \init \to \init$
which is uniformly continuous on every ball in $\mfG^{0,\rho}\times \init$.
If in addition $(\eta,\beta,\delta)$ satisfies~\eqref{eq:CI},
then $X^g \sim X$ for all $(g,X)\in\mfG^{0,\rho}\times \init$.
Finally, if $(\rho,\eta,\beta,\delta,\alpha,\theta)$ satisfies~\eqref{eq:CGS},
then $(g,X) \mapsto X^g$ defines a continuous left group action $\mfG^{0,\rho}\times \state \to \state$
which is uniformly continuous on every ball in $\mfG^{0,\rho}\times \state$.
\end{proposition}

We break up the proof into several lemmas.

\begin{lemma}\label{lem:commutators_Schwartz}
Let $g\in \CC^\rho(\T^3)$, $h\in\CC^\eta(\T^3)$,
which we identify with periodic distributions on $\R^3$.
For every Schwartz function $\phi\in\CS(\R^3)$
\begin{equ}\label{eq:phi_Schwartz}
\sup_{x\in\R^3} \sup_{\lambda\in(0,1]}
\lambda^{-\eta-\rho}|\scal{gh,\phi^\lambda_x} - g(x)\scal{h,\phi^\lambda_x}| \lesssim |g|_{\Hol\rho}|h|_{\CC^\eta}\;,
\end{equ}
where 
$\phi^\lambda_x = \lambda^{-3}\phi(\frac{\cdot-x}{\lambda})$
and the proportionality constant depends only on $\rho,\eta,\phi$.
\end{lemma}

\begin{proof}
The reconstruction theorem, using a modification of the proof of~\cite[Prop.~4.14]{Hairer14},
implies that
\begin{equ}\label{eq:phi_B1}
\sup_{x\in\R^3}\sup_{\lambda\in(0,1]}\sup_{\phi\in\CB^1} \lambda^{-\rho-\eta} |\scal{gh,\phi^\lambda_x}-g(x)\scal{h,\phi^\lambda_x}|\lesssim |g|_{\Hol\rho}|h|_{\CC^\eta}\;.
\end{equ}
To extend this to a Schwartz function $\phi$,
we can decompose $\phi=\sum_{k\in\Z^3}\phi_k(\cdot-k)$ where $\phi_k$ has support in a ball of radius $\sqrt d$ centered at $0$ and 
$|\phi_k|_{\CC^1}\lesssim |k|^{-d-1-\rho}$.
Then $\phi^\lambda_x=\sum_{k\in\Z^3}(\phi_k)_{x+\lambda k}^\lambda$
and~\eqref{eq:phi_B1} implies
\begin{equ}
|\scal{gh,(\phi_k)^\lambda_{x+\lambda k}} - g(x+\lambda k)\scal{h,(\phi_k)_{x+\lambda k}^\lambda}|\lesssim k^{-d-1-\rho}\lambda^{\rho+\eta}|g|_{\Hol\rho}|h|_{\CC^\eta}\;.
\end{equ}
Furthermore
\begin{equ}
|(g(x)-g(x+\lambda k))\scal{h,(\phi_k)_{x+\lambda k}^\lambda}|\lesssim |\lambda k|^{\rho}|g|_{\Hol\rho} k^{-d-1-\rho} \lambda^\eta|h|_{\CC^\eta}\;,
\end{equ}
from which~\eqref{eq:phi_Schwartz} follows by taking the sum over $k\in\Z^3$.
\end{proof}

\begin{corollary}\label{cor:commutator_heat}
Consider $g,h$ as in Lemma~\ref{lem:commutators_Schwartz}.
Then for all $t\in(0,1)$
\begin{equ}
|\CP_t(gh)-g\CP_t h|_\infty
\lesssim
t^{\frac{\eta+\rho}{2}}|g|_{\Hol\rho}|h|_{\CC^\eta}
\end{equ}
and
\begin{equ}
|\nabla \CP_t (gh)-g \nabla \CP_t h|_\infty
\lesssim
t^{\frac{\eta+\rho-1}{2}} |g|_{\Hol\rho}|X|_{\CC^\eta}\;,
\end{equ}
where the proportionality constants depend only on $\rho,\eta$.
\end{corollary}

\begin{proof}
Apply Lemma~\ref{lem:commutators_Schwartz} with $\phi= \CP_1$ and $\phi=\nabla \CP_1$, where we interpret $\CP_t$ as the heat kernel at time $t>0$.
\end{proof}

For $g\in G$ and $(a,\phi)\in E$, let us denote $g(a,\phi) \eqdef (\Ad_g a, g\phi)\in E$.
We extend this action to $G\times E^3\to E^3$ diagonally.
In particular, $gX$ is well-defined as an element of $\CC^{\eta}(\T^3,E)$ (resp.\ $\CC^{\eta}(\T^3,E^3)$)
provided $g\in\mfG^{\rho}$
and $X\in \CC^{\eta}(\T^3,E)$
(resp.\ $X\in \CC^{\eta}(\T^3,E^3)$). 
Similarly for the spaces $\CC^{0,\eta}$ and $\mfG^{0,\rho}$. 
Denote further
\begin{equ}
g\CN_t(X) \eqdef (g\otimes g)\CN_t(X) = (g\CP_t X )\otimes (g \nabla \CP_t X)\;.
\end{equ}

\begin{lemma}\label{lem:commutators}
For $X,Y\in \CC^\eta(\T^3,E)$, $g\in\mfG^\rho$, and $t\in(0,1)$
\begin{multline*}
|g(\CN_t(X) - \CN_t(Y)) - (\CN_t(gX) - \CN_t(gY))|_{\infty}
\\
\lesssim t^{\eta +\frac{\rho-1}2}
|g|_{\CC^\rho}|g|_{\Hol\rho}
|X-Y|_{\CC^\eta}(|X|_{\CC^\eta}+|Y|_{\CC^\eta})\;,
\end{multline*}
where the proportionality constant depends only on $\eta,\rho$.
\end{lemma}

\begin{proof}
We use the shorthand $a_i,b_i$ for $i=1,2$ to denote $g\CP_t X, g\CP_t Y$ if $i=1$ and $g\nabla \CP_tX, g\nabla \CP_t Y$ if $i=2$, and $\bar a_i,\bar b_i$ to denote the same symbols except with the $g$ inside the $\CP_t$, e.g.\ $\bar a_1 = \CP_tgX,\bar b_2 = \nabla \CP_t gY$.
With this notation, and henceforth dropping $\otimes$,
the quantity we aim to bound is
$|(a_1a_2-b_1b_2)-(\bar a_1 \bar a_2 - \bar b_1 \bar b_2)|_\infty$.

By Corollary~\ref{cor:commutator_heat},
\begin{equ}
|a_1-\bar a_1|_\infty = |\CP_t g X - g \CP_t X|_\infty
\lesssim
t^{\frac{\eta+\rho}{2}}|g|_{\Hol\rho}|X|_{\CC^\eta}
\end{equ}
and
\begin{equ}
|a_2-\bar a_2|_\infty = |\nabla \CP_t g X - g \nabla \CP_t X|_\infty 
\lesssim
t^{\frac{\eta+\rho-1}{2}} |g|_{\Hol\rho}|X|_{\CC^\eta}\;,
\end{equ}
with similar inequalities for the ``$b$'' terms.
It follows that 
\begin{equs}
|(b_1-\bar b_1)(a_2 - b_2)|_\infty
&\lesssim
t^{\eta + \frac{\rho-1}2}|g|_{\Hol\rho}|X-Y|_{\CC^\eta}|Y|_{\CC^\eta}
\\
|(\bar a_1-\bar b_1)(b_2 - \bar b_2)|_\infty
&\lesssim
t^{\eta + \frac{\rho-1}2}|g|_{\CC^\rho}|g|_{\Hol\rho}|X-Y|_{\CC^\eta}|Y|_{\CC^\eta}
\end{equs}
where in the first line we used 
 the fact that multiplication by $g$ on $E$ preserves the $L^\infty$ norm,
 and in the second line Young's theorem for $\CC^\rho\times \CC^\eta$ by $\eta>-\rho$.
 In the same way
\begin{equs}
|((a_1-\bar a_1)-(b_1-\bar b_1))a_2|_\infty
&\lesssim
t^{\eta + \frac{\rho-1}2}|g|_{\CC^\rho}|X-Y|_{\CC^\eta}|X|_{\CC^\eta}\;,
\\
|\bar a_1((a_2-\bar a_2)-(b_2-\bar b_2))|_\infty
&\lesssim
t^{\eta + \frac{\rho-1}2}|g|_{\CC^\rho}|g|_{\Hol\rho}|X-Y|_{\CC^\eta}|X|_{\CC^\eta}\;.
\end{equs}
It remains to observe that $(a_1a_2-b_1b_2)-(\bar a_1 \bar a_2 - \bar b_1 \bar b_2)$
is the sum of the previous four terms inside the norms $|\cdot|_\infty$.
\end{proof}

\begin{lemma}\label{lem:gX_gY_bound}
Suppose $\eta + \frac{\rho-1}2 \geq -\delta$.
Then for $X,Y\in\CC^\eta(\T^3,E)$ and $g\in\mfG^\rho$
\begin{equ}
\fancynorm{gX;gY}_{\beta,\delta} 
\lesssim |g|_{\CC^\rho}^2\fancynorm{X;Y}_{\beta,\delta} +
|g|_{\CC^\rho}|g|_{\Hol\rho}|X-Y|_{\CC^\eta}(|X|_{\CC^\eta} + |Y|_{\CC^\eta})\;,
\end{equ}
where the proportionality constant depends only on $\eta,\beta,\rho,\delta$.
\end{lemma}

\begin{proof}
For
$Z \in \CC^\beta(\T^3,E\otimes E^3)$ and $g\in\mfG^\rho$,
\begin{equ}\label{eq:g_diagonal}
|(g\otimes g) Z|_{\CC^\beta} \lesssim |g|_{\CC^\rho}^2|Z|_{\CC^\beta}\;.
\end{equ}
The conclusion now follows by applying~\eqref{eq:g_diagonal} to $Z=\CN_t(X)-\CN_t(Y)$
together with Lemma~\ref{lem:commutators}.
\end{proof}

\begin{lemma}\label{lem:X_gX_bound}
Suppose $\eta + \frac{\rho-1}2 \geq -\delta$.
For $X\in\CC^\eta(\T^3,E)$ and $g\in\mfG^\rho$,
\begin{equ}
\fancynorm{X;gX}_{\beta,\delta} \lesssim |g-1|_{\CC^\rho}|g|_{\CC^\rho}\fancynorm{X}_{\beta,\delta}
+|g|_{\CC^\rho}|g|_{\Hol\rho}|X|_{\CC^\eta}^2\;, 
\end{equ}
where the proportionality constant depends only on $\eta,\rho,\beta,\delta$.
\end{lemma}

\begin{proof}
For
$Z \in \CC^\beta(\T^3,E\otimes E^3)$ and $g\in\mfG^\rho$,
\begin{equ}\label{eq:gZ_Z_diff}
|(g\otimes g) Z-Z|_{\CC^\beta} \lesssim |g|_{\CC^\rho}|g-1|_{\CC^\rho}|Z|_{\CC^\beta}\;.
\end{equ}
The conclusion follows by combining~\eqref{eq:gZ_Z_diff} with Lemma~\ref{lem:commutators}.
\end{proof}

\begin{lemma}\label{lem:etas_to_beta=zero}
Consider $\bar\eta,\eta \leq 0$, $t\in(0,1)$, and $X,Y,\bar X,\bar Y\in\CD'(\T^3,E)$.
Then
\begin{multline*}
|\CP_t X \otimes \nabla \CP_t \bar{X} - \CP_t Y \otimes \nabla \CP_t \bar{Y} |_{\infty}
\lesssim t^{(\eta+\bar\eta-1)/2}(|X-Y|_{\CC^\eta}|\bar X|_{\CC^{\bar\eta}}
\\
+ |Y|_{\CC^{\eta}}|\bar X-\bar Y|_{\CC^{\bar\eta}})\;,
\end{multline*}
where the proportionality constant depends only on $\bar\eta$ and $\eta$.
\end{lemma}

\begin{proof}
Using $|\CP_tX|_{\infty} \lesssim t^{\eta/2}|X|_{\CC^\eta}$ and $|\nabla \CP_tX|_{\infty} \lesssim t^{(\eta+1)/2}|X|_{\CC^\eta}$, we obtain
\begin{equs}
|\CP_t X \otimes \nabla \CP_t \bar{X}
&- \CP_t Y \otimes \nabla \CP_t \bar{Y} |_{\infty} 
\\
&\lesssim
|\CP_t X \otimes \nabla \CP_t \bar{X} - \CP_t Y \otimes \nabla \CP_t \bar{X} |_{\infty}
\\
&\qquad+
|\CP_t Y \otimes \nabla \CP_t \bar{X} - \CP_t Y \otimes \nabla \CP_t \bar{Y} |_{\infty}
\\
&\lesssim t^{\frac{\eta+\bar\eta-1}{2}}(|X-Y|_{\CC^\eta}|\bar X|_{\CC^{\bar\eta}} + |Y|_{\CC^{\eta}}|\bar X-\bar Y|_{\CC^{\bar\eta}})\;,
\end{equs}
as claimed. 
\end{proof}

\begin{lemma}\label{lem:perturbation}
Consider $\bar\eta \leq 0$ and $X,Y,\bar X,\bar Y\in\CD'(\T^3,E)$. 
\begin{enumerate}[label=(\roman*)]
\item\label{pt:perturbation} 
Suppose $\eta+\bar\eta\geq 1-2\delta$.  Then
\begin{equs}
\fancynorm{X+\bar X;Y+\bar Y}_{\beta,\delta}
&\lesssim
\fancynorm{X;Y}_{\beta,\delta} + \fancynorm{\bar X;\bar Y}_{\beta,\delta}
\\
&\qquad
+ |X-Y|_{\CC^\eta}(|\bar X|_{\CC^{\bar \eta}}+|\bar Y|_{\CC^{\bar \eta}})
\\
&\qquad
+|\bar X-\bar Y|_{\CC^{\bar \eta}}(| X|_{\CC^\eta}+| Y|_{\CC^\eta})\;.
\end{equs}

\item\label{pt:fancynorm_0} 
Suppose $\bar\eta\ge \frac12-\delta$.
Then
\begin{equ}
\fancynorm{X;Y}_{0,\delta} \lesssim |X-Y|_{\CC^{\bar \eta}}(|X|_{\CC^{\bar \eta}}+|Y|_{\CC^{\bar \eta}})\;.
\end{equ}
\end{enumerate} 
The proportionality constants in both statements depend only on $\eta,\bar\eta,\beta,\delta$.
\end{lemma}

\begin{proof}
\ref{pt:perturbation}
follows from applying
Lemma~\ref{lem:etas_to_beta=zero} and the embedding $L^\infty\hookrightarrow \CC^\beta$
to the two ``cross terms''
and from the fact that $\frac{\eta+\bar\eta-1}{2} \geq -\delta \Leftrightarrow \eta+\bar\eta \geq 1-2\delta$.
\ref{pt:fancynorm_0} follows directly from Lemma~\ref{lem:etas_to_beta=zero} with $\eta=\bar\eta$.
\end{proof}

\begin{lemma}\label{lem:X_Xg_Yg}
Suppose $\eta \geq 2-2\delta-\rho$
and $X,Y\in\CC^\eta(\T^3,E)$ and $g\in\mfG^\rho$.
Then
\begin{equation}\label{eq:X_Xg_bound}
\begin{split}
\fancynorm{X;X^g}_{\beta,\delta} \lesssim
|g|_{\CC^\rho}
 \Big(   &  |g-1|_{\CC^\rho}\fancynorm{X}_{\beta,\delta}
+|g|_{\CC^\rho}|g|_{\Hol\rho}^2
\\
&+|g|_{\Hol\rho} |X|_{\CC^\eta}(|g|_{\CC^\rho}+|X|_{\CC^\eta})
\Big)\;,
\end{split}
\end{equation}
and
\begin{equ}\label{eq:Xg_Yg_bound}
\fancynorm{X^g;Y^g}_{\beta,\delta} \lesssim
|g|_{\CC^\rho}^2(\fancynorm{X;Y}_{\beta,\delta} + |X-Y|_{\CC^\eta}(|X|_{\CC^\eta} + |Y|_{\CC^\eta}+|g|_{\Hol\rho}))\;,
\end{equ}
where the proportionality constants depend only on $\eta,\beta,\rho,\delta$.
\end{lemma}

\begin{proof}
Recall our assumption \eqref{e:paramG} on the parameters.
Observe that $\eta\in(-\rho,\rho-1]$, $\rho\le 1$, and $\eta+\rho \ge  2-2\delta$
together imply $\eta + \frac{\rho-1}2 \geq -\delta$
and $\rho\ge \frac32-\delta$.
Writing $\mrd g g^{-1}\in\CC^{\rho-1}(\T^3,E)$ as the element with zero second component in $E=\mfg^3\oplus \higgsvec$,
\begin{equs}
 \fancynorm{X; X^g}_{\beta, \delta}
&= \fancynorm{X;g X-\mrd gg^{-1}}_{\beta,\delta}
\\
&\lesssim \fancynorm{X;gX}_{\beta,\delta} + \fancynorm{\mrd g g^{-1}}_{\beta,\delta}
+|g|_{\Hol\rho}|g|_{\CC^\rho}(|X|_{\CC^\eta}+|gX|_{\CC^\eta})
\;,
\end{equs}
where we used Lemma~\ref{lem:perturbation}\ref{pt:perturbation}
with $\bar\eta=\rho-1$
(using $\eta+\rho \ge  2-2\delta$).
We then obtain~\eqref{eq:X_Xg_bound} from Lemmas~\ref{lem:X_gX_bound} 
and~\ref{lem:perturbation}\ref{pt:fancynorm_0} again with $\bar\eta=\rho-1$ (using $\rho\geq\frac32-\delta$).
In a similar way
\begin{equs}
\fancynorm{X^g;Y^g}_{\beta,\delta}
&= \fancynorm{gX-\mrd gg^{-1};gY-\mrd gg^{-1}}_{\beta,\delta}
\\
&\lesssim \fancynorm{gX;gY}_{\beta,\delta}
+ |gX-gY|_{\CC^\eta}|g|_{\Hol\rho}|g|_{\CC^\rho}
\\
&\lesssim |g|_{\CC^\rho}^2(\fancynorm{X;Y}_\beta + |X-Y|_{\CC^\eta}(|X|_{\CC^\eta} + |Y|_{\CC^\eta}+|g|_{\Hol\rho}))\;,
\end{equs}
where we used Lemma~\ref{lem:perturbation}\ref{pt:perturbation}
with $\bar\eta=\rho-1$ in the first bound
(using again $\eta+\rho > 2-2\delta$)
and Lemma~\ref{lem:gX_gY_bound}
in the second bound.
\end{proof}

\begin{lemma}\label{lem:heatgr_dg}
Let $\alpha\in(0,1]$ and $\theta\geq0$
such that $\rho \geq 1+2\theta(\alpha-1)$.
Then for all $g\in\mfG^\rho$
\begin{equ}
\heatgr{(\mrd g) g^{-1}}_{\alpha,\theta} \lesssim |g|_{\CC^\rho}|g|_{\Hol\rho}\;,
\end{equ}
where the proportionality constant depends only on $\rho,\alpha,\theta$.
\end{lemma}

\begin{proof}
Since $\rho>\frac12$, $|(\mrd g) g^{-1}|_{\CC^{\rho-1}}\lesssim |g|_{\CC^\rho}|g|_{\Hol\rho}$.
Since $\rho-1\geq 2\theta(\alpha-1)$, the conclusion follows by Lemma~\ref{lem:heatgr_Besov_embed}.
\end{proof}

\begin{lemma}\label{lem:heatgr_A_Ad_g_A}
Let $\alpha\in(0,1]$ and $\theta\geq0$
such that $\rho \geq -\tilde\eta\eqdef -(1+2\theta)(\alpha-1)$.
Let $A\in\Omega\CC^\eta$ and $g\in\mfG^\rho$.
Then
\begin{equ}
\heatgr{A-\Ad_gA}_{\alpha,\theta}
\lesssim
(|g-1|_{\infty}+|g|_{\Hol\rho})\heatgr{A}_{\alpha,\theta}
\;,
\end{equ}
where the proportionality constant depends only on $\rho,\alpha,\theta$.
\end{lemma}

\begin{proof}
By Lemma~\ref{lem:heatgr_Besov_embed},
$|A|_{\CC^{\tilde\eta}} \lesssim \heatgr{A}_{\alpha,\theta}$.
Since $\tilde\eta+\rho>0$, by Corollary~\ref{cor:commutator_heat}, for all $t\in(0,1)$
\begin{equs}
|\CP_t \Ad_g A-\Ad_g\CP_t A|_{\gr\alpha;<t^\theta}
&\leq
|\CP_t \Ad_g A-\Ad_g\CP_t A |_{\gr1}
\\
&\asymp
|\CP_t \Ad_g A-\Ad_g\CP_t A|_{L^\infty} 
\lesssim
t^{\frac{\eta+\rho}{2}}|g|_{\Hol\rho}\heatgr{A}_{\alpha,\theta}\;.
\end{equs}
On the other hand, $-\tilde\eta\geq 1-\alpha$, hence
$\rho+\alpha>1$,
and thus for all $t\in(0,1)$
\begin{equ}
|\CP_t A-\Ad_g \CP_tA|_{\gr\alpha;<t^\theta} \lesssim 
(|g-1|_{\infty}+|g|_{\Hol\rho})|\CP_tA|_{\gr\alpha;<t^\theta}
\end{equ}
(this follows by restricting the proof of~\cite[Lem.~3.32]{CCHS2d} to lines of length less than $t^\theta$),
from which the conclusion follows.
\end{proof}

\begin{proof}[of Proposition~\ref{prop:group_action}]
For $X,Y\in\init$ and $g\in\mfG^{\rho}$, by~\eqref{eq:X_Xg_bound} and the fact that $|g|_{\Hol\rho} \leq |g-1|_{\CC^\rho}$, 
\begin{equ}
\fancynorm{X;Y^g}_{\beta,\delta} \leq \fancynorm{X;Y}_{\beta,\delta}
+\fancynorm{Y;Y^g}_{\beta,\delta} \to 0
\end{equ}
as $(g,Y) \to (1,X)$ in $\mfG^{\rho}\times \init$.
Furthermore, for $h\in\mfG^{\rho}$, clearly $|X^g - Y^h|_{\CC^\eta} \to 0$ as $(g,X)\to(h,Y)$.
Therefore, by~\eqref{eq:Xg_Yg_bound},
\begin{equ}
\fancynorm{X^g;Y^h}_{\beta,\delta} = \fancynorm{X^g;(Y^{g^{-1}h})^g}_{\beta,\delta} \to 0
\end{equ}
as $(g,X)\to(h,Y)$.
It follows that, if $g\in\mfG^{0,\rho}$, then $X^g\in\init$.
Furthermore $(g,X) \mapsto X^g$ is a continuous left group action $\mfG^{0,\rho}\times \init \to \init$
and it is easy to see from the explicit form of the estimates~\eqref{eq:X_Xg_bound} and~\eqref{eq:Xg_Yg_bound} that this map is uniformly continuous on every ball in $\mfG^{0,\rho}\times\init$.

Suppose further that $(\eta,\beta,\delta)$ satisfies~\eqref{eq:CI}.
Then $X^g \sim X$ for all smooth $(g,X)\in\mfG^{0,\rho}\times \init$.
The fact that $X^g \sim X$ for all $(g,X)\in\mfG^{0,\rho}\times \init$ now follows from Proposition~\ref{prop:back_unique} together with the density of smooth functions in $\mfG^{0,\rho} \times \init$ and continuity of $\ymhflow$ on $\init$ by Proposition~\ref{prop:YM_flow_minus_heat}.

Finally, suppose further that $(\rho,\eta,\beta,\delta,\alpha,\theta)$ satisfies~\eqref{eq:CGS}.
Fix $(g,X) \in \mfG^{\rho}\times\state$.
Then for $(h,Y) \in \mfG^{\rho}\times\state$, by Lemmas~\ref{lem:heatgr_dg} and~\ref{lem:heatgr_A_Ad_g_A}
\begin{equ}
\heatgr{X-Y^h}_{\alpha,\theta} \leq \heatgr{X-Y}_{\alpha,\theta}
+
\heatgr{Y-\Ad_h Y}_{\alpha,\theta}
+ \heatgr{(\mrd h)h^{-1}}_{\alpha,\theta}
\to 0
\end{equ}
as $|h-1|_{\CC^\rho} +\heatgr{X-Y}_{\alpha,\theta} \to 0$.
Furthermore, by Lemma~\ref{lem:heatgr_A_Ad_g_A},
$\heatgr{X^g-Y^g}_{\alpha,\theta}\to 0$ as $\heatgr{Y- X}_{\alpha,\theta}\to0$,
and therefore
\begin{equ}
\heatgr{X^g-Y^h}_{\alpha,\theta} = \heatgr{X^g-(Y^{g})^{hg^{-1}}}_{\alpha,\theta} \to 0
\end{equ}
as $|h-g|_{\CC^\rho} +\heatgr{X-Y}_{\alpha,\theta} \to 0$.
It follows that $(g,X) \mapsto X^g$ is a continuous left group action $\mfG^{0,\rho}\times \state \to \state$
and uniform continuity on every ball in $\mfG^{0,\rho}\times\state$ is clear again from the explicit form of the estimates in Lemmas~\ref{lem:heatgr_dg} and~\ref{lem:heatgr_A_Ad_g_A}.
\end{proof}

\subsection{Estimates on gauge transformations}
\label{subsec:est_gauge_trans}

Recall that, by Proposition~\ref{prop:group_action}, if $(\rho,\eta,\beta,\delta)$ satisfies~\eqref{eq:CGI}, then $(g,X)\mapsto X^g$ is a continuous group action $\mfG^{0,\rho}\times\init\to \init$.

\begin{theorem}\label{thm:g_bound}
Let $\theta=0+$, $\alpha=\frac12-$, $\delta\in(\frac34,1)$, $\beta=-2(1-\delta)-$,
and
$\eta=-\frac12-$.
There exists $\nu=\frac12-$ such that
$|g|_{\CC^\nu} \leq \Poly(\Sigma(X)+\Sigma(X^g))$ for all $g\in\mfG^{0,\rho}$, $X\in\init$, and $\rho\in(\frac12,1]$ such that $(\rho,\eta,\beta,\delta)$ satisfies~\eqref{eq:CGI}.
\end{theorem}

\begin{remark}
The conditions on $(\eta,\beta,\delta)$ in the above theorem
ensure that they satisfy \eqref{eq:CI}
and that there exists $\rho\in(\frac12,1]$ such that $(\rho,\eta,\beta,\delta)$ satisfies~\eqref{eq:CGI}.
On the other hand, we do not suppose in this subsection that $(\rho,\eta,\beta,\delta,\alpha,\theta)$ satisfies~\eqref{eq:CGS} since the continuity of the group action $\mfG^{0,\rho}\times\state\to\state$ will not be needed.
\end{remark}


We prove the theorem at the end of this subsection.
The main ingredients are Lemma~\ref{lem:init_cond_bound}\ref{pt:Hol_estimate_g},
which provides an estimate on $|g|_{\CC^\nu}$ in terms of a nonlinear PDE $\CH^X(g)$ with initial condition $g$ (this is rather generic and similar to how the heat flow characterises H{\"o}lder spaces),
and Lemma~\ref{lem:holonomy_estimates},
which estimates the (spatial) H{\"o}lder norm of $\CH^X(g)$ in terms of $\ymhflow(X)$ and $\ymhflow(X^g)$
(this estimate is based on holonomies and Young ODE theory, and was used extensively in~\cite{CCHS2d}).

We first require several lemmas and definitions.
Recall the notation $\ymhflow(X)=(\ymhflow_A(X),\ymhflow_\Phi(X))$
from Definition~\ref{def:ymhflow}.

\begin{definition}\label{def:CH}
Let $(\eta,\beta,\delta)$ satisfy~\eqref{eq:CI}, $X\in\init$ and $g\in\mfG^\nu$ for some $\nu>0$.
We denote by $\CH^X(g)\colon(0,T_{X,g})\to \mfG^\infty$
the solution to
\begin{equs}[eq:CH_equation]
\partial_t H H^{-1}
&= -\mrd^*_{\ymhflow_A(X)^H} (\mrd H H^{-1})
\\
&=
\partial_j((\partial_j H) H^{-1}) + [\Ad_H(\ymhflow_A(X)_j),(\partial_j H) H^{-1}]\;,
\\
H(0) &= g\;,
\end{equs}
where $T_{X,g}$ is the minimum between $T_X$ and the blow up time of $\CH^X(g)$ in $\mfG^\infty$.
\end{definition}

\begin{remark}
Using the bound $\sup_{t\in(0,T)}t^{-\frac{\eta\wedge\hat\beta}{2}}|\ymhflow_t(X)|_\infty<\infty$ for some $T>0$,
standard arguments show that, if $g \in \mfG^\nu$ for some $\nu>0$, then indeed a classical solution to~\eqref{eq:CH_equation} exists and is continuous with respect to $g \in \mfG^\nu$ (cf.\ Lemma~\ref{lem:init_cond_bound}\ref{pt:short_time_exist_g}).
\end{remark}

\begin{remark}
``$\CH$'' stands for ``harmonic'' since, if $\ymhflow_A(X)=0$, then $\CH^X(g)$ is the harmonic map flow with initial condition $g$.
\end{remark}

\begin{remark}\label{rem:H_meaning}
The significance behind Definition~\ref{def:CH}
is the identity
\begin{equ}\label{eq:CS_gauge_transform}
\ymhflow_t(X)^{\CH^X_t(g)} = \ymhflow_t(X^g)\;,\qquad \forall\; t\in [0,T_{X,g}\wedge T_{X^g})\;,
\end{equ}
which is valid for all smooth $(g,X)$ and therefore for all $X\in\init$ and $g\in\mfG^{0,\rho}$ whenever $(\rho,\eta,\beta,\delta)$ satisfies both~\eqref{eq:CI} and~\eqref{eq:CGI}.
In particular, $\CH^X(g)$ also solves
\begin{equs}
\partial_t H H^{-1}
&= -\mrd^*_{\ymhflow_A(X^g)} (\mrd H H^{-1})
\\
&= \partial_j((\partial_j H) H^{-1}) + [\ymhflow_A(X^g),(\partial_j H) H^{-1}]\;,
\end{equs}
which follows from  $ [\Ad_H(\ymhflow_A(X)_j),(\partial_j H) H^{-1}]=
[\ymhflow_A(X)^H_j,(\partial_j H) H^{-1}]$  in \eqref{eq:CH_equation}
and then
\eqref{eq:CS_gauge_transform}.
Moreover, a direct computation shows that $\CH^X_t(g)^{-1}=\CH^{X^g}_t(g^{-1})$ for all
$t\in[0,T_{X,g}\wedge T_{X^g,g^{-1}})$.
\end{remark}

The relation~\eqref{eq:CS_gauge_transform} yields the following lemma.

\begin{lemma}\label{lem:holonomy_estimates}
Let $(\rho,\eta,\beta,\delta)$ satisfy~\eqref{eq:CI} and~\eqref{eq:CGI}, $g\in\mfG^{0,\rho}$, and $X\in\init$.
Then $T_{X,g}=T_{X^g,g^{-1}} = T_X\wedge T_{X^g}$, and, for all $t\in(0,T_{X,g})$ and $\gamma\in(\frac12,1]$,
\begin{equ}\label{eq:Hol_g_bound}
|\CH^X_t(g)|_{\Hol\gamma} \lesssim |\ymhflow_{A,t}(X)|_{\gr\gamma} + |\ymhflow_{A,t}(X^g)|_{\gr\gamma}\;,
\end{equ}
where the proportionality constant depends only on $\gamma$.
\end{lemma}

\begin{proof}
The bound~\eqref{eq:Hol_g_bound} follows from~\eqref{eq:CS_gauge_transform} and~\cite[Prop.~3.35, Eq.~3.24]{CCHS2d}.
Then $T_{X}\wedge T_{X^g} \leq T_{X,g}$
follows from the fact that, for any given $k > 0$, $\|\CH^X(g)\|_{\CC^k}$ only blows up if $|\CH^X(g)|_{\CC^\nu}$ does for all $\nu>0$.
Since $T_{X,g}\leq T_X$ by definition, it remains only to show that $T_{X,g} \leq T_{X^g}$
(which will show $T_{X,g}=T_X\wedge T_{X^g}$ and the fact that $T_{X^g,g^{-1}} = T_X\wedge T_{X^g}$ follows by symmetry).
But this follows from the facts that, for all $\gamma\in(\frac12,1]$,
$\ymhflow(X)$ can only blow up if $|\ymhflow_A(X)|_{\gr\gamma}+|\ymhflow_\Phi(X)|_{\CC^{\gamma-1}}$ does and that
$|A^g|_{\gr\gamma}+|g\Phi|_{\CC^{\gamma-1}} \leq K(|A|_{\gr\gamma}+|\Phi|_{\CC^{\gamma-1}}+|g|_{\Hol\gamma})$ for some increasing function $K\colon\R_+\to\R_+$.
\end{proof}

For $\eta\in\R$, $T>0$, a normed space $B$,
and $X\colon(0,T) \to \mcC(\T^3,B)$, 
denote
\begin{equ}
|X|_{\eta;T} \eqdef \sup_{t\in(0,T)} t^{-\eta/2} |X_t|_{\infty}\;.
\end{equ}

\begin{lemma}\label{lem:init_cond_bound}
Let $(\eta,\beta,\delta)$ satisfy~\eqref{eq:CI}, $\nu\in(0,1]$,
$g\in\mfG^{\nu}$, $X\in\init$, and $T>0$.

\begin{enumerate}[label=(\alph*)]
\item\label{pt:short_time_exist_g} Let $\gamma\in[\nu,2)$,
and define $
\kappa \eqdef \frac12\min\{\eta+1,\nu\}$.
Then there exists $c>0$ such that for all $S\in(0,T)$ such that $0<S^\kappa \le c |g|_{\mcC^{\nu}}^{-2}(|\ymhflow_A(X)|_{\eta;T}+1)^{-1}$,
\begin{equ}[e:defNorm]
|\CH^X(g)|_{\gamma,\nu;S} \eqdef \sup_{t\in(0,S)} t^{(\gamma-\nu)/2}|\CH^X_t(g)|_{\mcC^\gamma} \lesssim |g|_{\mcC^{\nu}}\;,
\end{equ}
where the proportionality constant depends only on $\eta,\nu,\gamma$.

\item\label{pt:Hol_estimate_g} Conversely, if $\gamma-\nu \in [0,\frac{2\kappa}3)$,
then
\begin{equ}
|g|_{\mcC^\nu}\leq \Poly ( |\ymhflow_A(X)|_{\eta;T} + |\CH^X(g)|_{\gamma,\nu;T}+T^{-1} )\;,
\end{equ}
where the relevant constants depend only on $\eta,\nu,\gamma$.
\end{enumerate}
\end{lemma}

\begin{proof}
\ref{pt:short_time_exist_g}
Observe that $h\eqdef\CH^X(g)$ solves a fixed point of the form
\begin{equ}
h_t = \CM(h)_t \eqdef e^{t\Delta}g + \int_0^t e^{(t-s)\Delta}( (\nabla h_s)^2 h_s + (\nabla h_s) h_s^2 \ymhflow_{A,s}(X))\, \mrd s\;.
\end{equ}
It is easy to check that, for the given values of $S$ and for $c$ small enough, $\CM$
stabilises the ball of radius $2 |g|_{\CC^\nu}$ in the Banach space
$\CC_{\gamma,\nu;S}\cap\CC_{1,\nu;S}\cap\CC_{\nu,\nu;S}$
(using the notation obvious from \eqref{e:defNorm})
and is a contraction on this ball.


\ref{pt:Hol_estimate_g}
Observe that, for all $S>0$ and $t\in(0,S)$
\begin{equ}
|e^{(t-s)\Delta}(\nabla h_s) h_s^2\ymhflow_{A,s}(X)|_{\mcC^\gamma}
\lesssim
(t-s)^{-\gamma/2} s^{\frac\eta2+\frac{\nu-1}{2}}|\ymhflow_{A,s}(X)|_{\eta;S}|h|_{\nu,\nu;S}^2|h|_{1,\nu;S}\;,
\end{equ}
and
\begin{equ}
|e^{(t-s)\Delta}(\nabla h_s)^2 h_s|_{\mcC^\gamma}
\lesssim (t-s)^{-\gamma/2} s^{\nu-1}|h|_{1,\nu;S}^2|h|_{\nu,\nu;S}\;.
\end{equ}
By part~\ref{pt:short_time_exist_g}, $|h|_{\nu,\nu;S}+|h|_{1,\nu;S}\lesssim |g|_{\mcC^\nu}$ whenever $S\leq \Poly (|\ymhflow_A(X)|_{\eta;S}^{-1}) |g|_{\CC^{\nu}}^{-2/\kappa}$,
and thus
\begin{equ}
\sup_{t\in(0,S)}t^{(\gamma-\nu)/2}
|h_t-e^{t\Delta}g|_{\mcC^\gamma}
\lesssim S^{\kappa}(|\ymhflow_A(X)|_{\eta;S}+1)|g|^3_{\mcC^\nu}\;.
\end{equ}
On the other hand,
\begin{equ}
|g|_{\mcC^\nu}\asymp \sup_{t\in(0,1)}t^{(\gamma-\nu)/2}|e^{t\Delta}g|_{\mcC^{\gamma}}
\end{equ}
and,
for all $t\geq S$, $|e^{t\Delta}g|_{\mcC^\gamma}\leq |e^{S\Delta}g|_{\mcC^\gamma}$.
Therefore,
denoting $\bar\kappa \eqdef \kappa-\frac{\gamma-\nu}{2}\in(\frac{2\kappa}3,\kappa)$, 
\begin{equs}
|g|_{\mcC^\nu}
&\lesssim S^{-(\gamma-\nu)/2}\sup_{t\in(0,S)} t^{(\gamma-\nu)/2}|e^{t\Delta}g|_{\mcC^\gamma}
\\
&\lesssim S^{-(\gamma-\nu)/2}|h|_{\gamma,\nu;S} + S^{\bar\kappa}(|\ymhflow_A(X)|_{\eta;S}+1)|g|_{\mcC^\nu}^3\;.
\end{equs}
Hence, whenever $S \leq \Poly (|\ymhflow_A(X)|_{\eta;S}^{-1}) |g|_{\mcC^\nu}^{-2/\bar\kappa}$,
\begin{equ}\label{eq:g_h_bound}
|g|_{\mcC^\nu}\lesssim S^{-(\gamma-\nu)/2}|h|_{\gamma,\nu;S}\;.
\end{equ}
If $T\leq \Poly (|\ymhflow_A(X)|_{\eta;T}^{-1}) |g|_{\mcC^\nu}^{-2/\bar\kappa}$, then the conclusion follows.
Otherwise, we can choose $S\in(0,T)$ such that $S \asymp (|\ymhflow_A(X)|_{\eta;T}+1)^q |g|_{\mcC^\nu}^{-2/\bar\kappa}$ for some $q<0$
and such that~\eqref{eq:g_h_bound} holds,
and thus
\begin{equ}
|g|_{\mcC^\nu} \leq \Poly (|\ymhflow_A(X)|_{\eta;T}) |g|_{\mcC^\nu}^{(\gamma-\nu)/\bar\kappa}|h|_{\gamma,\nu;T}\;.
\end{equ}
Since $\gamma-\nu \in [0,\frac{2\kappa}{3}) \Leftrightarrow (\gamma-\nu)/\bar\kappa \in [0,1)$, the conclusion follows.
\end{proof}

The proof of the following lemma is routine.

\begin{lemma}[Interpolation]\label{lem:interpolation}
For $\alpha,\zeta\in(0,1]$, $\kappa\in[0,1]$, 
and $A\in\Omega$,
\begin{equ}
|A|_{\gr{(\kappa\alpha+(1-\kappa)\zeta)}} \leq |A|^\kappa_{\gr\alpha} |A|^{1-\kappa}_{\gr\zeta}\;.
\end{equ}
\end{lemma}

We will apply the following lemma with $\alpha=\frac12-$ and $\gamma=\frac12+$.

\begin{lemma}\label{lem:heat_interpolate}
For all $\eta\leq 0$, $\alpha\in(0,1]$, $\gamma\in[\alpha,1]$, $\theta\geq0$, and $A\in\Omega\CD'$,
\begin{equ}
\sup_{t\in(0,1)}t^{-(1-\kappa)\eta/2 + \theta(1-\alpha)\kappa} |\CP_t A|_{\gr\gamma}
\lesssim \heatgr{A}_{\alpha,\theta}^{\kappa}|A|^{1-\kappa}_{\CC^\eta}
\end{equ}
where $\kappa \eqdef \frac{\gamma-1}{\alpha-1}\in[0,1]$
and the proportionality constant depends only on $\eta$.
\end{lemma}

\begin{proof}
Observe that $|\CP_tA|_{\gr1} \lesssim |\CP_tA|_\infty \lesssim t^{\eta/2}|A|_{\CC^\eta}$,
and, by Lemma~\ref{lem:local_to_global_gr},
$|\CP_t A|_{\gr\alpha} \leq t^{-\theta(1-\alpha)}\heatgr{A}_{\alpha,\theta}$.
Since $\gamma=\kappa\alpha+(1-\kappa)$,
it follows from Lemma~\ref{lem:interpolation}
that
\begin{equ}
|\CP_tA|_{\gr\gamma}
\leq |\CP_tA|_{\gr\alpha}^{\kappa}|\CP_t A|^{1-\kappa}_{\gr1}
\lesssim t^{-\theta(1-\alpha)\kappa}\heatgr{A}_{\alpha,\theta}^{\kappa}t^{(1-\kappa)\eta/2}|A|^{1-\kappa}_{\CC^\eta} \;,
\end{equ}
as claimed.
\end{proof}

\begin{proof}[of Theorem~\ref{thm:g_bound}]
Take
$\gamma=\frac12+$
and denote $X=(A,\Phi)$.
Then, for all $T\leq \Poly (\Theta(X)^{-1})$ and $t\in(0,T)$,
\begin{equs}
|\ymhflow_{A,t}(X)|_{\gr\gamma}
&\leq |\CP_t A|_{\gr\gamma} + |\ymhflow_{A,t}(X)-\CP_tA|_{\gr\gamma}
\\
&\lesssim t^{(1-\kappa)\eta/2 - \theta(1-\alpha)\kappa}\heatgr{A}_{\alpha,\theta}^{\kappa}|A|^{1-\kappa}_{\CC^\eta} + t^{\hat\beta/2}(\fancynorm{X}_\beta+1)
\\
&\leq t^{\omega}(\Sigma(X)+1)\;,
\end{equs}
where $\kappa=\frac{\gamma-1}{\alpha-1}$ and $\omega=0-$,
and where we used Lemma~\ref{lem:heat_interpolate}
and Proposition~\ref{prop:YM_flow_minus_heat} in the second bound along with the trivial bound $|\cdot|_{\gr\gamma}\leq|\cdot|_{\gr1}\asymp |\cdot|_{\infty}$.

Define $\nu\eqdef \gamma+2\omega$.
Observe that $\nu=\frac12-$ for our choice of parameters.
It follows from Lemma~\ref{lem:holonomy_estimates} that  
\begin{equ}
|\CH^X(g)|_{\gamma,\nu;T} = \sup_{t\in(0,T)}t^{-\omega}|\CH^X_t(g)|_{\CC^\gamma} \lesssim \Sigma(X)+\Sigma(X^g)+1\;.
\end{equ}
Since $|\ymhflow_A(X)|_{\eta;T} \lesssim \Theta(X)+1$ by Proposition~\ref{prop:YM_flow_minus_heat},
it follows from Lemma~\ref{lem:init_cond_bound}\ref{pt:Hol_estimate_g} that $|g|_{\CC^\nu}\leq \Poly (\Sigma(A)+\Sigma(A^g) )$.
\end{proof}

\subsection{Gauge orbits}
\label{subsec:gauge_orbits}

Let us fix $(\eta,\beta,\delta)$ satisfying~\eqref{eq:CI} for the rest of the subsection.
Recall the relation $\sim$ on $\init$ from Definition~\ref{def:ymhflow} which extends  the usual notion of gauge equivalence.
Recall also the space $\state$ in Definition~\ref{def:state}.

\begin{definition}\label{def:orbits}
Define the quotient space under $\sim$ by $\mfO\eqdef \state/{\sim}$.
\end{definition}

\begin{lemma}\label{lem:gauge_orbits_closed}
The set $\{(X,Y)\in\init^2\,:\,X\sim Y\}$ is closed in $\init^2$.
\end{lemma}

\begin{proof}
Suppose $(X_n,Y_n)\to(X,Y)$ in $\init^2$ with $X_n\sim Y_n$.
Then,
by Proposition~\ref{prop:YM_flow_minus_heat},
\begin{equ}
|\ymhflow_t(X_n) - \ymhflow_t(X)|_{\Omega_{\gr 1}\times \CC(\T^3,\higgsvec)}\to 0
\end{equ}
for some $t>0$ sufficiently small, and similarly for $Y_n,Y$.
The conclusion now follows from Lemma~\ref{lem:gauge_orbits_closed_rho}.
\end{proof}

\begin{proposition}\label{prop:Hausdorff}
With the quotient topology, $\mfO$ is separable and completely Hausdorff.
\end{proposition}

\begin{proof}
Separability follows from the fact that $\state$ is separable.
The relation $\sim$ is of the form $X\sim Y \Leftrightarrow f(X) = f(Y)$ for a continuous function
$f \colon \state \to \mfU$, where $\mfU$ a metric space;
for instance,
 $f=\tilde\ymhflow_t$ for $t>0$ (recall Proposition~\ref{prop:back_unique})
 and $\mfU=\mfO_\alpha$ for $\alpha\in(\frac23,1]$ as defined in~\cite[Sec.~3.6]{CCHS2d}.
Since $\mfU$ is completely Hausdorff, so is $\mfO$.
\end{proof}

\begin{remark}
In addition to the quotient topology,
another natural topology on $\mfO$
is given by taking one of the functions $f\colon\state\to\mfU$ mentioned in 
the proof of Proposition~\ref{prop:Hausdorff},
which necessarily induces a bijection $f\colon\mfO\to f(\mfO)\subset\mfU$,
and equip $\mfO$ with the corresponding subspace topology.
This leads to a weaker topology than the quotient topology
but which is metrisable and separable.
One loses, however, continuity of solutions to SYMH (modulo gauge equivalence) with respect to the initial condition under this topology,
and it is not clear whether there is a choice for $f$ (or some family of such functions) for which this weaker topology is \textit{completely} metrisable.
\end{remark}

Recall (see the last paragraph of Section~\ref{sec:notation}) the metrisable space $\hat\state\eqdef \state\cup\{\skull\}$.
We extend $\sim$ to $\hat\state$ by $Y\sim\skull \Leftrightarrow Y=\skull$.
In particular, we identify $\hat\mfO\eqdef \hat\state/{\sim}$ with $\mfO\cup\{\skull\}$.
Equip $\hat\mfO$ with the quotient topology.

\begin{proposition}\label{prop:hat_mfO_top}
A strict subset $O\subsetneq\hat\mfO$ is open if and only if $\skull\notin O$ and $O$ is open as a subset of $\mfO$.  
\end{proposition}

\begin{proof}
Consider the collection $\tau\eqdef \{\hat\mfO\}\cup \{O\subset\mfO\,:\,O\text{ is open}\}$.
It is easy to verify that $\tau$ is a topology on $\hat\mfO$.
Let $\sigma$ denote the quotient topology on $\hat\mfO$.
We claim that the projection $\pi\colon\hat\state\to(\hat\mfO,\tau)$
is continuous,
from which it follows that $\tau\subset \sigma$.
Indeed, for $O\in\tau$, either $O=\hat\mfO$ in which case $\pi^{-1}(O)=\hat\state$, or $O\subset\mfO$ is open in which case $\pi^{-1}(O)$ is an open subset of $\state$ and therefore also an open subset in $\hat\state$.
Hence $\pi\colon\hat\state\to(\hat\mfO,\tau)$ is continuous as claimed.

We now claim that also $\sigma\subset\tau$, which then completes the proof.
Indeed, every equivalence class $[A]\in\mfO$ contains a sequence which converges to $\skull$ in $\hat\state$ (see~\cite[Lem.~3.46]{CCHS2d} for the proof of a similar statement).
Therefore, if $O\in\sigma$ and $\skull\in O$, then $O=\hat\mfO\in\tau$ since $\pi^{-1}(O)$ 
is open by definition of the quotient topology and contains $\skull$,
so that it must contain some element of every equivalence class $[A]\in\mfO$.
On the other hand, suppose $O\in\sigma$ and $\skull\notin O$.
Then $\pi^{-1}(O)\subset \state$ and $\pi^{-1}(O)$ is open as a subset of $\hat\state$ (by definition of $\sigma$).
It is easy to see that every open subset of $\hat\state$ which does not contain $\skull$ is also open in $\state$.
Hence $\pi^{-1}(O)$ is also open in $\state$, and therefore $O$ is open in $\mfO$, thus $O\in\tau$.
Hence $\sigma\subset \tau$ as claimed.
\end{proof}

\subsection{Embeddings and heat flow}
\label{subsec:embeddings}

In this subsection, we study compact embeddings for the spaces
$\init$ and $\state$ and the effect of the heat flow.
We first record the following lemma, the proof of which is straightforward.

\begin{lemma}[Lower semi-continuity]\label{lem:Theta_Sigma_lower_semicont}
Let $\eta,\beta,\delta,\theta,\kappa\in\R$, $\alpha\in(0,1]$.
The functions
\begin{equ}
|\cdot|_{\CC^{\eta}}\;,
\;\;
\fancynorm{\cdot}_{\beta,\delta}\;,\;\;
\heatgr{\cdot}_{\alpha,\theta}
\;\;
\colon\CC^\kappa\to [0,\infty]
\end{equ}
are lower semi-continuous.
\end{lemma}

\begin{lemma}\label{lem:CP_contraction}
\begin{enumerate}[label=(\roman*)]
\item\label{pt:Theta_limit}
For all $\eta,\beta\in\R$ and $\delta\geq0$, there exists $C>0$ such that, for all $h\in[0,1]$ and $X\in\CD'(\T^3,E)$,
$\Theta(\CP_hX)\leq (1+Ch)\Theta(X)$.
\item\label{pt:heatgr_contraction} For all $\alpha\in(0,1]$, 
$A\in\Omega\CD'$, and $t,s\geq 0$,
$|\CP_sA|_{\gr\alpha;<t} \leq |A|_{\gr\alpha;<t}$.
In particular, $\CP_s$ is a contraction for $\heatgr{\cdot}_{\alpha,\theta}$ for any $\alpha\in(0,1]$ and $\theta\in\R$.
\end{enumerate}
\end{lemma}

\begin{proof}
\ref{pt:Theta_limit}
Note that $|\CP_hX|_{\CC^\eta} \leq (1+Ch)|X|_{\CC^\eta}$.
Furthermore
\begin{equs}
\fancynorm{\CP_hX}_{\beta,\delta}
&=\sup_{t\in(0,1)}t^\delta|\CN_t\CP_hX|_{\CC^\beta}
\\
&= \sup_{t\in(h,1+h)}(t-h)^\delta|\CN_tX|_{\CC^\beta}
\leq\fancynorm{X}_{\beta,\delta} + Ch|X|_{\CC^\eta}^2\;,
\end{equs}
where we used that $\sup_{t\in[1,1+h)}|\CN_tX-\CN_1 X|_{\CC^\beta}\leq C h |X|_{\CC^\eta}^2$ to get the last bound.
\ref{pt:heatgr_contraction} Consider $\ell=(x,v)$ with $|\ell|<t$.
Then, by the definition~\eqref{eq:A_ell_def},
\begin{equs}
|\CP_s A(\ell)| &= \Big| \int_0^{1} \int_{\T^3} \sum_{i=1}^3 v_i \CP_s(y)A_i(x+rv-y) \mrd y \mrd r \Big|
\\
&= \Big|\int_{\T^3} \CP_s(y) A(x-y,v)\mrd y\Big|
\leq |A|_{\gr\alpha;<t}|\ell|^\alpha\;,
\end{equs}
where we used that $|A(x-y,v)|\leq|A|_{\gr\alpha;<t}|\ell|^\alpha$
and $\int_{\T^3}\CP_s(y)\mrd y=1$.
\end{proof}

The proof of the following lemma is obvious.

\begin{lemma}\label{lem:diff_times}
For $s\in(0,1)$, $0<\alpha\leq \bar\alpha\leq1$, $\theta\geq 0$, and $A\in\Omega\CD'$,
\begin{equ}
\sup_{t\in (0,s]} |\CP_t A|_{\gr{\alpha};<t^\theta} \leq s^{\theta(\bar\alpha-\alpha)} \heatgr{A}_{\bar\alpha,\theta}\;.
\end{equ}
\end{lemma}

\begin{proposition}[Compact embeddings]
\label{prop:compact}
Let $\eta<\bar\eta,\bar\delta<\delta$, $\beta,\kappa\in\R$,
$0<\alpha<\bar\alpha\leq1$, and $\theta,R>0$.
\begin{enumerate}[label=(\roman*)]
\item\label{pt:Theta_compact} If $X_n\to X$ in $\CC^\kappa(\T^3,E)$ and $\sup_n \Theta_{\bar\eta,\beta,\bar\delta}(X_n)<\infty$, then $\Theta(X_n,X)\to0$.
In particular, $\{X\in\CD'(\T^3,E)\,:\, \Theta_{\bar\eta,\beta,\bar\delta}(X) \leq R\}$
is a compact subset of $(\init,\Theta)$.

\item\label{pt:heatgr_compact} If $A_n\to A$ in $\Omega\CC^\kappa$ and $\sup_n \heatgr{A_n}_{\bar\alpha,\theta}<\infty$, then $\heatgr{A_n-A}_{\alpha,\theta}\to0$.
In particular, $\{A\in\Omega\CD'\,:\, \heatgr{A}_{\bar\alpha,\theta} \leq R\}$ is a compact subset of the closure of smooth functions under $\heatgr{\cdot}_{\alpha,\theta}$.

\item\label{pt:Sigma_compact} $\{X\in\CD'(\T^3,E)\,:\, \Sigma_{\bar\eta,\beta,\bar\delta,\bar\alpha,\theta}(X) \leq R\}$ is a compact subset of $(\state,\Sigma)$.
\end{enumerate}
\end{proposition}

\begin{proof}
\ref{pt:Theta_compact} Recall the compact embedding $\CC^{\bar\eta}\hookrightarrow \CC^{0,\eta}$.
Furthermore, for every
$s\in(0,1)$
\begin{equs}
\fancynorm{X_n ; X}_{\beta,\delta}
&\leq\sup_{t\in (0,s]} t^{\delta}\{|\CN_tX_n|_{\CC^\beta} + |\CN_tX|_{\CC^\beta}\}
+ \sup_{t\in (s,1)}t^\delta|\CN_{t}X_n-\CN_tX|_{\CC^\beta}\;.
\end{equs}
The first term is bounded above by $2s^{\delta-\bar\delta}\sup_n\fancynorm{X_n}_{\beta,\bar\delta}$ due to the lower semi-continuity of $\fancynorm{\cdot}_{\beta,\bar\delta}$ (Lemma~\ref{lem:Theta_Sigma_lower_semicont})
and can therefore be made arbitrarily small by choosing $s$ sufficiently small,
while the second term converges to $0$ as $n\to\infty$ for any fixed $s\in(0,1)$ 
since the operator norm of $\CP_t \colon \CC^\kappa \to \CC^1$ is bounded uniformly in $t>s$.
It follows that $\Theta(X_n,X)\to0$.
The final claim in~\ref{pt:Theta_compact} now follows from Lemma~\ref{lem:Theta_Sigma_lower_semicont}.

\ref{pt:heatgr_compact}
In a similar way, for every
$s\in(0,1)$
\begin{equs}
\heatgr{A_n - A}_{\alpha,\theta}
&\leq\sup_{t\in (0,s]} \{|\CP_tA_n|_{\gr\alpha;<t^{\theta}} + |\CP_tA|_{\gr\alpha;<t^{\theta}}\}
\\
&\qquad + \sup_{t\in (s,1)}|\CP_{t}A_n-\CP_tA|_{\gr\alpha}\;.
\end{equs}
The first term is bounded above by $2s^{\theta(\bar\alpha-\alpha)}\sup_n\heatgr{A_n}_{\bar\alpha,\theta}$ by Lemma~\ref{lem:diff_times}
and lower semi-continuity of $\heatgr{\cdot}_{\alpha,\theta}$ (Lemma~\ref{lem:Theta_Sigma_lower_semicont})
and can therefore be made arbitrarily small,
while the second term converges to $0$ as $n\to\infty$.
Hence $\heatgr{A_n-A}_{\alpha,\theta}\to0$.
The final claim in~\ref{pt:heatgr_compact} now follows from Lemma~\ref{lem:Theta_Sigma_lower_semicont}.

\ref{pt:Sigma_compact} This claim is a consequence of~\ref{pt:Theta_compact} and~\ref{pt:heatgr_compact}.
\end{proof}

\begin{proposition}\label{prop:Theta_cont_zero}
Let $\beta,\eta,\theta\in\R$, $\delta\geq 0$, and $\alpha\in(0,1]$.
\begin{enumerate}[label=(\roman*)]
\item\label{pt:heat_Theta} $\lim_{h\downarrow0}\Theta(\CP_hX,X)=0$ for all $X\in \init$.

\item\label{pt:heat_Sigma} $\lim_{h\downarrow0}\Sigma(\CP_hX,X)=0$ for all $X\in \state$.
\end{enumerate}
\end{proposition}

\begin{proof}
To prove~\ref{pt:heat_Theta}, observe that
for all $h\in[0,1]$
\begin{equs}[eq:Ph_X_Y]
\fancynorm{\CP_hX;\CP_hY}_{\beta,\delta}
&=\sup_{t\in(0,1)}t^\delta|\CN_t\CP_hX-\CN_t\CP_hY|_{\CC^\beta}
\\
&= \sup_{t\in(h,1+h)}(t-h)^\delta|\CN_tX-\CN_tY|_{\CC^\beta}
\\
&\lesssim |X-Y|_{\CC^\eta}(|X|_{\CC^\eta}+|Y|_{\CC^\eta}) + \fancynorm{X;Y}_{\beta,\delta}
\end{equs}
for a proportionality constant depending only on $\eta,\beta,\delta$.
It readily follows that the set of $X\in\init$ such that $\lim_{h\downarrow0}\fancynorm{X;\CP_h X} = 0$ is closed in $\init$.
Furthermore, this set contains all smooth functions, from which~\ref{pt:heat_Theta} follows.
The claim \ref{pt:heat_Sigma} follows in the same way once we remark that $\heatgr{\CP_hA}_{\alpha,\theta} \leq \heatgr{A}_{\alpha,\theta}$ by Lemma~\ref{lem:CP_contraction}\ref{pt:heatgr_contraction} and therefore the set of $X\in\state$ for which $\lim_{h\downarrow0}\Sigma(\CP_hX,X)=0$ is closed in $\state$ and contains the smooth functions.
\end{proof}

\begin{lemma}\label{lem:heatgr_heat_flow}
Let $0<\bar\alpha\leq\alpha\leq1$,
$\theta\geq 0$, $\kappa\in(0,1)$, $h>0$,
and $A\in\Omega\CD'$.
Then
\begin{equ}
\heatgr{\CP_h A -A}_{\bar\alpha,\theta}
\lesssim h^{\kappa\theta(\alpha-\bar\alpha)}\heatgr{A}_{\alpha,\theta}\;,
\end{equ}
where the proportionality constant depends only on $\bar\alpha,\alpha,\theta,\kappa$.
\end{lemma}

\begin{proof}
We can assume that $h\in(0,1)$.
For all $0<t\leq s<1$
\begin{equ}
|\CP_{t+h}A-\CP_tA|_{\gr{\bar\alpha};<t^\theta}
\leq
s^{\theta(\alpha-\bar\alpha)}
|\CP_{t+h}A-\CP_tA|_{\gr\alpha;<t^\theta}
\leq
2s^{\theta(\alpha-\bar\alpha)}\heatgr{A}_{\alpha,\theta}
\end{equ}
where we used Lemma~\ref{lem:local_lower_exponent} in the first bound and Lemma~\ref{lem:CP_contraction}\ref{pt:heatgr_contraction} in the second bound.
On the other hand,
denote $\eta\eqdef (1+2\theta)(\alpha-1)$ so that $|A|_{\CC^\eta}\lesssim \heatgr{A}_{\alpha,\theta}$
by Lemma~\ref{lem:heatgr_Besov_embed}.
Then, for $0<s<t<1$ and $\bar\eta\leq \eta$,
\begin{equs}
|\CP_{t+h}A-\CP_tA|_{\gr{\bar\alpha};<t^\theta}
&\leq
|\CP_{t+h}A-\CP_tA|_{\gr1} \asymp |\CP_{t+h}A-\CP_tA|_{L^\infty}
\\
&\lesssim
t^{\bar\eta/2}|\CP_hA-A|_{\CC^{\bar\eta}}
\lesssim
t^{\bar\eta/2}h^{(\eta-\bar\eta)/2}|A|_{\CC^{\eta}}
\\
&\lesssim
s^{\bar\eta/2}h^{(\eta-\bar\eta)/2}\heatgr{A}_{\alpha,\theta}\;.
\end{equs}
Optimising over $s\in(0,1)$ so that $s^{\theta(\alpha-\bar\alpha)}=s^{\bar\eta/2}h^{(\eta-\bar\eta)/2}
\Leftrightarrow s = h^\gamma$
where $\gamma=\frac{\eta-\bar\eta}{2\theta(\alpha-\bar\alpha)-\bar\eta}$,
we see that
\begin{equ}
\heatgr{\CP_hA - A}_{\bar\alpha,\theta}
\lesssim
h^{\gamma\theta(\alpha-\bar\alpha)}
\heatgr{A}_{\alpha,\theta}\;.
\end{equ}
Note that $\eta\leq 2\theta(\alpha-\bar\alpha)$ (with strict inequality if $\theta>0$ and $\bar\alpha<1$),
hence $\gamma\nearrow1$ as $\bar\eta\to-\infty$,
from which the conclusion follows.
\end{proof}

\begin{lemma}\label{lem:fancynorm_heat_flow}
Let $\kappa,h\in(0,1)$,
$\beta,\eta\leq 0$,
$\delta<\bar\delta$ with $\bar\delta\geq0$,
and $X\in\CD'(\T^3,E)$.
Then		
\begin{equ}
\fancynorm{\CP_hX ; X}_{\beta,\bar\delta}
\lesssim
h^{\kappa(\bar\delta-\delta)}
(\fancynorm{X}_{\beta,\delta} + |X|_{\CC^\eta}^2)\;.
\end{equ}
where the proportionality constant depends only on $\eta,\kappa$.
\end{lemma}

\begin{proof}
For all $0<t\leq s<1$
\begin{equ}
t^{\bar\delta}|\CN_{t+h}X-\CN_tX|_{\CC^\beta}
\leq s^{\bar\delta-\delta}t^\delta(|\CN_{t+h}X|_{\CC^\beta}+|\CN_tX|_{\CC^\beta})
\lesssim
s^{\bar\delta-\delta}(\fancynorm{X}_{\beta,\delta}+|X|_{\CC^\eta}^2)\;.
\end{equ}
On the other hand,
for $0<s<t<1$ and any $\bar\eta\leq\eta$
\begin{equs}
t^{\bar\delta}|\CN_{t+h}X-\CN_tX|_{\CC^\beta}
&\leq
|\CN_{t+h}X-\CN_tX|_{\infty}\label{eq:drop_t}
\\
&\lesssim
t^{(\bar\eta+\eta-1)/2}(|\CP_hX|_{\CC^\eta}+|X|_{\CC^\eta})|\CP_hX-X|_{\CC^{\bar\eta}}
\\
&\lesssim
s^{(\bar\eta+\eta-1)/2}h^{(\eta-\bar\eta)/2}|X|_{\CC^\eta}^2\;,
\end{equs}
where we used Lemma~\ref{lem:etas_to_beta=zero} in the second bound.
Optimising over $s\in(0,1)$ so that $s^{\bar\delta-\delta}=s^{(\bar\eta+\eta-1)/2}h^{(\eta-\bar\eta)/2}
\Leftrightarrow s = h^\gamma$
where $\gamma=\frac{\eta-\bar\eta}{2(\bar\delta-\delta)+1-\eta-\bar\eta}$,
we see that
\begin{equ}
\fancynorm{\CP_hX ; X}_{\beta,\bar\delta}
\lesssim
h^{\gamma(\bar\delta-\delta)}
(\fancynorm{X}_{\beta,\delta} + |X|_{\CC^\eta}^2)\;.
\end{equ}
Note that $\eta < 2(\bar\delta-\delta)+1-\eta$,
hence $\gamma\nearrow1$ as $\bar\eta\to-\infty$,
from which the conclusion follows.
\end{proof}

\begin{remark}
One could sharpen Lemma~\ref{lem:fancynorm_heat_flow} by not immediately dropping $t^{\bar\delta}$ in~\eqref{eq:drop_t},
but this only improves the result in the case $2\eta\geq 1-2\delta$,
which is not the case in which we will apply this lemma.
\end{remark}

\subsection{Smoothened gauge-invariant observables}
\label{subsec:regularised_Wilson}

We conclude this section with a discussion on the gauge-invariant observables defined on the space $\init$ (and therefore on $\state$).

\begin{definition}
A function $\CO\colon\CC^\infty(\T^3,E)\to \R$ is called \emph{gauge-invariant}
if $\CO$ factors through a function $\CO\colon\mfO^\infty\to\R$ (which we denote by the same symbol),
i.e.\
if $\CO(X)=\CO(Y)$ whenever $X\sim Y$.
A collection $\{\CO^i\}_{i\in I}$ of gauge-invariant observables on $\CC^\infty(\T^3,E)$ is called \emph{separating} if for all $X,Y\in\CC^\infty(\T^3,E)$
\begin{equ}
\CO^i(X)=\CO^i(Y) \;\;\forall i\in I\; \Leftrightarrow\; X\sim Y\;.
\end{equ}
The same definitions apply with $\CC^\infty(\T^3,E)$ replaced by $\init$.

If $\CO\colon\CC^\infty(\T^3,E)\to \R$
is a gauge-invariant observable then, recalling the YMH flow $\tilde\ymhflow\colon(0,\infty)\to \mfO^\infty$ from Definition~\ref{def:ymhflow}, we define its ``smoothened'' version $\CO_t\colon\init\to\R$ for $t>0$ by
\begin{equ}
\CO_t \eqdef \CO\circ\tilde \ymhflow_t\;.
\end{equ}
\end{definition}

For $X,Y\in\init$, recall that $X\sim Y$  if and only if $\tilde \ymhflow_t(X)=\tilde \ymhflow_t(Y)$ for some $t>0$ due to Proposition~\ref{prop:back_unique}.
Hence, given a separating collection $\{\CO^i\}_{i\in I}$ of gauge-invariant observables on $\CC^\infty(\T^3,E)$ and any $t>0$,
$\{\CO^i_t\}_{i\in I}$ is a separating collection of gauge-invariant observables on $\init$.

In the rest of this subsection, we describe a natural separating collection of gauge-invariant observables on $\CC^\infty(\T^3,E)$.\footnote{This construction is a simple adaptation of a well-known construction for pure gauge fields; for completeness, and since it is easy to do so, we give the details.}
For $x\in\T^3$ denote by $\CP_{x}$ the set of piecewise smooth paths $\gamma\colon[0,1]\to\T^3$ such that $\gamma_0=x$
and by $\CL_x\subset\CP_x$ the subset of loops, i.e.\ those $\gamma\in\CP_x$ such that $\gamma_1=x$.

Given $\gamma\in\CP_{x}$ for some $x\in\T^3$
and
$A\in\Omega\CC^\infty$,
the holonomy of $A$ along $\gamma$ is defined as 
$\hol(A,\gamma)\eqdef y_1$ where
$y$ satisfies the ODE
$\mrd y_t = y_t A(\gamma_t)\dot\gamma_t$ with $y_0=\id$
(in fact $A\in\Omega_\alpha$ for $\alpha\in(\frac12,1]$ suffices to define the holonomy as shown in~\cite[Thm.~2.1(i)]{CCHS2d}.

\begin{definition}
For $n\geq1$ let $\CA_{n}$ denote the set of all continuous functions $f\colon G^n\times\higgsvec^n\to\R$
such that for all $g,h_1,\ldots, h_n\in G$
and $\phi_1,\ldots,\phi_n\in\higgsvec$
\begin{equ}
f(gh_1g^{-1},\ldots,gh_ng^{-1},
g\phi_1,\ldots,g\phi_n)
=
f(h_1,\ldots,h_n,
\phi_1,\ldots,\phi_n)\;.
\end{equ}
Then, for any $x\in\T^3$, let $\mbW_x$ denote the set of all functions on $\CC^\infty(\T^3,E)$
of the form
\begin{equ}
(A,\Phi)\mapsto f(\hol(A,\ell^1), \ldots,\,\hol(A,\ell^n),\hol(A,\gamma^1)\Phi(\gamma^1_1),\ldots, \hol(A,\gamma^n)\Phi(\gamma^n_1))\;,
\end{equ}
where $n\geq 1$, $f\in\CA_n$, $\ell^1,\ldots,\ell^n\in\CL_{x}$, and $\gamma^1,\ldots,\gamma^n\in\CP_x$.
\end{definition}

\begin{remark}\label{rem:gaueg_invar}
Every $\CO\in\mbW_x$ is gauge-invariant
due to the identity $\hol(A^g,\gamma)=g(\gamma_0)\hol(A,\gamma)g(\gamma_1)^{-1}$.
In the absence of the Higgs field (i.e.\ $\higgsvec=\{0\}$ or when $f$ is independent of its $\higgsvec^n$-argument),
$\mbW_x$ 
is a separating collection for $\Omega\CC^\infty/{\sim}$, see~\cite[Prop.~2.1.2]{Sengupta92}.
\end{remark}

\begin{remark}
Elements of $\mbW_x$ of the form $(A,\Phi)\mapsto f(\hol(A,\ell))$, where $\ell\in\CL_x$ and $f\colon G\to\R$ is a class function,
are known as Wilson loops.
For many choices of $G$, Wilson loops also form a separating collection for $\Omega\CC^\infty$,
but this is harder to show than for $\mbW_x$;
see~\cite{Durhuus80,Sengupta94,Levy04} where this question is investigated.
\end{remark}

\begin{remark}
Elements of $\mbW_x$ of the form $(A,\Phi)\mapsto \scal{\Phi(x),\hol(A,\gamma)\Phi(\gamma_1)}_{\higgsvec}$
are known as string observables.
\end{remark}

We now show that $\mbW_x$ is a separating collection for all of $\CC^\infty(\T^3,E)/{\sim}$.

\begin{proposition}
Consider $x\in\T^3$ and
$(A,\Phi),(B,\Psi)\in\CC^\infty(\T^3,E)$.
The following statements are equivalent.
\begin{enumerate}[label=(\roman*)]
\item\label{pt:sim} $(A,\Phi)\sim(B,\Psi)$.
\item\label{pt:CO} $\CO(A,\Phi)=\CO(B,\Psi)$ for all $\CO\in\mbW_x$.
\end{enumerate}
\end{proposition}

\begin{proof}
This is similar to~\cite[Prop.~2.1.2]{Sengupta92}.
The direction~\ref{pt:sim}$\Rightarrow$\ref{pt:CO}
is Remark~\ref{rem:gaueg_invar}.
Suppose now that~\ref{pt:CO} holds.
For $\gamma\in\CP_x$
define $\Gamma_\gamma\eqdef\{g\in G\,:\,g\hol(A,\gamma)\Phi(\gamma_1)=\hol(B,\gamma)\Psi(\gamma_1)\}$,
and for $\ell\in\CL_x$ define
$\Lambda_\ell\eqdef \{g\in G\,:\,g\hol(A,\ell)g^{-1}=\hol(B,\ell)\}$.
Then~\ref{pt:CO} implies
\begin{equ}\label{eq:inter_nonempty}
\Big(\bigcap_{\gamma\in I}\Gamma_\gamma\Big)\cap \Big(\bigcap_{\ell\in J}\Lambda_\ell\Big)\neq\emptyset 
\end{equ}
for any finite subsets $I\subset\CP_x$, $J\subset\CL_x$.
This is due to the fact that if 
\[
f(h_1,\ldots,h_n, \phi_1,\ldots,\phi_n)
=f(\bar h_1,\ldots,\bar h_n,\bar\phi_1,\ldots,\bar\phi_n)
\qquad
\forall f\in \CA_{n}\;,
\]
then
there exists $g\in G$ such that $h_i=g \bar h_i g^{-1}$ and $\phi_i = g\bar \phi_i$ for all $1\le i\le n$.
Since $\Gamma_\gamma$ and $\Lambda_\ell$ are compact,~\eqref{eq:inter_nonempty} holds for any countable $I\subset\CP_x$, $J\subset\CL_x$ by Cantor's intersection theorem.
Since all objects are smooth and 
$\CP_x$ and $\CL_x$ are separable (say in the $C^1$-metric), 
one can choose 
$I\subset\CP_x$, $J\subset\CL_x$ to be countable dense subsets, and it follows that there exists
$g\in (\bigcap_{\gamma\in \CP_x}\Gamma_\gamma)\cap (\bigcap_{\ell\in \CL_x}\Lambda_\ell)$.

Define $h\in\CC^\infty(\T^3,G)$ by the identity $g\hol(A,\gamma)=\hol(B,\gamma)h(y)$, where $\gamma$ is any path in $\CP_{x}$ such that $\gamma_1=y$ (note that $h(y)$ is independent of the choice of $\gamma$).
Then $\hol(B,\gamma)h(\gamma_1)\Phi(\gamma_1)=g\hol(A,\gamma)\Phi(\gamma_1)=\hol(B,\gamma)\Psi(\gamma_1)$
for all $\gamma\in\CP_x$, and thus $\Psi=\Phi^h$.
Similarly $h(z)\hol(A,\gamma)h(y)^{-1}=\hol(B,\gamma)$ for all $z,y\in\T^3$ and $\gamma\in\CP_z$ with $\gamma_1=y$,
from which it follows that $A^h=B$.
\end{proof}

\section{Stochastic heat equation}
\label{sec:SHE}

The main goal of this section is to show that the stochastic heat equation (SHE)
is a continuous process with values in $(\state,\Sigma)$.
We analyse separately the two parts $\Theta$ and $\heatgr{\cdot}_{\alpha,\theta}$
which add to give $\Sigma$.

\subsection{The first half}
\label{subsec:1st-half}

Throughout this subsection, let $\xi$ denote either 
the $E$-valued space-time white noise $((\xi_i)_{i=1}^3,\zeta)$
 on $\R\times \T^3$
or its smoothened version $((\xi_i^\e)_{i=1}^3,\zeta^\e)$.
Then there exists $C_\xi > 0$ such that
$
\E[|\scal{\xi,\phi}|^2] \leq C_\xi|\phi|_{L^2(\R\times\T^3)}^2
$
for all smooth compactly supported $\phi \colon \R\times\T^3\to E$.
Let $\Psi$ solve $(\partial_t - \Delta) \Psi = \xi$ on $\R_+ \times \T^3$ with initial condition $\Psi(0)\in\CD'(\T^3,E)$. We will often write $\Psi_t=\Psi(t)$ for short.

\begin{definition}\label{def:space-CI'}
We say that $(\eta,\beta,\delta)\in \R^3$ satisfies condition~\eqref{eq:CI'} if
\begin{equation}\label{eq:CI'}
\begin{split}
&\eta\in(-2/3,-1/2)\;,\quad
\beta \in (-1/2,0]\;,\quad
\delta\in(3/4,1)\;,\\
&\hat\beta \eqdef \beta+2(1-\delta) < \eta+1/2 \;.
\end{split}\tag{$\CI'$}
\end{equation} 
\end{definition}

We introduce the homogenous version of the metric $\Theta$ by
\begin{equ}
\tilde\Theta_{\eta,\beta,\delta}(X,Y) = |X-Y|_{\CC^\eta} \vee \fancynorm{X;Y}_{\beta,\delta}^{1/2}\;.
\end{equ}

\begin{proposition}\label{prop:SHE_Nt_estimates}
Assume~\eqref{eq:CI'} and that $\Psi_0\in \init_{\eta,\beta,\delta}$.
Then for all  $T > 0$, there is a modification of $\Psi$
such that $\Psi\in \CC([0,T],\init_{\eta,\beta,\delta})$
and, for every $\bar\eta<\eta$, $\bar\beta<\beta$ and $\bar\delta>\delta$ there exists $\kappa>0$ such that $p\in [1,\infty)$,
\begin{equ}[eq:SHE_Theta_Holder]
\Big( \E \Big|\sup_{0\leq s<r \leq T} |r-s|^{-\kappa}\tilde\Theta_{\bar\eta,\bar\beta,\bar\delta}(\Psi_r,\Psi_s)\Big|^p\Big)^{1/p}\lesssim \fancynorm{\Psi_0}_{\beta,\delta}^{1/2}
+|\Psi_0|_{\CC^{\eta}}
+C_{\xi}^{1/2}
\end{equ}
and
\begin{equ}[eq:SHE_Theta_zero]
\Big( \E \sup_{r\in [0,T]} r^{-\kappa}\tilde\Theta_{\eta,\beta,\delta}(\Psi_r,\CP_r \Psi_0)^p \Big)^{1/p} \lesssim C_\xi^{1/4}(C_\xi^{1/4} + |\Psi_0|_{\CC^\eta}^{1/2})\;.
\end{equ}
\end{proposition}

For the proof, we require several lemmas.
Below we denote by $\hstar$ the spatial convolution (in contrast with $*$ for space-time convolution).
Below we write $G$ for the heat kernel associated to the heat semigroup $\CP$.\footnote{Writing $G$ for both heat kernel and the Lie group will not cause any confusion since it is clear from the context.}

\begin{lemma}\label{lem:CPf}
Let $d\geq 1$, $\gamma\in (0,d)$ and
$\alpha\in (0,\gamma)$.
Let  $\CP_t^{(k)}$ be the 
$k$th spatial derivative of the heat semigroup $\CP$
for  multiindex $|k|\le 1$.
Let  $f\in \CD'(\T^d)$ which is smooth except for the origin,
and satisfy the bound  $|f(x)|\lesssim |x|^{-\gamma}$ 
uniformly in $x\in \T^d\backslash \{0\}$.
Then for any $\omega\in \big(0,\min(\frac{\alpha}2,\frac12) \big)$ one has
\begin{equs}
|t^{\frac{\alpha+|k|}{2}} (\CP_t^{(k)}  f) (x)|
& \lesssim  |x |^{\alpha-\gamma}\;,
\\
|(t^{\frac{\alpha+|k|}{2}}\CP_t^{(k)}  f 
- \bar t^{\frac{\alpha+|k|}{2}}\CP_{\bar t}^{(k)}  f ) (x)|
& \lesssim
|t-\bar t|^\omega 
 |x |^{\alpha-\gamma-2\omega}\;,
\end{equs}
uniformly in $t\in (0,1)$ and $x\in \T^d\setminus \{0\}$.
\end{lemma}

\begin{proof}
By the bound $|G^{(k)}(t,x)|\lesssim \big(|x|+\sqrt{t}\big)^{-d-|k|} $
one has
\begin{equs}
|(\CP_t^{(k)}  f) (x)|
&\lesssim \Big(\big(|\cdot |+\sqrt{t}\big)^{-d-|k|}  \hstar |f| \Big)(x)
\lesssim
 t^{-\frac{\alpha+|k|}{2}}\Big( |\cdot|^{\alpha-d}  \hstar |f| \Big)(x)
\\
& \lesssim t^{-\frac{\alpha+|k|}{2}} |x |^{\alpha-\gamma}
\end{equs}
uniformly in $t\in (0,1)$,  where in the last step we applied \cite[Lem.~10.14]{Hairer14}.
The second bound follows in the same way using 
the bound on  $\partial_t G^{(k)}(t,x)$ by 
$\big(|x |+\sqrt{t}\big)^{-d-|k|-2}$ and interpolation.
\end{proof}

\begin{lemma}\label{lem:prob-obs} 
Suppose that $\tilde \Psi$ solves SHE with initial condition $\tilde\Psi(0)=0$.
Let $T>0$,
 $\delta\in (\frac12,1)$,
 and $\beta < 2(\delta-1)$.
 Then for $\kappa \eqdef \min\{2(\delta-1)-\beta,\frac12\}$,
 one has
\begin{equ}[e:N2diff]
\E \Big|
\Big(t^\delta \CN_t \tilde\Psi_r - \bar{t}^\delta\CN_{\bar t} \tilde\Psi_s \Big)
(\phi^\lambda)
\Big|^2
 \,\lesssim\, C_\xi^2\, \lambda^{2\beta}
(|t-\bar t|+|r-s|)^{\kappa}
\end{equ}
uniformly in $t,\bar t\in (0,1)$, $r,s\in[0,T]$, and $\lambda\in (0,1)$,
where $\phi^\lambda = \lambda^{-3}\phi(\cdot / \lambda)$ and $\phi\in\CC^\infty(\T^3)$ with support in $\{z\in\T^3\,:\,|z|\leq \frac14\}$.
\end{lemma}


\begin{proof}
Define $C_{r,s}(x) = \E \langle \tilde\Psi(r,0),\tilde\Psi(s,x)\rangle_E$.
For any $t,\bar t, r,s\ge 0$ and $x\in \T^3$, recalling $[3]=\{1,2,3\}$, we write
\begin{equs}
J_{t,\bar t,x}^{(1)}(r,s) &\eqdef 
\sum_{k\in [3]}(t\bar t)^\delta\,(G_t \hstar C_{r,s}\hstar \partial_k G_{\bar t})(x)\,
(\partial_k G_t \hstar C_{r,s}\hstar G_{\bar t})(x)\;,
\\
J_{t,\bar t,x}^{(2)}(r,s) &\eqdef 
\sum_{k\in [3]} (t \bar t)^\delta\,(G_t \hstar  C_{r,s}\hstar G_{\bar t})(x)\,
(\partial_k G_t \hstar C_{r,s} \hstar \partial_k G_{\bar t})(x)\;.
\end{equs}
Here note that 
$\CN_t (\tilde\Psi_r)$ has zero mean 
thanks to the spatial derivative in its definition
\eqref{eq:CN_def}.
By Wick's theorem for $x,\bar x\in \T^3$ one has
\begin{equs}[e:corr-N-rs]
{} & \E \Big\langle
\big( t^\delta\CN_t \tilde\Psi_r - \bar t^\delta\CN_{\bar t} \tilde\Psi_s \big)(x),
\big( t^\delta\CN_t \tilde\Psi_r - \bar t^\delta\CN_{\bar t} \tilde\Psi_s \big)(\bar x)
\Big\rangle_{E\otimes E^3}
\\
&=\sum_{\ell=1}^2 
\Big(
J_{t,t,x-\bar x}^{(\ell)}(r,r)
-J_{t,\bar t,x-\bar x}^{(\ell)}(r,s)
-J_{\bar t, t,x-\bar x}^{(\ell)}(s,r)
+J_{\bar t,\bar t,x-\bar x}^{(\ell)}(s,s)\Big)\;.
\end{equs}
It is easy to check that
\begin{equ}\label{eq:C_bounds}
|C_{r,s}(x)| \lesssim C_\xi |x|^{-1}\;,
\qquad
|(C_{r,r}-C_{r,s})(x)| 
\lesssim 
C_\xi
|r-s|^{1/2} |x|^{-2}\;,
\end{equ}
uniformly in $r,s\in [0,T]$ and $x\in \T^3\setminus\{0\}$.
Since $\kappa\in(0,\frac12]$,
interpolation between the two bounds in~\eqref{eq:C_bounds} implies
\begin{equ}
|(C_{r,r}-C_{r,s})(x)| 
\lesssim 
C_\xi
|r-s|^{\kappa} |x|^{-1-2\kappa}\;.
\end{equ}
Applying  Lemma~\ref{lem:CPf}
twice, with $\gamma=1$
and then with $\gamma=1-\alpha$ for any $\alpha\in(0,\frac12)$ therein, one has, for any $\tau,\bar\tau\in \{t,\bar t\}$, $\iota,\bar\iota\in \{r,s\}$,  $i,j \in \{0,1\}$,
\begin{equs}[e:PCP]
|(\tau^{\frac{\alpha+i}{2}}G_\tau^{(i)})
 \hstar C_{\iota,\bar\iota}\hstar
(\bar\tau^{\frac{\alpha+j}{2}}  G_{\bar\tau}^{(j)})(x)|
& \lesssim
C_\xi\,
|x|^{2\alpha -1}\;,
\\
|(\tau^{\frac{\alpha+i}{2}} G_\tau^{(i)} )
\hstar (C_{r,r}-C_{r,s})\hstar 
(\bar\tau^{\frac{\alpha+j}{2}} G_{\bar\tau}^{(j)})(x)|
& \lesssim
C_\xi\,
|x|^{2\alpha -1-2\kappa} |r-s|^{\kappa} \;,
\\
|(\tau^{\frac{\alpha+i}{2}}G_\tau^{(i)} )
\hstar C_{\iota,\bar\iota}\hstar
 (t^{\frac{\alpha+j}{2}}G_{t}^{(j)}-\bar{t}^{\frac{\alpha+j}{2}}G_{\bar t}^{(j)})(x)|
& \lesssim
C_\xi\,
|x|^{2\alpha -1-2\kappa} |t-\bar t|^\kappa\;.
\end{equs}
Here $G^{(0)}=G$ and $G^{(1)}$ denotes $\partial_k G$ for a generic $k\in [3]$.
Choose $\alpha\eqdef\delta-\frac12$.
By these bounds and obvious telescoping,
one then has the bound
\begin{equ}
|J_{t,t,x-\bar x}^{(\ell)}(r,r)
-J_{t,\bar t,x-\bar x}^{(\ell)}(r,s)|
\lesssim
C_\xi^2\,
(|t-\bar t|+|r-s|)^{\kappa}  |x-\bar x|^{4\delta-4-2\kappa}
\end{equ}
for $\ell\in \{1,2\}$.
The same bound holds with $r,s$ and $t,\bar t$ swapped
and thus holds for \eqref{e:corr-N-rs}.

It remains to integrate this bound
on \eqref{e:corr-N-rs} against the test functions $\phi^\lambda$.
Bounding $\phi^\lambda$
by constant $\lambda^{-3}$
over its support which has diameter proportional to
$\lambda$,
one obtains that the integral of $\eqref{e:corr-N-rs} $ against $\phi^\lambda(x)\phi^\lambda(\bar x)$
is bounded above by a multiple of 
$C_\xi^2 (|t-\bar t|+|r-s|)^{\kappa} $ times $\lambda^{4\delta-4-2\kappa} \leq  \lambda^{2\beta}$,
where used that $\kappa\leq 2(\delta-1)-\beta$ in the final bound.
\end{proof}

\begin{lemma}\label{lem:determ-part}
Let 
$\beta,\eta\leq 0$,
$\delta<\bar\delta$ with $\bar\delta\geq0$,
and $\Psi_0\in\CD'(\T^3,E)$.
Then		for  
$\kappa\in(0,2(\bar\delta-\delta))$,
writing 
\begin{equ}[e:def-CM]
\CM_{t,\bar t,r,s,\bar\delta}(\Psi_0) \eqdef t^{\bar\delta} \CN_t (\CP_r\Psi_0)
-{\bar t}^{\bar\delta} \CN_{\bar t} (\CP_s\Psi_0)\;,
\end{equ}
 one has 
\begin{equ}[e:determ-part]
| \CM_{t,\bar t,r,s,\bar\delta}(\Psi_0)|_{\CC^\beta}
\lesssim
(|t-\bar t|+|r-s|)^{\frac\kappa2}
(\fancynorm{\Psi_0}_{\beta,\delta} + |\Psi_0|_{\CC^\eta}^2)
\end{equ}
where the proportionality constant depends only on 
$\beta,\eta,\delta,\bar\delta,\kappa$.
\end{lemma}

\begin{proof}
By Lemma~\ref{lem:fancynorm_heat_flow},
\begin{equ}
t^{\bar\delta} |\CN_t (\CP_r\Psi_0)-\CN_t (\CP_s\Psi_0) |_{\CC^\beta}
\lesssim
| r-s|^{\tilde\kappa(\bar\delta-\delta)}
(\fancynorm{\Psi_0}_{\beta,\delta} + |\Psi_0|_{\CC^\eta}^2)
\end{equ}
for any $\tilde\kappa\in(0,1)$ and
$\delta<\bar\delta$,
where we used $\fancynorm{\CP_s\Psi_0}_{\beta,\delta}\lesssim
\fancynorm{\Psi_0}_{\beta,\delta}+|\Psi_0|_{\CC^\eta}^2$ and similarly for the $\CC^\eta$ norm.
By assumption on the indices,
$ |r-s|^{\tilde\kappa(\bar\delta-\delta)}\lesssim |r-s|^{\frac\kappa2}$ for $\tilde\kappa$ sufficiently close to $1$,
where we used $\kappa<2(\bar\delta-\delta)$.
Similarly,
\begin{equs}
t^{\bar\delta} |\CN_t (\CP_s\Psi_0) - \CN_{\bar t} (\CP_s\Psi_0) |_{\CC^\beta}
&=
t^{\bar\delta} |\CN_t (\CP_s\Psi_0) - \CN_{t} (\CP_{s+\bar t-t}\Psi_0) |_{\CC^\beta}
\\
&\lesssim
| t-\bar t|^{\frac\kappa2}
(\fancynorm{\Psi_0}_{\beta,\delta} + |\Psi_0|_{\CC^\eta}^2)\;.
\end{equs}
Finally, assuming $t\le \bar t$ without loss of generality,
\begin{equ}
|(t^{\bar\delta}-{\bar t}^{\bar\delta})
\CN_{\bar t} (\CP_s\Psi_0)|_{\CC^\beta}
\lesssim
| t-\bar t|^{\frac\kappa2}
|\bar t^{\bar\delta-\frac\kappa2} \CN_{\bar t}( \CP_s\Psi_0)|_{\CC^\beta}
\lesssim
| t-\bar t|^{\frac\kappa2}
\fancynorm{\Psi_0}_{\beta,\delta}
\end{equ}
where we again used $\kappa<2(\bar\delta-\delta)$ in the last bound.
Combining the above three bounds we obtain \eqref{e:determ-part}.
\end{proof}

\begin{lemma}\label{lem:mom-Psi-diff}
Assume~\eqref{eq:CI'} and that $\delta<\bar\delta < 1$ and
 $\beta> \bar\beta>-\frac12$.
Let
$0<\kappa < \min(2(\bar\delta-\delta),\eta-\beta+2\delta-3/2)$, 
 and $T>0$. Then for all  $r,s\in[0,T]$ and $p\in [1,\infty)$
\begin{equs}
\{\E 
 |t^{\bar\delta} \CN_t \Psi_r -\bar{t}^{\bar\delta} \CN_{\bar t} \Psi_s|_{\CC^{\bar\beta}}^{2p}\}^{1/p}
\lesssim
(|r-s|+|t-\bar t|)^\kappa
( \fancynorm{\Psi(0)}_{\beta,\delta}^2
+|\Psi(0)|_{\CC^\eta}^4
+C_{\xi}^{2})\;,
\end{equs}
where the proportionality constant depends only on 
$\eta,\beta,\bar\beta,\delta,\bar\delta,\kappa,T$.
\end{lemma}

\begin{proof}
Write $\Psi_s=\tilde\Psi_s+ \CP_s \Psi_0$
where $\tilde\Psi$ solves SHE with  $\tilde\Psi(0)=0$,
and assume $r>s$.
Since $\tilde\Psi$ and 
$\CN_t (\tilde\Psi_r)$ are mean-zero,
it is straightforward to check that
\begin{equs}[e:corr-N-diff]
 \E &\Big\langle
 \big(t^{\bar\delta}\CN_t (\Psi_r)-\bar t^{\bar\delta}\CN_t (\Psi_s) \big)(x),
\big(t^{\bar\delta}\CN_t (\Psi_r)-\bar t^{\bar\delta}\CN_t (\Psi_s) \big)(\bar x)
\Big\rangle_{E\otimes E^3}
\\
&= \eqref{e:corr-N-rs}
+ \langle
\CM_{t,\bar t,r,s,\bar\delta}(\Psi_0)(x),
\CM_{t,\bar t,r,s,\bar\delta}(\Psi_0)(\bar x)
\rangle_{E\otimes E^3}
\\
&\qquad \quad\;+
\sum_{i,j\in \{0,1\}}
\Big(
J_{t,t,x,\bar x}^{(i,j)}\,(r,r)
-J_{t,\bar t, x,\bar x}^{(i,j)}\,(r,s)
-J_{\bar t, t,x,\bar x}^{(i,j)}\,(s,r)
+J_{\bar t,\bar t,x,\bar x}^{(i,j)}\,(s,s)\Big)
\end{equs}
where in the first term on the right-hand side we replace $\delta$ in 
\eqref{e:corr-N-rs} by $\bar\delta$, 
and $\CM_{t,\bar t,r,s,\bar\delta}(\Psi_0)$ is as in \eqref{e:def-CM}.
Here for $i,j\in \{0,1\}$ we define
\begin{equs}
J_{t,\bar t, x,\bar x}^{(i,j)}\,&(r,s)
\\
&\eqdef
(t\bar t)^{\bar\delta}\,
\E \langle
\CP_{t+r}^{(i)}  \Psi_0(x)\otimes
\CP_t^{(1-i)} \tilde\Psi_r(x) ,
 \CP_{\bar t+s}^{(j)} \Psi_0(\bar x) \otimes
\CP_{\bar t}^{(1-j)} \tilde\Psi_s(\bar x) 
\rangle_{E\otimes E^3}
\end{equs}
where $\CP^{(i)}$ denotes the $i$-th spatial derivative of $\CP$
as in the proof of Lemma~\ref{lem:prob-obs}.

By Lemma~\ref{lem:prob-obs},
the integral of \eqref{e:corr-N-rs}
against 
$\phi^\lambda(x)\phi^\lambda(\bar x)$
 satisfies the  bound
\eqref{e:N2diff}.
By 
Lemma~\ref{lem:determ-part},
the integration of the
second term on the right-hand side of \eqref{e:corr-N-diff}
against $\phi^\lambda(x)\phi^\lambda(\bar x)$
is bounded by
\begin{equ}[e:mom-2nd]
 \lambda^{2\beta}
(|t-\bar t|+|r-s|)^{\frac\kappa2}
(\fancynorm{\Psi_0}_{\beta,\delta}^2
 + |\Psi_0|_{\CC^\eta}^4)\;.
\end{equ}

It remains to bound the last term in \eqref{e:corr-N-diff}.
We can prove for each $i,j\in \{0,1\}$,
\begin{equs}[e:Jij-Jij]
{}&
\Big| J_{t,t,x,\bar x}^{(i,j)}(r,r) 
-J_{t,\bar t, x,\bar x}^{(i,j)}(r,s)\Big| 
\\
&\lesssim
C_\xi (|t-\bar t|+|r-s|)^\kappa (t\vee \bar t)^{2\bar\delta-\beta-\frac32+\eta-\kappa} |x-\bar x|^{2\beta}
|\Psi_0|_{\CC^\eta}^2\;.
\end{equs}
Indeed, if $\bar t\notin (\frac{t}{2},2t)$,
we bound the two terms on the left-hand side of \eqref{e:Jij-Jij} separately,
and by the first inequality in \eqref{e:PCP} (with $\alpha=\beta+\frac12$)
and $|\CP_{t+r}^{(i)}  \Psi_0|_\infty
 \lesssim t^{\frac{\eta-|i|}{2}} |\Psi_0|_{\CC^\eta}$,
we have
\[
|J_{t,\bar t, x,\bar x}^{(i,j)}(r,s)|
\lesssim
C_\xi  (t \bar t)^{\bar\delta-\frac{\beta}{2}-\frac34+\frac{\eta}{2}} |x-\bar x|^{2\beta}
|\Psi_0|_{\CC^\eta}^2\;.
\]
This is then bounded by the right-hand side of \eqref{e:Jij-Jij} since 
$\bar\delta-\frac{\beta}{2}-\frac34+\frac{\eta}{2}>0$ by~\eqref{eq:CI'}
and $(t\vee \bar t)^\kappa \lesssim |t-\bar t|^\kappa$ by our assumption on $t,\bar t$.
The other term $J_{t,t,x,\bar x}^{(i,j)}(r,r) $ is bounded in the same way.
On the other hand,
if $\bar t \in (\frac{t}{2},2t)$,
we bound the left-hand side of \eqref{e:Jij-Jij} by obvious telescoping.
Using
 \eqref{e:PCP} (with  $\alpha=\beta+\frac12$ or $\alpha=\beta+\frac12+\kappa$),
$|\CP_{t+r}^{(i)}  \Psi_0|_\infty
 \lesssim t^{\frac{\eta-|i|}{2}} |\Psi_0|_{\CC^\eta}$,
 and
\begin{equ}
 |\CP_{t+r}^{(i)} \Psi_0-\CP_{t+s}^{(i)} \Psi_0|_\infty
 \lesssim 
 |r-s|^\kappa t^{\frac{\eta-|i|}{2}-\kappa} |\Psi_0|_{\CC^\eta} \;,
\end{equ}
we obtain a bound by
\[
C_\xi (|t-\bar t|+|r-s|)^\kappa
\, t^{q} \, \bar t^{\bar q} |x-\bar x|^{2\beta}
|\Psi_0|_{\CC^\eta}^2\;,
\]
where $q,\bar q\in \R$ are such that 
$q+\bar q=2\bar\delta-\beta-\frac32+\eta-\kappa$.
By our assumption on $t,\bar t$, one has $t^{q} \, \bar t^{\bar q} \lesssim (t\vee \bar t)^{q+\bar q}$,
so we obtain  \eqref{e:Jij-Jij}.
Here we recall again that $\eta-\beta+2\bar\delta-3/2>0$ 
by~\eqref{eq:CI'}, and we assume $\kappa<
\eta-\beta+2\delta-3/2$.
So one has $(t\vee \bar t)^{2\bar\delta-\beta-\frac32+\eta-\kappa} \le 1$. 
 The same bound holds with $r,s$ and $t,\bar t$ swapped.
So the last line of \eqref{e:corr-N-diff} is bounded by 
\[
C_\xi (|t-\bar t|+|r-s|)^\kappa  |x-\bar x|^{2\beta}
|\Psi_0|_{\CC^\eta}^2\;,
\]
The integration of the last line of \eqref{e:corr-N-diff}
against test functions $\phi^\lambda(x)\phi^\lambda(\bar x)$
is then bounded by $C_\xi$ times \eqref{e:mom-2nd}.
The lemma
then follows by collecting the above bounds
and using the equivalence of Gaussian moments.
\end{proof}

\begin{proof}[or Proposition~\ref{prop:SHE_Nt_estimates}]
It is standard to show that
for all $\bar\eta \in(-\frac23, \eta)$,
$0<\kappa <\frac{\eta-\bar\eta}{2}$,
$p\geq 1$ and $T > 0$
\begin{equ}
\E\Big[ \sup_{0 \leq s < r \leq T} 
\frac{|\Psi(r)-\Psi(s)|_{\CC^{\bar\eta}}^p}{|r-s|^{p\kappa}}\Big]^{1/p}
 \lesssim 
|\Psi_0|_{\CC^\eta} + C_\xi^{1/2}\;,
\end{equ}
where the proportionality constant depends only on $p,\eta,\bar\eta,\kappa,T$.
Recalling Def.~\ref{def:init},
it remains to bound the moments of, for a suitable modificaiton of $\Psi$ and for 
$\delta<\bar\delta< 1$
and $\kappa>0$ sufficiently small,
\begin{equ}[e:sup-sup]
 \sup_{0 \leq s < r \leq T} 
\frac{\fancynorm{\Psi(r);\Psi(s)}_{\beta,\bar\delta}}{|r-s|^{\kappa}}
=
\sup_{t\in (0,1)}  \sup_{0 \leq s < r \leq T} 
\frac{
 |t^{\bar\delta}(\CN_t \Psi_r - \CN_t \Psi_s)|_{\CC^\beta}
}{|r-s|^{\kappa}}\;.
\end{equ}
By Lemma~\ref{lem:mom-Psi-diff} below
 and the classical Kolmogorov continuity theorem over $(0,1)\times [0,T]$, one has, for every $p\in [1,\infty)$ and $\kappa\in (0,(\bar\delta-\delta)\wedge
 \frac{\eta-\beta+2\delta-3/2}{2})$,
\begin{equ}
\bigg(\E\Big|\sup_{\substack{t,\bar t\in (0,1),s,r\in[0,T] \\ (t,r)\neq (\bar t,s)}}
\frac{
 |t^{\bar\delta} \CN_t \Psi_r - \bar t^{\bar\delta}  \CN_{\bar t} \Psi_s |_{\CC^{\bar\beta}}
}{(|t-\bar t|+|r-s|)^{\kappa}} \Big|^{2p}
\bigg)^{1/p}
\lesssim \fancynorm{\Psi(0)}_{\beta,\delta}^2
+|\Psi(0)|_{\CC^\eta}^4
+C_{\xi}^{2}
\end{equ}
where the proportionality constant depends only on $p,\eta,\bar\eta,\beta,\delta,\bar\delta,\kappa,T$,
which in particular bounds the moments of \eqref{e:sup-sup} by restricting to $t=\bar t$
and proves~\eqref{eq:SHE_Theta_Holder}.

To prove~\eqref{eq:SHE_Theta_zero}, writing $\Psi_r = \CP_r\Psi_0 + \tilde\Psi_r$ as in the proof of Lemma~\ref{lem:mom-Psi-diff}, we have
\begin{equ}
\sup_{r\in[0,T]}r^{-\kappa}|\Psi_r - \CP_r\Psi_0|_{\CC^\eta} = \sup_{r\in[0,T]}r^{-\kappa}|\tilde\Psi_r|_{\CC^\eta} \leq |\tilde\Psi|_{\CC^\kappa([0,T],\CC^\eta)}
\end{equ}
which has finite moments of all orders for $\kappa>0$ sufficiently small.
Moreover
\begin{equs}[eq:PrPsi0]
\fancynorm{\CP_r\Psi_0+\tilde\Psi_r;\CP_r\Psi_0}_{\beta,\delta}
&\leq
\sup_{t\in (0,1)}t^{\delta}
\big(
|\CP_t\tilde\Psi_r\otimes \nabla \CP_{t+r} \Psi_0|_{\CC^\beta}
\\
&\quad+
|\CP_{t+r}\Psi_0\otimes \nabla \CP_{t} \tilde\Psi_r|_{\CC^\beta}
+
|\CP_t\tilde\Psi_r\otimes \nabla \CP_{t} \tilde\Psi_r|_{\CC^\beta}
\big)
\end{equs}
By Lemma~\ref{lem:prob-obs}, Kolmogorov's theorem, and equivalence of Gaussian moments,
we have
\begin{equ}
\Big(
\E \Big|
\sup_{0\leq s<r\leq T} |s-r|^{-\kappa} |\CP_t\tilde\Psi_r\otimes \nabla \CP_{t} \tilde\Psi_r-\CP_t\tilde\Psi_s\otimes \nabla \CP_{t} \tilde\Psi_s|
\Big|^p
\Big)^{1/p}
\lesssim
C_\xi\;.
\end{equ}
Furthermore, by~\eqref{e:Jij-Jij},
\begin{multline*}
\Big(
\E \Big|
\sup_{0\leq s<r\leq T} |s-r|^{-\kappa} |\CP_t\tilde\Psi_r\otimes \nabla \CP_{t+r} \Psi_0-\CP_t\tilde\Psi_s\otimes \nabla \CP_{t+s} \Psi_0|
\Big|^p
\Big)^{1/p}
\\
\lesssim
C_\xi^{1/2} |\Psi_0|_{\CC^\eta}
\end{multline*}
and likewise with $\CP_t\tilde\Psi_r\otimes \nabla \CP_{t+r} \Psi_0$ replaced by $\CP_{t+r}\Psi_0\otimes \nabla \CP_{t} \tilde\Psi_r$.
Restricting to $s=0$ and combining with~\eqref{eq:PrPsi0} concludes the proof of~\eqref{eq:SHE_Theta_zero}.

For the final claim that $\Psi$ is in $\CC([0,T],\init)$, remark that $[0,T]\ni r\mapsto \Psi_r \in \init$ is continuous at $0$ due to~\eqref{eq:SHE_Theta_zero} and due to continuity of $[0,T]\ni r\mapsto \CP_r\Psi_0 \in \init$ (Proposition~\ref{prop:Theta_cont_zero}\ref{pt:heat_Theta}).
Continuity at $r>0$ follows from the fact that $\tilde\Psi \in \CC([0,T],\init)$ due to~\eqref{eq:SHE_Theta_Holder}
and from Lemma~\ref{lem:perturbation}\ref{pt:perturbation}
once we write $\Psi_r = \tilde\Psi_r + \CP_r\Psi_0$ where $\CP\Psi_0\in \CC((0,T],L^\infty)$.
\end{proof}

\subsection{The second half}
\label{subsec:2nd-half}

Throughout this subsection, let $\xi$ be
an $\mfg^3$-valued
Gaussian random distribution on $\R\times \T^3$.
Suppose there exists $C_\xi > 0$ such that
\begin{equ}\label{eq:covariance_bound}
\E[|\scal{\xi,\phi}|^2] \leq C_\xi|\phi|_{L^2(\R\times\T^3)}^2
\end{equ}
for all smooth compactly supported $\phi \colon \R\times\T^3\to \mfg^3$.
Let $\Psi$ solve $(\partial_t - \Delta) \Psi = \xi$ on $\R_+ \times \T^3$ with initial condition $\Psi(0)\in\CD'(\T^3,\mfg^3)$.

\begin{proposition}\label{prop:SHE_heatgr_estimates}
Let $0 < \bar\alpha < \alpha < \frac12$,
$\theta\in(0,1]$,
and
$0<\kappa <\theta\min\{\alpha-\bar\alpha,\frac{1-2\bar\alpha}{4}\}$.
Then for all $p\geq 1$ and $T > 0$
\begin{equ}
\E\Big[ \sup_{0 \leq s < t \leq T} \frac{\heatgr{\Psi(t)-\Psi(s)}_{\bar\alpha,\theta}^p}{|t-s|^{p\kappa}}\Big]^{1/p} \lesssim \heatgr{\Psi(0)}_{\alpha,\theta} + C_\xi^{1/2}\;,
\end{equ}
where the proportionality constant depends only on $p,\alpha,\bar\alpha,\theta,\kappa,T$.
\end{proposition}

We give the proof at the end of this subsection
after several preliminary results.
Recall the space of line segments $\mcX$
defined by~\eqref{eq:def_mcX}.
For a line $\ell = (x,v)\in \mcX$,
define the line integral along $\ell$ as the distribution
\begin{equ}
\scal{\delta_\ell,\psi} \eqdef \int_0^1 |v|\psi(x+tv) \,\mrd t\;.
\end{equ}

\begin{lemma}\label{lem:delta_ell_Sobolev}
For all $\kappa\in(\frac12,1)$,
there exists $C>0$ such that for all $t\in(0,1)$ and $\ell\in\CX$ 
\begin{equ}
|e^{t\Delta}\delta_\ell|_{H^{-\kappa}}^2 \leq  C t^{2\kappa-2}|\ell|\;.
\end{equ}
\end{lemma}

\begin{proof}
By translation and rotation\footnote{There is a bijection between distributions on $\T^d$ supported in a ball centered
at zero of radius $\frac14$ and
distributions on $\R^d$ (where Sobolev norms are rotation invariant) with the same property,
and the corresponding Sobolev norms on $\T^d$ and $\R^d$ are equivalent.}
invariance, we can assume $\ell=(0,r e_1)$.
Then 
for $k = (k_1,k_2,k_3) \in \Z^3$ with $k\neq 0$
\begin{equ}
\hat\delta_\ell(k)\eqdef \scal{\delta_\ell, e^{2\pi i \scal{k,\cdot}}}= \frac{e^{2\pi irk_1}-1}{2\pi i k_1}\;.
\end{equ}
Therefore
\begin{equs}
|e^{t\Delta}\delta_\ell|_{H^{-\kappa}}^2
&\asymp r^2 + \sum_{k\in\Z^3\setminus\{0\}} e^{-2t|k|^2}|\hat\delta_\ell(k)|^2 |k|^{-2\kappa}
\\
&\lesssim r^2 + \sum_{k\in\Z^3\setminus\{0\}} e^{-2t|k|^2}(r^2 \wedge k_1^{-2}) |k|^{-2\kappa}\;.
\end{equs}
If $|k|\geq t^{-1}$, then $e^{-2t|k|^2}\leq e^{-2|k|}$, which is summable so that
these values of $k$ make a contribution to the sum of order at most $r^2$ which in
turn is bounded by $rt^{2\kappa-2}$ since $r,t \le 1$.
For the rest of the sum we simply bound the exponential by $1$, so that is bounded above by
\begin{equ}
\sum_{\substack{k\in\Z^3\\0<|k|<t^{-1}}} (r^2 \wedge k_1^{-2}) (1\wedge |k_2^2+k_3^2|^{-\kappa})
\lesssim
\sum_{\substack{k\in\Z\\0<k<t^{-1}}} (r^2 \wedge k^{-2}) t^{2\kappa-2} \lesssim rt^{2\kappa-2}\;,
\end{equ}
as required. Here the second inequality holds since
when $r\le t$, the sum of
$r^2 \wedge k^{-2}$ over $0<k<t^{-1}$
is bounded by $r^2 t^{-1}\le r$,
and when $r> t$,
this sum is bounded by 
$\sum_{0<k\le r^{-1}} r^2 $ plus 
$\sum_{r^{-1}<k<t^{-1}} k^{-2}$ which is bounded by $r+(r-t)\lesssim r$.
The first inequality follows by splitting the sum into the regime $|k_1|\ge  |k_2|+|k_3|$ where we first sum
over $k_2,k_3$ and use $|k_1|\asymp |k|$,
and the regime $|k_1|\le |k_2|+|k_3|$ where we first sum over $k_1$ and use $|k_2|+|k_3|\asymp |k|$.
\end{proof}

For a triangle $P=(x,v,w)$, $x\in \T^3$, $v,w\in\R^3$ with $|v|,|w|\leq \frac14$,
define the surface integral along $P$ as the distribution
\begin{equ}
\scal{\delta_P,\psi} \eqdef \int_0^1\mrd s\int_{0}^{1-s}\mrd t |P|\psi(x+sv+tw) \;.
\end{equ}

\begin{lemma}\label{lem:delta_P_Sobolev}
For all $\kappa>0$, $t\in(0,1)$, and triangles $P$
\begin{equ}
|e^{t\Delta}\delta_P|_{H^\kappa}^2
\lesssim
t^{-\kappa}
|e^{\frac t2\Delta}\delta_P|_{L^2}^2 \lesssim |P|t^{-\frac12-\kappa}\;,
\end{equ}
where the first proportionality constant depends on $\kappa$ and the second is universal.
\end{lemma}

\begin{remark}
The bound $|e^{t\Delta}\delta_P|_{L^2}^2 \lesssim |P|t^{-\frac12}$
is likely optimal.
On the other hand, the bound $|e^{t\Delta}\delta_P|_{H^\kappa}^2\lesssim |P|t^{-\frac12-\kappa}$
is likely suboptimal but suffices for our purposes.
\end{remark}

\begin{proof}
The first bound follows from
the heat flow estimate $|e^{t\Delta}f|_{H^{\kappa}}^2 \lesssim t^{-\kappa}|f|_{L^2}$,
so it remains to show that
$|e^{t\Delta}\delta_P|_{L^2}^2 \lesssim |P|t^{-1/2}$. 
Moreover, it suffices to consider right-angled triangles, in which case we can assume $P=(0,re_1,he_2)$.
Note that $|e^{t\Delta}\delta_P|_{L^2} \leq |e^{t\Delta}\delta_R|_{L^2}$
where $\scal{\delta_R,\psi}\eqdef \int_0^1\mrd s\int_{0}^{1}\mrd t rh\psi(x+sre_1+the_2)$ is the surface integral along the rectangle with sides $re_1$ and $he_2$.
Then 
for $k \in \Z^3$
\begin{equ}
\hat\delta_R(k)\eqdef \scal{\delta_R, e^{2\pi i \scal{k,\cdot}}}=
\Big(
\frac{e^{2\pi irk_1}-1}{2\pi i k_1}
\Big)
\Big(
\frac{e^{2\pi ihk_2}-1}{2\pi i k_2}
\Big)\;,
\end{equ}
and therefore
\begin{equ}
|e^{t\Delta}\delta_R|_{L^2}^2
= \sum_{k\in\Z^3} e^{-2t|k|^2}|\hat\delta_R(k)|^2
\lesssim \sum_{k\in\Z^2} t^{-1/2}(r^2 \wedge k_1^{-2})(h^2\wedge k_2^{-2})\;.
\end{equ}
The sum over $k_1$ splits into $|k_1|\leq r^{-1}$ and $|k_1|>r^{-1}$, each of which are of order 
at most $r$, and likewise for $k_2$, which shows that $|e^{t\Delta}\delta_R|_{L^2}^2\lesssim rht^{-1/2} =2|P|t^{-1/2}$.
\end{proof}

\begin{lemma}\label{lem:SHE_ell_P_bound}
Suppose $\Psi(0)=0$ and let $T>0$.
Then for all $s\in[0,T]$, $t\in(0,1)$, $\kappa\in(0,\frac12)$, and $\ell\in\mcX$
\begin{equ}\label{eq:SHE_ell_bound}
\E[|e^{t\Delta} \Psi(s)(\ell)|^2]
\lesssim C_\xi s^{\kappa} |\ell| t^{-2\kappa}\;,
\end{equ}
and for all $\kappa\in(0,1)$ and triangles $P$
\begin{equ}\label{eq:SHE_P_bound}
\E[|e^{t\Delta} \Psi(s)(\partial P)|^2]
\lesssim C_\xi s^{\kappa}|P|t^{-\frac12-\kappa}\;,
\end{equ}
where $A(\partial P)$ denotes the line
integral of $A\in\Omega\CC^\infty$ around the boundary of $P$ under an arbitrary orientation (see~\cite[Def.~3.10]{CCHS2d})
and where the proportionality constants depend only on $\kappa$ and $T$.
\end{lemma}

\begin{proof}
Observe that, for any $A\in\Omega\CC^\infty$,
$A(\ell)=\sum_{i=1}^3 |v|^{-1}v_i\scal{A_i,\delta_\ell}$.
Furthermore,
\begin{equ}
\scal{e^{t\Delta}\Psi_i(s),\delta_\ell} = \int_{\R\times \T^3}\xi_i(u,y) \bone_{u \in [0,s]} [e^{(t+s-u)\Delta}\delta_\ell](y) \,\mrd u\,\mrd y\;.
\end{equ}
Hence, by~\eqref{eq:covariance_bound},
\begin{equ}
\E\big[\scal{e^{t\Delta}\Psi_i(s),\delta_\ell}^2\big] \leq C_\xi \int_{0}^s |e^{(t+u)\Delta}\delta_\ell|_{L^2}^2 \,\mrd u\;.
\end{equ}
Since
$|e^{(t+u)\Delta}\delta_\ell|_{L^2}^2 \lesssim u^{\kappa-1} |e^{t\Delta}\delta_\ell|_{H^{\kappa-1}}^2 \lesssim u^{\kappa-1}|\ell| t^{-2\kappa}$,
where we used Lemma~\ref{lem:delta_ell_Sobolev}  in the final bound, we obtain
\begin{equ}
\E\big[\scal{e^{t\Delta}\Psi_i(t),\delta_\ell}^2\big]
\lesssim C_\xi \int_0^s u^{\kappa-1}|\ell| t^{-2\kappa} \,\mrd u \lesssim C_\xi s^{\kappa} |\ell| t^{-2\kappa}\;,
\end{equ}
from which~\eqref{eq:SHE_ell_bound} follows.
To show~\eqref{eq:SHE_P_bound},
we can suppose that $P$ is in the $(x_1,x_2)$-plane.
Then,
by Stokes' theorem,
\begin{equ}
|e^{t\Delta}\Psi(s)(\partial P)|
=
|\scal{e^{t\Delta}(\partial_1 \Psi_2(s)-\partial_2 \Psi_1(s)), \delta_P}|\;.
\end{equ}
Furthermore,
\begin{equ}
\scal{e^{t\Delta}\partial_1 \Psi_2(s),\delta_P}
=
-\int_{\R\times \T^3}\xi_2(u,y)\bone_{u \in [0,s]} [e^{(t+s-u)\Delta}\partial_1 \delta_P](y)\, \mrd u\,\mrd y\;,
\end{equ}
and thus by~\eqref{eq:covariance_bound}, Lemma~\ref{lem:delta_P_Sobolev}, and the bound $|e^{u\Delta}\partial_if|_{L^2}^2 \lesssim u^{\kappa-1}|f|_{H^{\kappa}}^2$,
\begin{equs}
\E[|\scal{e^{t\Delta}\partial_1 \Psi_2(s),\delta_P}|^2]
&\leq C_\xi\int_{0}^s |e^{(t+u)\Delta}\partial_1\delta_{P}|_{L^2}^2 \,\mrd u
\\
&\lesssim C_\xi\int_{0}^s u^{\kappa-1}|P|t^{-\frac12-\kappa} \,\mrd u
\lesssim  C_\xi s^{\kappa}|P|t^{-\frac12-\kappa}\;.
\end{equs}
The same applies to $\scal{\partial_2 \Psi_1(s), \delta_P}$, from which the conclusion follows.
\end{proof}

\begin{lemma}\label{lem:Kolmogorov_fixed_time}
Let $A$ be a $\mfg$-valued stochastic process indexed by $\CX$
such that, for all joinable $\ell,\bar\ell\in\CX$, $A(\ell\sqcup\bar\ell) = A(\ell)+A(\bar\ell)$ 
almost surely.
Suppose that there exist $r,s\in(0,1)$, $p,M>0$, and $\alpha \in (0,1]$ such that for all $\ell\in\CX$ with $|\ell|<r$
\begin{equ}
\E[|A(\ell)|^p] \leq M |\ell|^{p\alpha}\;,
\end{equ}
and for all triangles $P$ with $|P| < s$
\begin{equ}
\E[|A(\partial P)|^p ] \leq M |P|^{p\alpha/2}\;.
\end{equ}
Then there exists a modification of $A$ (which we denote by the same letter) which is a.s.\ a continuous function on $\CX$.
Furthermore, for every $\bar\alpha \in (0,\alpha-\frac{24}{p})$, there exists $\lambda > 0$, depending only on $p,\alpha,\bar\alpha$, such that
\begin{equ}
\E[|A|_{\gr{\bar \alpha};<r}^p] \leq \lambda M ((s^{-6}+r^{-12})r^{p(\alpha-\bar\alpha)}+1)\;.
\end{equ}
\end{lemma}
\begin{proof}
For $\ell=(x,v)\in\mcX$, let us denote $\ell_i\eqdef x$ and $\ell_f\eqdef x+v$,
and define the metric $d$ on $\mcX$ by
\begin{equ}
d(\ell,\bar\ell) \eqdef |\ell_i-\bar\ell_i|\vee|\ell_f-\bar\ell_f|\;.
\end{equ}
Recall that $\ell,\bar\ell \in \mcX$ are called \emph{far} if $d(\ell,\bar\ell) > \frac14 (|\ell| \wedge |\bar\ell|)$.
For $(\ell,\bar\ell)\in\mcX^2$, let $T(\ell,\bar\ell)= |P_1|+|P_2|$ where $P_1,P_2$ are the triangles
$P_1 = (\ell_i,\ell_f,\bar\ell_f)$ and $P_2=(\ell_i,\bar\ell_f,\bar\ell_i)$.
Let $\Area(\ell_i,\bar\ell_i) = T(\ell,\bar\ell)\wedge T(\bar\ell,\ell)$.
Similar to~\cite[Def.~3.3]{CCHS2d},
we now define
$\rho \colon \mcX^2 \to [0,\infty)$ by
\begin{equ}
\rho(\ell,\bar\ell) \eqdef
\begin{cases}
|\ell| + |\bar\ell| &\text{if $\ell,\bar\ell$ are far,}
\\
|\ell_i - \bar\ell_i| + |\ell_f-\bar\ell_f|  + \Area(\ell,\bar\ell)^{1/2} &\text{otherwise.}
\end{cases}
\end{equ}
By the same argument as in~\cite[Remark~3.4]{CCHS2d}, note that there exists $C>0$ such that $\rho(a,b)\leq C(\rho(a,c)+\rho(c,b))$.

By definition, for any $\ell,\bar\ell\in\mcX$,
there exist $a,b\in\mcX$ and triangles $P_1,P_2$
such that $A(\ell)-A(\bar\ell)=A(\partial P_1)+A(\partial P_2) + A(a) - A(b)$
and $|P_1|+|P_2| \leq \rho(\ell,\bar\ell)^2$ and $|a|+|b|\leq \rho(\ell,\bar\ell)$
(if $\ell,\bar\ell$ are far, then $a=\ell$, $b=\bar\ell$, and $P_1,P_2$ are empty;
if $\ell,\bar\ell$ are not far, then $a$ is the chord from $\ell_f$ to $\bar\ell_f$ and $b$ is the chord from $\bar\ell_i$ to $\ell_i$).
It follows that for all $\ell,\bar\ell \in \mcX$ with $\rho(\ell,\bar\ell)<\sqrt s \wedge r$
\begin{equ}\label{eq:rho_p_bound}
\E[|A(\ell)-A(\bar\ell)|^p ] \lesssim M \rho(\ell,\bar\ell)^{p\alpha}\;,
\end{equ}
where the proportionality constant depends only on $p,\alpha$.

Let $\CX_{<r}$ denote the set of line segments of length less than $r$.
For $N \geq 1$ let $D_{N;<r}$ denote the set of line segments in $\mcX_{<r}$
whose start and end points have dyadic coordinates of scale $2^{-N}$,
and let $D_{<r}=\cup_{N\geq 1} D_{N;<r}$.
For any $\ell\in D_{<r}$ and $L \geq 1$, we can write $A(\ell) = A(\ell_0) + \sum_{i=1}^m A(\ell_i)-A(\ell_{i-1})$
for some finite $m \geq 1$ and where $\ell_i\in D_{2(L+i);<Kr}$
and $\rho(\ell_i,\ell_{i-1}) \leq K2^{-(L+i)}$
for a constant $K>0$.
Moreover, the first $\ell_i$ which has non-zero length satisfies $|\ell_i|\asymp |\ell|$.
Taking $L$ as the smallest positive integer such that $2^{-L} < \sqrt s \wedge r$,
it follows that, for any $\bar\alpha\in(0,1]$,
\begin{equ}\label{eq:A_dyadic_bound}
\sup_{|\ell|\in D_{<r}} \frac{|A(\ell)|}{|\ell|^{\bar\alpha}}
\lesssim \sup_{\ell\in D_{2L;<r}} \frac{|A(\ell)|}{|\ell|^{\bar\alpha}}
+ \sup_{N > L} \sup_{\substack{a,b\in D_{2N;<r}\\ \rho(a,b) \leq K2^{-N}}} \frac{|A(a)-A(b)|}{2^{-N\bar\alpha}}\;,
\end{equ}
where the proportionality constant depends only on $\bar\alpha$.
Observe that $|D_{2N}| \asymp 2^{12N}$.
Therefore,
raising both sides of~\eqref{eq:A_dyadic_bound} to the power $p$ and replacing the suprema on the right-hand side by sums, we obtain from~\eqref{eq:rho_p_bound}
\begin{equ}
\E
\Big[
\Big(\sup_{|\ell|\in D_{<r}} \frac{|A(\ell)|}{|\ell|^{\bar\alpha}}
\Big)^p
\Big]
\lesssim M(s^{-6} + r^{-12}) r^{p(\alpha-\bar\alpha)} + M\sum_{N \geq 1} 2^{N(24 -p(\alpha-\bar\alpha))}\;,
\end{equ}
where the proportionality constant depends only on $p,\bar\alpha$.
Since $24-p(\alpha-\bar\alpha)<0 \Leftrightarrow \bar\alpha\in(0,\alpha-\frac{24}{p})$,
the conclusion follows as in the classical Kolmogorov criterion, see e.g.\ the proof of~\cite[Thm.~4.23]{Kallenberg21}.
\end{proof}

\begin{lemma}\label{lem:SHE-heatgr-mom}
Suppose $\Psi(0)=0$.
Let $\alpha\in(0,\frac12)$, $\bar\alpha\in(0,\alpha-\frac{24}{p})$,
$\theta\in(0,1]$, and $T>0$.
Then for all $p \geq p_0(\alpha,\bar\alpha,\theta)>0$ is sufficiently large,
\begin{equ}
\sup_{s \in (0,T)} s^{-p\theta(1-2\alpha)/4}\E[\heatgr{\Psi(s)}_{\bar\alpha,\theta}^p]
\lesssim
C_{\xi}^{p/2}\;,
\end{equ}
where the proportionality constant depends only on $\alpha,\bar\alpha,\theta,p,T$.
\end{lemma}

\begin{proof}
Take $\kappa\eqdef\theta(\frac12-\alpha)\in(0,\frac12)$.
Then $|\ell|t^{-2\kappa}\leq|\ell|^{2\alpha}$
whenever $|\ell|<t^\theta< 1$
and
$|P|t^{-\frac12-\kappa} \leq |P|^{\alpha}$
whenever $|P|<t<1$,
in which case Lemma~\ref{lem:SHE_ell_P_bound}
and equivalence of Gaussian moments imply
\begin{equ}
\E[|e^{t\Delta} \Psi(s)(\ell)|^p]
\lesssim C_\xi^{p/2} s^{p\kappa/2} |\ell|^{p\alpha}\;,
\end{equ}
and
\begin{equ}
\E[|e^{t\Delta} \Psi(s)(\partial P)|^p]
\lesssim C_\xi^{p/2} s^{p\kappa/2}|P|^{p\alpha/2}\;.
\end{equ}
Applying Lemma~\ref{lem:Kolmogorov_fixed_time}
(with $r,s$ therein equal to $t^\theta,t$ respectively)
and taking $p$ sufficiently large so that $\sup_{t\in(0,1)}(t^{-6}+t^{-12\theta})t^{\theta p(\alpha-\bar\alpha)}\leq 2$,
we obtain for all $\beta\in(0,\alpha-\frac{24}{p})$
\begin{equ}\label{eq:sup_t_Delta}
\sup_{t\in(0,1)}\E[|e^{t\Delta}\Psi(s)|^p_{\gr{\beta};<t^\theta}]
\lesssim C_\xi^{p/2}s^{p\kappa/2}\;.
\end{equ}
It is easy to see that for all $A\in\Omega\CD'$
\begin{equ}\label{eq:heatgr_dyadics}
\heatgr{A}_{\bar\alpha,\theta} \lesssim \sup_{k\geq 1} |e^{2^{-k}\Delta}A|_{\gr{\bar\alpha};<2^{-\theta k}}\;,
\end{equ}
where the proportionality constant depends only on $\bar\alpha,\theta$.
Combining~\eqref{eq:heatgr_dyadics},~\eqref{eq:sup_t_Delta}, and the fact that $|A|_{\gr{\bar\alpha};<t^\theta} \leq t^{\theta(\beta-\bar\alpha)}|A|_{\gr\beta;<t^\theta}$ for $\beta\in[\bar\alpha,1]$ (Lemma~\ref{lem:local_lower_exponent})
we eventually obtain
\begin{equ}
\E[\heatgr{\Psi(s)}_{\bar\alpha,\theta}^p]
\lesssim
\E
\Big[
\sum_{k=1}^\infty |e^{2^{-k}\Delta}\Psi(s)|_{\gr{\bar\alpha};<2^{-k\theta}}^p
\Big]
\lesssim C_{\xi}^{p/2}s^{p\kappa/2}\;,
\end{equ}
as claimed.
\end{proof}

\begin{proof}[of Proposition~\ref{prop:SHE_heatgr_estimates}]
Let $0 \leq s \leq t \leq T$ and observe that
\begin{equ}
\Psi(t)-\Psi(s) = (e^{(t-s)\Delta}-1)e^{s\Delta}\Psi(0) + (e^{(t-s)\Delta}-1)\tilde\Psi(s) + \hat\Psi(t)\;,
\end{equ}
where $\tilde\Psi \colon [0,s] \to \Omega\CD'$ and $\hat\Psi \colon [s,t] \to \Omega\CD'$ solve the SHE driven by $\xi$ with $\tilde\Psi(0)=0$ and $\hat\Psi(s)=0$.
By Lemmas~\ref{lem:CP_contraction}\ref{pt:heatgr_contraction} and~\ref{lem:heatgr_heat_flow},
for any $\gamma\in(0,1)$,
\begin{equ}
\heatgr{(e^{(t-s)\Delta}-1)e^{s\Delta}\Psi(0)}_{\bar\alpha,\theta} 
\lesssim |t-s|^{\gamma\theta(\alpha-\bar\alpha)}
\heatgr{\Psi(0)}_{\alpha,\theta}\;.
\end{equ}
Likewise, by Lemmas~\ref{lem:heatgr_heat_flow} and~\ref{lem:SHE-heatgr-mom}, for any $\zeta \in (\alpha+\frac{24}{p},\frac12)$ and $\gamma\in(0,1)$,
\begin{equ}
\E[\heatgr{(e^{(t-s)\Delta}-1)\tilde\Psi(s)}_{\bar\alpha,\theta}^p] \lesssim C_\xi^{p/2} |t-s|^{p\gamma\theta(\alpha-\bar\alpha)}s^{p\theta(1-2\zeta)/4}\;.
\end{equ}
Finally, by Lemma~\ref{lem:SHE-heatgr-mom},
for any $\beta \in (\bar\alpha+\frac{24}{p},\frac12)$,
\begin{equ}
\E[\heatgr{\hat\Psi(t)}_{\bar\alpha,\theta}^p] \lesssim C_\xi^{p/2} |t-s|^{p\theta(1-2\beta)/4}\;.
\end{equ}
Taking $\gamma\in(0,1)$, $\beta\in(\bar\alpha+\frac{24}{p},\frac12)$,
and denoting $\bar\kappa\eqdef\theta\min\{\gamma(\alpha-\bar\alpha),\frac{1-2\beta}{4}\}$,
we obtain
\begin{equ}
\E[\heatgr{\Psi(t)-\Psi(s)}_{\bar\alpha,\theta}^p] \lesssim |t-s|^{p\bar\kappa}(\heatgr{\Psi(0)}_{\alpha,\theta}^p+C_\xi^{p/2})\;.
\end{equ}
We can choose $p$ sufficiently large
and $\gamma\in(0,1)$ sufficiently close to $1$
such that $\kappa<\bar\kappa$ for some $\beta\in(\bar\alpha+\frac{24}{p},\frac12)$,
and the conclusion follows by Kolmogorov's continuity theorem.
\end{proof}

\subsection{Convergence of mollifications}

We conclude this section with a corollary on the convergence of mollifications of the SHE in $\CC([0,T],\state)$ (albeit with no quantitative statement).

\begin{lemma}\label{lem:equiv_metrics}
Let $\eta<\bar\eta,\bar\delta<\delta$, $\beta\in\R$,
$0<\alpha<\bar\alpha\leq1$, and $\theta,R>0$.
Denote $\bar\Sigma\eqdef\Sigma_{\bar\eta,\beta,\bar\delta,\bar\alpha,\theta}$
and define the set
$B_R\eqdef \{Y\in\state\,:\, \bar\Sigma(Y)\leq R\}$.
Then for every $c>0$ there exists $\bar c>0$ such that for all $X,Y\in B_R$, $|X-Y|_{\CC^{\eta}}<\bar c \Rightarrow \Sigma(X,Y)<c$.
\end{lemma}

\begin{proof}
By Proposition~\ref{prop:compact}\ref{pt:Sigma_compact},
$B_R$
is compact in $\state$.
The metric $\Sigma$ is stronger than $|\cdot|_{\CC^{\eta}}$ so
the identity map $\id \colon (B_R,\Sigma )\to (B_R,|\cdot|_{\CC^\eta})$ is continuous.
In particular, $B_R$ is compact also in $\CC^{\eta}$.
Recall that if $\tau,\sigma$ are Hausdorff topologies on a set $A$ such that (i) $\tau$ is stronger than $\sigma$ and (ii) $(A,\tau)$ and $(A,\sigma)$ are both compact, then $\tau=\sigma$.
It follows that the identity map $\id \colon (B_R,|\cdot|_{\CC^\eta})\to (B_R,\Sigma )$ is also 
continuous and thus uniformly continuous by the Heine--Cantor theorem.
\end{proof}

\begin{corollary}\label{cor:SHE_conv_state}
Let $\xi$ be as in Section~\ref{subsec:1st-half}.
Let $\moll$ be a mollifier and $\xi^\eps = \moll^\eps * \xi$.
Suppose that $\Psi^\eps$ (resp.\ $\Psi$) solves SHE driven by $\xi^\eps$ (resp.\ $\xi$),
with $\Psi^\eps_0=\Psi_0 \in \state$.
Suppose~\eqref{eq:CI'}
and let $\theta\in(0,1]$, $\alpha\in(0,\frac12)$, $T>0$.
Then
$\Psi^\eps \to \Psi$
in probability in
$\CC([0,T],\state)$
as $\eps\to 0$.
\end{corollary}

\begin{proof}
By Propositions~\ref{prop:SHE_Nt_estimates}
and~\ref{prop:SHE_heatgr_estimates},
$\Psi$ and $\Psi^\eps$ a.s. take values in $\CC([0,T],\state)$.
Let $\bar\eta\in(\eta,-\frac12)$, $\bar\delta\in(0,\delta)$, and $\bar\alpha\in(\alpha,\frac12)$ such that $(\bar\eta,\beta,\bar\delta)$ satisfies~\eqref{eq:CI'}, denote $\bar\Sigma\eqdef\Sigma_{\bar\eta,\beta,\bar\delta,\bar\alpha,\theta}$.
Then, for every $c,R>0$, by Lemma~\ref{lem:equiv_metrics},
there exists $\bar c>0$ such that
\begin{equs}{}
&\P
\Big[\sup_{t\in[0,T]}
\Sigma(\Psi^\eps_t,\Psi_t)> c
\Big]
\leq
\P\Big[
\sup_{t\in[0,r]} \{\Sigma(\Psi^\eps_t,\CP_t\Psi_0) + \Sigma(\Psi_t,\CP_t\Psi_0)\} > c
\Big]
\\
&\qquad+
\P
\Big[\sup_{t\in[r,T]}|\Psi^\eps_t-\Psi_t|_{\CC^{\eta}}>\bar c
\Big]
+ \P
\Big[\sup_{t\in[r,T]}\{\bar\Sigma(\Psi_t)+\bar\Sigma(\Psi^\eps_t)\} > R
\Big]\;.
\end{equs}
By~\eqref{eq:SHE_Theta_zero}, the first probability on the right-hand side can be made arbitrarily small, uniformly in $\eps\in(0,1)$,
by taking $r$ small.
On the other hand, for any fixed $r>0$, the final probability,
by~\eqref{eq:SHE_Theta_Holder} and Proposition~\ref{prop:SHE_heatgr_estimates}, converges to $0$ uniformly in $\eps\in(0,1)$ as $R\to\infty$,
while the second probability converges to $0$ as $\eps\to 0$, see e.g.~\cite[Prop.~9.5]{Hairer14}.
\end{proof}

\begin{remark}\label{rem:Sourav_flow}
The main result of~\cite{Sourav_flow} is that the YM flow $\ymhflow$ started from smooth approximations of the Gaussian free field (GFF)
converges locally in time to a process which one can interpret as the flow started from the GFF.
We recover essentially the same result from Proposition~\ref{prop:YM_flow_minus_heat}
and Corollary~\ref{cor:SHE_conv_state}
once we start the SHE from the GFF, which is its invariant measure.

In fact, after showing that $\ymhflow$ is a locally Lipschitz function on $\init\supset\state$ in Proposition~\ref{prop:YM_flow_minus_heat},
the claim that $\ymhflow$ is well-defined on the GFF reduces to showing that the latter takes values in
$\init$, which is a Gaussian moment computation similar and simpler to that of Section~\ref{subsec:1st-half}.

A minor difference with the results of~\cite{Sourav_flow} is that therein convergence is shown in $L^p$ while we only show convergence in probability; it would not be difficult to modify our arguments to show convergence in $L^p$  but we refrain from doing so since it is not needed in the sequel.
\end{remark}

\section{Symmetry and renormalisation in regularity structures}\label{sec:sym_renorm_reg_struct}

In this section we formulate algebraic arguments which verify that the symmetries of a system of equations are preserved by the BPHZ renormalisation procedure.

For what follows we are in the setting of \cite[Sec.~5]{CCHS2d};
we fix a collection of solution types $\mfL_{+}$, noise types $\mfL_{-}$, a target space assignment $W = (W_{\mft})_{\mft \in \Lab}$ and 
a kernel space assignment $\mathcal{K} = (\mathcal{K}_{\mft})_{\mft \in \Lab_{+}}$, with  $\mfL=\mfL_+\sqcup \mfL_-$.
We assume that all of these space assignments are finite-dimensional. 

We then fix a corresponding space assignment $V = (V_{\mft})_{\mft \in \Lab}$ as in \cite[Eq.~5.25]{CCHS2d}, that is
\begin{equ}\label{eq:space_assignment}
V_{\mft} = 
\begin{cases}
W_{\mft}^{\ast} & \textnormal{ if }\mft \in \Lab_{-}\;,\\
\mathcal{K}_{\mft}^{\ast} & \textnormal{ if }\mft \in \Lab_{+}\;.
\end{cases}
\end{equ}

The symmetries that we study are simultaneous transformations of the target spaces where the solution and the noise take values in and transformations on the underlying (spatial) base space $\Lambda \subset \R^{d}$. 
We assume that $\Lambda$ is invariant under reflections across coordinate axes and permutations of canonical basis vectors \dash to keep our presentation simpler we will only consider transformations of the base space that are compositions of such reflections and permutations.

\begin{definition}\label{def:Tran}
Let $\Tran$ be the collection of quartets $\mathbf{T} = (T,O,\sigma,r)$ where
\begin{itemize}
\item  $T = (T_\mfb)_{\mfb \in \Lab}$ with $T_{\mfb}\in   L(W_{\mfb},W_{\mfb})$ invertible. $T$ determines our transformation of the target spaces for the noise and solution.
\item  $O = (O_{\mft})_{\mft \in \Lab_{+}}$ with $O_\mft \in L(\mathcal{K}_{\mft})$
invertible. The role of $O$ is to specify transformations of our kernels. 
\item $\sigma \in S_{d}$ where $S_{d}$ is the set of permutations of the set $[d]$, and $r \in \{-1,1\}^{d}$. The role of $\sigma$ and $r$ is to determine our transformation of the base space \dash we also overload notation and, for $\sigma$ and $r$ as above, define $\sigma,r \in L(\R^{d},\R^{d})$ by setting $(r x)_{i} = r_{i}x_{i}$ and $(\sigma x)_i = x_{\sigma^{-1}(i)}$ for $i=1\ldots d$.
We have $\sigma,r\colon\T^{d} \rightarrow \T^{d}$ and also view them as maps on our space-time domain $\R_{+} \times \T^{d}$ by acting on the spatial coordinates only. 
\end{itemize}
\end{definition}

One should view an element $\mathbf{T} = (T,O,\sigma,r) \in \Tran$ as acting on solutions\slash noises $A_{\mft}$ and kernels $K_{\mft}$ by mapping
\[
A_{\mft}(\cdot) \mapsto T_{\mft} A_{\mft} (\sigma^{-1} r \cdot) \enskip \text{for} \enskip \mft \in \Lab \enskip \text{and} \enskip 
K_{\mft}(\cdot) \mapsto O_{\mft} K_{\mft}(\sigma^{-1} r \cdot) \enskip \text{for} \enskip  \mft \in \Lab_{+}\;. 
\] 
We endow $\Tran$ with a group structure compatible with the above: given $\mathbf{T} = (T,O,\sigma,r)$
and $\mathbf{T}' = (T',O',\sigma',r')$ we set 
$\mathbf{T} \mathbf{T}' = (T T',OO',\sigma \sigma',\sigma r' \sigma^{-1} r)$.


As above, we have a left group action of $S_{d}$ on multi-indices $\N^{d+1}$ by 
setting  
$\sigma (p_{0},\cdots,p_{d}) = (p_{0},p_{\sigma^{-1}(1)},\dots,p_{\sigma^{-1}(d)})$,
yielding an action on the set of edge types $\CE = \Lab \times \N^{d+1}$. \label{def CE}
It also yields an action of $S_{d}$ on $\N^{\CE}$, viewed as multi-sets of elements of $\CE$, 
given by applying $\sigma$ to each element of any given multiset in  $\N^{\CE}$. 
We then fix a rule $R$ that is $S_{d}$-invariant in the sense that, for any $\sigma \in S_{d}$, $ \mft \in \Lab$, and $\CN \in R(\mft) \subset \N^{\CE}$, we have $\sigma \CN \in R(\mft)$. 

Recall that a rule $R$ assigns a subset of $\N^{\CE}$ to each $\mft \in \Lab$
and determines a set of conforming trees $\mathfrak{T}$ and  forests $\mathfrak{F}$.
Our trees have edge decorations $e \mapsto (\mft(e),\mfn(e)) \in \CE$ and node decorations $v \mapsto \mfn(v) \in \N^{d+1}$, the
latter also being referred to as ``polynomial decorations''.
Loosely speaking, such a decorated tree conforms to the rule $R$ if, given any inner node, if $\mft \in \Lab$ is the 
type of the unique edge leaving $v$, the multiset given by the decorations of the edges entering $v$ belongs to $R(\mft)$.
Combining this with a space-assignment $V$ yields a regularity structure $\mcb{T}$ which admits a decomposition into subspaces $\mcb{T}[\tau]$ indexed by trees $\tau \in \mathfrak{T}$ (and an algebra $\mcb{F}$ decomposing into linear subspaces $\mcb{F}[f]$ indexed by  $f \in \mathfrak{F}$). We refer to \cite[Definition~5.8]{BHZ19} and \cite[Section~5.5]{CCHS2d}
for more details.

In order to formulate our arguments we now define a right action $\Tran \ni \mathbf{T} \mapsto \mathbf{T}^{\ast} \in L(\mcb{F},\mcb{F})$ by specifying three transformations on $\mcb{F}$: one that encodes $T$ and $O$, one that encodes $\sigma$, and one that encodes $r$. 

For encoding the transformation of the target space we use \cite[Remark~5.19]{CCHS2d}.
Given an  ``operator assignment''
\begin{equ}\label{eq:operator_assignment}
L = \bigoplus_{\mfb \in \Lab} L_{\mfb} \in L(V) \eqdef \bigoplus_{\mfb \in \Lab}  L(V_{\mfb},V_{\mfb})\;,
\end{equ}
that remark says that, for any $\Lab$-typed symmetric set $\symset$, we can apply $L$ ``component-wise'' to obtain a linear operator $L[\symset] \in L(V^{\otimes \symset},V^{\otimes \symset})$.
Since $\mcb{F}$ decomposes as a direct sum of such spaces, this defines a linear operator $L \in L(\mcb{F},\mcb{F})$. 
\begin{example}
We present a simple example of the above construction. 
We work with $\Lab_{-} = \{\mfl\}$, $\Lab_{+} = \{\mft\}$.
Suppose we fix some linear transformation acting on the target space of our noise, that is some $U \in L(W_{\mfl},W_{\mfl})$,
and define an operator assignment $L$ by $L_{\mft} = \id \in  L(V_{\mft},V_{\mft})$ and $L_{\mfl} = U^{\ast} \in L(V_{\mfl},V_{\mfl})$. 
Then, to demonstrate the corresponding linear transformation on our regularity structure, 
we look at the $\Lab$-typed symmetric set $\scal{ \<IXi^2> }$ associated to the tree $\<IXi^2>$.
We have $V^{\otimes  \scal{ \<IXi^2> }} = \mcb{T}[ \<IXi^2> ] \simeq V_{\mfl} \otimes_{s} V_{\mfl} \subset V_{\mfl} \otimes V_{\mfl}$
and, given $u,v \in V_{\mfl}$, we have $ u \otimes v + v \otimes u \in \mcb{T}[\<IXi^2>] $ and 
\[
L( u \otimes v + v \otimes u) = (U^*u) \otimes (U^*v) + (U^*v) \otimes (U^*u) \in \mcb{T}[\<IXi^2>] \;.
\] 
\end{example}

Given $\mathbf{T} \in \Tran$, we define the operator assignment
\[
I(\mathbf{T})^{\ast}
=
\Big(I(\mathbf{T})^{\ast}_{\mfb \in \Lab}: \mfb \in \Lab\Big)
\eqdef
\Big( \bigoplus_{\mfb \in \Lab_{-}} T^{\ast}_{\mfb} \Big)
\oplus
\Big( \bigoplus_{\mfb \in \Lab_{+}} O_{\mfb}^{\ast} \Big)
 \in L(V)\;.
 \]
We write $I(\mathbf{T})^{\ast}_{\mfb \in \Lab} \in L(\mcb{F},\mcb{F})$ for the corresponding linear operator given by \cite[Remark~5.19]{CCHS2d} which encodes the transformation of our target space on $\mcb{F}$.  
\begin{remark}
It is natural that the operators $(T_{\mft}: \mft \in \Lab_{+})$ do not play a role in how we transform the regularity structure because the regularity structure is constructed to study the structure of the noise and does not depend on our choice of the spaces $(W_{\mft}: \mft \in \Lab_{+})$.  
The role of the operators $(T_{\mft}: \mft \in \Lab_{+})$ will be to act on the $W_{\mft}$-valued modelled distributions with $\mft \in \Lab_{+}$ that describe the solution.
\end{remark}
Given $r \in \{-1,1\}^{d}$ we define $r^{\ast} \in L(\mcb{F},\mcb{F})$ by setting, for each $f \in \mfF$, 
\[
r^{\ast} [f] \eqdef   r^{\ast} \restriction_{\mcb{F}[f]} =   r^{n(f)} \id_{\mcb{F}[f]}\;,
\]
where for any $q = (q_{i})_{i=0}^{d} \in \N^{d+1}$ we write $r^q =\prod_{i=1}^{d} r_{i}^{q_{i}}$ and  $n(f) \in \N^{d+1}$ is given by 
\begin{equ}[e:def-nf]
n(f)= \sum_{e \in E_{f}} \mfn(e) + \sum_{u \in N_{f}} \mfn(u)\;.
\end{equ}
Here $E_{f}$ is the set of edges of $f$, $N_{f}$ is the set of nodes of $f$, and $\mfn$ denotes the edge and polynomial decorations.

We define a left group action $(\sigma,f) \mapsto \sigma f$ of $S_{d}$ on $\mfF$
 by setting $\sigma f$ to be the forest obtained by performing the replacement $\mfn(a) \mapsto \sigma \mfn(a)$ for any edge or node $a \in E_{f} \sqcup N_{f}$.  
Note that, for any $f \in \mfF$, there is a canonical isomorphism $\sigma^{\ast}[f]\colon \mcb{F}[f] \rightarrow \mcb{F}[\sigma^{-1} f]$ since the edges of $\sigma^{-1} f$ are in natural correspondence with those of $f$.
We write $\sigma^{\ast} = \bigoplus_{f \in \mcb{F}} \sigma^{\ast}[f] \in L(\mcb{F},\mcb{F})$.

We then set $\mathbf{T}^{\ast} = I(\mathbf{T})^{\ast} \sigma^{\ast} r^{\ast} \in L(\mcb{F},\mcb{F})$. 
Note that, as elements of $L(\mcb{F},\mcb{F})$,  $I(\mathbf{T})^{\ast}$ commutes with $r^{\ast}$ and with $\sigma^{\ast}$. 
We abuse notation and write $I(\mathbf{T}), \sigma,r , \mathbf{T}\in L(\mcb{F}^{\ast},\mcb{F}^{\ast})$ for the adjoints of the operators we just introduced. 

We also have a natural left action of $\Tran$ on $\mcb{A}\eqdef \prod_{o\in\CE}W_o$ given by setting, \label{mcb A}
for $\mathbf{T} = (T,O,\sigma,r) \in \Tran$ and $\mathbf{A} = (A_{(\mfb,p)})_{(\mfb,p) \in \CE}$, 
\begin{equ}\label{eq:transform_on_jets}
\mathbf{T}\mathbf{A} = ( \bar{A}_{o} )_{o \in \CE} 
\enskip
\textnormal{with}
\enskip 
\bar{A}_{(\mfb,p)} = r^{p} T_{\mfb}  A_{(\mfb,\sigma^{-1} p)}\;.
\end{equ}%

We now define the various ways in which we impose that a given system 
should respect a given transformation.
\begin{definition}
Fix $\mathbf{T} \in \Tran$, with $\Tran$ defined as in Definition~\ref{def:Tran}.
\begin{itemize}
\item Given a nonlinearity $F(\cdot) = \bigoplus_{\mft \in \Lab} F_{\mft}(\cdot)$ we say \emph{$F$ is $\mathbf{T}$-covariant} if, for every $\mfl \in \Lab_{-}$, $F_{\mfl}(\cdot) = \id_{W_{\mfl}}$ and, for every $\mft \in \Lab_{+}$, 
\[
(O_{\mft}^{\ast} \otimes T_{\mft}^{-1}) F_{\mft}(\mathbf{T} \mathbf{A})
=
F_{\mft}(\mathbf{A}) \;. 
\]
where on the left-hand side we are using the action \eqref{eq:transform_on_jets}.

\item Given a kernel assignment $K = (K_{\mft})_{\mft \in \Lab_{+}}$ we say \emph{$K$ is $\mathbf{T}$-invariant} if for every $\mft \in \Lab_{+}$ and $z \in \R \times \T^{d}$, $O_{\mft} K_{\mft}(\sigma^{-1} r z) = K_{\mft}(z)$.

\item Given a random noise assignment $\zeta = (\zeta_{\mfl})_{\mfl \in \Lab_{-}}$ we say \emph{$\zeta$ is $\mathbf{T}$-invariant} if the tuple of random fields $\big( \zeta_{\mfl}(\cdot) \big)_{\mfl \in \Lab_{-}}$ and $\big( T_{\mfl} \zeta_{\mfl}( \sigma^{-1} r \cdot) \big)_{\mfl \in \Lab_{-}}$ have the same probability law.

\item Given $\ell \in \mcb{F}^{\ast}$ or\footnote{$\mcb{F}_{-} \subset \mcb{F}$ is defined in \cite[Section~5.5]{CCHS2d}, and it is immediate that $\mathbf{T}^{\star}$ leaves $\mcb{F}_{-}$ invariant.} $\ell \in \mcb{F}^{\ast}_{-}$, we say that \emph{$\ell$ is $\mathbf{T}$-invariant} if $\mathbf{T}\ell \eqdef \ell \circ \mathbf{T}^{\ast} = \ell$. 
\end{itemize}
\end{definition}
We then have the following lemma, where the canonical lift is defined as in  \cite[Sec.~6.3]{BHZ19}, and the corresponding
BPHZ character is defined as in \cite[Thm~6.18]{BHZ19} (see also \cite[Sec.~5.7.2]{CCHS2d}).
\begin{lemma}\label{lem:ell-T-inv}
Fix $\mathbf{T} \in \Tran$. Suppose we are given a kernel assignment $K$ and a smooth, random noise assignment $\zeta$ which are both $\mathbf{T}$-invariant.  
Let $\bar{\PPi}_{\can} \in \mcb{F}^{\ast}$ be the corresponding canonical lift and $\ell_{\BPHZ} \in \mcb{F}^*_-$ its BPHZ character.
Then both $\bar{\PPi}_{\can}$ and $\ell_{\BPHZ}$ are $\mathbf{T}$-invariant.
\end{lemma}
\begin{proof}
To show that $\bar{\PPi}_{\can}$ is invariant it suffices to show that, for any $\tau \in \mfT$, one has   
$(\mathbf{T} \bar{\PPi}_{\can})[\tau] = \bar{\PPi}_{\can}[\sigma^{-1} \tau] \circ \mathbf{T}^{\ast}[\tau] = \bar{\PPi}_{\can}[\tau]$ \dash this is because both $\mathbf{T}$ and $\bar{\PPi}_{\can}$ ``factorise'' over forests appropriately.

Using the same notation convention as used in \cite[Eq.~5.31]{CCHS2d}, one has that $\bar{\PPi}_{\can}[\sigma^{-1} \tau] \circ \mathbf{T}[\tau]$ is given by
\begin{equs}\label{eq:bold_faced_pi_transformation}
{}&
r^{n(\tau)}\E
\int_{(\R^{d+1})^{N(\tau)}} 
\mrd x_{N(\tau)}
\delta(x_{\rho})
\Big(
\prod_{v \in N(\tau)} x_{v}^{\sigma^{-1}\mfn(v)}
\Big)\\
{}&
\qquad\qquad
\Big(
\bigotimes_{
e \in K(\tau)}
\big(D^{\sigma^{-1} \mfn(e)} O_{\mft(e)}K_{\mft(e)} \big)(x_{e_{+}} - x_{e_{-}})
\Big)
\\
{}&
\qquad\qquad
\Big(
\bigotimes_{e \in L(\tau)}
\big(D^{\sigma^{-1} \mfn(e)} T_{\mft(e)}\zeta_{\mft(e)}\big)(x_{e_{+}})
\Big)\\
{}&=
r^{n(\tau)}\E
\int_{(\R^{d+1})^{N(\tau)}} 
\mrd x_{N(\tau)}
\delta(x_{\rho})
\Big(
\prod_{v \in N(\tau)} (\sigma^{-1} r x_{v})^{\sigma^{-1}\mfn(v)}
\Big)\\
{}&
\qquad\qquad
\Big(
\bigotimes_{
e \in K(\tau)}
\big(D^{\sigma^{-1} \mfn(e)} O_{\mft(e)} K_{\mft(e)} \big) (\sigma^{-1} r x_{e_{+}} - \sigma^{-1} r x_{e_{-}})
\Big)
\\
{}&
\qquad\qquad
\Big(
\bigotimes_{e \in L(\tau)}
\big(D^{\sigma^{-1} \mfn(e)} T_{\mft(e)}\zeta_{\mft(e)}\big)(\sigma^{-1} r x_{e_{+}})
\Big)\\
{}&=
r^{n(\tau)}\E
\int_{(\R^{d+1})^{N(\tau)}} 
\mrd x_{N(\tau)}
\delta(x_{\rho})
\Big(
\prod_{v \in N(\tau)} r^{n(v)} x_{v}^{\mfn(v)}
\Big)\\
{}&
\qquad\qquad
\Big(
\bigotimes_{
e \in K(\tau)}r^{\mfn(e)}
D^{\mfn(e)} O_{\mft(e)} \big[ K_{\mft(e)} (\sigma^{-1} r \cdot ) \big] \big|_{x_{e_{+}} - x_{e_{-}}}
\Big)
\\
{}&
\qquad\qquad
\Big(
\bigotimes_{e \in L(\tau)}
r^{\mfn(e)}
D^{ \mfn(e)} \big[ T_{\mft(e)}\zeta_{\mft(e)}(\sigma^{-1} r \cdot )\big] \big|_{x_{e_{+}}}
\Big).
\end{equs}
By cancelling powers of $r$, using the $\mathbf{T}$-invariance of $K$ and  $\zeta$ we see that the last line is precisely $\bar{\PPi}_{\can}[\tau]$.

To prove the statement regarding $\ell_{\BPHZ}$ we observe that $\ell_{\BPHZ} = \bar{\PPi}_{\can} \circ \tilde{\mcb{A}}_{-}$ where $ \tilde{\mcb{A}}_{-} $ is the negative twisted antipode defined in \cite[Eq.~6.8]{BHZ19}. 
From the definition of $\Delta^{-}$ as in \cite[Sec.~5.5]{CCHS2d}, it is straightforward 
to verify recursively that $(\mathbf{T}^{\ast} \otimes \mathbf{T}^{\ast}) \Delta^{-} = \Delta^{-} \mathbf{T}^{\ast}$ and, combining this with the inductive definition of $\tilde{\mcb{A}}_{-}$, it follows that $\tilde{\mcb{A}}_{-} \circ  \mathbf{T}^{\ast} = \mathbf{T}^{\ast} \circ \tilde{\mcb{A}}_{-} $, so that 
\[
\mathbf{T} \ell_{\BPHZ} = \bar{\PPi}_{\can} \circ \tilde{\mcb{A}}_{-} \circ  \mathbf{T}^{\ast} =  \bar{\PPi}_{\can} \circ  \mathbf{T}^{\ast} \circ \tilde{\mcb{A}}_{-} =  \bar{\PPi}_{\can} \circ \tilde{\mcb{A}}_{-} = \ell_{\BPHZ}\;,
\]
concluding the proof.
\end{proof}

Note that, for the linear operator $\mcb{I}_{\mft}\colon  \mcb{T}  \otimes V_{\mft} \rightarrow \mcb{T}$ (see \cite[Sec.~5.8.2]{CCHS2d} for its definition) we have, for any $\mathbf{T} \in \Tran$, the identity 
\begin{equ}\label{eq:integration_intertwining}
\mcb{I}_{\mft} \circ ( \mathbf{T}^{\ast}  \otimes I(\mathbf{T})^{\ast}_{\mft}) =  \mathbf{T}^{\ast} \mcb{I}_{\mft}\;.
\end{equ}

Lemma~\ref{lem:upsilon_behaves} below shows how covariance of our nonlinearity propagates through coherence. 
We refer to \cite[Sec.~5.8.2]{CCHS2d} for the notion of coherence,
as well as the definitions of the maps $\bUpsilon$ and $\bar\bUpsilon$ which basically describe the coefficient in front of $\tau$ in the expansion of the solution as a modelled distribution.
(See also \cite[Section~2]{BCCH21} where the notion of coherence and the maps  $\bUpsilon$ and $\bar\bUpsilon$ are discussed in more details.)
For our purposes it suffices to recall that these maps $\bUpsilon$ and $\bar{\bUpsilon}$
satisfy the following inductive property: if
$\tau$ is of the form 
\begin{equ}\label{eq:generic_tau}
\mbX^k \prod_{i=1}^m \mcb{I}_{o_i}(\tau_{i})\;,
\end{equ}
where
$m \geq 0$, $\tau_{i} \in \mfT$, and $o_i \in \CE$, then\footnote{Recall from \cite[Sec.~5.8.2]{CCHS2d} 
that $\bUpsilon_{o}[\tau] =\Upsilon_{o}[\tau]/S(\tau)$ and recursive formulae
for $\Upsilon$ and $S$ are given therein. Here, we instead give a recursive formula for $\bUpsilon$. It is easy to check that 
these recursive formulae are consistent
and one may be more convenient than the other for different purposes.
The formula here
  involves a factor 
$\mathring{S}(\tau)$
and we give a formula for it in \eqref{e:sym_at_root}.}
\begin{equs}[eq:upsilon_induction]
\bar{\bUpsilon}_{\mft} [\tau] 
&\eqdef 
\frac{1}{\mathring{S}(\tau)} 
\mbX^k 
\Big[ 
\partial^k D_{o_1}\cdots D_{o_m}
\bar{\Upsilon}_{\mft}[\bone]
\Big]
\big(\bUpsilon_{o_1}[\tau_1],\ldots,\bUpsilon_{o_m}[\tau_m]\big)
 ,\\
\bUpsilon_{(\mft,p)}[\tau]
&\eqdef
(\mcb{I}_{(\mft,p)} \otimes \id_{W_{\mft}}) (\bar{\bUpsilon}_{\mft} [\tau] )\;,
\end{equs}
where 
\[
\mathring{S}(\tau) = \frac{S(\tau)}{\prod_{i=1}^{m} S(\tau_i)}
\]
and $S(\tau)$ is defined as in \cite[Eq.~5.56]{CCHS2d}. More explicitly, if, for each $i \in [m]$, we define $\beta_{i} \in \N$ to be the multiplicity of $\mcb{I}_{o_{i}}(\tau_{i})$ in $\beta_{i} \in \N$ in \eqref{eq:generic_tau}, then 
\begin{equ}[e:sym_at_root]
\mathring{S}(\tau) = 
k! \prod_{i=1}^{m} (\beta_{i}!)^{\frac{1}{\beta_{i}}}\;.	
\end{equ}

\begin{remark}\label{rem:suppress_ident}
In what follows we will often suppress tensor products with an identity operator from our notation. 
For instance, given  a tensor product of vector spaces $H \otimes H'$  a third vector space $\tilde{H}$, and operators on $H \otimes H'$ of the form  $(L \otimes \id_{H'})$ for $L \in L(H, \tilde{H})$ or $(\id_H \otimes Q)$ for $Q \in L(H',\tilde{H})$,  we will often write, for $v \in H \otimes H'$, $Lv  \in \tilde{H} \otimes H'$ or $Q v \in H \otimes \tilde{H}$ where $L$ and $Q$ are only acting on the appropriate factor of the tensor product. 
\end{remark}

\begin{lemma}\label{lem:upsilon_behaves}
Let $\mathbf{T} \in \Tran$ and suppose that $F$ is a $\mathbf{T}$-covariant nonlinearity.  
Then we have, for any $\mft \in \Lab$, 
\[
(\mathbf{T}^{\ast} \otimes I(\mathbf{T})^{\ast}_{\mft} \otimes \id_{W_{\mft}})  \bar\bUpsilon_{\mft}(\mathbf{A})
=  (\id_{\mcb{T}} \otimes \id_{V_{\mft}} \otimes  T_{\mft}) \bar\bUpsilon_{\mft}(\mathbf{T}^{-1} \mathbf{A})\;.
\]
\end{lemma}
\begin{proof}
We will prove that, for every $\tau \in \mfT$, $\mft \in \Lab$, and $q \in \N^{d+1}$, we have the identities 
\begin{equs}[eq:upsilon_is_brackets]
(\mathbf{T}^{\ast} \otimes I(\mathbf{T})^{\ast}_{\mft} \otimes \id_{W_{\mft}})  \bar\bUpsilon_{\mft}[\tau](\mathbf{A})
&= (\id_{\mcb{T}} \otimes \id_{V_{\mft}} \otimes T_{\mft})  \bar\bUpsilon_{\mft}[\sigma^{-1} \tau](\mathbf{T}^{-1} \mathbf{A})\;,\\
(\mathbf{T}^{\ast} \otimes  \id_{W_{\mft}}) \bUpsilon_{(\mft,q)}[\tau](\mathbf{A})
&= (\id_{\mcb{T}}  \otimes T_{(\mft,q)})  \bUpsilon_{(\mft,\sigma^{-1} q)}[\sigma^{-1} \tau](\mathbf{T}^{-1} \mathbf{A})
\end{equs}
where we set $T_{(\mft,q)} = r^{q}T_{\mft}$. 
The desired claim then follows after summing over $\tau \in \mfT$. 

We will prove the identities \eqref{eq:upsilon_is_brackets} by induction in 
$
 |E_{\tau}| + \sum_{u \in N_{\tau}} |\mfn(\tau)|
$.
We start by proving the first identity of \eqref{eq:upsilon_is_brackets} when $\tau = \bone$ and $\mft = \mfl \in \Lab_{-}$.
Fix some basis $(e_{i})_{i \in I}$ of $W_{\mfl}$, we then have 
\[
\bbUpsilon_{\mfl}[\one]=
\id_{W_{\mfl}}
=
\sum_{i \in I} e_{i}^{\ast} \otimes e_{i}
\in 
 V_{\mfl} \otimes W_{\mfl} \simeq \mcb{T}[\bone] \otimes V_{\mfl} \otimes W_{\mfl} \;.
\] 
On the other hand, we have 
\begin{equs}
T_{\mfl}  \bar\bUpsilon_{\mfl}[\bone](\mathbf{T}\mathbf{A}) 
&=
T_{\mfl} \bar\bUpsilon_{\mfl}[\bone]
=
\sum_{i \in I} e_{i}^{\ast} \otimes T_{\mfl} e_{i}
=
\sum_{i \in I} T_{\mfl}^{\ast} e_{i}^{\ast} \otimes e_{i}\\
{}&=
(\mathbf{T}^{\ast} \otimes I(\mathbf{T})^{\ast}_{\mft} \otimes \id_{W_{\mft}})  \bar\bUpsilon_{\mft}[\bone](\mathbf{A})\;.
\end{equs}

When $\tau = \bone$ and $\mft \in \Lab_{+}$ we have 
\begin{equs}
T_{\mft}  \bbUpsilon_{\mft}[\bone](\mathbf{T}^{-1} \mathbf{A})
&= 
T_{\mft} F_{\mft}(\mathbf{T}^{-1} \mathbf{A})
=   I(\mathbf{T})^{\ast}_{\mft} F_{\mft}( \mathbf{T} \mathbf{T}^{-1} \mathbf{A})
=   I(\mathbf{T})^{\ast}_{\mft} F_{\mft}(\mathbf{A})\\
&=  I(\mathbf{T})^{\ast}_{\mft} \bbUpsilon_{\mft}[\bone](\mathbf{A})
= ( \mathbf{T}^{\ast} \otimes I(\mathbf{T})^{\ast}_{\mft} \otimes \id_{W_{\mft}}) \bbUpsilon_{\mft}[\bone](\mathbf{A})\;,
\end{equs}
where, in the second equality above, we used the $\mathbf{T}$-covariance of $F$. 
Note that in both computations above we are using the fact that $\mathbf{T}^{\ast}[\bone] = \id_{\R}$.

From \eqref{eq:upsilon_induction},
we immediately see that, for fixed $\mft \in \Lab$ and $\tau \in \mfT$, the second identity of \eqref{eq:upsilon_is_brackets} follows from the first and \eqref{eq:integration_intertwining}.

To finish our proof we just need to prove the inductive step for the first identity of \eqref{eq:upsilon_is_brackets} \dash for this we may assume that $\mft \in \Lab_{+}$. 
We then write, using  \eqref{eq:upsilon_induction}, 
\begin{equs}\label{eq:symmetry_inductive_step}
{}&
\mathring{S}(\tau) (\mathbf{T}^{\ast} \otimes I(\mathbf{T})^{\ast}_{\mft})\bar{\bUpsilon}_{\mft} [\tau] (\mathbf{A})\\
{}&=
(\mathbf{T}^{\ast} \otimes I(\mathbf{T})^{\ast}_{\mft})
\mbX^k 
\Big[ 
\partial^k D_{o_1}\cdots D_{o_m}
I(\mathbf{T})^{\ast}_{\mft} \bar{\bUpsilon}_{\mft}[\bone] 
\Big]_{\mathbf{A}}
\big(\bUpsilon_{o_1}[\tau_1] (\mathbf{A}),\ldots,\bUpsilon_{o_m}[\tau_m] (\mathbf{A})\big)\\
{}&=
r^{k}
\mbX^{\sigma^{-1} k}
\Big[ 
\partial^k D_{o_1}\cdots D_{o_m}
I(\mathbf{T})^{\ast}_{\mft}
\bar{\bUpsilon}_{\mft}[\bone]
\Big]_{\mathbf{A}}
\big(\mathbf{T}^{\ast} \bUpsilon_{o_1}[\tau_1] (\mathbf{A}),\ldots,\mathbf{T}^{\ast} \bUpsilon_{o_m}[\tau_m] (\mathbf{A})\big)\\
{}&=
r^{k}
\mbX^{\sigma^{-1} k}
\Big[ 
\partial^k D_{o_1}\cdots D_{o_m}
I(\mathbf{T})^{\ast}_{\mft}
\bar{\bUpsilon}_{\mft}[\bone]
\Big]_{\mathbf{A}}
\\
& \qquad\qquad\qquad  \big(T_{o_1} \bUpsilon_{\sigma^{-1} o_1}[\sigma^{-1} \tau_1] (\mathbf{T}^{-1} \mathbf{A}),\ldots, T_{o_{m}} \bUpsilon_{ \sigma^{-1} o_m}[\sigma^{-1} \tau_m] (\mathbf{T}^{-1}\mathbf{A})\big)\;.
\end{equs}
Here the subscript $[ \cdot ]_{\mathbf{A}}$ indicates where the derivative is being evaluated.
In the second equality of \eqref{eq:symmetry_inductive_step} we used that the operator $\mathbf{T}^{\ast}$ is appropriately multiplicative and in the third equality we used our inductive hypothesis. 

Using the the $\mathbf{T}$-covariance of $F$ we have $I(\mathbf{T})^{\ast}_{\mft} \bar{\bUpsilon}_{\mft}[\bone](\mathbf{A}) = T_{\mft}\bar{\bUpsilon}_{\mft}[\bone](\mathbf{T}^{-1} \mathbf{A})$. 
From this it follows that
\begin{equs}\label{eq:derivative_change_of_var}
{}&\Big[ 
\partial^k D_{o_1}\cdots D_{o_m}
I(\mathbf{T})^{\ast}_{\mft} \bar{\bUpsilon}_{\mft}[\bone]
\Big]_{\mathbf{A}}
\big( \bullet, \dots, \bullet \big)\\
&=
r^{k}
\Big[ 
\partial^{\sigma^{-1} k} D_{ \sigma^{-1} o_1}\cdots D_{\sigma^{-1} o_m}
T_{\mft} \bar{\bUpsilon}_{\mft}[\bone]
\Big]_{\mathbf{T}^{-1} \mathbf{A}}
\big( T_{o_{1}}^{-1} \bullet, \dots, T_{o_{m}}^{-1} \bullet \big)\;.
\end{equs}
Inserting \eqref{eq:derivative_change_of_var} into the last line of \eqref{eq:symmetry_inductive_step} gives
\begin{equs}
{}&
\mathring{S}(\tau) (\mathbf{T}^{\ast} \otimes I(\mathbf{T})^{\ast}_{\mft}) \bar{\bUpsilon}_{\mft} [\tau] (\mathbf{A})\\
&=
\mbX^{\sigma^{-1} k}
\Big[ 
\partial^{\sigma^{-1} k} D_{\sigma^{-1} o_1}\cdots D_{\sigma^{-1} o_m}
T_{\mft} \bar{\bUpsilon}_{\mft}[\bone]
\Big]_{\mathbf{T}^{-1} \mathbf{A}}\\
{}&\qquad \qquad \qquad \qquad 
\big( \bUpsilon_{ \sigma^{-1} o_1}[\sigma^{-1} \tau_1] (\mathbf{T}^{-1} \mathbf{A}),\ldots, \bUpsilon_{\sigma^{-1} o_m}[\sigma^{-1} \tau_m] (\mathbf{T}^{-1}\mathbf{A})\big)\\
&=
\mathring{S}(\sigma^{-1} \tau) T_{\mft} \bar{\bUpsilon}_{\mft} [\sigma^{-1} \tau] (\mathbf{T}^{-1} \mathbf{A})\;,
\end{equs}
and we are done since $S(\tau) = S(\sigma^{-1} \tau)$. 
\end{proof}

Let $\mathfrak{R} \subset \mcb{F}^{\ast}$ be the collection of algebra homomorphisms $\ell$ from $\mcb{F}$ to $\R$ that satisfy $\ell[\tau] = 0$ for any $\tau \in \mfT \setminus \mfT_{-}$.\footnote{As in \cite{CCHS2d}
$\mfT_- \eqdef  \{\tau \in \mfT\,:\, \deg \tau < 0,\ \mfn(\rho) = 0,\ \tau \textnormal{ unplanted}\}$ where $\rho$ is the root of $\tau$.}
Recall that $\mathfrak{R}$ can be identified with the renormalisation group via an action $\mathfrak{R} \ni \ell \mapsto M_{\ell} \in L(\mcb{T},\mcb{T})$ and that \cite[Prop.~5.68]{CCHS2d} describes an action of the renormalisation group on nonlinearities, which we write $F \mapsto M_{\ell}F$.

We can then state the following proposition. 
\begin{proposition}\label{prop:renormalisation is covariant}
Let $\mathbf{T} \in \Tran$, let $F$ be a $\mathbf{T}$-covariant nonlinearity, and let $\ell \in \mathfrak{R}$ be $\mathbf{T}$-invariant. Then, for every $\mft \in \Lab_{+}$, 
\[
(\ell \otimes O_{\mft}^{\ast} \otimes \id_{W_{\mft}} ) \bar\bUpsilon_{\mft}(\mathbf{A})
= (\ell \otimes \id_{V_{\mft}} \otimes T_{\mft} ) \bar\bUpsilon_{\mft}( \mathbf{T}^{-1} \mathbf{A})\;.
\] 
In particular, $M_{\ell}F$ is $\mathbf{T}$-covariant.
\end{proposition}
\begin{proof}
Since $F$ is already $\mathbf{T}$-covariant we see that the second statement follows from the first. 
Now observe that
\begin{equs}
 (\ell \otimes \id_{V_{\mft}} \otimes T_{\mft} ) \bar\bUpsilon_{\mft}( \mathbf{T}^{-1} \mathbf{A}) 
&=  (\ell \circ \mathbf{T}^{\ast} \otimes O_{\mft}^{\ast} \otimes \id_{W_{\mft}} ) \bar\bUpsilon_{\mft}(  \mathbf{A})\\ 
&=
 (\mathbf{T} \ell  \otimes O_{\mft}^{\ast} \otimes \id_{W_{\mft}} ) \bar\bUpsilon_{\mft}(  \mathbf{A}) 
 =
  ( \ell  \otimes  O_{\mft}^{\ast} \otimes  \id_{W_{\mft}} ) \bar\bUpsilon_{\mft}(  \mathbf{A}) \;,
 \end{equs}
where in the first equality we used Lemma~\ref{lem:upsilon_behaves} and the last equality we used the $\mathbf{T}$-invariance of $\ell$. 
\end{proof}

\section{The stochastic Yang--Mills--Higgs  equation} \label{sec:renorm-A}

In this section we study the system \eqref{eq:SPDE_for_A} and prove Theorem~\ref{thm:local_exist}.
We will first formulate  \eqref{eq:SPDE_for_A}  in terms of the blackbox theory of \cite{BHZ19,CH16,BCCH21} and vectorial regularity structures of \cite{CCHS2d}, but in order to obtain our desired result we will need to handle the three issues described in Remark~\ref{rem:BPHZ_consts}. 
  
Throughout the rest of the paper, let us fix $\eta\in(-\frac23,-\frac12)$, $\delta\in(\frac34,1)$,  $\beta\in(-\frac12,-2(1-\delta))$, $\theta\in(0,1]$, and $\alpha\in(0,\frac12)$ such that
\begin{equs}
\hat\beta &\eqdef \beta+2(1-\delta)\geq 2\theta(\alpha-1)\;,\label{eq:bar_beta_cond1}
\\
\hat\beta &\geq \hat\eta \eqdef 1-2\delta-\eta\;,\label{eq:bar_beta_cond2}
\\
\hat\beta &< \eta+\f12\;,\label{eq:bar_beta_cond3}
\end{equs}
and such that the conditions of Theorem~\ref{thm:g_bound} hold.
Note that such a choice of parameters is always possible \dash we first choose $\alpha,\theta$, then choose $\eta$ sufficiently close to $-\frac12$ such that
$2\theta(\alpha-1)<\eta+\frac12$,
then choose $\delta$ sufficiently close to $1$ such that $\hat\eta<\eta+\frac12$,
and finally choose $\beta$ such that $\hat \beta$ satisfies~\eqref{eq:bar_beta_cond1}--\eqref{eq:bar_beta_cond3}.
We also remark that the tuple $(\rho,\eta,\beta,\delta,\alpha,\theta)\in\R^6$ satisfies conditions~\eqref{eq:CI}, \eqref{eq:CI'} and~\eqref{eq:CGS} for some $\rho\in(\frac12,1)$.  
We use these parameters to define the metric space $(\state,\Sigma)\equiv (\state_{\eta,\beta,\delta,\alpha,\theta},\Sigma_{\eta,\beta,\delta,\alpha,\theta})$ as in Definition~\ref{def:state}.

We rewrite \eqref{eq:SPDE_for_A} as 
\begin{equs}[e:SPDE-for-X]
\partial_t X &= \Delta X + X \partial X + X^3 
+( (C_{\A}^{\eps})^{\oplus 3} \oplus C_{\Phi}^{\eps} ) X + \sig^\eps \xi^\eps
\\
X_0 & = (a,\phi) \in \state \;.  
\end{equs}
Here we introduced constants $\sig^\eps \in \R$ which will be useful in Section~\ref{sec:gauge}. 

\begin{theorem}\label{thm:local-exist-sigma}
Fix any space-time mollifier $\moll$.   
Then there exist  
\[
(C_{j,\YM}^{\eps}, C_{j,\Higgs}^{\eps})_{ j \in \{1,2\}, \eps \in (0,1]}
\] 
with $C_{j,\YM}^{\eps} \in L_{G}(\mfg,\mfg)$ and  $C_{j,\Higgs}^{\eps} \in L_{G}(\higgsvec,\higgsvec)$ such that the following statements hold.

\begin{enumerate}[label=(\roman*)]
\item\label{pt:A_conv_state}
For any $(\mathring{C}_{\A},\sig) \in L_{G}(\mfg,\mfg)  \times \R$ 
and any sequence $(\mathring{C}_{\A}^{\eps},\sig^{\eps}) \rightarrow (\mathring{C}_{\A},\sig)$
 as $\eps \downarrow 0$, the solutions $X$ to the system \eqref{e:SPDE-for-X}
where
\begin{equs}\label{e:SYM_constants}
C_{\A}^{\eps} 
&= 
	(\sig^{\eps})^{2} C_{1,\YM}^{\eps}
	+(\sig^{\eps})^{4} C_{2,\YM}^{\eps} 
	+\mathring{C}^{\eps}_{\A} \;,
\\
C_{\Phi}^{\eps} 
&=  (\sig^{\eps})^{2} C_{1,\Higgs}^{\eps}
+(\sig^{\eps})^{4} C_{2,\Higgs}^{\eps} -m^2 
\end{equs}
converge in $\state^\sol$ in probability as $\eps \to 0$. 

\item\label{pt:depend_on_lim}
The limit of the solutions depends on the sequence $(\mathring{C}^{\eps}_{\A},\sig^{\eps})$ only through its limit $(\mathring{C}_{\A},\sig)$.
\end{enumerate}
\end{theorem} 

%

Remark that Theorem~\ref{thm:local_exist} clearly follows from Theorem~\ref{thm:local-exist-sigma}.
The rest of this section is devoted to the proof of Theorem~\ref{thm:local-exist-sigma}, which we give in Section~\ref{sec:proof_of_localexist_sym}.
To set up the regularity structure,
we  represent  the Yang--Mills--Higgs field $X$ 
 by a single type denoted as $\mfz$, and we set $W_{\mfz} =E$.  \label{W additive}
The noise $\xi$
is represented by a type $\bar{\mfl}$ and accordingly we set $W_{\bar{\mfl}} = E$.
Our set of types is then specified as $\mfL=\mfL_+\cup \mfL_-$ with
$\mfL_+ = \{\mfz\}$
and $ \mfL_-  = \{\bar{\mfl}\}$.
We write $\CE=\Lab\times \N^{3+1}$,
and $(e_i)_{i=0}^3$ for the generators of $ \N^{3+1}$.

Let $\kappa \in (0,1/100]$.\footnote{We will impose additional, more stringent, smallness requirements on $\kappa$ for other purposes later.}
The associated degrees are defined as
$\deg( \mfz )
=  2 - \kappa $
and $\deg(\bar{\mfl})  =-5/2 - \kappa $,
and the map $\reg: \Lab \rightarrow \R$ verifying the subcriticality of our system is given by
$\reg(\mfz)
=  -1/2 - 4\kappa $
and $\reg(\bar{\mfl})  =-5/2 - 2\kappa $.

We also set our kernel space assignment by setting $\mathcal{K}_{\mfz} = \R$  
and we define a corresponding space assignment $(V_{\mft})_{\mft \in \mfL}$ according to \eqref{eq:space_assignment} \dash this space assignment is then used for the construction of our concrete regularity structure via the functor  $\Func_{V}$ as in  \cite[Sec.~5]{CCHS2d}.

\begin{remark}
Here we have taken a smaller set of types,
namely, if we followed the setting as in  \cite[Sec.~6]{CCHS2d}
we would have to associate a type
to each component of the connection and the Higgs field,
and each component of the noises,
and the target space  $W$ would be
either $\mfg$ or $\higgsvec$ depending on the specific component.
\end{remark}


The form of \eqref{eq:SYMH} motivates us to consider the rule $\mathring{R}$ given by $\mathring{R}(\bar{\mfl}) \eqdef \{\emptyset\}$ and
\begin{equ}\label{e:ymh_rule}
\mathring{R}( \mfz)
\eqdef
\left\{
\bar{\mfl},\ 
\mfz \p_j \mfz  ,\
\mfz \mfz \mfz 
: 1 \le j \le 3 \right\}
\;.
\end{equ}
Here for any $\mft \in \Lab$ we write $\mft = (\mft,0)$ and $\p_j \mft = (\mft,e_j)$ as shorthand for edge types.
We also write products to represent multisets, for instance $ \mfz \p_j \mfz = \{ (\mfz,0), (\mfz,e_j)\}$. 
It is straightforward to verify that $\mathring{R}$ is subcritical, and
has a smallest normal extension which admits a completion $R$
in the sense of \cite[Sec.~5]{BHZ19}
that is also subcritical and will be used to define our regularity structure. 

We also obtain a corresponding set of trees $\mfT=\mfT(R)$ conforming to $R$. 
As in \cite[Sec.~7]{CCHS2d}, we write 
$\Xi=\mcb{I}_{(\bar\mfl,0)}(\bone)$,
and write 
$\bXi \in \mcb{T}[\Xi] \otimes E$
for the corresponding $E$-valued modelled distribution. 
Recall that $\bXi = \id$, the identity map under the canonical identifications
$\mcb{T}[\Xi] \otimes E \simeq V_{\bar \mfl} \otimes E  \simeq E^* \otimes E \simeq L(E,E)$.
For this section and the rest of the paper,
we will adopt the usual graphical notation (e.g.\ \cite{CCHS2d}) for trees:
circles $\<Xi>$ for noises, 
thin and thick lines for $\mcb{I}_\mfz$ and $\mcb{I}_{(\mfz,e_j)}$ for some $j\in[3]$; see Section~\ref{sec:Mass renormalisation} for examples. We will also sometimes write $\partial_j \mcb{I}_\mfz$ for $\mcb{I}_{(\mfz,e_j)}$.

For the kernel assignment, we set  \label{truncation of heat}
$K_{\mfz} = K$ where we fix $K$ to be a truncation of the Green's function $G(z)$ of the heat operator as in \cite[Sec.~6.2]{CCHS2d}.
In particular we choose $K$ so that it satisfies all the symmetries of the heat kernel 
so that, with the notation in Section~\ref{sec:sym_renorm_reg_struct}, for any 
 $\sigma \in S_{3}$, $r \in \{-1,1\}^{3}$ 
and $z \in \R \times \T^{d}$, one has $K_{\mft}(\sigma^{-1} r z) = K_{\mft}(z)$. 
We also fix a random smooth noise assignment 
 $\xi_{\bar{\mfl}} =\sig^{\eps} \xi^{\eps} =\sig^{\eps} \xi \ast \moll^{\eps}$.

 \begin{remark}
Within the rest of this section, to lighten our notation, we will simply assume that 
 $\mathring{C}^{\eps}_{\A} =\mathring{C}_{\A}$ and 
 $\sig^{\eps} =\sig $ (except for the very end of the proof of Theorem~\ref{thm:local-exist-sigma}).
This will not affect any of the arguments in this section. 
\end{remark}

 Following the notation in  \cite[Sec.~7]{CCHS2d},
 for $\pr{\mathbf{A}} \in \mcb{A} \eqdef \prod_{o\in\CE} W_o$
 we write its components in the following  way:
 \begin{equ}[e:pr-X-dX]
 \pr{X} \eqdef  \pr{\mathbf{A}_{(\mfz,0)}} \;, \quad
 \pr{\p_j X} \eqdef \pr{\mathbf{A}_{(\mfz,e_j)}} \;, \quad
 \pr{\bar{\Xi}} \eqdef \pr{\mathbf{A}_{(\bar{\mfl},0)}} 
 \end{equ}
and $\pr{\p X} = ( \pr{\p_1 X} , \pr{\p_2 X} , \pr{\p_3 X} ) \in E^3$. 
 We then define the nonlinearity as $F_{\bar{\mfl}} (\pr{\mathbf{A}})= \id_E$
  and, in the notation of \eqref{e:simpleNotation},
 \begin{equ}[e:YMH-nonlinearity]
F_{\mfz} (\pr{\mathbf{A}})= 
\pr{X \p X} + \pr{X^3} + \mathring C_{\mfz} \pr{X}
+\pr{\bar{\Xi}}
\end{equ}
 where $\mathring C_{\mfz} 
 \eqdef \mathring{C}_{\A}^{\oplus 3} \oplus (-m^2)$.
 
\begin{remark}
Here and below, we always use purple colour to identify elements $\pr{\mathbf{A}}  \in \mcb{A}$ and
their components.
\end{remark}

\subsection{Mass renormalisation}
\label{sec:Mass renormalisation}

Recall that $\mathfrak{T}_{-}=\mathfrak{T}_{-}(R)$ is the set of all unplanted trees in $\mathfrak{T}$ with vanishing polynomial label at the root and negative degree. 
%
Define the subset
\begin{equ}\label{eq:even-ymh-trees}
\mfT_{-}^{\even} \eqdef \Big\{ \tau \in \mfT_{-}:\; \sum_{i=1}^{3} n_{i}(\tau)   \textnormal{ and } |\{ e \in E_{\tau}: \mft(e) = \bar{\mfl} \}|\; \mbox{are  both even} \Big\}
\end{equ}
where $n(\tau) = (n_{0}(\tau),\dots, n_{3}(\tau)) \in \N^{3+1}$ is defined in \eqref{e:def-nf}, $E_{\tau}$ denotes the set of edges of $\tau$, and we write $\mft:E_{\tau} \rightarrow \Lab$ for the map that labels edges with their types. 

Recall from \cite[Prop.~5.68]{CCHS2d} that
to find the renormalised equation, we aim to compute the counterterms
\begin{equ}[e:counterterms-A]
\sum_{
\tau \in \mfT_{-}}
(\ell_{\BPHZ}^{\eps}[\tau] \otimes \id_{W_{\mft}})
\bar\bUpsilon_{\mft}[\tau](\pr{\mathbf{A}})\;.
\end{equ}
Thanks to Lemma~\ref{lem:symbols_vanish} below,
only trees in  $\mfT_{-}^{\even}$
 will actually contribute to \eqref{e:counterterms-A}.

\begin{lemma}\label{lem:symbols_vanish}
Let $\tau\in \mfT$, then $\ell_{\BPHZ}^{\eps}[\tau] \not = 0 \Rightarrow \tau \in \mfT_{-}^{\even}$.
\end{lemma}
\begin{proof}
We first note that by convention $\ell_{\BPHZ}^{\eps}[\tau]  = 0$ for $\tau \in \mfT \setminus \mfT_{-}$.  

Now, if we define $\mathbf{T} = (T,O,\sigma,r)\in \Tran$ by setting $T$, $O$, and $\sigma$ to be identity operators and $r = (-1,-1,-1)$ then we have that $\xi^{\eps}$ and $K$ are both $\mathbf{T}$-invariant,  
so we  have $\ell_{\BPHZ}^{\eps} \circ \mathbf{T}^{\ast}[\tau] =   \ell_{\BPHZ}^{\eps}[\tau]$ by Lemma~\ref{lem:ell-T-inv}. On the other hand $\ell_{\BPHZ}^{\eps} \circ \mathbf{T}^{\ast}[\tau] = (-1)^{\sum_{i=1}^{3}n_{i}(\tau)} \ell_{\BPHZ}^{\eps}[\tau]$.
Therefore, if $\sum_{i=1}^{3}n_{i}(\tau)$ is not even, we must have $ \ell_{\BPHZ}^{\eps}[\tau] = 0$. 

The second constraint defining $\mfT_{-}^{\even}$ can be argued similarly by setting $T_{\bar{\mfl}} = - \id_{W_{\bar{\mfl}}}$ and then $T_{\mfz}$, $S$, $\sigma$, and $r$ to be identity operators. 
\end{proof}
\begin{remark}\label{rem:bare-m-renorm}
Note that, since $d=3$,  
the term $\mathring{C}_{\mfz} X$ 
 leads to trees of the form $\mcb{I}_\mfz (\bar\mfl) \mcb{I}_{(\mfz,e_j)} (\mcb{I}_\mfz ( \bar\mfl ))$
with degree $-\kappa$.  
However, by Lemma~\ref{lem:symbols_vanish}, the BPHZ character vanishes on each of these trees since it has an odd number of derivatives.
This is important because it is convenient for the constants $C^{\eps}_{\YM}$ and $C^{\eps}_{\Higgs}$ appearing below to be independent of the constant $\mathring{C}_{\mfz}$. 
\end{remark} 
One of our key results for this section is the following proposition. 
 
\begin{proposition}\label{prop:mass_term}
Suppose $\tau \in  \mfT_{-}^{\even}$.
Then, for any $v \in \mcb{T}^{\ast}$, there exists $C_{v,\tau} \in L(W_{\mfz},W_{\mfz})$ such that
\[
 (v \otimes \id_{W_{\mfz}} ) \bar\bUpsilon_{\mfz}[\tau](\pr{\mathbf{A}}) = C_{v,\tau}\pr{X}\;.
\]
Furthermore, there exist $C^{\eps}_{\YM} \in L_{G}(\mfg,\mfg)$ 
and $C^{\eps}_{\Higgs} \in  L_{G}(\higgsvec,\higgsvec)$ such that 
\[
\sum_{\tau\in\mfT_{-}}
(\ell^{\eps}_{\BPHZ} [\tau] \otimes \id_{W_{\mfz}} ) \bar\bUpsilon_{\mfz}[\tau](\pr{\mathbf{A}}) = \Big( \bigoplus_{i=1}^{3} C^{\eps}_{\YM} \pr{A_i} \Big) \oplus C^{\eps}_{\Higgs} \pr{\Phi}\;,
\]
where $\pr{X} =((\pr{A_{i}})_{i=1}^{3} , \pr{\Phi}) \in \mfg^{3} \oplus \higgsvec = W_{\mfz}$,
and the linear maps $C^{\eps}_{\YM}$, $C^{\eps}_{\Higgs} $ have the forms
\begin{equ}[e:C-YM-Higgs-sig]
C^{\eps}_{\YM}= 
	\sig^{2} C_{1,\YM}^{\eps}
	+\sig^{4} C_{2,\YM}^{\eps} 
\quad \mbox{and}\quad
C^{\eps}_{\Higgs} = 
  \sig^{2} C_{1,\Higgs}^{\eps}
+\sig^{4} C_{2,\Higgs}^{\eps}
\end{equ}
where  for $j\in \{1,2\}$, the maps $C^{\eps}_{j,\YM} \in L_{G}(\mfg,\mfg)$ 
and $C^{\eps}_{j,\Higgs} \in  L_{G}(\higgsvec,\higgsvec)$  are independent of 
$\sig$.
\end{proposition}

\begin{proof}
The first statement, that the left-hand side is a linear function of $\pr{X}$, is shown in Lemma~\ref{lem:linearX}. 
The second statement is proven in Lemma~\ref{lem:LX-diag}. 
\end{proof}

%
%

Before going into the proof of Proposition~\ref{prop:mass_term}, we
give an explicit calculation of
\eqref{e:counterterms-A} at the ``leading order''.
This will not be used elsewhere but demonstrates how the ``basis-free'' framework developed in \cite[Sec.~5]{CCHS2d} applies to an SPDE where the nonlinearities are given by three intrinsic operations: 
the Lie brackets, the bilinear form $\mathbf{B}$, and the natural action of $\mfg$ on $\higgsvec$. 

Until the start of Section~\ref{subsec:ymh_renorm_linear}, consider $d=2,3$.
 As in 
\cite[Eq.~(6.7)]{CCHS2d}, $ \mathfrak{T}_{\lead}= \{\<I[I'Xi]I'Xi_notriangle> ,\<IXiI'[I'Xi]_notriangle> , \<IXi^2>\}$ consists of
all the ``leading order'' trees, i.e.\ trees of degree $-1-2\kappa$  for $d=3$ (resp.\ $-2\kappa$ for $d=2$) that contribute to \eqref{e:counterterms-A}.
Denoting $(t_i)$ (resp.\ $(v_\mu)$) an orthonormal basis of $\mfg$ (resp.\ $\higgsvec$),
our claim is that,
at the leading order,
the renormalisation is
 given by the following natural $G$-invariant linear operators:
\begin{alignat*}{3}
\ad_{\Cas} \colon\mfg & \to \mfg 
&\quad
{\bf C}_{\Cas}\colon \mfg &\to \mfg 
&\quad
 \brho(\Cas)  \colon \higgsvec & \to \higgsvec
\\
A &\mapsto [t_i,[t_i,A]]
&\quad
A &\mapsto \mathbf{B}(v_\mu \otimes A v_\mu)
&\quad
\Phi  &\mapsto\mathbf{B} (v_\mu \otimes \Phi) v_\mu
\end{alignat*}
with Einstein's convention of summation.

\begin{remark}
The first map is indeed $\ad_{\Cas}$ with $\Cas=t_i\otimes t_i \in \mfg\otimes_s\mfg$ the Casimir element, see \cite[Remark~6.8]{CCHS2d}.
For the last map we note that  by \eqref{e:def-B} and the fact that the representation is orthogonal,
\begin{equs}
\mathbf{B} (v_\mu \otimes \Phi)\, v_\mu
&= \brho(\scal{\mathbf{B}(v_\mu \otimes \Phi),t_i}_\mfg t_i) \, v_\mu
= \scal{v_\mu,\brho(t_i) \Phi}_{\higgsvec}\, \brho(t_i) v_\mu \\
&= \scal{\brho(t_i)^* v_\mu,\brho(t_i) \Phi}_{\higgsvec}\, \brho(t_i)\brho(t_i)^* v_\mu
= \scal{v_\mu,\brho(t_i)^2 \Phi}_{\higgsvec}\, v_\mu \\
&= \brho(t_i)^2 \Phi = \brho(\Cas) \Phi\;.
\end{equs}
We remark that  $\ad_{\Cas}$ is  a multiple of $\id_{\mfg}$
when $\mfg$ is simple, and
$\brho(\Cas)$ is a multiple of $\id_{\higgsvec}$ when the representation is irreducible (see \cite[Prop.~10.6]{Hall15} for a formula of this multiple).
For the  map ${\bf C}_{\Cas}$, one has
\begin{equ}
{\bf C}_{\Cas} A =
\scal{\mathbf{B}(v_\mu \otimes \brho(A) v_\mu) , t_i }_\mfg t_i
\stackrel{\eqref{e:def-B}}{=} \scal{v_\mu , \brho(t_i)\brho( A) v_\mu }_{\higgsvec} t_i
= \tr (\brho(t_i)\brho( A)) t_i \;,
\end{equ}
which is just the contraction of $\Cas$ with $A\in \mfg$
using the Hilbert--Schmidt inner product on $\End(\higgsvec)$.
It is easy to check their $G$-invariance. 
\end{remark}

Following the notation of \cite[Sec.~6]{CCHS2d},
but with dimension $d$ and the Higgs field $\Phi$,
one has\footnote{To avoid any confusion with other renormalisation operators we consider later, we remark that $\bar C^\eps$ and $\hat{C}^\eps$ defined here are only used within Lemma~\ref{lem:deg-1} and Remark~\ref{rem:finite-in-2d}.}
\begin{equ}
\ell_{\BPHZ}^{\eps}[\<I[I'Xi]I'Xi_notriangle>] = 
-\ell_{\BPHZ}^{\eps}[\<IXiI'[I'Xi]_notriangle>] = - \hat{C}^{\eps}\bCov\;,\qquad
\ell_{\BPHZ}^{\eps}[\<IXi^2>] = -\bar{C}^\eps \bCov\;,
\end{equ}
where  $\bCov$ is as in \eqref{e:def-bCov} and 
\begin{equ}
\bar C^\eps \eqdef \int dz\ K^\eps(z)^{2},
\quad
\hat{C}^\eps
\eqdef \int dz\ \partial_{j}K^\eps(z)(\partial_{j}K*K^\eps)(z)\;.
\end{equ}
Here, the index $j \in [d]$ is fixed (no summation) and $\hat{C}^\eps$ is clearly independent of it.

\begin{lemma}\label{lem:deg-1}
We have
$\sum_{\tau\in \mathfrak{T}_{\lead}}
(\ell_{\BPHZ}^{\eps}[\tau] \otimes \id_{E})
\bar\bUpsilon_{\mfz}[\tau](\pr{\mathbf{A}})
= \sig^2 \mcL_{\lead} \pr{X}$
where 
\begin{equs}
\mcL_{\lead} \eqdef 
\Big( ( (1-d)\bar C^{\eps} + (6d-8)\hat{C}^\eps) \ad_{\Cas} 
 &+ ( 4\hat C^\eps - \bar{C}^\eps) {\bf C}_{\Cas} \Big)^{\oplus d}
 \\
&   \oplus \Big( (2 d \hat{C}^\eps -d \bar C^{\eps})\brho(\Cas) - \bar C^{\eps} \id_{\higgsvec}\Big)\;.
\end{equs}
\end{lemma}

\begin{proof}[(Sketch)]
By straightforward calculation one has\footnote{
For instance, in the first identity, $2 \hat C^\eps {\bf C}_{\Cas}$
 is obtained by ``substituting''
$2  \mcb{I}  (A_j \p_j \Phi)$ into
 $ \mathbf{B}(\partial_{i} \Phi  \otimes \Phi)$ for $\Phi$,
and
$2 d  \hat C^\eps   \brho(\Cas) $
is obtained by 
substituting $\mcb{I} (\mathbf{B} (\p_j \Phi \otimes \Phi))$
into $2A_j \p_j \Phi$ for $A_j$.}
\begin{equs}
(\ell_{\BPHZ}^{\eps}[\<I[I'Xi]I'Xi_notriangle>] \otimes \id)
\bar{\bUpsilon}_{\mfz}[\<I[I'Xi]I'Xi_notriangle>](\pr{\mathbf{A}})
&=
 \hat C^\eps 
 \Big(\big((4d-5)  \ad_{\Cas} 
 + 2 {\bf C}_{\Cas} \big)^{\oplus d} \oplus  2 d  \brho(\Cas) \Big) \pr{X}\;,
\\
(\ell_{\BPHZ}^{\eps}[\<IXiI'[I'Xi]_notriangle>] \otimes \id)
\bar{\bUpsilon}_\mfz [\<IXiI'[I'Xi]_notriangle>](\pr{\mathbf{A}})
&=
\hat C^\eps 
\Big(\big((2d-3) \ad_{\Cas} 
+2 {\bf C}_{\Cas} \big)^{\oplus d}\oplus 0 \Big)\pr{X}\;,
\\
 (\ell_{\BPHZ}^{\eps}[\<IXi^2>] \otimes \id) 
 \bar\bUpsilon_\mfz [\<IXi^2>](\pr{\mathbf A}) 
&=
 \bar{C}^\eps
 \Big( \big((1-d) \ad_{\Cas} 
 -{\bf C}_{\Cas} \big)^{\oplus d} \oplus   (-d \brho(\Cas) -\id_{\higgsvec} )  \Big)\pr{X}.
\end{equs}
Adding all these identities we obtain the claimed map $\mcL_{\lead}$.
\end{proof}

%
%
%

\begin{remark}\label{rem:finite-in-2d}
$\bar C^{\eps} -2d\hat{C}^\eps$ converges to a finite value as $\eps\to 0$, which essentially follows from \cite[Proof of Lemma~6.9]{CCHS2d}.
 So when $d=3$, both the coefficients of $\ad_{\Cas} $ and ${\bf C}_{\Cas}$ are  {\it divergent} (at rate $O(\eps^{-1})$).
Interestingly, if $d=2$,  the coefficients of $\ad_{\Cas} $ and ${\bf C}_{\Cas}$ both
converge to finite limits (which was shown in the case without Higgs field in \cite[Sec.~6]{CCHS2d}).
With this remark and Lemma~\ref{lem:deg-1} one should be able to extend the main results of \cite{CCHS2d} to the
Yang--Mills--Higgs case in 2D by following the arguments therein.
\end{remark}

Note that 
there are  a large number of trees in  $\mfT_{-}^{\even}$ 
besides the ones in Lemma~\ref{lem:deg-1}, for instance
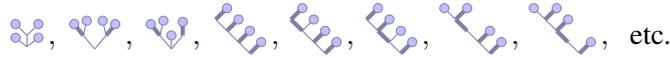
\begin{equ}
\<Psi2I[Psi2]_notriangle>\,,\;
\<YY_notriangle>\,,\;
\<PsiYI[Psi']_notriangle>\,,\;
\<Psi'I[Psi'I[LPsi']]_notriangle>\,,\,
\<Psi'I[Psi'I[L'Psi]]_notriangle>\,,\,
\<Psi'I[PsiI'[LPsi']]_notriangle>\,,\,
\<Psi'I[Psi'I[Y']]_notriangle>\,,\,
\<Psi'I[I'[YPsi']]_notriangle>\,,\;\;
 \mbox{etc.}
\end{equ}
Below we develop more systematic arguments to find their contribution to the renormalised equation.

\subsubsection{Linearity of renormalisation}\label{subsec:ymh_renorm_linear}

We prove the first statement of Proposition~\ref{prop:mass_term},
namely we only have linear renormalisation.
Note that for each $\tau \in \mfT_{-}(R)$,
$\bar\bUpsilon_{\mfz}[\tau](\pr{\mathbf{A}})$
is polynomial in $\pr{\mathbf{A}}$ (namely, in $\pr{X}$ and its derivatives).
Obviously it suffices to show that each term of the polynomial 
is linear in $\pr{X}$.
Fixing such a term,
we write $p_X$ and $p_\p $ for the total powers of $\pr{X}$ and derivatives respectively in this   monomial, for instance 
for $\pr{X \p X} $ we have  $p_X=2$ and $p_\p =1$.

It turns out to be convenient to introduce a formal parameter $\lambda$
and to write the nonlinear terms of our SPDE\footnote{One can eventually take $\lambda=1$, but having these coefficients will be helpful for power counting: pretending that $\lambda$ has ``degree $-1/2$'', these nonlinear terms all have the same degree as white noise.}
as $\lambda X\p X + \lambda^2 X^3$. 
Since $\bar\bUpsilon_{\mfz}[\tau]$ is generated by iterative substitutions
with these nonlinear terms, each term of $\bar\bUpsilon_{\mfz}[\tau]$
comes with a coefficient   $\lambda^{n_\lambda}$ for some positive integer $n_\lambda$. 
For instance  $\bar\bUpsilon_{\mfz}[\<YY_notriangle>]$ is associated with  $n_\lambda=4$. 

For a fixed tree $\tau$ 
we also write 
 $k_\xi  =   |\{ e \in E_{\tau}: \mft(e) = \bar{\mfl} \}|$ for the total number of noises
 and $k_\p = \sum_{i=1}^{3} |n_{i}(\tau)| $ for the total number of derivatives plus the total powers of $\mathbf{X}$.  (Recall \eqref{e:def-nf} for relevant notation.)
 For instance for the tree $\<YY_notriangle>$ 
 we have  $k_\xi=4$ and $k_\p =2$.

\begin{lemma}\label{lem:counting}
Let $\tau \in \mfT_{-}(R)$.
For each term of $\bar\bUpsilon_{\mfz}[\tau]$  
which is associated with numbers $(n_\lambda, k_{\p}, k_\xi,p_X, p_\p)$ as above, we have
\minilab{counting-identities}
\begin{equs}
{} k_\xi &= n_\lambda + 1 -p_X \;,		\label{e:iden1}
\\
n_\lambda + k_\p + p_\p  &\equiv 0 \; \mod 2 \;,		\label{e:iden2}
\\
 \deg(\tau) &= \frac{n_\lambda}{2} - \frac{5}{2} + \frac{p_X}{2} + p_\p  - k_\xi \kappa\;.
		\label{e:iden3}
\end{equs}
\end{lemma}

\begin{proof}
For $\tau =\mbX^q \prod_{i=1}^m \mcb{I}_{o_i}(\tau_{i})$ where  $o_i \in \CE$, recall our recursive definition 
\begin{equ}[e:def-barUp]
 \bar{\bUpsilon}_{\mfz} [\tau] 
\eqdef \mbX^q
\Big[ 
\partial^q D_{o_1}\cdots D_{o_m}
\bar{\bUpsilon}_{\mfz}[\bone]
\Big]
\big(\bUpsilon_{o_1}[\tau_1],\ldots,\bUpsilon_{o_m}[\tau_m]\big)
\end{equ}
where $\bar{\bUpsilon}_{\mfz}[\bone] = \bone \otimes (\lambda\pr{X \p X} +\lambda^2 \pr{X^3} + \pr{\xi})$. 
Since our SPDE has additive noise, we only need to
consider the case $o_i \in \Lab_+ \times \N^{d+1}$ for all $i\in [m]$.
By linearity it suffices to prove the lemma assuming that $\bar{\bUpsilon}_{\mfz} [\tau] $ is monomial.  Also note that it suffices to show the case $q=0$,
since increasing $|q|$ by $1$ amounts to
increasing each of $p_{\p}, k_{\p},\deg(\tau) $ by $1$, which preserves 
\eqref{counting-identities}.

We prove the lemma by induction. 
The base case of the induction is $\tau = \Xi = \CI_{\bar{\mfl}} [\bone]$
which has $\deg(\tau) = -5/2-\kappa$, and its associated five numbers
$(n_\lambda, k_{\p}, k_\xi,p_X, p_\p) = (0,0,1,0,0)$
obviously satisfy the three identities.\footnote{We do not start the induction from $\tau=\bone$ since it does not satisfy the last identity: $\bar{\bUpsilon}_{\mfz}[\bone] $ has a term $\bone \otimes \pr{\xi}$,
for which
$n_\lambda=k_\xi=p_X=p_\p =0$, and $\deg(\bone)=0$, so \eqref{e:iden3} would read $0=-\frac52$.}

For each $1\le i\le m$
denote by $(n_\lambda^{(i)}, k_{\p}^{(i)}, k_\xi^{(i)} ,p_X^{(i)}, p_\p^{(i)})$
the numbers associated to each $\bar{\bUpsilon}_{o_i}[\tau_i]$,
and our induction assumption is that they all satisfy \eqref{counting-identities}.
Denote by $(n_\lambda^{(0)}, k_{\p}^{(0)}, k_\xi^{(0)} ,p_X^{(0)}, p_\p^{(0)})$
the numbers associated to either the term 
$\bone \otimes (\lambda\pr{X \p X})$ or the term
 $\bone \otimes  (\lambda^2 \pr{X^3})$,
 which are $(1,0,0,2,1)$ and $(2,0,0,3,0)$ respectively.
 It is easy to check that these satisfy \eqref{counting-identities}
 since $\deg(\one)=0$.

From the recursion \eqref{e:def-barUp} we  observe that 
\begin{equ}[e:knp-ind]
k_\xi = \sum_{i=0}^m k_\xi^{(i)}\;, \quad
n_\lambda = \sum_{i=0}^m n_\lambda^{(i)}\;,\quad
p_X = \Big(\sum_{i=0}^m p_X^{(i)}\Big)-m\;, 
\end{equ}
and $p_{\p} + k_{\p} = \sum_{i=0}^m (p_{\p}^{(i)} + k_{\p}^{(i)})$.
Using these relations together with our induction assumption, one easily checks that
 \eqref{e:iden1} and \eqref{e:iden2} hold for $\tau$.

Finally, observe that 
if the factor $\pr{\p X}$ in $\pr{ X \p X}$ is not substituted (i.e.\ $o_i=(\mfz,0)$ for all $i\in [m]$), one has
$\deg(\tau) = \sum_{i=1}^m \deg(\tau_i) + 2m$
and $p_{\p} = \sum_{i=0}^m p_{\p}^{(i)}$,
where the term $2m$ arises from the increase of degree by the heat kernel.
On the other hand
if $\pr{\p X}$ is substituted by some $\bUpsilon_{o_i}[\tau_i]$ (i.e.\ $o_i=(\mfz,e_j)$ for some $i\in [m]$ and $j\in \{1,2,3\}$),  one has
$\deg(\tau) = \sum_{i=1}^m \deg(\tau_i) + 2m-1$
and $p_{\p} = (\sum_{i=0}^m p_{\p}^{(i)}) -1$.
Using these relations and \eqref{e:knp-ind} as well as our inductive assumption,
we see that  in both cases the last identity \eqref{e:iden3} is preserved for $\tau$.
\end{proof}

The following lemma shows that we only have linear renormalisation.

\begin{lemma}\label{lem:linearX}
There exists $\mcL\in L (E,E)$ such that 
\begin{equ}
\sum_{\tau \in \mfT_{-}} 
(\ell_{\BPHZ}^{\eps}[\tau] \otimes \id_{E})
\bar\bUpsilon_{\mfz}[\tau](\pr{\mathbf{A}})
=\mcL \pr{X}\;.
\end{equ}
The map $\mcL$ has the form
$\mcL = \sig^{2} \mcL_1
	+\sig^{4} \mcL_2$
where for each $j\in \{1,2\}$, the map $\mcL_j \in L(E,E)$ 
is independent of $\sig$.
\end{lemma}

\begin{proof}
It suffices to prove that for each tree $\tau  \in \mfT_{-}(R)$, we have either $(p_X, p_\p) = (1,0)$ or $\ell_{\BPHZ}^{\eps}[\tau]=0$.

The proof easily follows from analysing each possible value of $n_\lambda$ (which is the advantage of introducing this parameter). By the ``parity'' Lemma~\ref{lem:symbols_vanish}, both $k_\xi\ge 2$ and $k_\p$  must be even,
or $\ell_{\BPHZ}^{\eps}[\tau]=0$.  Obviously we also have $p_X = 0 \Rightarrow p_\p =0$.

Let $n_\lambda=2$. We look for solutions to\footnote{In this proof we drop the term $-k_{\xi} \kappa$ in \eqref{e:iden3} since it is irrelevant after we have chosen $\kappa>0$ sufficiently small.}
\[
\deg(\tau) =   - \frac{3}{2} + \frac{p_X}{2} + p_\p \le 0\;.
\]
To have even $k_\xi\ge 2$, by  \eqref{e:iden1} of Lemma~\ref{lem:counting}, $p_X$ must be odd.
To have even $k_\p$, by  \eqref{e:iden2}  $p_\p$ must be even.
Thus we only have one solution $(p_X, p_\p) = (1,0)$ to the above inequality.
For this solution, by  \eqref{e:iden1}, we have $k_\xi =2$.

Let $n_\lambda=3$ and we solve
\[
\deg(\tau) =   -1 + \frac{p_X}{2} + p_\p \le 0\;.
\]
As above, to have even $k_\xi\ge 2$, $p_X \in \{2,0\}$.
To have even $k_\p$, the number $p_\p$ must be odd.  
Since whenever $p_X = 0$ we must have  $p_\p =0$ as mentioned above, there is then no solution to the above inequality.

Consider $n_\lambda=4$ and
\[
\deg(\tau) =   -\frac12 + \frac{p_X}{2} + p_\p \le 0\;.
\]
To have even $k_\xi\ge 2$, $p_X \in \{1,3\}$.
To have even $k_\p$, the number $p_\p$ must be even. So the only solution is $(p_X, p_\p) = (1,0)$.
For this solution, by  \eqref{e:iden1}, we have $k_\xi =4$.

Consider $n_\lambda=5$ and
$\deg(\tau) =   \frac{p_X}{2} + p_\p \le 0$.
To have even $k_\p$, the number $p_\p$ must be odd, so there is no solution. 
Finally for $n_\lambda\ge 6$, $\deg(\tau)>0$.
Putting these together proves the first claim.
Moreover, since we have shown above that $k_\xi$ can only be $2$ or $4$, 
and $\sig$ is the coefficient in front of the noise, the map $\mcL$ has the claimed form.
\end{proof}

\subsubsection{Symmetries of the renormalisation}\label{sec:Sym-and-con}

Having proved that the renormalisation
is linear in $X$, we now show that 
$\CL$ is subject to a number of constraints by symmetries, allowing us
to complete the proof of Proposition~\ref{prop:mass_term}.
\begin{lemma}\label{lem:LX-diag}
For any $\eps > 0$, the linear map $\mcL\in L(E,E) $ in Lemma~\ref{lem:linearX} 
 has  the form $\CL = C^\eps_{\YM} \oplus C_{\YM}^\eps \oplus C_{\YM}^\eps \oplus C_{\Higgs}^\eps $ 
 with $C^\eps_{\YM} \in L(\mfg,\mfg)$ and $C_{\Higgs}^\eps \in L(\higgsvec,\higgsvec)$.
 Moreover $\mcL \in L_G(E,E) $, namely $\mcL$
 commutes with the action $\Ad\oplus \brho$ of $G$ on $E$. In particular,  $C_{\YM}^\eps \in L_{G}(\mfg,\mfg)$ and  $C_{\Higgs}^\eps \in L_{G}(\higgsvec,\higgsvec)$ . 
\end{lemma}

\begin{proof}
Our proof will repeatedly reference Proposition~\ref{prop:renormalisation is covariant}, with various choices of $\mathbf{T}  = (T,O,\sigma,r) \in \mathrm{Tran}$ to prove the various properties of $\mathcal{L}$. 
We will always choose $T$ of the form $T = T_{\mfz} \oplus T_{\bar{\mfl}} = \tilde{T} \oplus \tilde{T}$ for some $\tilde{T} \in L(E,E)$ \dash recall that $W_{\mfz}\simeq W_{\bar{\mfl}} \simeq E = \mfg_1 \oplus \mfg_2 \oplus \mfg_3\oplus \higgsvec$ where each $\mfg_i$ is a copy of $\mfg$. 

We first show that $\mathcal{L}$ is appropriately block diagonal. 
Fix $i\in \{1,2,3\}$ and choose $\mathbf{T} = (T,O,\sigma,r) \in \mathrm{Tran}$ as follows.
The map $\tilde{T}$ acts on $E$ 
by flipping the sign of the $i$-th component, namely
for every $u\in E$, let $(\tilde{T} u) |_{\mfg_i} = - u |_{\mfg_i} $ and
 $(\tilde{T} u) |_{\mfg_i^\perp} =  u |_{\mfg_i^\perp} $. 
We then also flip the sign of the $i$-th spatial coordinate,
namely $\sigma=\id $, $O=\id$ and  $r_j =1_{j \neq i} - 1_{j=i} $ for every $j\in \{1,2,3\}$.

Observe that our nonlinearity 
 $F(\pr{\mathbf{A}})$ 
   is then $\mathbf{T}$-covariant, namely $ F_\mfz(\mathbf{T}\pr{ \mathbf{A}}) =  T_\mfz F_\mfz(\pr{ \mathbf{A}})$; this is because, in any term for the $\mfg_{i}$ component in \eqref{e:XdX-X3}, the spatial index $j$ (appearing either as a subscript of $\pr{A_{\bullet}}$ or as a partial derivative $\pr{\partial_{\bullet}}$) appears an even number of times if $j \not = i$ and an odd number of times if $j=i$.
 For instance, one of the terms in the $\mfg_i$ component
$[ \pr{A_j}, \pr{\p_j A_i}]$ flips sign when $\pr{\mathbf{A}}$ is replaced by $\mathbf{T} \pr{\mathbf{A}}$,
because according to the definition
of $\mathbf{T} \pr{\mathbf{A}}$ given by \eqref{eq:transform_on_jets},
 if $j\neq i$, $\pr{A_j}$ is fixed and $\pr{\p_j A_i}$ flips sign,
and  if $j= i$, $\pr{A_j}$ flips sign and $\pr{\p_j A_i}$ is fixed.

Also, $K$ and $\xi_\eps$ are both  $\mathbf{T}$-invariant,
so by Lemma~\ref{lem:ell-T-inv}, 
both $\bar{\PPi}_{\can}$ and $\ell_{\BPHZ}$ are $\mathbf{T}$-invariant. 
Invoking Proposition~\ref{prop:renormalisation is covariant}
and  Lemma~\ref{lem:linearX}  we conclude that $\mcL X$ is $\mathbf{T}$-covariant.
Since this holds for every $i \in \{1,2,3\}$, one has
$\CL = C_\eps^{(1)} \oplus C_\eps^{(2)} \oplus C_\eps^{(3)} \oplus C_{\Higgs}^\eps$ 
 with $C_\eps^{(1)},C_\eps^{(2)},C_\eps^{(3)} \in L(\mfg,\mfg)$ and $C_{\Higgs}^\eps \in L(\higgsvec,\higgsvec)$.

We now show that the first three blocks are identical. 
Fixing $i\neq j \in \{1,2,3\}$,  
we choose another $\mathbf{T} = (T, O,\sigma,r) \in \mathrm{Tran}$ 
where $\sigma \in S_3$ is defined by swapping $i$ and $j$,
 $r=1$, $O=\id$, and $T$ is given 
by swapping the $\mfg_i$ and $\mfg_j$ components,
namely
\[
(\tilde{T} u) |_{\mfg_i} = u |_{\mfg_j} \;,
\quad
(\tilde{T} u) |_{\mfg_j} = u |_{\mfg_i} \;,
\quad
 (\tilde{T} u) |_{(\mfg_i\oplus \mfg_j)^\perp} =  u |_{(\mfg_i\oplus \mfg_j)^\perp} \;.
 \] 
It is easy to check that $F(\pr{\mathbf{A}})$ 
   is again  $\mathbf{T}$-covariant,
   and the kernel $K$ and noise $\xi$ are both  $\mathbf{T}$-invariant.
   Invoking Proposition~\ref{prop:renormalisation is covariant} again
 it follows that $\mcL X$ is $\mathbf{T}$-covariant
 which implies that $C_\eps^{(1)}=C_\eps^{(2)}=C_\eps^{(3)} \eqdef C^\eps_{\YM}$.

To show $\mathcal{L} \in L_{G}(E,E)$
we choose $\mathbf{T} = (T, O,\sigma,r) \in \mathrm{Tran}$ by 
taking  $r=1$, $O=\id$ and $\sigma=\id$,
and, for any fixed $g \in G$, define $\tilde{T}$
to be the action by $g\in G$ on $E$. 
Note that  $F(\pr{\mathbf{A}})$  is   $\mathbf{T}$-covariant
(since each term in \eqref{eq:SYMH} is covariant),
and the noise $\xi_\eps$ is $\mathbf{T}$-invariant.
Since this holds for any $g \in G$, we thus have $\mcL \in L_G(E,E) $.
\end{proof} 

\begin{remark}
It is possible to show that $\mcL$
in Lemma~\ref{lem:linearX} -- \ref{lem:LX-diag} is symmetric 
with respect to the inner product given on $E$.
Indeed, recalling that $\higgsvec \cong \higgsvec^*$ is canonically
given by the scalar product $\langle \;,\;\rangle_\higgsvec$,
 the action \eqref{eq:YM_energy}
 is invariant under
$\Phi\mapsto \Phi^*$ (with connection replaced by its dual connection, which is locally still represented by $A$), 
and  \eqref{eq:SYMH} is covariant i.e.
the transformation $(\Phi,\zeta) \mapsto (\Phi^*,\zeta^*)$
just amounts to applying the dual operation to the second equation.
A similar argument as above can show that $\mcL^* \Phi^* = (\mcL\Phi)^*$
which means that $\mcL$ is symmetric. 
If  $\mfg$
is simple and $\brho$ is  irreducible, and assuming
$\brho$ is surjective, then $ C_{\YM}^\eps , C_{\Higgs}^\eps$
commute with all the orthogonal 
transformations and therefore must be multiples of the identity on $\mfg $ and $\higgsvec$ respectively.
\end{remark}

\begin{remark}
One may wonder if  our model ``decouples'' 
as $\mfg$ splits into simple and abelian components and $\higgsvec$ decomposes into irreducible subspaces.
The ``pure YM'' part (i.e.\ the first term in \eqref{eq:YM_energy}) decouples
under the decomposition of $\mfg$, as observed in \cite[Remarks~2.8,~2.10]{CCHS2d}.
The term $|\mrd_A \Phi (x)|^2$ (and the corresponding terms in our SPDE) decouples into orthogonal irreducible components,
but the $|\Phi|^4$ term does not.
%
%
On the other hand, assuming that $\higgsvec$ is  irreducible but 
$\mfg=\mfg_1\oplus \mfg_2$, 
by \cite[Theorem~3.9]{SepanskiLieBook},
$\higgsvec = \higgsvec_1\otimes \higgsvec_2$ for  irreducible representations $\higgsvec_i$ of $\mfg_i$.
So for $A_j\in \mfg_j$
and  $\Phi_j \in \higgsvec_j$ we have
$\mrd_{A_1+A_2} (\Phi_1\otimes \Phi_2) = (\mrd_{A_1} \Phi_1 )\otimes \Phi_2 + \Phi_1 \otimes (\mrd_{A_2}  \Phi_2)$;  the two terms are generally not orthogonal, and one does not have any decoupling.
\end{remark}

\subsection{Solution theory}
\label{sec:solution theory additive}

We now turn to posing the analytic fixed point problem in an appropriate space of modelled distributions
for \eqref{e:YMH-nonlinearity}.
A naive formulation would be
 \begin{equs}[e:abs-fix-pt-X0]
 \mcX
&=
\CG_\mfz \mathbf{1}_{+}\big(
\mcX \partial \mcX + \mcX^3
 + \mathring{C}_{\mfz} \mcX+  \boldsymbol{\Xi} \big)
  + \CP X_0\;.
 \end{equs}
 Here, $\mathbf{1}_{+}$ is the restriction of modelled distributions to non-negative times,  
$\CG_\mfz \eqdef \mcb{K}_{\mfz} + R \mathcal{R}$
where
$ \mcb{K}_{\mfz}  $ is the abstract integration operator,
and $R$  the operator realising convolution with $G - K$ 
as a map from appropriate H{\"o}lder--Besov functions into modelled distributions as in \cite[Eq.~7.7]{Hairer14},
and   \label{harmonic extension}
$\CP X_0$ is the ``harmonic extension'' of $X_0$ as in \cite[Eq.~7.13]{Hairer14}. 
However, 
  \eqref{e:abs-fix-pt-X0} can not be closed in any $\cD^{\gamma,\eta}_{\alpha}$ space,
even for smooth initial data.
This is because 
$\mathbf{1}_{+} \boldsymbol{\Xi} \in \cD^{\infty,\infty}_{-5/2-}$, so $\mcX \in \cD^{\gamma,\eta}_{\alpha}$ would require $\eta, \alpha < -1/2$,
but then $\mcX \partial \mcX \in \cD^{\gamma-\frac32-, -2-}_{-2-}$ at best. Unfortunately, exponents below $-2$ represent a 
non-integrable singularity in the time variable so that we cannot apply the standard integration result \cite[Prop.~6.16]{Hairer14}
(which requires $\eta\wedge\alpha>-2$) for the modelled distributions $\mathbf{1}_{+} \boldsymbol{\Xi}$ and $\mcX \partial \mcX$.\footnote{
Another way to bypass this problem is to let $\Xi$ represent the noise {\it restricted to positive times}, that is we could put the indicator function for positive times inside the model instead of the fixed point problem for modelled distributions.
However, working with a non-stationary noise would create serious technical difficulties 
since we would not be able to use \cite{CH16} to control the given models 
\dash in general the trees in $\mfT_{-}$ would need time-dependent renormalisation counter-terms.} 

Recall the local reconstruction operator $\tilde{\CR}$ and the global reconstruction operator $\CR$  described in
Appendix~\ref{app:Singular modelled distributions}. The reason that the proof of \cite[Prop.~6.16]{Hairer14} fails in this case is that the modelled distributions $\mathbf{1}_{+} \boldsymbol{\Xi}$  and $\mcX \partial \mcX$ canonically only admit  local (but not global) reconstructions $\tilde{\CR} \mathbf{1}_{+} \boldsymbol{\Xi}$ and $\tilde{\CR} (\mcX \partial \mcX)$ which are defined as space-time distributions only away from the $t=0$ hyperplane. 
However, Lemma~\ref{lem:Schauder-input} allows us to bypass this difficulty if we also specify space-time distributions that extend  $\tilde{\CR} \mathbf{1}_{+} \boldsymbol{\Xi}$ and $\tilde{\CR} (\mcX \partial \mcX)$ to $t=0$.  

More precisely, for fixed $\eps > 0$, we  can easily define such an extension for  $\tilde{\CR} \mathbf{1}_{+} \boldsymbol{\Xi}$, and by linearising around the SHE we can similarly handle the product $\mcX \partial \mcX$. 
Let 
\begin{equ}[e:def-tilde-Psi]
\tilde{\Psi}_{\eps} = \sig K \ast (1_{t > 0} \xi^{\eps})
\end{equ}
 and consider 
 \begin{equs}[e:abs-fix-pt-X0']
 \mcX &= \tilde{\mcX} + \tilde{\bPsi} \;, \quad 
 \tilde{\bPsi}  =  \mcb{K}_{\mfz}^{\sig 1_{t > 0} \xi^{\eps}}(\one_{+} \bXi)\;,\\
 \tilde{\mcX}
&=
\CG_\mfz \mathbf{1}_{+}\big(\mcX^3
 + \mathring{C}_{\mfz} \mcX \big)
 +
 \CG^{\tilde{\Psi}_{\eps} \partial \tilde{\Psi}_{\eps}}_\mfz 
  \big( \tilde{\bPsi}   \partial  \tilde{\bPsi}  \big)\\
   {}&\quad
  +
 \CG_\mfz \mathbf{1}_{+} 
 \big( \tilde{\mcX} \partial \tilde{\bPsi}  +  \tilde{\bPsi} \partial  \tilde{\mcX} + \tilde{\mcX} \partial \tilde{\mcX} \big) 
  +R  ( \sig 1_{t > 0} \xi^{\eps}) +\CP X_0\;,
 \end{equs}
where, for a  space-time distribution $\omega$, the notation
 $\CG_{\mfz}^{\omega} f = \mcb{K}_{\mfz}^{\omega} f  + R \omega$
 is defined 
 as in Appendix~\ref{app:Singular modelled distributions}.  The space-time distributions in the superscripts here
play the role of ``inputs by hand''  to the integration operators which replace the standard reconstructions that aren't defined a priori.
 Pretending for now that the initial condition $X_0$ is sufficiently regular,
the fixed point problem \eqref{e:abs-fix-pt-X0'} can be solved for $\tilde{\mcX}\in\cD^{3/2+,0-}_{0-}$ and we can apply Lemma~\ref{lem:Schauder-input} once we check its condition, 
that is, 
$1_{t > 0} \xi^{\eps}$ and $\tilde{\CR}(\one_{+} \bXi)$ agree away from $t=0$ (which is obvious),
and the same holds for
$\tilde{\Psi}_{\eps} \partial \tilde{\Psi}_{\eps}$ and  $\tilde{\CR}\big( \tilde{\bPsi}   \partial  \tilde{\bPsi}  \big)$.
The fact that $\CR \mcX$ solves \eqref{e:SPDE-for-X} then follows by combining \cite{BCCH21} and \cite[Sec.~5.8]{CCHS2d} along with  Proposition \ref{prop:mass_term}. 

Note that we are slightly outside of the setting of \cite{BCCH21} because we have replaced the standard  integration operators $\CG_{\mfz}$ and $\mcb{K}_{\mfz}$ with non-standard ones with ``inputs''. However the results of \cite{BCCH21} still hold because $\CX$ is still coherent\footnote{See \cite[Def.~5.64]{CCHS2d}.} with respect to the nonlinearity \eqref{e:YMH-nonlinearity}. Coherence is a completely local algebraic property and for each $(t,x)$ with $t > 0$, $\CX(t,x)$ solves an {\it algebraic} fixed point problem of the form 
\begin{equ}\label{eq:coherence_identity}
 \mcX(t,x)
=
\mcb{I}_\mfz \Big(
\mcX(t,x) \partial \mcX(t,x) + \mcX(t,x)^3
 + \mathring{C}_{\mfz} \mcX(t,x)+ \boldsymbol{\Xi} \Big) + (\cdots)
 \end{equ}
where $(\cdots)$ takes values in the polynomial sector of the regularity structure. The relation above is all that is needed to deduce that $\CX$ is coherent at $(t,x)$. 

\begin{remark}
Note that there is a degree of freedom in the ``$t=0$ renormalisation'' of \eqref{e:abs-fix-pt-X0'} that could be exploited: one could add to the $ \tilde{\Psi}_{\eps} \partial \tilde{\Psi}_{\eps}$ appearing in the superscript any fixed distribution supported on $t=0$. This does not affect at all the coherence of $\CX$ with the nonlinearity, but in fact changes our initial condition. The fact that we do not add such a distribution in the superscript means that $\CR \CX$ really does solve \eqref{e:SPDE-for-X} with the prescribed initial data. 
\end{remark}

Combining this with the probabilistic convergence of the BPHZ models and the distributions $1_{t > 0} \xi^{\eps}$ and $\tilde{\Psi}_{\eps} \partial \tilde{\Psi}_{\eps}$ in the appropriate spaces as $\eps \downarrow 0$, one also gets stability of the solution in this limit. 
However, the analysis above depends on having fairly regular initial data which is {\it not} sufficient for our purposes. In order to use this dynamic to construct our Markov process we will need to start the dynamic from an arbitrary  $X_0 \in \state$.
 
As in our analysis of the deterministic equation \eqref{eq:ymh_flow_deturck} with rough initial data, the $t=0$ behaviour of the term $X \partial X$ requires us to linearise about $\CP X_0$ and take advantage of the control over $\CN(X_0)$ given by the metric on $\state$. 
We thus introduce the decomposition
\begin{equ}[eq:sym_solution_expansion]
\mcX=   \CP X_0 +  \tilde{\bPsi} + \hat{\mcX}\;,
\end{equ}
where $ \tilde{\bPsi}$ is as in 
\eqref{e:abs-fix-pt-X0'}, and consider
 the fixed point problem 
\begin{equs}[e:fix-pt-X1]
\hat \CX
&=
 \CG_{\mfz}
\Big( 
\CP X_0 \partial \CP X_0 
+
\CP X_0\, \partial \hat\mcX+  \hat\mcX \, \partial \CP X_0 +\hat\mcX \partial \hat\mcX
\Big) 
+ 
R  (1_{t > 0} \xi^{\eps}) \\
{}&
\quad
+
\CG_{\mfz} \Big(
 \tbPsi \partial \hat\mcX +  \hat\mcX \partial \tbPsi + \CX^3
 + \mathring{C}_{\mfz} \CX \Big)\\
{}&
\quad + \CG_{\mfz}^{\tilde\Psi_\eps \partial \tilde\Psi_\eps} 
(\tbPsi\partial \tbPsi)
+ \CG_{\mfz}^{\CP X_0 \partial \tilde\Psi_\eps} (\CP X_0 \partial \tbPsi)
+ \CG_{\mfz}^{ \tilde\Psi_\eps \partial \CP X_0} ( \tbPsi \partial \CP X_0 )\;.
\end{equs}
Instead of using the spaces $\cD^{\gamma,\eta}_{\alpha}$  of \cite{Hairer14}, it will be convenient
to use a slightly smaller class of ``$\hat{\cD}$ spaces'' with $\hat{\cD}^{\gamma,\eta}_{\alpha} \subsetneq \cD^{\gamma,\eta}_{\alpha}$
which impose a vanishing condition near $t=0$.
These spaces were introduced in~\cite{MateBoundary} and used in~\cite{CCHS2d} for the SYM in dimension $d=2$.
We collect their important properties in Appendix~\ref{app:Singular modelled distributions}.

Imposing $X_{0} \in \state $ will give us control over the first term on the right-hand side of \eqref{e:fix-pt-X1}. 
We will see that
 the other products in the first and second lines  of \eqref{e:fix-pt-X1} will belong to $\hat{\cD}$ spaces
  with good enough exponents
for Theorem~\ref{thm:integration} (for non-anticipative kernels) to apply,
 thanks to Lemma~\ref{lem:multiply-hatD} which gives
more refined power-counting for multiplication in $\hat{\cD}$ spaces
than that  of for instance \cite[Prop.~6.12]{Hairer14} for the usual $\cD$ spaces.
While we choose to use $\hat{\cD}$ spaces here as a matter of convenience, it will serve as a warm-up for Section~\ref{subsec:analytic_fp} where it is crucial. 

%
%

Finally the products of modelled distributions in the last line of \eqref{e:fix-pt-X1} give us  non-integrable singularities, similarly as in our discussion for \eqref{e:abs-fix-pt-X0'},  but we can again appeal to the integration result Lemma~\ref{lem:Schauder-input}
instead of the standard result \cite[Prop.~6.16]{Hairer14}.
Note that the distributions $\CP X_{0}$ and $\tilde \Psi_{\eps}$ are explicit objects,
 so we can again show  by hand that $\CP X_0 \partial \tilde \Psi_{\eps}$ and  $\tilde\Psi_\eps \partial \CP X_0$ converge probabilistically as $\eps \downarrow 0$ to some well-defined distributions over the entire space-time. 
We can argue exactly as for \eqref{e:abs-fix-pt-X0'} that $\CR \CX$ solves \eqref{e:SPDE-for-X} for every fixed $\eps > 0$ whenever $\CX$ solves \eqref{eq:sym_solution_expansion}--\eqref{e:fix-pt-X1}.


Below we first prove the necessary probabilistic convergences mentioned above,
and then
close the analytic fixed point problem \eqref{eq:sym_solution_expansion}--\eqref{e:fix-pt-X1}.
This will complete the proof of Theorem~\ref{thm:local-exist-sigma}. 

%

\subsubsection{Probabilistic estimates for solution theory}

We start with the probabilistic convergence of models. 
Let $Z^{\eps}_{\BPHZ}$ be the BPHZ models
determined by the kernel and noise assignments as in the beginning of this section.

\begin{lemma}\label{lem:conv_of_models}
The random models $Z^{\eps}_{\BPHZ}$ converge in probability to a limiting random model $Z_{\BPHZ}$  as $\eps \downarrow 0$.
\end{lemma}
\begin{proof}
We take a choice  of scalar noise decomposition
and  check the criteria of \cite[Theorem~2.15]{CH16}
which are
 insensitive to this choice. 
 First, it is clear that
 for any scalar noise decomposition,  the random smooth noise assignments here are a uniformly compatible family of Gaussian noises that converge to the Gaussian white noise. 
We then note that 
\[
\min \{ \deg(\tau): \tau \in \mfT(R), |N(\tau)| > 1\}
=
-2 - 2\kappa > -5/2 = -|\s|/2
\]
and the minimum is achieved for $\tau$ of the form $\<IXiI'Xi_notriangle>$.
Here $N(\tau)$ is the set of vertices of $\tau$
 such that $v \neq  e_{-} $ 
 for any $e$ with $\mft(e)\in \Lab_-$,
so the third criterion is satisfied.
Combining this with the fact that $\deg(\mfl) = -5/2 - \kappa$ for every $\mfl \in \Lab_{-}$ guarantees that the second criterion is satisfied. 
Finally, the worst case scenario for the first condition is for $\tau$ of the form $\<IXi^3_notriangle>$ and $A = \{a\}$ with   $\mft(a) = \mfl$ for  $\mfl \in \Lab_{-}$ for which we have  
$\deg(\tau) + \deg(\mfl) + |\s| = 1 - 4\kappa > 0$
as required.
\end{proof}

We now give the promised statement about the probabilistic definition of products involving the initial data and the solution to SHE.
We skip proving the convergence of $1_{t > 0} \xi^{\eps}$ since this follows by a straightforward application of Kolmogorov's argument combined with second moment computations. 
In the next lemmas, we recall that $G$ is the Green's function of the heat operator
and we let $(\Omega^{\noise},\P)$ denote the probability space on which $\xi$ is defined.

\begin{lemma}\label{lem:tildePsi-cov}
For each $\kappa>0$ there exists $\bar\kappa>0$ such that, for all $T>0$ and $p\in[1,\infty)$, $G*(\tilde\Psi_\eps \partial \tilde\Psi_\eps) $
converges in $L^p(\Omega^\noise;\CC^{\bar\kappa}([0,T],\CC^{-\kappa}(\T^3)))$ to a limit denoted by 
$G*(\tilde\Psi \partial \tilde\Psi)$.
In particular,  $\tilde\Psi_\eps \partial \tilde\Psi_\eps$
converges in $L^p(\Omega^\noise;\CC^{-2-\kappa}(\KK))$ to a limit denoted by 
$\tilde\Psi \partial \tilde\Psi$ for any compact $\KK\subset\R\times\T^3$.
\end{lemma} 

\begin{proof}
By equivalence of Gaussian moments, it suffices to consider $p=2$.
Dropping reference to $\eps$ and denoting ${\mathcal Y} = G*(\tilde\Psi \partial \tilde\Psi) $, one has
\begin{equ}\label{eq:CY_bound}
 {\mathcal{Y}}_t -  {\mathcal{Y}}_s 
 = ( \CP_{t-s}-1) {\mathcal{Y}}_s 
 + \int_s^t \CP_{t-r} ( \tilde\Psi_r\partial \tilde\Psi_r) \,\mrd r \;.
\end{equ}
Using \cite[Lem.~10.14]{Hairer14}, followed by Lemma~\ref{lem:CPf} (with $\gamma=2$ and $\alpha=2-2\kappa$ therein), one has
\begin{equs}
\E & \Big| \Big\langle  \int_s^t \CP_{t-r} ( \tilde\Psi_r\partial \tilde\Psi_r) \,\mrd r  ,\phi^\lambda \Big\rangle \Big|^2
\\
&\lesssim 
 \int_{r,\bar r\in [s,t]} \!\!\!\!\!\!\!\!\!\! G_{t-r} (x-y) G_{t-\bar r} (\bar x-\bar y)
 |(r,y)-(\bar r,\bar y)|^{-4}
| \phi^\lambda (x) \phi^\lambda (\bar x)| \,\mrd r \mrd \bar r  \mrd y \mrd \bar y\mrd x\mrd \bar x 
 \\
&\lesssim
\int_{r\in [s,t]} \!\!\!\!\!\!\!
 G_{t-r}(x-y) |y-\bar x|^{-2} | \phi^\lambda (x) \phi^\lambda (\bar x) | \,  \mrd r \mrd y\mrd x\mrd \bar x
\\
&\lesssim 
\int_{r\in [s,t]} \!\!\!\!\!\!\! (t-r)^{-\frac{2-2\kappa}{2}} |x-\bar x|^{-2\kappa} |\phi^\lambda (x) \phi^\lambda (\bar x)| \,  \mrd r \mrd x\mrd \bar x
\lesssim |t-s|^{\kappa} \lambda^{-2\kappa}
\end{equs}
 for $\kappa>0$ arbitrarily small,
where the first bound follows from Wick's theorem, the fact that $\E (\tilde{\Psi}\partial \tilde{\Psi})=0$ thanks to the derivative,
and the uniform in $\eps>0$ bound
\begin{equ}[e:tildePsi-cov]
|\E (\partial^k\tilde{\Psi}(z) \, \partial^{\bar k}\tilde{\Psi}(\bar z))| \lesssim
|z-\bar z|^{-1-|k|-|\bar k|} \qquad  (|k|,|\bar k|\in \{0,1\})\;.
\end{equ}
Here  $\phi^\lambda$ is the rescaled spatial test function as in Lemma~\ref{lem:prob-obs}. 
Setting $s=0$ 
one then has
$\E  | \langle {\mathcal{Y}}_t ,\phi^\lambda \rangle|^2
\lesssim t^\kappa \lambda^{-2\kappa}  $.
Therefore $\E |( \CP_{t-s}-1) {\mathcal{Y}}_s |^2_{\CC^{-\kappa}} \lesssim s^{\kappa/4}|t-s|^{\kappa/4}$, which handles the first term on the right-hand side of \eqref{eq:CY_bound}.
Since $\CY_r=0$ for $r\leq 0$,
it follows from a Kolmogorov argument that
$\E |{\mathcal{Y}}|_{\CC^{\bar\kappa}([0,T], \CC^{-\kappa}(\T^3))}^2$ is bounded uniformly in $\eps>0$ for any $T>0$ and $\kappa>0$.

To show that $\mcY$ converges in $\CC^{\bar\kappa}([0,T], \CC^{-\kappa}(\T^3))$ as $\eps\downarrow0$, it suffices to extract a small power of $\bar\eps$ in the corresponding bounds on $\CP_{t-r} (\tilde\Psi_\eps\partial \tilde\Psi_\eps) -\CP_{t-r} ( \tilde\Psi_{\bar\eps}\partial \tilde\Psi_{\bar\eps})$ for $0<\eps<\bar\eps$.

The final claim follows by applying to $\mcY$ the heat operator, which is a bounded operator from $\CC([-T,T],\CC^{-\kappa}(\T^3))$ to $\CC^{-2-\kappa}((-T,T)\times\T^3)$
and remarking that $\mcY$ is continuous over $[-T,T]$ once we extend it by $\mcY_t=0$ for $t<0$.
\end{proof}

\begin{lemma}\label{lem:PX0Psi-prob}
For every $\kappa>0$ and $\eta\in (-1,-\frac12)$,
there exists $\bar\kappa>0$ such that, for all $T>0$, $p\in[1,\infty)$, uniformly in $0<\eps<\bar\eps<1$,
\begin{equs}
\E|G*(\CP X_0  \partial \tilde\Psi_\eps - \CP X_0  \partial \tilde\Psi_{\bar\eps})|_{\CC^{\bar\kappa}([0,T],\CC^{\eta+\frac12-\kappa})}^p
&\lesssim
\bar\eps^{\bar\kappa}|X_0|_{\CC^\eta}^p
\\
\E|G*( \tilde\Psi_\eps \partial \CP X_0 -  \tilde\Psi_{\bar\eps} \partial \CP X_0)|_{\CC^{\bar\kappa}([0,T],\CC^{\eta+\frac12-\kappa})}^p
&\lesssim
\bar\eps^{\bar\kappa}|X_0|_{\CC^\eta}^p\;.
\end{equs}
In particular,
$G*(\CP X_0  \partial \tilde\Psi_\eps)$
and $G*(\tilde\Psi_\eps \partial \CP X_0)$
converge in $\CC^{\bar\kappa}([0,T],\CC^{\eta+\frac12-\kappa})$ in $L^p(\Omega^{\noise},\P)$ to limits denoted respectively by 
$G*(\CP X_0  \partial \tilde\Psi)$ 
and $G*(\tilde\Psi \partial \CP X_0)$,
and
$\CP X_0  \partial \tilde\Psi_\eps$
and $\tilde\Psi_\eps \partial \CP X_0$
converge in $\CC^{\eta-\frac32-\kappa}(\KK)$ in $L^p(\Omega^{\noise},\P)$ to limits denoted respectively by 
$\CP X_0  \partial \tilde\Psi$ 
and $\tilde\Psi \partial \CP X_0$
for any compact $\KK\subset\R\times\T^3$.
The maps sending $X_0$ to any of these limits is a bounded linear map from
$\CC^\eta$ to the corresponding $L^p(\Omega^\noise,\P)$ space.
\end{lemma}

\begin{proof}
We apply the same trick as \eqref{eq:CY_bound}.
Dropping again reference to $\eps$,
one has
\begin{equs}[e:PX0dPsi-21]
\E  \Big| \Big\langle  \int_s^t \CP_{t-r} ( \CP_r X_0 & \partial \tilde\Psi_r) \,  \mrd r  ,\phi^\lambda \Big\rangle \Big|^2
\lesssim 
 |X_0|_\eta^2 \int_{r,\bar r\in [s,t]} \!\!\!\!\!\!\!\! G_{t-r} (x-y) G_{t-\bar r} (\bar x-\bar y)
 \\
 &\times r^{\eta/2} {\bar r}^{\eta/2}  |(r,y)-(\bar r,\bar y)|^{-3}
 |\phi^\lambda (x) \phi^\lambda (\bar x)| \,\mrd r \mrd y \mrd \bar r \mrd \bar y\mrd x\mrd \bar x \;.
\end{equs}
Since $r,\bar r$ are symmetric, we just consider the regime $r\le \bar r$,
so that ${\bar r}^{\eta/2} \le  r^{\eta/2}$.
By \cite[Lem.~10.14]{Hairer14} followed by Lemma~\ref{lem:CPf} (with $\gamma=1$ and $\alpha=2\eta + 2 -2 \kappa\in(0,1)$ therein for $\kappa$ sufficiently small), the above quantity is bounded by a multiple of
\begin{equs}
{}& |X_0|_\eta^2  \int_{r\in [s,t]} \!\!\!\!\!\!\!\!
G_{t-r} (x-y) r^{\eta}  
|\bar x-y|^{-1}
 |\phi^\lambda (x) \phi^\lambda (\bar x) |\,\mrd r \mrd y \mrd x\mrd \bar x 
\\ 
&\lesssim
|X_0|_\eta^2
\int_{r\in [s,t]} \!\!\!\!\!\!\!\! (t-r)^{-\frac{\alpha}{2}}  r^\eta |x-\bar x|^{\alpha-1}
 |\phi^\lambda (x) \phi^\lambda (\bar x)| \,\mrd r \mrd x\mrd \bar x 
 \lesssim 
 |t-s|^{\kappa}   \lambda^{2\eta+1-2\kappa}.
\end{equs}
 Setting $s=0$ and denoting ${\mathcal Y} = G*(\CP X_0 \partial \tilde\Psi) $,
one then has
$\E  | \langle {\mathcal{Y}}_t ,\phi^\lambda \rangle|^2
\lesssim  t^\kappa\lambda^{2\eta+1-2\kappa}  $.
The claim for $G*(\CP X_0 \partial \tilde\Psi)$ and $\CP X_0 \partial \tilde\Psi$ now follows in the same way as in the proof of Lemma~\ref{lem:tildePsi-cov}.

The argument for $\tilde\Psi \partial \CP X_0$ is similar except we have
\begin{equs}[e:PX0dPsi-22]
\E  \Big| \Big\langle  \int_s^t \CP_{t-r} ( &\tilde\Psi_r \partial \CP_r X_0 ) \,  \mrd r  ,\phi^\lambda \Big\rangle \Big|^2
\lesssim 
 |X_0|_\eta^2 \int_{r,\bar r\in [s,t]} \!\!\!\!\!\!\!\! G_{t-r} (x-y) G_{t-\bar r} (\bar x-\bar y)
 \\
 &\times r^{\frac\eta2-\frac12} {\bar r}^{\frac\eta2-\frac12}   |(r,y)-(\bar r,\bar y)|^{-1}
 |\phi^\lambda (x) \phi^\lambda (\bar x)| \,\mrd r \mrd y \mrd \bar r \mrd \bar y\mrd x\mrd \bar x
\end{equs}
and instead of \cite[Lem.~10.14]{Hairer14}  we use  Lemma~\ref{lem:CPf} twice (first with $\gamma=1$ and $\alpha=1+\eta \in (0,\frac12)$ and then with $\gamma=-\eta$ and $\alpha=\eta+1-2\kappa\in (0,-\eta)$ therein for $\kappa$ sufficiently small)
to bound the above quantity by a multiple of
\begin{equs}
{}& |X_0|_\eta^2 \int_{r,\bar r\in [s,t]} \!\!\!\!\!\!\!\!\!\!
G_{t-r} (x-y)  r^{\frac\eta2-\frac12}  \bar r^{\frac\eta2-\frac12}  (t-\bar r)^{-\frac12-\frac\eta2} |\bar x-y|^{\eta}
 |\phi^\lambda (x) \phi^\lambda (\bar x)| \,\mrd r \mrd \bar r \mrd y \mrd x\mrd \bar x
\\ 
&\lesssim |X_0|_\eta^2  \int_{r\in [s,t]} \!\!\!\!\!\!\!\!
G_{t-r} (x-y)  r^{\frac\eta2-\frac12}
|\bar x-y|^{\eta}
| \phi^\lambda (x) \phi^\lambda (\bar x) |\,\mrd r \mrd y \mrd x\mrd \bar x 
\\ 
&\lesssim
|X_0|_\eta^2
\int_{r\in [s,t]} \!\!\!\!\!\!\!\! (t-r)^{-\frac{\eta}{2}-\frac12+\kappa}  r^{\frac\eta2-\frac12}|x-\bar x|^{2\eta+1-2\kappa}
 |\phi^\lambda (x) \phi^\lambda (\bar x)| \,\mrd r \mrd x\mrd \bar x 
\\
&\lesssim 
 |t-s|^{\kappa}   \lambda^{2\eta+1-2\kappa}\;.
\end{equs}
The rest of the proof is again the same as that of Lemma~\ref{lem:tildePsi-cov}.
\end{proof}

\subsection{Proof of Theorem~\ref{thm:local-exist-sigma}}\label{sec:proof_of_localexist_sym}


\begin{proof}[of Theorem~\ref{thm:local-exist-sigma}]
We first show that the fixed point problem \eqref{e:fix-pt-X1} is well-posed, and then argue that the reconstructed solutions converge in $\state^\sol$.  

Fix $\gamma\in(\frac32+2\kappa,2)$. 
Writing $\CX$ as \eref{eq:sym_solution_expansion}, we will solve the fixed point problem \eqref{eq:sym_solution_expansion} and \eqref{e:fix-pt-X1}  for $\hat \CX$ in $\hat{\cD}^{\gamma,\hat\beta}_{-5\kappa}$ 
with $\hat\beta > -1/2$ as in  \eqref{eq:CI}.  
Note that
\begin{equ}[e:PX0D]
\CP X_0 \in  \hat{\cD}^{\infty,\eta}_{0} \;,
\quad
\partial \CP X_0 \in  \hat{\cD}^{\infty,\eta-1}_{0}\;.
\end{equ}
Moreover, since  $X_0\in\init$,
we use the bound \eqref{e:PN-bound}
 with $|\cdot|_{\CC^{\hat\beta}}$ replaced by  $|\cdot|_{\CC^k}$ (for $|k|<\gamma$)  
 to conclude\footnote{Note that the bound  \eqref{e:PN-bound} references classical H\"{o}lder-Besov spaces but the space of modelled distributions taking values in the polynomial sector coincides with these classical spaces.} that
$\CG_{\mfz} \big( \CP X_0 \partial \CP X_0 \big)$
is well-defined as an element of $\hat\cD^{\gamma,\hat\beta}_0$.

For the other terms in the first line on the right-hand side of \eqref{e:fix-pt-X1}, 
since $\partial \hat \CX \in
 \hat{\cD}^{\gamma-1,\hat\beta-1}_{-1-5\kappa}$,
by \eqref{e:PX0D} and  
 Lemma~\ref{lem:multiply-hatD} one has
\begin{equ}[e:PX0-hatX-add]
\CP X_0\, \partial \hat\mcX
 \in
 \hat{\cD}^{\gamma-1,\hat\beta+\eta-1}_{-1 -5\kappa}\;,
\quad
 \hat\mcX \, \partial \CP X_0  
 \in
 \hat{\cD}^{\gamma,\hat\beta+\eta-1}_{ -5\kappa}\;,
\quad
  \hat\mcX \partial \hat\mcX 
 \in
 \hat{\cD}^{\gamma-1-4\kappa,2\hat\beta-1}_{-1 -10\kappa}\;,
\end{equ}
where we took $\kappa>0$ sufficiently small such that $-5\kappa>\hat\beta$.
Note that $\hat\beta>-1/2$ due to \eqref{eq:CI}, and thus
$\min(\hat\beta+\eta-1,2\hat\beta-1)
> -2$.
Therefore, by Theorem~\ref{thm:integration}, 
\[
\CG_{\mfz} \Big( \CP X_0\, \partial \hat\mcX+  \hat\mcX \, \partial \CP X_0 +\hat\mcX \partial \hat\mcX \Big) \in \hat{\cD}^{\gamma,\hat\beta}_{-2\kappa}\;,\]
provided that $\kappa>0$ is sufficiently small.

For the cubic and linear term on the second line of \eqref{e:fix-pt-X1}, 
by the definition of  $\tbPsi$  in \eqref{e:abs-fix-pt-X0'} and
Lemma~\ref{lem:Schauder-input}, we have $\tbPsi \in \cD^{\gamma,-\f12-2\kappa}_{-\f12-2\kappa}$ and thus $\tbPsi \in \hat\cD^{\gamma,-\f12-2\kappa}_{-\f12-2\kappa}$.
Then by the fact that $\CP X_0  , \hat \CX$ belong to the $\hat\cD$-spaces with exponents stated above,
and taking $\kappa>0$ sufficiently small such that
$\eta<-\frac12-2\kappa$,
one has
\begin{equ}[e:tildeX3-add]
\CX \in \hat\cD^{\gamma,\eta}_{-\frac12-2\kappa}\;,
\qquad
\mbox{and thus}
\qquad
\CX^3  \in \hat\cD^{\gamma-1-2\kappa,3\eta}_{-\frac32-6\kappa}\;.
\end{equ}
By \eqref{eq:CI},
$3\eta>-2$ so that
Theorem~\ref{thm:integration} applies,
and we have 
$\CG_{\mfz} \big(  \CX^3 + \mathring{C}_{\mfz} \CX \big)
\in \hat{\cD}^{\gamma,\hat\beta}_{-2\kappa}$.  
Moreover, one has
\begin{equ}[e:F-hatX-add]
\tbPsi \partial \hat\mcX , \hat\mcX \partial \tbPsi
\in
\hat{\cD}^{\gamma -\frac32-2\kappa,\hat\beta-\frac32-2\kappa}_{-\frac32-7\kappa}\;.
\end{equ}
The condition  \eqref{eq:CI}
again guarantees that 
Theorem~\ref{thm:integration} applies to these terms provided that $\kappa>0$
 is small enough.
This concludes our analysis for the second line of \eqref{e:fix-pt-X1}.

The terms in the third line of  \eqref{e:fix-pt-X1} require extra care. Indeed, applying Lemma~\ref{lem:multiply-hatD} for multiplication as before
we have 
\begin{equ}[e:FpartialF]
\tbPsi\partial \tbPsi \in 
\hat \cD^{\gamma-\frac32-2\kappa, -2-4\kappa}_{-2-4\kappa}\;,
\end{equ}
and
\begin{equ}[e:PX0dF]
\CP X_0 \partial \tbPsi
 \in \hat{\cD}^{\gamma-1, \eta-\frac32-2\kappa}_{-\frac32-2\kappa}\;,
\qquad
 \tbPsi \partial \CP X_0 
  \in \hat{\cD}^{\gamma, \eta-\frac32-2\kappa}_{-\frac12-2\kappa}\;,
\end{equ}
but then the reconstruction
Theorem~\ref{thm:reconstructDomain} and thus Theorem~\ref{thm:integration}
do not  apply because $ -2-4\kappa < -2$
and $ \eta-\frac32-\kappa < -2$.
To control these terms in \eqref{e:fix-pt-X1} we use instead Lemma~\ref{lem:Schauder-input}, 
for which we need to check that $\omega$ is compatible with $f$ for every term of the form $\CG^\omega(f)$.

For the last two terms  this is immediate since for $|k| \in \{0,1\}$ and $t \not = 0$,  $(\tilde{\CR} \partial^{k}\tbPsi)(t,x) = \partial^{k} \tilde{\Psi}_{\eps}(t,x)$, and since renormalisation commutes with multiplication by polynomials, for any modeled distribution $H$ taking values in the span of the polynomial sector, one has $\tilde{\CR} (H \partial^{k}\tbPsi)(t,x) = (\tilde{\CR}H)(t,x) (\tilde{\CR}\partial^{k}\tbPsi)(t,x)$. 

For the term $\CG^{\tilde\Psi_\eps \partial \tilde\Psi_\eps} 
(\tbPsi\partial \tbPsi)$, since the action of the model on $ \<IXiI'Xi_notriangle>$ is unaffected by BPHZ renormalisation (see Lemma~\ref{lem:symbols_vanish}), one has
\[
\tilde{\CR}(\tbPsi\partial \tbPsi)(t,x) = \tilde{\CR}(\tbPsi)(t,x)\tilde{\CR}(\partial \tbPsi)(t,x) = \tilde{\Psi}_{\eps} \partial \tilde{\Psi}_{\eps}
\]
for $t \neq 0$.
This verifies the conditions of Lemma~\ref{lem:Schauder-input} which
together with \eqref{e:FpartialF}--\eqref{e:PX0dF} shows
 that the fixed point problem \eqref{e:fix-pt-X1} is well-posed for $\eps > 0$. 

The stability of the fixed point problem (in modelled distribution space)
as $\eps \downarrow 0$
for a short (random) time interval $[0,\tau]$
 then follows from the convergence of models (Lemma~\ref{lem:conv_of_models}), Lemmas~\ref{lem:tildePsi-cov}--\ref{lem:PX0Psi-prob} and \eqref{e:conti-in-zeta}.
Here $\tau>0$ depends only on the size of model in the time interval $[-1,2]$ and the size of the initial condition in $\state$.

We write $X^{\eps} = \CR \CX$ in the rest of the argument, where $\CX$ is given by \eqref{eq:sym_solution_expansion} and $\hat{\CX}$ is the solution to the fixed point problem \eqref{e:fix-pt-X1} for the model $Z^{\eps}_{\BPHZ}$
over the interval $[0,\tau]$.

We now show that $X^{\eps}$ converges as $\eps \downarrow 0$ in $\CC([0,\tau],\state)$.
To this end,
let $\Psi^\eps$ (resp.\ $\Psi$) solve
the stochastic heat equation driven by $\xi^\eps$ (resp.\ $\xi$)
 with initial condition $(a,\phi) \in\state$, and
 let us decompose $X^\eps = \Psi^\eps + \hat{X}^\eps$.
 By the above construction, $\hat{X}^\eps = \CR \hat{\CX}$
 where $\CR$ is the reconstruction map for  $Z^{\eps}_{\BPHZ}$.
 By convergence of models given by Lemma~\ref{lem:conv_of_models}, 
and  continuity of the reconstruction map $\mcR$,
$\hat{X}^\eps$ converges in probability to a limit denoted by $\hat{X}$ in $\CC^{-\kappa}((0,\tau)\times\T^3)$.

We claim further that $\hat{X}^\eps$ converges to $\hat X$ 
in $\CC([0,\tau],\CC^{\eta+\frac12-\kappa}(\T^3))$.
Indeed, one has $\hat{X} = \mathcal{Y} + \hat{X}_R$ where
$ \mathcal{Y}  \eqdef G*(\Psi\partial\Psi)$ 
with $G$ the heat kernel and $\Psi$ as above, and $\hat{X}_R \in \CC([0,\tau],\CC^{\f12 - \kappa})$.
Then writing $\Psi=\CP X_0+\tilde\Psi$ with $X_0=(a,\phi)$
and $\tilde\Psi$ the solution to SHE with $0$ initial condition,
we can split $\mathcal Y$ into four terms. The term quadratic in $\CP X_0$
can be bounded as in the proof of Proposition~\ref{prop:YM_flow_minus_heat},
while the term quadratic in $\tilde\Psi$ and the cross terms between $\CP X_0$ and $\tilde\Psi$
converge in $\CC([0,T], \CC^{\eta+\frac12-\kappa})$ due to Lemmas~\ref{lem:tildePsi-cov} and \ref{lem:PX0Psi-prob} respectively.
In conclusion, $\mcY^\eps \to \mcY$ in probability (even in $L^p$ for any $p\in[0,\infty)$) in $\CC([0,T], \CC^{\eta+\frac12-\kappa}(\T^3))$,
and therefore $\hat X^\eps \to \hat X$ in probability in $\CC([0,\tau],\CC^{\eta+\frac12-\kappa})$
for any $\kappa>0$ as claimed.

To show that $X^{\eps}$ converges as $\eps \downarrow 0$ in $\CC([0,\tau],\state)$,
note that,
by assumption~\eqref{eq:bar_beta_cond2} (and the condition for $\beta$ therein),
we can choose $\kappa>0$ small enough 
such that $\eta +(\eta+\frac12-\kappa )  > 1-2\delta$  (which follows from \eqref{eq:bar_beta_cond2}-\eqref{eq:bar_beta_cond3}),
 $\eta+\frac12-\kappa>\frac12-\delta$ (see above \eqref{eq:bar_beta_cond1}), 
 and $\eta+\frac12-\kappa>2\theta(\alpha-1)$ (see below \eqref{eq:bar_beta_cond3}).
Then, by items~\ref{pt:perturbation} and~\ref{pt:fancynorm_0} of Lemma~\ref{lem:perturbation} respectively, pointwise in $[0,\tau]$,
\begin{equs}\label{eq:perturbation_X}
{} &
\fancynorm{X^\eps; X}_{\beta,\delta}
\lesssim
\fancynorm{\Psi^\eps;\Psi}_{\beta,\delta} 
+ \fancynorm{\hat{X}^\eps; \hat{X}}_{\beta,\delta}
\\
&\;
+ |\Psi^\eps-\Psi|_{\CC^\eta}
\Big(|\hat{X}^\eps|_{\CC^{\eta+\frac12-\kappa}}\!\! +|\hat{X}|_{\CC^{\eta+\frac12-\kappa}}\Big)
+|\hat{X}^\eps- \hat{X}|_{\CC^{\eta+\frac12-\kappa}}
\Big(| \Psi^\eps|_{\CC^\eta}+| \Psi |_{\CC^\eta}\Big)
\end{equs}
and
\begin{equ}
 \fancynorm{\hat{X}^\eps; \hat{X}}_{\beta,\delta}
  \lesssim |\hat{X}^\eps -\hat{X}|_{\CC^{\eta+\frac12-\kappa}}
  \Big(|\hat{X}^\eps|_{\CC^{\eta+\frac12-\kappa}}+|\hat{X}|_{\CC^{\eta+\frac12-\kappa}}\Big)\;,
\end{equ}
and by
Lemma~\ref{lem:heatgr_Besov_embed},
\begin{equ}
 \heatgr{\hat{X}^\eps- \hat{X}}_{\alpha,\theta}
  \lesssim |\hat{X}^\eps -\hat{X}|_{\CC^{\eta+\frac12-\kappa}}\;.
\end{equ}
Furthermore, by Corollary~\ref{cor:SHE_conv_state}, for all $T>0$
\begin{equ}
\lim_{\eps\to0}\sup_{t\in[0,T]} \big(\heatgr{\Psi^\eps_t-\Psi_t}_{\alpha,\theta} + \fancynorm{\Psi^\eps_t;\Psi_t}_{\beta,\delta} +|\Psi^\eps_t-\Psi_t|_{\CC^\eta}\big)=0
\end{equ} 
in probability. It follows that there exists a random variable $C_\eps>0$ (with moments of all orders bounded uniformly in $\eps>0$) such that, for all $R>1$ and $t\in[0,\tau]$,
on the event $|\hat X_t|_{\CC^{-\kappa}}+|\hat X^\eps_t|_{\CC^{-\kappa}} < R$, it holds that
\begin{itemize}
\item $\Sigma(X_t)+\Sigma(X^\eps_t) < C_\eps R^2$, and
\item $|\hat X_t-\hat X^\eps_t|_{\CC^{\eta+\frac12-\kappa}} < \eps \Rightarrow \Sigma(X_t,X^\eps_t) < c_\eps R + C_\eps\eps$,
\end{itemize}
where $c_\eps$ is another random variable (with finite moments of all orders) 
such that $c_\eps \to 0$ in probability as $\eps\to 0$.
Using that $\hat X^\eps\to \hat X$ in probability in $\CC([0,\tau],\CC^{\eta+\frac12-\kappa})$,
and using Proposition~\ref{prop:Theta_cont_zero}\ref{pt:heat_Sigma} to handle continuity at time $t=0$ for the stochastic heat equation,
it follows that $X^\eps\to X$ in probability in $\CC([0,\tau],\state)$ as claimed.

Finally, we note that $X^{\eps}$ indeed solves equation \eqref{e:SPDE-for-X} on the time interval $[0,\tau]$, since for smooth initial data the equation \eqref{e:fix-pt-X1} reconstructs to the same equation that \eqref{e:abs-fix-pt-X0'} does. 

After the initial time interval $[0,\tau]$,
we can restart the equation ``close to stationarity'' by solving for the remainder.
More specifically, let $V$ denote the solution
in the space of modelled distributions 
for the remainder equation arising from the ``generalised Da Prato--Debussche trick'' in~\cite{BCCH21}
and associated to the model $Z^\eps_{\BPHZ}$ (we allow $\eps=0$).
Recall that this equation removes the stationary ``distributional'' objects, which in our case are $\<IXi>_\eps$
and $\<I[IXiI'Xi]_notriangle>_\eps$,
and solves for the remainder in the space of modelled distributions $\mathscr{U}^{\gamma,\eta}_{+}$ specified in~\cite[Sec.~5.5, Eq.~5.16]{BCCH21}. 

We start the equation from time $\tau$ and with initial condition $v^\eps(\tau) \eqdef \hat X_R^\eps(\tau) + f^\eps(\tau)$
where $\hat X^\eps_R\in\CC([0,\tau],\CC^{\frac12-\kappa})$ is defined as above
and $f^\eps$ is defined by
\begin{equ}\label{eq:f_DPD_trick}
f^\eps(t) = \tilde\Psi_\eps(t) - \<IXi>_\eps(t) + K*(\tilde\Psi_\eps\partial\tilde\Psi_\eps)(t) - \<I[IXiI'Xi]_notriangle>_\eps(t)\;,
\end{equ}
where the symbols represent the stationary objects
\begin{equ}
\<IXi>_\eps \eqdef K*\xi^\eps\;,\qquad \<I[IXiI'Xi]_notriangle>_\eps \eqdef K*(\<IXi>_\eps \partial \<IXi>_\eps)\;.
\end{equ}
Note that $\tilde\Psi_\eps - \<IXi>_\eps$ converges in $\CC^\infty((0,\infty)\times\T^3)$
and
therefore $K*(\tilde\Psi_\eps\partial\tilde\Psi_\eps) - \<I[IXiI'Xi]_notriangle>_\eps$ converges in $\CC^{\frac12-\kappa}((0,\infty)\times\T^3)$ in probability as $\eps\downarrow0$.
Hence $f^\eps$ converges in probability in $\CC^{\frac12-\kappa}((0,\infty)\times\T^3) $
as $\eps\downarrow 0$.

Since the initial condition of $V$ is H\"older continuous with exponent $\f12-\kappa$,
we obtain a maximal solution for $V$
for which $\CR^\eps V$ converges in $\CC^{\frac12-\kappa}(\T^3)^\sol$ to $\CR^0 V$
(we use here that $\CR^\eps V$ is continuous with respect to the initial data $v^\eps(\tau)$,
and that $v^\eps(\tau)\to v^0(\tau)$).
Furthermore, $f$ is chosen in such a way that
\begin{equ}
\CR^\eps V(\tau) + \<IXi>_\eps(\tau) + \<I[IXiI'Xi]_notriangle>_\eps(\tau) = X^\eps(\tau)\;.
\end{equ}

Finally, for $\eps>0$, it follows from coherence (specifically from~\cite[Thm.~5.7]{BCCH21}) that
$\CR^\eps V + \<IXi>_\eps+\<I[IXiI'Xi]_notriangle>_\eps$ solves the equation~\eqref{e:SPDE-for-X} with initial condition $X^\eps(\tau)$ on $[\tau,T^*_\eps)$, where $T^*_\eps$ is the maximal existence time of $\CR^\eps V$.
By the same argument as above concerning $\hat X^\eps$ and $\Psi^\eps$
we further see that $\CR^\eps V + \<IXi>_\eps+\<I[IXiI'Xi]_notriangle>_\eps$ converges as $\eps\downarrow0$ in probability in $\state^\sol$ to $\CR^0 V + \<IXi>_0+\<I[IXiI'Xi]_notriangle>_0$.
This completes the proof of part~\ref{pt:A_conv_state}.

Part~\ref{pt:depend_on_lim} simply follows from the stability of 
the fixed point problem in space of modelled distributions with respect to coefficients in the fixed point problem.
\end{proof}

\section{The gauge transformed system}\label{sec:gauge}

We now formulate the first part of Theorem~\ref{theo:meta} precisely.
Throughout the section, let $\moll$ be a space-time mollifier.
Define the mapping 
\begin{equ}\label{eq:hatmfG^rho_def}
\mfG^{\rho} \ni g \mapsto (U,h) \in  \tilde{\mfG}^{\rho} \eqdef \CC^\rho \big(\T^3,L_{G}(\mfg,\mfg) \oplus L_{G}(\higgsvec,\higgsvec) \big) \times \CC^{\rho - 1} \big(\T^3,\mfg^3)\;,
\end{equ}  \label{hatmfG^rho page ref}
by setting    
\begin{equ}\label{eq:h_and_U_def}
h \eqdef (\mrd g )g^{-1} 
\qquad\textnormal{and}\qquad
U=(U_{\mfg},U_{\higgsvec}) 
\eqdef
(\Ad_g,\brho(g))\;.
\end{equ}
Observe that~\eqref{eq:hatmfG^rho_def} maps $\mfG^{0,\rho}$ into $\tilde{\mfG}^{0,\rho}$, the closure of smooth functions in $\tilde{\mfG}^\rho$.\label{hat_mfG^0rho page ref}

By Lemma~\ref{lem:flow_for_rough_g},
for any $(x,g_0) \in \state \times \mfG^{0,\rho}$, if we write $(X^{\eps},g^{\eps})$ for the solution to \eqref{e:SPDE-for-X_with_g}+\eqref{eq:SPDE_for_g_coupled}
 starting with initial data $(x,g_0)$ and with $c^{\eps} = 0$ and $\mathring{C}_{\A}^{\eps}$, $\mathring{C}_{\Phi}^{\eps}$ as in the statement of the lemma, then we have that $(X^{\eps},g^{\eps})$ converges in probability, as $\eps \downarrow 0$, in $(\state \times \mfG^{0,\rho})^{\sol}$ to a limit $(X,g)$. We write $
  \mathcal{A}^{\BPHZ}[\mathring{C}_{\A},x,g_0] \eqdef (X,g)$. 
We analogously write $\mathcal{\tilde{A}}^{\BPHZ}[\mathring{C}_{\A},x,g_0]$ for the analogous construction, obtained via Lemma~\ref{lem:flow_for_rough_g2} to take the limit of the solutions to  \eqref{e:SPDE-for-X_with_g2}+\eqref{eq:SPDE_for_g_coupled2}
instead of  \eqref{e:SPDE-for-X_with_g}+\eqref{eq:SPDE_for_g_coupled}. 

\begin{theorem}\label{thm:gauge_covar}
Suppose $\moll$ is 
non-anticipative and let $C_{\YM}^{\eps},C_{\Higgs}^{\eps}$ be as in~\eqref{e:C-YM-Higgs-sig} with $\sig=1$.
Fix $\mathring{C}_{\A}\in L_{G}(\mfg,\mfg)$,
$ (a,\phi) \in \state$,
and $g(0)\in\mfG^{0,\rho}$.
Let $C^{\eps}_{\A} = C^{\eps}_{\YM} + \mathring{C}_{\A}$
and
$C_{\Phi}^{\eps} =  C_{\Higgs}^{\eps} -m^2 $,
and let $(B,\Psi,g)$ be the solution to~\eqref{eq:SPDE_for_B}.
Furthermore, for $\check{C} \in L_{G}(\mfg,\mfg)$,
let $(\bar{A},\bar{\Phi},\bar{g})$ be the solution to   
\begin{equs}[eq:SPDE_for_bar_A]
\partial_t \bar{A}_i &= 
\Delta \bar{A}_i  +[\bar{A}_j,2\partial_j \bar{A}_i - \partial_i \bar{A}_j + [\bar{A}_j,\bar{A}_i]] -\mathbf{B}((\partial_{i} \bar{\Phi} + \bar{A}_{i}\bar{\Phi}) \otimes \bar{\Phi})
\\
& \qquad\quad + C^{\eps}_{\A} \bar{A}_i
+ (\mathring{C}_{\A} - \check{C})(\partial_i \bar{g} )\bar{g}^{-1}
+ \moll^\eps * (\bar{g}\xi_i \bar{g}^{-1}) \;,
\\
\partial_t \bar{\Phi} &= 
\Delta \bar{\Phi} 
+ 2 \bar{A}_{j} \partial_{j}\bar{\Phi} 
+ \bar{A}_{j}^{2}\bar{\Phi} 
 - |\bar{\Phi}|^2 \bar{\Phi} + C_{\Phi}^{\eps} \bar{\Phi}
+\moll^\eps * (\bar{g}\zeta)\;,
\\
(\partial_t \bar{g}) & \bar{g}^{-1} 
= \partial_j((\partial_j \bar{g})\bar{g}^{-1})+ [\bar{A}_j, (\partial_j \bar{g})\bar{g}^{-1}]\;,
\\
& \qquad (\bar{A}(0),\bar{\Phi}(0))  = g(0)\act (a, \phi) \;,
\quad \bar{g}(0) = g(0)\;,
\end{equs}
where we set $\bar g\equiv 1$ on $(-\infty,0)$.

\begin{enumerate}[label=(\roman*)]
\item\label{pt:delta_to_zero} For every $\eps>0$, there exists a smooth maximal solution to~\eqref{eq:SPDE_for_bar_A} in $(\state\times \mfG^{0,\rho})$
obtained by replacing $\xi = \big( (\xi_{i})_{i=1}^{3}, \zeta \big)$ by $\tilde{\xi}^{\delta} \eqdef \xi \bone_{t<0}+\bone_{t\geq0}\xi^\delta$ and taking the $\delta\downarrow0$ limit.

\item \label{pt:gauge_covar} 
There exists a unique operator $\check{C}\in L_{G}(\mfg,\mfg)$
such that for all choices of $(B(0),\Psi(0), g(0)) = (\bar{A}(0),\bar{\Phi}(0), \bar{g}(0)) \in \state \times \mfG^{0,\rho}$, $(B,\Psi,U,h)$
and $(\bar A,\bar\Phi,\bar U,\bar h)$
converge in probability to the same limit in $(\state\times\tilde{\mfG}^{0,\rho})^\sol$ as $\eps\downarrow 0$.
Here $\bar U,\bar h$ and $U,h$ are determined by $\bar g$ and $g$ respectively via~\eqref{eq:h_and_U_def}.
 
\item \label{pt:indep_moll} The operator $\check{C}$ above is independent of $\mathring C_\A$ and our choice of non-anticipative mollifier $\chi$. 

\item \label{pt:final_gauge_covar} $\check{C}$ is the unique value of $\mathring{C}_{\A}$ for which the following property holds: for any $(x,g_{0}) \in \state \times \mfG^{0,\rho}$, if one writes $(\tilde{X},\tilde{g}) =   \tilde{\mathcal{A}}^{\BPHZ}[\mathring{C}_{\A},g_0 \act x,g_0]$ and $(X, g) =  \mathcal{A}^{\BPHZ}[\mathring{C}_{\A}, x, g_0]$, then $(g \act X, U,h) \eqlaw (\tilde{X},\tilde{U},\tilde{h}) $
in $(\state\times \tilde{\mfG}^{0,\rho})^{\sol}$, where $\tilde U,\tilde h$ and $U,h$ are determined by $\tilde g$ and $g$ respectively via~\eqref{eq:h_and_U_def}. 
\end{enumerate} 
\end{theorem}

The proof of Theorem~\ref{thm:gauge_covar} will be given at the end of this section on page~\pageref{proof:gauge_covar}.
If $\moll$ is non-anticipative,
then $\lim_{\delta \downarrow 0} \bar U \tilde{\xi}^{\delta} $ is
 equal in law to $\xi$
by It{\^o} isometry since $\bar{U}$ is adapted and is orthogonal on $E$. 
Therefore, when we choose
 $\mathring{C}_{\A}=\check{C}$, 
the law of $(\bar{A},\bar{\Phi})$ \dash which does not depend on $\bar g$ anymore \dash is equal  to the 
solution to \eqref{eq:SPDE_for_A}
with initial condition $g(0)\act (a,\phi)$, and
Theorem~\ref{thm:gauge_covar}
implies the desired property of gauge covariance described in Theorem~\ref{theo:meta}\ref{pt:gauge_covar} \dash this property is illustrated in Figure~\ref{fig:gauge_covar}.
The desired Markov process on gauge orbits in Section~\ref{sec:Markov} will therefore be constructed from the $\eps\downarrow0$
limit $X=(A,\Phi)$ of the solutions to~\eqref{eq:SPDE_for_A} where the constants are defined by~\eqref{e:SYM_constants} with  $\mathring C_\A=\check C$ and $\sig^\eps=1$.
Remark that, by Theorems~\ref{thm:local-exist-sigma}\ref{pt:depend_on_lim}
and~\ref{thm:gauge_covar}\ref{pt:indep_moll},
this $X$ is canonical in that it does not depend on $\moll$.
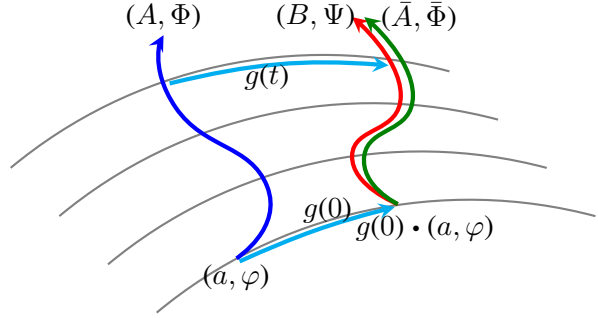
\begin{figure}[h]
\centering
  \begin{tikzpicture}[scale=0.8]
  
    \draw[thick,gray] (0,2.4) arc (75:130:8cm);
    \draw[thick,gray] (.8,1.6) arc (75:130:8cm);
    \draw[thick,gray] (1.6,0.8) arc (75:130:8cm);
    \draw[thick,gray] (2.4,0) arc (75:130:8cm);
    
 \draw [blue,ultra thick,->] plot [smooth, tension=1] coordinates 
 { (-3.5,-0.7) (-3, .5) (-4.5 , 1.8) (-4.7,  3)};   
\node at (-3.5,-1) {$(a,\phi)$}; 
\node at (-4.7,  3.3) {$(A,\Phi)$}; 
     
 \draw [red,ultra thick,->] plot [smooth, tension=1] coordinates 
 { (-0.85,0.2) (-1.6,1) (-0.8,2) (-1.6 , 3.3)};
\node at (-0.4, -.2) {$g(0)\act (a,\phi)$}; 
\node at (-2.2,  3.3) {$(B,\Psi)$}; 
 
 \draw [darkgreen,ultra thick,->] plot [smooth, tension=1] coordinates 
 { (-0.85, 0.2) (-1.4,1) (-0.6,2) (-1.4 , 3.3)};
\node at (-0.5,  3.3) {$(\bar A,\bar\Phi)$}; 

\draw [cyan, ultra thick,-> ,bend left =5]  (-3.45,-0.75)  to (-0.9,0.15);
 \node at (-2,0.1) {$g(0)$}; 
 
\draw [cyan, ultra thick,-> ,bend left =9]  (-4.6,  2.2)  to (-1, 2.5);
 \node at (-3,2.3) {$g(t)$};  

\end{tikzpicture}

\caption{An illustration of the $(A,\Phi)$ system (blue) \eqref{eq:SPDE_for_A},
the $(B,\Psi)$ system (red) \eqref{eq:SPDE_for_B}, $g$ (cyan) given by~\eqref{eq:SPDE_for_B} (or equivalently~\eqref{eq:SPDE_for_g_wrt_A}),
and the $(\bar A,\bar\Phi)$ system (green)~\eqref{eq:SPDE_for_bar_A}.
The black curves stand for gauge orbits.
With the choice $\mathring{C}_{\A}=\check{C}$, $(A,\Phi)$ and $(\bar A,\bar \Phi)$ have the same law modulo initial condition.
As $\eps\downarrow 0$, the red and green curves on the right converge to the same limit.}\label{fig:gauge_covar}
\end{figure}

\begin{remark}\label{rem:non-anticipative-assump}
The non-anticipative assumption on the mollifier $\moll$ is used in various places in the proof 
of Theorem~\ref{thm:gauge_covar}\ref{pt:gauge_covar}.
We use it in Proposition~\ref{prop:SPDEs_conv_zero_state}
to relate~\eqref{eq:SPDE_for_bar_A} in law 
to an SPDE with additive noise, for which the short-time analysis is simpler.
Furthermore, under this assumption, $\check{C}$ defined by the limit in  \eqref{e:def-C-bar} 
is an $\eps$-independent finite constant; this relies on 
Proposition~\ref{prop:consts-soft2} which requires a non-anticipative
mollifier. 
\end{remark}

Before continuing we describe some of the important differences between the proof of Theorem~\ref{thm:gauge_covar} and the analogous theorem in $d=2$, namely \cite[Thm.~2.9]{CCHS2d}. 
For this comparison, the fact that \cite{CCHS2d} does not involve a Higgs field is an immaterial difference. 

Firstly, the short-time analysis of the systems~\eqref{eq:SPDE_for_B} and~\eqref{eq:SPDE_for_bar_A}
is significantly more involved for $d=3$ than $d=2$.
We therefore require control on more stochastic objects (in addition to the models) which arise
from ill-defined products between terms involving the noise and the initial condition;
some of these objects require non-trivial renormalisation not present in $d=2$ 
(see~\eqref{e:omega123} and~\eqref{e:bar_omega0123}).

We furthermore require a more involved decomposition of the abstract fixed point problem due to these singularities than for $d=2$.
This is particularly the case in the analysis of the maximal solution to the system~\eqref{eq:SPDE_for_bar_A} performed in Proposition~\ref{prop:bar_X_max_sols}, where
we require a new fixed point problem which solves for a suitable `remainder'.
While a similar but simpler change in the fixed point problem was also required in~\cite[Sec.~7.2.4]{CCHS2d},
the main difference here is that we need to leverage
the fact that the fixed point for the `remainder' has an improved initial condition.
(We already met similar considerations in the proof of Theorem~\ref{thm:local-exist-sigma},
where we used two different strategies for the time intervals $[0,\tau]$ and $[\tau,\infty)$
and were able to fall back on the results of~\cite{BCCH21} for the latter due to the additive nature of the noise --
in the proof of Proposition~\ref{prop:bar_X_max_sols}, 
the results of~\cite{BCCH21} are not available because the noise is multiplicative.)

Next, note that the parameters $\bar{C}$ appearing in \cite[Thm~2.9]{CCHS2d} and  $\check{C}$ in Theorem~\ref{thm:gauge_covar} are related as follows:  in $d=2$, one has the convergence $\lim_{\eps \downarrow 0} C_{\YM}^{\eps} = C_{\YM}^{0}$ where $C_{\YM}^{\eps}$ is the BPHZ constant appearing in Theorem~\ref{thm:local_exist}.\footnote{This is not true for $C_{\Higgs}^{\eps}$ in $d=2$ but  again, this makes no difference as the Higgs doesn't play an important role in this argument.} 
The parameter $\bar{C}$ appearing in \cite{CCHS2d} would then correspond to $\check{C} + C_{\YM}^{0}$.

The key difference in our arguments is the use of explicit formulae for renormalisation constants in $d=2$ vs. small noise limits and injectivity arguments in $d=3$ \dash without either of these the best one would be able to do in both $d=2,3$ would be to obtain Theorem~\ref{thm:gauge_covar}\ref{pt:gauge_covar} with the modification that  $\check{C}$ would have to be replaced by, say, $\check{C}_{\eps}$ which in principle would be allowed to diverge as $\eps \downarrow 0$.

In $d=2$, $\bar{C}_{\eps}$ can be computed explicitly in terms of a small number of trees, and it can be shown that one has the convergence $\lim_{\eps \downarrow 0} \bar{C}_{\eps} = \bar{C}$. 
This gives \cite[Thm~2.9(i)]{CCHS2d}. 
Moreover  one can directly verify that, for a non-anticipative mollifier, $\bar{C} - C_{\YM}^{0}$ is independent of the mollifier. 
This would prove Theorem~\ref{thm:gauge_covar}\ref{pt:indep_moll}  in $d=2$ since in the limit, BPHZ renormalisation removes all the dependence on the mollifier. 

In $d=3$, $\check{C}_{\eps}$ has contributions from dozens of trees and explicit computation is not feasible. 
As mentioned earlier, by using small noise limits we can argue that $\check{C}_{\eps}$ remains bounded as $\eps \downarrow 0$ (Lemmas~\ref{lem:consts-soft1_bounded} and~\ref{lem:consts-soft2_bounded}). 
Injectivity of our solution theory also provides some rigidity and prevents $\check{C}_{\eps}$ from having distinct subsequential limits (Propositions~\ref{prop:consts-soft} and~\ref{prop:consts-soft2}). 
This gives the convergence $\lim_{\eps \downarrow 0}  \check{C}_{\eps} = \check{C}$ which establishes  Theorem~\ref{thm:gauge_covar}\ref{pt:gauge_covar}. 
These same propositions also show that the limiting $\check{C}$ does not depend on our particular choice of non-anticipative mollifier which gives part of Theorem~\ref{thm:gauge_covar}\ref{pt:indep_moll}\footnote{That $\check{C}$ doesn't depend on $\mathring{C}_{\A}$ can be argued with power-counting and parity, see Remark~\ref{rem:bare-m-renorm_2}.}. 

Note that while $\check{C}$ is independent of the mollifier in both $d=2,3$ we do not claim that it is a universal\footnote{Such as the $1/24$ that appears in KPZ in $d=1$.} constant \dash our constant $\check{C}$, as well as the constants $C_{\YM}^{\eps}$, $C^{\eps}_{\Higgs}$, is fixed only once 
one has fixed a BPHZ lift of space-time white noise, but the latter is not canonical since we must prescribe a 
large-scale truncation of the heat kernel to perform this lift.

\subsection{Setting up regularity structures}
\label{subsec:reg_struct_guage_transform}

As in \cite[Sec.~7]{CCHS2d} we will replace the evolutions of gauge transformations in \eqref{eq:SPDE_for_B} and \eqref{eq:SPDE_for_bar_A} with evolutions in linear spaces \dash to that end, recall the mapping~\eqref{eq:hatmfG^rho_def}.
We next state a lemma about writing the dynamics \eqref{eq:SPDE_for_B} and  \eqref{eq:SPDE_for_bar_A} in terms of the variables \eqref{eq:h_and_U_def}.  

Below it will be convenient to overload notation and sometimes write $\brho$
 for the direct sum of the adjoint representation of $G$ on $\mfg$ and the representation of $G$ on $\higgsvec$. Note that we also write $\brho$ for its derivative representation (except in the proof of Lemma~\ref{lemma:linear_g_system} where 
 we will use $\bbrho$ instead since the distinction is conceptually more important there), but the meaning of $\brho$ will be clear from the context. 
  
%
\begin{lemma}\label{lemma:linear_g_system}
Consider any smooth $B\colon(0,T]\to\Omega\mcC^\infty$ and suppose $g$ solves the third equation in~\eqref{eq:SPDE_for_B}.
Then $h$ and $U$ defined by \eqref{eq:h_and_U_def} satisfy  
\begin{equs}[e:final_hU]
(\p_t - \Delta) h_{i} &= 
 - [h_j,\p_j h_i] + [[B_j, h_j],h_i] + \p_i [B_j, h_j]\;,
\\
(\p_{t} - \Delta)U &=
 - \brho(h_j)^{2} U+ \brho([B_j, h_j]) U \;.
\end{equs}
\end{lemma}

\begin{proof}
The derivation 
for the equation of $h$ 
is exactly the same as in  \cite[Lem.~7.2]{CCHS2d}  since it does not rely on any specific representation, and
in particular the third equation in~\eqref{eq:SPDE_for_B} for $g$ can be rewritten as
\begin{equ}[e:ptgg-1]
(\p_t g)g^{-1} = \p_j h_j + [B_j , h_j]\;.
 \end{equ}
Given \textit{any} Lie group representation $\brho : G \to GL(V)$ for a vector space $V$ and its derivative representation $\bbrho : \mfg \to L(V,V)$, \eqref{e:ptgg-1} implies that the function
$U \colon \T^{3} \rightarrow GL(V)$
defined by $U=\brho(g)$ satisfies \eqref{e:final_hU}.

Indeed 
if $\bar{h}_i =(\p_i g) g^{-1} $ for $i\in \{0\}\cup [d]$,
then 
$\bbrho (\bar{h}_i)  = \p_i (\brho(g))  \brho(g)^{-1} $ and so in particular
$\p_i U = \bbrho(\bar{h}_i) U$. This identity has two consequences:
\begin{equs}
\Delta U &= \bbrho(\p_i h_i) U +\bbrho (h_i)^2  U \;,
\\
\p_t U &= \bbrho(\p_t g \,g^{-1} ) U 
\stackrel{\eqref{e:ptgg-1}}{=} \bbrho( \p_j h_j + [B_j , h_j])  U\;.
\end{equs}
These two identities together  then yield the desired equation.
\end{proof}

Define $U$ and $h$ as in \eqref{eq:h_and_U_def} using $g$ from \eqref{eq:SPDE_for_B}, 
and $\bar{U}$ and $\bar{h}$ as in \eqref{eq:h_and_U_def}  using $\bar{g}$ from \eqref{eq:SPDE_for_bar_A}. 
We can then rewrite the  ``unrenormalised'' equations of~\eqref{eq:SPDE_for_B} and~\eqref{eq:SPDE_for_bar_A} respectively as   
\begin{equs}\label{eq: B system}
(\p_t  - \Delta) B_i &=   [B_j, 2\p_j B_i - \p_i B_j + [B_j,B_i]]
\\
& \qquad - \mathbf{B}((\partial_{i} \Psi + B_{i}\Psi)\otimes \Psi) 
+ \mathring{C}_{\A} B_i + \mathring{C}_{\h} h_{i}+ U_{\mfg} (\moll^{\eps} *\xi_i)\;,
\\ 
(\partial_t - \Delta) \Psi  &=  
2 B_{j} \partial_{j}\Psi + B_{j}^{2}\Psi -  m^2 \Psi - |\Psi|^2 \Psi + U_{\higgsvec} (\moll^{\eps} *\zeta) \;,
\end{equs} 
combined with \eqref{e:final_hU} for the evolution of $(U,h)$ and
\begin{equs}\label{eq: bar a system}
(\p_t - \Delta) \bar{A}_i &=   [\bar{A}_j, 2\p_j \bar{A}_i - \p_i \bar{A}_j + [\bar{A}_j,\bar{A}_i]]
\\
& \qquad - \mathbf{B}((\partial_{i} \bar{\Phi} + \bar{A}_{i}\bar{\Phi})\otimes \bar{\Phi}) 
+ \mathring{C}_{\A} \bar{A}_i + \mathring{C}_{\h} \bar{h}_{i}+ \moll^{\eps} *(\bar{U}_{\mfg} \xi_i),\\
(\partial_t - \Delta) \bar{\Phi}  &=  
 2 \bar{A}_{j} \partial_{j}\bar{\Phi} + \bar{A}_{j}^{2}\bar{\Phi} -  m^2 \bar{\Phi} - |\bar{\Phi}|^2 \bar{\Phi} + \moll^{\eps}  *( \bar{U}_{\higgsvec} \zeta)\;,
\end{equs}
again combined with \eqref{e:final_hU}. Note that 
in the nonlinear terms such as $B_{j}^{2} \Psi$ we are implicitly referencing the derivative representation of $\mfg$ on $\higgsvec$.

To set up the regularity structure, as in the previous section, 
we  combine the components of the connection and the Higgs field into a single variable:   $\bar{X} = ((\bar{A}_i)_{i=1}^{3},\bar{\Phi})$
and $Y = ((B_i)_{i=1}^{3},\Psi)$. Recall the notation $\xi \eqdef ((\xi_i)_{i=1}^{3},\zeta)$.
Our approach is  to work with {\it one single} regularity structure to simultaneously study the systems  \eqref{e:final_hU}+\eqref{eq: B system} for $(Y,U,h)$ 
and \eqref{e:final_hU}+\eqref{eq: bar a system}  for $(\bar{X}, \bar U, \bar h)$, allowing us to compare their solutions at the abstract level of modelled distributions. In fact, we will set it 
up in such a way that in the $\eps \downarrow 0$ limit, the two
solution maps associated to each of these systems converge to the same limit as modelled distributions.

Let $\kappa>0$.
We introduce the label sets 
$\mfL_+ \eqdef \{\mfz, \mfm, \mfh, \mfh', \mfu\}  $ and
$\mfL_- \eqdef \{\mfl, \bar{\mfl} \}$ 
and set their degrees to
\[
\deg(\mft) \eqdef 
\begin{cases}
2 - \kappa
& \mft \in \{\mfz, \mfm\}\;,\\
2
& \mft \in \{\mfh, \mfu\}\;,\\
1
& \mft =\mfh'\;,\\
-5/2 - \kappa
&\mft \in \{\mfl, \bar{\mfl}\}\;.
\end{cases}
\] 
%
We briefly explain the roles of the types introduced above. 
Regarding noises, $\bar{\mfl}$ represents the noise 
$\moll^{\eps} \ast \xi$   
in the $(B,\Psi)$ equation \eqref{eq: B system}, while  $\mfl$ represents the noise $\xi$  
in the $(\bar{A},\bar{\Phi})$ equation \eqref{eq: bar a system}.\footnote{The ``bar'' in the notation $\bar{\mfl}$ indicates the ``mollification at scale $\eps$'', and has nothing to do with the ``bars'' in $(\bar{A},\bar{\Phi})$.} 

To understand the roles of the types $\mfz$ and $\mfm$, observe that when writing \eqref{eq: bar a system} as an integral equation we are in an analogous situation to the one described in \cite[Sec.~7.2]{CCHS2d}, namely the term 
$\bar{U} \xi = ( (\bar{U}_{\mfg} \xi_i)_{i=1}^3, \bar{U}_{\higgsvec} \zeta)$ 
appearing on the right-hand side is convolved with a mollified heat kernel $G \ast \moll^{\eps}$ while the remaining terms are convolved with an un-mollified heat kernel. 
This requires us to use two types to track the RHS of  \eqref{eq: bar a system}: we use $\mfm$ for the term 
$\bar{U} \xi$   
and $\mfz$ for everything else.
As in \cite[Sec.~7.2]{CCHS2d}, we also use $\mfz$ to keep track of the entire RHS of \eqref{eq: B system}. 

One difference in our use of types versus what was done in \cite[Sec.~7]{CCHS2d} is that here the RHS of the equation for $h$ in \eqref{e:final_hU} is also split into two pieces using the types $\mfh$ and $\mfh'$. 
Since we are working in dimension $d=3$, power counting \emph{after} an application of Leibniz 
rule would suggest that the terms $\partial_{i} [B_{j},h_{j}]$ could generate 
a number of counterterms, but by keeping them written as total derivatives 
we can easily verify that no counterterms are generated. 
To implement this we use $\mfh'$ to keep track of the terms $[B_{j},h_{j}]$ with the understanding that in the integral equation for $h$ coming from \eqref{e:final_hU}, the terms associated to $\mfh'$ should be integrated with the gradient of the heat kernel, while  $\mfh$ tracks the
 rest of the RHS of the $h$ equation in \eqref{e:final_hU} while $\mfu$ tracks the $U$ equation.

We will specify the corresponding rule $R$ later, but it will be subcritical with respect to the map $\reg: \Lab \rightarrow \R$ given by
\[
\reg(\mft)
\eqdef
\begin{cases}
-1/2 -  4\kappa
& \mft \in \{\mfz, \mfm\}\;,\\
1/2-5\kappa
& \mft \in  \{\mfh, \mfh'\}\;,\\
3/2-5\kappa
& \mft = \mfu \;,\\
-5/2  -2\kappa
&\mft \in \{\mfl, \bar{\mfl}\}\;,
\end{cases}
\]
provided that $\kappa < \frac{1}{100}$ (more stringent requirements on $\kappa$ will be imposed for other purposes later).
Our target space assignment  $(W_{\mft})_{\mft \in \mfL}$ is given by
\begin{equ}[e:system-target space assignment]
W_{\mft} \eqdef 
\begin{cases}
E
	& \mft \in \{ \mfz, \mfm, \mfl, \bar{\mfl} \}\;,
\\
\mfg^3 
	& \mft \in \{\mfh, \mfh'\}\;,
\\
L(\mfg,\mfg) \oplus L(\higgsvec,\higgsvec)
	&\mft = \mfu\;,
\end{cases} 
\end{equ} 
and our kernel space assignment $(\CK_{\mft})_{\mft \in \mfL_{+}}$ is given by 
\begin{equ}[e:system-kernel space assignment]
\CK_{\mft}
\eqdef
\begin{cases}
\R^{d}& \mft = \mfh' \;,\\
\R & \text{otherwise.}
\end{cases}
\end{equ}
The assignment $(V_{\mft})_{\mft \in \mfL}$ used to build our regularity structure is then 
given by \eqref{eq:space_assignment}. 

We now specify our rule\footnote{As in \cite{CCHS2d}, the choice to include $\{\mfu \mfl\} \in \mathring{R}(\mfz)$ isn't directly motivated by the systems associated to~\eqref{eq: B system} and~\eqref{eq: bar a system} but is needed to compare these two systems, see e.g.\ \eqref{e:GUhatXi-compare}.}
\begin{equs}[e:rule-gauged]
\mathring{R}(\mfl) &= \mathring{R}(\bar{\mfl}) = \{\bone\}\;,\quad 
\mathring{R}(\mfm) = \big\{\mfu \mfl \big\}\;,
\\
\mathring{R}(\mfu) &= \{\mfu\mfr^2,\; \mfu\mfq \mfr\,:\, \mfq \in \{\mfz,\mfm\},\, \mfr \in \{\mfh,\mfh'\} \}\;,
\\
\mathring{R}(\mfh) &= \{\mfr \p_j \mfr,\; \mfq \mfr^2 \,:\, \mfq \in \{\mfz,\mfm\},\, \mfr \in \{\mfh,\mfh'\},\, j \in [d]\}\;,
\\
\mathring{R}(\mfh') &= \{\mfq \mfr\,:\, \mfq \in \{\mfz,\mfm\},\,  \mfr \in \{\mfh,\mfh'\} \}\;,
\\
\mathring{R}(\mfz) &= 
\{\mfq,\; \mfr,\; \mfq \hat{\mfq} \tilde{\mfq},\; \mfq \p_j \hat{\mfq},\;\mfu \mfl,\; \mfu \bar{\mfl}:\,  \mfq,\hat{\mfq},\tilde{\mfq} \in \{\mfz,\mfm\},\,  \mfr \in \{\mfh,\mfh'\},\, j \in [d]\}\;.
\end{equs}
It is straightforward to verify that $\mathring{R}$ is subcritical with respect to $\deg$ (for $\reg$ defined as above), and
has a smallest normal extension which admits a completion $R$ and 
that is also subcritical. We use $R$ to define our regularity structure
and write $\mfT$ for the corresponding set of trees conforming to $R$. 
As in Section~\ref{sec:renorm-A} we write 
$\Xi=\mcb{I}_{(\mfl,0)}(\bone)$ and
$\bar\Xi=\mcb{I}_{(\bar\mfl,0)}(\bone)$,
and later we will also use the notations
$\bXi \in \mcb{T}[\Xi] \otimes E$
and $\bar\bXi \in \mcb{T}[\bar\Xi] \otimes E$
for the corresponding $E$-valued modelled distributions.

The kernel assignment  $K^{(\eps)} = (K_{\mft}^{(\eps)}: \mft \in \Lab_{+})$
is given by  \label{Keps-assign}
\begin{equ}[e:K-Keps-assign]
K_\mft^{(\eps)} = 
\left\{\begin{array}{cl}
	K & \quad \mft \in \{ \mfz, \mfh, \mfu\}\;, \\
	K^\eps  &  \quad \mft =\mfm\;, \\
	\nabla K &  \quad \mft =\mfh'\;,
\end{array}\right.
\end{equ}
where $K^\eps = K \ast \moll^{\eps}$. 

Our noise assignment $\zeta^{\delta,\eps} = (\zeta_{\mfl}: \mfl \in \mfL_{-})$ is given by
\begin{equ}[eq:noise_assign_gauge]
\zeta_{\mfl} = \sig^{\eps} \moll^{\delta} \ast \xi = \sig^{\eps} \xi^{\delta} \;,
\qquad
\zeta_{\bar{\mfl}} = \sig^{\eps} \moll^{\eps} \ast \xi =\sig^{\eps} \xi^{\eps} \;.
\end{equ}
Namely, $\bar{\mfl}$ indexes the noise $\xi^\eps= \sig^{\eps} \moll^\eps * \xi$ in 
\eqref{eq: B system},
while $\mfl$ indexes the white noise  $\sig^{\eps} \xi$ in \eqref{eq: bar a system} \dash we have replaced $\xi$ with $\xi^{\delta}$ above since it is convenient to work with smooth models, we will later take, for $\eps >0$, the benign limit $\delta \downarrow 0$ to obtain \eqref{eq: bar a system} . 
The label $\mfm$ is used to track the term $\bar U \xi$ in \eqref{eq: bar a system}
  \dash these terms are treated separately because they are hit by a mollified heat kernel, which is why we defined $K_\mfm^{(\eps)} = K^\eps$ above.

We now fix two nonlinearities $F = \bigoplus_{\mft \in \Lab} F_{\mft},\ \bar{F} = \bigoplus_{\mft \in \Lab} \bar{F}_{\mft} $ which encode the systems $(Y,U,h)$
and $(\bar X,\bar U, \bar h)$ separately.
For $\mft\in \Lab_-$ we set $F_\mft = \bar F_\mft = \id_E$ (recalling that $\mcb{T}[\Xi] \otimes E\simeq L(E,E)$), and 
set, similarly as in \eqref{e:pr-X-dX},
 \begin{equs}[2][e:purple-jets]
 \pr{Y} &\eqdef  \pr{\mathbf{A}_{(\mfz,0)}}, \;& \quad
 \pr{\partial_j Y} &\eqdef  \pr{\mathbf{A}_{(\mfz,e_j)}}, \;
 \\
 \pr{h}  &\eqdef  \pr{\mathbf{A}_{(\mfh,0)}} +\pr{\mathbf{A}_{(\mfh',0)}} ,\;& \quad
 \pr{\p_jh}  &\eqdef \pr{\mathbf{A}_{(\mfh,e_j)}} +\pr{\mathbf{A}_{(\mfh',e_j)}},\;\\
\pr{U}  &\eqdef\pr{\mathbf{A}_{(\mfu,0)}}, \;& \quad
\pr{ \p_j U}  &\eqdef\pr{\mathbf{A}_{(\mfu,e_j)}}\;, 
\qquad \qquad \pr{\bar{\Xi}} \eqdef \pr{\mathbf{A}_{(\bar{\mfl},0)}} \;.
\end{equs}
We also write $\pr{U} = \pr{U_{\mfg}} \oplus \pr{U_{\higgsvec}} \in L(\mfg,\mfg) \oplus L(\higgsvec,\higgsvec)$ and similarly decompose $\pr{ \p_j U} = \pr{\p_j U_{\mfg}} \oplus \pr{\p_j U_{\higgsvec}}$. 

Recalling the convention \eqref{e:XdX-X3}, 
overloading  notation $\pr{h}=(\pr{h},0)$ and writing
\begin{equ}[e:ring-C-zh]
\mathring{C}_{\mfz} = 
\mathring{C}_{\A}^{\oplus 3} 
\oplus (-m^2) \;,\quad
\mathring{C}_{\mfh} = 
\mathring{C}_{\h}^{\oplus 3}
\oplus 0\;,
\end{equ}
we define
 \begin{equ}[e:F_mfz]
F_{\mfz} (\pr{\mathbf{A}})\eqdef 
\pr{Y \p Y} + \pr{Y^3}
 + \mathring{C}_{\mfz} \pr{Y} + \mathring{C}_{\mfh} \pr{h} 
+\pr{U}  \pr{\bar{\Xi}}\;.  
\end{equ}
 Moreover, writing 
 $\pr{h_i} \eqdef \pr{h}|_{\mfg_i}$,
$\pr{\p_jh_i} \eqdef \pr{\p_jh}|_{\mfg_i}$,
and $\pr{B_i} \eqdef \pr{Y}|_{\mfg_i}$, 
we define
 \begin{equs}[e:F_mfh]
F_{\mfh} (\pr{\mathbf{A}}) |_{\mfg_i} 
&\eqdef 
- [\pr{h_j}, \pr{\p_jh_i}] + [[\pr{B_j}, \pr{h_j}],\pr{h_i}] \;,
\\
F_{\mfh'} (\pr{\mathbf{A}}) |_{{\textnormal{\tiny span}} \{e_k\} \otimes \mfg_i}  &\eqdef  
1_{i=k}[\pr{B_j}, \pr{h_j}]\;,
\end{equs}
 where $(e_k)$ is the canonical basis of $V_{\mfh'}=\R^3$
 and the index $j$ is summed over $\{1,2,3\}$;
 and finally
 \begin{equ}[e:F_mfu]
F_{\mfu} (\pr{\mathbf{A}})\eqdef
- \brho(\pr{h_j})^2 \pr{U} +\brho( [\pr{B_j}, \pr{h_j}]) \pr{U} \;,
\qquad
F_{\mfm} (\pr{\mathbf{A}})= 0 \;.
\end{equ}

\begin{remark}
By \cite[Sec.~5.8]{CCHS2d} the nonlinearity
$F_{\mfh'} $ takes values in 
$V_{\mfh'} \otimes W_{\mfh'} \cong \R^3 \otimes \mfg^3 \cong \mfg^{3\times 3}$,
although it has a special diagonal\footnote{The framework of Section~\ref{sec:sym_renorm_reg_struct} allows for cases where the nonlinearity would have a non-diagonal form and the ``components mix'', for instance the Navier--Stokes equations with $\mbox{div}(u\otimes u)$
where $u\otimes u$ takes values in $\R^{3\times 3}$.} form with identical entries.
This is because $F_{\mfh'} $ here describes a gradient of a $\mfg$-valued function.
\end{remark}

Regarding $\bar F$, we set 
\[
\pr{\bar{X}} \eqdef  \pr{\mathbf{A}_{(\mfz,0)}}+ \pr{\mathbf{A}_{(\mfm,0)}}\;,
\qquad
  \pr{\partial_i \bar{X}}  \eqdef \pr{\mathbf{A}_{(\mfz,e_i)}}+ \pr{\mathbf{A}_{(\mfm,e_i)}} \;,
\qquad
 \pr{\Xi} \eqdef \pr{\mathbf{A}_{(\mfl,0)}} \;.
 \] 
and $\pr{\bar h}$, $\pr{\p_j \bar h}$,  $\pr{\bar U}$ in the same 
way as \eqref{e:purple-jets} above,
 \[
\bar{F}_{\mfz} (\pr{\mathbf{A}})\eqdef 
\pr{\bar{X} \p \bar{X}} + \pr{\bar{X}^3}
 + \mathring{C}_{\mfz} \pr{\bar{X}} + \mathring{C}_{\mfh} \pr{\bar{h}}  \;,
 \qquad
 \bar F_{\mfm}(\pr{\mathbf{A}}) \eqdef \pr{\bar{U}}    \pr{\Xi} \;,  
\]
 and define $\bar{F}_{\mfh} $, $\bar{F}_{\mfh'}$, $\bar{F}_{\mfu} $
in the same way as $F_{\mfh}$,  $F_{\mfh'}$, $F_{\mfu}$
except that $\pr{h}$, $\pr{ U}$, $\pr{B}$ are replaced by
$\pr{\bar h}$, $\pr{\bar U}$, $\pr{\bar A}$.

We write $\mathscr{M}$ for the space of all models and, for $\eps \in [0,1]$, we write $\mathscr{M}_{\eps} \subset \mathscr{M}$ 
 for the family of $K^{(\eps)}$-admissible models.\label{model page ref}
We also define $\ell_{\BPHZ}^{\delta,\eps} \in \mcb{R} \subset \mcb{F}^{\ast}$ to be the BPHZ character associated to the kernel assignment $K^{(\eps)}$ and the noise assignment $\zeta^{\delta,\eps}$.

\subsection{Renormalisation of the gauge transformed system}\label{sec:renorm_gauge_transformed}

The following proposition is one of the main results of this section and describes the form of the BPHZ counterterms that appear in \eqref{eq: B system}
and \eqref{eq: bar a system} {\it after} renormalisation.  
In order to write these counterterms in a clean form, we will want to impose a nonlinear constraint on $\pr{U}$ and $\pr{h}$ that is consistent with \eqref{eq:h_and_U_def}. 
Recalling that $\mcb{A} = \prod_{e \in \CE} W_{e}$, we define
\begin{definition}\label{def:bar-mcb-A}
We define $\bar{\mcb{A}} \subset \mcb{A}$ to be the collection of $\pr{\mathbf{A}} \in \mcb{A}$ such that $\pr{U} \in \Im(\Ad \oplus \brho)$ and $\brho(\pr{h_{i}}) = (\pr{\partial_{i}U_{\mfg}})\pr{U}_{\pr{\mfg}}^{-1}$ for each $i \in\{1,2,3\}$.%
\footnote{Here, when we write $\brho(\pr{h_{i}})$ the notation $\brho$ means the adjoint representation of $\pr{h_{i}}$ on $\mfg$.}
\end{definition}
\begin{proposition}\label{prop:renorm_g_eqn}
For $\mft \in \{\mfu,\mfh,\mfh'\}$ and any $\pr{\mathbf{A}} \in \mcb{A}$ one has
\begin{equ}\label{eq:no_u_h_renorm}
(\ell^{\delta,\eps}_{\BPHZ} \otimes \id_{W_{\mft}}) 
\bar{\bUpsilon}^{F}_{\mft}(\pr{\mathbf{A}})  =
(\ell^{\delta,\eps}_{\BPHZ} \otimes \id_{W_{\mft}}) 
\bar{\bUpsilon}^{\bar{F}}_{\mft}(\pr{\mathbf{A}}) = 0\;.
\end{equ}
Moreover, there exist
 operators $C^{\delta,\eps}_{\YM},C^{\eps}_{\Gauge},C^{\delta,\eps}_{\Gauge}  \in L_{G}(\mfg,\mfg)$  and  $C^{\delta,\eps}_{\Higgs} \in L_{G}(\higgsvec,\higgsvec)$
such that, for any $\pr{\mathbf{A}} \in \bar{\mcb{A}}$,   
\begin{equs}[eq:renorm_of_g_system]
(\ell^{\delta,\eps}_{\BPHZ} \otimes \id_{W_{\mfz}}) 
\bar{\bUpsilon}^{F}_{\mfz}(\pr{\mathbf{A}}) 
&= 
\Big( \bigoplus_{i=1}^{3} C^{\eps}_{\YM} \pr{B_i}+C^{\eps}_{\Gauge} \pr{h_i}
 \Big) \oplus C^{\eps}_{\Higgs} \pr{\Psi}\;,
\\
\sum_{\mft\in \{\mfz,\mfm\}}(\ell^{\delta,\eps}_{\BPHZ} \otimes \id_{W_{\mft}}) 
\bar{\bUpsilon}^{\bar{F}}_{\mft}(\pr{\mathbf{A}})
&= 
\Big( \bigoplus_{i=1}^{3}  C^{\delta,\eps}_{\YM}\pr{\bar{A_i}} 
+ C^{\delta,\eps}_{\Gauge} \pr{\bar {h}_{i}}
\Big) \oplus C^{\delta,\eps}_{\Higgs} \pr{\bar{\Phi}}  \;,
\end{equs}
where $C^{\eps}_{\YM}$ and $C^{\eps}_{\Higgs}$ are the same maps as in Proposition~\ref{prop:mass_term}.
One also has the convergence
\begin{equ}\label{eq:delta_to_0}
\lim_{\delta \downarrow 0} C^{\delta,\eps}_{\YM} = C^{\eps}_{\YM}\;,\quad 
\lim_{\delta \downarrow 0} C^{\delta,\eps}_{\Higgs} = C^{\eps}_{\Higgs}\;,\quad
\text{ and } 
\quad
C^{0,\eps}_{\Gauge} \eqdef \lim_{\delta \downarrow 0} C^{\delta,\eps}_{\Gauge}\;.
\end{equ}
Finally there are $C^{\delta,\eps}_{j,\YM},C^{\eps}_{j,\Gauge},C^{\delta,\eps}_{j,\Gauge}  \in L_{G}(\mfg,\mfg)$  and  $C^{\delta,\eps}_{j,\Higgs} \in L_{G}(\higgsvec,\higgsvec)$ for $j\in \{1,2\}$
which are all independent of $\sig$, such that
\begin{equ}[e:Cs-sig]
C^{\eps}_{\Gauge}= 
	\sig^{2} C_{1,\Gauge}^{\eps}
	+\sig^{4} C_{2,\Gauge}^{\eps} 
\quad \mbox{and}\quad
C^{\delta,\eps}_{\bullet} = 
  \sig^{2} C_{1,\bullet}^{\delta,\eps}
+\sig^{4} C_{2,\bullet}^{\delta,\eps}
\end{equ}
where $\bullet \in \{  \textnormal{\footnotesize {YM, Gauge, Higgs}}   \}$.
\end{proposition}

The proof of Proposition~\ref{prop:renorm_g_eqn} requires many intermediate computations and is delayed to the end of Section~\ref{sec:Computation of counterterms} on page~\pageref{proof:renorm_g_eqn}.
We note that the situation here is more complicated than  Proposition~\ref{prop:mass_term}, in 
particular the number of trees of negative degree that appear is much larger and 
power-counting arguments as in Lemma~\ref{lem:counting} do not get us quite as far because 
of the presence of components $\pr{U}$ and $\pr{h}$ of positive regularity. 
We will have to combine power-counting and parity arguments in a more sophisticated way. 

As before, we define $\mathfrak{T}_{-}$  the set of all unplanted trees in $\mathfrak{T}$ with vanishing polynomial label at the root and negative degree.
To impose parity constraints, we define the following sets of trees where we recall the definition of $n(\tau)$ in \eqref{e:def-nf} and write $n(\tau) = (n_{0}(\tau),\dots,n_{d}(\tau)) \in \N^{d+1}$:
\begin{equs}\label{eq:g-trees-parity}
\mfT^{\mathrm{ev,sp}} &= 
\Big\{ \tau \in \mfT:\; \sum_{i=1}^{d}n_{i}(\tau)  + |\{e \in E_{\tau}: \mft(e) = \mfh'\}| \textnormal{ is even }\Big\}\;,
\\
\mfT^{\mathrm{ev,noi}} &= \big\{ \tau \in \mfT:  |\{ e \in E_{\tau}: \mft(e) = \bar{\mfl} \}|\ \textnormal{ is even} \big\}\;.
\end{equs}
The first set above enforces that a tree has ``even parity in space'' while the second enforces that a tree has ``even parity in the noise''.
We also define $\mfT^{\mathrm{od,sp}}$ and $\mfT^{\mathrm{od,noi}}$ analogously with ``even'' replaced by ``odd''. 
We set $\mfT^{\even}_{-} = \mfT_{-} \cap \mfT^{\mathrm{ev,sp}}  \cap \mfT^{\mathrm{ev,noi}}$. 
Note that some of our notation in this section, like $\mfT$ and $\mfT^{\even}_{-}$, collide with notation we used in Section~\ref{sec:renorm-A} but which definition we are referring to should be clear from context \dash namely in this section we are, unless we explicitly say otherwise, always referring to the definitions made in this section.\footnote{Later, in Section~\ref{sec:grafting}, we will refer to sets of tree introduced in  Section~\ref{sec:renorm-A}  but use a different notation to distinguish them.}

We have the following version of Lemma~\ref{lem:symbols_vanish} which is proven in exactly the same way.  
\begin{lemma}\label{lem:parity-bphz-gsym}
Let $\tau\in \mfT$, then $\ell_{\BPHZ}^{\delta,\eps}[\tau] \not = 0 \Rightarrow \tau \in \mfT_{-}^{\even}$.\qed
\end{lemma}
%
\begin{remark}\label{rem:bare-m-renorm_2}
A generalisation of Remark~\ref{rem:bare-m-renorm} also holds in the present setting. 
In particular,  recalling $\mathring{C}_{\mfz}$ and $\mathring{C}_{\mfh}$ from  \eqref{e:ring-C-zh},
 $\mathring{C}_{\mfh}$ cannot generate any new renormalisation by power counting and the counterterms generated by $\mathring{C}_{\mfz}$ are of the same type as in Remark~\ref{rem:bare-m-renorm}, with the BPHZ characters vanishing on them by parity. 
This implies again that the constants on the right-hand side of \eqref{eq:renorm_of_g_system} do not depend on $\mathring{C}_{\mfz}$ or $\mathring{C}_{\mfh}$. 
\end{remark}
%
%
%
%
In what follows, we will first show that the $(Y,U,h)$ system only requires
renormalisation terms that are linear in $(Y,h)$ in the $Y $ equation.
We then point out how to get the same result for the $(\bar X, \bar U, \bar h)$ system.
%
%
\begin{remark}\label{rem:bracket_overload}
We overload our Lie bracket notation in various ways.  
For $(a_{i})_{i=1}^{d}$, $(b_{i})_{i=1}^{d} \in \mfg^{d}$,
 we write $[a,b] \eqdef \sum_{i=1}^{d} [a_{i},b_{i}] \in \mfg$. 
For $u = v \oplus w \in \mfg^{d} \oplus \higgsvec = W_{\mfz}$ and $b \in \mfg^{d}$ we set $[u,b] \eqdef [v,b] \in \mfg$. 
For $h\in \mfg$ and $u\in \mfg^d$,
we define $[h,u] \in \mfg^d$ componentwise. 
\end{remark}

The equations for $h_i$ and $U$ 
contain several products 
that are not classically defined, namely, the term
$ [h_j,\p_j h_i] $, and  $[B_j, h_j]$ which appears three times.
However, the following lemma shows that BPHZ renormalisation doesn't generate any renormalisation\footnote{This also holds in 2D, see \cite[Lem.~7.28]{CCHS2d}.}  in the equations for $h_i$ and $U$.
\begin{lemma}\label{lem:hU-no-renor}
For every $\tau\in\mfT_-^{\even}$ one has 
$\bar{\bUpsilon}^{F}_{\mfh} [\tau] =\bar{\bUpsilon}^{F}_{\mfh'} [\tau]= \bar{\bUpsilon}^{F}_\mfu [\tau]=0$.
\end{lemma}


Recall \cite[Lemma~5.65, Remark~5.66]{CCHS2d}
that the map $\bUpsilon^{F}$ can be computed by an (algebraic) Picard iteration as follows.  
An element  $\mcA \in \expan$, where $\expan$ is the space of expansions as in \cite[Sec.~5.8]{CCHS2d}, has the form
$\mcA_{\mft} = \mcA^{R}_{\mft} +
\sum_{k} \frac{1}{k!}
\mathbf{X}^{k}
\otimes \mathbf{A}_{(\mft,k)}
\in \expan_\mft$ where $\mcA^{R}_{\mft} \in \mcb{B}_\mft\otimes W_\mft$,\footnote{See \cite[Sec.~5.8.2]{CCHS2d} for the definition of $\mcb{B}_\mft$.}
and if  $\mcA \in \expan$ is coherent then (up to truncation at some order),
\begin{equ}[e:explain-substitute]
\mcA^{R}_\mft = \bUpsilon^{F}_\mft(\mathbf{A}^{\mcA})
=
(\mcb{I}_{\mft} \otimes \id_{W_\mft}) 
 \bbUpsilon^{F}_\mft(\mathbf{A}^{\mcA})
=
(\mcb{I}_{\mft} \otimes \id_{W_\mft}) 
F_\mft(\mcA)\;,
\end{equ}
where $\mathbf{A}^{\mcA} \eqdef (\mathbf{A}_{(\mft,k)}) \in \mcb{A}$.
Given $\pr{\mathbf{A}} \in \mcb{A}\subset\expan$, 
we substitute it into the right-hand side of \eqref{e:explain-substitute}, 
and obtain an element $\mcA^{R}$, which together with 
$\pr{\mathbf{A}}$ yields a new element $\mcA\in \expan$ which is then substituted  again into the right-hand side of \eqref{e:explain-substitute}.
By subcriticality, this iteration stabilises after a finite number of steps
to an element which gives $\bUpsilon^{F}(\pr{\mathbf{A}})$.  

We will use a ``substitution'' notation to keep track of how terms are produced by expanding the polynomial $F$ on the right-hand side of \eqref{e:explain-substitute}. As an example, we write 
\begin{equ}\label{eq:substitution-example}
\pr{U} \mcb{I}_{\mfz} (\bar{\bXi})
\leadsto  \brho \big( [ \uwave{\pr{B}\,},\pr{h}] \big) \pr{U}\;.
\end{equ} 
The notation  $ \uwave{\pr{B}\,}$ indicates that we replace $\pr{B}$ with the term on the left-hand side. 
Thus \eqref{eq:substitution-example} corresponds to 
\[
\brho([ \pr{U} \mcb{I}_{\mfz} (\bar{\bXi}),\pr{h}])\pr{U}\;
\]
which is a contribution to $\bbUpsilon_{\mfu}^{F}[\mcb{I}_{\mfz}(\bar{\mfl})]$ (since $\bUpsilon_{\mfz}^{F}[\bar{\mfl}] = \pr{U}\bar{\bXi}$ and $\brho \big( [\pr{B},\pr{h}] \big) \pr{U}$ appears on the RHS for the equation for $U$.)

Note that above we are using the convention for Lie brackets given in Remark~\ref{rem:bracket_overload} and
we set $\bar{\bXi} = \id_E \in \mcb{T}[\bar\mfl]\otimes E$.
Below, when multiple factors are being substituted 
as in \eqref{e:only-ren-U} the substitutions should be read from left to right on both sides. 


\begin{proof}[of Lemma~\ref{lem:hU-no-renor}]
Denote $\mcY \eqdef \mcA_{\mfz}$ and $\mcH \eqdef \mcA_\mfh +\mcA_{\mfh '}$.
We start with the coherent expansion
$q_{\le 0}\mcY = \pr{U} \mcb{I}_{\mfz} (\bar{\bXi})  + Y_0 + \bone\otimes \pr{Y}$,
where $Y_0 $ is obtained by substituting both $\pr{Y}$ in $\pr{Y\p Y}$
by $\pr{U} \mcb{I}_{\mfz} (\bar{\bXi}) $. Here $q_{\le L}$ denotes the projection
onto degrees $\le L$.
A substitution 
$\pr{U} \mcb{I}_{\mfz} (\bar{\bXi})
\leadsto  [ \uwave{\pr{B}\,},\pr{h}]$ in the $h$ equation
yields 
$q_{\le \frac12}\mcH
 =\bone \otimes \pr{h} + h_{1/2}$    
where 
\begin{equ}[e:def-h12]
h_{1/2} \eqdef [ \pr{U} \mcb{I}_{\mfh'}( \mcb{I}_{\mfz} \bar{\bXi}), \pr{h} ]  
 \in \mcb{T}[\tau]\otimes \mfg^3
 \quad  \mbox{with } \tau 
 \in \mfT^{\mathrm{od,sp}}\cap \mfT^{\mathrm{od,noi}}\;.
\end{equ}

Considering the term $\brho([\pr{B_j}, \pr{h_j}]) \pr{U} $ in the $U$ equation:
by a simple power counting the only substitutions 
relevant to $\bar{\bUpsilon}^F_\mfu [\tau]$
with
$\tau\in\mfT_-$ are
\minilab{e:only-ren-U}
\begin{equs}
 \pr{U} \mcb{I}_{\mfz} (\bar{\bXi})
& \leadsto
 \brho([ \uwave{\pr{B}\,}, \pr{h}]) \pr{U} \;,	\label{e:only-ren-U1}
 \\
 Y_0 \leadsto \brho([\uwave{\pr{B}\,}, \pr{h}]) \pr{U} ,
 \qquad&
(  \pr{U} \mcb{I}_{\mfz}( \bar{\bXi}), h_{1/2})
 \leadsto
  \brho([ \uwave{\pr{B}\,}, \uwave{\pr{h}\,}]) \pr{U} \;,
  	\label{e:only-ren-U2}
 \end{equs}
 which only contribute to  $\bar{\bUpsilon}^F_\mfu [\tau]$
 with $\tau \in \mfT^{\mathrm{od,noi}}$
 and $\tau \in \mfT^{\mathrm{od,sp}}$ respectively.
Thus the claim for $\bar{\bUpsilon}^F_\mfu$ holds.

We turn to the equation for $h_i$.
Considering the term $[[\pr{B_j}, \pr{h_j}],\pr{h_i}]$,
by power counting the relevant substitutions are 
$\pr{U} \mcb{I}_{\mfz} (\bar{\bXi}) \leadsto
[[\uwave{\pr{B}\,}, \pr{h}],\pr{h}] $ and
\begin{equ}
(\pr{U} \mcb{I}_{\mfz}( \bar{\bXi}),h_{1/2})
\leadsto
[[\uwave{\pr{B}\,}, \uwave{\pr{h}}],\pr{h}] ,
\qquad
(\pr{U} \mcb{I}_{\mfz}( \bar{\bXi}),h_{1/2})
\leadsto
[[\uwave{\pr{B}\,}, \pr{h}],\uwave{\pr{h}\,}] ,
\end{equ}
 which only contribute to  $\bar{\bUpsilon}^F_\mfh [\tau]$
 with $\tau \in \mfT^{\mathrm{od,noi}}$
 or $\tau \in \mfT^{\mathrm{od,sp}}$.
Concerning the term  $\pr{\p_i} [\pr{B_j},\pr{h_j}]$,
the  substitution 
\[
(\pr{U} \mcb{I}_{\mfz}( \bar{\bXi}),  h_{1/2})
\leadsto
 [\uwave{\pr{B}\,},\uwave{\pr{h}\,}]
\]
gives an updated  expansion
$q_{\le 1}\mathcal{H} 
 =
\bone \otimes \pr{h} + h_{1/2} + h_{1}
+ \mathbf{X}\otimes \pr{\partial h}$
where 
\begin{equ}[e:def-h1]
h_{1} \eqdef  
\mcb{I}_{\mfh'} [ \pr{U}  \mcb{I}_{\mfz}( \bar{\bXi}),
 h_{1/2} ] \quad
 \in \mcb{T}[\tau]\otimes \mfg^3
 \quad  \mbox{with } \tau 
 \in \mfT^{\mathrm{ev,sp}}\cap \mfT^{\mathrm{ev,noi}}\;.
\end{equ}
%
%
Now for the term 
$ [\pr{h_j},\pr{\p_j h_i}]$
 in the equation for $h_i$, the relevant substitutions
\[
h_{1/2} \leadsto  [\pr{h},\pr{\p }\uwave{\pr{h}\,}],
\quad
(h_{1/2},h_{1/2}) \leadsto  [\uwave{\pr{h}\,},\pr{\p} \uwave{\pr{h}\,}],
\quad
h_{1} \leadsto  [\pr{h},\pr{\p} \uwave{\pr{h}\,}]
\]
again lead to trees in
$\mfT^{\mathrm{od,noi}}$
 or $ \mfT^{\mathrm{od,sp}}$, 
 which proves 
 the claim for $\bar{\bUpsilon}^F_\mfh$ and $\bar{\bUpsilon}^F_{\mfh'}$.
%
%
\end{proof}



We now prepare for our analysis of the renormalisation in the equation for $\pr{Y}$ equation by keeping track of the key substitutions involved. 

The main difference between renormalising the $\pr{Y}$ equation versus renormalising the YMH equation of Section~\ref{sec:renorm-A} comes from substituting for $\pr{U}$. 
The coherent expansion (up to order $5/2$) for $\pr{U}$ is written as 
\begin{equ}[e:U-expansion]
q_{\le \frac52} \mcA_{\mfu}
 =
\bone \otimes \pr{U}   
+ \mathbf{X}_{j} \otimes \pr{\p_j U}
+ U_{3/2}+ U_{2}  
+ (\tfrac12 \mathbf{X}_{i} \mathbf{X}_{j} \otimes \pr{\p_i \p_j U}
+ \mathbf{X}_{0}  \otimes \pr{\p_0 U})
+  U_{5/2}
\end{equ}
where the subscripts indicate the degrees (minus some multiple of $\kappa > 0$ which is taken to be arbitrarily small).
Here the term $U_{3/2}$ is obtained by
\eqref{e:only-ren-U1} and 
$U_{2}$ is obtained by
\eqref{e:only-ren-U2}. 
%
The term $U_{5/2}$ is obtained by summing
\minilab{e:U52}
\begin{equs}[2]
(\pr{U} \mcb{I}_{\mfz}( \bar{\bXi}), \mathbf{X}\otimes \pr{\p U} ) 
&\leadsto 
\brho([\uwave{\pr{B}\,}, \pr{h}]) \uwave{\,\pr{U}\, }
\qquad  \qquad  &
\mfT^{\mathrm{od,sp}} &\cap \mfT^{\mathrm{od,noi}}
\\
(\pr{U} \mcb{I}_{\mfz}( \bar{\bXi}), h_1  )
 &\leadsto 
 \brho([\uwave{\pr{B}\,}, \uwave{\pr{h}\,}]) \pr{U} 
\qquad  \qquad   &
\mfT^{\mathrm{ev,sp}} &\cap \mfT^{\mathrm{od,noi}}
			\label{e:U521}
\\
(\pr{U} \mcb{I}_{\mfz}( \bar{\bXi}), \mathbf{X}\otimes \pr{\p h} ) 
&\leadsto
 \brho([\uwave{\pr{B}\,}, \uwave{\pr{h}\,}]) \pr{U} 
\qquad  \qquad  &
\mfT^{\mathrm{od,sp}} &\cap \mfT^{\mathrm{od,noi}}
\\
h_{1/2} 
&\leadsto 
\brho([\pr{B}, \uwave{\pr{h}}]) \pr{U} 
\qquad  \qquad  &
\mfT^{\mathrm{od,sp}} &\cap \mfT^{\mathrm{od,noi}}
\\
( Y_0, h_{1/2} )
&\leadsto 
\brho([\uwave{\pr{B}\,}, \uwave{\pr{h}}]) \pr{U} 
\qquad  \qquad  &
\mfT^{\mathrm{ev,sp}} &\cap \mfT^{\mathrm{od,noi}}
			\label{e:U523}
\\
Y_{1/2}
 &\leadsto 
 \brho([\uwave{\pr{B}\,}, \pr{h}]) \pr{U} 
\qquad  \qquad  &
&\quad \mfT^{\mathrm{od,noi}}
			\label{e:U522}
\\
h_{1/2}
  \leadsto
  \brho(\uwave{\pr{h}})\brho(\pr{h}) & \pr{U}
 \,\mbox{ or } \,
 \brho(\pr{h}) \brho(\uwave{\pr{h}}) \pr{U}
\qquad  \qquad     &
\mfT^{\mathrm{od,sp}} &\cap \mfT^{\mathrm{od,noi}}\;.
\end{equs}
Above we have listed parity information for the substitutions as well \dash recall that the parities of $h_{1/2}$ and $h_{1}$ are given in \eqref{e:def-h12} and \eqref{e:def-h1}.
Here $Y_{1/2}$ is the degree $\frac12 -$
term in the coherent expansion of
$\mcY \eqdef \mcA_{\mfz}$, which arises from
\begin{equs}
(\pr{U} \mcb{I}_{\mfz}( \bar{\bXi}) ,\pr{U} \mcb{I}_{\mfz}( \bar{\bXi}) ,\pr{U} \mcb{I}_{\mfz}( \bar{\bXi}) )
\leadsto 
\uwave{\,\pr{Y}}^3
\qquad 
\mfT^{\mathrm{ev,sp}} &\cap \mfT^{\mathrm{od,noi}}
\\
(Y_{0} ,\pr{U} \mcb{I}_{\mfz}( \bar{\bXi}) )
\; \mbox{or} \;
(\pr{U} \mcb{I}_{\mfz}( \bar{\bXi}) , Y_{0})
\leadsto 
\uwave{\,\pr{Y}} \pr{\p}  \uwave{\,\pr{Y}}
\qquad 
\mfT^{\mathrm{ev,sp}} &\cap \mfT^{\mathrm{od,noi}}
\\
\pr{U} \mcb{I}_{\mfz}( \bar{\bXi}) 
\leadsto 
\pr{Y}\pr{\p} \uwave{\,\pr{Y}},
\qquad
(\mathbf{X} \pr{\p U}	 ,\bar{\bXi})
\leadsto 
\uwave{\pr{\,U}} \uwave{\pr{\bar{\Xi}}}
\qquad \mfT^{\mathrm{od,sp}} 
&\cap \mfT^{\mathrm{od,noi}}\;.
\end{equs}

To find all the relevant trees contributing to the renormalisation of the $(Y,U,h)$ system, we can simply take the  
trees for the YMH system discussed in Section~\ref{sec:renorm-A}
such as  $\<IXiI'Xi_notriangle>$
and replace some of their noises by  the trees that appear in  $q_{\le \frac52} \mcA_{\mfu} \pr{\bar{\Xi}}$.
Intuitively one could think of the $Y$ equation
as driven by ``additive noises'' which are $\xi$ multiplied
by the terms in \eqref{e:U-expansion}.

\begin{remark}
Note that $U_{2}$ depends on $Y_0$, and 
$U_{5/2}$  depends on  $Y_0$ and $Y_{1/2}$.
Fortunately, $Y_0$ and $Y_{1/2}$ only depend on the first two  terms on the right-hand side of \eqref{e:U-expansion} (since
$Y_r$ depends on $U_s$
if $s-\frac52+2 \le r$). This makes some of the explicit calculations below not too difficult.  
\end{remark}




\subsubsection{Grafting}\label{sec:grafting}

In this section we make precise the statement above that we compute the renormalisation of the $Y$ equation by using the trees for the simpler YMH equation of Section~\ref{sec:renorm-A}.
The purpose of this subsection is to facilitate the proof of Lemma~\ref{lem:renorm_linear_in_Y_and_h}.
Instead of trying to compute the renormalisation counterterms by hand, we instead use arguments that only leverage space and noise parity arguments and the relation of the $Y$ equation to the YMH equation of  Section~\ref{sec:renorm-A}. 

Let $\tilde{\mfT}$ be the set of trees that was referred to as $\mfT$ in Section~\ref{sec:renorm-A}.
Since our label set and rule is just an extension of what was introduced in Section~\ref{sec:renorm-A}, we have that $\tilde{\mfT} \subset \mfT$ where $\mfT$ refers to the set of trees we introduced in this section. 
Since our target and kernel space assignments in this section are also just an extension 
of those introduced in Section~\ref{sec:renorm-A}, we also have canonical 
inclusions $\mcb{T}[\tilde{\mfT}] \subset \mcb{T}[\mfT]$. 
We also write $\bUpsilon^{\YMH}$ and $\bar{\bUpsilon}^{\YMH}$ for the corresponding  maps for the YMH system which are defined on $\mcb{T}[\tilde{\mfT}]$. 

We write $\mathfrak{T}^{\YMH}$ for the set of all trees $\tau \in \tilde{\mfT} \cap \mathfrak{T}_{-}$ with $\bUpsilon^{\YMH}[\tau] \neq 0$. 
\begin{remark}
We make some observations about the structure of trees $\tau \in \mfT^{\YMH}$. 
First, any $\tau \in \mfT^{\YMH}$ cannot contain any instances of $\mcb{I}_{(\mfz,p)}(\mbX^{k})$ for any $p,k \in \N^{d+1}$. 
If this was the case then one would be able to replace an instance of $\mcb{I}_{(\mfz,p)}(\mbX^{k})$ in $\tau$ with $\mcb{I}_{(\mfz,p)}(\mcb{I}_{\bar{\mfl}}(\bone))$  to obtain a new tree $\bar{\tau} \in \mfT^{\YMH}$ for which  $\deg(\bar{\tau}) \le \deg(\tau) + \deg(\bar{\mfl})$. But since $\deg(\tau) < 0$, this would violate subcriticality. 

Second, any $\tau \in \mfT^{\YMH}$ cannot have instances of $\mbX^{k}\mcb{I}_{\bar{\mfl}}(\bone)$ for $k \not = 0$ \dash since $\bUpsilon^{\YMH}[\mbX^{k}\mcb{I}_{\bar{\mfl}}(\bone)] = 0$, this would force $\bUpsilon^{\YMH}[\tau]  = 0$.
 
Putting these two observations together, $\tau \in \mfT^{\YMH}$ imposes that the leaves of $\tau$ are given by noises, 
and there are no products of a polynomial with noises. This will be useful to keep in mind for some of the lemmas (and inductive proofs) to follow. 
\end{remark}


We define $\mfT_{g}$ ($g$ stands for ``grafting'') to be the collection of all the trees
used to describe the RHS of \eqref{e:U-expansion}, namely,
\begin{equ}\label{eq:def_of_Tg}
\mfT_{g} \eqdef
 \{ \mbX^{p}: |p|_{\s} \le 5/2\}
\sqcup \{ \mcb{I}_{\mfu}(\tilde{\tau}) \,:\, \bUpsilon_{\mfu}[\tilde{\tau}] \neq 0\,,\; \deg(\tilde{\tau}) \le 1/2
\}\;.
\end{equ}
In particular $\mfT_{g}$ has as elements the abstract  monomials,
$\hat\tau_{3/2} \eqdef \mcb{I}_{\mfu} (\<IXi>)$,
 and
\begin{equ}[e:hat-tau-2]
\hat\tau_{2}^{(1)}
\eqdef
\mcb{I}_{\mfu} \big(\<IXi> \,
	\mcb{I}_{\mfh'} (\<IXi> )\big) \;,
\qquad
\hat\tau_{2}^{(2,i)}
\eqdef
\mcb{I}_{\mfu} \big(\,\<I[IXiI'Xi]_notriangle>\,\big)\;,
\end{equ}
which can be seen by the discussion below \eqref{e:U-expansion} (or \eqref{e:only-ren-U}), and
\begin{equs}[e:hat-tau-52]
\hat\tau_{5/2}^{(1)}
\eqdef 
\mcb{I}_{\mfu} \big( \<IXi> \, \mcb{I}_{\mfh'} 
\big( \<IXi> \mcb{I}_{\mfh'} (\<IXi>) \big) \big)\;,
\quad
\hat\tau_{5/2}^{(2)}
\eqdef
\mcb{I}_\mfu \big(
\<I[IXiI'Xi]_notriangle>  \mcb{I}_{\mfh'} (\<IXi> )
\big)\;,
\\
\hat\tau_{5/2}^{(3)}
\eqdef
\mcb{I}_\mfu \big(\<I[IXi^3]>\big) \;,
\quad
\hat\tau_{5/2}^{(4,i,j)}
\eqdef
\mcb{I}_\mfu \big(\<I[I[IXiI'Xi]I'Xi]>\big)\;,
\quad
\hat\tau_{5/2}^{(5,i,j)}
\eqdef
\mcb{I}_\mfu \big(
\<I[IXiI'[IXiI'Xi]]>
\big)\;,
\end{equs}
which originate from
\eqref{e:U521}, \eqref{e:U523}, \eqref{e:U522}.
Here, $i,j \in \{1,2,3\}$ represent space indices which specify the directions of the
derivatives appearing on the given symbol. 
Again, the subscripts indicate their degrees (modulo a multiple of $\kappa$).  
Our list of the  $\frac52$ trees is not exhaustive;
we listed only those in $\mfT^{\mathrm{ev,sp}} \cap \mfT^{\mathrm{od,noi}}$
since the rest of them will turn out to be irrelevant in renormalisation.

As described earlier, we will modify the trees  $\bar\tau \in \mfT^{\YMH}$ by replacing 
an instance of $\bar{\mfl}$ in $\bar\tau $ 
by $\hat{\tau} \bar{\mfl}$  where $\hat{\tau} \in \mfT_{g}$.
We will also say
that $\hat{\tau}$ is ``grafted'' onto $\bar{\mfl}$ in the tree
$\bar{\tau}$.
Note that grafting $\bone$ onto some $\bar\tau \in \mfT^{\YMH}$
yields $\bar\tau$ itself.
For instance, for $j \in \{1,2,3\}$, we have 
$\partial_{j}\mcb{I}_{\mfz}(\bar{\mfl})\mcb{I}_{\mfz}(\bar{\mfl})  \in \mfT^{\YMH}$
 and two modifications of the tree given by grafting 
 $\hat{\tau} = \hat{\tau}_{3/2}= \mcb{I}_{\mfu} (\mcb{I}_{\mfz}(\bar{\mfl}))$ are 
\[
\partial_{j}\mcb{I}_{\mfz}\left(\,\mcb{I}_{\mfu}(\mcb{I}_{\mfz}(\bar{\mfl}))\,\bar{\mfl}\,\right)\mcb{I}_{\mfz}(\bar{\mfl})
\quad
\textnormal{and}
\quad
\partial_{j}\mcb{I}_{\mfz}(\bar{\mfl})\mcb{I}_{\mfz}\left(\mcb{I}_{\mfu}(\mcb{I}_{\mfz}(\bar{\mfl}))\, \bar{\mfl}\,\right)\;.
\]

%

For any $\bar{\tau} \in \mfT^{\YMH}$ and $\hat{\tau} \in \mfT_{g}$, 
we define $\mfT [\hat{\tau};\bar{\tau}] \subset \mfT$ to be the collection of all the trees $\tau \in \mfT$ which are obtained by precisely one grafting of $\hat{\tau}$  onto an instance of $\bar{\mfl}$ in $\bar{\tau}$. 
More precisely, $\mfT [\hat{\tau};\bar{\tau}]$ is defined inductively in the number of instances of $\bar{\mfl}$ in $\bar{\tau}$. 
For the base case we have $\bar{\tau} = \bar{\mfl}$ and set $\mfT[ \hat{\tau}; \bar{\mfl}] = \{ \hat{\tau}\bar{\mfl} \}$. 
For the inductive step, we note that any $\bar{\tau} \in \mfT^{\YMH}$ with $\bar{\tau} \not = \bar{\mfl}$ can be written
\begin{equ}\label{eq:induction_for_ymh}
\bar{\tau} =  \mbX^{k} \prod_{j=1}^{m} \mcb{I}_{o_{j}}(\bar{\tau}_{j})\;,
\enskip m \ge 1\;, \enskip o_{j} = (\mfz, p_{j}) \in \CE\;, \enskip \bar{\tau}_{j} \in \mfT^{\YMH}\;,
\end{equ}
and we then set  
\begin{equ}\label{eq:inductive_graft_set}
\mfT[\hat{\tau} ; \bar{\tau}] =
\left\{ \tau \in \mfT: 
\begin{array}{c}
\tau 
= \mbX^{k}
  \mcb{I}_{o_{1}}(\bar{\tau}_{1}) \cdots 
 \mcb{I}_{o_{j}}(\mathring{\tau}) 
\cdots   \mcb{I}_{o_{m}}(\bar{\tau}_{m})\\
\textnormal{ for some } 1 \le j \le m \textnormal{ and } \mathring{\tau} \in \mfT[\hat{\tau};\bar{\tau}_{j}]\;
\end{array}
\right\}\;.
\end{equ} 
\begin{remark}\label{rem:grafting-adds}
Note that the parity in space / parity in the number of noises / degree of every tree in $\mfT[\hat{\tau};\bar{\tau}]$ is the sum of the corresponding quantities for $\hat{\tau}$ and $\bar{\tau}$. 
We also note that, for $(\hat{\tau}_{1},\bar{\tau}_{1}),\ (\hat{\tau}_{2},\bar{\tau}_{2} ) \in \mfT_{g} \times \mfT^{\YMH}$, the condition $(\hat{\tau}_{1},\bar{\tau}_{1}) \not =  (\hat{\tau}_{2},\bar{\tau}_{2} )$ forces $\mfT[ \hat{\tau}_{1} ; \bar{\tau}_{1}]$ and $\mfT[ \hat{\tau}_{2} ; \bar{\tau}_{2}]$ to be disjoint.
\end{remark}


%
%

We write  $\mfT_{g}[\bar{\tau}] 
= \bigsqcup_{\hat{\tau} \in \mfT_{g}}
	 \mfT[ \hat{\tau} ;\bar{\tau}]$,
and define the following subspaces of $\mcb{T}$
\begin{equs}
\mcb{T}_{g} &= \mcb{T}[\mfT_{g}]\;,\quad
\mcb{T}^{\YMH} = \mcb{T}[\mfT^{\YMH}]\;,\quad
\textnormal{ and }\quad
\mcb{T}[\hat{\tau} ; \bar{\tau}] = \mcb{T} \big[\mfT[ \hat{\tau} ; \bar{\tau}] \big]\;.
\end{equs} 
We then have the following lemma which classifies the trees in $ \mathfrak{T}^{\even}_{-} $.
\begin{lemma}\label{lemma:ymh_trees_generate}
Let $\tau \in \mfT^{\even}_{-}$ with $\bar{\bUpsilon}^{F}_{\mfz}[\tau] \not = 0$, then $\tau \in \mfT_{g}[\bar{\tau}] $ for  some $\bar{\tau} \in \mfT^{\YMH}$.

Moreover, precisely one of the following statements hold:
\begin{enumerate}
\item $\tau \in  \mathfrak{T}^{\YMH}$.
\item 
$\tau \in \mathfrak{T}[\mathbf{X}_j;\bar{\tau}]$ for some $j \in \{1,2,3\}$ and
\begin{equ}\label{eq:ymhtrees-minus1}
\bar{\tau} \in \mfT^{\mathrm{od, sp}} \cap \mfT^{\mathrm{ev, noi}} \cap \mfT^{\YMH}
\textnormal{ with }
\deg(\bar{\tau}) \le -1\;.
\end{equ}
\item  
$\tau \in \mathfrak{T}[\mcb{I}_{\mfu} (\<IXi>), \bar{\tau}]$ for some 
\begin{equ}\label{eq:ymhtrees-minus3/2}
\bar{\tau} \in \mfT^{\mathrm{ev,sp}}\cap \mfT^{\mathrm{od, noi}} \cap \mfT^{\YMH} \textnormal{ with } \deg(\bar{\tau}) \le -3/2\;.
\end{equ} 
\item 
$\tau \in \mathfrak{T}[\hat{\tau};\<IXiI'Xi_notriangle>]$ for some $\hat{\tau}$
given in \eqref{e:hat-tau-2}.
\item 
$\tau \in  \mathfrak{T}[\hat{\tau};\<Xi>]$ for some $\hat{\tau}$
given in \eqref{e:hat-tau-52}.
\end{enumerate}
%
\end{lemma}

\begin{remark}
Note that the assumption $\bar{\bUpsilon}^{F}_{\mfz}[\tau] \not = 0$ is needed because the set of trees $\mfT$ includes many trees that don't even appear in the coherent expansion of the $\pr{Y}$ system (such as $\mfl$) and the lemma would fail in these cases. 
\end{remark}

\begin{remark}\label{rem:bar-trees-cases23}
The trees satisfying \eqref{eq:ymhtrees-minus3/2}
are
 \<IXi^3_notriangle>,
\<IXiI'[IXiI'Xi]_notriangle>,
\<I[IXiI'Xi]I'Xi_notriangle>.
Some examples of trees satisfying \eqref{eq:ymhtrees-minus1}
are 
\begin{equ}
\<IXiI'Xi_notriangle>, \quad
\<Psi'I[YPsi']_notriangle>,\quad
\<Y'Y_notriangle>,\quad
\<PsiPsiY_notriangle>,\quad
\<PsiI'[Psi3]_notriangle>,\quad
\<Psi'I[Psi3]_notriangle>
\end{equ}
In fact, \eqref{eq:ymhtrees-minus1} implies that $\bar\tau$ can only have $2$ or $4$ noises.
Indeed, since 
$\bar{\tau} \in \mfT^{\mathrm{ev, noi}} \cap \mfT^{\YMH}$,
in the notation of Lemma~\ref{lem:counting},
$k_\xi$ is even, so by \eqref{e:iden1} $n_\lambda-p_X$ must be odd;
then using $\deg(\bar{\tau}) \le -1$ and  \eqref{e:iden1}, we see that the only solutions to \eqref{counting-identities}
are $k_\xi=2$ and $k_\xi=4$.
\end{remark}

\begin{proof} 
In this proof we keep in mind that parities and degrees 
are additive under grafting (Remark~\ref{rem:grafting-adds}).

By the discussion in the beginning of this subsection,
for every $\tau \in \mfT$ with $\bar{\bUpsilon}^{F}_{\mfz}[\tau] \not = 0$, there is a unique way of obtaining $\tau$ by choosing some $\bar{\tau} \in \tilde{\mfT}$ with $\bUpsilon^{\YMH}[\bar{\tau}] \not = 0$, $m \ge 0$, and $\hat{\tau}_{1},\dots,\hat{\tau}_{m}$ with $\hat{\tau}_{j} \not = \bone$ and $\bUpsilon^{F}_{\mfu}[\hat{\tau}_{j}] \not = 0$, and then choosing $m$ instances of $\bar{\mfl}$ in $\bar{\tau}$ and replace them by $\hat{\tau}_{1}\bar{\mfl},\dots,\hat{\tau}_{m}\bar{\mfl}$.  
Recalling that each tree
in $\mfT [\hat{\tau};\bar{\tau}]$ is obtained by only one grafting by definition,
to prove the first statement we must show $m=1$.

Now, since we impose $\tau \in \mfT_{-}$, then, for $\kappa>0$ sufficiently small, one necessarily has $\deg(\hat{\tau}_j) \ge 1$ for all $j\in \{1,\cdots,m\}$ and therefore one must have $m \le 1$ unless $\bar{\tau} = \<IXiI'Xi_notriangle>$. 
This is because every $\bar\tau$ satisfying the above condition with at least two noises has degree strictly larger than $-2$, except for  $\bar{\tau} = \<IXiI'Xi_notriangle>$.

Suppose $\bar{\tau} = \<IXiI'Xi_notriangle>$, for which $\deg(\bar{\tau})=-2-2\kappa$.
Then since $\tau \in \mfT_{-}$, one must have $m\le 2$.
If $m = 2$ then one must have $\hat{\tau}_{j} = \mbX_{i_j}$ for both $j=1,2$ (since, for $\hat{\tau}_{j}$ not of this form, one has $\deg(\hat{\tau}_{j}) \ge 3/2-$). 
However, in this case $\tau \not \in \mfT^{\mathrm{ev},\mathrm{sp}}$, therefore we must have $m=1$. 
Since  $\tau \in \mfT_{-}$ and $\deg(\hat{\tau}) \ge 0$ furthermore imply $\deg(\bar{\tau}) < 0$,
one has $\bar{\tau} \in \mfT^{\YMH}$, thus proving the first statement. 

We now 
prove the second statement.
Note that the claim that we fall into {\it precisely one} of these cases follows 
from the disjointness statement in Remark~\ref{rem:grafting-adds}. 
The key point to prove is that the sets in the statement of the lemma cover $\mathfrak{T}^{\even}_{-}$.
Write $\tau \in \mfT[\hat{\tau}; \bar{\tau}]$ with $\hat{\tau} \in \mfT_{g}$ and $\bar{\tau} \in \mfT^{\YMH}$. 

\begin{itemize}
\item
If $\hat{\tau} = \bone$, then $\tau = \bar{\tau}$ so $\tau \in \mfT^{\YMH}$. 
\item
If $\hat{\tau} = \mbX^{p}$ with $|p|_{\s} =2$ then we must have $\deg(\bar{\tau}) \le -2$, which forces $\bar{\tau} = \<Xi>$ or $\bar{\tau} = \<IXiI'Xi_notriangle>$.  But in both cases $\tau\notin \mathfrak{T}^{\even}_{-}$.  
\item
If $\hat{\tau} = \mbX^{p}$ for $|p|_{\s} = 1$ then we must have $\bar{\tau}$ as in \eqref{eq:ymhtrees-minus1}, since parity and degree are additive under grafting.
\end{itemize}

Thus all that remains are the cases where $\hat{\tau}$ belongs to the second set on the right-hand side of \eqref{eq:def_of_Tg} so we assume this from now on. 

\begin{itemize}
\item
There is a single such $\hat{\tau}$ of minimal degree given by 
$\hat{\tau} = \mcb{I}_{\mfu} (\<IXi>)$ with $\deg(\hat{\tau} )=3/2-$. 
This choice of $\hat{\tau}$ forces $\bar{\tau}$ to be as in \eqref{eq:ymhtrees-minus3/2}. 
\item
If $\deg(\hat{\tau}) = 2-$, it has to be one of the trees in \eqref{e:hat-tau-2},
 so we have $\hat\tau \in \mfT^{\mathrm{ev, noi}} \cap \mfT^{\mathrm{od, sp}}$, 
 which forces $\bar{\tau}$ to be in $\mfT^{\mathrm{ev, noi}} \cap \mfT^{\mathrm{od, sp}}$ as well and $\deg(\bar{\tau})\le -2$,
 so $\bar{\tau} = \<IXiI'Xi_notriangle>$.
\item
If $\deg(\hat{\tau}) = 5/2-$ then we must have $\bar{\tau} = \<Xi>$ in order for $\deg(\tau) < 0$.
This in turn forces $\hat{\tau} \in  \mfT^{\mathrm{od, noi}} \cap  \mfT^{\mathrm{ev, sp}}$, and all such $\hat{\tau} $ are listed in \eqref{e:hat-tau-52}.
\end{itemize}
Thus we  proved that $\tau$ is in precisely one of the five cases claimed by the lemma.
\end{proof}

Below we study the effect of the $\bUpsilon$ map acting on the trees 
that appeared in the previous proof.

The following lemma states that the contribution to the renormalised
equation from trees in $\mathfrak{T}^{\YMH}$ is the same as what was observed in Section~\ref{sec:renorm-A}.
In what follows, we overload notation and write $\pr{U}^{\ast} \in L(\mcb{F},\mcb{F})$ for the operator  constructed via \cite[Remark~5.19]{CCHS2d}  using the operator assignment $L = (L_{\mft})_{\mft \in \Lab}$ given by setting $L_{\mft} = \pr{U}^{\ast}$ if $\mft \in \{ \bar{\mfl},\mfl\}$  and $L_{\mft} = \id_{V_{\mft}}$ otherwise; namely, $\pr{U}^{\ast}$ acts on $\mcb{F}$ by acting on each factor $V_{\mfl}$ and $V_{\bar{\mfl}}$.

We also note that, while $\bUpsilon^{\YMH}_{\mfz}[\tau](\cdot)$ should take as an argument an element of $\prod_{ o = (\mfz,p), (\bar{\mfl},p)} W_{o}$, there is a natural projection to this space from $\mcb{A} = \prod_{o \in \CE} W_{o}$ by dropping all other components, so below we write $\bUpsilon^{\YMH}_{\mfz}[\tau](\cdot)$ with arguments $\pr{\mathbf{A}} \in \mcb{A}$ with this projection implicitly understood.

\begin{lemma}\label{lem:gauge_transformed_ymhtrees}
For any $\tau \in \mathfrak{T}^{\YMH}$ 
one has
\begin{equs}
\bar\bUpsilon^{F}_{\mfz}[\tau](\pr{\mathbf{A}}) 
&= \pr{U}^{\ast}  \bar\bUpsilon^{\YMH}_{\mfz}[\tau](\pr{\mathbf{A}})\;.
\end{equs}
\end{lemma}  
\begin{proof}
We can  proceed by an induction in the number of instances of $\bar{\mfl}$ in the tree $\tau$, appealing to the inductive definition \eqref{eq:upsilon_induction}. 
The base case of the induction is given by $\tau= \mcb{I}_{\bar{\mfl}}(\bone) = \bar{\mfl}$ in which case\footnote{Note that if $p \not = 0 $ then the desired statement does not hold for $\tau = \mbX^{p}\mcb{I}_{\bar{\mfl}}[\bone]$ but $\mbX^{p}\mcb{I}_{\bar{\mfl}}[\bone] \not \in \mfT^{\YMH}$.}
\[
\bar\bUpsilon^{F}_{\mfz}[\bar{\mfl}]=  \pr{U} \bar{\bXi} = 
\pr{U}^{\ast} \bar{\bXi} = \pr{U}^{\ast} \bar\bUpsilon^{\YMH}_{\mfz}[\bar{\mfl}]\;,
\]
where $\bar{\bXi} = \id_{W_{\mfz}} \in V_{\bar{\mfl}} \otimes W_{\mfz}$. 
Now for the inductive step we write $\bar{\tau} \in \mfT^{\YMH} \setminus \{\bar{\mfl}\}$ as in \eqref{eq:induction_for_ymh} and we have, since $m \ge 1$ and and $o_{j}= (\mfz,p_{j})$, 
\[
\partial^{k} D_{o_1}\cdots D_{o_m}
\bar{\bUpsilon}^{F}_{\mfz}[\bone] = \partial^{k} D_{o_1}\cdots D_{o_m}
\bar{\bUpsilon}^{\YMH}_{\mfz}[\bone]\;.
\] 
Then by \eqref{eq:upsilon_induction} one has 
\begin{equs}
\bar{\bUpsilon}^{F}_{\mfz} [\bar{\tau}] 
&= 
\frac{1}{\mathring{S}(\bar{\tau})}
\mbX^{k}
\Big[ 
\partial^{k} D_{o_1}\cdots D_{o_m}
\bar{\bUpsilon}^{\YMH}_{\mfz}[\bone]
\Big]
\big(\bUpsilon_{o_1}[\tau_1],\ldots,\bUpsilon_{o_m}[\tau_m]\big)\\
{}&
=
\frac{1}{\mathring{S}(\bar{\tau})}
\mbX^{k}
\Big[ 
\partial_{k} 
D_{o_1}\cdots D_{o_m}
\bar{\bUpsilon}^{\YMH}_{\mfz}[\bone]
\Big]
\big(\pr{U}^{\ast}\bUpsilon^{\YMH}_{o_1}[\tau_1],\ldots, \pr{U}^{\ast}\bUpsilon^{\YMH}_{o_m}[\tau_m]\big)\\
{}&=
\frac{1}{\mathring{S}(\bar{\tau})}
\pr{U}^{\ast} 
\mbX^{k}\Big[ \partial_{k}
D_{o_1}\cdots D_{o_m}
\bar{\bUpsilon}^{\YMH}_{\mfz}[\bone]
\Big]
\big(\bUpsilon^{\YMH}_{o_1}[\tau_1],\ldots, \bUpsilon^{\YMH}_{o_m}[\tau_m]\big)\\
{}&=
\pr{U}^{\ast}
\bar\bUpsilon^{\YMH}_{\mfz}[\bar{\tau}]\;,
\end{equs}
where in the second equality we used the induction hypothesis, 
and in the third equality we used the definition of $\pr{U}^{\ast} $
given before this lemma.
\end{proof}
The above lemma shows that if we don't expand $U$ in the multiplicative noise $U\xi$, the corresponding trees are the same as those seen for the YMH system in the last section, and the renormalisation they produce will be  the same 
as that stated in 
Proposition~\ref{prop:mass_term}.
In the remainder of this subsection we will actually expand $U$
(thus generalising Lemma~\ref{lem:gauge_transformed_ymhtrees}),
and performing a single such expansion can be formulated using the grafting procedure we described earlier. 
For each tree $\tau$ that comes from grafting, i.e.\ 
$\tau \in \mfT[\hat{\tau};\bar{\tau}]$,
we aim to describe $\bar{\bUpsilon}^{F}_{\mfz}[\tau]$.

Before we proceed, a simple example may be helpful to 
motivate the upcoming lemmas. 
Consider 
$\tau \in \mathfrak{T}[\hat\tau, \bar\tau]$
with $\hat\tau=\mcb{I}_{\mfu} (\<IXi>)$
and $\bar\tau=\<IXi^3_notriangle>$, as in \eqref{eq:ymhtrees-minus3/2}.
Pretending that we only have 
a nonlinearity $[B_j,[B_j,B_i]]$ in the equation for $Y$
and ignoring the other terms, 
$\bar{\bUpsilon}^{F}_{\mfz}[\tau]$ is given by
\begin{equs}[e:example-Ups-graft]
{} [\pr{U} \mcb{I}_{\mfz}\bar\bXi  ,  [\pr{U} \mcb{I}_{\mfz}\bar\bXi ,
\mcb{X}\mcb{I}_{\mfz}\bar\bXi ]]
&+
[\pr{U} \mcb{I}_{\mfz}\bar\bXi  ,  [ \mcb{X} \mcb{I}_{\mfz}\bar\bXi ,
\pr{U}\mcb{I}_{\mfz}\bar\bXi ]]
\\
&+
[ \mcb{X}\mcb{I}_{\mfz}\bar\bXi  ,  [\pr{U} \mcb{I}_{\mfz}\bar\bXi ,
\pr{U}\mcb{I}_{\mfz}\bar\bXi ]]\;,
\end{equs}
where $\mcb{X} = \brho ([\pr{U}\mcb{I}_{\mfu}  \mcb{I}_{\mfz}\bar\bXi ,\pr{h}]) \pr{U} $  (recall \eqref{e:only-ren-U1}).
We can think of this as being obtained by the usual 
inductive construction of $\bar{\bUpsilon}^{F}$
except for a tweak that one starts as the base case 
with not only $\bar\bXi$ but also $\mcb{X}\bar\bXi$; this will be made precise in Lemma~\ref{lem:Ups-graft}.
In order to factor out $\pr{U}$ (to eventually show that the renormalisation
{\it will not} depend on $\pr{U}$), 
note that \eqref{e:example-Ups-graft}
can be thought of as being obtained by taking
$\bar{\bUpsilon}^{\YMH}_{\mfz}[\bar\tau]=
[ \mcb{I}_{\mfz}\bar\bXi  ,  [ \mcb{I}_{\mfz}\bar\bXi ,
 \mcb{I}_{\mfz}\bar\bXi ]]$
and inserting $ \brho ([\mcb{I}_{\mfu}  \mcb{I}_{\mfz}\bar\bXi ,\pr{h}]) $
into each appropriate location,
and finally applying $\pr{U}^{\ast}$ in the way explained above Lemma~\ref{lem:gauge_transformed_ymhtrees}
\dash this will be formulated as a linear map 
which brings $\bar{\bUpsilon}^{\YMH}_{\mfz}[\bar\tau]$ to 
\eqref{e:example-Ups-graft}; see Lemma~\ref{lem:grafting_operators} for the precise formulation.

Given $\mcb{X} \in \mcb{T}_{g} \otimes W_{\mfu}$ and $\bar{\tau} \in \mfT^{\YMH}$ we will define objects $\bar{\bUpsilon}^{\bar{\tau},\mcb{X}}_{\mfz}(\pr{\mathbf{A}})  \in \mcb{T} \big[\mfT_{g}[\bar{\tau}] \big]\otimes W_{\mfz}$. 
Our definition of $\bar{\bUpsilon}^{\bar{\tau},\mcb{X}}_{\mfz}$ is inductive in the number of $\bar{\mfl}$ edges in $\bar{\tau}$. 
For the base case, we set 
\begin{equ}[e:mcbX-Xi]
\bar{\bUpsilon}_{\mfz}^{\bar{\mfl},\mcb{X}}(\pr{\mathbf{A}}) \eqdef \mcb{X} \boldsymbol{\bar{\Xi}}\;.
\end{equ}
For the inductive step, for $\bar{\tau}$ as in \eqref{eq:induction_for_ymh}, we set
\begin{equs}\label{e:def-Ups-mcbX}
{}&\bar{\bUpsilon}_{\mfz}^{\bar{\tau},\mcb{X}}(\pr{\mathbf{A}})
\\
&= 
\frac{1}{\mathring{S}(\bar{\tau})}
\sum_{j=1}^{m} 
\mbX^{k}
\Big[ 
\partial^{k}
D_{o_1}\cdots D_{o_m}
\bar{\bUpsilon}^{F}_{\mfz}[\bone]
\Big]
\big(\bUpsilon_{o_1}[\bar{\tau}_1],
\ldots, \bUpsilon_{o_{j}}^{\bar{\tau}_{j},\mcb{X}}, \ldots,\bUpsilon_{o_m}[\bar{\tau}_m]\big)\;,
\end{equs}
where $\bUpsilon_{(\mfz,p)}^{\bar{\tau}_{j},\mcb{X}}
= \mcb{I}_{(\mfz,p)}(\bar{\bUpsilon}^{\bar{\tau}_{j},\mcb{X}}_{\mfz})$.
From this inductive definition it is clear that the mapping 
\[
\mcb{T}_{g} \otimes W_{\mfu} \ni \mcb{X} \mapsto  \bar{\bUpsilon}^{\bar{\tau},\mcb{X}} \in \mcb{T} \otimes W_{\mfz}
\]
 is linear, and takes $\mcb{T}[\hat{\tau}] \otimes W_{\mfu}$ into $\mcb{T}[\hat{\tau};\bar{\tau}] \otimes W_{\mfz}$ for every  $\hat{\tau} \in \mfT_{g}$.  

In the next lemma we specify 
$\mcb{X}$ to be the terms appearing in the expansion for $U$
in \eqref{e:U-expansion}.
Recall the shorthand notation $\pr{\p^p U}$
for the components $\pr{A_{(\mfu,p)}}$ of $\pr{\mathbf{A}}$.

\begin{lemma}\label{lem:Ups-graft}
For any $\bar{\tau} \in \mfT^{\YMH}$ and $\hat{\tau} \in \mfT_{g}$, we have the consistency relation
\[
\bar{\bUpsilon}_{\mfz}^{\bar{\tau},\ \mathcal{U}[\hat{\tau}](\pr{\mathbf{A}})}(\pr{\mathbf{A}})
=
\sum_{\tau \in \mfT[\hat{\tau};\bar{\tau}]}
\bar{\bUpsilon}^{F}_{\mfz}[\tau](\pr{\mathbf{A}})\;,
\]
where
\begin{equ}[e:CU-tau-hat]
\mathcal{U}[\hat{\tau}](\pr{\mathbf{A}}) 
= 
\begin{cases}
\frac{1}{p!}\pr{\p^p U} \mbX^{p} & 
\textnormal{ if }\hat{\tau} = \mbX^{p}\;,
\\
\bUpsilon_{\mfu}[\tilde{\tau}](\pr{\mathbf{A}}) & 
\textnormal{ if } \hat{\tau} = \mcb{I}_{\mfu}(\tilde{\tau})\;.
\end{cases}
\end{equ}
\end{lemma}
\begin{proof}
The proof follows by induction in the number of $\bar{\mfl}$ edges in $\bar{\tau}$. 
The base case is given by $\bar{\tau} = \bar{\mfl}$, then if $\tau \in \mfT_{g}[\bar{\mfl}]$ one has $\tau =  \hat{\tau} \bar{\mfl}$ for some $\hat{\tau} \in \mfT_{g}$ and the result follows by a straightforward computation:
indeed recalling the definition of $\mfT_{g}$ in \eqref{eq:def_of_Tg} and by \eqref{eq:upsilon_induction}  one has
\begin{equs}
\bar{\bUpsilon}^{F}_{\mfz}[ \mbX^p \mcb{I}_{\bar{\mfl}}(\bone)]
&= \frac{1}{p!}\mbX^p \big[\p^p D_{\bar{\mfl}} \bar{\bUpsilon}^{F}_{\mfz}[\bone]\big]
	(\bUpsilon^F_{\bar{\mfl}} [\bone])
= \frac{1}{p!}\pr{\p^p U} \mbX^p \bar{\bXi}\;,
\\
\bar{\bUpsilon}^{F}_{\mfz}[ \mcb{I}_{\mfu} (\tilde\tau) \mcb{I}_{\bar{\mfl}}[\bone]]
&=  \big[D_{\mfu} D_{\bar{\mfl}} 
\bar{\bUpsilon}^{F}_{\mfz}[\bone]\big]
(\bUpsilon^F_{\mfu}[\tilde\tau] ,\bUpsilon^F_{\bar{\mfl}} [\bone])
= \bUpsilon^F_{\mfu}[\tilde\tau]  \bar{\bXi}.
\end{equs}
which are consistent with \eqref{e:mcbX-Xi}. 
For the inductive step, we can assume that $\bar{\tau}$ is of the form \eqref{eq:induction_for_ymh}, and then the claim follows by combining 
 the inductive definition \eqref{eq:upsilon_induction}, \eqref{eq:inductive_graft_set} and \eqref{e:def-Ups-mcbX} along with the induction hypothesis.
 Indeed, by \eqref{eq:upsilon_induction} and \eqref{eq:inductive_graft_set}
\begin{equs}
{}& \sum_{\tau \in \mfT[\hat{\tau};\bar{\tau}]}
\bar{\bUpsilon}^{F}_{\mfz}[\tau](\pr{\mathbf{A}})
= 
\frac{1}{\mathring{S}(\bar{\tau})} \sum_{j=1}^m  \frac{1}{\bar\beta_j}
\sum_{\mathring{\tau} \in \mfT[\hat{\tau};\bar{\tau}_{j}]}
\frac{\mathring{S}(\bar{\tau})}{\mathring{S}(\tau)}
\\
& \qquad \times
\mbX^{k}
\Big[ 
\partial^{k}
 D_{o_1}\cdots D_{o_m}
\bar{\Upsilon}_{\mfz}[\bone]
\Big]
\big(\bUpsilon_{o_1}[\bar\tau_1],\ldots, 
\bUpsilon_{o_j} [\mathring\tau] ,
  \ldots,\bUpsilon_{o_m}[\bar\tau_m]\big)\;,
\end{equs}
where $\bar\beta_j$ denotes the multiplicity of $\mcb{I}_{o_j}[\bar\tau_j]$
in $\bar\tau$ \dash we need to divide by this number because our sum above is overcounting the elements of \eqref{eq:inductive_graft_set} when any $\mcb{I}_{o_j}[\bar\tau_j]$ appears with multiplicity. 
Here on the right-hand side, we are writing $\tau = \mbX^{k} \mcb{I}_{o_{1}}(\bar{\tau}_{1})\cdots \mcb{I}_{o_{j}}(\mathring{\tau}) \cdots \mcb{I}_{o_{m}}(\bar{\tau}_{m})$
for every fixed $j$ and $\mathring{\tau}$. 
We  then observe that 
$\frac{\mathring{S}(\bar{\tau})}{\mathring{S}(\tau)} = \bar{\beta}_{j}$ 
since $\mathring{S}$ introduced in  \eqref{eq:upsilon_induction}+\eqref{e:sym_at_root} is precisely
a product of factorials of multiplicities.
Now using the inductive hypothesis which states that
\[
\sum_{\mathring{\tau} \in \mfT[\hat{\tau};\bar{\tau}_{j}]} \bUpsilon_{o_j} [\mathring\tau] 
=
\bar{\bUpsilon}_{\mfz}^{\bar{\tau}_{j},\ \mathcal{U}[\hat{\tau}](\pr{\mathbf{A}})}(\pr{\mathbf{A}})\;,
\] 
we obtain the claimed identity by \eqref{e:def-Ups-mcbX}.
\end{proof}
%


We now introduce the operator on our regularity structure that are related to the grafting procedure on trees we described earlier. 
\begin{lemma}\label{lem:grafting_operators}
There exists $\mathcal{L}_{g} \in L \big( \mcb{T}_{g} \otimes W_{\mfu}, L( \mcb{T}^{\YMH}, \mcb{T}) \big)$ with the following properties
\begin{enumerate}
\item  For any $\hat{\tau} \in \mfT_{g}$, $\bar{\tau} \in \mfT^{\YMH}$, and $\mcb{X} \in \mcb{T}[\hat{\tau}] \otimes W_{\mfu}$, $\mathcal{L}_{g} \mcb{X}$ maps $\mcb{T}[\bar{\tau}]$ into $\mcb{T}[\hat{\tau}; \bar{\tau}] $. 
\item 
For any $\bar{\tau} \in \mfT^{\YMH}$, 
\begin{equ}[e:LgX-YMH]
\bUpsilon_{\mfz}^{\bar{\tau}, (\pr{U}^{\ast} \otimes \bullet \pr{U}) \mcb{X}}(\pr{\mathbf{A}})
=
\pr{U}^{\ast} \big( \mathcal{L}_{g} \mcb{X} \big) \bUpsilon_{\mfz}^{\YMH}[\bar{\tau}](\pr{\mathbf{A}})\;,
\end{equ}
where $\pr{U}$ is given by the relevant component of $\pr{\mathbf{A}}$.
Here we are applying the convention of Remark~\ref{rem:suppress_ident} so $\big( \mathcal{L}_{g} \mcb{X} \big)$ and $\pr{U}^{\ast}$ act only on the left factor of  $\mcb{T} \otimes W_{\mfz}$; and the notation  $\bullet \pr{U}$ denotes right multiplication \slash composition by $\pr{U}$, that is the mapping $W_{\mfu} \ni M \mapsto  M \pr{U} \in W_{\mfu}$. 
\end{enumerate}
\end{lemma}
\begin{proof} 
We first define, for fixed $\mcb{X} \in \mcb{T}_{g} \otimes W_{\mfu}$, $\mathcal{L}_g\mcb{X} \in  L(\mcb{T}^{\YMH} , \mcb{T})$  by defining  $\big (\mathcal{L}_g\mcb{X} \big)\restriction_{\mcb{T}[\bar{\tau}]}$ for $\bar{\tau} \in \mfT^{\YMH}$  inductively in the number of $\bar{\mfl}$ edges in $\bar{\tau}$. 
For the base case $\bar{\tau} = \bar{\mfl}$, we have, for any $v \in \mcb{T}[\bar{\mfl}] $, $\hat{\tau} \in \mfT_{g}$, and $\mcb{X} = a \otimes b \in \mcb{T}[\hat{\tau}] \otimes W_{\mfu}$, 
\begin{equ}\label{eq:base-case-linmap}
 \big(\mathcal{L}_{g} a \otimes b \big) v 
 = a \cdot (b^{\ast}  v) \in \mcb{T}[ \hat{\tau} ; \bar{\tau}] \;.
\end{equ}
Here, to interpret $b^{\ast} v$,  we note that $W_{\mfu}$ has an adjoint action on $V_{\bar{\mfl}} \simeq W_{\mfz}^{\ast}$ and the product outside of the parentheses above is just the product in the regularity structure. 
This extends to arbitrary $\mcb{X} \in \mcb{T}[\hat{\tau}] \otimes W_{\mfu}$ by linearity. 

For the inductive step, given $\bar{\tau}$ as in \eqref{eq:induction_for_ymh} and $v_{j} \in \mcb{T}[\bar{\tau}_{j}] $ we define 
\begin{equ}[e:inductive_def_of_Ls]
\big(\mathcal{L}_{g} \mcb{X} \big)  
\mbX^{k} \prod_{j=1}^{m} \mcb{I}_{o_{j}}(v_j)
=
\frac{\mbX^{k}}{\mathring{S}(\bar{\tau})}
\sum_{j=1}^m \big( \mcb{I}_{o_{1}}(v_1)\big) 
\cdots 
\mcb{I}_{o_{j}}
\big(  (\mathcal{L}_{g} \mcb{X}) v_{j} \big) 
\cdots 
\big(\mcb{I}_{o_{m}}(v_m)\big)\;.
\end{equ}
By the definition of the space $\mcb{T}[\bar{\tau}]$ and the symmetry of the above expressions this suffices to define $\big(\mathcal{L}_{g} \mcb{X} \big)\restriction_{\mcb{T}[\bar{\tau}]}$ and the first statement of the lemma is easy to verify.

The second statement is a straightforward to  verify by induction.
For the base case $\bar\tau =\bar{\mfl}$, we note that for $\mcb{X} = a \otimes b \in \mcb{T}_{g} \otimes W_{\mfu}$, 
\begin{equs}
\pr{U}^{\ast} \big(\mathcal{L}_{g} \mcb{X} \big) \bUpsilon_{\mfz}^{\YMH}[\bar{\mfl}](\pr{\mathbf{A}})
&=
\pr{U}^{\ast} \big(\mathcal{L}_{g} \mcb{X} \big) \boldsymbol{\bar{\Xi}}
= \pr{U}^{\ast} ( a \cdot (b^{\ast} e_{i}^{\ast})) \otimes e_{i} 
\\
& =(\pr{U}^{\ast} a \cdot \pr{U}^{\ast}b^{\ast} e_{i}^{\ast} )\otimes e_{i}
 = (\pr{U}^{\ast} a \otimes  b\pr{U})( e_{i}^{\ast} \otimes e_{i})
 \\
& = \big[ (\pr{U}^{\ast} \otimes \bullet \pr{U})\mcb{X} \big] \boldsymbol{\bar{\Xi}} = \bar{\bUpsilon}^{\bar{\mfl},(\pr{U}^{\ast} \otimes \bullet \pr{U}) \mcb{X}}_{\mfz}(\pr{\mathbf{A}})
\end{equs}
by \eqref{e:mcbX-Xi}, where the $\{e_{i}\}_{i=1}^{\dim(E)}$ are a basis for $E$ and  $i$ is summed implicitly. 
For the inductive step, by definition of $\bar{\bUpsilon}^{\YMH}_{\mfz}$,
the right-hand side of \eqref{e:LgX-YMH} equals 
 \begin{equs}
{}& \pr{U}^{\ast}  \sum_{j=1}^m
\frac{\mbX^k }{\mathring{S}(\bar\tau)} 
\Big[ 
\partial^k D_{o_1}\cdots D_{o_m}
\bar{\Upsilon}_{\mft}[\bone]
\Big]
\big(\bUpsilon^{\YMH}_{o_1}[\bar\tau_1],\ldots,
 \big( \mathcal{L}_{g} \mcb{X} \big) \bUpsilon^{\YMH}_{o_j}[\bar\tau_j] ,\ldots,
\bUpsilon^{\YMH}_{o_m}[\bar\tau_m]\big)
\\
&=
 \sum_{j=1}^m
\frac{\mbX^k }{\mathring{S}(\bar\tau)} 
\Big[ 
\partial^k D_{o_1}\cdots D_{o_m}
\bar{\Upsilon}_{\mft}[\bone]
\Big]
\big(\bUpsilon^{F}_{o_1}[\bar\tau_1],\ldots,
\bUpsilon_{o_j}^{\bar{\tau}_j, (\pr{U}^{\ast} \otimes \bullet \pr{U}) \mcb{X}},\ldots,
\bUpsilon^{F}_{o_m}[\bar\tau_m]\big)\;,
\end{equs}
where in the last step we applied the inductive hypothesis
and Lemma~\ref{lem:gauge_transformed_ymhtrees}. 
This is precisely  the left-hand side of \eqref{e:LgX-YMH} 
by definition \eqref{e:def-Ups-mcbX}.
\end{proof}

\subsection{Computation of counterterms}
\label{sec:Computation of counterterms}

\begin{lemma}\label{lem:const_upsilon}
Suppose that $\bar{\tau}$ is of the form $\<IXiI'Xi_notriangle>$ or $\<Xi>$ or satisfies \eqref{eq:ymhtrees-minus1} or \eqref{eq:ymhtrees-minus3/2}. 
Then, $\bUpsilon_{\mfz}^{\YMH}[\bar{\tau}](\pr{\mathbf{A}})$ has no dependence on $\pr{\mathbf{A}}$. 
\end{lemma}
\begin{proof}
The cases of $\<IXiI'Xi_notriangle>$ or $\<Xi>$ follow trivially from a short, straightforward computation.

To prove the lemma in the cases of \eqref{eq:ymhtrees-minus1} or \eqref{eq:ymhtrees-minus3/2} we invoke  Lemma~\ref{lem:counting}
and show that $p_X = p_\p =0$.

Suppose that $\bar{\tau}$ satisfies \eqref{eq:ymhtrees-minus1}.
Since $\deg(\bar{\tau}) \le -1$, $k_\xi $ is even and $k_\p$ is odd,
the only solutions to the equations in Lemma~\ref{lem:counting}
are $(n_\lambda,k_\xi,k_\p) = (1,2,1)$, 
$(n_\lambda,k_\xi,k_\p) = (3,4,3)$,
$(n_\lambda,k_\xi,k_\p) = (3,4,1)$, and $(p_X,p_\p)=(0,0)$ for all these three solutions.

Now suppose that $\bar{\tau}$ satisfies  \eqref{eq:ymhtrees-minus3/2}. 
 Since $\deg(\bar{\tau}) \le -3/2$, $k_\xi $ is odd and $k_\p$ is even,
the only solutions to the equations in Lemma~\ref{lem:counting}
are are $(n_\lambda,k_\xi,k_\p) = (0,1,0)$, 
$(n_\lambda,k_\xi,k_\p) = (2,3,0)$,
$(n_\lambda,k_\xi,k_\p) = (2,3,2)$, and again $(p_X,p_\p)=(0,0)$ for all these three solutions.
\end{proof}
In the next lemma we apply Lemma~\ref{lem:grafting_operators} to study the $\mathcal{U}[\hat{\tau}](\pr{\mathbf{A}})$ introduced in \eqref{e:CU-tau-hat}.
\begin{lemma}\label{lem:hats-linear-h}
For every $\hat\tau\in\{ \hat\tau_{3/2}, \hat\tau_{2}^{\bullet},
\hat\tau_{5/2}^{\bullet} \}$ (see
\eqref{e:hat-tau-2}, \eqref{e:hat-tau-52}),
writing  $\hat{\tau} = \mcb{I}_{\mfu}(\tilde{\tau})$, we have 
\begin{equ}\label{eq:factorization}
\mathcal{U}[\hat{\tau}](\pr{\mathbf{A}}) 
= (\pr{U}^{\ast} \otimes \bullet \pr{U})\tilde{\mathcal{U}}[\hat{\tau}](\pr{h})
\end{equ}
for some $\tilde{\mathcal{U}}[\hat{\tau}](\pr{h})$ which depends on $\pr{\mathbf{A}}$ only through a linear dependence on $\pr{h}$.

Moreover, if $\pr{\mathbf{A}} \in \mcb{A}$ satisfies $\brho(\pr{h_{i}}) = \pr{(\partial_{i}U)U^{-1}}$ then \eqref{eq:factorization} also holds for $\hat{\tau} = \mbX_{i}$. 
\end{lemma}

\begin{proof}
We start with the first statement of the lemma in which case $\hat{\tau}$ is of the form $\hat{\tau} = \mcb{I}_{\mfu}(\tilde{\tau})$.
If 
$\tilde\tau\in \{\<IXi>,\<I[IXiI'Xi]_notriangle>,\<I[IXi^3]>,\<I[I[IXiI'Xi]I'Xi]>,\<I[IXiI'[IXiI'Xi]]> \}$,  since $\tilde\tau \in \mfT^{\YMH}$,
by 
Lemma~\ref{lem:gauge_transformed_ymhtrees}
one has
$\bar\bUpsilon_{\mfz}[\tilde\tau](\pr{\mathbf{A}}) 
= \pr{U}^{\ast}  \bar\bUpsilon^{\YMH}_{\mfz}[\tilde\tau](\pr{\mathbf{A}})$
where, following as in Lemma~\ref{lem:const_upsilon},
 $\bUpsilon_{\mfz}^{\YMH}[\tilde\tau](\pr{\mathbf{A}})$ has no dependence on $\pr{\mathbf{A}}$. Therefore for each $\tilde\tau$ in this set,
 \begin{equs}
\bUpsilon_{\mfu}[\tilde{\tau}](\pr{\mathbf{A}})  
& =
\mcb{I}_{\mfu}
\big[D_{\mfz} (\brho ([\pr{B},\pr{h}])\pr{U})\big] (\bUpsilon_{\mfz} [\tilde\tau](\pr{\mathbf{A}}) )
=
 \brho ([\pr{U}^{\ast} \mcb{I}_{\mfu}\bUpsilon_{\mfz}^{\YMH}[\tilde\tau],\pr{h}])\pr{U}
 \\
 &=
 (\pr{U}^{\ast} \otimes \bullet \pr{U})
(  \brho ([ \mcb{I}_{\mfu}\bUpsilon_{\mfz}^{\YMH}[\tilde\tau],\pr{h}]))
 \end{equs}
which is in the desired form. Moreover,
\[
\bUpsilon_{\mfh'}[\<IXi>](\pr{\mathbf{A}})  =
\mcb{I}_{\mfh'} \big[ D_{\mfz} ([\pr{B},\pr{h}])\big]
(\bUpsilon_{\mfz} [\<Xi>](\pr{\mathbf{A}}) )
=
 [\pr{U}^{\ast} \mcb{I}_{\mfh'} \bUpsilon_{\mfz}^{\YMH} [\<Xi>],
 \pr{h}]
\]
where $ \bUpsilon_{\mfz}^{\YMH} [\<Xi>]$ 
has no dependence on $\pr{\mathbf{A}}$.
Therefore for 
$\tilde\tau\in \{\<IXi>, \<I[IXiI'Xi]_notriangle>\}$,    
\begin{equs}
\bUpsilon_{\mfu}[\tilde{\tau}\mcb{I}_{\mfh'}(\<IXi>)   ](\pr{\mathbf{A}})  
& =
\mcb{I}_{\mfu}
\big[D_{\mfz} D_{\mfh'}(\brho ([\pr{B},\pr{h}])\pr{U})\big] 
(\bUpsilon_{\mfz} [\tilde\tau](\pr{\mathbf{A}}) , \bUpsilon_{\mfh'} [\<IXi>](\pr{\mathbf{A}}) )
\\
&=
\mcb{I}_{\mfu}
\Big( \brho \big([\pr{U}^{\ast} \bUpsilon_{\mfz}^{\YMH}[\tilde\tau],
[\pr{U}^{\ast} \mcb{I}_{\mfh'} \bUpsilon_{\mfz}^{\YMH}[\<Xi>],\pr{h} ]  ]\big)\pr{U}\Big)
 \\
 &=
 (\pr{U}^{\ast} \otimes \bullet \pr{U})
\mcb{I}_{\mfu}
\Big( \brho \big([\bUpsilon_{\mfz}^{\YMH}[\tilde\tau],
[\mcb{I}_{\mfh'} \bUpsilon_{\mfz}^{\YMH}[\<Xi>],\pr{h} ]  ]\big)\Big)
 \end{equs}
 which is also in the desired form.
 Finally the claim for
$\mcb{I}_{\mfu} \big( \<IXi> \, \mcb{I}_{\mfh'} 
\big( \<IXi> \mcb{I}_{\mfh'} (\<IXi>) \big) \big)$ can be checked
in an analogous way.

To prove the second statement of the lemma, we note that $\pr{U}^{\ast} \restriction_{\mcb{T}[\mbX_{i}]} = \id_{\R}$ so 
we have $\mathcal{U}[\mbX_i](\pr{\mathbf{A}}) = (\pr{U}^{\ast} \otimes \bullet \pr{U})  (\pr{\partial_{i} U})\pr{U}^{-1}\mbX_{i} =  (\pr{U}^{\ast} \otimes \bullet \pr{U})   \brho(\pr{h_{i}}) \mbX_{i}$ which is of the desired form. 
\end{proof}

To state the next lemma, for every $g \in G$, one can define a natural action of $\brho(g)$
on $\mcb{T}^{\ast}$ as in \cite[Remark~5.19]{CCHS2d}
or as above Lemma~\ref{lem:gauge_transformed_ymhtrees},
with the operator assignment $L = (L_{\mft})_{\mft \in \Lab}$ with $L_{\mft} = \brho(g)$ for $\mft \in \Lab_{-}$ and $L_{\mft} = \id_{V^{\ast}_{\mft}}$ otherwise.
The following lemma states that under certain conditions of $\ell \in \mcb{T}^{\ast}$
(which in particular will be satisfied by $\ell^\eps_{\BPHZ}$),
the equation for $Y$ only requires  renormalisation that is linear in $Y$ and $h$:

\begin{lemma}\label{lem:renorm_linear_in_Y_and_h}
Fix $\ell \in \mcb{T}^{\ast}$. Suppose that $\brho(g) \ell = \ell$ for every $g \in G$,
and that $\ell[\tau] = 0$ for every 
$ \tau \notin \mfT^{\even}_{-}$.

Let $C_{\ell,\tau} \in L(E,E)$ be the same map\footnote{Here we also use the symbol $\ell$ to refer to its restriction to $\mcb{T}[\tilde{\mfT}]^{\ast}$.} that appears in Proposition~\ref{prop:mass_term}.  

Then there exists $C_{\ell} \in L(\mfg^{3},E)$ such that, for any $\pr{\mathbf{A}} \in \bar{\mcb{A}}$, 
\begin{equ}\label{eq:renorm_linear_in_Y_and_h}
(\ell \otimes \id_{W_{\mfz}}) 
\bar{\bUpsilon}^{F}_{\mfz}(\pr{\mathbf{A}})
=
\sum_{\tau \in \mfT^{\even}_{-} \cap \mfT^{\YMH}} \!\!\!\! C_{\ell,\tau} \pr{Y} 
\; + \;
C_{\ell} \pr{h}\;.
\end{equ}
\end{lemma}
\begin{proof}
We first note that we have
\begin{equs}[e:ell-Ups-z]
(\ell \otimes \id_{W_{\mfz}}) 
\bar{\bUpsilon}^{F}_{\mfz}(\pr{\mathbf{A}})
&= 
\sum_{\tau \in \mfT^{\YMH} \cap \mfT_{-}^{\even}}
(\ell \otimes \id_{W_{\mfz}}) 
\bar{\bUpsilon}^{F}_{\mfz}[\tau](\pr{\mathbf{A}})\\
{}&\enskip + 
\sum_{\bar{\tau} \in \mfT^{\YMH,h}}
\sum_{\hat{\tau} \in \mfT_{g,h}}
\sum_{\tau \in \mfT[\hat{\tau} ; \bar{\tau}]}
(\ell \otimes \id_{W_{\mfz}}) \bar\bUpsilon^{F}_{\mfz}[\tau](\pr{\mathbf{A}})\;.	
\end{equs}
Here we define $\mfT^{\YMH,h} \subset \mfT^{\YMH}$ to consist of all of those $\bar{\tau}$ as in \eqref{eq:ymhtrees-minus1} or \eqref{eq:ymhtrees-minus3/2}, or of the form $\<IXiI'Xi_notriangle>$ or $\<Xi>$. We also define $\mfT_{g,h} \subset \mfT_{g}$ to be given by all the symbols that we graft in Lemma~\ref{lemma:ymh_trees_generate}, namely it consists of $\mbX_{j}$ for $j \in \{1,2,3\}$ and $\hat{\tau}_{3/2}= \mcb{I}_{\mfu} (\<IXi>)$, along with all symbols appearing in \eqref{e:hat-tau-2} and \eqref{e:hat-tau-52}. 

By Proposition~\ref{prop:mass_term} and Lemma~\ref{lem:gauge_transformed_ymhtrees},
one has for each $\tau \in \mfT^{\YMH} \cap \mfT_{-}^{\even}$,
\[
(\ell \otimes \id_{W_{\mfz}}) 
\bar{\bUpsilon}^{F}_{\mfz}[\tau](\pr{\mathbf{A}})
=
C_{\pr{U}\ell,\tau}\pr{Y}
=
C_{\ell,\tau}\pr{Y}\;.
\]
For the second contribution on the right-hand side of \eqref{e:ell-Ups-z}, by Lemma~\ref{lem:Ups-graft} and linearity,  
\[
\sum_{\bar{\tau} \in \mfT^{\YMH,h}}
\sum_{\hat{\tau} \in \mfT_{g,h}}
\sum_{\tau \in \mfT[\hat{\tau} ; \bar{\tau}]}
\bar\bUpsilon^{F}_{\mfz}[\tau](\pr{\mathbf{A}}) 
=
\sum_{\bar{\tau} \in \mfT^{\YMH,h}}
\bar{\bUpsilon}_{\mfz}^{\bar{\tau},\ \mcb{X}(\pr{\mathbf{A}})}(\pr{\mathbf{A}})
\]
where $\mcb{X}(\pr{\mathbf{A}}) 
=
\sum_{\hat{\tau} \in \mfT_{g,h}}
\mathcal{U}[\hat{\tau}](\pr{\mathbf{A}})$. 
Now, by Lemma~\ref{lem:hats-linear-h} we can write $\mcb{X}(\pr{\mathbf{A}}) = (\pr{U}^{\ast} \otimes \bullet \pr{U}) \tilde{\mcb{X}}(\pr{h})$
for some $\tilde{\mcb{X}}$ which depends on $\pr{\mathbf{A}}$ only through a linear dependence on $\pr{h}$. 

By Lemma~\ref{lem:grafting_operators},
the second term on the right-hand side of \eqref{e:ell-Ups-z} then equals to
\begin{equ}
(\ell \otimes \id_{W_{\mfz}})  \pr{U}^{\ast}  \big( \mathcal{L}_{g}  \tilde{\mcb{X}}(\pr{h}) \big) \sum_{\bar{\tau} \in \mfT^{\YMH,h}} \!\!\!\! \bUpsilon_{\mfz}^{\YMH}[\bar{\tau}]
=
( \pr{U}\ell \otimes \id_{W_{\mfz}})  \big(\mathcal{L}_{g}   \tilde{\mcb{X}}(\pr{h}) \big) \sum_{\bar{\tau} \in \mfT^{\YMH,h}} \!\!\!\! \bUpsilon_{\mfz}^{\YMH}[\bar{\tau}]\;.
\end{equ}
Using the fact that $\pr{U}$ is in the image of $\brho(\cdot)$,
one has $ \pr{U}\ell  = \ell$.
Recall that $\bUpsilon_{\mfz}^{\YMH}[\bar{\tau}]$ has no dependence on $\pr{\mathbf{A}}$ thanks to Lemma~\ref{lem:const_upsilon}.
Therefore the last line is of the form $C_{\ell} \pr{h}$ as desired. 
\end{proof}

We now turn to the renormalisation of the $(\bar{X}, \bar{U}, \bar{h})$ system. 
There are slight differences between our book-keeping for $(Y,U,h)$ versus $(\bar{X}, \bar{U}, \bar{h})$: (i) in the latter system we use the label $\mfl$ instead of the label $\bar{\mfl}$ to represent the noise; and (ii) the RHS of the equation for $\bar{X}$ is broken as a sum of two terms \dash the first term being $\bar{U} \xi$ which is labelled by $\mfm$ and  the second term which contains everything else being labelled then by $\mfz$.  

We write $\mfT^{F}$ for the collection of all $\tau \in \mfT$ that contain no instance of $\mfl$ or $\mfm$. We write $\mfT^{\bar{F}}$ for the collection of all $\tau \in \mfT$ which  (i) contain no instance of $\bar{\mfl}$ and (ii) satisfy the constraint that in any expression of the form $\mcb{I}_{o}(\tau \mcb{I}_{\mfl}(\bone))$ we must have $o \in \{\mfm\} \times \N^{d+1}$. 
Trees generated by the $F$ system belong to $\mfT^{F}$ and trees generated by the $\bar{F}$ system belong to $\mfT^{\bar{F}}$. In particular, $\tau  \in  \mfT \setminus \mfT^{F} \Rightarrow \bar\bUpsilon^{F}[\tau] = 0$ and $\tau \in \mfT \setminus \mfT^{\bar{F}} \Rightarrow \bUpsilon^{\bar{F}}[\tau] = 0$.

Moreover, there is a natural bijection  $\theta\colon \mfT^{\bar{F}} \rightarrow \mfT^{F}$ obtained by replacing any edge of type $(\mfm,p)$ with $(\mfz,p)$ and noise edge of type $\mfl$ with one of type $\bar{\mfl}$.
Moreover, there is an induced isomorphism $\Theta: \mcb{T}[\mfT^{\bar{F}}] \rightarrow \mcb{T}[\mfT^{F}]$ which maps $\mcb{T}[\tau]$ into $\mcb{T}[\theta(\tau)]$. \label{pageref:Theta}
We then have the following lemma.
\begin{lemma}\label{lemma:Y_to_barX}
Given any $\tau \in \mfT^{\bar{F}}$, one has, for any $\mft \in \{\mfu,\mfh,\mfh'\}$, 
\begin{equs}
\bUpsilon^{F}_{\mft}[\theta(\tau)](\pr{\mathbf{A}})
&=
(\Theta \otimes \id_{W_{\mft}})
\bUpsilon^{\bar{F}}_{\mft}[\tau](\pr{\mathbf{A}})
\;,\\
\bUpsilon^{F}_{\mfz}[\theta(\tau)](\pr{\mathbf{A}})
&=
(\Theta \otimes \id_{W_{\mfz}})
\big(
\bUpsilon^{\bar{F}}_{\mfz}[\tau](\pr{\mathbf{A}})
+
\bUpsilon^{\bar{F}}_{\mfm}[\tau](\pr{\mathbf{A}})
\big)
\;.
\end{equs}

\end{lemma}
\begin{proof}
This is straightforward to prove by appealing to the inductive formulae for $\bUpsilon^{\bar{F}}$ and $\bUpsilon^{F}$. 
\end{proof}

We can now give the proof of the main result of this subsection. 

\begin{proof}[of Proposition~\ref{prop:renorm_g_eqn}]\label{proof:renorm_g_eqn}
Note that $\ell = \ell^{\delta,\eps}_{\BPHZ}$ satisfies the conditions of Lemma~\ref{lem:renorm_linear_in_Y_and_h} (using Lemma~\ref{lem:parity-bphz-gsym} for the parity constraint and arguing similarly for the invariance of $\ell$ under $\brho(\cdot)$). 
In particular, we see that \eqref{eq:no_u_h_renorm}  follows from Lemma~\ref{lem:hU-no-renor};
and we also have \eqref{eq:renorm_linear_in_Y_and_h} 
by Lemma~\ref{lem:renorm_linear_in_Y_and_h}.

We now turn to proving the first line of \eqref{eq:renorm_of_g_system}. 
We will now argue in a similar way as Lemma~\ref{lem:LX-diag} to show that the right-hand side of \eqref{eq:renorm_linear_in_Y_and_h} has a block structure of the form
\begin{equ}\label{eq:block_structure}
\Big( \bigoplus_{i=1}^{3} \check{C}^{\delta,\eps}_{\YM} \pr{B_i}+ \check{C}^{\delta,\eps}_{\Gauge} \pr{h_i}
 \Big) \oplus \check{C}^{\delta,\eps}_{\Higgs} \pr{\Psi}\;
\end{equ}
with $\check{C}^{\delta,\eps}_{\YM}, \check{C}^{\delta,\eps}_{\Gauge} \in L_{G}(\mfg,\mfg)$, and $\check{C}^{\delta,\eps}_{\Higgs} \in L_{G}(\higgsvec,\higgsvec)$. 

Consider $\mathbf{T} =(T,\sigma,r,O)\in \mathrm{Tran}$ chosen by fixing $i \in \{1,2,3\}$, 
and then choosing $T$ to act on $W_{\mfz}\simeq  W_{\mfm} \simeq W_{\bar{\mfl}} = \mfg_1 \oplus \mfg_2 \oplus \mfg_3\oplus \higgsvec$, on $W_{\mfh} \simeq W_{\mfh'} = \mfg_1 \oplus \mfg_2 \oplus \mfg_3$
by flipping the sign of the $\mfg_i$ component, 
and act as the identity on $W_{\mfu}$. 
We set $r$ which flips the $i$-th spatial coordinate,
and set $\sigma=\id$. 
For $O = (O_{\mft})_{\mft \in \Lab_{+}}$, we have $O_{\mft} = \id$ except for $\mft = \mfh'$ where $\mathcal{K}_{\mfh'} \simeq \R^{3}$ and we set $O_{\mfh'}$ to flip the sign of the $i$-th spatial component.  

One can check that that our nonlinearity 
 $F(\pr{\mathbf{A}})$ 
   is then $\mathbf{T}$-covariant, and the kernels and noises are $\mathbf{T}$-invariant.
Considering such transformations for every $i \in \{1,2,3\}$ and applying Proposition~\ref{prop:renormalisation is covariant} then implies that 
 the linear renormalisation must be appropriately block diagonal. 
  
 We then show that the blocks associated to different spatial indices are the same. Again as in Lemma~\ref{lem:LX-diag}, we consider $\mathbf{T} =(T,\sigma,r,O)\in \mathrm{Tran}$ chosen by fixing $1 \le i < j \le 3$ 
and then choosing $T$ to act on $W_{\mfz} \simeq W_{\mfm} \simeq W_{\bar{\mfl}} = \mfg_1 \oplus \mfg_2 \oplus \mfg_3\oplus \higgsvec$ and $W_{\mfh} \simeq W_{\mfh'} = \mfg_1 \oplus \mfg_2 \oplus \mfg_3$ by swapping 
 the $\mfg_{i}$ and  $\mfg_j$ components, and act as an identity on $W_{\mfu}$.
We also set $r$ given by the identity, and  $\sigma$ exchanges $i$ and $j$. 
For $O = (O_{\mft})_{\mft \in \Lab_{+}}$, we have $O_{\mft} = \id_{\mathcal{K}_{\mft}}$ except for $\mft = \mfh'$ where $\mathcal{K}_{\mfh'} \simeq \R^{3}$ and we set $O_{\mfh'}$ to swap the $i$ and $j$ spatial components.
One can check that  our nonlinearity 
 $F(\pr{\mathbf{A}})$ 
   is then $\mathbf{T}$-covariant, and kernels and noises are $\mathbf{T}$-invariant.
Considering such transformations for every $1 \le i < j \le 3$ and applying Proposition~\ref{prop:renormalisation is covariant}, we conclude that the blocks indexed by different spatial indices must be identical. 

Finally, to show that the block operators all commute with the appropriate actions of $G$, we again argue using Proposition~\ref{prop:renormalisation is covariant}, fixing $g \in G$ and looking at $\mathbf{T} \in \mathrm{Tran}$
where $T$ acts on $W_{\mfz}$ and $W_{\mfm}$ as $\brho(g)$, on $W_{\mfu}$ by right composition\slash multiplication with $\brho(g)$, 
 and on $W_{\mfh}$ and $W_{\mfh'}$ by $\Ad_g$. We set $r$, $\sigma$, and $O$ to be the  appropriate identity operators. 
 
To finish justifying the first line of \eqref{eq:renorm_of_g_system}, we note that our choice of kernel and noise assignments guarantee that, for any $\tau \in \mfT^{\YMH}$ one has $\ell_{\BPHZ}^{\delta,\eps}[\tau] = \ell_{\BPHZ}^{\eps}[\tau]$ where on the right-hand side we are referring to the character in \eqref{e:counterterms-A}  \dash this justifies replacing $\check{C}^{\delta,\eps}_{\YM}$ and $\check{C}^{\delta,\eps}_{\Higgs}$ with $C^{\eps}_{\YM}$ and $C^{\eps}_{\Higgs}$.
(Note that 
$C^{\eps}_{\YM}$ and $C^{\eps}_{\Higgs}$ are the same maps as that in Proposition~\ref{prop:mass_term}, which follows from Lemma~\ref{lem:renorm_linear_in_Y_and_h}.)
More generally, for any $\tau \in \mfT$ with $\bUpsilon_{\mfz}^{F}[\tau] \not = 0$, $\ell_{\BPHZ}^{\delta,\eps}[\tau]$ does not depend on $\delta > 0$,
therefore $\check{C}^{\delta,\eps}_{\Gauge}$ does not depend on $\delta > 0$ either,
and we can just denote this map by $C^{\eps}_{\Gauge}$
as in  the first line of \eqref{eq:renorm_of_g_system}.

Repeating the argument above combined with Lemma~\ref{lemma:Y_to_barX} gives us the second line of \eqref{eq:renorm_of_g_system}, although in this case the operators should be allowed to depend on $\delta$.
It then only remains to prove \eqref{eq:delta_to_0}. 

For the first two statements we note that, if $\PPi_{\delta,\eps}$ is the canonical lift of the kernel assignment
 $K^{(\eps)}$
and noise assignment $\zeta^{\delta,\eps}$ then for any $\tau \in \mfT^{\bar{F}}$ with $\theta(\tau) \in \mfT^{\YMH}$, one has    
\begin{equ}\label{eq:work_delta_to_0}
\lim_{\delta \downarrow 0}
\bar{\PPi}_{\delta,\eps}[\tau] 
=
\bar{\PPi}_{\delta,\eps}[\theta(\tau)] \circ \Theta[\tau]\;,
\end{equ}
 where $\bar{\PPi}_{\delta,\eps} = \E[\PPi_{\delta,\eps}(0)]$ and the equality above is between elements of $\mcb{T}[\tau]^{\ast}$. Here $\theta$ and $\Theta$ are as in the paragraph before \eqref{lemma:Y_to_barX}. 
 One can verify \eqref{eq:work_delta_to_0} with a straightforward computation. The key observation is that the condition  $\theta(\tau) \in \mfT^{\YMH}$ means that the only way any instance of $\mfl$ appears in $\tau$ is as $\mcb{I}_{(\mfm,p)}(\mfl)$ for some $p \in \N^{d+1}$, then to argue \eqref{eq:work_delta_to_0} one just observes that
 \[
 \lim_{\delta \downarrow 0}
 D^{p}K^{\eps} \ast \xi^{\delta} = D^{p}K \ast \xi^{\eps}\;.
 \]
 Heuristically, as $\delta \downarrow 0$, ``$\mcb{I}_{(\mfm,p)}(\mfl)$ is the same as $\mcb{I}_{(\mfz,p)}(\bar{\mfl})$''.
 Since the map $\Theta$ interacts well with coproducts, it is easy to see that \eqref{eq:work_delta_to_0} holds if we replace $\bar{\PPi}_{\delta,\eps}[\cdot]$ with $\ell^{\delta,\eps}_{\BPHZ}[\cdot]$. 
 
To prove the last statement of \eqref{eq:delta_to_0}, it suffices to show that for any $\eps > 0$ and $\tau \in \mfT^{\bar{F}}$ the limit 
 \begin{equ}\label{eq:bphz_char_delta_0}
 \ell^{0,\eps}_{\BPHZ}[\tau]
 =
 \lim_{\delta \downarrow 0}\ell^{\delta,\eps}_{\BPHZ}[\tau]
\end{equ}
 exists. 
 Note that  $\tau \in \mfT^{\bar{F}}$ implies that every noise in $\tau$ not incident to the root is incident to an $\eps$-regularised kernel. 
This then means that
$\E[ \PPi_{\delta,\eps}[\tau](\cdot)]$ remains smooth in the limit $\delta \downarrow 0$. 
Indeed, if $\tau$ doesn't contain an instance of the noise $\mfl$ at the root then $\PPi_{\delta,\eps}[\tau](\cdot)$ itself remains smooth in the $\delta \downarrow 0$ limit.
On the other hand, if $\tau$ does contain an instance of $\mfl$ at the root, then one can write $\E[\PPi_{\delta,\eps}[\tau](\cdot)]$ as a sum over Wick contractions. 
Each term in the sum can be written a convolution of kernels where each integration vertex is incident to at most one $\delta$-regularized kernel (the covariance of $\xi^{\delta}$ while the rest of the kernels are of the form $D^{p} K^{\eps}$ and smooth.
It is then straightforward to argue that, as $\delta \downarrow 0$, this convolution converges to one written completely in terms of the smooth kernels of the form $D^{p} K^{\eps}$. 
We then obtain \eqref{eq:bphz_char_delta_0} by expanding its RHS in terms of the negative twisted antipode and applying the above observations. 

Finally, to show \eqref{e:Cs-sig}, it is enough to recall \eqref{e:C-YM-Higgs-sig},
and observe that the trees $\tau$ in Lemma~\ref{lemma:ymh_trees_generate}
all have $2$ or $4$ noises (in particular see Remark~\ref{rem:bar-trees-cases23}).
Moreover, since the block structure of the operators in \eqref{eq:renorm_of_g_system} holds for \emph{every} choice of $\sig$,  all the blocks appearing must decompose
into orders $\sig^2$ and $\sig^4$ in the same way.
\end{proof}

\subsection{Solution theory}
\label{sec:sol-multi}

We start by trying to pose the analytic fixed point problems
associated to  \eqref{eq: B system}
 for the $E \oplus \mfg^3 \oplus (L(\mfg,\mfg)\oplus L(\higgsvec,\higgsvec))$-valued modelled distributions 
 $(\mathcal{Y}, \mathcal{H},\mathcal{U})$.

As before, for each $\mft \in \Lab_{+}$ we write 
$\mcb{K}_{\mft}$ for the corresponding 
 abstract integration operator on modelled distributions realising convolution with $K_{\mft}^{(\eps)}$ in \eqref{e:K-Keps-assign}.
 We also sometimes write $\mcb{K}$ and $\bar{\mcb{K}}$ for
 $\mcb{K}_\mfz$ and $\mcb{K}_\mfm$ respectively.  \label{bar-mcbK}
Moreover, we write 
$\CG_\mft \eqdef \mcb{K}_{\mft} + R \mathcal{R}$
 for $\mft \in \{\mfz,\mfh,\mfu\}$ 
 and
 $\CG_{\mfh'} \eqdef \mcb{K}_{\mfh'} + \nabla R \mathcal{R}$, where
 $R$ represents convolution with $G-K$ as in Section~\ref{sec:solution theory additive}. 
Recall the convention \eqref{e:XdX-X3}.

In addition to the difficulties described in the beginning of Section~\ref{sec:solution theory additive}, we now have a multiplicative noise term $U\xi$ which requires additional treatments.
First, note that the worst term in $\CU  \boldsymbol{\bar\Xi}$ is of degree  $-\frac52-\kappa$ which is below $-2$, so the general integration theorem \cite[Prop.~6.16]{Hairer14}\footnote{Recall the conditions $\eta\wedge \alpha>-2$ and $\bar\eta= \eta\wedge\alpha+\beta$ in \cite[Prop.~6.16]{Hairer14}.} 
would not  apply. 
The second problem is that, 
 even if we could apply the integration operator on it, we would have $\mcb{K}_{\mfz} \big( \CU \boldsymbol{\bar{\Xi}}\big) \in \cD^{\gamma,\eta}$ for  $\eta=-\frac12-2\kappa$,
but as in Section~\ref{sec:solution theory additive} having a closed fixed point problem requires us to work in
$\hat{\cD}^{\gamma,\hat\beta}$ 
with $\hat\beta > -1/2$ due to the term $Y\partial Y$.


To handle these difficulties
we start by decomposing
$\CU = \CP U_{0} + \hat{\CU}$ and solve for $\hat\CU\in \hat\cD^{\gamma,\rho}$ for suitable exponents $\gamma,\rho$. 
The component $\hat\CU$ has improved behaviour near $t=0$ 
since we have subtracted the initial condition $U_0$,
so that we are in the scope of 
Theorem~\ref{thm:integration} for the term $\hat{\CU}\boldsymbol{\bar{\Xi}}$.
On the other hand,  $\CP U_{0} \bar{\bXi}$ can be handled with Lemma~\ref{lem:Schauder-input},
since, thanks to Lemma~\ref{lem:omegas_converge} below,
\begin{equ}\label{eq:omega_0_def}
\omega_0\eqdef  \sig\CP U_{0}\xi^{\eps}
\end{equ}
probabilistically converges as $\eps \downarrow 0$ in $\CC^{-\frac52-\kappa}$
on the entire space-time (not just $t>0$).
For $\rho \in (1/2,1)$, 
we have   
$\CP U_0 \in  \cD^{\infty,\rho}_{0}$,
and since $\bar{\boldsymbol{\Xi}}\in \cD^{\infty,\infty}_{-\frac52-\kappa}$
one has $\CP U_0 \bar{\boldsymbol{\Xi}}\in \cD^{\infty,\rho-\frac52-\kappa}_{-\frac52-\kappa}$. 
So by Lemma~\ref{lem:Schauder-input}
\begin{equ}[e:bPsi_U0]
 \bPsi^{U_0} \eqdef  \mcb{K}_{\mfz}^{\omega_0} (\CP U_0 \bar{\boldsymbol{\Xi}})
 \in \cD^{\infty,-\frac12-2\kappa}_{-\frac12-2\kappa}
\end{equ}
(the compatibility condition of Lemma~\ref{lem:Schauder-input} will be checked in Lemma~\ref{lem:compatible} below).
We write
\begin{equ}
\Psi^{U_{0}}_{\eps} =  K \ast \omega_0 = \CR  \bPsi^{U_0} \;.
\end{equ}
We write $\omega_0$ above instead of $\sig \CP U_{0}\xi^{\eps}$ 
not only as shorthand, but also because later in Lemma~\ref{lem:fixedptpblmclose} we will
consider $\omega_0$ as part of a given deterministic input to a fixed point problem
(which is necessary in the $\eps\downarrow0$ limit).

Furthermore, 
analogously to Section~\ref{sec:solution theory additive}, we will linearise $\CY$ around  $\CP Y_{0}$ and $\bPsi^{U_0}$,
 and then apply Lemma~\ref{lem:Schauder-input} with appropriate ``input'' distributions, in a similar way as we did for the last line of \eqref{e:fix-pt-X1}. 

Putting things together, for fixed initial data  
\begin{equ}\label{eq:initial_data_for_gauge_sys}
(Y_0,U_0,h_0) 
\in
\Omega^\initial\eqdef
\state
\times 
\tilde{\mfG}^{\rho}\;, 
\end{equ}
where we recall that $\state \subset \CC^{\eta}$
for $\eta<-1/2$ and $\rho \in (1/2,1)$, we decompose 
\begin{equ}[e:def-tildeY]
\CY = \CP Y_0 
+ \bPsi^{U_0}
+ \hat \CY\;,
\qquad
\CU = \CP U_0 + \hat \CU\;,
\end{equ}
where we want
\begin{equ}[e:hatCYhatCU]
\hat \CY \in \hat \cD^{\gamma_\mfz, \hat\beta}_{-\frac12-2\kappa}\;,\quad
\hat \CU \in \hat\cD^{\gamma_\mfu, \rho}_0
\qquad
\mbox{for }\gamma_\mfz>\frac32+\kappa\;, \; \gamma_\mfu>\frac52+\kappa\;,
\end{equ}
along with $\CH \in \cD^{\gamma_{\mfh},\eta_{\mfh}}_{\alpha_{\mfh}}$ for some appropriate exponents $\gamma_{\mfh},\eta_{\mfh},\alpha_{\mfh}$ (these exponents will be specified in \eqref{eq:fixed_pt_space}; note that solving $\CH$ only requires a ``standard'' modelled distribution space as in \cite{Hairer14}), and solve the fixed point problem  
\begin{equs}
 \hat \CY
&=
\CG_{\mfz}^{\omega_1}  
(\CP Y_0 \partial\bPsi^{U_0})
+ \CG_{\mfz}^{\omega_2}  
(  \bPsi^{U_0} \partial \CP Y_0 )
+
\CG_{\mfz}^{ \omega_3}   
( \bPsi^{U_0}\partial \bPsi^{U_0})
\label{e:fix-pt-multi-Y}
\\
&\quad
+\CG_{\mfz} \mathbf{1}_{+}\Big( R_Q
  + \CY^3+ \mathring{C}_{\mfz} \CY +\mathring C_\mfh\CH
   + \hat{\CU} \bar{\boldsymbol{\Xi}} +\CP Y_0 \partial \CP Y_0\Big)
  +
  R  ( \omega_0) \;,	
\end{equs}
and
\begin{equs}[e:fix-pt-HU]
\mathcal{H}
&=
\CG_{\mfh}
\mathbf{1}_{+}
\Big(
 [\mathcal{H}_j,\p_j \mathcal{H}] + [[\mathcal{B}_j, \mathcal{H}_j],\mathcal{H}] 
 \Big) 
  +\CG_{\mfh'}
\mathbf{1}_{+}  [\mathcal{B}_j, \mathcal{H}_j]+ \CP h_0\;,
  \\
 \hat{\CU} &=
\CG_{\mfu}
\mathbf{1}_{+} \Big( - [\mathcal{H}_j,[\mathcal{H}_j, \cdot]] \circ \mathcal{U}
+ [[\mathcal{B}_j, \mathcal{H}_j],\cdot] \circ \mathcal{U} \Big) \;,
\qquad \mathcal{B}=\CY\restriction_{\mfg^3}\;.
 \end{equs}
Here we have defined
\begin{equ}[e:def-RQ]
R_Q
\eqdef 
\CP Y_0 \partial \hat{\CY}
+
 \hat{\CY}  \partial \CP Y_0 
+
 \hat{\CY} \partial \hat{\CY}
 +
 \bPsi^{U_0} \partial \hat{\CY} 
 +
 \hat{\CY} \partial \bPsi^{U_0}\;.
\end{equ}
The 
$E$-valued space-time distributions $\omega_1,\omega_2,\omega_3$ are compatible with the respective modelled distributions and are part of the input to the fixed point problem.
We will later take
\begin{equs}[e:omega123]
\omega_1
&\eqdef \CP Y_0  \partial \Psi^{U_{0}}_{\eps}\;,
\\
\omega_2
&\eqdef \Psi^{U_{0}}_{\eps}   \partial \CP Y_0\;,
\\
\omega_3
&\eqdef
\Psi^{U_{0}}_{\eps} 
\partial \Psi^{U_{0}}_{\eps}  - \sig^2(c_1^\eps +c_2^\eps) \CP U_0\partial \CP U_0\;,
\end{equs}
where, writing $\moll_\eps^{(2)} \eqdef \moll^\eps*\moll^\eps$,
\begin{equs}[e:c1c2]
c_1^\eps & \eqdef \int_{\R^2 \times \T^6}  
y_i \,K_{t-s}(-y)\partial_i  K_{t-\bar s}(-\bar y)  \,
\moll_\eps^{(2)} (s-\bar s,y-\bar y) \,\mrd s\mrd \bar s\mrd y \mrd \bar y\;,
\\
c_2^\eps &
 \eqdef \int_{\R^2 \times \T^6} 
\bar y_i \, K_{t-s}(-y)\partial_i  K_{t-\bar s}(-\bar y)  \,
\moll_\eps^{(2)} (s-\bar s,y-\bar y) \,\mrd s\mrd \bar s\mrd y \mrd \bar y\;.
\end{equs}
Here the values of $c_1^\eps, c_2^\eps$ clearly do not depend on $i\in \{1,2,3\}$.
Note that in \eqref{e:omega123} we use the convention 
\eqref{e:XdX-X3} except for the term $\CP U_0\partial \CP U_0$
for which that convention does not apply since $ \CP U_0$ is not $E$-valued. 
To define $\CP U_0\partial \CP U_0$, for any 
\begin{equs}[2]
u &= u_\mfg \oplus u_{\higgsvec}  & \quad &\in L(\mfg,\mfg)\oplus L(\higgsvec,\higgsvec) \;,
\\
\partial v &= (\partial_i v_\mfg)_{i=1}^3 \oplus (\partial_i u_{\higgsvec})_{i=1}^3 
& \quad &\in L(\mfg,\mfg)^3\oplus L(\higgsvec,\higgsvec)^3\;,
\end{equs}
we set $u\partial v \in E=\mfg^3\oplus \higgsvec$ as being given by
\begin{equs}
u\partial v\restriction_{\mfg_i} &\eqdef -[u_\mfg e_\alpha,\partial_i v_\mfg e_\alpha]
- \mathbf{B} (\partial_i v_{\higgsvec} e_a\otimes u_{\higgsvec} e_a) \quad \mbox{for } i\in \{1,2,3\} \;,
\\
u\partial v\restriction_{\higgsvec} &\eqdef 0 \;,  	\label{e:udv}
\end{equs}
where $(e_\alpha)_\alpha$ is an orthonormal basis of $\mfg$ and $(e_a)_a$ 
is an orthonormal basis of $\higgsvec$. Here $\alpha$ and $a$ are being summed as usual.

The reason to define the above space-time distributions
will become clear in Lemmas~\ref{lem:omegas_converge} and~\ref{lem:compatible}.
We will write $\omega^\eps_\ell$ for $\ell\in\{0,1,2,3\}$
when we wish to make the dependence of $\omega_\ell$ on $\eps$ explicit.

The abstract fixed point problem for the $(\bar X, \bar U,\bar h)$ system 
\eqref{eq: bar a system}
is given by~\eqref{e:fix-pt-HU} with $\CU=\CP U_0 + \hat\CU$ as before but with
$\CY$ in~\eqref{e:def-tildeY} replaced by
\begin{equ}[eq:CY_bar_equation]
\CY = \CM
+ \bar\bPsi^{U_0}
+ \hat \CY\;,
\end{equ}
where
\begin{equ}\label{eq:CM_def}
\CM = \CP Y_0 + \CW\;,
\end{equ}
and $\CW\in\cD^{\frac74,-\frac12-\kappa}$ is the canonical lift of a smooth function on $(0,\infty)$,
and~\eqref{e:fix-pt-multi-Y} replaced by
\begin{equs}
 \hat \CY
 &=
\CG_{\mfz}^{\bar\omega_1}  
(\CM \partial\bar\bPsi_{U_0})
+ \CG_{\mfz}^{\bar\omega_2}  
(  \bar\bPsi^{U_0} \partial \CM )
+\CG_{\mfz}^{\bar\omega_3}   
(\bar\bPsi^{U_0} \partial \bar\bPsi_{U_0}) \label{e:fix-pt-multi-Xbar}
\\
&\quad
+ \CG_\mfz^{\bar\omega_4}(\CP Y_0 \partial\CW + \CW \partial \CP Y_0 + \CW\partial \CW)
\\
&\quad
+\CG_{\mfz} \mathbf{1}_{+}\Big( \bar R_Q
  + \CY^3+ \mathring{C}_{\mfz} \CY +\mathring C_\mfh\CH +\CP Y_0 \partial \CP Y_0 \Big)
+\CG_{\mfm} \mathbf{1}_{+}
  \big( \hat{\CU} \boldsymbol{\Xi} \big) +  R  ( \bar\omega_0) \;,
\end{equs}
where $\CG_{\mfm}\eqdef \bar{\mcb{K}} + \bar{R} \mathcal{R}$ with
$\bar{R}$  defined just like $R$ but with $G - K$ replaced by $G^{\eps} - K^{\eps}$.
Here  we define (recall from \eqref{e:K-Keps-assign}
that the kernel assigned to $\mfm$ is $K^\eps$)
\begin{equs}[e:bbPsi_U0]
\bar\bPsi^{U_0}  &\eqdef \mcb{K}_{\mfm}^{\bar\omega_0} (\CP U_0 \boldsymbol{\Xi})
 \in \cD^{\infty,-\frac12-2\kappa}_{-\frac12-2\kappa} \;,
\end{equs}
$\bar R_Q$ is defined in the same way as $R_Q$ in~\eqref{e:def-RQ}
with $\bPsi^{U_0} $ replaced by $\bar\bPsi^{U_0} $
and $\CP Y_0$ replaced by $\CM$,
and $\bar\omega_i$ are distributions compatible with the respective modelled distributions.
The modelled distribution $\CW$ is part of the input for the fixed point problem
and we will later take
\begin{equ}\label{eq:CW_def}
\CW = G*(\bone_{t\geq0} \moll^\eps*(\sig \xi\bone_{t<0}))\;,
\end{equ}
understood as $\CW=0$ if $\eps=0$.
The role of $\CW$ is to encode the behaviour of $\xi$ on negative times so as to obtain the correct reconstruction.
We will also later take
\begin{equs}[e:bar_omega0123]
\bar\omega_0 &\eqdef \sig\CP U_0 \xi^\delta\;,
\\
\bar\omega_1
&\eqdef \CM  \partial \bar\Psi^{U_{0}}_{\delta}\;,
\\
\bar\omega_2
&\eqdef \bar\Psi^{U_{0}}_{\delta}   \partial \CM\;,
\\
\bar\omega_3
&\eqdef
\bar \Psi^{U_{0}}_{\delta} 
\partial \bar \Psi^{U_{0}}_{\delta}
 - \sig^2(c_1^{\eps,\delta} +c_2^{\eps,\delta}) \CP U_0\partial \CP U_0\;,
\\
\bar\omega_4 &\eqdef
\CP Y_0 \partial\CW + \CW \partial \CP Y_0 + \CW\partial \CW\;,
\end{equs}
where, similar to before, we write
\begin{equ}[e:bar-Psi-U0-delta]
\bar\Psi_{\delta}^{U_0} \eqdef K^\eps \ast \bar\omega_0
= \CR \bar\bPsi^{U_0}  \;.
\end{equ}
Here $c_1^{\eps,\delta} ,c_2^{\eps,\delta}$ are defined as  in \eqref{e:c1c2}
except that $K$ is replaced by $K^\eps$ and $\moll_\eps^{(2)}$
is replaced by $\moll_\delta^{(2)}$.
We will write $\bar\omega_0^\delta$ and $\bar\omega^{\eps,\delta}_\ell$ for $\ell\in\{1,2,3\}$ and $\bar\omega^\eps_4$ when we wish to make the dependence of $\bar\omega$ on $\eps,\delta$ explicit
(remark that $\bar\omega^\delta_0$ and $\bar\omega^\eps_4$ are independent of $\eps$ and $\delta$ respectively).

\subsubsection{Probabilistic input}
\label{subsubsec:prob}
We write $Z^{\delta,\eps}_{\BPHZ}  \in \mathscr{M}_{\eps}$ to be the BPHZ models associated to the kernel assignment $K^{(\eps)}$ and random noise assignment $\zeta^{\delta,\eps}$.
We will first take $\delta \downarrow 0$ followed by $\eps \downarrow 0$.
Recall that here $\eps$ is the mollification parameter in Theorem~\ref{thm:gauge_covar}, and $\delta$
is an additional mollification for the white noise
in \eqref{eq:SPDE_for_bar_A}
which is introduced just for technical reasons, so that the models we explicitly construct are smooth models.

Recall from \cite[Sec.~7.2.2]{CCHS2d}    
that one can encode the smallness of $K-K^\eps$ and $\xi-\xi^\eps$ by  
introducing $\eps$-dependent norms on our regularity structure.
We then have corresponding 
 $\eps$-dependent seminorms  $\$\act\$_{\eps}$ \label{norm_eps_pageref}
 and pseudo-metrics  $d_{\eps}(\act,\act) = \$\act;\act\$_\eps$ \label{d_eps_pageref} on models as in Appendix~\ref{app:Singular modelled distributions}.
We will write $\varsigma\in(0,\kappa]$ for the small parameter `$\theta$' appearing in \cite[Sec.~7.2.2]{CCHS2d} to avoid confusion with $\theta$ as defined in Section~\ref{sec:renorm-A}.
For the reader unfamiliar with  \cite[Sec.~7]{CCHS2d}, we remark that all we need with these  $\eps$-dependent (semi)norms
is that one can extract a factor $\eps^\varsigma$
in an abstract Schauder estimate for $K-K^\eps$
(see Lemma~\ref{lem:K-barK-hat} and Lemma~\ref{lem:Schauder-input-KK})
 and, similarly, an estimate of the type $\$\boldsymbol\Xi - \bar {\boldsymbol\Xi}\$_\eps \lesssim \eps^\varsigma$ encoding the corresponding
 bound on $\xi-\xi^\eps$ at the level of models \slash modelled distributions (see \eqref{eq:diff_of_noises} below).

We then also have
$\eps$-dependent seminorms  on  $\cD^{\gamma,\eta} \ltimes \mathscr{M}_\eps$
denoted by $|\cdot |_{\cD^{\gamma,\eta,\eps}}$,
and
on $\hat\cD^{\gamma,\eta} \ltimes \mathscr{M}_\eps$
denoted by $|\cdot |_{\hat\cD^{\gamma,\eta,\eps}}$.
These seminorms are indexed by compact subsets of $\R\times \T^3$, which, as in Appendix~\ref{app:Singular modelled distributions}, we will always take of the form $O_\tau=[-1,\tau]\times\T^3$ for $\tau\in(0,1)$ and will sometimes keep implicit.

\begin{lemma}\label{lem:conv_of_models2}
One has, for any $p \ge 1$, 
\begin{equ}\label{eq:eps_control_of_model}
\sup_{\eps \in (0,1]}
\sup_{\delta \in (0,\eps)}
\E[ \|Z^{\delta ,\eps}_{\BPHZ}\|_{\eps}^{p} ] 
< \infty\;.
\end{equ}
Moreover, there exist models $Z^{0,\eps}_{\BPHZ} \in \mathscr{M}_{\eps}$ for $\eps \in (0,1]$
and a model $Z^{0,0}_{\BPHZ} \in \mathscr{M}_{0}$
 such that one has the following convergence in probability:
\begin{equs}
\lim_{\delta \downarrow 0} 
d_{\eps} (Z^{\delta,\eps}_{\BPHZ},
Z^{0,\eps}_{\BPHZ})
&=0 \quad (\forall \eps \in (0,1])\;,\label{e:conv-Z-delta}
\\
\lim_{\eps \downarrow 0}
d_{1}
(Z^{0,\eps}_{\BPHZ},
 Z^{0,0}_{\BPHZ})
&=0\;.\label{e:conv-Z-eps}
\end{equs}
\end{lemma}

\begin{proof}
The proof follows in  exactly the same way as \cite[Lem.~7.24]{CCHS2d},
by invoking \cite[Thm.~2.15]{CH16} and \cite[Thm.~2.31]{CH16}
and checking the criteria of these theorems as in Lemma~\ref{lem:conv_of_models}.
The only tweak for the proof  of \cite[Lem.~7.24]{CCHS2d}
is that the covariance of the noise $\xi^\eps$ 
is  measured in the $\|\act\|_{-5-2\kappa,k}$ kernel norm here since $d=3$.
\end{proof}
The next lemma shows that $\CW$ in~\eqref{eq:CM_def} vanishes as $\eps\downarrow0$.
\begin{lemma}
Let $\CW$ be defined by~\eqref{eq:CW_def}
and denote by the same symbol its lift to the polynomial sector.
Then, for all $\kappa\in (0,1)$ and $\tau>0$, there exists a random variable $M$ with moments bounded of all orders such that
\begin{equ}
|\CW|_{\frac74,-\frac12-2\kappa; \tau} \leq M\eps^{\kappa}\;.
\end{equ}
\end{lemma}

\begin{proof}
Recall that $\sup_{s\in[0,\eps^2]}|\moll^\eps*(\xi \bone_{-})(s)|_{L^\infty} \leq M \eps^{-\frac52-\kappa}$.
Therefore, for all $\gamma \in [0,2)\setminus\{1\}$
and $t\in [0,2\eps^2]$
\begin{equs}
|\CW(t)|_{\CC^\gamma} &\leq
\int_0^t | G_{t-s}(\moll^\eps*(\xi \bone_{-})(s))|_{\CC^\gamma}\mrd s
\lesssim \int_0^t (t-s)^{-\frac\gamma2} M \eps^{-\frac52-\kappa}\mrd s
\\
&\lesssim
Mt^{1-\frac\gamma2} \eps^{-\frac52-\kappa} 
\lesssim Mt^{-\frac\gamma2-\frac14-\kappa}\eps^\kappa\;.
\end{equs}
Recall further that $\sup_{s\in[0,\eps^2]}|\moll^\eps*(\xi \bone_{-})(s)|_{\CC^{-\frac12-2\kappa}} \leq M \eps^{-2+\kappa}$.
Therefore, for all $\gamma$ as above
and $t\in (2\eps^2,\tau]\;$,
\begin{equs}
|\CW(t)|_{\CC^\gamma}
\lesssim \int_0^{\eps^2} (t-s)^{-\frac\gamma2-\frac14-\kappa} M \eps^{-2+\kappa}\mrd s
&\leq
\eps^2(t-\eps^2)^{-\frac\gamma2-\frac14-\kappa}
M\eps^{-2+\kappa}
\\
&\lesssim Mt^{-\frac\gamma2-\frac14-\kappa}\eps^\kappa\;,
\end{equs}
where it suffices to integrate from $0$ to $\eps^2$ since
$\moll^\eps*( \xi\bone_-)$ is supported on the time interval $(-\infty,\eps^2]$.
\end{proof}

The next lemma gives the probabilistic convergence of the products that are ill-defined due to $t=0$ singularities but are needed in the formulation of our fixed point problem.
Recall $\omega^\eps_\ell$ defined in~\eqref{eq:omega_0_def} and~\eqref{e:omega123} and $\bar\omega_\ell^{\eps,\delta}$ defined in~\eqref{e:bar_omega0123}.
In the definition of $\bar\omega_\ell^{\eps,\delta}$ we take $\CM$ as in~\eqref{eq:CM_def}
where $\CW$ is defined by~\eqref{eq:CW_def}. 

\begin{lemma}\label{lem:omegas_converge}
For all $0 < \delta \le \eps < 1$, there exists a random variable $M>0$,
of which every moment is bounded uniformly in $\eps,\delta$, and such that
\begin{equs}[eq:PU_0xi_bounds]
|\omega^\eps_0|_{\CC^{-\frac52-\kappa}}
&\leq  M \sig |U_0|_{\infty}\;,
\\
|\omega^\eps_0 - \omega^\delta_0|_{\CC^{-\frac52-\kappa}}
&\leq  M \sig |U_0|_{\CC^{\kappa/2}} \eps^{\kappa/2}\;,
\end{equs}
and
\begin{equ}[e:zeta-hat-diff]
| \omega_\ell^\eps - \bar \omega_\ell^{\eps,\delta}  |_{\CC^{\alpha-2\kappa}}
 \leq M \sig^n |U_0|_{\CC^{\rho}}^{n} (|Y_0|_{\CC^\eta} +\eps^{\kappa/2})^m 
\eps^\kappa\;,
\end{equ}
where $\alpha=\eta-\frac32,m=1,n=1$ for $\ell \in \{1,2\}$ and $\alpha=-2,m=0,n=2$ for $\ell=3$,
and
\begin{equ}
|\bar\omega^\eps_4|_{\CC^{\eta-\frac32-2\kappa}} \leq M(\sig |Y_0|_{\CC^\eta} + \sig^2)\eps^\kappa\;.
\end{equ}
In particular, $\omega^0_0\eqdef \lim_{\eps\downarrow 0}\omega^\eps_0$ exists in probability in $\CC^{-\frac52-\kappa}$, and $\lim_{\eps\downarrow0}\bar\omega^\eps_4 = 0$ in probability in $\CC^{\eta-\frac32-2\kappa}$.

Moreover $\omega_1^\eps$ and $\omega_2^\eps$
converge in probability as $\eps\to 0$
in $\CC^{\eta-\f32-\kappa}$,
and $\omega_3^\eps$
converges in probability  as $\eps\to 0$ in $\CC^{-2-\kappa}$.
We denote their limits by $\omega_1^0,\omega_2^0,\omega_3^0$. 
Finally, for $\ell\in \{1,2,3\}$,
the limit $\bar\omega^{\eps,0}_\ell\eqdef \lim_{\delta\downarrow0}\bar\omega^{\eps,\delta}_\ell$
exists and
$\omega_\ell^0 = \lim_{\eps\downarrow 0} \bar\omega_\ell^{\eps,0}$  under the same topologies as above.
\end{lemma}

\begin{proof}
Uniformly in $\psi\in\mcB^3$, $\lambda,\eps\in(0,1)$, and $\delta\in(0,\eps)$,
\begin{equ}
\E\big[
\scal{\CP U_0
\xi^\eps,\psi^\lambda_z}^2
\big]
= |\moll^\eps *(\CP U_0\psi^\lambda_z)|^2_{L^2}\lesssim |\CP U_0|_{L^\infty}^2 \lambda^{-5}\;,
\end{equ}
and
\begin{equ}
\E\big[\scal{\CP U_0(\xi^\eps - \xi^{\delta}),\psi^\lambda_z}^2
\big]
=
|(\moll^\eps-\moll^\delta) *(\CP U_0\psi_z^\lambda)|^2_{L^2}
\lesssim
\eps^{\kappa}\lambda^{-5-\kappa}|U_0|^2_{\CC^{\kappa/2}}\;.
\end{equ}
We thus obtain~\eqref{eq:PU_0xi_bounds} by equivalence of Gaussian moments and Kolmogorov's theorem.

The convergence of $\omega_1^\eps, \omega_2^\eps$ follows in exactly the same way as 
Lemma~\ref{lem:PX0Psi-prob} (for the special case $U_0=1$)
except that in 
\eqref{e:PX0dPsi-21}-\eqref{e:PX0dPsi-22}
one also uses $|\CP U_0|_{\infty} \lesssim 1$.

The convergence of 
the second Wiener chaos of $\omega_3^\eps$ also follows 
in  the same way  as Lemma~\ref{lem:tildePsi-cov} 
except that in \eqref{e:tildePsi-cov}  
one simply bounds $|\CP U_0|_{\infty} \lesssim 1$.
Regarding the zeroth chaos of $\omega_3^\eps$, we note that
\begin{equs}[e:zeta3-0chaos]
\E \Big( & \big(K*(\CP U_0 \xi^\eps) \big)
\big(\partial K*(\CP U_0 \xi^\eps)\big)(z)\Big)
\\
=
&\int \Big(K(z-w)  \CP U_0(w)\Big)
\partial \Big( K(z-\bar w)  \CP U_0(\bar w) \Big)\moll_\eps^{(2)} (w-\bar w)\mrd w\mrd \bar w
\end{equs}
where the first line should be understood with the notation \eqref{e:XdX-X3}
while the second line should be understood with the notation 
\eqref{e:udv}. In particular both sides are $E$-valued functions of $z$.
Indeed, \eqref{e:zeta3-0chaos}
 can be checked using  \eqref{e:XdX-X3}+\eqref{e:udv}
and the facts that 
the components of the white noise $\xi$
are independent and the basis in \eqref{e:udv} are orthogonal.

\eqref{e:zeta3-0chaos} would be zero if $U_0=1$ (since the integrand would be odd in the spatial variable thanks to  the derivative), but generally not.
Write
\begin{equ}[e:PU0-expand]
\CP U_0 (w) = \sum_{|k|\le 2} \frac{(w-z)^k}{k!} \partial^{k}\CP U_0 (z) + r_u(w,z)
\end{equ}
and the same for $\CP U_0 (\bar w)$ for which we call 
 the multiindex $\bar k$ instead of $k$. 
Fixing $i\in [3]$, we consider the $i$-th component of  \eqref{e:zeta3-0chaos}.
Below
we write $s,\bar s$ for the time variables of $w, \bar w$ and $z=(t,x)$.

 We substitute  \eqref{e:PU0-expand} into \eqref{e:zeta3-0chaos}.
The contributions from terms with $k_i+\bar k_i\in \{0,2\}$
all vanish since $K$, $\moll_\eps^{(2)}$ are even and $\partial_i K$ is odd in the $i$-th spatial variables.

The contributions from terms with $k_i+\bar k_i =1$
combined with 
the renormalisation terms introduced in \eqref{e:omega123}-\eqref{e:c1c2}
converge in $\CC^{-2-\kappa}$.
Indeed, consider the contribution of the case 
 $(k_i,\bar k_i)=(0,1)$ to \eqref{e:zeta3-0chaos}
  minus the renormalisation term $c_2^\eps \CP U_0\partial \CP U_0$
  (the argument for the other case is identical).
Note that here $s,\bar s$ need to be positive for the integrand in  \eqref{e:zeta3-0chaos} to be non-zero,
whereas in \eqref{e:c1c2} $s,\bar s \in \R$. Thus
 the contribution we consider here is bounded by 
$|\CP U_0(z)\partial \CP U_0(z)|$ times
\begin{equs}
\Big|\int_{\R_- \times \T^3} K_{t-s}(y) \partial_i K_{t-s}(y) \,y_i  \,\mrd y \mrd s\Big|
&\lesssim
\int_{\R_- \times \T^3} |y| (\sqrt{t-s} + |y|)^{-7} \,\mrd y \mrd s
\\
&\lesssim 
\int_{\R_-} (t-s)^{-3/2} \,\mrd s 
\lesssim t^{-\frac12}\;.
\end{equs}
This implies that this part of contribution is in $\CC^{-1}\subset \CC^{-2-\kappa}$.

All the other terms involve $r_u$ and they converge in $\CC^{-2-\kappa}$.  
Indeed, 
one has that for $s, \bar s \le t$,
$|r_u(w,z)|\lesssim s^{\frac{\rho-5/2}{2}} |w-z|^{\frac{5}{2}}  |U_0|_{\CC^\rho}$.
Then the ``worst'' terms are the two cross terms between $\CP U_0 (z)$ (i.e.\ $k=0$) and $r_u$,
one of which (the other one is almost identical) is bounded by  $  |U_0|_{\CC^\rho}^2$ times
\begin{equs}
\int & K(z-w) 
\partial_i K(z-\bar w) s^{\frac{\rho-5/2}{2}} |w-z|^{\frac52} \moll_\eps^{(2)} (w-\bar w)\mrd w\mrd \bar w
\\
&\lesssim
\int  s^{\frac{\rho-5/2}{2}} |w-z|^{-9/2}\mrd w
\lesssim 
\int  s^{\frac{\rho-5/2}{2}} (\sqrt{t-s} + |y-x|)^{-9/2}\mrd s\mrd y
\\
&\lesssim
\int  s^{\frac{\rho-5/2}{2}} (t-s)^{-3/4}\mrd s
\lesssim
t^{\frac{\rho-2}{2}}
\end{equs}
since $\rho>1/2$, where the integration variables $s \in [0, t]$.
Regarding the  cross terms between $(w-z)^k \CP^{(k)}U_0 (z)$ with $|k|\in \{1,2\}$ and $r_u$,
one can bound $|\CP^{(k)}U_0 (z)| \lesssim t^{\frac{\rho-|k|}{2}}|U_0|_{\CC^\rho}$,
and then proceeding as above one obtains a better bound of the form
\[
t^{\frac{\rho-|k|}{2}}  \int_0^t  s^{\frac{\rho-5/2}{2}} (t-s)^{-\frac34+\frac{|k|}{2}}\mrd s
\lesssim
t^{\rho-1}\;.
\]
The cross terms between $r_u(w,z)$ and $r_u(\bar w,z)$ are then even easier to bound.
Finally, note that
 as a function of $z$ \eqref{e:zeta3-0chaos} converges in $\CC^{\rho-2}\subset \CC^{-2-\kappa}$ as desired,
 using the fact that $\$\delta-\moll^\eps\$_{-5-\kappa;m} \lesssim \eps^\kappa$ in the notation of~\cite[Sec.~10.3]{Hairer14}.

By the same arguments one can show the desired convergences of
$\CP Y_0  \partial \bar\Psi^{U_{0}}_{\delta}$,
$\bar\Psi^{U_{0}}_{\delta}   \partial \CP Y_0$,
and $\bar\omega_{3}^{\eps,\delta}$.
To prove the claimed convergences of $\bar\omega_{1,2}^{\eps,\delta}$
it suffices to
 show that 
$\CW  \partial \bar\Psi^{U_{0}}_{\delta}$ and
$\bar\Psi^{U_{0}}_{\delta}   \partial \CW$ vanish in the limit.
Below we write $w<z$ (resp. $w\le z$) for space-time variables $w,z$ 
if the time coordinate of $w$ is less than (resp. less or equal to)
the time coordinate of $z$. 

Using the definitions of $\CW$ in \eqref{eq:CW_def}
and  $\bar\Psi^{U_{0}}_{\delta}$ in \eqref{e:bar-Psi-U0-delta},
and non-anticipativeness of $K$ and $\chi$,
the zeroth chaos of $\CW  \partial \bar\Psi^{U_{0}}_{\delta}$
can be bounded by
\begin{equ}[e:WPsi0chaos]
\int_D |K(z-w_1)\nabla K(z-y_1) \chi^\eps(w_1-w_2)\chi^\eps(y_1-y_2)
\chi^\delta(y_2-w_2) |\,\mrd w_1\mrd w_2 \mrd y_1 \mrd y_2
\end{equ}
where $D\eqdef \{(w_1,w_2,y_1,y_2): w_2<0<w_1\le z,
0<y_2<y_1\le z\}$.
For any $\eps>0$, as $\delta\downarrow 0$, the integrand
converges to a smooth function but $D$ shrinks to an empty set,
since $\chi^\delta$ has support size $\delta$
and $w_2<0<y_2$, so the above integral vanishes in the limit. 
The zeroth chaos of 
$\bar\Psi^{U_{0}}_{\delta}   \partial \CW$ vanishes by the same argument,
with the derivative in \eqref{e:WPsi0chaos} taken on the first $K$. 

The second chaos of $\CW  \partial \bar\Psi^{U_{0}}_{\delta}$
and $\bar\Psi^{U_{0}}_{\delta}   \partial \CW$,
as well as the term $\CW\partial \CW$ in
$\bar\omega_4$, 
are bounded similarly as Lemma~\ref{lem:tildePsi-cov},
except that instead of \eqref{e:tildePsi-cov} one uses the bound
\begin{equs}
|\E (\partial^k \CW (z) \,\partial^k \CW(\bar z))| 
&\lesssim
\int |\partial^k K(z-w) \partial^k  K(\bar z-\bar w) (\chi^\eps*\chi^\eps)(w-\bar w)|
\mrd w\mrd \bar w
\\
&\lesssim
\eps^\kappa |z-\bar z|^{-1-2|k|-\kappa}	\label{e:EWW}
\end{equs}
for $|k|\in \{0,1\}$,
where the time variables of $w$ and $\bar w$
are restricted to $[0,\eps^2]$,
due to the cutoff functions 
$\bone_{t\geq0}$, $\bone_{t<0}$ 
and the compact support of $\moll^\eps$ in the definition of $\CW$.
(Note that 
$\CW$ defined in~\eqref{eq:CW_def} and $\tilde\Psi_\eps$ defined in \eqref{e:def-tilde-Psi} only differ by a cutoff function $\bone_{t<0}$.)
Finally, regarding the terms 
$\CP Y_0 \partial\CW$ and $\CW \partial \CP Y_0$
in the definition of $\bar\omega_4$,
they vanish
by the same proof of Lemma~\ref{lem:PX0Psi-prob},
except that
in \eqref{e:PX0dPsi-21}-\eqref{e:PX0dPsi-22},
one obtains an extra factor $\eps^\kappa$ thanks to \eqref{e:EWW}.

The difference estimates~\eqref{e:zeta-hat-diff}
between  $\omega_\ell^\eps $ and $\bar\omega_\ell^{\eps,\delta}$
follow  the same arguments using $\$K-K^\eps\$_{-3-\kappa;m} \lesssim \eps^\kappa$ and $\$\moll^{\delta}-\moll^\eps\$_{-5-\kappa;m} \lesssim (\eps\vee\delta)^\kappa$.
\end{proof}

The next lemma ensures that the input distributions in our fixed point problem are compatible with the corresponding modelled distributions.

\begin{lemma}\label{lem:compatible}
Consider $\eps\in[0,1]$ and $\delta\in [0,\eps]$.
Then, for the model $Z^{\delta,\eps}_{\BPHZ}$, $\omega^\eps_0$ and $\bar\omega^{\delta}_0$ is compatible with $\CP U_0\bar\bXi$ and $\CP U_0\bXi$ respectively.
Furthermore
\begin{equ}
\omega^\eps_1,\omega^\eps_2,\omega^\eps_3
\;\text{ and }\; \bar\omega^{\eps,\delta}_1,\bar\omega^{\eps,\delta}_2,\bar\omega^{\eps,\delta}_3
\end{equ}
is compatible with
\begin{equ}
\CP Y_0 \partial \bPsi^{U_0}, \bPsi^{U_0}\partial \CP Y_0, \bPsi^{U_0}\partial \bPsi^{U_0}
\;\text{ and }\;
\CM \partial \bar \bPsi^{U_0}, \bar \bPsi^{U_0}\partial \CM, \bar \bPsi^{U_0}\partial \bar \bPsi^{U_0}, 
\end{equ}
respectively.
Finally, $\bar\omega^{\eps,\delta}_4$ is compatible with $\CP Y_0 \partial\CW + \CW \partial \CP Y_0 + \CW\partial \CW$.
\end{lemma}

\begin{proof}
We first consider $\eps,\delta>0$.
For $\CP U_0 \bar{\bXi}$,
it is clear that for $t>0$
\begin{equ}[e:PU0Xi-local]
\tilde\CR(\CP U_0 \bar{\boldsymbol{\Xi}})(t,x)
= \sig \CP U_0 \xi_\eps(t,x) = \omega^\eps_0
\end{equ}
because 
for any modelled distribution $f$ in the span of the polynomials,
 one has $\tilde{\CR} (f \bar{\boldsymbol{\Xi}})(t,x) 
 = (\tilde{\CR}f)(t,x) (\tilde{\CR} \bar{\boldsymbol{\Xi}})(t,x)$.
This also proves the claim for $\bar\omega^{\delta}_0$ since $\bar\omega^{\delta}_0=\omega_0^\delta$.
(Note that we required $\eps>0$ above since we cannot in general multiply arbitrary $\xi\in\CC^{-\frac52-\kappa}$ and $\CP U_0$ for $U_0\in\CC^{\rho}$ to obtain a distribution on all of $\R\times \T^3$.)

The case for $\eps=\delta=0$ follows from the
fact that the first equality in~\eqref{e:PU0Xi-local} still holds on $(0,\infty)\times\T^3$ (in a distributional sense)
and from the existence of $\omega^0_0 \eqdef \lim_{\eps\to0} \CP U_0 \xi_\eps$
in probability in $\CC^{-\frac52-\kappa}$ obtained in Lemma~\ref{lem:omegas_converge}.

Consider again $\eps,\delta>0$.
For the same reason, we have for $t>0$
\[
\tilde\CR (\CP Y_0  \partial \bPsi^{U_0})(t,x)
 = \CP Y_0 (t,x)
\tilde\CR (\partial \bPsi^{U_0}) (t,x)
=\omega_1^\eps(t,x)
\]
and similarly 
$\tilde\CR ( \bPsi^{U_0} \partial \CP Y_0 )(t,x)=\omega_2^\eps(t,x)$.
Moreover, we claim that for $t>0$
\begin{equ}[e:to-check-zeta3]
\tilde\CR ( \bPsi^{U_0} \partial \bPsi^{U_0}) = 
\omega_3^\eps \;.
\end{equ}
In fact,  recalling the definition of $ \bPsi^{U_0}$ in \eqref{e:bPsi_U0}  and writing the polynomial lift of
 $\CP U_0 $ as  $u \bone + \langle \nabla u, \bf{X} \rangle$ plus higher order polynomials, one has 
\[
\bPsi^{U_0}\partial \bPsi^{U_0} = 
 u \partial u \<IXXiI'Xi_notriangle>
+ u \partial u \<IXiI'XXi_notriangle>  + (\cdots)
\]
where $(\cdots)$ stands for trees which do not require renormalisation.
Here a crossed circle $\<XXi>$
denotes
$\mbX_j\Xi$ with the index $j \in \{1,2,3\}$ being equal to the derivative index for the derivative (thick line) of the tree.
The BPHZ renormalisation of the two trees $\<IXXiI'Xi_notriangle>$ and $\<IXiI'XXi_notriangle> $ 
yields renormalisation constants that are precisely given by \eqref{e:c1c2}, and therefore one has \eqref{e:to-check-zeta3}.

The corresponding claims for $\bar\omega^{\eps,\delta}_\ell$ with $\ell=1,2,3$ follow in an identical way,
while the claim for $\bar\omega_4$ is obvious.
Finally, the case $\eps>0$ and $\delta=0$ follows from the convergence of models~\eqref{e:conv-Z-delta} and the corresponding convergence $\bar\omega^{\eps,0}_\ell=\lim_{\delta\downarrow0}\bar\omega^{\eps,\delta}_\ell$ in Lemma~\ref{lem:omegas_converge},
and the case $\eps=\delta=0$ follows from~\eqref{e:conv-Z-eps}
and the convergences
\begin{equ}
\omega^{0}_\ell=\lim_{\eps\downarrow0}\omega_\ell^\eps=\lim_{\eps\downarrow0}\bar\omega^{\eps,0}_\ell
\end{equ}
in Lemma~\ref{lem:omegas_converge}.
\end{proof}

\begin{remark}\label{rmk:multi-coherent}

As in the discussion in Section~\ref{sec:solution theory additive},
 for each $z=(t,x)$ with $t > 0$, $\CY(z)$ solves an algebraic fixed point problem of the form 
\[
 \mcY(z)
=
\mcb{I}_\mfz \Big(
\mcY(z) \partial \mcY(z) 
+ \mcY(z)^3
 + \mathring{C}_{\mfz} \mcY(z)+\mathring C_\mfh\CH
 + \mcU(z)\bar{\boldsymbol{\Xi}}(z) \Big) + (\cdots)\;,
 \]
where $(\cdots)$ takes values in the polynomial sector, and similarly for $\CU(z)$ and $\CH(z)$. 
From this we see that the solution $\CY(z)$ must be coherent with respect to 
\eqref{e:F_mfz}--\eqref{e:F_mfu}.
Therefore, the renormalised equations are still given by Proposition~\ref{prop:renorm_g_eqn}.
The appearance of the renormalisations in \eqref{e:omega123}-\eqref{e:c1c2} are an artifact of 
our decomposition \eqref{e:def-tildeY}.
Namely, they correspond to the BPHZ renormalisations of $\<IXXiI'Xi_notriangle>$ and $\<IXiI'XXi_notriangle>$
(as pointed out in the proof of Lemma~\ref{lem:compatible})
 which contribute to the term $C^{\eps}_{\Gauge}h_i$ in Proposition~\ref{prop:renorm_g_eqn};
even if one were to renormalise the decomposed equations, 
 terms such as $(c_1^\eps +c_2^\eps) \hat\CU\partial\hat\CU$  would arise in addition to the ones in \eqref{e:omega123},
which would still contribute to $C^{\eps}_{\Gauge}h_i$.
%
Also we note that the inputs $\omega$'s
match the global reconstruction with the BPHZ model in the smooth setting, therefore the reconstructed solution satisfies the correct initial condition.
\end{remark}

%
%
\subsubsection{Closing the analytic fixed point problem}\label{subsec:analytic_fp}

In the rest of this subsection, we collect the  deterministic analytic results needed to solve and compare~\eqref{e:fix-pt-multi-Y}+\eqref{e:fix-pt-HU}
and \eqref{e:fix-pt-multi-Xbar}+\eqref{e:fix-pt-HU}.
The input necessary to solve the two systems is a model $Z$,
initial condition~\eqref{eq:initial_data_for_gauge_sys}, distributions $\omega_\ell$ and $\bar\omega_\ell$ compatible with the corresponding $f$ in every term of the form $\CG^\omega(f)$,
and the modelled distribution $\CW$ in~\eqref{eq:CM_def} which takes values in the polynomial sector.

We say that an \emph{$\eps$-input of size $r>0$ over time $\tau>0$}
is a collection of such objects which furthermore satisfy
\begin{equs}
\$Z&\$_{\eps;\tau}
+ \Sigma(Y_0) + |U_0|_{\CC^\rho} + |h_0|_{\CC^{\rho-1}} + \eps^{-\varsigma}|\CW|_{\frac74,-\frac12-\kappa;\tau}
\\
&+ \sum_{\ell=0}^4
\big(
|\omega_\ell|_{\CC^{\alpha_\ell}_\tau} + |\bar\omega_\ell|_{\CC^{\alpha_\ell}_\tau} + \eps^{-\varsigma}|\omega_\ell-\bar\omega_\ell|_{\CC^{\alpha_\ell}_\tau}
\big) \label{eq:input_def}
\\
&+
\sup_{t\in [0,\tau]}
\big\{\Sigma(\PPi^Z \<IXi>(t)) + \Sigma(\PPi^Z \<IXiS>(t))
+ |\PPi^Z \<I[IXiI'Xi]_notriangle>(t)|_{\CC^{-\kappa}}
+ |\PPi^Z\<I[IXiSI'XiS]_notriangle>(t)|_{\CC^{-\kappa}}
\big\}
\leq r
\;,
\end{equs}
where $\omega_4\eqdef 0$,
$\varsigma>0$ is the fixed parameter from Section~\ref{subsubsec:prob} (which we later take sufficiently small), $\PPi^Z=\Pi_0$ where we denote as usual $Z=(\Pi,\Gamma)$,
and
\begin{multline*}
(\alpha_0,\alpha_1,\alpha_2,\alpha_3,\alpha_4)
\\
=
\Big(-\frac52-\kappa, \eta-\frac32-2\kappa,\eta-\frac32-2\kappa,-2-4\kappa,
\eta-\frac32-\kappa
\Big)\;.
\end{multline*}
Above and until the end of the subsection we write
\begin{equs}
\<IXi>
&\eqdef \mcb{I}_\mfz[\bar \bXi]
\;,\quad \<I[IXiI'Xi]_notriangle> \eqdef \mcb{I}_\mfz[\mcb{I}_{\mfz}[\bar \bXi]\partial_j\mcb{I}_{\mfz}[\bar \bXi]]
\\
\<IXiS>
&\eqdef \mcb{I}_\mfm[ \bXi]
\;,\quad 
\<I[IXiSI'XiS]_notriangle>\eqdef \mcb{I}_\mfz[\mcb{I}_{\mfm}[\bXi]\partial_j\mcb{I}_{\mfm}[\bXi]]\;.
\end{equs}
Note that the thick line in \<I[IXiI'Xi]_notriangle> and \<I[IXiSI'XiS]_notriangle> represents a spatial derivative,
so each of these really represents a collection of three basis vectors, one for each direction, indexed by $j$. In
\eqref{eq:input_def}, the corresponding norm should be interpreted as including a supremum over $j$.

\begin{remark}
The choice `$0$' in the definition $\PPi^Z=\Pi_0$ is completely arbitrary since the structure group acts trivially on the set of symbols $\mcJ\eqdef \{\<IXi>,\<IXiS>,\<I[IXiI'Xi]_notriangle>,\<I[IXiSI'XiS]_notriangle>\}$ and therefore $\Pi_0 = \Pi_x$ for all $x\in \R\times \T^3$ when restricted to $\mcJ$.
\end{remark}
\begin{remark}\label{rem:PPi_not_needed}
The term on the left-hand side of~\eqref{eq:input_def} involving $\sup$ plays no role in Lemma~\ref{lem:fixedptpblmclose},
but will be important in Proposition~\ref{prop:improve_local} below.
\end{remark}

We assume throughout this subsection that $\moll$ is non-anticipative, which implies that $K^\eps$ is also non-anticipative.

\begin{lemma}\label{lem:fixedptpblmclose}
Consider $r>0$, $\eps\in[0,1]$, and the bundle of modelled distributions
 \begin{equ}[e:space]
 \Big(
\hat\cD^{\gamma_{\mfz},\eta_{\mfz}}_{\alpha_{\mfz}} 
\oplus
\hat\cD^{\gamma_{\mfu},\eta_{\mfu}}_{\alpha_{\mfu}} 
\oplus
\cD^{\gamma_{\mfh},\eta_{\mfh}}_{\alpha_{\mfh}} 
\Big)
 \ltimes \mathscr{M}_\eps
 \end{equ}
 where
\begin{equ}\label{eq:fixed_pt_space}
(\gamma_{\mft},\alpha_{\mft}, \eta_{\mft}) = 
\begin{cases}
(\frac32+3\kappa, -\frac12-2\kappa, \hat\beta) 
			& \textnormal{ if }\mft = \mfz \;,\\
(\frac52 + 2 \kappa, 0, \rho) 
			& \textnormal{ if }\mft = \mfu \;,\\
(1 + 5\kappa, 0, \rho-1) 
			& \textnormal{ if }\mft = \mfh \;.
\end{cases}
\end{equ}
Then there exists $\tau>0$, depending only on $r$,
such for all $\eps$-inputs of size $r$ over time $\tau$
there exist solutions $\cS$ and $\bar\cS$ in~\eqref{e:space} on the interval $(0,\tau)$ to~\eqref{e:fix-pt-multi-Y}+\eqref{e:fix-pt-HU} 
and \eqref{e:fix-pt-multi-Xbar}+\eqref{e:fix-pt-HU} respectively.
Furthermore
\begin{equ}\label{e:SS-RSRS}
|\cS - \bar{\cS}|_{\vec{\gamma},\vec{\eta},\eps;\tau} 
\lesssim 
\eps^{\varsigma}
\end{equ}
uniformly in $\eps$ and $\eps$-inputs of size $r$, 
where $| \act |_{\vec{\gamma},\vec{\eta},\eps;\tau} $ is the corresponding multi-component modelled distribution (semi)norm for \eqref{eq:fixed_pt_space} on the interval $(0,\tau)$.
Finally, when the space of models is equipped with the metric $d_{1;\tau}$,
both $\cS$ and $\bar\cS$ are locally uniformly continuous with respect to the input.
\end{lemma}

 \begin{proof} 
We first prove well-posedness of the fixed point problem 
\eqref{e:fix-pt-multi-Y}+\eqref{e:fix-pt-HU}.
For simplicity we write $\gamma $ for $\gamma_\mfz$.
As in the additive case \eqref{e:PX0D},
one has $\CP Y_0 \in \hat{\cD}^{\infty,\eta}_{0} $, $\partial \CP Y_0 \in \hat{\cD}^{\infty,\eta-1}_{0} $,
and, since $Y_0\in\init$,
$ \CG_{\mfz} \big( \CP Y_0 \partial \CP Y_0 \big)$ is well-defined as an element of $\hat\cD^{\gamma,\hat\beta}_0$ due to~\eqref{e:PN-bound}.
Also, for $\rho \in (1/2,1)$,
we recall that
$\CP U_0 \in  \cD^{\infty,\rho}_{0}$
and $\bar{\boldsymbol{\Xi}}\in \cD^{\infty,\infty}_{-\frac52-\kappa}$.
Thus $\CP U_0 \bar{\boldsymbol{\Xi}}\in \cD^{\infty,\rho-\frac52-\kappa}_{-\frac52-\kappa}$
and, by Lemma~\ref{lem:Schauder-input},
we have~\eqref{e:bPsi_U0}, i.e., $\bPsi^{U_0}\in\cD^{\infty,-1/2-2\kappa}_{-1/2-2\kappa}$.

Considering the terms in $R_Q$ as defined in \eqref{e:def-RQ},
Lemma~\ref{lem:multiply-hatD} implies
\begin{equs}[eq:CPY_0_hat_Y]
\CP Y_0 \partial \hat{\CY} \in \hat{\cD}^{\gamma-1,\eta+\hat\beta-1}_{-\frac32-2\kappa}\;,
\quad
 \hat{\CY}  \partial \CP Y_0 
& \in \hat{\cD}^{\gamma,\eta+\hat\beta-1}_{-\frac12-2\kappa}\;,
 \quad
 \hat{\CY} \partial \hat{\CY} \in \hat{\cD}^{\gamma-\frac32-2\kappa,2\hat{\beta}-1}_{-2-4\kappa}
\\
\mbox{and}\qquad
 \bPsi^{U_0}\partial \hat{\CY} \;,\; &
 \hat{\CY} \partial \bPsi^{U_0}
 \in \hat{\cD}^{\gamma-\frac32-2\kappa , \hat\beta-\frac32-2\kappa}_{-2-4\kappa}\;.
\end{equs}
Since $(\eta+\hat\beta-1)\wedge (2\hat{\beta}-1) \wedge (\hat\beta-\frac32-2\kappa)> -2$ by \eqref{eq:CI},
using Theorem~\ref{thm:integration}, one has
$\CG_{\mfz} \mathbf{1}_{+} R_Q \in 
\hat\cD^{\gamma,\hat\beta}_{-\frac12-2\kappa}$.

On the other hand, by \eqref{e:bPsi_U0}, one has
\begin{equs}[eq:CP_Y_0_Psi]
\CP Y_0 \partial \bPsi^{U_0} &\in \cD^{\infty,\eta-\frac32-2\kappa}_{-\frac32-2\kappa},
\quad
 \bPsi^{U_0}\partial \CP Y_0 \in \cD^{\infty,\eta-\frac32-2\kappa}_{-\frac12-2\kappa}, \quad \text{and}
\\
\bPsi^{U_0}\partial \bPsi^{U_0}
&\in \cD^{\infty,-2-4\kappa}_{-2-4\kappa}
\end{equs}
where $\eta-\frac32-2\kappa < -2$.
Taking $\kappa>0$ sufficiently small such that $\eta+\frac12<-2\kappa$,
Lemma~\ref{lem:Schauder-input}
therefore implies
  that  the first line on the right-hand side of \eqref{e:fix-pt-multi-Y}
  belongs to 
  $\cD^{\infty, \eta+\frac12-3\kappa}_{-5\kappa}$ 
  and thus to
$\hat \cD^{\infty, \eta+\frac12-3\kappa}_{-5\kappa}$ by Remark~\ref{rem:DDhat-same}.

Moreover,
using $\eta<-\frac12-2\kappa$,
it follows from~\eqref{e:bPsi_U0},~\eqref{e:def-tildeY},~\eqref{e:hatCYhatCU},
and Remark~\ref{rem:DDhat-same} that for $\rho \in (1/2,1)$
\begin{equ}
\CY \in \hat\cD^{\gamma,\eta}_{-\frac12-2\kappa}
\;
\Rightarrow
\;
\CY^3  \in \hat\cD^{\gamma-1-4\kappa,3\eta}_{-\frac32-6\kappa}\;,
\qquad
\mbox{and}
\qquad
\CH \in \hat\cD^{1+5\kappa,\rho-1}_{0}\;,
\end{equ}
so Theorem~\ref{thm:integration} applies 
and one has
$\CG_{\mfz}  \mathbf{1}_{+}  \big(  \CY^3 + \mathring{C}_{\mfz} \CY+\mathring C_\mfh\CH \big)
\in \hat{\cD}^{\gamma,\hat\beta}_{-\frac12-2\kappa}$.

Finally, recalling that $\gamma_\mfu = \frac52+2\kappa$,
we have 
\[
 \mathbf{1}_+ \bar{\boldsymbol{\Xi}}  \in \hat{\cD}^{\infty,-\frac52-\kappa}_{-\frac52-\kappa}
 \quad\mbox{and thus}\quad
\hat\CU  \bar{\boldsymbol{\Xi}} \in \hat{\cD}^{\kappa, \rho-\frac52-\kappa}_{-\frac52-\kappa}\;.
\]
Using our assumption $\rho>1/2$ one has $\rho-\frac52-\kappa  > -2$ for $\kappa>0$ sufficiently small, so
Theorem~\ref{thm:integration} again applies
and one has
$\CG_{\mfz}  \big( \hat\CU  \bar{\boldsymbol{\Xi}} \big)
\in \hat{\cD}^{\gamma,\hat\beta}_{-\frac12-2\kappa}$.

Turning to the equation for $\CH$, using our assumptions \eqref{eq:fixed_pt_space}, we have
\begin{equs}
{} [\mathcal{B}_j, \mathcal{H}_j]  & 
\in \cD^{\gamma_\mfh-\frac12-2\kappa, \eta+\eta_{\mfh}}_{-\frac12-2\kappa},
\qquad
 [\mathcal{H}_j,\p_j \mathcal{H}] 
  \in \cD^{\gamma_\mfh-1, 2\eta_\mfh-1}_{-\frac12-2\kappa},
\\
{} [[\mathcal{B}_j, \mathcal{H}_j],\mathcal{H}] & 
\in \cD^{\gamma_\mfh-\frac12-2\kappa, \eta+2\eta_\mfh}_{-\frac12-2\kappa}\;.
\end{equs}
For the term $[\mathcal{H}_j,\p_j \mathcal{H}]$ above, we used that the lowest degree term appearing in the expansion for $\CH$ is $\mcb{I}_{\mfh'}[\mcb{I}_{\mfz}[\bar\bXi]]$.
Recalling that $\rho \in (\frac12,1)$ and $\eta_\mfh=\rho-1$, 
as well as $\eta+\rho>0$ (since \eqref{eq:CGI} implies \eqref{e:paramG} by Remark~\ref{rem:GI-equiv}),
one has 
$\eta+\eta_{\mfh}>-1$ and
 $(2\eta_\mfh-1) \wedge  (\eta+2\eta_\mfh)>-2$.
Thus the right-hand side of  the equation for $\CH$
belongs to $\cD^{\gamma_{\mfh},\eta_{\mfh}}_{\alpha_{\mfh}} $.

Regarding the equations for $\CU$ and $\hat\CU$, our assumption \eqref{eq:fixed_pt_space}
on $\hat\CU$ and the fact that $\CP U_0 \in \cD^{\infty,\rho}_0$
implies $\CU \in \cD^{\gamma_\mfu,\eta_\mfu}_0$, so
\begin{equs}
{} [\mathcal{H}_j,[\mathcal{H}_j, \cdot]] \circ \mathcal{U}
&\in \cD^{\gamma_\mfh,2\eta_\mfh}_0 \cong 
\hat\cD^{\gamma_\mfh,2\eta_\mfh}_0 ,
\\
{} [[\mathcal{B}_j, \mathcal{H}_j],\cdot] \circ \mathcal{U}
& \in \cD^{\gamma_\mfh-\frac12-2\kappa, \eta+ \eta_\mfh}_{-\frac12-2\kappa}
 \cong 
 \hat\cD^{\gamma_\mfh-\frac12-2\kappa,
\eta+ \eta_\mfh}_{-\frac12-2\kappa}
\end{equs}
where $2\eta_\mfh \in (-1,0)$, and
$\eta+\eta_{\mfh}>-1$ as mentioned above.
Noting that $(\gamma_\mfh-\frac12-2\kappa)+2-\kappa= \gamma_\mfu$,
Theorem~\ref{thm:integration} then implies that 
 the right-hand side of the equation for $\hat\CU$
belongs to $\hat\cD^{\gamma_{\mfu},\eta_{\mfu}}_{\alpha_{\mfu}} $.
The fact that these maps stabilise and are contractions on a ball follows from the short-time convolution estimates in Theorem~\ref{thm:integration} as well as the refinement of~\cite[Thm.~7.1]{Hairer14} given by~\cite[Prop.~A.4]{CCHS2d}.
Local uniform continuity of the unique fixed point with respect to the inputs follows in the same way as the proof of~\cite[Thm.~7.8]{Hairer14}.

The well-posedness for the other system
\eqref{e:fix-pt-multi-Xbar}+\eqref{e:fix-pt-HU} also follows in the same way, with the following differences.
\begin{itemize}
\item  $\CM \in \hat\cD^{7/4,\eta}_0$ now plays the role of $\CP Y_0$ in $\bar R_Q$.
For the two corresponding terms in~\eqref{eq:CPY_0_hat_Y}, by Lemma~\ref{lem:multiply-hatD}, 
\begin{equ}
\CM \partial \hat{\CY}\in \hat\cD^{\frac14-2\kappa,\eta+\hat\beta-1}_{-\frac32-2\kappa}\;,\quad \text{and} \quad
\hat{\CY}  \partial \CM \in \hat\cD^{\frac14-2\kappa,\eta+\hat\beta-1}_{-\frac12-2\kappa}\;.
\end{equ}
Therefore $\CG_{\mfz} \mathbf{1}_{+} \bar R_Q \in 
\hat\cD^{\gamma,\hat\beta}_{-\frac12-2\kappa}$
as before.

\item $\CM$ likewise plays the role of $\CP Y_0$ in the first two terms in~\eqref{eq:CP_Y_0_Psi}, for which, using~\eqref{e:bbPsi_U0},
\begin{equ}
\CM \partial \bar\bPsi^{U_0} \in \cD^{\frac14-2\kappa,\eta-\frac32-2\kappa}_{-\frac32-2\kappa}\;,
\quad \text{and} \quad
 \bar\bPsi^{U_0}\partial \CM \in \cD^{\frac14-2\kappa,\eta-\frac32-2\kappa}_{-\frac12-2\kappa}\;.
\end{equ}
By Lemma~\ref{lem:Schauder-input}, the first line of the right-hand side of~\eqref{e:fix-pt-multi-Xbar}
is therefore in
  $\cD^{\gamma, \eta+\frac12-3\kappa}_{-5\kappa}$ 
  and thus in
$\hat \cD^{\gamma, \eta+\frac12-3\kappa}_{-5\kappa}$ by Remark~\ref{rem:DDhat-same}.

\item Since $\CW\in\hat\cD_0^{7/4,-\frac12-\kappa}$ and $\CP Y_0\in\hat\cD^{\infty,\eta}_0$, we have $\CP Y_0 \partial\CW \in \hat\cD_0^{3/4,\eta-\frac32-\kappa}$,
$ \CW \partial \CP Y_0 \in \hat\cD_0^{7/4,\eta-\frac32-\kappa}$,
and $\CW\partial \CW\in \hat\cD_0^{3/4,-2-2\kappa}$.
Therefore, for $\kappa>0$ sufficiently small, the second line of \eqref{e:fix-pt-multi-Xbar}
is in $\hat\cD^{\gamma,\hat\beta}_0$
by Lemma~\ref{lem:Schauder-input}.

\item For the term
$\CG_{\mfm} \mathbf{1}_{+}\big( \hat{\CU} \boldsymbol{\Xi} \big)$,
we recall
that $\mcb{K}_\mfm=\bar{\mcb{K}}$ is the abstract integration
realising $K_{\mfm}^{(\eps)} = K^\eps = K*\moll^\eps$,
and  since 
$K^\eps$ is non-anticipative,
we can use the statement in Theorem~\ref{thm:integration} for non-anticipative kernels in the same way as for equation \eqref{e:fix-pt-multi-Y}.

\end{itemize}

It remains to compare the two systems and prove \eqref{e:SS-RSRS}.
First, 
by \eqref{e:bPsi_U0}, \eqref{e:bbPsi_U0},
 \eqref{e:conti-in-zeta} and Lemma~\ref{lem:Schauder-input-KK},
\begin{equs}
| & \bPsi^{U_0} -\bar{\bPsi}^{U_0}|_{\cD^{\infty,-\frac12-2\kappa,\eps}}
\\
&\lesssim
|\mcb{K}_{\mfz}^{\omega_0} (\CP U_0 \bar{\boldsymbol{\Xi}})-
\mcb{K}_{\mfz}^{\bar\omega_0} (\CP U_0 \boldsymbol{\Xi})|_{\cD^{\infty,-\frac12-2\kappa,\eps}}
\\
&\qquad\qquad\qquad\qquad +
| (\mcb{K}_{\mfz}^{\bar\omega_0} 
- \mcb{K}_{\mfm}^{\bar\omega_0}) (\CP U_0 \boldsymbol{\Xi}) |_{\cD^{\infty,-\frac12-2\kappa,\eps}}
\\
&\lesssim 
|\omega_0 - \bar\omega_0|_{\CC^{-\frac52-\kappa}}
+ |\CP U_0 (\boldsymbol{\Xi}- \bar{\boldsymbol{\Xi}})|_{\cD^{\infty,-\frac52-\kappa,\eps}}
+ \eps^\varsigma\;.
\end{equs}
Also, as in \cite[Eq.~(7.19)-(7.20)]{CCHS2d}  one has
\begin{equ}\label{eq:diff_of_noises}
|\boldsymbol{\Xi} - \boldsymbol{\bar{\Xi}}|_{\cD^{\infty,\infty,\eps}}
+
|\mathbf{1}_+ \boldsymbol{\Xi} - \mathbf{1}_+ \boldsymbol{\bar{\Xi}}|_{\hat\cD^{\infty,-\frac52-\kappa,\eps}}
 \lesssim \eps^{\varsigma}\;.
\end{equ}
By \eqref{eq:diff_of_noises} and Remark~\ref{rem:DDhat-same} one then  has
\begin{equ}[e:PsiU-barPsiU]
| \bPsi^{U_0}-\bar{\bPsi}^{U_0}|_{\hat\cD^{\infty,-\frac12-2\kappa,\eps}}
\lesssim \eps^\varsigma\;.
\end{equ}
The same bounds hold for $\partial(\bPsi^{U_0}-\bar{\bPsi}^{U_0})$
with $-\frac12-2\kappa$ replaced by $-\frac32-2\kappa$.
Moreover,  since
\begin{equ}[e:GUhatXi-compare]
\CG (\hat\CU  \bar{\boldsymbol{\Xi}})
 - \bar\CG  (\hat\CU \boldsymbol{\Xi})
 =
 \CG (\hat\CU ( \bar{\boldsymbol{\Xi}} - \boldsymbol{\Xi}) )
 + (\CG -\bar\CG) (\hat\CU \boldsymbol{\Xi})\;,
\end{equ}
its $\hat\cD^{\gamma_{\mfz},\eta_{\mfz}}_{\alpha_{\mfz}} $ norm
is of order $\eps^\varsigma$
using Theorem~\ref{thm:integration}, \eqref{eq:diff_of_noises},
and Lemma~\ref{lem:K-barK-hat}.

Using \eqref{e:PsiU-barPsiU}, and again by  Theorem~\ref{thm:integration}, the 
$\hat\cD^{\gamma_{\mfz},\eta_{\mfz}}_{\alpha_{\mfz}} $ norm of 
\[
\CG (R_Q ) - \CG(\bar R_Q)  = 
\CG (\hat\CY (\partial  \bPsi^{U_0}-  \partial \bar{\bPsi}^{U_0}))
+\CG ( (  \bPsi^{U_0}-   \bar{\bPsi}^{U_0}) \partial \hat\CY)
\]
is also of order $\eps^\varsigma$.

Turning to the first lines of \eqref{e:fix-pt-multi-Y} and \eqref{e:fix-pt-multi-Xbar},
by \eqref{e:conti-in-zeta} and Lemma~\ref{lem:Schauder-input-KK}
we have
\begin{equs}[e:G-che-bar-che]
|\CG_{\mfz}^{ \omega_3}   &
( \bPsi^{U_0}\partial \bPsi^{U_0})
 -
\CG_{\mfz}^{\bar\omega_3}   
(\bar\bPsi^{U_0} \partial \bar\bPsi^{U_0})|_{\cD^{\infty, -5\kappa,\eps}_{-5\kappa}}
\\
&\lesssim 
| \omega_3- \bar\omega_3|_{\CC^{-2-4\kappa}}
+ 
| \bPsi^{U_0}\partial \bPsi^{U_0}
-  \bar\bPsi^{U_0} \partial \bar\bPsi^{U_0}|_{\cD^{\infty,-2-4\kappa,\eps}_{-2-4\kappa}}\;.
\end{equs}
By \eqref{e:bPsi_U0} and \eqref{e:PsiU-barPsiU}, one has
\[
| \bPsi^{U_0}\partial \bPsi^{U_0}
-  \bar\bPsi^{U_0} \partial \bar\bPsi^{U_0}|_{\cD^{\infty,-2-4\kappa,\eps}_{-2-4\kappa}}
\lesssim
\eps^\varsigma\;.
\]
Therefore
\eqref{e:G-che-bar-che} is of order $\eps^\varsigma$.
The differences for the other two terms in
 the first lines of \eqref{e:fix-pt-multi-Y} and \eqref{e:fix-pt-multi-Xbar}
 are bounded in the same way.
We readily obtain~\eqref{e:SS-RSRS}
from the above bounds and the short-time convolution estimates.
\end{proof}

With further assumptions on the models, we can obtain the following improved convergence.
Recall the definition of $\mcY$ from~\eqref{e:def-tildeY}. 

\begin{proposition}\label{prop:improve_local}
Suppose we are in the setting of Lemma~\ref{lem:fixedptpblmclose}.
For every $\tau\in(0,1)$,
equip the space of models $\mathscr{M}_1$ with the pseudo-metric\footnote{More precisely, we redefine $\mathscr{M}_1$ to be the closure of smooth models under this pseudo-metric with $\tau=1$.}
\begin{equ}
d_{1;\tau}(Z,\bar Z) + \sup_{t\in [0,\tau]}\Sigma(\PPi^Z \<IXi>(t), \PPi^{\bar Z}\<IXi>(t))
+ |\PPi^Z \<I[IXiI'Xi]_notriangle> - \PPi^{\bar Z}\<I[IXiI'Xi]_notriangle>|_{\CC([0,\tau],\CC^{-\kappa})}\;.
\end{equ}
Let $\lambda\in(0,1)$ and $r>0$.
Then there exists $\tau\in(0,1)$, depending only on $r$, such that $I\mapsto \CR\mcY \in \CC([\lambda\tau,\tau],\state)$
is a uniformly continuous function on the set
\begin{multline*}
\Big\{
I = (Z,\omega,(Y_0,U_0,h_0))
\in \mathscr{M}_1\times \bigoplus_{\ell=0}^4 \CC^{\alpha_\ell}\times \Omega^\initial
\\
:\, I \textnormal{ is a $1$-input of size $r$ over time $\tau$}
\Big\}\;.
\end{multline*}
The same statement holds with $\mcY$ replaced by $\bar\mcY \eqdef \CP Y_0 + \bar\bPsi^{U_0} + \hat\mcY$
i.e.\ by the $\mcY$-component of $\bar\cS$,
$\omega$ replaced by $\bar\omega$, and $\<IXi>$ and $\<I[IXiI'Xi]_notriangle>$ replaced by
$\<IXiS>$ and $\<I[IXiSI'XiS]_notriangle>$.

Furthermore, as $\eps\downarrow 0$,
\begin{equ}\label{eq:conv_in_state_eps}
\sup_{t\in[\lambda\tau,\tau]}\Sigma(\CR\CY(t),\CR\bar\CY(t)) \to 0
\end{equ}
uniformly over all $\eps$-inputs of size $r$ over time $\tau$
which satisfy 
\begin{equ}\label{eq:consistency_models}
\PPi^Z \<IXi> = \PPi^Z \<IXiS>\;,\quad
\PPi^Z \<I[IXiI'Xi]_notriangle> = \PPi^Z \<I[IXiSI'XiS]_notriangle>\;.
\end{equ}
\end{proposition}

\begin{proof}
We prove first the statement regarding uniform continuity of $I\mapsto \CR\CY$.
It follows from Lemma~\ref{lem:fixedptpblmclose} and continuity of the reconstruction map
that, for some $\tau\in(0,1)$, $\CR \CU$ is a uniformly continuous function into $\CC([\lambda\tau,\tau],\CC^{\frac32-3\kappa})$ of the input.
Truncating $\CY$ at level $\gamma=\frac12-6\kappa$, we can write
\begin{equ}
\mcY (t,x) = u(t,x) \<IXi> + v(t,x)\<I[IXiI'Xi]_notriangle> + b(t,x) \mathbf 1\;,
\end{equ}
where $u = \CR\CU$ and $v=u\otimes u$.
Remark that the structure group acts trivially on $\<IXi>$, $\<I[IXiI'Xi]_notriangle>$, and $\mathbf 1$,
and in particular $b$ is a uniformly continuous function into $\CC([\lambda\tau,\tau],\CC^{\frac12-6\kappa})$ of the input.

The reconstruction of $\CY$ on $[\lambda\tau,\tau]\times \T^3$ is then
\begin{equ}
\CR\CY = u \PPi^Z \<IXi> + v\PPi^Z\<I[IXiI'Xi]_notriangle>+ b\;.
\end{equ}
Observe that $\PPi^Z \<IXi>$ and $\PPi^Z\<I[IXiI'Xi]_notriangle>$
are uniformly continuous functions into $\CC([\lambda\tau,\tau],\state)$ and $\CC([\lambda\tau,\tau],\CC^{-\kappa})$ respectively of the input due to our choice of metric on $\mathscr{M}_1$.
Hence $v\PPi^Z\<I[IXiI'Xi]_notriangle>+b$ is a uniformly continuous function into $\CC([\lambda\tau,\tau],\CC^{-\kappa})$ of the input.
Furthermore, it follows from
Lemmas~\ref{lem:gX_gY_bound},~\ref{lem:X_gX_bound}, and~\ref{lem:heatgr_A_Ad_g_A}
that $(g,Y) \mapsto gY$ is uniformly continuous from every ball in $\mfG^\rho \times \state$  into $\state$.
Therefore
$u \PPi^Z \<IXi>$ is a uniformly continuous functions into $\CC([\lambda\tau,\tau],\state)$
of the input.
It follows from
the same perturbation argument as in~\eqref{eq:perturbation_X}
that $\CR\CY$ is a uniformly continuous function into $\CC([\lambda\tau,\tau],\state)$ of the input.
The proof of the analogous statement for $I\mapsto \CR\bar\mcY$ is identical.

It remains to prove~\eqref{eq:conv_in_state_eps}.
Due to~\eqref{e:SS-RSRS}, $|b-\bar b|_{\CC([\lambda\tau,\tau],\CC^{1/2-6\kappa})}\to 0$,
and, combining with continuity of the reconstruction map, $|\CR\CU-\CR\bar\CU|_{\CC([\lambda\tau,\tau],\CC^{3/2-3\kappa})} \to 0$ as $\eps\downarrow0$.
Both convergences are uniform over all inputs of a given size.
Then~\eqref{eq:conv_in_state_eps} follows from assumption~\eqref{eq:consistency_models}
and a similar argument as above.
\end{proof}

\begin{remark}\label{rem:lambda_zero}
Similar to~\cite[Rem.~7.16]{CCHS2d}, one should be able to set $\lambda=0$ in Proposition~\ref{prop:improve_local} once the initial conditions are taken appropriately (namely $h_0= (\mrd g)g^{-1}$ for some $g\in \mfG^{0,\rho}$)
and extra assumptions on $\omega_\ell,\bar\omega_\ell$ are made of the type $G*\omega_0\in\CC([0,1],\state)$ and $G*\omega_\ell\in\CC([0,1],\CC^{-\kappa})$ for $\ell=1,2,3,4$.
However, the weaker statement with $\lambda>0$ suffices for our purposes.
\end{remark}

\subsection{Renormalised local solutions}\label{subsec:renorm_sol} 

The main result of this subsection is Proposition~\ref{prop:SPDEs_conv_zero_state}
which shows that the SPDEs \eqref{eq:renorm_bar_a} and \eqref{eq:renorm_b} below, which we will later take to coincide with~\eqref{eq:SPDE_for_bar_A} (up to $o(1)$)
and~\eqref{eq:SPDE_for_B} respectively, converge to the same limit locally in time in $\CC([0,\tau],\state\times\tilde{\mfG}^{0,\rho})$.

We start by introducing some decomposition for our renormalisation constants.
Recalling that we have fixed a space-time mollifier $\moll$,  Proposition~\ref{prop:renorm_g_eqn} defines the renormalisation operators $C^\eps_{\Higgs} \in L_G(\higgsvec,\higgsvec)$ and 
$C^\eps_{\YM}, C^\eps_{\Gauge}, C^{0,\eps}_{\Gauge} \in L_{G}(\mfg,\mfg)$,
which depend on $\sig^\eps \in \R$. 
Recall from Propositions \ref{prop:mass_term} and \ref{prop:renorm_g_eqn} that
for $j \in \{1,2\}$ there exist
$C^\eps_{j,\Higgs} \in L_G(\higgsvec,\higgsvec)$ and
$C^\eps_{j,\YM}, C^\eps_{j,\Gauge}, C^{0,\eps}_{j,\Gauge} \in L_{G}(\mfg,\mfg)$
which are all independent of $\sig^{\eps}$, such that
for $C^{\bullet}_{\bullet} \in \{ C^\eps_{\YM},\ C^\eps_{\Higgs},\ C^\eps_{\Gauge},\ C^{0,\eps}_{\Gauge}\}$,
\begin{equ}
C^{\bullet}_{\bullet} =  (\sig^\eps)^2  C^{\bullet}_{1,\bullet} +  (\sig^\eps)^4  C^{\bullet}_{2,\bullet} \;.
\end{equ}


We first show in Proposition~\ref{prop:SPDEs_conv_zero_space-time}
that the SPDEs~\eqref{eq:renorm_bar_a} and~\eqref{eq:renorm_b} converge (to the same limit)
locally in time as \textit{space-time distributions} in $\CC^{-\frac12-\kappa}((0,\tau)\times\T^3)$.
This allows us to use a small noise limit argument in Lemmas~\ref{lem:consts-soft1_bounded} and~\ref{lem:consts-soft2_bounded} to show that $C^\eps_{j,\YM}-C^\eps_{j,\Gauge}$  and $C^{0,\eps}_{j,\Gauge}$ are bounded in $\eps$.
This in turn allows us to improve the mode of convergence in Proposition~\ref{prop:SPDEs_conv_zero_state}
to $\CC([0,\tau],\state\times\tilde{\mfG}^{0,\rho})$ by comparing~\eqref{eq:renorm_bar_a} and~\eqref{eq:renorm_b} to an SPDE with additive noise.

We define operators in $L_G (E,E)$  given by 
\begin{equs}
C_{\YMH}^\eps \eqdef  
(C^\eps_{\YM} )^{\oplus 3} \oplus C^\eps_{\Higgs}\;, \quad 
C_{\GaugeH}^\eps & \eqdef (C^\eps_{\Gauge})^{\oplus 3} \oplus 0\;,\\
 \text{and} 
\qquad
C_{\GaugeH}^{0,\eps} & \eqdef 
(C^{0,\eps}_{\Gauge})^{\oplus 3} \oplus 0\;.
\end{equs}

To state Proposition~\ref{prop:SPDEs_conv_zero_space-time},
we fix $\sig \in \R$, and 
some choice of constants $(\sig^{\eps} \in \R: \eps \in (0,1])$ 
with $\sig^{\eps} \rightarrow \sig$.
We also fix
$\mathring{C}_{\A}, \mathring{C}_{\h} \in L_{G}(\mfg,\mfg)$, and
some choice of  $\mathring{C}^{\eps}_{\A}, \mathring{C}^{\eps}_{\h} \in L_{G}(\mfg,\mfg)$ 
with $\mathring{C}^{\eps}_{\A} \rightarrow \mathring{C}_{\A}$, and $\mathring{C}^{\eps}_{\h} \rightarrow \mathring{C}_{\h}$ as $\eps \downarrow 0$,
and then define as in
\eqref{e:ring-C-zh} 
\begin{equ}[e:defConstbarX]
\mathring{C}_{\mfz}^\eps = 
(\mathring{C}_{\A}^\eps)^{\oplus 3}
\oplus (-m^2) \;,\quad
\mathring{C}_{\mfh}^\eps = 
(\mathring{C}_{\h}^\eps)^{\oplus 3}
\oplus 0\;.
\end{equ}
Finally we will simply write $\mrd g g^{-1}$ for $\mrd g g^{-1}\oplus 0 \in E$ and recall the shorthand notation~\eqref{e:simpleNotation}
and that $\mbF$\label{pageref:mbF} is the filtration generated by the white noise.

\begin{proposition}\label{prop:SPDEs_conv_zero_space-time}
Suppose $\moll$ is non-anticipative
and fix any  initial data
$\bar{x} = (\bar a,\bar\phi) \in \state$
and $g(0) \in {\mfG}^{0,\rho}$.
Consider the system of equations for $\bar X=(\bar A,\bar \Phi)$
\begin{equ}[eq:renorm_bar_a]
\partial_t \bar X
= \Delta \bar X
+ \bar X \partial \bar X
+ \bar X^3
+ C_{\mfz}^\eps \bar X
+ \bar{C}_{\mfh}^\eps \mrd \bar g \bar g^{-1}
+ \sig^\eps \moll^\eps * (\bar g \xi)
\end{equ}
where $\bar g $ solves the last equation in \eqref{eq:SPDE_for_bar_A} with $\bar A$ given by 
the first component of the $\bar X$ from \eqref{eq:renorm_bar_a} and we define $\bar g \xi$ 
as in~\eqref{eq:group_action_E} with $\bar g\equiv 1$ on $(-\infty,0)$.
Consider further
the system of equations  for $Y=(B,\Psi)$
\begin{equ}[eq:renorm_b]
\partial_t Y
= \Delta Y
+ Y \partial Y
+ Y^3
+ C_{\mfz}^\eps Y
+ C_{\mfh}^\eps \mrd g g^{-1}
+ \sig^\eps  (g \xi^\eps)
\end{equ}
where $g $ solves the last equation in  \eqref{eq:SPDE_for_B} with $B$ given by the first component of $Y$,
and $g \xi^\eps$ is again defined by~\eqref{eq:group_action_E}.
Both systems are taken
with initial conditions $\bar X(0)=Y(0)=\bar{x}$ and $\bar{g}(0) = g(0)$,
and with renormalisation operators
\begin{equ}[e:translated-consts]
C_{\mfz}^\eps = \mathring{C}_{\mfz}^\eps
 +C_{\YMH}^\eps \;,
\qquad
C_{\mfh}^\eps = \mathring{C}_{\mfh}^\eps
 +C_{\GaugeH}^{\eps} \;,
 \qquad
 \bar{C}_{\mfh}^\eps = \mathring{C}_{\mfh}^\eps
 +C_{\GaugeH}^{0,\eps} \;.
\end{equ}
Then, for $\eps>0$, the system~\eqref{eq:renorm_bar_a} is well-posed in the same sense as Theorem~\ref{thm:gauge_covar}\ref{pt:delta_to_zero}.

Furthermore, let $(U,h)$ (resp.\ $(\bar U,\bar h)$) be defined by $g$ (resp.\ $\bar g$) as in \eqref{eq:h_and_U_def}.
Then
there exists $\tau>0$, depending only on $\sup_{\eps\in(0,1)}(|\mathring{C}_{\A}^{\eps}|+
|\mathring{C}_{\h}^{\eps}| +|\sig^{\eps}|)$,
the realisation of the noise on $[-1,2]\times\T^3$,
and the size of $(\bar x,g(0))$ in $\state\times\mfG^\rho$,
such that $\tau$ is an $\mbF$-stopping time and such that
$(\bar X,\bar U,\bar h)$ and 
$(Y,U,h)$ converge in probability in $\CC^{-\frac12-\kappa}((0,\tau)\times \T^3)$ to the same limit as $\eps \downarrow 0$.
More precisely, there exists a $\CC^{-\frac12-\kappa}((0,\infty)\times\T^3)$-valued random variable $(X,V,f)$ such that
\begin{equ}
|(X,V,f)-(\bar X,\bar U,\bar h)|_{\CC^{-\frac12-\kappa}((0,\tau)\times\T^3)}+|(X,V,f)-(Y,U,h)|_{\CC^{-\frac12-\kappa}((0,\tau)\times\T^3)} \to 0 
\end{equ}
in probability as $\eps\downarrow 0$.
Moreover, $(X,V,f)$ depends on the sequences
 $\mathring{C}_{\A}^{\eps}$, $\mathring{C}_{\h}^{\eps}$ and $\sig^{\eps}$ 
 only through their limits $\mathring{C}_{\A}$, $\mathring{C}_{\h}$, and $\sig$. 
\end{proposition}
%
\begin{proof}
We first claim that \eqref{eq:renorm_bar_a} is well-posed and
is moreover the renormalised equation satisfied by 
 the reconstruction  of $\CY$ in \eqref{e:def-tildeY}
 with respect to $Z^{0,\eps}_{\BPHZ}$, where $\hat\CY$ solves the fixed point problem
  \eqref{e:fix-pt-multi-Xbar}.
The other inputs $\bar\omega_\ell$ and $\CW$ for~\eqref{e:fix-pt-multi-Xbar} are defined by~\eqref{eq:CW_def} and~\eqref{e:bar_omega0123}.
Since $Z^{0,\eps}_{\BPHZ}$ is not a smooth model,  
  the proof of the claim for \eqref{eq:renorm_bar_a}  
  goes via obtaining the corresponding renormalised equation for the model $Z^{\delta,\eps}_{\BPHZ}$ and then taking the limit $\delta \downarrow 0$, which is justified by the convergence \eqref{e:conv-Z-delta}
and Lemmas~\ref{lem:omegas_converge}-\ref{lem:compatible}.

By \cite[Thm.~5.7]{BCCH21} and \cite[Prop.~5.68]{CCHS2d},
together with the discussion in Section~\ref{sec:sol-multi} (see in particular Remark~\ref{rmk:multi-coherent}),
 the  reconstruction  with respect to $Z^{0,\eps}_{\BPHZ}$
 of $(\CY,\CU,\CH)$  yields 
  $(\bar X, \bar U, \bar h)$ where we write $\bar X=(\bar A,\bar\Phi)$,
 and by
\eqref{eq:no_u_h_renorm} from 
Proposition~\ref{prop:renorm_g_eqn},
$(\bar U,\bar h)$ satisfies the ``barred'' version\footnote{That is with $\bar{U}$ and $\bar{h}$ instead of $U$ and $h$ and also with $B$ replaced by $\bar{A}$.} of  \eqref{e:final_hU} which is not renormalised.

We first claim that $\bar U$ satisfies the conditions of Definition~\ref{def:bar-mcb-A}. 
Indeed, if $\bar g$  solves~\eqref{eq:SPDE_for_bar_A} for the given $\bar A$,
then Lemma~\ref{lemma:linear_g_system} with this  given $\bar A$
implies that $(\mrd \bar g) \bar g^{-1}$ and $\brho (\bar g)$  solve
 the ``barred'' version of \eqref{e:final_hU}
 with the same initial condition 
 as that of $(\bar U,\bar h)$,
thus $(\bar U,\bar h)=(\brho (\bar g) ,(\mrd \bar g) \bar g^{-1})$.

We  apply Proposition~\ref{prop:renorm_g_eqn}, and take the limit $\delta \downarrow 0$ of the renormalisation constants
to conclude that
$\bar X$ solves~\eqref{eq:renorm_bar_a} but with 
$\bar g \xi_i $ replaced by $\bar U\xi_i$ 
and $(\mrd \bar g) \bar g^{-1}$ replaced by $\bar h$.
It remains to observe that if $(\bar X', \bar g)$ is the solution to~\eqref{eq:renorm_bar_a} and the last equation in~\eqref{eq:SPDE_for_bar_A},
where $\bar g $ is such that $\brho(\bar g) = \bar U$ and $\mrd \bar g \bar g^{-1}=\bar h$,
then similarly  as above using Lemma~\ref{lemma:linear_g_system},
we have  $\bar X=\bar X'$, which proves the claim that 
$\bar X$ solves \eqref{eq:renorm_bar_a}.

A similar argument shows that 
\eqref{eq:renorm_b}
is  the renormalised equation satisfied by 
 the reconstruction  of $\CY$ in \eqref{e:def-tildeY}
 with respect to $Z^{0,\eps}_{\BPHZ}$, where $\hat\CY$ solves the fixed point problem
\eqref{e:fix-pt-multi-Y}.

The probabilistic input from Lemmas~\ref{lem:conv_of_models2},~\ref{lem:omegas_converge}, and~\ref{lem:compatible},
together with the deterministic Lemma~\ref{lem:fixedptpblmclose}, Remark~\ref{rem:PPi_not_needed}, and continuity of the reconstruction map $\CR\colon \cD^{\gamma_\mft,\eta_\mft}_{\alpha_\mft} \to \CC^{\alpha_\mft\wedge\eta_\mft}((0,\tau)\times\T^3)$,
prove the statement concerning convergence in probability as $\eps\downarrow0$.
The fact that $\tau$ is a stopping time follows from the fact that, since $\moll$ and $K$ are non-anticipative, one only needs to know the noise up to time $t>0$ to determine if the corresponding $\eps$-input is of size $r>0$ over time $t$.

Finally, the fact that $(X,V,f)$ in the proposition statement can be chosen to depend only on the limits $\mathring{C}_{\A}$, $\mathring{C}_{\h}$, and $\sig$
follows from the continuity of the fixed point problems~\eqref{e:fix-pt-multi-Y}+\eqref{e:fix-pt-HU}
and \eqref{e:fix-pt-multi-Xbar}+\eqref{e:fix-pt-HU}
with respect to the coefficients of the non-linearity, and with respect to the input model $Z^\eps_{\BPHZ}$ and distributions $\omega^\eps_\ell,\bar\omega^{\eps,0}_\ell$ which are multilinear functions of $\sig^\eps$.
\end{proof}

Our next goal is to show boundedness of  $C_\mfz^\eps - C^\eps_\mfh$ and $\bar{C}^\eps_\mfh$ appearing in Proposition~\ref{prop:SPDEs_conv_zero_space-time}.
This will allow us to relate $\bar X$ and $Y$ to a variant of equation~\eqref{eq:SPDE_for_A}, which we study in Appendix~\ref{app:evolving_rough_g},
and for which continuity into $\state$ at time $t=0$ is simpler to show since the noise appears additively.

We now take advantage of the parameter $\sig$. For clarity we include the following remark. 

\begin{remark}\label{rem:role_of_sigma}
In  Sections~\ref{sec:renorm-A} and~\ref{sec:gauge}, the parameter $\sig^{\eps}$ which multiplies the noise terms in \eqref{e:SPDE-for-X} and \eqref{eq:renorm_bar_a}--\eqref{eq:renorm_b} is incorporated into our noise assignment,\footnote{For instance, the noise term in the abstract fixed point problem in  Section~\ref{sec:renorm-A}  is just $\bar{\bXi}$, not $\sig^{\eps}\bar{\bXi}$} that is we look at the BPHZ lifts of the noise $\sig^{\eps} \xi^{\eps}$ in Section~\ref{sec:renorm-A}  (and $\sig^{\eps} \xi^{\eps}$, $\sig^{\eps}\xi$ in Section~\ref{sec:gauge}). 

In the results below we perform small noise limits $\lim_{\eps \downarrow 0 } \sig^{\eps} = 0$. 
An important fact is that in this limit the corresponding BPHZ lifts converge to the BPHZ lift of the ``noise'' $0$ (which is just the canonical lift of $0$, that is the unique admissible model that vanishes on any tree containing a label in $\Lab_{-}$). 
In fact, the full ``$\eps$-inputs''  converge to $0$ (see the $\sig$ dependence in the bounds of Lemma~\ref{lem:omegas_converge}), which allows us to argue that the corresponding solutions converge (in the appropriate sense), as $\eps \downarrow 0$ to the solution we obtain from the zero noise. 
\end{remark}
\begin{lemma}\label{lem:consts-soft1_bounded}
For $j \in \{1,2\}$, let $D_{j}^{\eps}\eqdef
  C^{\eps}_{j,\YM} - C^{\eps}_{j,\Gauge}\in L_G(\mfg,\mfg)$.
Then
\begin{equ}
\sup_{\eps\in(0,1]} \big(|D_1^\eps| + |D_2^\eps|\big) < \infty\;.
\end{equ}
\end{lemma}

\begin{proof}
We proceed by contradiction, supposing that we have a subsequence of $\eps \downarrow 0$ along which  $\lim_{\eps \downarrow 0} r_{1}^{\eps} + r_{2}^{\eps} = \infty$ where  $r_{j}^{\eps} \eqdef |D_{j}^{\eps}|^{\frac{1}{2j}}$. 
We then set $\bar{\sig}^{\eps} = (r_{1}^{\eps} + r_{2}^{\eps})^{-1}$. 
By passing to another subsequence  we can assume that 
\begin{equ}\label{eq:subseq_conv_constants1}
\lim_{\eps \downarrow 0} 
\left(
(\bar{\sig}^{\eps})^{2} D_{1}^{\eps}, 
(\bar{\sig}^{\eps})^{4} D_{2}^{\eps}
\right)
=
(\hat{C}_{1},\hat{C}_{2})
\not = (0,0)\;.
\end{equ}
It follows that we can find a constant $z \in \R$ such that if we set 
$\sig^{\eps} \eqdef z \bar{\sig}^{\eps}$ 
and $\hat{C}^{\eps} \eqdef (\sig^{\eps})^{2} D_{1}^{\eps}+
(\sig^{\eps})^{4} D_{2}^{\eps}$ then $\hat{C}^{\eps} \rightarrow \hat{C} = z^{2} \hat{C}_{1} + z^{4} \hat{C}_{2} \not = 0$ along the subsequence for which we have \eqref{eq:subseq_conv_constants1}.

Now we set 
\begin{equ}[e:pf-soft1]
\mathring{C}^{\eps}_{\A} = 0\;,
\qquad 
\mathring{C}^{\eps}_{\h} = \hat{C}^{\eps}\;,
\end{equ}
as the constants in \eqref{e:defConstbarX} and Proposition~\ref{prop:SPDEs_conv_zero_space-time}.
We then note that  the solution $Y$ to \eqref{eq:renorm_b} for initial data $(\bar x,g(0)) \in \state \times \mfG^{0,\rho}$, with renormalisation constants $C^{\eps}_{\mfz},C^{\eps}_{\mfh}$ given by \eqref{e:translated-consts},
namely the solution to
\begin{equs}[e:Y-contra]
\partial_t Y
= \Delta Y
 &+ Y \partial Y
+ Y^3 
+ ((C_{\YM}^\eps)^{\oplus 3} \oplus C_{\Phi}^\eps ) Y 
\\
&+ ((\hat{C}^{\eps} + C_{\Gauge}^\eps)^{\oplus 3} \oplus 0) \,\mrd g g^{-1}
+ \sig^\eps  g \xi^\eps\;,
\end{equs}
 is equal to $g\act X$ where $X$ solves \eqref{e:SPDE-for-X} with initial data $x\eqdef \bar g(0)^{-1}\act \bar x$,
with constants \eqref{e:SYM_constants} and $\mathring{C}^{\eps}_{\A} = 0 $ and $\sig^{\eps}$ as above, namely $X$ solves
\begin{equ}[e:X-contra]
\partial_t X = \Delta X + X \partial X + X^3 
+( (C_{\YM}^{\eps})^{\oplus 3} \oplus C_{\Phi}^{\eps} ) X + \sig^\eps \xi^\eps\;.
\end{equ}
The key fact is that with our choices of constants one has $C^{\eps}_{\YM}=\hat{C}^{\eps} + C_{\Gauge}^\eps$.
 
As  $\eps \downarrow 0 $ we have that $\sig^{\eps} \rightarrow 0$ and so, by Remark~\ref{rem:role_of_sigma}, the solutions to \eqref{e:X-contra}  converge (in the sense given by Proposition~\ref{prop:SPDEs_conv_zero_space-time}) to the solution of the deterministic PDE 
\begin{equ} 
\partial_t X = 
\Delta X+ X\partial X + X^3  - (0^{\oplus 3} \oplus m^2) X  \;.
\end{equ}
(Recall that $C_{\Phi}^{\eps} 
=C_{\Higgs}^{\eps} -m^2 $.)
On the other hand, as $\eps \downarrow 0$ along our subsequence, the solutions to  \eqref{e:Y-contra} 
converge to the solution of the deterministic PDE
\begin{equ}
\partial_t Y = 
\Delta Y + Y\partial Y + Y^3 
-(0^{\oplus 3} \oplus m^2) Y
+  ( \hat{C}^{\oplus 3} \oplus 0)\, \mrd g g^{-1} \;,
\end{equ}
where $\mrd g g^{-1}= h$ solves the deterministic PDE \eqref{e:final_hU}, which is non-zero for generic initial conditions.
We remind the reader that convergence as $\eps \downarrow 0$ of the solutions to \eqref{e:final_hU} is a
straightforward consequence of Lemma~\ref{lem:fixedptpblmclose} and continuity of the reconstruction operator.

Since on the one hand the relation $g\act X = Y$ is preserved under the limit $\eps \downarrow 0$
but on the other hand $g\act X$ solves the same equation as $X$ (i.e.\ without the last term appearing in the equation for $Y$), this gives us the desired contradiction.\footnote{Note that while Proposition~\ref{prop:SPDEs_conv_zero_space-time} gives us convergence in a space where we don't have continuity at $t=0$, the limiting deterministic equations certainly have continuity at $0$ for smooth initial data and therefore must be different as space-time distributions on $(0,T)\times\T^3$ for sufficiently small $T>0$.}
\end{proof}
\begin{lemma}\label{lem:consts-soft2_bounded}
Suppose $\moll$ is non-anticipative. Then
\begin{equ}
\sup_{\e\in(0,1]} \big(|C_{1,\Gauge}^{0,\eps}|+|C_{2,\Gauge}^{0,\eps}|\big) < \infty\;.
\end{equ}
\end{lemma}
\begin{proof}
We proceed similarly to the proof of Lemma~\ref{lem:consts-soft1_bounded},
 except that now
instead of comparing  $Y$ with a gauge transformation of $X$,  we 
compare $\bar X$ with $X$.

Arguing by contradiction,  we suppose as in the proof of Lemma~\ref{lem:consts-soft1_bounded} that we have a subsequence $\eps \downarrow 0$  and  $\sig^{\eps} \rightarrow 0$ with 
\[
(\sig^{\eps})^{2} C_{1,\Gauge}^{0,\eps}
+(\sig^{\eps})^{4} C_{2,\Gauge}^{0,\eps}
= C_{\Gauge}^{0,\eps} \rightarrow C_{\Gauge}^{0,0} \not = 0 \;.
\]
This time, we consider $(\bar X,\bar g)$ as in Proposition~\ref{prop:SPDEs_conv_zero_space-time} \dash namely, with $\bar{X}= (\bar{A},\bar{\Phi})$ solving \eqref{eq:renorm_bar_a} and $\bar{g}$ as in  \eqref{eq:SPDE_for_bar_A} \dash and choose
\begin{equ}[e:choiceConstants]
\mathring{C}^{\eps}_{\A} \eqdef 0
\quad \text{and} \quad
\mathring{C}^{\eps}_{\h} \eqdef   - C_{\Gauge}^{0,\eps}     
\end{equ}
as the constants in \eqref{e:defConstbarX}.
For any $\eps > 0$, define the
transformation of the white noise $T^\eps\colon \xi\mapsto T^\eps(\xi)$ by $T^\eps(\xi)=\bar g \xi$
on $[0,\tau]$ and $T^\eps(\xi)=\xi$ on $\R\setminus[0,\tau]$.
Recall here that $\tau>0$ is the $\mbF$-stopping time as in Proposition~\ref{prop:SPDEs_conv_zero_space-time}.
In particular, since $\moll$ is non-anticipative, $\bar g$ is adapted to $\mbF$, it follows 
that this operation is well defined and that moreover $T^\eps(\xi)\eqlaw \xi$.\footnote{This equality in law follows from applying It\^o's isometry to $T^\eps(\xi)$, using that $\bar{g}$ acts by an orthogonal representation, and then appealing to L\'{e}vy's characterisation of Brownian motion.}

Then, by definition, $\bar X= \tilde X\eqdef (\tilde A,\tilde \Phi)$ on $[0,\tau]$ where 
$\tilde X$ and $\tilde g$ solve
\begin{equs}[eq:Xg-gaugecovar]
\partial_t \tilde X &= \Delta \tilde X + \tilde X \partial \tilde X + \tilde X^3 
+ C^\eps_\mfz \tilde X +
\sig^\eps \moll^\eps* T^\eps(\xi)\;,
\\
(\partial_t \tilde {g})  \tilde{g}^{-1} 
&= \partial_j((\partial_j \tilde{g})\tilde{g}^{-1})+ [\tilde{A}_j, (\partial_j \tilde{g})\tilde{g}^{-1}]\;,
\\
\tilde X(0) & = \bar x \in \state \;, \qquad \tilde g(0) = \bar g(0)\in\mfG^{0,\rho}\;. 
\end{equs}
(The reason why the term $\bar C^\eps_{\mfh} \mrd\tilde g \tilde g^{-1}$ is absent is that 
$\bar C^\eps_{\mfh} = 0$ thanks to our choice \eqref{e:choiceConstants}.)
Remark that, for every $\eps>0$, 
the above equations are well-posed and $(\tilde X,\tilde g)$ are smooth for $t>0$ since $T^\eps(\xi)\eqlaw \xi$,
and thus $(\tilde X,\tilde g)$ can be extended to maximal solutions in $(\state\times\mfG^{0,\rho})^\sol$.

On the one hand, by Remark~\ref{rem:role_of_sigma},
$\bar X$ converges in probability on $[0,\tau]$ (in the sense given by Proposition~\ref{prop:SPDEs_conv_zero_space-time}) to the solution of the deterministic equation
\begin{equ}\label{eq:deterministic1}
\partial_t \bar{X} = 
\Delta \bar{X} + \bar{X} \partial \bar{X} + \bar{X}^3 - (0^{\oplus 3}\oplus m^2) \bar{X}
  - ( (C_{\Gauge}^{0,0})^{\oplus 3}\oplus 0) \,\mrd \bar g \bar g^{-1} \;,
\end{equ}
where $\bar{g}$ solves the relevant equation in \eqref{eq:SPDE_for_bar_A}.
%
On the other hand, since $T^\eps(\xi)\eqlaw \xi$, it follows from Theorem~\ref{thm:local-exist-sigma}
that $\tilde X$ converges in law (and thus in probability) in $\state^\sol$
to the solution of the deterministic equation 
\begin{equ}\label{eq:deterministic2}
\partial_t \tilde X = 
\Delta \tilde X+ \tilde X \partial \tilde X + \tilde X^3 - (0^{\oplus 3}\oplus m^2)\tilde X
 \;.
\end{equ}
Since we can choose $\bar{g}_0$ so that $\mrd \bar{g}_{0} g_{0}^{-1} \not = 0$, 
we can find initial data for which the (deterministic) solution of \eqref{eq:deterministic1} differs 
from that of \eqref{eq:deterministic2} for all sufficiently small times, which gives the desired contradiction.  
\end{proof}

We are now ready to improve the mode of convergence in Proposition~\ref{prop:SPDEs_conv_zero_space-time}.

\begin{proposition}\label{prop:SPDEs_conv_zero_state}
The conclusion of Proposition~\ref{prop:SPDEs_conv_zero_space-time} holds with the improvement that $(\bar X,\bar U,\bar h)$ and $(Y,U,h)$ converge in probability in $\CC([0,\tau],\state\times\tilde{\mfG}^{0,\rho})$ as $\eps\downarrow0$.
More precisely, the
random variable $(X,V,f)$ in Proposition~\ref{prop:SPDEs_conv_zero_space-time}
takes values in $\CC([0,\tau],\state\times\tilde{\mfG}^{0,\rho})$ and
\begin{equ}
|(X,V,f)-(\bar X,\bar U,\bar h)|_{\CC([0,\tau],\state\times\tilde{\mfG}^{0,\rho})}+|(X,V,f)-(Y,U,h)|_{\CC([0,\tau],\state\times\tilde{\mfG}^{0,\rho})} \to 0 
\end{equ}
in probability as $\eps\downarrow 0$.
\end{proposition}
\begin{proof}
Continuing from the proof of Proposition~\ref{prop:SPDEs_conv_zero_space-time} \dash 
Corollary~\ref{cor:SHE_conv_state} and Proposition~\ref{prop:improve_local} provide the additional probabilistic and deterministic input respectively
to conclude the desired statement with $\CC([0,\tau],\state\times\tilde{\mfG}^{0,\rho})$ replaced by $\CC([\lambda\tau,\tau],\state\times\tilde{\mfG}^{0,\rho})$ for any fixed $\lambda\in(0,1)$.

To lighten notation, we will only consider the $\bar X$ and $Y$ components henceforth; how to add back $\bar U,\bar h$ and $U,h$ will be clear.
We assume henceforth that $\tau\in(0,1)$.

For $\eps>0$, we extend $Y$ to a function in $\CC([0,1],\state)$ by stopping it at $\tau$.
It follows from the above that there exists a $\CC((0,1],\state)$-valued random variable $Z$ such that $\sup_{t\in (\lambda,1]}\Sigma(Z(t),Y(t))\to 0$ in probability as $\eps\downarrow0$ for every $\lambda>0$.
We now claim that 
\begin{equ}\label{eq:Y_conv_at_zero}
\lim_{\lambda\downarrow0}\limsup_{\eps\downarrow0} \P
\Big[
\sup_{t\in[0,\lambda]} \Sigma(Y(t),Y(0)) > \delta
\Big]
= 0\qquad \forall \delta>0\;,
\end{equ}
and likewise for $\bar X$.
To prove~\eqref{eq:Y_conv_at_zero},
a similar calculation as in~\cite[Sec.~2.2]{CCHS2d} implies that $Y=g\act X$ where $X,g$ solve~\eqref{e:SPDE-for-X_with_g}-\eqref{eq:SPDE_for_g_coupled} with operator
\begin{equ}\label{eq:c^eps_def}
c^\eps\eqdef  C^\eps_\mfh - C^{\eps}_\mfz\restr_{\mfg^3} = (\mathring{C}_{\h}^\eps)^{\oplus 3} + (C^\eps_{\Gauge})^{\oplus3}- (\mathring C^\eps_{\A})^{\oplus3}-(C^\eps_{\YM})^{\oplus3}
\end{equ}
and initial condition $x= g(0)^{-1} \act\bar x$.

Observe that $\sup_{\eps\in(0,1]} |c^\eps|<\infty$ due to Lemma~\ref{lem:consts-soft1_bounded}.
It follows from a similar proof to that of Theorem~\ref{thm:local-exist-sigma} with minor changes as indicated in the proof of Lemma~\ref{lem:flow_for_rough_g},
that~\eqref{eq:Y_conv_at_zero} holds with $\Sigma(Y(t),Y(0))$ replaced by $\Sigma(X(t),X(0)) + |g(t)-g(0)|_{\mfG^{0,\rho}}$.
Recalling from Proposition~\ref{prop:group_action} that $(g,X)\mapsto g \act X$ is a uniformly continuous function from every ball in $\mfG^{0,\rho}\times \state$ into $\state$,
we obtain~\eqref{eq:Y_conv_at_zero} from the identity $Y=g \act X$.

Observe that~\eqref{eq:Y_conv_at_zero} implies that $Z$ can be extended to a $\CC([0,1],\state)$-valued random variable by defining $Z(0)=Y(0)$
and $\sup_{t\in[0,1]}\Sigma(Y(t),Z(t)) \to 0$ in probability as $\eps\downarrow0$.

We now claim that~\eqref{eq:Y_conv_at_zero} holds with $Y$ replaced by $\bar X$, where we also extend $\bar X$ to a function in $\CC([0,1],\state)$ by stopping it at $\tau$.
Recall that $\tau>0$ is the $\mbF$-stopping time, which depends only on $\xi\restr_{[-1,2]\times\T^3}$ (through the size of its BPHZ model),
for which $\bar X$ is well-defined and converges in probability in $\CC([\lambda\tau,\tau],\state)$ as $\eps\downarrow0$ to the same limit $Z$ as $Y$ for all $\lambda\in(0,1)$.

To prove the claim, for $\eps>0$, define as in the proof of Lemma~\ref{lem:consts-soft2_bounded}
the transformation of the white noise $T\colon \xi\mapsto T(\xi)$,
where $T(\xi)=\bar g \xi$ on $[0,\tau]$ and $T(\xi)=\xi$ on $\R\setminus[0,\tau]$.
Then $\bar X=\tilde X\eqdef(\tilde A,\tilde \Phi)$ on $[0,\tau]$
where $(\tilde X,\tilde g)\in (\state\times\mfG^{0,\rho})^\sol$ denotes the maximal solution to
\begin{equs}
\partial_t \tilde X &= \Delta \tilde X + \tilde X \partial \tilde X + \tilde X^3 
+ C^\eps_\mfz \tilde X + \bar{C}^\eps_\mfh\mrd \tilde g \tilde g^{-1}  +
\sig^\eps \moll^\eps* T(\xi)\;,
\\
(\partial_t {\tilde g})  {\tilde g}^{-1} 
&= \partial_j((\partial_j {\tilde g}){\tilde g}^{-1})+ [{\tilde A}_j, (\partial_j {\tilde g}){\tilde g}^{-1}]\;,
\\
\tilde X(0) & = (\bar a,\bar \phi) \in \state \;, \qquad \tilde g(0) = \bar g(0)\in\mfG^{0,\rho}\;. 
\end{equs}
Since $\moll$ is non-anticipative,
Lemma~\ref{lem:consts-soft2_bounded} yields
$\sup_{\eps\in(0,1]}|\bar{C}^\eps_\mfh|<\infty$.
Since $T(\xi)\eqlaw \xi$,
it follows again from the same proof as Theorem~\ref{thm:local-exist-sigma} with small 
changes as indicated in the proof of Lemma~\ref{lem:flow_for_rough_g}\slash\ref{lem:flow_for_rough_g2} 
that~\eqref{eq:Y_conv_at_zero} holds with $Y$ replaced by $\tilde X$.
(In fact, if we knew that $ \bar{C}^\eps_\mfh $ converges to a finite limit as $\eps\downarrow0$,
then we would know that $(\tilde X,\tilde g)$ converges in $(\state\times \mfG^{0,\rho})^\sol$.)
Since $\bar X= \tilde X$ on $[0,\tau]$, it also follows that~\eqref{eq:Y_conv_at_zero} holds with $Y$ replaced by $\bar X$ as claimed.

Finally, recalling that $\sup_{t\in(\lambda,1]}\Sigma(\bar X(t),Y(t)) \to 0$ in probability as $\eps\downarrow0$ for all $\lambda\in(0,1)$,
it follows that $\sup_{t\in[0,1]}\Sigma(\bar X(t),Y(t))\to 0$ in probability as $\eps\downarrow0$ as required.
\end{proof}

\subsection{Convergence of maximal solutions}
The main result of this subsection is the following.
\begin{proposition}\label{prop:bar_X_max_sols}
Suppose we are in the setting of Proposition~\ref{prop:SPDEs_conv_zero_space-time}.
Then $(\bar X,\bar U, \bar h)$, obtained as in Theorem~\ref{thm:gauge_covar}\ref{pt:delta_to_zero},
converge in probability as maximal solutions in $(\state\times\tilde{\mfG}^{0,\rho})^\sol$
as $\eps\downarrow0$.
\end{proposition}

\begin{remark}
The proof of Proposition~\ref{prop:bar_X_max_sols}
applies with several changes to
show convergence of $(Y,U,h)$ as maximal solutions.
However, we will derive this result below more simply in Corollary~\ref{cor:Y_system_conv} using pathwise gauge equivalence.
\end{remark}

\begin{proof}
As in~\cite[Sec.~7.2.4]{CCHS2d}, it will be convenient to consider a 
different abstract fixed point problem from~\eqref{e:fix-pt-multi-Xbar}, which solves instead for the `remainder' of $\bar X$.
Like in~\cite{CCHS2d}, this new fixed point will encode information of $\bar U$ on an earlier time interval
which is necessary to make sense of the term $\moll^\eps * (\bar g \xi)$.
In addition to this, and unlike in~\cite{CCHS2d},
it is important that this new fixed point problem will have an improved regularity in the initial condition for $\bar X$.

To motivate the new fixed point problem,
consider the system of SPDEs given by \eqref{eq:renorm_bar_a}+\eqref{e:final_hU}.
Note that if we start with an arbitrary initial condition at some earlier time $t_0 < 0$, then the solution
at time $0$ is necessarily of the form
\begin{equ}
\bar X_0 = \big(K^\eps * (\bar U\xi)\big)(0,\act)
+ K*\big( K^\eps *(\bar U\xi)\partial  K^\eps *(\bar U\xi)\big)(0,\act)
+ \tilde X_0\;,
\end{equ}
where $\tilde X_0 \in \CC^{-\kappa}$ (even in $\CC^{\frac12-}$, but we don't need to use this).
Looking now at positive times, it is straightforward to see that $X = \tilde X + Z_1+Z_2$ where $Z_1 = K^\eps * (\bar U\xi)$, $Z_2 = K*(Z_1\partial Z_1)$,
and $\bar U$ solves~\eqref{eq:h_and_U_def} for positive times and is equal to its previous value for negative times,
and $\tilde X$ solves the fixed point problem
\begin{equ}[e:FPclassical]
\tilde X = \CP \tilde X_0
+ G * \one_+ (\tilde X^3 + \tilde X\p \tilde X + \mathring C_\mfz \tilde X + \mathring C_\mfh h - Z_1\partial Z_1)\;.
\end{equ}

We now lift this to a fixed point problem in a suitable space of modelled distributions. 
For this, assume that we are given $\tilde X_0 \in \CC^{-\kappa}$, $h_0 \in \CC^{\frac12-3\kappa}$, and 
a modelled distribution $\tilde \CU \in \cD^{\f52+2\kappa}$
defined on $[-T,0]\times \T^3$ for some $T\in(0,1)$
and taking values in the corresponding sector of our regularity structure.
One should think of $\tilde\CU$ as the `previous' $\bar U$ restricted to a sufficiently short interval $[-T,0]$.

As earlier, we decompose $\CU = \CP \bar U_0 + \hat \CU$.
We therefore extend
$\tilde \CU$ to positive times by the Taylor lift of the harmonic extension of its
reconstruction $\bar U_0$ at time $0$. Note that $\bar U_0$ is simply given by the component 
of $\tilde \CU$ multiplying $\bone$ since the sector corresponding to $\CU$ is 
of function type. We also have $\bar U_0 \in \CC^{\f32-3\kappa}$ which implies that,
as a modelled distribution on $[-T,\infty)\times \T^3$, $\tilde \CU \in \bar \cD^{\f52+2\kappa,\f32-3\kappa}_0$.

Let us write now
\begin{equ}
\CZ_1 = \CG_\mfm(\CU\bXi)\quad \text{and} \quad
\CZ_2 = \CG_\mfz (\CZ_1 \partial \CZ_1)\;,
\end{equ}
where $\CU = \tilde\CU + \hat\CU$.
Here $\hat\CU$ and $\CH$ solve~\eqref{e:fix-pt-HU},
and $\CY \eqdef \CZ_1+\CZ_2+\tilde \CY$ where $\tilde \CY$ solves
\begin{equ}[eq:tilde_X_FPP]
\tilde \CY = \CP \tilde X_0 + \CG_\mfz \bone_+
(\CY^3 + \CY\partial\CY + \mathring C_\mfz \CY + \mathring C_\mfh \CH - \CZ_1\partial \CZ_1)\;.
\end{equ}
We claim that we can solve for
\begin{equ}
(\tilde \CY,\hat\CU,\CH) \in \cD^{\frac32+3\kappa,-\kappa}_0 \times \hat\cD^{\frac52+2\kappa, \frac32-3\kappa}_0\times \cD^{1+5\kappa,\frac12-3\kappa}_0\;.
\end{equ} 
Indeed, by Lemma~\ref{lem:multiply-hatD} and Theorem~\ref{thm:integration}, $\CU\bXi \in \bar \cD^{\kappa,-1-4\kappa}_{-\frac52-\kappa}$ and $\CZ_1 \in \bar\cD^{2-\kappa,1-5\kappa}_{-\frac12-2\kappa}$.
Hence
\begin{equ}
\CZ_1\partial \CZ_1 \in \bar\cD^{\frac12 - 3\kappa, -\frac12 -7\kappa}_{-2-4\kappa}
\Rightarrow
\CZ_2 \in \cD^{\frac52 - 4\kappa, \frac32 - 8\kappa}_{-5\kappa}\;.
\end{equ}
Furthermore $\tilde X_0 \in \CC^{-\kappa}$
and the `right-hand side' is
\begin{equ}
\CY^3 + \CY\partial\CY + \mathring C_\mfz \CY + \mathring C_\mfh \CH - \CZ_1\partial \CZ_1
\in \bar\cD^{\kappa, -\frac32-3\kappa}_{-\frac32-6\kappa}
\subset \cD^{\kappa,-\frac32-6\kappa}_{-\frac32-6\kappa}\;. 
\end{equ}
Note how $-\CZ_1\partial \CZ_1$ cancels the worst term in $\CY\partial\CY$, so that the worst terms are now $\CZ_1^3 \in \bar\cD^{1-5\kappa,-9\kappa}_{-\frac32-6\kappa}$
and
$\CZ_1 \partial \tilde \CY\in \bar\cD^{\kappa, -\frac32-3\kappa}_{-\frac32-2\kappa}$.
Since $\cD^{\kappa,-\frac32-6\kappa}_{-\frac32-6\kappa}$ is stable under multiplication by $\bone_+$,
the fixed point problem~\eqref{eq:tilde_X_FPP} is well-posed.

Furthermore, under the same pseudo-metric on models as in Proposition~\ref{prop:improve_local},
the reconstruction $(\bar X,\tilde X,Z_1,Z_2,\bar U,\bar h)$ of the solution $(\CY,\tilde \CY,
\CZ_1,\CZ_2,\CU,\CH)$
is a locally uniformly continuous function into $\CC([\lambda\tau,\tau],\state\times\CC^{-\kappa}\times \state\times\CC^{-5\kappa}\times
\tilde{\mfG}^{0,\rho})$ (for any fixed $\lambda>0$)
of the tuple
\begin{equ}
(Z,\tilde\CU,\tilde X_0, h_0) \in \mathscr{M}_1\times \cD^{\frac52+2\kappa}\times \CC^{-\kappa}\times \CC^{\frac12-3\kappa}
\end{equ}
where $\tau$ is locally uniform in the tuple.
(Note how we do \textit{not} require the input distributions $\bar\omega$ any more because the initial condition is in $\CC^{-\kappa}$ instead of $\CC^{\eta}$ for $\eta<-\frac12$.)


For $\eps\in[0,1]$, the construction of the maximal solutions (modelled distributions) $(\CY,\CU,\CH)$
is then similar to~\cite[Def.~7.20]{CCHS2d}.
Namely, we begin by solving for $(\CY,\CU,\CH)$ using the original equation~\eqref{e:fix-pt-multi-Xbar}+\eqref{eq:CY_bar_equation}+\eqref{e:fix-pt-HU}
on an interval $[0,\sigma_1]\eqdef [0,2\tau_1]$.
The underlying model is $Z^{0,\eps}_{\BPHZ}\in\mathscr{M}_1$ as in Lemma~\ref{lem:conv_of_models2},
$\CW$ is taken as in~\eqref{eq:CW_def},
and $\bar\omega^0_0$, $\bar\omega^{\eps,0}_\ell$ for $\ell\in\{1,2,3\}$, and $\bar\omega^\eps_4$ are defined through~\eqref{e:bar_omega0123} and Lemma~\ref{lem:omegas_converge}. 

We then take $T=\tau_1$ with time centred around $\sigma_1$ and $\CU \restr_{[\sigma_1-T,\sigma_1]}$
playing the role of $\tilde \CU$
above,
and
\begin{equ}
\tilde X_0 \eqdef \CR(\CY-\CZ_1-\CZ_2)(\sigma_1)\quad
\text{and} \quad
h_0\eqdef (\CR \CH)(\sigma_1)\;.
\end{equ}
Solving now~\eqref{eq:tilde_X_FPP}+\eqref{e:fix-pt-HU} with this data (and the same model $Z^{0,\eps}_{\BPHZ}$),
we extend the solution to the interval $[0,\sigma_2]\eqdef [0,2\tau_1+2\tau_2]$.

Then we set $T=\tau_2$, centre time around $\sigma_2$, and solve again~\eqref{eq:tilde_X_FPP}+\eqref{e:fix-pt-HU}
with $\tilde \CU = \CU\restr_{[\sigma_2-T,\sigma_2]}$,
$\tilde X_0 = (\CR\tilde \CY)(\sigma_2)$ and $h_0=(\CR \CH)(\sigma_2)$.
We thus extend the solution $(\CY,\CU,\CH)$ to the interval $[0,\tau^*)$ where $\tau^\star = \sum_{k=1}^\infty\sigma_k$.

The proof of the proposition is then similar to that of~\cite[Prop.~7.23]{CCHS2d}.
In particular, similar to~\cite[Remarks~7.21,~7.22]{CCHS2d},
the reconstruction $(\bar X, \bar U,\bar h)$ solves the renormalised PDEs~\eqref{eq:renorm_bar_a}+\eqref{e:final_hU}
on the interval $[0,\sigma_n]$ provided that $\eps^2 < \tau_i$ for all $i=1,\ldots, n$.
Moreover, by the convergence of models from Lemma~\ref{lem:conv_of_models2} and Corollary~\ref{cor:SHE_conv_state},
for every $n\geq 1$, there exists $\eps_n$ such that we can choose the same $\tau_1,\ldots,\tau_n$ for all $\eps\in[0,\eps_n]$.

\label{page:reconstructions_match}The final important remark is that for $\eps=0$, the reconstructions of the above maximal solutions on each interval $(\sigma_k,\sigma_{k+1}]$
coincide with the reconstructions of the solutions to the original
fixed point problem~\eqref{e:fix-pt-multi-Xbar} with $\CW=0$
and $\bar\omega_\ell$ defined again through~\eqref{e:bar_omega0123} and Lemma~\ref{lem:omegas_converge}
\dash note that, since we take $\delta\downarrow0$ first and then $\eps\downarrow0$,
the initial condition at each time step $\sigma_k$ is independent of $\xi \restr_{[\sigma_k,\infty)}$ 
(modulo centring time around $\sigma_k$),
therefore the estimates and convergence results from Lemma~\ref{lem:omegas_converge}
still hold since $\sigma_k$ can be taken as a stopping time.
In particular, the maximal solution for $\eps=0$ is in $(\state\times\tilde{\mfG}^{0,\rho})^\sol$,
which, combined with Proposition~\ref{prop:SPDEs_conv_zero_state}, implies that $(\bar X,\bar U,\bar h)$ converges as $\eps\downarrow0$ in $(\state\times\tilde{\mfG}^{0,\rho})^\sol$ (in probability).
\end{proof}

\subsection{Gauge covariance}

In this subsection we prove Theorem~\ref{thm:gauge_covar}.
We keep using the notation of Section~\ref{subsec:renorm_sol}.
Following Proposition~\ref{prop:SPDEs_conv_zero_state}, we will use the main result of Appendix~\ref{app:injectivity} to show that $C^\eps_{j,\YM}-C^\eps_{j,\Gauge}$  and $C^{0,\eps}_{j,\Gauge}$ both converge to finite limits as $\eps\downarrow0$.
This in turn will allow us to conclude the proof of Theorem~\ref{thm:gauge_covar}.

%
\begin{proposition}\label{prop:consts-soft} 
Suppose $\moll$ is non-anticipative.
With notation as in Lemma~\ref{lem:consts-soft1_bounded}, 
there exists $(D_{1},D_{2}) \in L_{G}(\mfg,\mfg)^2$ such that
\begin{equ}
\lim_{\eps \downarrow 0}\;
(D_{1}^{\eps},D_{2}^{\eps}) = (D_{1},D_{2})\;.
\end{equ}
Moreover, $(D_{1},D_{2})$ does not depend on the choice of mollifier. 
\end{proposition}
\begin{proof}
We start by proving the first statement regarding $\eps \downarrow 0$ convergence. 
Recalling that $
\sup_{\eps \in (0,1]}\big(
|D_{1}^{\eps}| + |D_{2}^{\eps}| \big)< \infty$
by Lemma~\ref{lem:consts-soft1_bounded}, we argue by contradiction and suppose that  there exist sequences $\eps_{n}, \tilde{\eps}_{n} \downarrow 0$ along which we have  {\it distinct}  limits 
\begin{equ}\label{eq:subsequential}
(D_{1}^{\eps_n}, D_{2}^{\eps_n}) \rightarrow (D_{1},D_{2}) 
\quad
\textnormal{and} 
\quad
(D_{1}^{\tilde{\eps}_n}, D_{2}^{\tilde{\eps}_n}) \rightarrow (\tilde{D}_{1},\tilde{D}_{2})\;,
\end{equ} 
as $n \rightarrow \infty$. We will show that this gives rise to a contradiction with the results of Appendix~\ref{app:injectivity}.
We first note that we can find $\sig$ such that, setting 
$D = \sig^{2} D_{1} + \sig^{4} D_{2}$ and $\tilde{D} =  \sig^{2}\tilde{D}_{1} + \sig^{4} \tilde{D}_{2}$, we have $D \not = \tilde{D}$.
%
We fix $g_{0} \in \mfG^{0,\rho}$ such that $D^{\oplus 3} \mrd g_{0} g_{0}^{-1} \not = \tilde{D}^{\oplus 3} \mrd g_{0} g_{0}^{-1}$ and some arbitrary $Y_{0} \in \state$. 

Now for $\eps > 0$ let $(X^{\eps},g^{\eps})$ be the maximal solution to~\eqref{e:SPDE-for-X} (with $\sig^\eps=\sig$ and $\mathring{C}^{\eps}_{\A} = \mathring{C}_{\A} = 0$) and \eqref{eq:SPDE_for_g_wrt_A}, started from the  initial data $(X_0,g_0)$ where $Y_0 = g_0 \act X_0$.
Observe that, by Lemma~\ref{lem:flow_for_rough_g},
$(X^{\eps},g^{\eps})$ converges in probability in $(\state\times \mfG^{0,\rho})^\sol$ as $\eps\downarrow0$ to a limit we denote $(X,g)$.
On the other hand, $Y^{\eps}\eqdef g^{\eps} \act X^{\eps}$ solves
\[
\partial_t Y^{\eps}
= \Delta Y^{\eps}
+ Y^{\eps}\partial Y^{\eps}
+ (Y^{\eps})^3
+ C_{\mfz} Y^{\eps}
+ C_{\mfh}^\eps \mrd g g^{-1}
+ \sig^\eps  (g \xi^\eps)\;, \quad \tilde Y^{\eps}(0)=Y_0\;,
\]
where in the definitions of $C_\mfz\eqdef C_{\mfz}^{\eps}$ and  $C_{\mfh}^\eps$ one takes  $\mathring{C}^{\eps}_{\A} = \mathring{C}_{\A} = 0$ and $\mathring{C}^{\eps}_{\h} = \sig^{2} D^{\eps}_{1} + \sig^{4} D^{\eps}_{2}$.
Then, by \eqref{eq:subsequential}, we have $\mathring{C}^{\eps_{n}}_{\h} \rightarrow D$ and $\mathring{C}^{\tilde{\eps}_n}_{\h} \rightarrow \tilde{D}$.

For the rest of the argument we recall the space of local solutions $(\state \times \tilde{\mfG}^{0,\rho})^{\lsol}$ and local solution map $\mathcal{A}_{\sig}^{\BPHZ, \scriptscriptstyle{\lsol}}$ introduced in Remark~\ref{rem:injectivity_local}.

Then there exists some strictly positive $\mathbf{F}$-stopping time $\tau$
such that, in probability, we have
\begin{equs}
\lim_{n \rightarrow \infty}
(\tau, Y^{\eps_n}, U^{\eps_n}, h^{\eps_n})  &\rightarrow \mathcal{A}_{\sig}^{\BPHZ, \scriptscriptstyle{\lsol}}[0, D,X_0, g_0]\\
\text{and}
\quad  
\lim_{n \rightarrow \infty}
(\tau, Y^{\tilde\eps_n}, U^{\tilde\eps_n}, h^{\tilde\eps_n})  & \rightarrow \mathcal{A}_{\sig}^{\BPHZ, \scriptscriptstyle{\lsol}}[0, \tilde{D},X_0, g_0]\;.
 \end{equs}
Here $U,h$ (resp.\ $U^{\eps_n}, h^{\eps_n}$ \slash $U^{\tilde{\eps}_n}, h^{\tilde{\eps}_n}$) are obtained from $g$ (resp.\ $g^{\eps_n}$ \slash $g^{\tilde{\eps}_n}$) via \eqref{eq:h_and_U_def}. 

On the other hand, since $(X^{\eps},g^{\eps})$ does not depend on $\mathring C^\eps_\h$ and by the continuity of the group action given in Proposition~\ref{prop:group_action}, it follows that 
\begin{equ}\label{eq:same_limit}
\lim_{n \rightarrow \infty}
(\tau, Y^{\eps_n}, U^{\eps_n}, h^{\eps_n}) =  \lim_{n \rightarrow \infty}
(\tau, Y^{\tilde{\eps}_n}, U^{\tilde{\eps}_n}, h^{\tilde{\eps}_n})
= 
(\tau, X,U,h)\;,
\end{equ}
in probability, as random elements of $(\state \times \tilde{\mfG}^{0,\rho})^{\lsol}$ \dash here $(\tau, X,U,h)$ is defined by stopping the maximal solution $(X,U,h) \in (\state \times \tilde{\mfG}^{0,\rho})^{\sol}$ at $\tau$. 

By \eqref{eq:observable_injectivity_local} and cancelling stochastic integrals and terms $\big( 0^{\oplus 3} \oplus (- m^2) \big)X$ on either side, we have the almost sure equality 
\[
D^{\oplus 3} h 1_{(0,\tau)}= \tilde{D}^{\oplus 3} h 1_{(0,\tau)}\;.
\]
However, this is a contradiction since $\lim_{t \downarrow 0} h(t) = \mrd g_{0} g_{0}^{-1}$ \dash this finishes the proof of the first statement of the proposition. 

The statement about independence with respect to our choice of mollifier $\chi$ can be proven with the same argument \dash the key point above was that the limit $(\tau, X,U,h)$ in \eqref{eq:same_limit} is independent of which $\eps \downarrow 0$ subsequence we choose, thanks to BPHZ convergence.\footnote{This is clearly true for the maximal solution $(X,U,h)$. The stopping time $\tau$ can also be chosen small enough in a way that only depends on the limiting BPHZ model and initial data $(X_{0},g_{0})$ and an a-priori bound on the constants $\mathring{C}^{\eps}_{\A}$, $\mathring{C}^{\eps}_{\h}$}
However, the BPHZ model that defines $(\tau,X,U,h)$ is also insensitive to the choice of mollifier one uses when taking $\eps \downarrow 0$. 
Therefore the argument above to prove uniqueness of the limiting constants obtained with two different $\eps \downarrow 0$ subsequences and a single mollifier can be repeated when taking  $\eps \downarrow 0$ with two different mollifiers.
\end{proof}

As a corollary of Proposition~\ref{prop:consts-soft} and the proof of Proposition~\ref{prop:SPDEs_conv_zero_state}, we obtain convergence of $(Y,U,h)$ as \textit{maximal} solutions.

\begin{corollary}\label{cor:Y_system_conv}
In the setting of Proposition~\ref{prop:SPDEs_conv_zero_space-time},
$(Y,U,h)$ converges in probability in $(\state\times\tilde{\mfG}^{0,\rho})^\sol$ as $\eps\downarrow0$.
\end{corollary}

\begin{proof}
Consider $X,g$ as in the proof of Proposition~\ref{prop:SPDEs_conv_zero_state}.
Since $c^\eps$ in~\eqref{eq:c^eps_def} converges as $\eps\downarrow0$ to a finite limit due to Proposition~\ref{prop:consts-soft}, $(X,g)$ converges in probability in $(\state\times\mfG^{0,\rho})^\sol$ as $\eps\downarrow0$
due to Lemma~\ref{lem:flow_for_rough_g}.
Since $Y=g\act X$ and $(U,h)$ are defined from $g$ by~\eqref{eq:h_and_U_def} for all positive times (until the blow-up of $(X,g)$),
the conclusion follows from the local uniform continuity of the group action (Proposition~\ref{prop:group_action}).
\end{proof} 

The next proposition concerns the $\eps \downarrow 0$ limit of the renormalisation constants
$C_{\Gauge}^{0,\eps}$  which appear in the renormalised equations for $\bar X$.

Assume that the mollifier  $\moll$ is non-anticipative.
For any fixed   
 $\eps > 0$, after taking $\delta \downarrow 0$,
since $\bar{g}$ is adapted 
one expects that 
the law of $\bar{X}$ should remain the same
if one replaces 
$\bar{g} \xi$ by $\xi$,
in which case 
  the equation for $\bar{X}$
would simply become    
the equation for $X$.
Therefore
one expects that 
$C_{\Gauge}^{0,\eps}$ should vanish when  the mollifier  $\moll$ is non-anticipative.
However the following result is sufficient for our needs.

\begin{proposition}\label{prop:consts-soft2}
Suppose $\moll$ is non-anticipative. 
Then, for $j\in\{1,2\}$, there exists $C_{j,\Gauge}^{0,0}\in L_{G}(\mfg,\mfg)$ such that 
\begin{equ}
\lim_{\eps \downarrow 0}\;
(C_{1,\Gauge}^{0,\eps},C_{2,\Gauge}^{0,\eps})
=
(C_{1,\Gauge}^{0,0},C_{2,\Gauge}^{0,0})\;.
\end{equ}
Moreover, $(C_{1,\Gauge}^{0,0},C_{2,\Gauge}^{0,0})$ is independent of our choice of non-anticipative $\moll$.
\end{proposition}

\begin{proof} 
We first prove the limit as $\eps \downarrow 0$ for a fixed choice of mollifier. 
Recall that by  Lemma~\ref{lem:consts-soft2_bounded} we have $\sup_{\e\in(0,1]} (|C_{1,\Gauge}^{0,\eps}|+|C_{2,\Gauge}^{0,\eps}|) < \infty$.
To prove convergence we again argue by contradiction and suppose that we have sequences $\eps_n$ and $\tilde\eps_n$ going to $0$ along which we have the distinct limits
\begin{equs}[eq:subsequential2]
{}(C_{1,\Gauge}^{0,\eps_n},C_{2,\Gauge}^{0,\eps_n})
& \rightarrow 
(C_{1,\Gauge}^{0,0},C_{2,\Gauge}^{0,0})\\
\text{and}
\quad
(C_{1,\Gauge}^{0,\tilde\eps_n},C_{2,\Gauge}^{0,\tilde\eps_n})
& \rightarrow 
(\tilde{C}_{1,\Gauge}^{0,0},\tilde{C}_{2,\Gauge}^{0,0})\;.
\end{equs}
We fix $\sig \in \R$ such that, if we set $D = \sig^{2} C_{1,\Gauge}^{0,0} + \sig^{4} C_{2,\Gauge}^{0,0}$ and  $\tilde{D} = \sig^{2} \tilde{C}_{1,\Gauge}^{0,0} + \sig^{4} \tilde{C}_{2,\Gauge}^{0,0}$, one has $D \not = \tilde{D}$.  
Fix $\bar{g}_{0} \in \mfG^{0,\rho}$ such that \begin{equ}\label{eq:diff_init_conds}
D^{\oplus 3} \mrd \bar{g}_{0} \bar{g}_{0}^{-1} \not = \tilde{D}^{\oplus 3} \mrd \bar{g}_{0} \bar{g}_{0}^{-1}
\end{equ}
along with some arbitrary $\bar{X}_{0} \in \state$.

Again, for $\eps > 0$, and this time making $\eps$-dependence explicit,
let $\bar{X}^\eps,\bar{g}^\eps$ be the maximal solution to \eqref{eq:renorm_bar_a} with $\sig^{\eps} = \sig$, $\mathring{C}^{\eps}_{\h} = -\sig^{2} C_{1,\Gauge}^{0,\eps} - \sig^{4} C_{2,\Gauge}^{0,\eps}$,
and $\mathring{C}^{\eps}_{\A} = 0$ started from initial data $(\bar{X}_{0},\bar{g}_0)$.
By Proposition~\ref{prop:bar_X_max_sols} we have the convergence in probability
\begin{equs}
 (\bar{X}^{\eps_n},\bar{U}^{\eps_n},\bar{h}^{\eps_n}) &\rightarrow \mathcal{A}_{\sig}^{\BPHZ}[0, -D,X_0, g_0]\\
\text{and}
\quad  
(\bar{X}^{\tilde\eps_n},\bar{U}^{\tilde\eps_n},\bar{h}^{\tilde\eps_n}) & \rightarrow \mathcal{A}_{\sig}^{\BPHZ}[0, -\tilde{D},X_0, g_0]\;.
 \end{equs}
Here $\bar{U}^{\eps}, \bar{h}^{\eps}$ are obtained from $\bar g^{\eps}$ as in \eqref{eq:h_and_U_def} and $\mathcal{A}_{\sig}^{\BPHZ}$ is the BPHZ solution map given in \eqref{equ:constants_to_solution_map}.

Now, let $(X^{\eps},g^{\eps})$ be the maximal solution to \eqref{e:SPDE-for-X_with_g2}--\eqref{eq:SPDE_for_g_coupled2}  with $\mathring{C}^{\eps}_{\A} = 0$ and $c^{\eps} = 0$ started from initial data $(\bar{X}_{0},\bar{g}_0)$. 
Due to our choice of renormalisation constants, $(X^{\eps},g^{\eps})$ is equal  to $(\bar{X}^{\eps},\bar{g}^{\eps})$ in law. 
Moreover, by Lemma~\ref{lem:flow_for_rough_g2}, $(X^{\eps},g^{\eps})$ converges in probability as $\eps \downarrow 0$ to a limit $(X,g)$. 

It follows that $\mathcal{A}_{\sig}^{\BPHZ}[0, -D,X_0, g_0]\eqlaw \mathcal{A}_{\sig}^{\BPHZ}[0, -\tilde{D},X_0, g_0]$ but this combined with \eqref{eq:diff_init_conds} gives a contradiction with Lemma~\ref{lemma:injectivity_in_law} since 
\begin{equs}
\lim_{n \rightarrow \infty} \E \big[ O_{n} (\mathcal{A}_{\sig}^{\BPHZ}[0, -D,X_0, g_0]) \big] &= -D^{\oplus 3} \mrd g_{0} g_{0}^{-1}\\
\text{and} \quad \lim_{n \rightarrow \infty} \E \big[ O_{n} ( \mathcal{A}_{\sig}^{\BPHZ}[0, -\tilde{D},X_0, g_0])  \big] 
&= -\tilde{D}^{\oplus 3} \mrd g_{0} g_{0}^{-1}\;.
 \end{equs}
The above argument finishes the proof of the first statement. 
The second statement, regarding independence of the limit with respect to our choice of non-anticipative mollifier $\chi$, can be also be proven using the above argument
\dash the relation between the first and second statement of this proposition mirrors what is explained in the final paragraph
of the proof of Proposition~\ref{prop:consts-soft2} \dash again the limit $(X,g)$ of $(X^{\eps},g^{\eps})$ above does not depend on the choice of mollifier. 
Thus we can repeat the argument above for two sequences obtained by taking $\eps \downarrow 0$ for two different non-anticipative mollifiers. 
\end{proof}

\begin{proof}[of Theorem~\ref{thm:gauge_covar}]\label{proof:gauge_covar}
Part~\ref{pt:delta_to_zero} was already proven in Proposition~\ref{prop:SPDEs_conv_zero_space-time}.

We now prove part~\ref{pt:gauge_covar}. We will apply Proposition~\ref{prop:SPDEs_conv_zero_state} (together with Proposition~\ref{prop:bar_X_max_sols}
and Corollary~\ref{cor:Y_system_conv}) with $\sig_\eps=\sig=1$.
Note that Proposition~\ref{prop:SPDEs_conv_zero_state} holds for any choice of converging sequence $(\mathring{C}_{\A}^\eps, \mathring{C}_{\h}^\eps)$.
We now find the claimed $\eps$-independent $\check{C}$, so that for any $\mathring{C}_{\A}$ as in Theorem~\ref{thm:gauge_covar},
we can choose suitable $(\mathring{C}_{\A}^\eps, \mathring{C}_{\h}^\eps)$ for the systems in Proposition~\ref{prop:SPDEs_conv_zero_state}
so that 
the two systems~\eqref{eq:SPDE_for_B}+\eqref{eq:SPDE_for_bar_A} being compared in Theorem~\ref{thm:gauge_covar}
and the two systems~\eqref{eq:renorm_b}+\eqref{eq:renorm_bar_a}
being compared in Proposition~\ref{prop:SPDEs_conv_zero_state} match, at least up to order $o(1)$.
To match~\eqref{eq:SPDE_for_B} and~\eqref{eq:renorm_b} we require
\begin{equ}[e:match-consts-goal_B]
C^\eps_{\YM} + \mathring C_\A = C^\eps_{\YM} + \mathring C^\eps_\A\;,
\qquad
C^\eps_{\YM} + \mathring{C}_{\A} =\mathring{C}_{\h}^\eps+ C^\eps_{\Gauge} \;,
\end{equ}
while to match~\eqref{eq:SPDE_for_bar_A} and~\eqref{eq:renorm_bar_a} to order $o(1)$ we require
\begin{equ}[e:match-consts-goal_A]
C^\eps_{\YM} + \mathring C_\A = C^\eps_{\YM} + \mathring C^\eps_\A\;,
\qquad
\mathring{C}_{\A} - \check{C} = \mathring{C}_{\h}^\eps + C^{0,\eps}_{\Gauge} + o(1)\;.
\end{equ}
To satisfy the first identities in~\eqref{e:match-consts-goal_B}--\eqref{e:match-consts-goal_A}, we take $\mathring{C}_{\A}^\eps = \mathring{C}_{\A}$.

For the second identities,
by Proposition~\ref{prop:consts-soft}, one has convergence of 
$C^{\eps}_{\YM}  - C^{\eps}_{\Gauge}$.
We therefore set
\begin{equ}
\mathring{C}_{\h}^\eps \eqdef 
 C^{\eps}_{\YM}  - C^{\eps}_{\Gauge} + \mathring{C}_{\A} \;,
\end{equ}
so that the second identity in~\eqref{e:match-consts-goal_B} is satisfied and $\lim_{\eps\downarrow0}\mathring{C}_{\h}^\eps $ exists.
Furthermore, by Proposition~\ref{prop:consts-soft2} and the non-anticipative assumption 
on $\moll$, one has convergence of $C^{0,\eps}_{\Gauge}$.
We therefore set
\begin{equ}[e:def-C-bar]
\check{C}  \eqdef  \lim_{\eps\downarrow 0}\;
 (C^{\eps}_{\Gauge} -C^{0,\eps}_{\Gauge}  - C^{\eps}_{\YM} )\;,
\end{equ}
so that the second identity in~\eqref{e:match-consts-goal_A} is satisfied.
Since the difference between the constants in the two systems~\eqref{eq:SPDE_for_bar_A} and~\eqref{eq:renorm_bar_a} vanishes,
stability of the fixed point problem~\eqref{e:fix-pt-multi-Xbar} with respect to $\mathring C_\mfh$
 implies that the two systems converge in probability to the same limit.

The claimed convergence in probability in part~\ref{pt:gauge_covar} therefore follows from Proposition~\ref{prop:SPDEs_conv_zero_state} with the above choice of $(\mathring{C}_{\A}^\eps, \mathring{C}_{\h}^\eps)$
together with Proposition~\ref{prop:bar_X_max_sols}
and Corollary~\ref{cor:Y_system_conv}.
Indeed, Proposition~\ref{prop:bar_X_max_sols}
and Corollary~\ref{cor:Y_system_conv} imply that the two systems converge separately as maximal solutions in probability in $(\state\times\tilde{\mfG}^{0,\rho})^\sol$ as $\eps\downarrow 0$,
while Proposition~\ref{prop:SPDEs_conv_zero_state} implies the two limits are equal on an interval $[0,\tau]$ where $\tau>0$ is an $\mbF$-stopping time
bounded stochastically from below by a distribution depending only on the size of the initial condition.
Since both limits are (time-homogenous) $(\state\times\tilde{\mfG}^{0,\rho})$-valued strong Markov process with respect to $\mbF$, it follows that they are equal almost surely as $(\state\times\tilde{\mfG}^{0,\rho})^\sol$-valued random variables.

Furthermore, $\check C$ is unique
by the injectivity in law given by  Lemma~\ref{lemma:injectivity_in_law} \dash any alternative choice of $\check{C}$ would mean that each of the each of the two systems again converge to solutions of the form
$\mathcal{A}_{1}^{\BPHZ}[\mathring{C}_{\A},\mathring{C}_{\h},x,g_0]$, but now with distinct choices of $(\mathring{C}_{\A},\mathring{C}_{\h})$ for each of the systems.
This completes the proof of part~\ref{pt:gauge_covar}.

For part~\ref{pt:indep_moll} of the theorem, observe that the right hand side of \eqref{e:def-C-bar} is independent of  $\mathring{C}_{\A}$ thanks to Remark~\ref{rem:bare-m-renorm_2} and independent of our specific choice of non-anticipative $\chi$ thanks to the second statements in Propositions~\ref{prop:consts-soft}+\ref{prop:consts-soft2}. 

The final part~\ref{pt:final_gauge_covar} of the theorem is mostly a reformulation of the other parts of the theorem. 
Note that  $(g \act X, U, h)$ is almost surely equal to the maximal solution $(Y,U,h) = (B,\Psi,U,h)$ (as given in part~\ref{pt:gauge_covar} of the theorem) started from the initial data $(g_0 \act x, \brho(g_0), \mrd g_0 g_0^{-1})$. 
On the other hand, if $(\bar{X},\bar{U},\bar{h})$ (again, the limiting solution given in part~\ref{pt:gauge_covar}) is started from initial data  $(g_0 \act x, \brho(g_0), \mrd g_0 g_0^{-1})$  then we have $(Y,U,h) = (\bar{X},\bar{U},\bar{h})$ almost surely, in particular $(Y,U,h) \eqlaw (\bar{X},\bar{U},\bar{h})$. 

Finally, we have that $(\bar{X},\bar{U},\bar{h})   \eqlaw (\tilde{X}, \tilde{U}, \tilde{h})$ for all choices of  $(x,g_0) \in \state \times \mfG^{0,\rho}$ if and only if $\mathring{C}_{\A} = \check{C}$. 
To see that  $\mathring{C}_{\A} = \check{C}$ is a sufficient condition for this equality in law, note that, for $\eps > 0$ and non-anticipative mollifiers,
we have equality in law for the solutions of \eqref{eq:SPDE_for_bar_A} and those of \eqref{e:SPDE-for-X_with_g2}+\eqref{eq:SPDE_for_g_coupled2} (with $c^{\eps} = 0$). 

To see that  $\mathring{C}_{\A} = \check{C}$ is a necessary condition
for $(\bar{X},\bar{U},\bar{h})   \eqlaw (\tilde{X}, \tilde{U}, \tilde{h})$, note that, by the injectivity in law given by Lemma~\ref{lemma:injectivity_in_law}, there exists a choice of $(g_0 \act x, g_0) \in \state \times \mfG^{0,\rho}$ for which limits of the law of $(\bar{X},\bar{U},\bar{h})$  started from  $(g_0 \act x, g_0)$ has a different law then that of the limit of the maximal solutions to a modification of \eqref{eq:SPDE_for_bar_A} where one again starts from $(g_0 \act x, g_0)$ but drops the term $(\mathring{C}_{\A} - \check{C})(\partial_i \bar{g} )\bar{g}^{-1}$. 
However, this latter law is equal to the law of $(\tilde{X}, \tilde{U}, \tilde{h})$ (again, since the dynamics only differ by a rotation of the noise by an adapted gauge transformation). 
This shows that if $\mathring{C}_{\A} \not = \check{C}$ then there exists initial data for which $(\bar{X},\bar{U},\bar{h})   \not \eqlaw (\tilde{X}, \tilde{U}, \tilde{h})$.
\end{proof}

\section{Markov process on gauge orbits}
\label{sec:Markov}

We finally turn to the existence and uniqueness of a Markov process on the orbit space $\mfO$ associated to the SPDE~\eqref{eq:SPDE_for_A}, which is the second part of Theorem~\ref{theo:meta}.

We first adapt to our setting the notion of a generative probability measure
from~\cite{CCHS2d} in order to specify the way in which a Markov process is canonical.
Our definition differs slightly from that of~\cite{CCHS2d} due to the different structure of the quotient space.

\begin{definition}\label{def:solves_SYMH}
Unless otherwise stated, a \emph{white noise} $\xi=(\xi_1,\xi_2, \xi_3,\zeta)$ is understood to be an $E=\mfg^3\oplus\higgsvec$-valued white noise on $\R\times\T^3$
defined on a probability space $(\Omega^{\noise},\mcF, \P)$.
A filtration $\mbF=(\mcF_t)_{t \geq 0}$ is said to be \emph{admissible} for $\xi$ if $\xi$ is adapted to $\mbF$ and $\xi\vert_{[t,\infty)}$ is independent of $\mcF_t$ for all $t \geq 0$.
By the \emph{SYMH solution map} (driven by $\xi$) we mean the random map
\begin{equ}
\Lambda \colon \{(s,t)\,:\,0\leq s\leq t<\infty\}\times \hat\state \to L^0(\Omega^\noise;\hat \state)
\end{equ}
for which $t\mapsto \Lambda_{s,t}(x)$ is the $\eps\downarrow0$ limit in probability in $\state^\sol$ of the solutions to~\eqref{eq:SPDE_for_A} with initial condition $x$ at initial time $s$, and
with a mollifier $\moll$ and operators $(C^{\eps}_{\A},C^{\eps}_{\Phi})$ defined by~\eqref{e:SYM_constants} with  $\mathring C_\A=\check C$ and $\sig^\eps=1$, where $\check C$ is the unique $\moll$-independent operator in
Theorem~\ref{thm:gauge_covar}\ref{pt:indep_moll}.
\end{definition}

We will frequently say that $X$ solves SYMH driven by $\xi$ with initial condition $x\in\hat\state$ and time $s\geq 0$ to mean that $X(t)=\Lambda_{s,t}(x)$.
Remark that $\Lambda_{s,s}(x)=x$ by construction.
We will drop the reference to $s$ whenever $s=0$.

For fixed $s,t,x$, note that $\Lambda_{s,t}(x)\colon \Omega^\noise \to \hat\state$ is a random variable and we make no claims about the $\P$-a.s. continuity in $x$
(the difficulty here being the extra stochastic objects in Lemma \ref{lem:PX0Psi-prob} that need to be constructed from the initial data).
Nonetheless, $t\mapsto \Lambda_{s,t}(x)$ is $\P$-a.s. continuous since it is $\P$-a.s. in $\state^\sol$.
Moreover, $\state\ni x\mapsto \Lambda_{s,s+\cdot}(x) \in L^0(\Omega^\noise;\hat\state)$ is continuous in probability for any $s\geq 0$.

Furthermore, for an admissible filtration $\mbF$ and an $\mbF$-stopping time $\tau$, if $X_0$ is an $\mbF_\tau$-measurable random variable, then $t\mapsto \Lambda_{\tau,\tau+t}(X_0)$ is well-defined as a $\state^\sol$-valued random variable and is adapted to the filtration $\{\mcF_{\tau+t}\}_{t\geq 0}$.
Indeed, this follows from independence of $\tilde{\Psi}_{\eps} $ and $X_0$ in the notation of Lemma \ref{lem:PX0Psi-prob}.
Furthermore, we can readily verify the flow property
\begin{equ}[eq:flow_prop]
\Lambda_{t,u}(\Lambda_{s,t}(X_0)) = \Lambda_{s,u}(X_0) \quad \textnormal{a.s.}
\end{equ}
for any $\mbF$-stopping times $s\leq t\leq u$ and $\mbF_s$-measurable $X_0$.
Indeed,~\eqref{eq:flow_prop} follows from the analogous identity at the level of the mollification $\xi^\eps$ combined with continuity of $x\mapsto \Lambda_{s,t}(x)$ in probability.

For a metric space $F$, denote by $D(\R_+,F)$
the usual Skorokhod space of c{\`a}dl{\`a}g functions.
We extend the gauge equivalent relation $\sim$ to $\hat\state=\state\sqcup\{\skull\}$
by imposing that $X\sim \skull\Leftrightarrow X=\skull$.

Recall that
we have fixed 
the parameters $\eta,\beta,\delta,\alpha,\theta$
in the way described in the beginning of Section~\ref{sec:renorm-A} 
which is required to define the space
$(\state,\Sigma)\equiv (\state_{\eta,\beta,\delta,\alpha,\theta},\Sigma_{\eta,\beta,\delta,\alpha,\theta})$ as in Definition~\ref{def:state}.

Consider $\bar\eta\in(\eta,-\frac12)$,
$\bar\delta\in(\frac34,\delta$) such that $\beta<-2(1-\bar\delta)-$,
and $\bar\alpha\in(\alpha,\frac12)$.
Note that $(\bar\eta,\beta,\bar\delta)$ also satisfies~\eqref{eq:CI}. 
Denote $\bar\Sigma\equiv \Sigma_{\bar\eta,\beta,\bar\delta,\bar\alpha,\theta}$.

\begin{definition}\label{def:gen}
A probability measure $\mu$ on $D(\R_+,\hat\state)$ is called \emph{generative} if 
there exists a filtered probability space $(\Omega^{\noise},\mcF,(\mcF_t)_{t \geq 0}, \P)$
supporting a white noise $\xi$
for which the filtration $\mbF=(\mcF_t)_{t \geq 0}$ is admissible
and a $D(\R_+,\hat\state)$-valued random variable $X$ on $\Omega^{\noise}$ with the following properties.
\begin{enumerate}
\item\label{pt:time_zero} The law of $X$ is $\mu$ and $X(0)$ is $\mcF_0$-measurable.

\item\label{pt:dynamics} Let $\Lambda$ denote the SYMH solution map.
There exists an non-decreasing sequence of 
$\mbF$-stopping times $(\varsigma_j)_{j=0}^\infty$,
such that a.s. $\varsigma_0=0$ and, for all $j \geq 0$,
\begin{enumerate}
\item\label{pt:solves_SPDE} one has $X(t) = \Lambda_{\varsigma_j,t}(X(\varsigma_j))$ for all $t\in[\varsigma_j,\varsigma_{j+1})$, and

\item\label{pt:gauge_at_jumps} the random variable $X(\varsigma_{j+1})$ is $\mcF_{\varsigma_{j+1}}$-measurable and
\begin{equ}
X(\varsigma_{j+1}) \sim \Lambda_{\varsigma_j,\varsigma_{j+1}}(X(\varsigma_j))\;.
\end{equ}
\end{enumerate} 

\item\label{pt:honest_blow_up} Let $T^* \eqdef \inf\{t \geq 0 \,:\, X(t) = \skull\}$.
Then a.s. $\lim_{j\to\infty}\varsigma_j=T^*$.
Furthermore, on the event $\{T^*<\infty\}$,
almost surely
\begin{itemize}
\item $X\equiv \skull$ on $[T^*,\infty)$, and
\item one has
\begin{equ}\label{eq:honest_blow_up}
\lim_{t\nearrow T^*} \inf_{\substack{Y\in\state\\Y\sim X(t)}} \bar\Sigma(Y,0) = \infty\;.
\end{equ}
\end{itemize}
\end{enumerate}
If there exists $x\in\hat\state$ such that 
$X(0)=x$ almost surely, then we call $x$ the initial condition of $\mu$.
\end{definition}

\begin{remark}
The meaning of item~\ref{pt:dynamics} is that, before blow-up, the dynamic $X$ runs on the interval $[\varsigma_j,\varsigma_{j+1})$ according to SYMH, after which it jumps to a new point in its gauge orbit and restarts SYMH at time $\varsigma_{j+1}$.
The main point of item~\ref{pt:honest_blow_up}
is that the whole gauge orbit $[X(t)]$ ``blows up''
at time $T^*$ provided we measure ``size'' in a slightly stronger sense than $\Sigma$.
We use $\bar\Sigma$ in~\eqref{eq:honest_blow_up} due to technical issues with measurable selections;
using $\Sigma$ in~\eqref{eq:honest_blow_up} would be more natural,
but we do not know if Theorem~\ref{thm:Markov_process}\ref{pt:gen_exists} below remains true in this case. The reason why this is nevertheless sensible is that even if we start with
an initial condition in $\CS$ for which $\bar \Sigma$ is infinite, it is immediately finite
so that the infimum in \eqref{eq:honest_blow_up} is almost surely finite for all $t \in (0,T^*)$.
\end{remark}

\begin{remark}
In the setting of Definition~\ref{def:gen},
if $Z\colon\Omega^{\noise}\to\hat\state$ is $\mcF_s$-measurable, 
then $t\mapsto \Lambda_{s,t}(Z)$ is adapted to $(\mcF_t)_{t \geq s}$.
Hence $X$ is adapted to $\mbF$, $T^*$ is an $\mbF$-stopping time,
the event~\eqref{eq:honest_blow_up} is $\mcF_{T^*}$-measurable, and $T^*$ is predictable.
\end{remark}

Define the quotient space $\hat\mfO\eqdef \hat\state/{\sim}$,
which we can readily identify, as a set, with $\mfO\sqcup\{\skull\}$.
Note that $\hat\mfO$ is not Hausdorff,
but there is a simple description of its topology in terms of the (completely Hausdorff) topology of $\mfO$.
Indeed, $O\subset\hat\mfO$ is open if and only if either ($O=\hat\mfO$) or ($\skull\notin O$ and $O\subset\mfO$ is open),
see Proposition~\ref{prop:hat_mfO_top}. 

Let $\pi\colon\hat\state\to\hat\mfO$ denote the canonical projection map.
If $\mu$ is generative, then thanks to Definition~\ref{def:gen},
the pushforward $\pi_*\mu$ is a probability measure
on $\CC(\R_+,\hat\mfO)$, instead of just on $D(\R_+,\hat\mfO)$. Our main result
regarding the construction of a suitable Markov process on $\mfO$ is as follows.

\begin{theorem}\label{thm:Markov_process}
\begin{enumerate}[label=(\roman*)]
\item \label{pt:gen_exists} For every $x\in\hat\state$, there exists a generative 
probability measure $\mu$ with initial condition $x$.

\item \label{pt:unique_Markov}
For every $z\in\hat\mfO$, there exists a (necessarily unique) probability measure $\P^z$ on $\CC(\R_+,\hat\mfO)$
such that, for every $x\in z$
and generative probability measure $\mu$ with initial condition $x$,
$\pi_*\mu=\P^z$.
Furthermore, $\{\P^z\}_{z\in\hat\mfO}$ define the transition functions of a time homogeneous Markov process on $\hat\mfO$.
\end{enumerate}
\end{theorem}

The proof of Theorem~\ref{thm:Markov_process} is postponed to Section~\ref{sec:MarkovProcess}.
Before that, we provide a number of results showing how to exploit the gauge covariance
obtained in Theorem~\ref{thm:gauge_covar} in this setting.

\subsection{Coupling SYMH with gauge-equivalent initial conditions}
\label{subsec:coupling}

Throughout this subsection, we fix a white noise $\xi$ defined on a probability space $(\Omega^{\noise},\CF,\P)$ and let $\mbF=(\mcF_t)_{t\geq0}$\label{pageref:mbF2} denote its canonical filtration.
Let furthermore $x,\bar x\in \state$
and $X=(A,\Phi)\in\state^\sol$ solve SYMH driven by $\xi$ with initial condition $X(0)=x$.

If $\bar x = x^{g(0)}$ for some $g(0)\in\mfG^1$,
it follows readily from Theorem~\ref{thm:gauge_covar}
that $\bar X \eqdef X^g$ solves SYMH driven by $g\xi = (\Ad_g\xi_1,\Ad_{g}\xi_2,\Ad_g\xi_3,g\zeta)$
with initial condition $\bar x$,  
where $g$ is the $\eps\downarrow0$ limit of the solutions to~\eqref{eq:SPDE_for_g_wrt_A}
with initial condition $g(0)$ (and where $(A,\Phi)=X$).
The point of this subsection is to prove that if only $\bar x \sim x$,
then a similar statement still holds true.
Namely, we have the following result.

\begin{proposition}\label{prop:gauge_covar_couple}
Suppose $x\sim\bar x$.
Then, on the same probability space $(\Omega^{\noise},\CF,\P)$,
there exists a process $(\bar X,g)$ and a white noise $\bar \xi$, both adapted to $\mbF$, 
such that $\bar X$ solves SYMH driven by $\bar \xi$ with initial condition $\bar x$
and such that $g(t)\in\mfG^{\frac32-}$ and $\bar X(t) = X(t)^{g(t)}$
for all $t>0$ inside the interval of existence of $(\bar X, g)\in\state\times \mfG^{\frac32-}$.
Furthermore, $g$ cannot blow up in $\mfG^{\frac32-}$ before $\Sigma(X)+\Sigma(\bar X)$ does.
\end{proposition}

\begin{remark}
The proof of Proposition~\ref{prop:gauge_covar_couple} will reveal that $g$ solves~\eqref{eq:SPDE_for_g_wrt_A}
with some initial condition $g^{(0)}\in\mfG^{\frac12-}$.
\end{remark}

\begin{lemma}\label{lem:Sigmas_converg}
For every $x\in\state$,
\begin{enumerate}[label=(\roman*)]
\item\label{pt:F_t} $\lim_{t\downarrow 0}\Sigma(\ymhflow_t(x),x)=0$,
\item\label{pt:P_F_t} $\lim_{s\downarrow0}\sup_{t\in [0,T]} \Sigma(\CP_s \ymhflow_t (x), \ymhflow_t (x)) = 0$,
where $T\leq \Poly (\Theta(x)^{-1})$ is a time of existence of $\ymhflow_t(x)$ appearing in Proposition~\ref{prop:YM_flow_minus_heat}.
\end{enumerate}
\end{lemma}

\begin{proof}
\ref{pt:F_t}
Writing $p_t \eqdef \CP_{t}x$ and
$\ymhflow_t x = p_t + r_t$, it follows from Proposition~\ref{prop:Theta_cont_zero}\ref{pt:heat_Sigma}
that $\lim_{t\to0}\Sigma(p_t,x)=0$. 
Observe further that $|r_t|_{\CC^{\hat\beta}}\to 0$ by Proposition~\ref{prop:YM_flow_minus_heat}.
Since $\hat\beta\geq 2\theta(\alpha-1)$ by assumption~\eqref{eq:bar_beta_cond1},
it follows from Lemma~\ref{lem:heatgr_Besov_embed} that $\lim_{t\to0}\heatgr{r_t}_{\alpha,\theta}=0$.
Furthermore $|r_t|_{\CC^\eta}\to 0$ since, by \eqref{eq:bar_beta_cond2} and the conditions on 
$\eta$ and $\delta$ (which in particular imply that $\eta + \delta < \f12$ and therefore $\hat \eta > \eta$), one has $\eta<\hat\beta$.
Finally, by Lemma~\ref{lem:perturbation}\ref{pt:perturbation},
\begin{equ}
\fancynorm{p_t+r_t;x}_{\beta,\delta} \lesssim \fancynorm{p_t;x}_{\beta,\delta}
+ \fancynorm{r_t}_{\beta,\delta}
+ (|p_t|_{\CC^\eta}+|x|_{\CC^\eta})|r_t|_{\CC^{\hat\eta}}\;,
\end{equ}
with $\hat \eta$ as in \eqref{eq:bar_beta_cond2}. As $t\to0$,
the first term on the right-hand side converges to $0$ since $\Theta(p_t,x)\to0$,
the second term converges converges to $0$ since
\begin{equ}
\fancynorm{r_t}_{\beta,\delta}\leq\fancynorm{r_t}_{0,\delta} \lesssim |r_t|^2_{C^{\hat\eta}}
\end{equ}
by Lemma~\ref{lem:perturbation}\ref{pt:fancynorm_0} and since $\hat\beta \geq \hat\eta>\frac12-\delta$ by assumption~\eqref{eq:bar_beta_cond2},
and the third term converges to $0$ also because of~\eqref{eq:bar_beta_cond2}.

\ref{pt:P_F_t} As above, write $\CF_t x= p_t+r_t$.
Since $\CP_t$ is a contraction for the norm $\heatgr{\cdot}_{\alpha,\theta}$ by Lemma~\ref{lem:CP_contraction}\ref{pt:heatgr_contraction},
it follows that
\begin{equ}
\sup_{t\in[0,T]}\heatgr{\CP_s p_t - p_t}_{\alpha,\theta}
= \sup_{t\in[0,T]}\heatgr{\CP_t (\CP_s x - x)}_{\alpha,\theta}
\leq \heatgr{\CP_s x - x}_{\alpha,\theta}\;,
\end{equ}
which converges to $0$ as $s\downarrow0$ by Proposition~\ref{prop:Theta_cont_zero}\ref{pt:heat_Sigma}.
By an identical argument,
\begin{equ}[eq:Ps_pt_bound]
\lim_{s\downarrow 0}\sup_{t\in[0,T]}|\CP_s p_t - p_t|_{\CC^\eta}= 0\;.
\end{equ}
Furthermore, by Proposition~\ref{prop:YM_flow_minus_heat},
$\lim_{t\downarrow0}|r_t|_{\CC^{\hat\beta}}= 0$ and
$\sup_{t\in [\eps,T]}|r_t|_{L^\infty}<\infty$ for every $\eps>0$.
Using that $|\CP_s f - f|_{\CC^{\hat\beta}}\lesssim s^{-\hat\beta/2}|f|_{L^\infty}$ and $\hat\beta<0$,
it follows from~\eqref{eq:Ps_rt}
that
\begin{equ}[eq:Ps_rt_bound]
\lim_{s\downarrow0}\sup_{t\in[0,T]}|\CP_s r_t - r_t|_{\CC^{\hat\beta}}=0\;.
\end{equ}
Therefore, using $\hat\beta\geq 2\theta(\alpha-1)$ and Lemma~\ref{lem:heatgr_Besov_embed} to estimate
$\heatgr{\CP_s r_t - r_t}_{\alpha,\theta}
\lesssim |\CP_s r_t - r_t|_{\CC^{\hat\beta}}$
we obtain
\begin{equ}[eq:Ps_rt]
\lim_{s\downarrow0}\sup_{t\in[0,T]}\heatgr{\CP_s r_t - r_t}_{\alpha,\theta} =0 \;.
\end{equ}
Since $\eta<-\frac12<\hat\beta$,~\eqref{eq:Ps_rt_bound} also implies
$\lim_{s\downarrow0}\sup_{t\in[0,T]}|\CP_s r_t - r_t|_{\CC^\eta}=0$.

It remains only to consider the terms with $\fancynorm{\cdot;\cdot}_{\beta,\delta}$. 
By Lemma~\ref{lem:perturbation}\ref{pt:perturbation},
uniformly in $t\geq0$,
\begin{equs}
\fancynorm{\CP_s (p_t+r_t) ; p_t+r_t}_{\beta,\delta}
&\lesssim
\fancynorm{\CP_s p_t ; p_t}_{\beta,\delta}
+ \fancynorm{\CP_s r_t ; r_t}_{\beta,\delta}
\\
&\qquad + |\CP_s p_t - p_t|_{\CC^\eta}|r_t|_{\CC^{\hat\eta}}
+|\CP_s r_t - r_t|_{\CC^{\hat\eta}}|x|_{\CC^\eta}\;,
\end{equs}
for $\hat\eta$ as in \eqref{eq:bar_beta_cond2}.
Since $\hat\beta\geq \hat\eta$, it follows from~\eqref{eq:Ps_pt_bound}-\eqref{eq:Ps_rt_bound} that the final two terms converge to $0$ as $s\downarrow0$ uniformly in $t\in[0,T]$.
For the first term, observe that, by~\eqref{eq:Ph_X_Y},
\begin{equ}
 \sup_{t\in[0,T]} \fancynorm{\CP_s p_t ; p_t}_{\beta,\delta}
=
 \sup_{t\in[0,T]} \fancynorm{\CP_t p_s ; \CP_t x}_{\beta,\delta}
\lesssim
|p_s-x|_{\CC^\eta}|x|_{\CC^\eta} + \fancynorm{p_s;x}_{\beta,\delta}\;.
\end{equ}
Since $\lim_{s\downarrow0}\Theta(p_s,x)=0$ by 
Proposition~\ref{prop:Theta_cont_zero}\ref{pt:heat_Theta},
\begin{equ}
\lim_{s\downarrow0}\sup_{t\in[0,T]}\fancynorm{\CP_s p_t ; p_t}_{\beta,\delta}
=0\;.
\end{equ}
Furthermore, by Lemma~\ref{lem:perturbation}\ref{pt:fancynorm_0}, uniformly in $t\geq0$,
\begin{equ}
\fancynorm{\CP_s r_t ; r_t}_{\beta,\delta}
\lesssim
|\CP_s r_t - r_t|_{\CC^{\hat\eta}}|r_t|_{\CC^{\hat\eta}}\;.
\end{equ}
The final term again converges to $0$ as $s\downarrow0$ uniformly in $t\in[0,T]$.
\end{proof}

\begin{proof}[of Proposition~\ref{prop:gauge_covar_couple}]
Let us define $x_n,\bar x_n\in \CC^\infty(\T^3,E)$ by
\begin{equ}
x_n\eqdef \ymhflow_{1/n}(x)\;,
\quad \bar x_n\eqdef \ymhflow_{1/n}(\bar x)\;,
\end{equ}
which exist for $n$ sufficiently large.
Let $g_{n}(0) \in \mfG^\infty$
such that $x_n^{g_{n}(0)} = \bar x_n$.
It follows from Lemma~\ref{lem:Sigmas_converg}\ref{pt:F_t} and Theorem~\ref{thm:g_bound} that, for some $\nu<\frac12$,
\begin{equ}
\limsup_{n\to\infty} |g_{n}(0)|_{\CC^\nu} < \infty\;.
\end{equ}
In particular, passing to a subsequence, we can assume $g_{n}(0)\to g(0)$ in $\mfG^{0,\nu-}$.
Let $X_n=(A_n,\Phi_n)$ solve SYMH driven by $\xi$ with initial condition $x_n$
and let $g_n$ solve~\eqref{eq:SPDE_for_g_wrt_A}
driven by $A_n$
with initial condition $g_{n}(0)$. 

Consider further the process $\xi_n$ given by $g_n\xi$ on $[0,T_{X_n,g_n})$ and by $\xi$ outside this interval, where we set, as in~\cite[Sec.~1.5.1]{CCHS2d},
\begin{equ}
T_{X,g}\eqdef\inf\{t\geq0\,:\,(X,g)(t)=\skull\}\;.
\end{equ}
Remark that $\xi_n$ is a white noise by It{\^o} isometry.
Let $\bar X_n$ 
solve SYMH driven by $\xi_n$
with initial condition $\bar x_n$.
Observe that $\bar X_n$ and $X_n^{g_n}$ are strong Markov processes,
for which, by Theorem~\ref{thm:gauge_covar},
$
\bar X_n = X_n^{g_n}
$
on an interval $[0,\tau]$ where $\tau$ is a stopping time bounded stochastically from below by the size of $(X_n(0), g_n(0))$ in $\state\times\mfG^{0,\rho}$.
It follows that
\begin{equ}\label{eq:bar_X_n_X_n_equal}
\bar X_n=X_n^{g_n} \;\; \text{ on the interval of existence of } \;(\bar X_n, g_n)\;.
\end{equ}

Next, we claim that
\begin{equ}\label{eq:X_n_to_X}
\lim_{n\to\infty}(X_n,X_n^{g_n}) = (X,X^g)\;\; \text{ in probability in }\; (\state\times\state)^\sol\;,
\end{equ}
where $g$ solves~\eqref{eq:SPDE_for_g_wrt_A}
with initial condition $g(0)$
and where we \textit{define} $X^g(0) = \bar x$ (for $t>0$, note that $g(t)\in\mfG^{\frac32-}$ and thus $X^g(t)$ makes sense by Proposition~\ref{prop:group_action}).
To prove~\eqref{eq:X_n_to_X}, remark that
$\Sigma(x_n,x)\to0$ by Lemma~\ref{lem:Sigmas_converg}\ref{pt:F_t}.
It thus follows from Lemma~\ref{lem:flow_for_rough_g} with $c^\eps=0$ therein
that $(X_n,g_n) \to (X,g)$ in probability
in $(\state\times \mfG^{0,\nu-})^\sol$.
Therefore, for any $\lambda,\delta>0$,
\begin{equ}
\lim_{n\to\infty }\P \Big[\sup_{t\in [\lambda, T_{X,g}-\lambda]}|g_n(t)-g(t)|_{\mcC^{\frac32-}} > \delta\Big] = 0\;.
\end{equ}
It follows by joint continuity of the group action from Proposition~\ref{prop:group_action} and the estimates on gauge transformations from Theorem~\ref{thm:g_bound}
(which ensures that $|g|_{\CC^{1/2-}}$ and thus $|g|_{\CC^{3/2-}}$
cannot blow up before $\Sigma(X)+\Sigma(X^g)$ does, hence $T_{X,g}=T_{X,X^g}$) that
\begin{equ}[eq:X_X_n_from_zero]
\lim_{n\to\infty }\P \Big[\sup_{t\in [\lambda, T_{X,X^g}-\lambda]} \{\Sigma(X_n(t),X(t)) + \Sigma(X_n^{g_n}(t),X^g(t)) \}> \delta\Big] = 0\;.
\end{equ}
It remains to handle the time interval $[0,\lambda]$ for which we can proceed similar to the proof of Proposition~\ref{prop:SPDEs_conv_zero_state}.
Specifically, it holds that
\begin{equ}\label{eq:X_n_at_zero}
\lim_{\lambda\downarrow0}\limsup_{n\to\infty} \P
\Big[
\sup_{t\in[0,\lambda]} \Sigma(X^{g_n}_n(t),X^{g_n}_n(0)) > \delta
\Big]
= 0\qquad \forall \delta>0\;.
\end{equ}
Indeed, this follows from $X_n^{g_n} = \bar X_n$, which solves the SYMH driven by $\xi_n$ with initial condition $\bar x_n$, and $\lim_{s\downarrow 0}\sup_{n}\Sigma(\CP_s \bar x_n,\bar x_n)\to 0$ by Lemma~\ref{lem:Sigmas_converg}\ref{pt:P_F_t},
combined with the arguments from the proof of Theorem~\ref{thm:local-exist-sigma} to handle the remainder terms.
The proof of~\eqref{eq:X_n_to_X} now follows from combining~\eqref{eq:X_X_n_from_zero}-\eqref{eq:X_n_at_zero} and the facts that $\Sigma(\bar x,\bar x_n)\to 0$ by Lemma~\ref{lem:Sigmas_converg}\ref{pt:F_t}
and that $X_n\to X$ in probability in $\state^\sol$.

Finally, since $\xi_n\eqlaw\xi$ and $\Sigma(\bar x_n,\bar x)\to 0$ by Lemma~\ref{lem:Sigmas_converg}, $\bar X_n\convlaw \bar X$ in $\state^\sol$, where $\bar X$  is equal in law to the solution to SYMH with initial condition $\bar x$.
By Remark~\ref{rem:recover_noise},
there exists a white noise $\bar \xi$ on the probability space $(\Omega^{\noise},\CF,\P)$ which is adapted to $(\CF_t)_{t \geq 0}$ such that $\bar X$ solves SYMH driven by $\bar \xi$
with initial condition $\bar x$.
Combining with~\eqref{eq:X_n_to_X} and~\eqref{eq:bar_X_n_X_n_equal},
we obtain the first claim in the proposition statement.

For the second claim we first note that Theorem~\ref{thm:g_bound} gives control of $|g|_{\CC^{1/2-}}$ in terms of $\Sigma(X)+\Sigma(\bar{X})$.
Next we point out that, by Lemma~\ref{lem:flow_for_rough_g}, for any $\nu > 0$,
we can start the limiting $\eps \downarrow 0$ dynamic \eqref{e:SPDE-for-X_with_g}-\eqref{eq:SPDE_for_g_coupled}  with initial data $g \in \mfG^{\nu}$. 
Moreover, the limiting dynamic has maximal solutions in a space where the $g$ component belongs to $\mfG^{\frac32-}$. 
Therefore, $g$ can blow up in $\mfG^{\frac32-}$ only if it blows up in $\mfG^{\nu}$ for every $\nu>0$.
\end{proof}

\subsection{Measurable selections}
\label{subsec:measurable_selections}

Define $m\colon \state \to [0,\infty]$ by
\begin{equ}\label{eq:def_m}
m(X) \eqdef \inf_{Y\sim X} \bar\Sigma(Y)\;,
\end{equ}
where $\sim$ is the equivalence relation on $\state\subset\init$ defined in Definition~\ref{def:ymhflow}.
The following is the main result of this subsection.

\begin{theorem}\label{thm:selection}
There exists a Borel function $S\colon \state \to\state$
such that for all $X\in\state$
\begin{itemize}
\item $S(X)\sim X$,
\item $\bar\Sigma(S(X)) \leq 2m(X)$ if $m(X) < \infty$, and
\item $S(X)=X$ if $m(X)=\infty$.
\end{itemize}
\end{theorem}

\begin{proof}[of Theorem~\ref{thm:selection}]
We start with the following preliminary statement.

\begin{lemma}\label{lem:m_lower_semicont}
The function $m$ given in \eqref{eq:def_m} is lower semi-continuous on $\state$.
\end{lemma}

\begin{proof}
Let $X_n\to X$ in $\state$ and consider $Y_n\in\state$ such that $Y_n\sim X_n$ and $\bar\Sigma(Y_n)-m(X_n)<\frac1n$.
Passing to a subsequence, we may suppose that $\lim_{n\to\infty}m(X_n)=\liminf_{n\to\infty}m(X_n)$.
If $m(X_n) \to \infty$, then there is nothing to prove.
Otherwise, $\sup_n \bar\Sigma(Y_n) < \infty$,
so by the compact embedding due to Proposition~\ref{prop:compact}\ref{pt:Sigma_compact},
there exists $Y\in\state$ such that $\Sigma(Y_n,Y)\to 0$
and
$\bar\Sigma(Y) \leq \liminf_n \bar\Sigma(Y_n) = \lim_n m(X_n)$.
Since $(X_n,Y_n)\to (X,Y)$ in $\state^2$
and $X_n\sim Y_n$, it follows that $X\sim Y$ by Lemma~\ref{lem:gauge_orbits_closed}.
\end{proof}

Consider the sets $B \eqdef \{X\in\state \,:\, \bar\Sigma(X) \leq 2m(X) < \infty\}$
and
$\Gamma\subset \state^2$
defined by the disjoint union
\begin{equ}
\Gamma \eqdef \{(X,X)\,:\, m(X)=\infty\} \cup \{(X,Y)\,:\, X\sim Y\,, \; Y\in B
\}\;.
\end{equ}

\begin{lemma}\label{lem:Gamma_compact}
For every $X\in\state$, $\{Y\in \state\,:\,(X,Y)\in\Gamma\}$
is non-empty and compact.
\end{lemma}

\begin{proof}
If $m(X)=\infty$, the result is obvious.
Otherwise the result follows from the invariance of $m$ on $[X]$,
the fact that $[X]$ is closed in $\state$ by Lemma~\ref{lem:gauge_orbits_closed},
and the fact that $\{Y\in\state\,:\, \bar\Sigma(Y)\leq 2m(X)\}$ is compact in $\state$ by Proposition~\ref{prop:compact}\ref{pt:Sigma_compact}.
\end{proof}

We now have all the ingredients in place to complete the proof of Theorem~\ref{thm:selection}.
Since $\bar\Sigma,m\colon\state\to[0,\infty]$
are lower semi-continuous (Lemmas~\ref{lem:Theta_Sigma_lower_semicont} and~\ref{lem:m_lower_semicont}),
$B$ is Borel in $\state$.
Moreover, $\{(X,Y)\in\state^2\,:\, X\sim Y\}$ is closed in $\state^2$
(Lemma~\ref{lem:gauge_orbits_closed})
and thus $\Gamma$ is Borel in $\state^2$.
The conclusion follows from Lemma~\ref{lem:Gamma_compact} and~\cite[Thm~6.9.6]{Bogachev07} since $\state$ is a Polish space.
\end{proof}

%

\subsection{Proof of Theorem~\ref{thm:Markov_process}}\label{sec:MarkovProcess}

\begin{proof}[of Theorem~\ref{thm:Markov_process}]
\ref{pt:gen_exists}
We extend the function $S$ in Theorem~\ref{thm:selection} to a Borel function $S\colon\hat\state \to \hat\state$ by $S(\skull)=\skull$.
Let $\xi$ be a white noise defined on a probability space $(\Omega^{\noise},\mcF,\P)$
and let $\mbF=(\mcF_t)_{t\geq0}$ be the canonical filtration generated by $\xi$.

Consider any $x \in \hat\state$.
We define a c{\`a}dl{\`a}g process $X\colon\Omega^{\noise}\to D(\R_+, \hat\state)$ with $X(0)=x$,
and a sequence of $\mbF$-stopping times $(\varsigma_j)_{j=0}^\infty$ as follows.
The idea is to run the SYMH equation until the first time $t\geq \varsigma_j$  such that $\bar\Sigma(X(t)) \geq
1+4m(X(t))$, at which point we use $S$ to jump to a new representative of $[X(t)]$
and then restart SYMH, setting $\varsigma_{j+1}=t$.
At time $t=0$, however, it is possible that $\bar\Sigma(x)=m(x)=\infty$, in which case this procedure does not work.
We therefore treat $t=0$ as a special case.

If $x=\skull$, we define $X(t)=\skull$ for all $t\in\R_+$ and $\varsigma_j=0$ for all $j\geq0$.
Suppose now that $x\neq\skull$.
Set $\varsigma_0 = 0$ and $X(0)=x$.
Define $Y \in \mcC([0,\infty),\hat\state)$
by
$Y(t) =
\Lambda_{0,t}(x)$
and $R \in \mcC([0,\infty),\hat\R_+)$ by
\begin{equ}
R(t)\eqdef \Sigma(Y(t))-\Sigma(x)\;.
\end{equ}
Define further
\begin{equ}
\varsigma_{1}
\eqdef
\inf
\{
t > 0 \,:\, R(t)=1)\;.
\end{equ}
We then define $X(t) = Y(t)$ for $t\in(0,\varsigma_{1})$
and $X(\varsigma_{1}) = S(Y(\varsigma_{1}))$.
Observe that a.s.
\begin{equ}
0<\varsigma_1 < \inf\{t>0\,:\, R(t)=\infty\}
\end{equ}
and thus $\bar\Sigma(X(\varsigma_1))<\infty$ because $\Sigma(Y(t))<\infty\Rightarrow\bar \Sigma (Y(t))<\infty$ for all $t>0$.

Having defined $\varsigma_1$,
consider now $j \geq 1$.
If $\varsigma_j=\infty$, we set $\varsigma_{j+1}=\infty$.
Otherwise, if $\varsigma_j<\infty$, 
suppose that $X$ is defined on $[0,\varsigma_{j}]$.
If $X(\varsigma_j)=\skull$, then define $\varsigma_j=\varsigma_{j+1}$.
Suppose now that $X(\varsigma_j)\neq\skull$.
Define $Y \in \mcC([\varsigma_{j},\infty),\hat\state)$
by
\begin{equ}
Y(t) \eqdef
\Lambda_{\varsigma_{j},t}(X(\varsigma_{j}))
\end{equ}
and
\begin{equ}
\varsigma_{j+1}
\eqdef
\inf
\{
t > \varsigma_j \,:\, \bar\Sigma(Y(t)) \geq
1+4m(Y(t))\}\;.
\end{equ}
We then define $X(t) = Y(t)$ for $t\in(\varsigma_j,\varsigma_{j+1})$
and $X(\varsigma_{j+1}) = S(Y(\varsigma_{j+1}))$.
Observe that the map $[\varsigma_j,\infty)\to\hat\R_+$, $t\mapsto \bar\Sigma(Y(t))$, is a.s. continuous.
Furthermore, $\bar\Sigma(Y(\varsigma_j)) \leq 2m(Y(\varsigma_j))$
and the map $[\varsigma_j,\infty) \ni t \mapsto m(Y(t))$ is lower semi-continuous (Lemma~\ref{lem:m_lower_semicont}).
Therefore, 
a.s. $\varsigma_{j+1}>\varsigma_j$.

Let $T^*\eqdef \lim_{j\to\infty}\varsigma_j$.
If $T^*<\infty$, we define $X(t)=\skull$ for all $t\in[T^*,\infty)$.
We have now defined $X(t)\in\hat\state$ for all $t\in \R_+$ and, by construction, $T^* = \inf\{t \geq 0 \,:\, X(t) = \skull\}$.
Moreover, $X$ takes values in $D(\R_+,\hat\state)$.

To prove~\ref{pt:gen_exists}, it remains to show that
\begin{equ}\label{eq:m_blow_up}
\P\Big[T^*<\infty\,,\, \liminf_{t\nearrow T^*}m(X(t)) \leq M\Big] = 0\qquad \forall M>0\;.
\end{equ}
Since, by construction, $\bar\Sigma(X(t))\in[m(X(t)),1+4m(X(t))]$ for all $t\geq \varsigma_1$, it suffices to prove that
\begin{equ}\label{eq:barSigma_blow_up}
\P\Big[T^*<\infty\,,\, \liminf_{t\nearrow T^*}\bar\Sigma(X(t)) \leq M
\Big] = 0\qquad \forall M>0\;.
\end{equ}
To this end,
consider $M>1$.
For any stopping time $\tau < T^*$, on the event $\bar\Sigma(X(\tau))\leq M$, we have
\begin{equ}[eq:fluctuation]
\sup_{t\in [\tau, \tau+\delta]}\bar\Sigma(\Lambda_{\tau,t}(X(\tau))) \leq \bar\Sigma(X(\tau))+1/2
\end{equ}
where $\delta\in (0,1)$ is random and depends on $X(\tau)$ and the realisation of $\xi$ on $[\tau,\tau+1]$,
but is stochastically bounded from below by a function of $M$
in the sense that there exists $\eps=\eps(M)\in (0,1)$ such that 
\begin{equ}[eq:delta_eps]
\P[\delta > \eps \mid \bar\Sigma(X(\tau))\leq M ] > \eps\;.
\end{equ}
Indeed,~\eqref{eq:fluctuation}-\eqref{eq:delta_eps} follows from the fact that $\bar\Sigma(\CP_r Z) \leq (1+Cr)\bar\Sigma(Z)$ for any $Z\in\state$ by Lemma~\ref{lem:CP_contraction},
from the stochastic bounds on $\sup_{r\in[0,h]}r^{-\kappa}\bar\Sigma(\Psi_r,\CP_r\Psi_0)$ for the SHE $\Psi$ from~\eqref{eq:SHE_Theta_zero} and Proposition~\ref{prop:SHE_heatgr_estimates} and the stochastic bounds in Lemmas~\ref{lem:tildePsi-cov}-\ref{lem:PX0Psi-prob} (we use that $\tau$ is a stopping time in all of these stochastic bounds),
and from the perturbation argument similar to that in the proof of Theorem~\ref{thm:local-exist-sigma} that uses Lemma~\ref{lem:perturbation}
(see~\eqref{eq:perturbation_X}).


Consider the sequence of stopping times $\tau_j = \inf\{t\geq \varsigma_j \,:\, \bar\Sigma (X(t)) \leq M\}$.
For $j\geq 2$, we have
\begin{equ}
\bar\Sigma(X(\varsigma_{j})) \leq \frac12\lim_{t\nearrow \varsigma_j}
\{\bar\Sigma( X(t)) - 1\}\;,
\end{equ}
and thus, if $\tau_j <\varsigma_{j+1}$ and
\begin{equ}
\sup_{t\in[\tau_j,\tau_j+\eps]}\bar\Sigma(\Lambda_{\tau_j,t}(X(\tau_j))) \leq \bar\Sigma(X(\tau_j))+1\;,
\end{equ}
then by the flow property~\eqref{eq:flow_prop} (which allows us to restart the dynamic at $\tau_j$), we have
\begin{equ}
\bar\Sigma(X(\varsigma_{j+1})) \leq \bar\Sigma(X(\tau_{j}))/2 \quad
\text{or} \quad \varsigma_{j+1} - \varsigma_j\geq \eps\;.
\end{equ}
Therefore, by~\eqref{eq:fluctuation}-\eqref{eq:delta_eps}, a.s.
\begin{equ}
\bone_{\tau_j \geq \varsigma_{j+1}} +
\P\big[
\bar\Sigma(X(\varsigma_{j+1}))\leq \bar\Sigma(X(\tau_j))/2
\text{ or } \varsigma_{j+1} - \tau_j\geq \eps
\,\big|\,
\mcF_{\tau_j}
\big] > \eps\;.
\end{equ}
Remark also that $\bar\Sigma(X(\varsigma_{j+1}))\leq \bar\Sigma(X(\tau_j))/2$ implies
$\bar\Sigma(X(\varsigma_{j+1}))\leq M/2$ and thus $\tau_{j+1}=\varsigma_{j+1}$. 
It now follows from the strong Markov property of $\{X_{\tau_1+t}\}_{t\geq0}$ that, for any integer $q\geq1$,
\begin{equ}[eq:q_bound]
\bone_{\tau_j \geq \varsigma_{j+1}} +
\P\big[
\bar\Sigma(X(\varsigma_{j+q}))\leq 2^{-q}M
\text{ or } \varsigma_{j+1} - \tau_j\geq \eps
\,\big|\,
\mcF_{\tau_j}
\big] > \eps^q\;.
\end{equ}
Finally, since $\bar\Sigma(X(\varsigma_k))\geq 1$ for all $k\geq 1$, by taking any $q$ in~\eqref{eq:q_bound} such that $2^{-q}M<1/2$ and by~\eqref{eq:fluctuation}-\eqref{eq:delta_eps}, we obtain
\begin{equ}
\bone_{\tau_j \geq \varsigma_{j+1}} +
\P\big[
\varsigma_{j+q+1} - \varsigma_{j} \geq \eps
\,\big|\,
\mcF_{\tau_j}
\big] > \eps^{q+1}\;.
\end{equ}
It readily follows that, on the event $T^*<\infty$, there are only finitely many $j$ such that $\tau_j< \varsigma_{j+1}$,
which concludes the proof of~\eqref{eq:barSigma_blow_up} and thus of~\eqref{eq:m_blow_up}.


\ref{pt:unique_Markov}
The idea, like in the proof of~\cite[Thm.~2.13(ii)]{CCHS2d}, is to couple any generative 
probability measure $\bar \mu$ to the law of the 
process $X$ constructed in part~\ref{pt:gen_exists}.
Consider $x,\bar x\in\hat\state$ with $x\sim \bar x$
and a generative probability measure $\bar\mu$ on $D(\R_+,\hat\state)$ with initial condition $\bar x$.
Let $(\Omega^{\noise},\mcF,\mbF=(\mcF_t)_{t \geq 0}, \P)$, $\bar\xi$, $\bar X\colon\Omega^{\noise}\to D(\R_+,\hat\state)$,
and $\bar T^*$
denote respectively the corresponding filtered probability space, white noise, random variable,
and blow-up time as in Definition~\ref{def:gen}.

It follows from Proposition~\ref{prop:gauge_covar_couple}
that
there exists, on the same filtered probability space $(\Omega^{\noise},\mcF,\mbF, \P)$,
a c{\`a}dl{\`a}g process $X\colon\R_+\to\hat\state$ constructed 
as in part~\ref{pt:gen_exists} using a white 
noise $\xi$ for which $\mbF$ is admissible and such that $X \sim \bar X$ and $X(0)=x$.
In particular, the pushforwards $\pi_*\bar\mu$ and $\pi_*\mu$ coincide,
where $\mu$ is the law of $X$.

To complete the proof, it remains to show that 
for the process $X$ from part~\ref{pt:gen_exists}
with any initial condition $x\in\hat\state$, the projected process
$\pi X\in\mcC(\R_+,\hat\mfO)$ is Markov.
However, this follows from
taking $\bar \mu$ in the above argument as the law of $X$ with initial condition $\bar x \sim x$. 
\end{proof}

\appendix

\section{Singular modelled distributions}
\label{app:Singular modelled distributions}

In this appendix we collect some useful results 
on singular modelled distributions.
We write $P = \{(t,x) \,:\, t=0\}$ for the time $0$ hyperplane.
Recall that the reconstruction operator $\tilde \CR \colon \cD^{\gamma,\eta}_{\alpha} \to \CD'(\R^{d+1}\setminus P)$\label{def:tildeR} 
 defined in \cite[Sec.~6.1]{Hairer14} is local, and there is in 
full generality no way of canonically extending $\tilde \CR$ to an operator $\CR \colon \cD^{\gamma,\eta}_{\alpha} \to \CD'(\R^{d+1})$. 
Here $\alpha$ denotes the lowest degree of the modelled distribution as usual.
 \cite[Prop.~6.9]{Hairer14} provides such a unique extension under the assumption $\alpha\wedge \eta >-2$ (which is then also required by the integration results \cite[Prop.~6.16]{Hairer14})
  which is insufficient for our purposes in Section~\ref{sec:renorm-A} and~\ref{sec:gauge}.
Below we collect some results which require weaker conditions.

As in~\cite[Sec.~7]{CCHS2d}, it will furthermore be important for us to have short-time convolution estimates on intervals $[0,\tau]$
which depend only on the model up to time $\tau$.
This is used to ensure that a time $\tau>0$ on which fixed point maps are contractions is a stopping time.

We write $\bar \cD^{\gamma,\eta} = \cD^{\gamma,\eta} \cap \cD^\eta$ 
as well as $\hat \cD^{\gamma,\eta} \subset \bar \cD^{\gamma,\eta}$ for the subspace of those  \label{def:Dhat-space}
functions $f\in \bar \cD^{\gamma,\eta}$ such that $f(t,x) = 0$ for $t \le 0$.
Similarly to \cite[Lem.~6.5]{Hairer14} one can show that these are closed subspaces
of $\cD^{\gamma,\eta}$, so that we endow them with the usual norms $\$f\$_{\gamma,\eta}$.
We also note that $\bone_+ f= f$ for all $f\in\hat\cD^{\gamma,\eta}$.

\begin{remark}\label{rem:DDhat-same}
If $\eta \le \alpha$, then 
 $\hat{\cD}^{\gamma,\eta}_\alpha$ simply coincides with the space of all $f \in \cD^{\gamma,\eta}_\alpha$ which vanish for $t \le 0$.
\end{remark}

Recall from~\cite[Appendix~A]{CCHS2d} that $\omega\in\CC^{\alpha\wedge\eta}$ is called compatible with $f\in\cD^{\gamma,\eta}_\alpha$, where $\gamma>0$, if $\omega(\psi)=\tilde\CR f(\psi)$ for all $\psi\in\CC^\infty_c(\R^{d+1}\setminus P)$.
For the proofs of the following results, we refer to 
\cite[Appendix~A]{CCHS2d}.

\begin{theorem}\label{thm:reconstructDomain}
Let $\gamma > 0$ and $\eta \in (-2,\gamma]$. 
There exists a unique continuous linear operator
$\CR:\hat\cD^{\gamma,\eta}_\alpha\to\CC^{\eta\wedge \alpha}$ such that
$\CR f$ is compatible with $f$.
\end{theorem}

\begin{lemma}\label{lem:multiply-hatD}
For $F_i \in \bar \cD^{\gamma_i,\eta_i}_{\alpha_i}$ with $\alpha_i \le 0 < \gamma_i$ and $\eta_i \le \gamma_i$, one has
$F_1 \cdot F_2 \in \bar\cD^{\gamma, \eta}_{\alpha_1+\alpha_2}$ with $\gamma = (\alpha_1 + \gamma_2)\wedge (\alpha_2 + \gamma_1)$ and $\eta = (\alpha_1 + \eta_2)\wedge (\alpha_2 + \eta_1) \wedge (\eta_1 + \eta_2)$.
Furthermore, if $F_i \in \hat \cD^{\gamma_i,\eta_i}_{\alpha_i}$, then
$F_1 \cdot F_2 \in \hat\cD^{\gamma, \eta_1 + \eta_2}_{\alpha_1+\alpha_2}$.
\end{lemma}

Remark that the  multiplication bound \cite[Prop.~6.12]{Hairer14}
and Lemma~\ref{lem:multiply-hatD} also hold for the $\eps$-dependent norms
in Section~\ref{sec:sol-multi}.

Below we assume that 
we have an abstract integration map $\CI$  of order $\beta$\footnote{Note that we write $\beta$ in this section following the convention as in Schauder estimates in~\cite{Hairer14}, and it shouldn't be confused with the parameter $\beta$ in Section~\ref{sec:state_space}.}
and admissible models $Z,\bar Z$ realising a non-anticipative kernel $K$ for $\CI$.
Following~\cite[Sec.~4.5]{MateBoundary} and~\cite[Appendix~A.3]{CCHS2d}, given a space-time distribution $\omega$
and a modelled distribution $f$,
we write $\mcb{K}^\omega f$ for the modelled distribution
defined as in  \cite[Sec.~5]{Hairer14} but with $\CR f$ in the definition replaced by $\omega$.

Working henceforth in the periodic setting $\R\times\T^d$,
we use the shorthand $O_T=[-1,T]\times\T^d$ for $T>0$.
Recall from~\cite[Appendix~A.1]{CCHS2d} the semi-norm $\$Z\$_T$ defined as the smallest constant $C$ such that, for all homogeneous elements $\tau$ in our regularity structure, 
\begin{equ}
\bigl| \bigl(\Pi_x \tau\bigr)(\phi_x^\lambda)\bigr| \le C \|\tau\| \lambda^{\deg\tau}\;,
\end{equ}
for all $\phi \in \CB^r$, $x \in O_T$, and $\lambda \in (0,1]$ such that
$B_\s(x,2\lambda) \subset O_T$, and
\begin{equ}
\|\Gamma_{x,y} \tau\|_\alpha \le C \|\tau\| \, \|x-y\|_\s^{\deg \tau - \alpha}\;,
\end{equ}
for all $x,y \in O_T$.
The pseudo-metric $\$Z;\bar Z\$_T$ is defined analogously.
We will write
$|\act|_{ \cD^{\gamma,\eta};T}$ for the modelled distribution semi-norm associated to the set $O_T$, and likewise for $\hat\cD^{\gamma,\eta}$.
We similarly define $|\omega|_{\CC^\alpha_T}$
for $\omega\in\CD'(\R\times\T^d)$ and $\alpha\in\R$ as the smallest constant $C$ such that
\begin{equ}
|\omega(\phi^\lambda_x)| \leq C\lambda^{\alpha}
\end{equ}
for all $\phi \in \CB^r$, $x \in O_T$, and $\lambda \in (0,1]$ such that
$B_\s(x,2\lambda) \subset O_T$.

\begin{lemma}\label{lem:Schauder-input}
Fix $\gamma>0$. Let $f\in \cD^{\gamma,\eta}_\alpha$ and let $\omega\in \CC^{\eta\wedge \alpha}$ be compatible with $f$.
Set $\bar\gamma=\gamma+\beta$, $\bar\eta=(\eta\wedge \alpha)+\beta$,
which are assumed to be non-integers,
 $\bar\alpha=(\alpha+\beta)\wedge 0$ 
and $\bar\eta\wedge\bar\alpha >-2$. Then $\mcb{K}^\omega f \in \cD^{\bar\gamma,\bar\eta}_{\bar \alpha}$,
and  one has $\CR\mcb{K}^\omega f = K * \omega$.

Furthermore, if $\bar f\in \cD^{\gamma,\eta}_\alpha$ is a modelled distribution
with respect to $\bar Z$ and $\bar\omega\in \CC^{\eta\wedge \alpha}$ is compatible with $\bar f$, then 
 \begin{equ}[e:conti-in-zeta]
 |\mcb{K}^\omega \bone_+f ; \mcb{K}^{\bar\omega}\bone_+\bar f|_{ \cD^{\bar\gamma,\bar\eta};T} 
 \lesssim
   |f;\bar f|_{\cD^{\gamma,\eta};T}
  + \$Z;\bar Z\$_{T}
+  |\omega-\bar\omega|_{\CC^{\eta\wedge \alpha}_T}
 \end{equ}
 locally uniformly in models, modelled distributions and space-time distributions $\omega$, $\bar\omega$.
 Finally, the above bound also holds uniformly in $\eps$
for the $\eps$-dependent norms 
on models and modelled distributions defined in  \cite[Sec.~7.2]{CCHS2d}.
\end{lemma}

 For the spaces $\hat \cD^{\gamma,\eta}$, we have the following version of Schauder estimate.

\begin{theorem}\label{thm:integration}
Let $\gamma>0$, and $\eta > -2$ such that $\gamma+\beta\notin \N$ and $\eta+\beta\notin\N$.
Then, there exists an operator 
$\mcb{K} \colon \bar \cD^{\gamma,\eta} \to \bar\cD^{\gamma+\beta,\eta+\beta}$ 
which also maps $\hat \cD^{\gamma,\eta}$ to $\hat\cD^{\gamma+\beta,\eta+\beta}$ and
such that 
$\CR \mcb{K} f = K * \CR f$ with the reconstruction $\CR$ from Theorem~\ref{thm:reconstructDomain}.
Furthermore, 
for $T\in(0,1)$ and $\kappa\geq 0$
\begin{equs}[eq:short_time]
|\mcb K f |_{\cD^{\gamma+\beta-\kappa,\eta+\beta-\kappa};T}
&\lesssim T^{\kappa/2}|f|_{\cD^{\gamma,\eta};T}\;,
\\
|\mcb K  f; \mcb K \bar f |_{\cD^{\gamma+\beta-\kappa,\eta+\beta-\kappa};T}
&\lesssim T^{\kappa/2}(|f;\bar f|_{\cD^{\gamma,\eta};T} + \$Z;\bar Z\$_{T})\;,
\end{equs}
where the first proportionality constant depends only on $\$Z\$_{T}$
and the second depends on $\$Z\$_{T}+\$\bar Z\$_{T}+|f |_{\cD^{\gamma,\eta};T}+| \bar f |_{\cD^{\gamma,\eta};T}$. 
\end{theorem}

Write $\mcb{K}$ and $\bar{\mcb{K}}$
for integration operators on modelled distributions
corresponding to 
$\mfz$ and $\mfm$.
Recall from \eqref{e:K-Keps-assign}
that they represent the kernels $K$ and $K^\eps = K \ast \moll^{\eps}$.
We assume that $\moll$ is non-anticipative and therefore so is $K \ast \moll^{\eps}$.
Recall the fixed parameter $\varsigma \in  (0,\kappa]$ and the norms $|\cdot|_{\hat\cD^{\gamma,\eta,\eps};T}$ in Section~\ref{subsubsec:prob}.
The following results are from  \cite[Sec.~7]{CCHS2d}.

\begin{lemma}\label{lem:K-barK-hat}
Fix $\gamma  > 0$ and $\eta < \gamma$  such that $\gamma + 2 - \kappa \not \in \N$, $\eta + 2 - \kappa \not \in \N$, and $\eta>-2$.
Suppose that $\moll^{\eps}$ is non-anticipative.
Then, for fixed $M > 0$, one has for all $T\in(0,1)$
\[
|\mcb{K} f - \bar{\mcb{K}} f|_{\hat\cD^{\gamma + 2 - \kappa, \eta  + 2 - \kappa,\eps};T} \lesssim \eps^{\varsigma} |f|_{\hat\cD^{\gamma,\eta,\eps};T}
\]
uniformly in $\eps \in (0,1]$, $Z \in \mathscr{M}_{\eps}$ with $\$Z\$_{\eps;T} \le M$, and $f \in \hat\cD^{\gamma,\eta} \ltimes Z$.
\end{lemma}

\begin{lemma}\label{lem:Schauder-input-KK}
Under the same assumptions as Lemma~\ref{lem:Schauder-input} with $\beta=2-\kappa$ and for fixed $M>0$
one has for all $T\in(0,1)$
\begin{equ}
| \mcb{K}^\omega \bone_+ f -\bar{\mcb{K}}^\omega \bone_+ f |_{\cD^{\bar\gamma-\kappa,\bar\eta,\eps};T}
\lesssim \eps^\varsigma
| f |_{\cD^{\gamma,\eta,\eps};T}
\end{equ} 
uniformly in $\eps \in (0,1]$, $Z \in \mathscr{M}_{\eps}$ with $\$Z\$_{\eps;T} \le M$, and $f \in \cD^{\gamma,\eta} \ltimes Z$.
\end{lemma}

\section{YMH flow without DeTurck term}
\label{app:YMH flow without DeTurck term}

In this appendix we collect some useful results on the YMH flow (without DeTurck and $\Phi^3$ term) 
which reads
\begin{equation}\label{eq:YMH_no_DeTurck}
\begin{split}
\partial_t A =& -\mrd_A^* F_A - \mathbf{B}(\mrd_{A}\Phi\otimes\Phi)\;,\\
\partial_t\Phi  =& - \mrd_{ A}^* \mrd_A \Phi\;,
\end{split}
\end{equation}
or, in coordinates,
\begin{equs}
\partial_t A_i
&= \Delta A_i  - \partial_{ji}^2 A_j + [A_j,2\partial_j A_i - \partial_i A_j + [A_j,A_i]]\\
{}& \qquad 
+ [\partial_j A_j, A_i] -\mathbf{B} ( (\partial_{i} \Phi + A_{i}\Phi)\otimes \Phi)\;,\\
\partial_t\Phi  &= 
\Delta \Phi + (\partial_{j}A_{j})\Phi + 2 A_{j} \partial_{j}\Phi + A_{j}^{2}\Phi \;.
\end{equs}

In this section we set $E=\mfg^d\oplus \higgsvec$ where $d \in \{2,3\}$. 

\begin{lemma}[Theorems~2.6 and~3.7 of~\cite{HongTian2004}]
\label{lem:global_sol_no_deturck} 
For $d=2,3$ and any $(A_0,\Phi_0)\in\CC^\infty(\T^d,E)$, there exists a unique solution to~\eqref{eq:YMH_no_DeTurck} in $\CC^\infty(\R_+\times \T^d, E)$
with initial condition $(A_0,\Phi_0)$.
\end{lemma}

\begin{remark}
It is assumed in~\cite{HongTian2004} that the Higgs bundle $\CV$ is the adjoint bundle, i.e.\ the setting of Remark~\ref{rem:V=g}, and that $d=3$.
The same proof, however, works in the case of a general Higgs bundle;
the case $d=2$ follows by considering fields $X=(A,\Phi)$ with $A_3=0$ and $\partial_3 X=0$.
\end{remark}

\begin{definition}\label{def:flow_no_deturck}
Let $\flow\colon\CC^\infty(\T^d,E)\to \CC^\infty(\R_+\times \T^d, E)$ be the map taking $X = (A_0,\Phi_0)$ to the smooth solution of~\eqref{eq:YMH_no_DeTurck}
with initial condition $X$.
\end{definition}

A classical way to obtain solutions to~\eqref{eq:YMH_no_DeTurck}
is given by the following lemma, the proof of which is standard.

\begin{lemma}\label{lem:deturck_to_no_deturck}
Let $X\in \CC^\infty(\T^d,E)$ and $\ymhflow(X)=(a,\phi)\in\CC^\infty([0,T_X)\times\T^d,E)$ as defined in Definition~\ref{def:ymhflow} (the solution of the YMH flow with DeTurck term and initial condition $X$).
Let $g\colon [0,T_X)\to \CC^\infty(\T^d,G)$ solve
\begin{equ}\label{eq:g_deturck}
g^{-1} \partial_t g = -\mrd^* a\;,\quad g = \id\;.
\end{equ}
Then $(A,\Phi)\eqdef \ymhflow(X)^g \colon [0,T_X) \to \CC(\T^d,E)$ solves~\eqref{eq:YMH_no_DeTurck}
with initial condition $X$.
\end{lemma}

\begin{corollary}\label{cor:flow_continuous}
For every $t\geq0$, $\CC^\infty\ni X\mapsto \flow_t(X)\in\CC^\infty$ is continuous.
\end{corollary}

\begin{proof}
For $X\in\CC^\infty$,
by Proposition~\ref{prop:YM_flow_minus_heat}
there exists $t>0$ sufficiently small (depending, say, only on $|X|_{L^\infty}$) such that $\CC^\infty\ni Y\mapsto (\ymhflow_t(Y),g_t) \in \CC^\infty$
is continuous at $X$, where $g$ solves~\eqref{eq:g_deturck}
with $\ymhflow(Y)=(a,\phi)$.
Hence $Y\mapsto \flow_t(Y)$ is continuous at $X$.
Continuity for arbitrary $t\geq0$ follows
from non-explosion of $|X|_{L^\infty}$ (Lemma~\ref{lem:global_sol_no_deturck}) and patching together intervals.
\end{proof}

For $\rho\in(\frac12,1]$, recall the action of $\mfG^{0,\rho}$ on $\Omega^\rho \supset\CC^\infty(\T^d,E)$
from
Section~\ref{subsec:backwards_unique}.

\begin{lemma}\label{lem:gauge_for_flow}
Suppose $X,Y\in\CC^\infty(\T^d,E)$
and $X^g=Y$ for some $g\in\mfG^\rho$ and $\rho\in(\frac12,1]$.
Then $g\in \CC^\infty(\T^d,G)$
and $\flow_t(X)^g = \flow_t(Y)$ for all $t \geq 0$.
\end{lemma}

\begin{proof}
$\CC^\alpha$ control on $g$ gives $\CC^{\alpha}$ control on $\mrd g$ by writing $\mrd g = (\Ad_g A - A^g)g$ where $X=(A,\Phi)$.
Hence $g$ is smooth.
The fact that $\flow(X)^g = \flow(Y)$
is clear.
\end{proof}
\section{Evolving rough gauge transformations}
\label{app:evolving_rough_g}

In this appendix, we extend the analysis performed in Section~\ref{sec:renorm-A}
for the system~\eqref{e:SPDE-for-X} to the coupled system
\begin{equs}
\partial_t X &= \Delta X + X \partial X + X^3 
+( (C_{\A}^{\eps})^{\oplus 3} \oplus C_{\Phi}^{\eps} ) X + (c^\eps g^{-1}\mrd g \oplus 0) +
 \xi^\eps\;,
\\
X_0 & = (a,\phi) \in \state \;, \label{e:SPDE-for-X_with_g} 
\end{equs} 
and
\begin{equ}\label{eq:SPDE_for_g_coupled}
g^{-1}(\partial_t g)
= \partial_j(g^{-1}\partial_j g)+ [A_j,g^{-1}\partial_j g]\;, \quad g_0 \in \mfG^{0,\nu}\;,
\end{equ}
where $X=(A,\Phi)$, $\nu>0$ and $c^\eps\in L(\mfg,\mfg)$.
Note that~\eqref{eq:SPDE_for_g_coupled} is identical to~\eqref{eq:SPDE_for_g_wrt_A}.
The following result provides well-posedness and stability of $(X,g)$,
which is used in Sections~\ref{sec:gauge} and~\ref{sec:Markov}.

\begin{lemma}\label{lem:flow_for_rough_g}
Let $(X^\eps,g^\eps)$ solve~\eqref{e:SPDE-for-X_with_g}-\eqref{eq:SPDE_for_g_coupled} with $\nu\in(0,\frac32)$, $c^\eps\in L(\mfg,\mfg)$ such that $\lim_{\eps\downarrow0}c^\eps$ exists,
and $C_{\A}^{\eps}, C_{\Phi}^{\eps}$ as in Theorem~\ref{thm:local_exist}.
Then $(X^\eps,g^{\eps})$ converges to a process $(X,g)$ in probability in $(\state\times\mfG^{0,\nu})^\sol$.
Furthermore, the solution map
\begin{equ}[eq:sol_map_X_g]
\state \times \mfG^{0,\nu} \ni ( X_0, g_0 ) \mapsto (X,g)\in L^0(\Omega^\noise;(\state\times\mfG^{0,\nu})^\sol)
\end{equ}
is continuous, where $L^0(\Omega^\noise;(\state\times\mfG^{0,\nu})^\sol)$ is the space of $(\state\times\mfG^{0,\nu})^\sol$-valued random variables defined on the underlying probability space $\Omega^\noise$ equipped with the topology of convergence in probability.
\end{lemma}

\begin{proof}
To circumvent the issue that the target space $G$ is nonlinear, we can assume without loss of generality
that for some $n \in \N$, we embed $G \subset O(n)$,  $\mfg\subset \mfo(n)$ (the Lie algebra of $O(n)$), and $\higgsvec\subset \R^n$.
Then $A_j$ and $g$ just take values in $\R^{n \times n}$ which is the linear space of $n$ by $n$ matrices. 
Therefore we can exchange the term $c^\eps g^{-1}\mrd g \oplus 0$ in~\eqref{e:SPDE-for-X_with_g} for $c^\eps g^*\mrd g \oplus 0$ and equation~\eqref{eq:SPDE_for_g_coupled} for the equation
\begin{equ}[e:g-equ-A]
(\partial_t - \Delta)g
= - (\partial_j g)  g^{\ast} \partial_j g 
+ g [A_j,g^{\ast}\partial_j g]\;,
\end{equ}
where $g$ takes values in $\R^{n \times n}$ and we are using the fact that, for $M \in O(n)$, $M^{\ast} = M^{-1}$. 
Here the bracket is just the matrix commutator.

Since the solution theory for $(X,g)$ is similar to that of~\eqref{e:SPDE-for-X}, we only sketch the main differences.
We enlarge the regularity structure introduced in Section~\ref{sec:renorm-A} to treat~\eqref{e:SPDE-for-X_with_g} and~\eqref{e:g-equ-A} together, and
verify that the corresponding BPHZ character does not introduce any renormalisation into \eqref{e:g-equ-A} nor any additional renormalisation into~\eqref{e:SPDE-for-X_with_g}.
This means that the dynamic for $g$ remains compatible with \eqref{eq:SPDE_for_g_coupled}, in particular it preserves $G$ as the target space if the initial data is $G$-valued, and 
the dynamic for $A_j$  also preserves the Lie subalgebra $\mfg$ as the target space if the initial data is $\mfg$-valued.
We then verify that the enlarged abstract fixed point problem is also well-posed. 
Arguing the convergence of the BPHZ models can be done in the same way as before.

To enlarge the regularity structure we introduce a new type $\mfk$, namely we set $\mfL_+ = \{\mfz,\mfk\}$. 
The target and kernel space assignments are enlarged by setting $W_\mfk =  \R^{n\times n}$ and $\mathcal{K}_{\mfk} = \R$. 
We also set
$\deg( \mfk )
=  2 $
and $\reg(\mfk)=  3/2 - 5\kappa $.
We also have an enlarged nonlinearity by reading off of the right-hand side of \eqref{e:g-equ-A}
and adding a term corresponding to $(c^\eps g^{*}\mrd g ,0)$ to the nonlinearity~\eqref{e:YMH-nonlinearity} for $X$.

We use the shorthand $\pr{g}$ for the component $A_{(\mfk,0)}$ of $\pr{\mathbf{A}} = (A_{o})_{o \in \CE} \in \mcb{A}$. Writing $\CH$ for the coherent expansion for the solution to \eqref{e:g-equ-A}, we then have 
\[
q_{<\frac32}\CH =\pr{g} \bone + \pr{\partial_j g} \mathbf{X}_j
 + \pr{g} [\mcb{I}_{\mfk} ( \mcb{I}_{\mfz} \bar{\bXi}), \pr{g}^{\ast} \pr{\partial g}]\;.
\]
Here and below $q_{< L}$ denotes the projection
onto degrees $< L$, and we used the summation convention in the bracket as in Remark~\ref{rem:bracket_overload}.
We also have $q_{<1}\CH^{\ast} = \pr{g}^{\ast}\bone$.     
A simple power counting argument shows that the only trees $\tau$ where one has both $\deg(\tau) < 0$ and $\bar\bUpsilon_{\mfk}[\tau] \not = 0$ are
$\<IXi>$, $\<I[IXiI'Xi]_notriangle>$ and 
$\<IXi> \, \partial_j \mcb{I}_{\mfk}( \<IXi>)$.
The first two trees  are planted and thus do not contribute to renormalisation.
By parity (similarly to Lemma~\ref{lem:symbols_vanish})
the last one  vanishes when hit by the BPHZ character.

To see that no additional renormalisation appears in~\eqref{e:SPDE-for-X_with_g} beyond that already shown for~\eqref{e:SPDE-for-X},
it suffices to observe that tree of lowest degree added to the expansion of $X$ by the presence of $(c^\eps g^{*}\mrd g,0)$ is $\<I'[IXi]_notriangle>$, which has positive degree.

The fixed point problem
 \begin{equ}\label{eq:mod_dis_for_g}
\CH =
\CG_\mfk \mathbf{1}_{+}\Big(
- \partial_j \CH  \CH^{\ast} \partial_j \CH
+ \CH [\CA_j,\CH^{\ast}\partial_j \CH]
 \Big)
  + \CP g_0
 \end{equ}
 can be solved in $\cD^{\gamma,\nu}$
with $\gamma=1+$
 since, for the term $\partial_j \CH  \CH^{\ast} \partial_j \CH$ one has
 \[
 \cD_0^{\gamma,\nu-1} \times \cD_0^{\gamma,\nu}\times \cD_0^{\gamma,\nu-1}\to \cD_0^{\gamma,2\nu-2}
 \]
 and for the term
 $\CH [\CA_j,\CH^{\ast}\partial_j \CH]$  one has
 \[
  \cD_0^{\gamma,\nu} \times  \cD_{-\frac12 -}^{\gamma,-\frac12 -}
 \times \cD_0^{\gamma,\nu}\times \cD_0^{\gamma,\nu-1}
 \to \cD_{-\frac12-}^{\gamma-\frac12-,\nu-\frac32-}\;.
 \]
Here $2\nu-2 > -2$, and
 $\nu-\frac32 > -\frac32$, so \cite[Prop.~6.16]{Hairer14} applies.
The abstract fixed point problem for the $X$-component is solved in essentially the same way as in the proof of Theorem~\ref{thm:local-exist-sigma}.

We note that continuity at time $t=0$ of $\mcR\mcH$ follows from the fact that the initial condition $g_0$ is in $\mfG^{0,\nu}$, on which the heat semigroup is strongly continuous,
and the fact that the first term on the right-hand side of~\eqref{eq:mod_dis_for_g} goes to zero in $\cD^{\gamma+\frac32-, \nu+\frac12-}_0$ for short times by~\cite[Thm.~7.1]{Hairer14}.

The final claim on continuity of the map~\eqref{eq:sol_map_X_g} follows from Lemma~\ref{lem:PX0Psi-prob}, which
shows continuity (in $L^p$ and thus in probability) with respect to $X_0$ of the singular products of $\CP X_0$ and the SHE $\tilde \Psi$ appearing in the fixed point problem,
combined with standard arguments showing (pathwise) continuity of the solution in the initial data.
\end{proof}

We state the following straightforward variant of Lemma~\ref{lem:flow_for_rough_g} which 
is used in the proof of Propositions~\ref{prop:SPDEs_conv_zero_state} and~\ref{prop:consts-soft2}. 

\begin{lemma}\label{lem:flow_for_rough_g2}
The statement of Lemma~\ref{lem:flow_for_rough_g} stills holds with \eqref{e:SPDE-for-X_with_g}-\eqref{eq:SPDE_for_g_coupled} replaced by 
\begin{equs}
\partial_t X &= \Delta X + X \partial X + X^3 
+( (C_{\A}^{\eps})^{\oplus 3} \oplus C_{\Phi}^{\eps} ) X + (c^\eps \mrd g g^{-1} \oplus 0) +
\xi^\eps\;,
\\
X_0 & = (a,\phi) \in \state \;, \label{e:SPDE-for-X_with_g2} 
\end{equs} 
and
\begin{equ}\label{eq:SPDE_for_g_coupled2}
(\partial_t g)  g^{-1} 
= \partial_j((\partial_j g)g^{-1})+ [A_j, (\partial_j g) g^{-1}]\;, \quad g_0 \in \mfG^{0,\nu}\;.
\end{equ}
\end{lemma}

\begin{remark}\label{rem:why_U_h}
It is natural to ask why we use two different approaches for formulating the evolution of our gauge transformation in linear spaces: 
for the proof of Lemma~\ref{lem:flow_for_rough_g} we view $G$ as a group of matrices while 
in Section~\ref{sec:gauge} we adopt the viewpoint of tracking the pair $(U,h)$ as in \eqref{eq:h_and_U_def}. 

In Section~\ref{sec:gauge}, $(U,h)$ are the natural variables for tracking how our gauge transformations act on our gauge field (which makes $h$ appear) and on our $E$-valued white noise (which makes $U$ appear).
This also makes them good variables for keeping track of the renormalisation in \eqref{eq: B system} and \eqref{eq: bar a system}. 

On the other hand, $h$ as introduced in \eqref{eq:h_and_U_def} is not well-defined for the initial data treated in Lemma~\ref{lem:flow_for_rough_g} when $\nu \le \f12$, which is the class of initial data that is used
in Proposition~\ref{prop:gauge_covar_couple}.
\end{remark}
\section{Injectivity of the solution map}\label{app:injectivity}
Our objective in this appendix is to prove that one can recover, from the outputs of the limiting solution maps for \eqref{eq:renorm_bar_a} and \eqref{eq:renorm_b}, the values of the constants $(\mathring{C}_{\A},\mathring{C}_{\h}) \in L(\mfg, \mfg)^{2}$. 
This tells us that these solution maps have injectivity properties as functions of these constants, and this is a key ingredient for the proof of  the convergence statements in Propositions~\ref{prop:consts-soft} and \ref{prop:consts-soft2}.
Our arguments here are in the same in spirit as \cite[Thm~3.5]{StringManifold}, but adapted to the present setting where our solution contains more singular components. 

We introduce the same notations for the set of types $\Lab = \Lab_{+} \sqcup \Lab_{-}$, rule $R$, kernel and target space assignments, corresponding regularity structure $\mcb{T}$, and kernel assignments $K^{(\eps)}$ and noise assignments $\zeta^{\delta,\eps}$ that we defined in 
Section~\ref{subsec:reg_struct_guage_transform} (see in particular page~\pageref{eq:noise_assign_gauge}).
Recall (see \cite{BHZ19})  that models can be described by non-recentred evaluation maps $\PPi$, and for the rest of this appendix we use this notation when referring to models. 
We then write $\PPi_{\delta,\eps}$ for the canonical lift of  the kernel assignment $K^{(\eps)}$ from~\eqref{e:K-Keps-assign} and noise assignment $\zeta^{\delta,\eps}$ from~\eqref{eq:noise_assign_gauge},  
$\PPi^{\BPHZ}_{\delta,\eps}$ for the corresponding BPHZ lift, and 
$\ell^{\delta,\eps}_{\BPHZ}$ for the BPHZ character that shifts $\PPi_{\delta,\eps}$ to $\PPi^{\BPHZ}_{\delta,\eps}$.
By BPHZ convergence (see Lemma~\ref{lem:conv_of_models2}) for the family of models $\PPi^{\BPHZ}_{\delta,\eps}$,  we obtain a limiting random model $\PPi^{\BPHZ}_{0,\eps} = \lim_{\delta \downarrow 0} \PPi^{\BPHZ}_{\delta,\eps}$, as
well as a limit $\PPi^{\BPHZ} = \lim_{\eps \downarrow 0} \PPi^{\BPHZ}_{0,\eps}$.

Fix initial conditions $\bar x = (\bar a,\bar \phi)\in\state$ and $g_0\in \mfG^{0,\rho}$
as well as $\mathring{C}^{}_{\A}, \mathring{C}^{}_{\h} \in L(\mfg,\mfg)$ and $\sig\in\R$.
We write ${\Omega}^{\noise}$
for the canonical probability space for the underlying white noise on which 
all the random models introduced above are defined.
By Proposition~\ref{prop:bar_X_max_sols}, we have a mapping
\begin{equs}[equ:constants_to_solution_map]
 L(\mfg,\mfg)^{2}\times \state\times\mfG^{0,\rho} \ni (\mathring{C}_{\A},\mathring{C}_{\h},\bar x, g_0) &\mapsto \mathcal{A}_{\sig}^{\BPHZ}[\mathring{C}_{\A},\mathring{C}_{\h},\bar x, g_0] \in L^{0}_{\sol} \;,
\end{equs}
where $L^{0}_{\sol}$ is the space of equivalence classes (modulo null sets) of measurable maps from ${\Omega}^{\noise}$ to $(\state \times \tilde{\mfG}^{0,\rho})^{\sol}$ and $\mathcal{A}_{\sig}^{\BPHZ}[\mathring{C}_{\A},\mathring{C}_{\h},\bar x, g_0] $ is the corresponding maximal solution map for both\footnote{See Corollary~\ref{cor:Y_system_conv}.} \eqref{eq:renorm_bar_a} and \eqref{eq:renorm_b} and the $(U,h)$ \slash $(\bar{U},\bar{h})$ obtained from their corresponding gauge transformations $g$ \slash $\bar{g}$ using \eqref{eq:h_and_U_def}.

Our main result for this section is the following lemma. 
\begin{lemma}\label{lemma:injectivity_in_law}
There exists a sequence of measurable maps $(O_{n})_{n \in \N}$ from $(\state \times \tilde{\mfG}^{0,\rho})^{\sol}$
to $\mcb{D}'(\T^{3},E)$ such that, for any $\sig \in \R$, $\mathring{C}_{\A},\mathring{C}_{\h} \in L(\mfg,\mfg)$,
and $(\bar x,g_{0}) \in \state \times \mfG^{0,\rho}$,
one has 
\[
\lim_{n \rightarrow \infty} \E[ O_{n} \big(\mathcal{A}_{\sig}^{\BPHZ}[\mathring{C}_{\A},\mathring{C}_{\h},\bar x,g_0]\big)]
=
\mathring{C}_{\mfz} \bar{x} +  \mathring{C}_{\mfh} (\mrd g_0) g_0^{-1}\;. 
\] 
Above  $\mathring{C}_{\mfz}$, $\mathring{C}_{\mfh}$ are defined from $\mathring{C}_{\A},\mathring{C}_{\h}$ as in \eqref{e:ring-C-zh}. 
\end{lemma}
The abstract fixed point problem used in the proof of Proposition~\ref{prop:bar_X_max_sols} will not be suitable for our analysis in this section and it will be preferable to use the abstract fixed point problem used to prove the $\eps \downarrow 0$ convergence of local solutions in Propositions~\ref{prop:SPDEs_conv_zero_space-time}+\ref{prop:SPDEs_conv_zero_state}. 

For that reason, we give another construction of maximal solutions obtained by patching together the limits of local solutions. Fix again initial conditions $\bar x = (\bar a,\bar \phi)\in\state$ and $g_0\in \mfG^{0,\rho}$
as well as $\mathring{C}^{}_{\A}, \mathring{C}^{}_{\h} \in L(\mfg,\mfg)$ and $\sig\in\R$.
Let $\mbF$ denote the canonical filtration generated by the noise $\xi$.

Let $\tau_1\in(0,1)$ be the $\mbF$-stopping time from Proposition~\ref{prop:SPDEs_conv_zero_state} corresponding to initial condition $( \bar{x},g_{0})$,
constants $\mathring C^\eps_\A\equiv \mathring C_\A$, $\mathring C^\eps_\h\equiv \mathring C_\h$, and $\sig^\eps=\sig$.
We first define $(\bar X,\bar U,\bar h)$ on $[0,\tau_1]$
as the $\eps\downarrow0$ limit of the solution to~\eqref{eq:renorm_bar_a}
as in Propositions~~\ref{prop:SPDEs_conv_zero_space-time}\slash\ref{prop:SPDEs_conv_zero_state}. 

We now let $\tau_2-\tau_1\in(0,1)$ be the stopping time from Proposition~\ref{prop:SPDEs_conv_zero_state} 
corresponding to the new `initial time' $\tau_1$ and initial condition $(\bar X,\bar U, \bar h)(\tau_1)$.\footnote{Propositions~\ref{prop:SPDEs_conv_zero_space-time} and~\ref{prop:SPDEs_conv_zero_state} were stated in terms of initial conditions in $\state\times\mfG^{0,\rho}$,
but of course can be equivalently stated in terms of initial conditions in $\state\times\tilde{\mfG}^{0,\rho}$.}
We then extend $(\bar X,\bar U,\bar h)$ to $[\tau_1,\tau_2]$
by defining it as the $\eps\downarrow0$ limit of the corresponding solution to~\eqref{eq:renorm_bar_a}.

We proceed in this way defining $(\bar X,\bar U,\bar h)$ as an element of $\CC([0,\tau^*) , \state\times\tilde{\mfG}^{0,\rho})$
where $\tau^* = \lim_{n\to\infty} \tau_n$.
We then extend $(\bar X,\bar U,\bar h)$ to a function $[0,\infty)\to (\state\times\tilde{\mfG}^{0,\rho}) \sqcup \{\skull\}$ by setting 
$(\bar X,\bar U,\bar h)(t) = \skull$ for all $t\geq \tau^*$.
We then have the following lemma.
\begin{lemma}\label{lem:patching}
With the above construction, one has $(\bar X,\bar U,\bar h) \in (\state\times\tilde{\mfG}^{0,\rho})^\sol$
and $(\bar{X},\bar{U},\bar{h}) = \mathcal{A}_{\sig}^{\BPHZ}[\mathring{C}_{\A},\mathring{C}_{\h},\bar x, g_0]$. 
\end{lemma}
\begin{proof}
For the first statement, observe that $\tau_{n+1}-\tau_n$ is independent of $\tau_1,\tau_2,\ldots, \tau_n$ when conditioned on $(\bar X,\bar U,\bar h)(\tau_n)$
and is stochastically bounded from below by a distribution depending only on
$\Sigma(\bar X(\tau_n)) +|(\bar U,\bar h)(\tau_n)|_{\tilde{\mfG}^{0,\rho}}$.
Furthermore, for any $\mbF$-stopping time $\sigma<\tau^*$, there exists another $\mbF$-stopping time $\bar \sigma>\sigma$ such that $\bar\sigma-\sigma$ is stochastically bounded from below by a distribution depending only on $\Sigma(\bar X(\sigma)) +|(\bar U,\bar h)(\sigma)|_{\tilde{\mfG}^{0,\rho}}$
and such that
\begin{equ}
\sup_{t\in [\sigma,\bar\sigma]} \{\Sigma(\bar X(t)) +|(\bar U,\bar h)(t)|_{\tilde{\mfG}^{0,\rho}}\}
\leq 1+2(\Sigma(\bar X(\sigma)) +|(\bar U,\bar h)(\sigma)|_{\tilde{\mfG}^{0,\rho}})\;.
\end{equ}
It follows that, on the event $\tau^*<\infty$, almost surely $\lim_{t\nearrow\tau^*}(\bar X,\bar U,\bar h)(t) = \skull$
and therefore $(\bar X,\bar U,\bar h) \in (\state\times\tilde{\mfG}^{0,\rho})^\sol$.

The second statement follows from the last paragraph of the proof of Proposition~\ref{prop:bar_X_max_sols} on pg. \pageref{page:reconstructions_match}. 
\end{proof}

We now state a version of \cite[Lemma~3.6]{StringManifold}. 
\begin{lemma}\label{lemma:noise_term} 
Fix $ (\mathring{C}_{\A},\mathring{C}_{\h}) \in L(\mfg,\mfg)^{2}$ and $(\bar{x} , g_0) \in \state \times \mfG^{0,\rho}$.

Let $(\CY,\CU,\CH)$ be the local solution, associated to the model $\PPi^{\BPHZ}$ along with the rest of the limiting $\eps$-inputs, of the abstract fixed point problem 
  \eqref{e:fix-pt-HU}+\eqref{e:fix-pt-multi-Xbar} with constants $\mathring{C}_{\A},\mathring{C}_{\h}$ 
and initial condition $(\bar{x},U_0,h_0)$ (the latter two are determined from $g_0$ via \eqref{eq:h_and_U_def}) promised by Lemma~\ref{lem:fixedptpblmclose}, and let $\tau$ be the associated $\mbF$-stopping time given in the statement of that lemma. 
Let  $(\bar{X}, \bar{U} ,\bar{h}) = \CR (\CY,\CU,\CH)$.   

Then, the (random) distribution
$\eta \eqdef  \tilde{\CR} \big( \CU \bXi \big) \in \mcb{D}'((0,\tau) \times \T^3,E)$, where $\tilde{\CR}$ is the local (defined away from time $t=0$) reconstruction operator for the model $\PPi^{\BPHZ}$, admits an extension to a random element of $\mcb{D}'((0,1) \times \T^3,E)$ such that for every (deterministic) $E$-valued space-time test function $\psi$,
\begin{equ}\label{eq:ItoIntegral}
\eta(\psi)= \sig \int_0^{\tau}   \scal{ \bar{U} \psi, \mrd W(s)}\;,
\end{equ}
where the integral is in the It\^o sense and
$\sig W$ is the $L^2(\T^{3},E)$-cylindrical Wiener process associated to the noise $\sig \xi = \CR \bXi$.
\end{lemma}

\begin{proof}
First note that the quantity \eqref{eq:ItoIntegral} is well-defined since
$s \mapsto U(s,\cdot)$ is continuous and adapted.
To obtain \eqref{eq:ItoIntegral}, we approximate $\PPi^{\BPHZ}$ with the models $\PPi_{0,\eps}^{\BPHZ}$ (using a non-anticipative mollifier $\chi$) and let $\eta^{\eps}$ be the analog of $\eta$ but with the role of $\PPi^{\BPHZ}$ replaced with $\PPi_{0,\eps}^{\BPHZ}$ along with the corresponding regularised $\eps$-inputs. 

Now, on the one hand, we have $\lim_{\eps \downarrow 0}  \eta^{\eps}(\psi) = \eta(\psi)$. 
On the other hand, thanks to \cite{BCCH21}, it follows that if we write $(\bar{X}^{\eps},\bar{U}^{\eps},\bar{h}^{\eps})$  for the relevant objects defined from $\PPi_{0,\eps}^{\BPHZ}$ along with the associated $\eps$-inputs, one has the identity
\begin{equs}
\eta^{\eps}(\psi) 
&=  \sig \int_0^{\tau}   \scal{\bar{U}^{\eps} \psi,\mrd W(s)} \\
{}& \quad 
+ 
\sum_{\tree \in \mfT_{-}} 
\int_0^{\tau}  \scal{\psi, (\ell^{0,\eps}_{\BPHZ} \otimes \id) \bbUpsilon^{\bar{F}}_{\mfm}[\tree](\bar{X}^{\eps},\bar{U}^{\eps},\bar{h}^{\eps})} \,\mrd s\;,
\end{equs}
where $\ell^{0,\eps}_{\BPHZ}$ is defined as in \eqref{eq:bphz_char_delta_0} \dash note that $\bbUpsilon^{\bar{F}}_{\mfm}[\tree]$ can be written as a function of the pointwise values of the smooth (until blow-up) functions $(\bar{X}^{\eps},\bar{U}^{\eps},\bar{h}^{\eps})$. 

Observe that $\tree \in \mfT_{-}$ and $\bbUpsilon_{\mfm}^{\bar{F}}[\tree] \not = 0$ forces $\tree$ to have an occurrence of $\mfl$ at the root, so by Lemma~\ref{lemma:noise_at_root_vanish} below  the second term above vanishes. 
The result then follows by the stability of the It\^o integral.  
\end{proof}
The above proof used the following lemma regarding properties of the renormalisation character $\ell^{0,\eps}_{\BPHZ}$ in \eqref{eq:bphz_char_delta_0}. 

\begin{lemma}\label{lemma:noise_at_root_vanish}
Suppose that $\moll$ is non-anticipative
and that $ \ell^{0,\eps}_{\BPHZ}$ is given by \eqref{eq:bphz_char_delta_0}.
Then, for any $\tree \in  \mfT^{\bar{F}} $ that has an occurrence of $\mfl$ at the root, one has $\ell^{0,\eps}_{\BPHZ}[\tree] = 0$.  
\end{lemma}
\begin{proof}
The action of the negative twisted antipode on $\tree$ generates a linear combination of forests, and each such forest must contain at least one tree  of the form $\mcb{I}_{\mfl}(\bone) \tilde{\tree}$ 
where every noise in $\tilde{\tree}$ is incident to an $\eps$-regularised kernel. 
For $\delta > 0$ sufficiently small, $\bar{\PPi}_{\delta,\eps}[\mcb{I}_{\mfl}(\bone) \tilde{\tree}] =  \bar{\PPi}_{\delta,\eps}[\mcb{I}_{\mfl}(\bone)] \cdot \bar{\PPi}_{\delta,\eps}[ \tilde{\tree}]$ by independence, and since $\bar{\PPi}_{\delta,\eps}[\mcb{I}_{\mfl}(\bone)]  = 0$ the proof is completed.
\end{proof} 
%
The key step in constructing the $O_{n}$'s promised in  Lemma~\ref{lemma:injectivity_in_law} is 
constructing, as function of the limiting $\eps \downarrow 0$ solution $(X,U,h)$ to \eqref{eq:renorm_bar_a}, the nonlinear expression $\CN(X,U,h) =  \text{\ ``}X^3 + X \partial X  \text{\ ''}$ \dash which are the limits of the corresponding terms on the right-hand side of \eqref{eq:renorm_bar_a}).
We will have to define $\CN(X,U,h)$ in a renormalised sense, that is by regularising $(X,U,h)$ at a scale $\hat{\eps} > 0$ and defining a regularised and renormalised $\CN_{\hat{\eps}}(\bullet, \bullet,\bullet)$. 

To do this we introduce modelled distributions that describe 
the regularisation of $(X,U,h)$\footnote{Which we note are different from the solutions to the regularised equation.}, and show that the reconstruction of $\CN$ applied to the regularised modelled distribution for $X$ converges to the reconstruction of $\CN$ applied to the modelled distribution for the original, unregularised $X$. 
At the algebraic level, this regularisation of $(X,U,h)$ is encoded by introducing new labels corresponding to regularised kernels. Intermediate quantities appearing when comparing the two reconstructions described above will involve \emph{both} regularised kernels and unregularised kernels which makes adding labels necessary (as opposed to just modifying our assignment of kernels to labels). 

Finally, to complete the argument one needs to know that reconstruction of our nonlinearity evaluated at these regularised modelled distributions makes sense as an observable on $(\state \times \tilde{\mfG}^{0,\rho})^{\sol}$ \dash but this last point can be justified with arguments from \cite{BCCH21}. 

Let $\MM$ be the set of models on $\mcb{T}$ that are admissible with the kernel assignment $K^{(0)}$ \dash note that $\PPi^{\BPHZ}  \in \MM$. 
We then define a larger set of labels by adding a duplicate type $\hat{\mft}$
 for each  $\mft \in \Lab_{+}$. 
 Our new set of kernel types is then $\hat{\mfL}_{+} = \mfL_{+} \sqcup \{ \hat{\mft}: \mft \in \Lab_{+}\}$, and our new set of types is $\hat{\mfL} = \hat{\mfL}_{+} \sqcup \mfL_{-}$. 
We also set, for each $\mft \in \Lab_{+}$, $\deg(\hat{\mft}) = \deg(\mft)$, $W_{\hat{\mft}} \simeq W_{\mft}$, and $\mathcal{K}_{\hat{\mft}} \simeq \mathcal{K}_{\mft}$.

We extend the rule $R$ to obtain a rule $\hat{R}$ on $\hat{\mfL}$ as follows: for any $\mft \in \Lab$ we set $\hat{R}(\mft) = R(\mft)$ and, for any $\mft \in \Lab_{+}$, we set $\hat{R}(\hat{\mft})$ to be the collection of all the node types $\hat{N}$ that can be obtained by taking $N \in R(\mft)$ and replacing any number of instances of edge types $(\mft,p) \in \Lab_{+} \times \N^{d+1}$ with $(\hat{\mft},p)$. 
 
We define the set of trees $\hat{\mfT} \supset \mfT$ to be those that strongly conform to $\hat{R}$ but where we also enforce that any edge of type $(\hat{\mft},p)$ has to be incident to the root. 
Since the latter constraint is preserved by coproducts, $\hat{\mfT}$ determines a regularity structure $\hat{\mcb{T}} \supset \mcb{T}$. 

Finally we introduce a parameter $\hat{\eps} \in [0,1]$ and define the kernel assignment $K^{[\hat{\eps}]} = (K^{[\hat{\eps}]}_{\mfb}: \mfb \in \hat{\mfL}_{+})$ by setting, for each $\mft \in \Lab_{+}$,  $K^{[\hat{\eps}]}_{\mft} = K^{(0)}_{\mft}$ and 
 $K^{[\hat{\eps}]}_{\hat{\mft}} = K^{(0)}_{\mft} \ast \moll_{\hat{\eps}}$. 
We write $\hat{\MM}_{\hat{\eps}}$ for the collection of $K^{[\hat{\eps}]}$-admissible models on $\hat{\mcb{T}}$ and then set $\hat{\MM} = \bigcup_{\hat \eps} \hat{\MM}_{\hat{\eps}}$. 
We denote by $\pi_{0}$ the natural restriction map $\pi_{0}: \hat{\MM} \rightarrow \MM$.

We define a mapping $\hat{\mfT} \ni \hat{\tree} \mapsto \hat{\pi}_{0}(\hat{\tree}) \in \mfT$ by replacing, for every $\mft \in \Lab_{+}$ and $p \in \N^{d+1}$, all edges of type $(\hat{\mft},p)$ in $\hat{\tree}$ with edges of type $(\mft,p)$. 
Overloading notation, this mapping on trees also induces a linear map\footnote{Note that $\hat{\pi}_{0}$ may not be an isomorphism because one may have $\dim(\mcb{T}[\hat{\tree}]) > \dim( \mcb{T}[\hat{\pi}_{0} (\hat{\tree})])$, there will be some symmetrisation in the action of $\hat{\pi}_{0}[\hat{\tree}]$ when $\hat{\pi}_{0}(\hat{\tree})$ has more symmetries than $\hat{\tree}$.} $\hat{\pi}_{0}[\hat{\tree}]\colon \mcb{T}[\hat{\tree}] \rightarrow \mcb{T}[\hat{\pi}_{0}( \hat{\tree})]$. 
We again denote by $\hat{\pi}_{0}$ the corresponding linear map $\hat{\pi}_{0}\colon \hat{\mcb{T}} \rightarrow \mcb{T}$. 
Right composition by $\hat{\pi}_{0}$ gives us another map $\iota: \MM \rightarrow \hat{\MM}_{0}$ and we note that $\pi_{0} \circ \iota = \id$. 

For every $\hat{\eps} > 0$ we define a map $\CE_{\hat{\eps}}\colon \MM \rightarrow \hat{\MM}_{\hat{\eps}}$ as follows. Given $\PPi \in \MM$ we fix $\CE_{\hat{\eps}}(\PPi)$ to be the unique model in $\hat{\MM}_{\hat{\eps}}$ that satisfies
\begin{itemize}
\item $\CE_{\hat{\eps}}(\PPi) [ \tree ] = \PPi[\tree]$ for $\tree \in \mfT$
\item  For $\tree = \bar{\tree} \prod_{j=1}^{n} \mcb{I}_{(\hat{\mft}_{j},p_{j})}(\tree_{j}) \in \mfT$ with $\bar{\tree},\tree_{j} \in \mfT$,  one has
\begin{equ}\label{eq:formula_for_extended_model}
\CE_{\hat{\eps}}(\PPi) [ \tree ] = \PPi[\bar{\tree}]  \Big(  \bigotimes_{j=1}^{n}\partial^{p_j}K_{\mft_j} \ast \moll^{\hat{\eps}}  \ast \PPi [\tree_{j}] \Big)\;. 
\end{equ}
\end{itemize}
Note that $\pi_{0} \circ \CE_{\hat{\eps}} = \id$. 
Below, we write $\hat{\mathfrak{R}}$ for the renormalisation group for $\hat{\mcb{T}}$.
We can now state and prove our analog of \cite[Lem~3.7]{StringManifold}.  
\begin{lemma}\label{lem:shifted_model}
There exists a sequence $( \hat{M}_{\hat{\eps}}: \hat{\eps} \in (0,1])$ of elements of $\hat{\mathfrak{R}}$ that all leave $\mcb{T}$ invariant and such that
the models
\begin{equ}\label{eq:shifted_model}
\hat{\PPi}_{\hat{\eps}} \eqdef \CE_{\hat{\eps}}(\PPi^{\BPHZ})  \circ \hat{M}_{\hat{\eps}}
\end{equ}
 converge in probability to $\iota (\PPi^{\BPHZ})$ as $\hat{\eps} \downarrow 0$. 
 Moreover, one has $\pi_0(\hat{\PPi}_{\hat{\eps}}) = \PPi^{\BPHZ}$. 
\end{lemma}
\begin{proof}
First note that the statement $\pi_0(\hat{\PPi}_{\hat{\eps}}) = \PPi^{\BPHZ}$ will follow from showing that  $M_{\hat{\eps}}$ leaves $\mcb{T}$ invariant thanks to the identity \eqref{eq:shifted_model}.

The general idea is to define $\hat{\PPi}_{\hat{\eps}}$ to be the  BPHZ renormalisation of the random model $\CE_{\hat{\eps}}(\PPi^{\BPHZ})$.
However, we can not directly BPHZ renormalise the random model $\CE_{\hat{\eps}}(\PPi^{\BPHZ})$ since $\CE_{\hat{\eps}}(\PPi^{\BPHZ})$ is not a smooth model (simply because $\PPi^{\BPHZ}$ is not a smooth model). 
We can suppose, by regularising the noise, that we have a collection of smooth random BPHZ models $\PPi_{\eps}^{\BPHZ}$ that converge to $\PPi^{\BPHZ}$  as $\eps \downarrow 0$ in every $L^{p}$. 
 
We then define $\hat{\PPi}_{\hat{\eps},\eps}$ to be the BPHZ renormalisation of $\CE_{\hat{\eps}}(\PPi^{\BPHZ}_\eps)$, and let $M_{\hat{\eps},\eps}$ be the corresponding element in renormalisation group that satisfies 
\[
\hat{\PPi}_{\hat{\eps},\eps} = \CE_{\hat{\eps}}(\PPi^{\BPHZ}_\eps) \circ \hat{M}_{\hat{\eps},\eps}\;.
\]
By stability of the BPHZ lift, $\hat{\PPi}_{\hat{\eps},\eps}$ converges to $\iota( \PPi^{\BPHZ})$
whichever way one takes $\eps,\hat{\eps} \downarrow 0$. 

Next, we argue that that $\hat{M}_{\hat{\eps},\eps}$ leaves $\mcb{T}$ invariant. 
Writing $\ell_{\hat{\eps},\eps}$ for the character that defines $\hat{M}_{\hat{\eps},\eps}$, 
it suffices to verify that $\tree \in \mfT_{-} \subset \mfT \Rightarrow \ell_{\hat{\eps},\eps}[\tree] = 0$.
Recall that $\ell_{\hat{\eps},\eps}$ is given by the composition of the negative twisted antipode with  $\E[\CE_{\hat{\eps}}(\PPi^{\BPHZ}_\eps)[\cdot](0)]$. 
The negative twisted antipode takes trees in $\tree \in \mfT_{-} $ into forests of trees in $\mfT$ which contain at least one tree $\tilde{\tree} \in \mfT_{-}$. 
Note that for $\tilde{\tree} \in \mfT_{-}$ we have $\E[\CE_{\hat{\eps}}(\PPi^{\BPHZ}_\eps)[\tilde{\tree}](0)] = \E[\PPi^{\BPHZ}_\eps[\tilde{\tree}](0)] = 0$ since $\PPi^{\BPHZ}_{\eps}$ is a BPHZ model in $\MM$.
Therefore, $\ell_{\hat{\eps},\eps}[\tree] = 0$.

What remains to be shown is that, for fixed $\hat{\eps} > 0$,  the limit $M_{\hat{\eps}}\eqdef  \lim_{\eps \downarrow 0} M_{\hat{\eps},\eps}  $ exists. 
Here, it suffices to show that for every $\tree \in \mfT$, $\E[\CE_{\hat{\eps}}(\PPi^{\BPHZ}_\eps)[\tree](0)]$ converges as $\eps \downarrow 0$.

Since the models $\CE_{\hat{\eps}}(\PPi^{\BPHZ}_\eps)$ themselves converge, we have the convergence of $\E\big[\CE_{\hat{\eps}}(\PPi^{\BPHZ}_\eps)[\tree](0)\big]$ for $\deg(\tree) > 0$ as $\eps \downarrow 0$. 
If $\tree \in \mfT_{-}$ then we may assume that $\tree \in \hat{\mfT}^{\even}_{-}$, where $\hat{\mfT}^{\even}_{-}$ consists of all $\tree \in \hat{\mfT}_{-}$ which have even parity in both space and noise, since otherwise the expectation vanishes. 

Now, if $\tree \in \hat{\mfT}^{\even}_{-}$ has a vanishing polynomial label then $\CE_{\hat{\eps}}(\PPi^{\BPHZ}_\eps)[\tree](\cdot)$ is space-time stationary, and so for any fixed smooth space-time test function $\phi$ integrating to $1$, we have
\begin{equ}\label{eq:stationary_symbol}
\E\big[\CE_{\hat{\eps}}(\PPi^{\BPHZ}_\eps)[\tree](0)\big]
=
\E\big[\CE_{\hat{\eps}}(\PPi^{\BPHZ}_\eps)[\tree](\phi)\big]
\end{equ}
Again, convergence of the random models $\CE_{\hat{\eps}}(\PPi^{\BPHZ}_\eps)[\tree]$ implies the convergence of the right-hand side above. 

Now suppose that $\tree \in \hat{\mfT}^{\even}_{-}$ has a non-vanishing polynomial label, then it must be the case that $\tree$ contains precisely one factor of $\mbX_{j}$ for some $j \in \{1,2,3\}$. 
Since $\CE_{\hat{\eps}}(\PPi^{\BPHZ}_\eps)$ is a ``stationary model'' in the sense of \cite[Def~6.17]{BHZ19} it follows that, for any $h \in \R^{d+1}$, one has
\[
\CE_{\hat{\eps}}(\PPi^{\BPHZ}_\eps)[\tree](\cdot + h)
\quad
\text{ and }
\quad
\CE_{\hat{\eps}}(\PPi^{\BPHZ}_\eps)[\tree](\cdot )
+
h_{j}
\CE_{\hat{\eps}}(\PPi^{\BPHZ}_\eps)[\bar{\tree}](\cdot )
\]
are equal in law, where $\bar{\tree}$ is obtained from $\tree$ by removing the factor $\mbX_{j}$. 
It follows that, for any fixed smooth space-time test function $\phi$ integrating to $1$ that satisfies $\int_{\R^{d+1}} z_{j} \phi(z) = 0$, one again has \eqref{eq:stationary_symbol}.
\end{proof}

We can now state the key construction in our argument for proving Lemma~\ref{lemma:injectivity_in_law}. 

\begin{lemma}\label{lemma:observable} 
Let $\tilde{\tau}$ be a stopping time with respect to the canonical filtration on $(\state \times \tilde{\mfG}^{0,\rho})^\sol$ dominated by the blow-up time $\tau^{\ast}$. 

There exists a measurable map 
$O^{\tilde{\tau}} \colon (\state \times \tilde{\mfG}^{0,\rho})^\sol \rightarrow \mcb{D}'((0,1) \times \T^{3},E)$ 
with the following property.
 
Fix any $(\mathring{C}_{\A},\mathring{C}_{\h}) \in L(\mfg,\mfg)^{2}$ and $(\bar{x},g_0) \in \state \times \mfG^{0,\rho}$ and write $(\bar{X},\bar{U},\bar{h}) =  \mathcal{A}_{\sig}^{\BPHZ}[\mathring{C}_{\A},\mathring{C}_{\h},\bar x,g_0]$.
Then 
\begin{equ}\label{eq:observable_injectivity}
O^{\tilde{\tau}}( \bar{X},\bar{U},\bar{h} )
=
 \big[ \mathring{C}_{\mfz} \bar{X} +  \mathring{C}_{\mfh} \bar{h} \big]1_{(0,\tilde{\tau})} + \sig \eta^{\tilde{\tau}} 
\end{equ}
for almost every $\xi \in {\Omega}^{\noise}$ where the random distribution $\eta^{\tilde{\tau}}$ is defined by setting, for any deterministic $\psi \in \mcb{D}( (0,1) \times \T^{3}, E)$, 
\[
\eta^{\tilde{\tau}}(\psi) = \int_{0}^{\tilde{\tau}}   \scal{ \bar{U} \psi, \mrd W(s)}\ \mrd s\;.
\]
Above, $W$ is the $L^2(\T^{3},E)$-cylindrical Wiener process associated to the noise $\xi$.
\end{lemma}

\begin{proof}
Without loss of generality, we set $\sig = 1$ and drop it from the notation. 
Heuristically we want to define $O(\bar{X},\bar{U},\bar{h})$ as
\begin{equ}
\text{``\ } 
\Big[ \partial_{t}\bar{X} - \Delta  \bar{X}  + (0^{\oplus 3} \oplus m^2) \bar{X}
- \bar{X}^{3}- \bar{X} \partial \bar{X}
 \Big]  \text{\ ''} \text{ on } (0, \tilde{\tau}) \times \T^{3}\;,
\end{equ}
where we are using the notation of \eqref{e:XdX-X3} for the fourth and fifth terms on the right-hand side.

The first three terms on the right-hand side are canonically well-defined but the fourth and fifth are not. 
However, we know how to construct, for $(\bar{X},\bar{U},\bar{h})$ in the statement of the lemma, the renormalised nonlinearities $\bar{X}^{3}$ and  $\bar{X} \partial \bar{X}$, and the crucial point is that the construction of these nonlinearities has no dependence on the choice of $\mathring{C}_{\A}$ and $\mathring{C}_{\h}$ because our renormalisation counter-terms do not depend on these constants.\footnote{For independence with respect to $\mathring{C}_{\A}$ our particular choice of the renormalisation character plays a role here, see Remark~\ref{rem:bare-m-renorm_2}.}

Let $\mathcal{X} = (\mathcal{X}_{\mft}: \mft \in \Lab_{+})$ be the limiting modelled distributions that solve the fixed point problems \eqref{e:fix-pt-HU} and \eqref{e:fix-pt-multi-Xbar} \dash in particular, 
\begin{equ}[e:component_reminder]
\CY = \mathcal{X}_{\mfa} + \mathcal{X}_{\mfm}\;,\quad
\CU = \mathcal{X}_{\mfu}\;,
\quad \text{ and }\quad
\CH = \mathcal{X}_{\mfh} + \mathcal{X}_{\mfh'}\;.
\end{equ}

We perform a trivial enlargement of our fixed point problems by including a second set of components which represent the $\hat{\eps}$-mollification of the first.

Note that, for $\hat{\eps} \ge 0$, one can pose and solve a fixed point problem for modelled distributions $\CX = (\CX_{\mft}: \mft \in \Lab_{+})$ and  $\CX^{[\hat{\eps}]} = (\CX^{[\hat{\eps}]}_{\hat{\mft}}: \mft \in \Lab_{+})$ of the form 
\begin{equs}[e:timeloc_fixpblm]
\mathcal{X}_{\mft} 
&= 
\mathcal{G}^{\vec{\omega}}_{\mft}
\big[
\bar{F}_{\mft} \big( (\mathcal{X}_{\mft}: \mft \in \Lab_{+}) \big) 
\big]\;,\\
\mathcal{X}^{[\hat{\eps}]}_{\hat{\mft}} 
&= 
\mathcal{G}^{[\hat{\eps}],\vec{\omega}}_{\hat{\mft}}
\big[
\bar{F}_{\mft} \big( ( \mathcal{X}_{\mft}: \mft \in \Lab_{+}) \big) 
\big]\;,
\end{equs}
where what we wrote for the system for $\CX$ is a shorthand for the system appearing in Lemma~\ref{lem:fixedptpblmclose}, and $\vec{\omega}$ represents the data regarding ill-defined products at $t=0$ that is needed to pose the fixed point problem.  

Then $\CG^{[\hat{\eps}],\vec{\omega}}_{\hat{\mft}}$ is an analogous shorthand for the abstract lift of the kernel $G \ast \moll^{\hat{\eps}}$ or $\nabla G \ast \moll^{\hat{\eps}}$ with rough part $K^{[\hat{\eps}]}_{\hat{\mft}}$ and augmented by the same data as the $\tilde{\mathcal{X}}$ \dash  note that the same collection of products ill-defined at $t=0$ appear in both equations above. 

Furthermore, for any $\hat{\eps} \ge 0$, the fixed point problem above, augmented by the $\eps$-inputs $\omega$ built from $\xi$ along with the model $\PPi^{\BPHZ}$, admits solutions in appropriate spaces of modelled distributions on $[0,\hat{\tau}) \times \T^{3}$ for some $\hat{\tau} > 0$.  
Henceforth we use the notation $\mathcal{X}$ and $\mathcal{X}^{[\hat{\eps}]}$ to refer to the solutions to these fixed point problems. 

We define $(\CY^{[\hat{\eps}]}, \CU^{[\hat{\eps}]},\CH^{[\hat{\eps}]})$ analogously to \eqref{e:component_reminder} but with the role of $\CX$ replaced by $\CX^{[\hat{\eps}]}$. 
We also set, for $\mathring{\CY} \in \{\CY, \CY^{[\hat{\eps}]}\}$, 
\[
\mathfrak{N}(\mathring{\CY})
\eqdef
\mathring{\CY} \partial \mathring{\CY} + \mathring{\CY}^3 -  (0^{\oplus 3} \oplus m^2) \mathring{\CY}\;.
\]
For any model $\hat{\PPi} \in \hat{\MM}$ which satisfies $\pi_{0}(\hat{\PPi}) = \PPi^{\BPHZ}$, we have $\mathcal{R} (\CY,\CU,\CH) = (\bar{X},\bar{U},\bar{h})$ on $[0,\hat{\tau}) \times \T^{3}$ since the equation for $\CX$ does not involve the trees of $\hat{\mfT} \setminus \mfT$,
and we also have the distributional identity
\[
 \big[ \mathring{C}_{\mfz} \bar{X} +  \mathring{C}_{\mfh} h \big] + 
\eta^{\hat{\tau}}
=
\partial_{t} \tilde{\CR} \CY - \Delta \tilde{\CR} \CY
 - \tilde{\CR} \mathfrak{N}(\CY) \quad \text{on} \quad (0,\hat{\tau}) \times \T^3  \;,
\] 
where again $\tilde{\CR}$ is the local reconstruction operator (which avoids $t=0$). 
To finish our proof we would need to construct $ \tilde{\CR} \mathfrak{N}(\CY)$ as a functional of $(\bar{X},\bar{U},\bar{h})$.  

If $\hat{\PPi} \in \hat{\MM}$ is of the form $\iota (\PPi)$ for $\PPi \in \MM$ then one has $\tilde{\CR} \circ \hat{\pi}_{0} = \tilde{\CR}$ and  $\hat{\pi}_{0} \CY^{[0]} = \CY$. 
In particular, we have  $\hat{\pi}
  \mathfrak{N}( \CY^{[0]} ) = 
\mathfrak{N}( \CY )$ and  $\tilde{\CR}    \mathfrak{N}( \CY^{[0]} )= \tilde{\CR} \mathfrak{N}( \CY )$. 
 
For each $\hat{\eps} > 0$, let  $\hat{\PPi}_{\hat{\eps}}$ be as given by \eqref{eq:shifted_model} in Lemma~\ref{lem:shifted_model}. 
We fix a subsequence $\hat{\eps}_{n} \downarrow 0$ such that the models $\hat{\PPi}_{\hat{\eps}_n}$ converge almost surely as $n \rightarrow \infty$. 

By the observations above and Lemma~\ref{lem:shifted_model} we have that, as distributions on $(0,\hat{\tau}) \times \T^{3}$,  
$\tilde{\CR} 
\mathfrak{N}(\CY^{[\hat{\eps}_n]})$ converges to $\tilde{\CR} 
\mathfrak{N}(\CY)$ almost surely as $n \rightarrow \infty$.

Moreover, we claim that there exist $c_{\hat{\eps},1}, c_{\hat{\eps},2} \in L(\mfg^3,\mfg^3)$ and $c_{\hat{\eps},3} \in L(\mfg^3,E)$ such that 
\begin{equs}
\tilde{\CR}  
\mathfrak{N}(\CY^{[\hat{\eps}]})
&=
\bar{X}^{\hat{\eps}}\partial \bar{X}^{\hat{\eps}} + (\bar{X}^{\hat{\eps}})^3  -  (0^{\oplus 3} \oplus m^2) \bar{X}^{\hat{\eps}}
+
\big[
c_{\hat{\eps},1} X
+
c_{\hat{\eps},2} \hat{X}^{\eps}
+
c_{\hat{\eps},3} h
\big] \\
{}&\eqdef 
N^{\hat{\eps}}(\bar{X},\bar{U},\bar{h})
\;,\label{e:mollified_renormalised_nonlin}
\end{equs}

In order to justify this claim one can argue as in  \cite[Sec.~2.8.3]{BCCH21} and include  $\mathfrak{N}(\CY^{[\hat{\eps}]})$ as yet another component of our system associated to some new label $\check{\mfz}$, we write $\check{F} = (\check{F}_{\mft} : \mft \in \hat{\mfL} \sqcup \{\check{\mfz}\})$ the non-linearity for this bigger system. 
Since the model $\CE_{\hat{\eps}}(\PPi^{\BPHZ})$ is a multiplicative model on the sector that $\mathfrak{N}(\CY^{[\hat{\eps}]})$ lives in and $\CE_{\hat{\eps}}(\PPi^{\BPHZ})$ agrees with $\hat{\PPi}_{\hat{\eps}}$ on this same sector, the renormalised nonlinearity that allows us to calculate the left hand side of \eqref{e:mollified_renormalised_nonlin} is 
\[
\check{F}_{\check{\mfz}}(\mathbf{A}) + 
\sum_{\tree \in \hat{\mfT}_{-}(R)}
(\ell_{\hat{\eps}}  \otimes \id_{E}) \bbUpsilon_{\check{\mfz}}[\tree](\mathbf{A})\;,
\]
where $\ell_{\hat{\eps}}$ is the character corresponding to the operator $\hat{M}_{\hat{\eps}}$ in Lemma~\ref{lem:shifted_model}. 
To argue that the second term above agrees with the functional form seen in the bracketed term appearing in the second line of \eqref{e:mollified_renormalised_nonlin}, one observes that the trees contributing to the above formula are obtained as graftings onto trees $\tilde{\tree}$ where $\tilde{\tree}$ is obtained from taking one of the trees in $\big\{ \theta(\hat{\tree}): \hat{\tree} \in \mfT^{\YMH} \setminus \{ \mcb{I}_{\bar{\mfl}}(\bone) \big\}$ and replacing all of its kernel edge $(\mft,p)$ incident to the root in $\theta(\hat{\tree})$  with their $\hat{\eps}$-regularised versions $(\hat{\mft},p)$ \dash here $\theta$ is the map that ``takes $(Y,U,h)$ trees to $(\bar{X},\bar{U},\bar{h})$ trees'' described on p.~\pageref{pageref:Theta}. 
Note that changing edges in this way does not change noise, spatial parity, or degree. 
One can then by hand, work through the recursive formula that defines $\bbUpsilon_{\check{\mfz}}$ and afterwards take advantage of the fact that $\bUpsilon^{\check{F}}_{\hat{\mft}} = \bUpsilon^{\check{F}}_{\mft} = \bUpsilon^{\bar{F}}_{\mft}$ for $\mft \in \{\mfm, \mfh, \mfh', \mfz\}$ to apply the same arguments as in the proof of Lemma~\ref{lem:renorm_linear_in_Y_and_h} to obtain the desired claim.

The key takeaway from \eqref{e:mollified_renormalised_nonlin} is that $N^{\hat{\eps}}(\bar{X},\bar{U},\bar{h})$ is well-defined on all of $(\state \times \tilde{\mfG}^{0,\rho})^{\sol}$.
Finally, we set $O^{\tilde{\tau}}(\bar{X},\bar{U},\bar{h})$ to be the limit of 
\begin{equ}\label{eq:observable_def}
\lim_{n \rightarrow \infty} \;
\big[ 
 \partial_{t} \bar{X}^{\hat{\eps}_{n}} - \Delta \bar{X}^{\hat{\eps}_{n}} 
 - N^{\hat{\eps}_n}(\bar{X},\bar{U},\bar{h}) 
 \big] 
 1_{(0,\tilde{\tau})}\;,
 \end{equ}
if this limit exists in $ \mcb{D}'((0,1) \times \T^{3},E)$, otherwise we set $O^{\tilde{\tau}}(\bar{X},\bar{U},\bar{h}) = 0$. 

Now, by Lemma~\ref{lem:patching}, $\mathcal{A}_{\sig}^{\BPHZ}[\mathring{C}_{\A},\mathring{C}_{\h},\bar x,g_0]$ can be obtained as the patching of the limits of local solutions obtained from Proposition~\ref{prop:SPDEs_conv_zero_state} .  
In particular, the argument above gives us \eqref{eq:observable_injectivity} almost everywhere on ${\Omega}^{\noise}$ if one replaces $\tilde{\tau}$ both in  \eqref{eq:observable_injectivity} and \eqref{eq:observable_def} with $\tau_{1} \wedge \tilde{\tau}$ where $\tau_{1} > 0$ is the $\mbF$-stopping time coming from Proposition~\ref{prop:SPDEs_conv_zero_state} as described in the patching construction of $(\bar{X},\bar{U},\bar{h})$. However, the same argument can again be applied on every epoch $[\tau_{j} \wedge \tilde{\tau} ,\tau_{j+1} \wedge \tilde{\tau}]$ and patching the observables obtained on each epoch together gives us the desired result.  
\end{proof}

\begin{proof}[of Lemma~\ref{lemma:injectivity_in_law}]
Define 
\begin{equ}\label{eq:tilde_tau_def}
\tilde{\tau} \eqdef \inf \Big\{ s > 0: \|(\bar{X}(s),\bar{U}(s),\bar{h}(s))\| > 1+4  \|(\bar{X}(0),\bar{U}(0),\bar{h}(0))\| \Big\}\;,
\end{equ}
where $\|(x,u,h)\|$ in~\eqref{eq:tilde_tau_def} is shorthand for $\Sigma(x) + |(u,h)|_{\tilde{\mfG}^{0,\rho}}$.

Given $\psi \in \CC^\infty(\T^3,E)$ and $Z \in  (\state \times \tilde{\mfG}^{0,\rho})^\sol$, we define 
\[
O_{n}(Z)(\psi) = O^{\tilde{\tau}}(Z)(\psi \tilde{\psi}_{n})\;,
\]
where we choose some $\tilde{\psi}_{n} \in \CC^\infty\big((1/n,2/n),\R\big)$ integrating to $1$.

The desired result then follows by using that $\eta^{\tilde{\tau}}(\psi \tilde{\psi}_{n})$ has expectation $0$
and that the $\mathring{C}_{\mfz} X +  \mathring{C}_{\mfh} h$ appearing on the right-hand side of \eqref{eq:observable_injectivity}
are actually functions of time with values in  $\CC^{\eta}(\T^3,E)$, are uniformly bounded over time and ${\Omega}^{\noise}$ by the definition of $\tilde{\tau}$, and converge almost surely to the desired deterministic values as one takes $t \downarrow 0$.
\end{proof}

\begin{remark}\label{rem:recover_noise}
For fixed $(\mathring{C}_{\A},\mathring{C}_{\h}) \in L(\mfg,\mfg)^{2}$ and $(X_0,g_0) \in \state \times \mfG^{0,\rho}$, Lemma~\ref{lemma:observable} allows us to recover $\eta^{\tilde{\tau}}$ directly from $ \mathcal{A}^{\BPHZ}[\mathring{C}_{\A},\mathring{C}_{\h},\bar x,g_0]$. 
In particular, $x \mapsto  \mathcal{A}^{\BPHZ}[\mathring{C}_{\A},\mathring{C}_{\h},\bar x, \id_{G}] = (X,U,h)$,
restricted to the $X$ component, is a local solution map to \eqref{eq:SPDE_for_A} (where $\mathring{C}_{\h}$ plays no role). 
By applying Lemma~\ref{lemma:observable}  with $\tilde{\tau} = \tau^{\ast}$, we have $\tilde\eta^{\tilde{\tau}} = \xi \restr_{(0, \tau^\star)}$  \dash this means that we can recover the noise $\xi$ directly from the solution of  \eqref{eq:SPDE_for_A} from positive times up to the blow-up time of $X$. 
\end{remark}

\begin{remark}\label{rem:injectivity_local}
Our proof of Lemma~\ref{lemma:observable}
can also be carried out for \textit{local} solutions to \eqref{eq:renorm_bar_a} and \eqref{eq:renorm_b}.
Indeed, one can define a space of stopped local solutions 
\[
(\state \times \tilde{\mfG}^{0,\rho})^{\lsol}
=
\big\{ (\tau,Z) \in (0,1] \times \CC([0,1],\state \times \tilde{\mfG}^{0,\rho}): Z_{s} = Z_{\tau} \text{ for all } s \ge \tau \big\}\;.
\]
Then Propositions~\ref{prop:SPDEs_conv_zero_space-time}+\ref{prop:SPDEs_conv_zero_state} give us, for each fixed $\sig \in \R$, a mapping
\begin{equ}[equ:constants_to_solution_map2]
 L(\mfg,\mfg)^{2}\times \state\times\mfG^{0,\rho} \ni (\mathring{C}_{\A},\mathring{C}_{\h},\bar x, g_0) \mapsto \mathcal{A}_{\sig}^{\BPHZ, \scriptscriptstyle{\lsol}}[\mathring{C}_{\A},\mathring{C}_{\h},\bar x, g_0] \in L^{0}_{\lsol} \;.
\end{equ}
Here $L^{0}_{\lsol}$ is the space of equivalence classes (modulo null sets) of measurable maps from ${\Omega}^{\noise}$ to $(\state \times \tilde{\mfG}^{0,\rho})^{\lsol}$.
Furthermore, writing $(\tau,Z) = \mathcal{A}_{\sig}^{\BPHZ, \scriptscriptstyle{\lsol}}[\mathring{C}_{\A},\mathring{C}_{\h},\bar x, g_0]$, we can enforce that $\tau$ is an $\mbF$-stopping time. 

One can then replace the role of $\tilde{\tau}$ with $\tau$ coming from $(\tau,Z) \in (\state \times \tilde{\mfG}^{0,\rho})^{\lsol}$ in  the argument of Lemma~\ref{lemma:observable}.
This gives a measurable map $O^{\lsol} \colon (\state \times \tilde{\mfG}^{0,\rho})^{\lsol} \rightarrow \mcb{D}'((0,1) \times \T^{3},E)$ such that, for any $(\mathring{C}_{\A},\mathring{C}_{\h}) \in L(\mfg,\mfg)^{2}$ and $(\bar{x},g_0) \in \state \times \mfG^{0,\rho}$, if we write $(\tau,\bar{X},\bar{U},\bar{h}) =  \mathcal{A}_{\sig}^{\BPHZ, \scriptscriptstyle{\lsol}}[\mathring{C}_{\A},\mathring{C}_{\h},\bar x, g_0]$,
then 
\begin{equ}\label{eq:observable_injectivity_local}
O^{\lsol}( \tau,\bar{X},\bar{U},\bar{h} )
=
 \big[ \mathring{C}_{\mfz} \bar{X} +  \mathring{C}_{\mfh} \bar{h} \big]1_{(0,\tau)} + \sig \int_{0}^{\tau}  \scal{ \bar{U} \psi, \mrd W(s)}\ \mrd s\;,
\end{equ}
where $W$ is the $E$-valued Wiener process associated to $\xi$. 
\end{remark}

\section{Symbolic index}

We collect in this appendix commonly used symbols of the article.

 \begin{center}
\renewcommand{\arraystretch}{1.1}
\begin{longtable}{lll}
\toprule
Symbol & Meaning & Page\\
\midrule
\endfirsthead
\toprule
Symbol & Meaning & Page\\
\midrule
\endhead
\bottomrule
\endfoot
\bottomrule
\endlastfoot
$\heatgr{\act}_{\alpha,\theta}$ & 
Extended norm on $\Omega\CD'$ & 
\pageref{def:heatgr}\\
$\$\act\$_{\eps}$ & $\eps$-dependent seminorm on models & \pageref{norm_eps_pageref}\\
$ \mathcal{A}^{\BPHZ}_{\act}(\act)$ &
BPHZ maximal solution map to SYM equation&
\pageref{SYM_solmap} \\
$ \mathcal{A}_{\sig}^{\BPHZ}[\act]$ &
BPHZ local solution map to gauge transformed system &
\pageref{equ:constants_to_solution_map} \\
$\mcb{A}$ &
Target space of jets of noise
and solution $\mcb{A}\eqdef \prod_{o\in\CE}W_o$ &
\pageref{mcb A}\\
$\CB^{\beta,\delta}$ &
Banach space for the lifted quadratic objects &
\pageref{def:B-beta-delta} \\
$\mathring{C}_{\A}$, $m^2$, $\mathring{C}_{\h}$ &
Parameters in unrenormalised equations &
\pageref{thm:local_exist},~\pageref{eq: B system} \\
$C^{\eps}_{\YM}$, $C^{\eps}_{\Higgs}$ &
BPHZ mass renormalisation maps for $(A,\Phi)$ &
\pageref{thm:local_exist},~\pageref{thm:local-exist-sigma} \\
$C_{\A}^{\eps}$, $C_{\Phi}^{\eps}$ &
Masses in BPHZ renormalised equations for $(A,\Phi)$ &
\pageref{eq:SPDE_for_A},~\pageref{e:SPDE-for-X} \\
$ C^{\delta,\eps}_{\YM}$, $C^{\delta,\eps}_{\Higgs}$ &
BPHZ mass renormalisation maps for $(\bar A,\bar\Phi)$ &
\pageref{prop:renorm_g_eqn} \\
$C^{\eps}_{\Gauge}$, $C^{\delta,\eps}_{\Gauge}$ & 
BPHZ  ``$h$-renormalisation'' in $B$ or $\bar A$ equation &
\pageref{prop:renorm_g_eqn} \\
$\check{C}$ &
The crucial map for obtaining gauge covariance &
\pageref{thm:gauge_covar} \\
$\hat\cD^{\gamma,\eta}$ &
A subspace of  $\cD^{\gamma,\eta}$ &
\pageref{def:Dhat-space} \\
$d_{\eps}(\act,\act)$ & $\eps$-dependent metric  on models  ($d_1$ is $d_\eps$ with $\eps=1$)& \pageref{d_eps_pageref}\\
$E$ & 
Shorthand for $E=\mfg^3\oplus \higgsvec$ & 
\pageref{def of E}\\
$\CE$ &
 The set of edge types $\CE = \Lab \times \N^{d+1}$ &
\pageref{def CE} \\
$\flow_t(X)$ &
Deterministic YMH flow without DeTurck term &
\pageref{def:flow_no_deturck} \\
$\ymhflow_t(X)$ & 
Deterministic YMH flow with DeTurck term & 
\pageref{def:ymhflow}\\
$\mbF$ &
Filtration $\mbF=(\mcF_t)_{t\geq 0}$ generated by the noise &
\pageref{pageref:mbF},\pageref{pageref:mbF2} \\
$F^\sol$ &
Functions into $\hat F\eqdef F\cup\{\skull\}$ with finite-time blow up &
\pageref{ref:Fsol} \\
$\mfG^\rho$ & 
The gauge group $\CC^\rho(\T^3,G)$ & 
\pageref{mfG page ref}\\
$\mfG^{0,\rho}$ & 
Closure of smooth functions in $\mfG^\rho$ & 
\pageref{mfG^0 page ref}\\
$\tilde{\mfG}^{\rho}$ & 
Space $\CC^\rho  \times \CC^{\rho - 1}$ ``of $(U,h)$'s''& 
\pageref{hatmfG^rho page ref}\\
$\tilde{\mfG}^{0,\rho}$ & 
Closure of smooth functions in $\tilde{\mfG}^\rho$ & 
\pageref{hat_mfG^0rho page ref}\\
$\CG_\mfz$ &  Integration operator   $\CG_\mfz \eqdef \mcb{K}_{\mfz} + R \mathcal{R}$ & \pageref{e:abs-fix-pt-X0}\\
$\init$ & 
``First half'' of the space for initial conditions & 
\pageref{def:init}\\
$\mcb{I}_{\mft}$ &
Integration operator for type $\mft$ on regularity structures &
\pageref{eq:integration_intertwining}\\
$K$ &
Truncation of the heat kernel $G$ &
\pageref{truncation of heat}\\
$K^{(\eps)}$ &
Kernel assignment for the gauge transformed system &
\pageref{Keps-assign}\\
$ \mcb{K}_{\mfz} $  & Integration operator for type $\mfz$ on modelled distributions & \pageref{e:abs-fix-pt-X0}\\
$\bar{\mcb{K}}$  &
Short-hand for $\mcb{K}_\mfm$, the abstract integration for $K^\eps$  &
\pageref{bar-mcbK} \\
 $\mcb{K}^\omega $ &
 Integration operator with an input distribution $\omega$  &
\pageref{lem:Schauder-input} \\
$L_G(V,V)$ &
$G$-invariant linear maps from a space $V$ to itself &
\pageref{eq:SPDE_for_A}\\
$\mathscr{M}_\eps$ &
The family of $K^{(\eps)}$-admissible models &
\pageref{model page ref} \\
$\CN$ & 
Quadratic mapping $X\mapsto \CP_t X\otimes \nabla \CP_t X$ & 
\pageref{eq:CN_def}\\
$\mfO^\infty$ &
Smooth quotient space &
\pageref{mfO infinity} \\
$\mfO$ &
Quotient space $\mfO\eqdef \state/{\sim}$ &
\pageref{def:orbits}\\
$\Omega^\rho$ & 
The Banach space $\Omega^1_{\gr\rho}\times \CC^{\rho-1}(\higgsvec)$ & 
\pageref{omega rho page ref}\\
$\CP$ &
$\CP_t X \eqdef e^{t\Delta}X$, or its lift as a modelled distribution&
\pageref{def:N-B},~\pageref{harmonic extension}\\
 $R$ &  The operator realising convolution with $G - K$ & \pageref{e:abs-fix-pt-X0}\\
$\tilde \CR$ &
``Local'' reconstruction operator away from $\{t=0\}$ &
\pageref{def:tildeR} \\
$\state$ & 
State space for stochastic YMH flow & 
\pageref{def:state}\\
$\Sigma$ & 
Metric (resp.\ extended metric) on $\state$ (resp.\ $\init$) & 
\pageref{e:def-Sigma}\\
$\Theta$ &
Metric on $\init=\init_{\eta,\beta,\delta}$ &
\pageref{def:init} \\
$W_\mft$ &
Target space assignment for type $\mft$ &
\pageref{W additive},~\pageref{e:system-target space assignment}
\end{longtable}
 \end{center}


\endappendix

\bibliographystyle{Martin}
\bibliography{./refs}


\end{document}